\documentclass[a4paper,12pt]{article}
\usepackage{amsfonts}
\usepackage{amsmath}
\usepackage{amsthm}
\usepackage{amssymb}
\usepackage{latexsym}
\begin{document}

\newtheorem{lemma}{Lemma}[subsection]
\newtheorem{corollary}{Corollary}[subsection]
\newtheorem{theorem}{Theorem}[subsection]
\newtheorem{definition}{Definition}[subsection]
\newtheorem{example}{Example}[subsection]
\newtheorem{remark}{Remark}[subsection]
\newtheorem{claim}{Claim}[subsection]
\newtheorem{proposition}{Proposition}[subsection]
\newtheorem{reference}{}

\def\GL{\mathop{\rm GL}\nolimits}
\def\PSL{\mathop{\rm PSL}\nolimits}
\def\PGL{\mathop{\rm PGL}\nolimits}
\def\SL{\mathop{\rm SL}\nolimits}
\def\SO{\mathop{\rm SO}\nolimits}
\def\PSO{\mathop{\rm PSO}\nolimits}
\def\Spin{\mathop{\rm Spin}\nolimits}
\def\PSp{\mathop{\rm PSp}\nolimits}
\def\Sp{\mathop{\rm Sp}\nolimits}
\def\Im{\mathop{\rm Im}\nolimits}
\def\Char{\mathop{\rm char}\nolimits}
\def\Lie{\mathop{\rm Lie}\nolimits}
\def\Ad{\mathop{\rm Ad}\nolimits}
\def\Inf{\mathop{\rm Inf}\nolimits}
\def\Irr{\mathop{\rm Irr}\nolimits}
\def\End{\mathop{\rm End}\nolimits}

\def \proof{{\bf Proof:~}}
\def \notation{{\bf Notation:~}}
\def \remark{{\bf Remark:~}}
\def \claim{{\underline{Claim:}}}
\def \examp{{\bf Example:~}}

\begin{titlepage}
%\vspace*{0.5ex}
\begin{center}

{\Huge\sc Exceptional Representations \vspace*{2.5ex}\\
       of Simple Algebraic Groups \vspace*{2.5ex}\\
       in Prime Characteristic\vspace*{8.5ex}\\}

{\Large Marin\^es Guerreiro \vspace{4ex}\\
Department of Mathematics\vspace{5ex}\\ 1997\vspace*{15ex}\\}

{\normalsize
\sc A thesis submitted to the University of Manchester \\ % \vspace{0.5ex}\\
for the degree of Doctor of Philosophy \\ % \vspace{0.5ex}\\
in the Faculty of Science}
\end{center}
\end{titlepage}

\newpage
\setcounter{page}{2}

\tableofcontents
\newpage
\listoftables

\newpage
\part*{Abstract}
\addcontentsline{toc}{section}{\protect\numberline{}{Abstract}}
\vspace{10ex}

Let $\,G\,$ be a simply connected simple algebraic
group over an algebraically closed field $\,K\,$ of
characteristic $\,p>0\,$ with root system $\,R,\,$
and let $\,{\mathfrak g}\,=\,{\cal L}(G)\,$ be its restricted Lie algebra. 
Let $\,V\,$ be a finite dimensional $\,{\mathfrak g}$-module over $\,K.\,$ 
For any point $\,v\,\in\,V$,
% the {\it isotropy subgroup} of $\,v\,$ in $\,G\,$ is 
%$\,G_v\,=\,\{\,g\,\in\,G\,/\,g\cdot v\,=\,v\,\}\,$ and 
the {\it isotropy subalgebra} of $\,v\,$ in $\,\mathfrak g\,$ is 
$\,{\mathfrak g}_v\,=\,\{\,x\,\in\,{\mathfrak g}\,/\,x\cdot
v\,=\,0\,\}\,$

A restricted $\,{\mathfrak g}$-module $\,V\,$ is called {\bf exceptional} if
for each $\,v\in V\,$ the isotropy subalgebra $\,{\mathfrak g}_v\,$ contains 
a non-central element (that is, $\,{\mathfrak g}_v\,\not\subseteq\,
{\mathfrak z(\mathfrak g)}\,$).

This work is devoted to classifying irreducible exceptional
$\,\mathfrak g$-modules.
A necessary condition for a $\,\mathfrak g$-module to be exceptional
is found and a complete classification of modules over 
groups of exceptional type is obtained.
For modules over groups of classical type, the general problem is
reduced to a short list of unclassified modules.

The classification of exceptional modules is expected to have applications
in modular invariant theory and in classifying modular simple Lie
superalgebras.

\newpage

\part*{Declaration}
\vspace{10ex}
\begin{center}
No portion of the work referred to in this thesis has been \\
submitted in support of an application for another degree\\
or qualification of this or any other institution of learning.
\end{center}

%\newpage

\part*{Copyright}
\vspace{10ex}

Copyright in text of this thesis rests with the Authoress. Copies (by
any process) either in full, or of extracts, may be made {\bf only} in
accordance with instructions given by the Authoress and lodged in the
John Rylands University Library of Manchester. Details may be obtained from
the Librarian. This page must form part of any such copies made. 
Further copies (by any process) of copies made in accordance with such
instructions may not be made without the permission (in writing) of
the Authoress.

\newpage
\part*{The Authoress}
\addcontentsline{toc}{section}{\protect\numberline{}{The Authoress}}

\vspace{5ex}

I graduated from Universidade Federal de Santa Maria, Santa Maria -
RS, Brazil in 1988 with a Teaching degree in
Mathematics (Licenciatura em Matem\'atica). 
I was awarded an M. Sc. in Mathematics (Mestre em Matem\'atica) 
from Universidade
de Bras\'{\i}lia, Bras\'{\i}lia - DF, Brazil in 1991. My
M. Sc. Dissertation was ``On finite groups admitting automorphisms of
prime order with centralizers of controlled order''. Since
August 1991 I have been a Lecturer in the Department of Mathematics of the
Universidade Federal de Vi\c cosa, Vi\c cosa - MG, Brazil.
I commenced the studies for this degree at the University of
Manchester in October 1993, under Professor Brian Hartley's
supervision. This thesis is the result of research work done 
since January 1995 under the supervision of Dr. Alexander Premet.

\begin{center}
\vspace*{12ex}

{\bf\Large\it Dedicated to \vspace{2ex}\\

Professor Brian Hartley\vspace{2ex}\\

(In Memoriam)\vspace{12ex}}

{\it I thank God for the opportunity 

of working with a great Master,

even for such a short time.} 

\end{center}

\newpage

\part*{Acknowledgements}
\addcontentsline{toc}{section}{\protect\numberline{}{Acknowledgements}}

I thank God for more this achievement in my life.

I thank my supervisor, Dr. Alexander Premet, for teaching me
the Representation Theory of Algebraic
Groups and Lie Algebras, and for his guidance throughout the
development of this work.

I also thank Dr. Grant Walker for his support 
during the most difficult time of my course, for his constant encouragement, 
and for very useful hints on solving combinatorial problems.

I would like to thank my office mates Ersan Nal\c cacio\u{g}lu, 
Mohammad Ali Asadi and Suo Xiao for the
good chats that helped to keep our spirits up.

I have very good friends that have helped me to stay sane
whilst doing my Ph. D.. In particular, I would like to
thank Noreen Pervaiz for being such a good hearted person and an excellent
flatmate and Jackie Furby for her friendship. 

A very special thanks to David D. Garside for helping me to cope with
the stress %so closely, for understanding my priorities and
and mainly for bringing so much happiness into my heart
and life over the last year or so. I also thank his family 
%- Jean, Roy and Grandad - 
for being so caring and generous. 

I thank the Universidade Federal de Vi\c cosa (MG - Brazil)
and all my colleagues of the Departamento de Matem\'atica for their
support throughout these years. In particular, I would like to 
thank Jos\'e Geraldo Teixeira for always attending promptly my
numerous enquiries.

I am forever in debt to my good friend Maria de Lourdes Carvalho for her
extreme dedication on looking after all my things in Brazil, as if they
were hers. 

I also would like to thank Dr. Nora\'{\i} Rocco for his advice during
the moments of important decisions and for always keeping one eye at my
process at CAPES. 

Finally, I most gratefully acknowledge over four years of financial support
given to me by CAPES - MEC - Brazil.

\newpage

\vspace{20ex}
\part*{Chapter 1}
\part*{Introduction}
\addcontentsline{toc}{section}{\protect\numberline{1}Introduction}
\vspace*{8ex}

Let $\,G\,$ be a simply connected simple algebraic
group over an algebraically closed field $\,K\,$ of
characteristic $\,p>0,\,$ and let
$\,{\mathfrak g}\,=\,{\cal L}(G)\,$ be its restricted Lie algebra. 
Let $\,V\,$ be a finite dimensional vector space over $\,K.\,$
If $\,G\,$ acts on $\,V\,$ via a rational representation
$\,\pi\,:\,G\,\longrightarrow\,GL(V),\,$ then $\,{\mathfrak g}\,$ acts
on $\,V\,$ via the differential 
$\,d\pi\,:\,{\mathfrak g}\,\longrightarrow\,{\mathfrak
{gl}}\;(V)$ (see Section~\ref{tlaag}). 
For any point $\,v\,\in\,V$, the {\it isotropy subgroup} 
of $\,v\,$ in $\,G\,$ is $\,G_v\,=\,\{\,g\,\in\,G\,/\,
g\cdot v\,=\,v\,\}\,$ and the {\it isotropy subalgebra} of
$\,v\,$ in $\,\mathfrak g\,$ is 
$\,{\mathfrak g}_v\,=\,\{\,x\,\in\,{\mathfrak g}\,/\,x\cdot
v\,=\,0\,\}\,$ (see Section~\ref{sec3.2}). 

This thesis is devoted to classifying irreducible rational
$\,G$-modules $\,V\,$ for which the isotropy subalgebras of all points have at
least one non-central element. We call such modules {\bf exceptional}
(Definition~\ref{excpmod}).
For Lie algebras of exceptional type a complete classification
is obtained. For Lie algebras of classical type partial results are
given.

The classification of such modules is interesting for the following
reasons. 

For a rational action of an algebraic group $\,G\,$ on  
a vector space $\,V$, we say that a subgroup $\,H\,\subset\,G\,$ is an
{\it isotropy subgroup in general position} (ISGP for short) if $\,V\,$
contains a non-empty Zariski open subset $\,U\,$ whose points have
their isotropy subgroups conjugate to $\,H$. The points of $\,U\,$ are
called {\it points in general position}.
Isotropy subalgebras $\,{\mathfrak h}\,\subset\,{\mathfrak g}\,$ in general
position are defined similarly.

A rational $\,G$-action $\,G\longrightarrow \GL(V)\,$ 
is said to be {\it locally free} if the ISGP $\,H\,$ exists
and equals $\,\{ e\}$. Similarly, a linear $\,\mathfrak g$-action $\,\mathfrak
g\longrightarrow\mathfrak{gl}(V)\,$ is said to be {\it locally free} if the
isotropy subalgebra in general position $\,\mathfrak h\,$ exists
and equals $\,\{ 0\}$.

In \cite[1967]{ave}, E.M. Andreev, E.B. Vinberg and A.G. \'Elashvili
have given a necessary and sufficient condition for an irreducible
action of a complex simple Lie group $\,G\,$ on a vector space $\,V\,$
to have zero isotropy subalgebra in general position. 
Their method is based on the notion of the {\it index of the trace
form} associated to a representation and cannot be used
in positive characteristic, as in this case the trace form may vanish.

Let $\,G\,$ be a simple Lie group and $\,V\,$ be a finite  
dimensional vector space over a field of characteristic zero.
Define the {\it trace form} associated to the representation  
$\,\psi\,:\,G\longrightarrow\GL(V)\,$ by $\,\kappa_{\psi}(X,\,Y)\,=\, 
{\rm tr}\,({\rm d}\psi(X)\circ {\rm d}\psi(Y))\,$ for all $\,X,\,Y\in
G\,$. 

Let $\,G\,$ act irreducibly on $\,V.\,$ 
The {\it index} $\,l_V\,$ of the trace form associated to $\,\psi\,$
is given by
$\,{\rm tr}({\rm d}\psi(X))^2\,=\,l_V\,{\rm tr}({\rm ad} 
({\rm d}\psi(X)))^2\,$, for $\,X\in G.\,$ 
$\,l_V\,$ is a positive 
rational number, independent of $\,X\,$ \cite[p. 257]{ave}.
The index of a direct sum of irreducible representations is the sum of
the indices of the irreducible representations \cite[p. 260]{ave}.

The index of a one-dimensional representation of a simple Lie algebra
$\,L\,$ is $\,0\,$, and, for irreducible representations, it takes
value $\,1\,$ only for the adjoint representation
\cite[p. 259]{ave},~\cite[1.4.3]{kac1}. 
The sufficient condition obtained by Andreev, Vinberg and 
\'Elashvili establishes that if the index of the trace form of
$\,\pi\,$ is greater than 1, then $\,\pi\,$ is locally free
\cite[Theorem]{ave}. 

Table~\ref{table1} has been taken from~\cite{ave} and 
lists the irreducible representations of the simple Lie algebras
for which $\,0\,<\,l_V\,<\,1\,$ (up to graph automorphism).

\begin{table}[p]
\begin{tabular}{llllll}
\hline\hline
Type &    & Rank  &  Highest   & $\dim\,V$   & Index \\ 
     &    &         &  Weight  &  & \\ \hline\hline \vspace{.5ex}
A    & $\mathfrak{sl}_{n}$ & $\,n-1\,$ & $\,\omega_1\,$   
& $n$ & $\displaystyle\frac{1}{2n}$ \vspace{.5ex}\\
    & $\textstyle\bigwedge^2\mathfrak{sl}_{n}$ & $\,n-1\,$ &    
 $\,\omega_2\,$ & $\displaystyle\frac{n\,(n-1)}{2}$ &  
$\displaystyle\frac{n-2}{2n}$ \vspace{.5ex} \\
    & $S^2\mathfrak{sl}_{n}$ & $\,n-1\,$ & $\,2\omega_1\,$ 
& $\displaystyle\frac{n\,(n+1)}{2}$ &  
$\displaystyle\frac{n+2}{2n}$ \vspace{.5ex} \\
    & $\textstyle\bigwedge^3\mathfrak{sl}_{n}$ & $\,5,6,7\,$ &
     $\,\omega_3\,$  & $20,\,35,\,56$  &
$\displaystyle\frac{1}{2},\,\frac{5}{7},\,\frac{15}{16} $\\
& $n\,=\,6,\,7,\,8$ & & & &  \vspace{.5ex}\\ \hline \vspace{.5ex}
B   & $\mathfrak{so}_{n}$ & $\,\displaystyle\frac{n-1}{2}\,$ & $\,\omega_1\,$ 
& $n$ & $\displaystyle\frac{1}{n-2}$ \vspace{.5ex} \\ 
  & ${\rm spin}_{n}$ & $\,3,\,4,\,5,\,6\,$ & $\,\omega_{\ell}\,$    
& $8,\,16,\,32,\,64$  &
$\displaystyle\frac{1}{5},\,\frac{2}{7},\,\frac{4}{9},\,\frac{8}{11}$ \\
& $n\,=\,7,\,9,\,11,\,13$ & & & &  \vspace{.5ex}\\ \hline \vspace{.5ex}
C    & $\mathfrak{sp}_{n}$ & $\,\displaystyle\frac{n}{2}\,$ & $\,\omega_1\,$  
& $n$ & $\displaystyle\frac{1}{n+2}$ \vspace{.5ex}\\
    & $\textstyle\bigwedge^2_0\mathfrak{sp}_{n}$ &  
$\,\displaystyle\frac{n}{2}\,$
&  $\,\omega_2\,$ & $\displaystyle\frac{n\,(n-1)}{2}-1$ &  
$\displaystyle\frac{n-2}{n+2}$ \vspace{.5ex} \\
    & $\textstyle\bigwedge^3_0\mathfrak{sp}_{6}$ &  $\,3\,$ &  $\,\omega_3\,$ 
& $\,14\,$ &  
$\displaystyle\frac{5}{8}$ \vspace{.5ex} \\ \hline  \vspace{.5ex}
D   & $\mathfrak{so}_{n}$ & $\,\displaystyle\frac{n}{2}\,$ & $\,\omega_1\,$ 
& $n$ & $\displaystyle\frac{1}{n-2}$ \vspace{.5ex} \\
    & ${\rm spin}_{n}$ & $\,5,\,6,\,7\,$ & $\,\omega_{\ell}\,$    
& $\,16,\,32,\,64$  &
$\displaystyle\frac{1}{4},\,\frac{2}{5},\,\frac{2}{3} $\\
& $n\,=\,10,\,12,\,14$ & & & &  \vspace{.5ex}\\ \hline \vspace{.5ex}
E & $\,E_6\,$  & $6$ & $\,\omega_1\,$  & $27$ & 
 $\displaystyle\frac{1}{4} $  \vspace{.5ex}\\ 
 & $\,E_7\,$  & $7$ & $\,\omega_7\,$  & $56$ & $\displaystyle\frac{1}{3} $  
\vspace{.5ex}\\ \hline \vspace{.5ex}
F & $\,F_4\,$  & $4$ & $\,\omega_4\,$ & $26$ & $\displaystyle\frac{1}{3} $  
\vspace{.5ex}\\ \hline \vspace{.5ex}
G & $\,G_2\,$  & $2$ & $\,\omega_1\,$  & $7$ & $\displaystyle\frac{1}{4} $  
\vspace{.5ex}\\ \hline \hline  
\end{tabular}
\caption[Irreducible Representations with index
$\,0\,<\,l_V\,<\,1\,$ - Characteristic Zero]{\label{table1}}
\end{table}
%\pagebreak

In Table~\ref{table1}, $\,\mathfrak{sl}_{n}\,$, $\,\mathfrak{sp}_{n}\,$, and 
$\,\mathfrak{so}_{n}\,$ stand for natural representations of these Lie
algebras; $\,S^k\,$ and $\,\textstyle\bigwedge^k\,$ denote  
$\,k$th symmetric and
exterior powers, respectively. 
One can show that the set of weights of the $\,G$-module %$\,S^k\,$ 
$\,\textstyle\bigwedge^k\,$ contains a unique maximal weight. We
denote by $\,\textstyle\bigwedge^k_0\,$
%$\,S^k_0\,$  (resp. $\,\textstyle\bigwedge^k_0\,$)
the corresponding irreducible component of $\,\textstyle\bigwedge^k\,$.
%$\,S^k\,$ (resp. $\,\textstyle\bigwedge^k\,$).
%Decompose $\,S^k\,$ (resp. $\,\textstyle\bigwedge^k\,$)
%into irreducible components. Let $\,\lambda_0^k\,$ be the maximal
%weight in the set $\,{\cal X}_{++}(S^k)\,$ (resp. $\,{\cal
%X}_{++}(\textstyle\bigwedge^k)\,$) and denote by $\,S^k_0\,$ (resp. 
%$\,\textstyle\bigwedge^k_0\,$)
%the corresponding irreducible component of $\,S^k\,$
%(resp. $\,\textstyle\bigwedge^k\,$).
$\,{\rm spin}_{n}\,$ stands for the irreducible
spinor representation of $\,\mathfrak{so}_{n}\,$; $\,G_2,\,F_4,\,
E_6,\, E_7\,$ denote the natural representations of the
corresponding Lie algebras.
%\pagebreak

This list was extensively used by V.G. Kac to classify
the simple Lie superalgebras in characteristic $\,0\,$~\cite[1977]{kac1}. 
%(Adv. Math. Vol.26, 8-96 (1977)).

In~\cite[1972]{ela1},\cite{ela2}, A.G. \'Elashvili lists the
simple and irreducible semisimple linear Lie groups $\,G\,$ such that
the ISGP's $\,H\,$ for the action of $\,G\,$ on $\,V\,$ have positive
dimension. He also proves the existence of 
isotropy subalgebras in general position for actions of simple, and 
irreducible semisimple, linear Lie groups, and finds an  
explicit form for them. 
As proved by Richardson~\cite[1972]{rich}, in characteristic zero, 
the ISGP always exists for an arbitrary rational (linear) action  
of a reductive group on an algebraic variety.
In~\cite[1986]{pop1}, \cite[1987]{pop2}, A.M. Popov classified the
irreducible linear actions of (semi)simple complex linear Lie groups
with finite (but nontrivial) ISGP's.
\medskip

These results find applications in Invariant Theory. 
The determination of ISGP's is interesting for the following reasons.
If $\,m_G\,$ is the maximal dimension of orbits of an algebraic $\,F$-group
$\,G\subset \GL(V)\,$ (acting on a finite dimensional vector space
$\,V\,$ over an algebraically closed field $\,F\,$), then the 
maximal number of algebraically independent rational invariants for the
action of $\,G\,$ on $\,V\,\cong\,F^N\,$ equals $\,N-m_G\,$ \cite{rosen}.
On the other hand, if $\,H\,$ is the ISGP, then $\,m_G\,=\,\dim\,G\,
-\,\dim\,H.\,$ In the case of a semisimple group $\,G\,$, the algebra
$\,F(V)^G\,$ of rational invariants of $\,G\,$ is the field of quotients of
the algebra $\,F[V]^G\,$ of polynomial invariants of $\,G.\,$ 
Therefore, the degree
of transcendence of $\,F(V)^G\,$ equals the Krull dimension of
$\,F[V]^G.\,$ 
% although the existence of ISGP for actions
%of (semi)simple algebraic groups over such fields is not proved and
%is is also not proved that the isotropy subgroups are conjugate.
Hence, $\,{\rm tr.deg.}\,F(V)^G\,=\,\dim\,V\,-\,\dim\,G\,+\,\dim\,H\,$.
Thus, to find the Krull dimension of the 
algebra $\,F[V]^G\,$ in the case of a semisimple group
$\,G\,$ it suffices to know $\,\dim\,H.\,$

Knowledge of the ISGP $\,H\,$ can be used to investigate the properties of
$\,\mathbb C[V]^G.\,$ % (see~\cite{schwarz1},\cite{schwarz2}) and for
%explicitly finding $\,\mathbb C[V]^G\,$ (see 10 and 9 Popov???). 
Namely,
if for an action of a connected reductive linear group $\,G\,$ on a vector
space $\,V\,$ the ISGP $\,H\,$ is reductive (for example, finite), then 
$\,V\,$ contains a Zariski open $\,G$-invariant subset whose points
all have closed orbits \cite{popvl}. 
%(This fact is expected to generalise to prime characteristic.)
Then, for the action of the normalizer $\,N(H)\,$ of the 
subgroup $\,H\,$ in $\,G\,$ on the subspace $\,V^H,\,$ the restriction 
homomorphism $\,\mathbb C[V]^G\longrightarrow \mathbb C[V^H]^{N(H)}\,$ is an 
isomorphism \cite{luna}. Then $\,W:=N(H)/H\,$ acts on $\,V^H\,$
and $\,\mathbb C[V^H]^{N(H)}\,\cong\,\mathbb C[V^H]^W.\,$ 
Thus, knowledge of the reductive (in particular,  
finite) ISGP's $\,H\,$ enables us to reduce the determination of 
$\,\mathbb C[V]^G\,$ to that of $\,\mathbb C[V^H]^W.\,$ 
(These results are expected to generalise to prime characteristic.)

Knowledge of the ISPG is important in constructing some moduli spaces 
in algebraic geometry \cite{dcm}, \cite{mum}. 
Also, information on the ISGP's may be very helpful in establishing
the rationality of the field of invariants $\,\mathbb C(V)^G\,$  
\cite{boka}. 

\bigskip

We now give an overview of this thesis. 
Chapter 2 is a background chapter, where we introduce
the notation used throughout this work.

In Section~\ref{centsection} we give 
some well-known properties of centralizers that are used in the main
sections of Chapter 3. In Section~\ref{sec3.2} we prove 
that, for Lie algebra actions, if the stabilizer of any
point of a Zariski open subset $\,W\,$ of the module $\,V\,$
contains a non-central element then the stabilizer of any point $\,v\in
V\,$ has the same property (see Proposition~\ref{3.1.1}). 
We define the exceptional modules in~\ref{excpmod}.
In Section~\ref{necess} we prove the following theorem, which gives 
a necessary condition for a module to be exceptional.\\
{\bf Theorem~\ref{???}}~~{\it 
Let $\,p\,$ be a non-special prime for $\,G$. 
Let $\,\pi\,:\,G\longrightarrow\GL(V)\,$
be a non-trivial faithful rational representation of $\,G\,$ such that
$\,\ker\,{\rm d}\pi\subseteq\mathfrak z(\mathfrak g)\,$.
If $\,V\,$ is an exceptional $\,\mathfrak g$-module, then it satisfies
the inequalities
% (with the assumptions we have done so far????), then it satisfies
\[
r_p(V)\,:=\,\sum_{\stackrel{\scriptstyle\mu\;good}{\mu\in{\cal X}_{++}(V)}}
\,m_{\mu}\,\frac{|W\mu|}{|R_{long}|}\,
|R_{long}^+-R^+_{\mu,p}|\,\leq |R|\,,
\]
and
\begin{equation}%\label{?ref123}
s(V)\,:=\,\sum_{\stackrel{\scriptstyle\mu\,good}{\mu\in{\cal X}_{++}(V)}}
\,m_{\mu}\,|W\mu|\;\leq\;\mbox{\bf limit}\,,\nonumber
\end{equation}
where $\,m_{\mu}\,$ denotes the multiplicity of the weight $\,\mu\in
{\cal X}_{++}(V)\,$ and the {\bf limit}s for the different types of 
Lie algebras are given in Table~\ref{table2} (see p.~\pageref{table2}).}

Observe that the necessary condition for a module to be exceptional is
given in terms of sums involving orbit sizes of weights. We start
Chapter 4 with some well-known facts on weights, their orbits and  
centralizers with respect to the natural action of the Weyl group.  
In Section~\ref{proced} we describe the procedure used
to classify exceptional modules 
(it relies heavily on Theorem~\ref{???}). Finally, in Section~\ref{resul} we start proving the main results of  
this thesis.

The complete classification of exceptional modules for Lie algebras of
exceptional type is as follows.\vspace{1ex}\\
{\bf Theorem~\ref{frlaet}}~~{\it 
Let $\,G\,$ be a simply connected simple algebraic group of exceptional
type, and $\,\mathfrak g\,=\,{\cal L}(G).\,$ 
Let $\,V\,$ be an infinitesimally irreducible $\,G$-module.
If the highest weight of $\,V\,$ is listed in Table~\ref{table3},
then $\,V\,$ is an exceptional $\,\mathfrak g$-module. If $\,p\,$
is non-special for $\,G\,$, then the modules listed
in Table~\ref{table3} are the only exceptional $\,\mathfrak g$-modules.

\begin{table*}[htb]\begin{center}
\begin{tabular}{llcll}
\hline\hline
Type     &  Rank  &  Weights   & $\dim\,V$ & Module 
 \\ \hline\hline \vspace{.5ex}
$\,E_6\,$  & $6$ & $\,\omega_1\,$  & $27$ & natural \\ \cline{3-5}
	   &     & $\,\omega_2\,$  & $78,\;p>3$ & adjoint \\ 
	   &     &           & $77,\;p=3$  & \\ \cline{3-5}
	   &     & $\,\omega_6\,$  & $27$
\vspace{.5ex} & twisted-natural\\ \hline \vspace{.5ex} 
$\,E_7\,$  & $7$ & $\,\omega_1\,$  & $133,\;p>2$ & adjoint \\
	   &     &                 & $132,\;p=2$ &\\ \cline{3-5}
	   &     & $\,\omega_7\,$  & $56$  
\vspace{.5ex} & natural \\ \hline \vspace{.5ex}
$\,E_8\,$  & $8$ & $\,\omega_8\,$  & $248$  
\vspace{.5ex} &  adjoint \\ \hline \vspace{.5ex}
$\,F_4\,$  & $4$ & $\,\omega_1\,$ & $52,\;p>2$  & adjoint \\ 
           &     &                 & $26,\;p=2$ & \\ \cline{3-5}
	   &     & $\,\omega_4\,$ & $26,\;p\neq 3$  & natural\\ 
	   &     &                 & $25,\;p=3$ 
\vspace{.5ex} & \\ \hline \vspace{.5ex}
$\,G_2\,$   & $2$ & $\,\omega_1\,$  & $7,\;p>2$ & natural\\
	    &     &                 & $6,\;p=2$  & \\ \cline{3-5}
	    &     & $\,\omega_2\,$  & $14,\;p\neq 3$  & adjoint\\ 
	    &     &                 & $7,\;p=3$ &
\vspace{.5ex}\\ \hline \hline  
\end{tabular}\end{center}\caption[The Exceptional Modules for Groups or Lie
Algebras of Exceptional Type]{Exceptional Modules for Groups of
Exceptional Type{\label{table3}}}
\end{table*}} 

Comparing Tables~\ref{table1} and~\ref{table3}, we see that in the case 
of exceptional groups the list of exceptional modules is the same
as in characteristic zero. (Observe that Table~\ref{table1} omits the
highest weights corresponding to the adjoint representations.)

For Lie algebras of type $\,A,\,B,\,C,\,D\,$ (classical types), 
certain reduction lemmas are proved in Section~\ref{lact}. 
%{alresults}, \ref{blresults}, \ref{clresults},\ref{resuldl}. 
For classical groups of low characteristics,
we expect some new highest weights to appear in the list of
exceptional modules (that is, the list of exceptional modules for
classical groups of low characteristics will differ from the list
obtained by Andreev-Vinberg-\'Elashvili).
The results obtained for the groups of classical type are as
follows.\vspace{2.2ex}\\
{\bf Theorem~\ref{anlist}}~~{\it 
If $\,V\,$ is an infinitesimally irreducible $\,A_{\ell}(K)$-module 
with highest weight listed
in Table~\ref{tablealall}, then $\,V\,$ is an exceptional 
$\,\mathfrak g$-module.
%If $\,V\,$ is an $\,A_{\ell}$-module having highest weight according
%to Table~\ref{leftan}, then $\,V\,$ is not classified. 
If $\,V\,$ has highest weight different from the ones listed in Tables
~\ref{tablealall} or~\ref{leftan}, then
$\,V\,$ is not an exceptional $\,\mathfrak g$-module.
\begin{table*}[htb]\begin{center}
$\begin{array}{c|c|c|c|c|c}\hline\hline
{\rm N.} & {\rm Rank} & {\rm Prime} & {\rm Weights} & {\rm Module} & 
\dim\,V \\  
\hline \hline
 1 &\ell=1     & {\rm any}
 & \omega_1       & {\rm natural}  &  2 \\  \hline
2 & \ell=1     & p\geq 3  & 2 \omega_1             & {\rm adjoint} & 
\ell^2\,+\,2\ell\,-\,\varepsilon \\ \cline{2-4}
  & \ell\geq 2 & {\rm any}  &  \omega_1\,+\,\omega_{\ell} &       &  
\varepsilon\,\in\,\{ 0,\,1\} \\ \hline
3 & \ell\geq 2 & p\geq 3  &  2\omega_1,\;2\omega_{\ell} & S^2,\,{S^2}^* & 
\displaystyle\binom{\ell+2}{2} \\  \hline
4 &\ell\geq 2 &  {\rm any}   &  \omega_1,\;\omega_{\ell}   &  
\mathfrak{sl},\,\mathfrak{sl}^*     &  \ell\,+\, 1\\  \hline
5 &\ell\geq 3 &  {\rm any}   &  \omega_2,\;\omega_{\ell-1} &  
\bigwedge^2,\,{\bigwedge^2}^*  & \displaystyle\binom{\ell +1}{2} \\  \hline
6 & 5\leq\ell\leq 7 &  {\rm any} &  \omega_3,\;\omega_{\ell-2} &  
\bigwedge^3,\,{\bigwedge^3}^*  &  \displaystyle \binom{\ell +1}{3}\\ 
 \hline\hline
%4\leq\ell\leq 6 &  {\rm any} &  \omega_4,\;\omega_{\ell-3} & 
%\bigwedge^4,\,{\bigwedge^4}^* &  \displaystyle\binom{\ell +1}{4}\\ 
\end{array}$
\caption[Exceptional $\,A_{\ell}$-Modules]
{Exceptional $\,A_{\ell}$-Modules \label{tablealall}}\end{center}
\end{table*}}
%%%\vspace*{2.2ex}\\ 
\newpage\noindent
{\bf Theorem~\ref{listbn}}~~{\it Suppose $\,p\,>\,2.\,$
If $\,V\,$ is an infinitesimally irreducible
$\,B_{\ell}(K)$-module with highest weight listed in
Table~\ref{tableblall}, then $\,V\,$ is an exceptional $\,\mathfrak g$-module.
If $\,V\,$ has highest weight different from the ones listed in Tables
~\ref{tableblall} or~\ref{leftbn}, then
$\,V\,$ is not an exceptional $\,\mathfrak g$-module.
\begin{table*}[htb]
\[
\begin{array}{c|c|c|c|c|c}\hline\hline
{\rm N.} & {\rm Rank}   &  {\rm Prime} & {\rm Weights} & {\rm Module}     &
\dim\,V \\  \hline\hline
 1 & \ell\geq 2
 &  {\rm any} & \omega_1 & {\rm natural} & 2\ell\,+\,1 \\ \hline
 2 & \ell=2  &  p\geq 3  & 2\,\omega_2 &  {\rm adjoint} & 
\ell\,(2\ell\,+\,1) \\ \cline{2-4}
   & \ell\geq 3  &  {\rm any} & \omega_2 &            &  \\ \hline
 3 & 2\leq\ell\leq 6  &  {\rm any} & \omega_{\ell} & {\rm spin}_{2\ell+1}
& 2^{\ell}   \\  
\hline\hline
\end{array}
\]
\caption[Exceptional $\,B_{\ell}$-Modules]
{Exceptional $\,B_{\ell}$-Modules\label{tableblall}}
\end{table*}
\vspace*{2.5ex}\\ 
%%\newpage\noindent
{\bf Theorem~\ref{listcn}}~~{\it Suppose $\,p\,>\,2.\,$
If $\,V\,$ is an infinitesimally irreducible
$\,C_{\ell}(K)$-module with highest weight listed in
Table~\ref{tableclall}, then $\,V\,$ is an exceptional $\,\mathfrak g$-module.
%If $\,V\,$ is a $\,C_{\ell}$-module having highest weight according
%to Table~\ref{leftcn}, then $\,V\,$ is not classified.
If $\,V\,$ has highest weight different from the ones listed in Tables
\ref{tableclall} or~\ref{leftcn}, then
$\,V\,$ is not an exceptional $\,\mathfrak g$-module.
\begin{table*}[htb]
\[
\begin{array}{c|c|c|c|c|c}\hline\hline
{\rm N.} & {\rm Rank} & {\rm Prime} & {\rm Weight} & {\rm Module}  & 
\dim\,V \\ \hline\hline
1 & \ell\geq 2  & p\geq 3  & 2\omega_1 &  {\rm adjoint}    &
\ell\,(2\ell\,+\,1)      \\ \hline
2 & \ell\geq 2 &  {\rm any}  & \omega_1 & {\rm natural}  & 2\,\ell  \\ \hline
3 & \ell\geq 2 &  
 {\rm any}  & \omega_2 &        &    \ell\,(2\ell-1)\,-\,\nu      \\
  & &      &          &       &  \nu\,\in\,\{1,\,2\} \\ \hline
4 & \ell=3        &  {\rm any}  & \omega_3 &      &   14      \\ \hline\hline
\end{array}
\]
\caption[Exceptional $\,C_{\ell}$-Modules]
{Exceptional $\,C_{\ell}$-Modules\label{tableclall}}
\end{table*}

\begin{table*}[htb]
\[
\begin{array}{c|c|c|c|c|c}\hline\hline
{\rm N.} & {\rm Rank} & {\rm Prime} &  {\rm Weights}  &  {\rm Module} & \dim\,V
\\  \hline\hline 
1 & \ell=2 & p\neq 3 & \omega_1\,+\,\omega_2 &      &    \\ \hline
2 & \ell\geq 4 &  {\rm any}   & \omega_3    &           &        \\ \hline  
3 & 7\leq \ell\leq 11 &  {\rm any} &  \omega_{\ell} &    &  \\ \hline
4 & \ell=3,\,4 &  {\rm any}  &  2\omega_{\ell}             &    &  \\ \hline
5 & \ell=3     & {\rm any} &  \omega_1\,+\,\omega_3   &    & \\ \hline\hline 
\end{array}
\]
\caption[Unclassified $\,B_{\ell}$-Modules]
{Unclassified $\,B_{\ell}$-Modules\label{leftbn}}
\end{table*}}

\begin{table*}[htb]
\[
\begin{array}{c|c|c|c|c|c}\hline\hline
{\rm N.} & {\rm Rank} &  {\rm Prime}  &  {\rm Weight}  &  {\rm Module}  &  
\dim\,V \\ \hline\hline
 1 & \ell= 2 &  p\neq 3  & \omega_1\,+\,\omega_2    &       & \\ \hline
 2 & \ell\geq 4   &  {\rm any} &   \omega_3            &       &   \\ \hline
 3 &\ell= 5 &  {\rm any}  &  \omega_{4}                &       &   \\ \hline 
 4 &\ell=4,\,5 &  {\rm any} &  \omega_{\ell}       &       &  \\ \hline \hline
\end{array}
\]
\caption[Unclassified $\,C_{\ell}$-Modules] 
{Unclassified $\,C_{\ell}$-Modules\label{leftcn}}
\end{table*}}
\vspace*{2ex}\noindent
{\bf Theorem~\ref{dnlist}}~~{\it 
If $\,V\,$ is an infinitesimally irreducible
$\,D_{\ell}(K)$-module with highest weight listed
in Table~\ref{tabledlall}, then $\,V\,$ is an exceptional 
$\,\mathfrak g$-module.
%If $\,V\,$ is a $\,D_{\ell}$-module having highest weight according
%to Table~\ref{leftdn}, then $\,V\,$ is not classified.
If $\,V\,$ has highest weight different from the ones listed in Tables
~\ref{tabledlall} or~\ref{leftdn}, then
$\,V\,$ is not an exceptional $\,\mathfrak g$-module.
\begin{table*}[tbp]
\[
\begin{array}{c|c|c|c|c|c}\hline\hline
{\rm N.} & {\rm Rank}    &  {\rm Prime}     &  {\rm Weights} & {\rm Module}   &
 \dim\,V    \\ \hline\hline
1 & \ell\geq 4  & {\rm any}   &  \omega_1  & {\rm natural}  & 2\,\ell \\ \hline
2 & \ell\geq 4  &  {\rm any}  &  \omega_2    & {\rm adjoint}  &
 2\,\ell^2\,-\,\ell\,-\,\nu  \\ 
 &           &   &    &   & \nu\,\in\,\{1,\,2\} \\ \hline
3 & \ell= 4  &  {\rm any}  &  \omega_3,\;\omega_4     &   {\rm twisted-natural}
&   8     \\ \hline
4 & 5\leq\ell\leq 7 &  {\rm any}  &  \omega_{\ell-1},\; \omega_{\ell}
 & {\rm semi-spinor}  & 2^{\ell -1}      \\ \hline\hline
%5 & 5\leq\ell\leq 7 &  {\rm any}  &  \omega_{\ell}      &  {\rm
%semi-spinor}    &   2^{\ell -1}     \\ \hline\hline
\end{array}
\]
\caption[Exceptional $\,D_{\ell}$-Modules]{Exceptional
$\,D_{\ell}$-Modules \label{tabledlall}}
\end{table*}

\begin{table*}[tbp]
\[
\begin{array}{c|c|c|c|c|c}\hline\hline
{\rm N.} &{\rm Rank}  &  {\rm Prime} &  {\rm Weights} &   {\rm Module} &
 \dim\,V    \\ \hline\hline
1 & \ell=5 & p=2, 5  &  \omega_1\,+\,\omega_{4}   &          &   \\ 
  &       &         &   \omega_1\,+\,\omega_5    &          &   \\ \hline
2 & \ell=5 & p=2     & \omega_{4}\,+\,\omega_{5}  &          &   \\ \hline
3 & \ell=4 &  {\rm any}   &  \omega_1\,+\,\omega_{3}   &          &    \\ 
  &    &         &   \omega_1\,+\,\omega_4    &          &    \\ 
  &    &         & \omega_{3}\,+\,\omega_{4}  &          &    \\ \hline
% & \ell\geq 4 & p\geq 3  & 2 \omega_{1}          &          &    \\ \hline
% & \ell=4 & p\geq 3      & 2 \omega_2            &          &    \\ \hline
4 & \ell=5 &  p\geq 3 & 2\omega_{4},\;2\omega_{5}  &   
    &    \\ \hline
5 & \ell\geq 5 & {\rm any}  & \omega_{3}  &  & \\ \hline
6 &  8\leq\ell\leq 10 & {\rm any} & \omega_{\ell-1},\; \omega_{\ell}
&    &    \\ \hline\hline
\end{array}
\]
\caption[Unclassified $\,D_{\ell}$-Modules]{Unclassified
$\,D_{\ell}$-Modules \label{leftdn}}
\end{table*}}

Thus, with a few exceptions, the problem of classifying
exceptional modules for the classical types is reduced to the groups
of small rank and a finite list of highest weights in each rank.  
Computer calculations can be used to sort out these remaining cases  
by producing central stabilizers of some
generic vectors (see e.g., \cite{cohwa}).

Due to the length and technicality of the proofs of the main theorems, we have
opted to give in Chapter 5 the complete proof of the classification of
exceptional modules just for the exceptional types. The  
proof of Theorem~\ref{frlaet} is given in a series of
lemmas, and for each type separately in Section~\ref{proofexcp}.
The proofs of Theorems~\ref{anlist}, \ref{listbn}, \ref{listcn}, \ref{dnlist} 
are given in the Appendix, where we also deal with
some combinatorial inequalities.

\begin{table*}[tbp]\begin{center}
$\begin{array}{c|c|c|c|c|c}\hline\hline
{\rm N.} & {\rm Rank} & {\rm Prime} & {\rm Weights} & {\rm Module} &
\dim\,V \\ \hline \hline 
1 & \ell=2       & p\geq 3 & 2\omega_1\,+\,\omega_2      &    &    \\
  &              &         & \omega_1\,+\,2\omega_2      &    &   \\ \hline
2 & \ell=3       & p=5     & 2\omega_1\,+\,\omega_3      &    &   \\
  &              &         & \omega_1\,+\,2\omega_3      &    &   \\ \hline
% & \ell=\,5 & p\neq 2 & \omega_1\,+\,\omega_3       &    &   \\ 
%	     &         &  \omega_{3}\,+\,\omega_{5}  &    &   \\ \hline
3 & \ell=5   & p= 2    & \omega_1\,+\,\omega_3       &    &   \\ 
  &          &         & \omega_{4}\,+\,\omega_{5}  &     &  \\ \hline
4 & \ell=\,4     & p=2,\,3   & \omega_2\,+\,\omega_3       &     &   \\ \hline
5 & \ell\geq 3 & p\neq 3    & \omega_1\,+\,\omega_2       &     &   \\ 
  &          &         & \omega_{\ell-1}\,+\,\omega_{\ell}  &     &  \\ \hline
6 & 4\leq\ell\leq 6 & p\geq 3    & \omega_1\,+\,\omega_{\ell-1}  &    &
	     \\ \cline{2-3}
  & \ell=7    & p=7   &               &  &  \\ \cline{2-3}
  & 4\leq\ell\leq 8 &   p=2  & \omega_2\,+\,\omega_{\ell}    &    &  \\ \hline
%any          & p\geq 5 & 3\omega_1,\;3\omega_{\ell}    &     &  \\ \hline
7 & \ell=3,\,4  & p\geq 3 & 2\omega_2,\;2\omega_{\ell-1}   &  &  \\ \hline
%\ell=4       &         & 2\omega_3,\;2\omega_{\ell-2}   &     &  \\ \hline   
8 & \ell = 9 & {\rm any}    &  \omega_5                &    &   \\ \hline
9 & 7\leq\ell\leq 11 &  {\rm any} &  \omega_4,\;\omega_{\ell-3} &  &
\\  \hline
10 & \ell\geq 8   &  {\rm any}   &  \omega_3,\;\omega_{\ell-2}    &    &  \\
 \hline \hline
\end{array}$\caption[Unclassified $\,A_{\ell}$-Modules]
{Unclassified $\,A_{\ell}$-Modules \label{leftan}}\end{center}
\end{table*}

\newpage
\vspace{50ex}
\part*{Chapter 2}
\part*{Preliminaries}
\addcontentsline{toc}{section}{\protect\numberline{2}Preliminaries}
\label{background}
\setcounter{section}{2}
\setcounter{subsection}{0}
As this work is about classifying certain representations of algebraic
groups in prime characteristic and the corresponding representations
of the associated restricted Lie algebras, 
I will give in this Chapter a brief description of the structure of
restricted Lie algebras as well as an introduction to the results I will
need from the representation theory of algebraic groups.
The main objective is to establish notation, while more detailed 
information on these structures appears in the references.
I will assume familiarity with the basic concepts of algebraic groups
and Lie algebras, while referring to the usual literature for proofs of the 
results. 
%~\cite{bor2},~\cite{bor1},~\cite{hum1},~\cite{hum2},~\cite{car1}

\subsection{Restricted Lie Algebras}\label{rla}

In this section I define restricted Lie algebra 
and discuss its main properties, as well as some results that are 
used in this work. The main references to this section are
\cite{jac}, \cite{stfa}, \cite{wint}.

Throughout this section, $\,F\,$ is a field of characteristic 
$\,p > 0\,$ ($\,p\,$ being a prime). Lie algebras and Lie modules are
finite dimensional over $\,F$.

\begin{definition}\label{restliealg}
A {\bf Lie $\,p$-algebra} ({\bf restricted Lie algebra of
characteristic $\,p\,$}) $\,L\,$ is a Lie algebra over a field $\,F\,$
of characteristic $\,p > 0\,$ in which there is defined a mapping
$\,x\,\longmapsto\,x^{[p]}$, called {\bf $\,p$-mapping} such that

(i) $\;\,(t\,x)^{[p]}\,=\,t^p\, x^{[p]},\;(\,t\,\in\,F,\,x\,\in\,L\,)$;

(ii) $\;\,(x\,+\,y)^{[p]}\,=\,x^{[p]}\,+\,y^{[p]}\,+\,\displaystyle
\sum_{i=1}^{p-1}\,i^{-1}s_i(x,\,y)$;

(iii) $\;\,{\rm ad}(x^{[p]})\,=\,({\rm ad}\,x)^p,\;(\,x\,\in\,L\,)$;\\
where $\,s_i(x,\,y)\,$ is the coefficient of $\,t^i\,$ in
$\,{\rm ad}(t x\,+\,y)^{p-1}(x),\;(\,x,\,y\in\,L\,)\,$.
\end{definition}
Formula {\it (ii)} is called {\bf Jacobson's Identity}.  
As usual, we write $\,{\rm ad}(x)(y)\,=\,[\,x,\,y\,]$. 
Note that when $\,[\,x,\,y\,]\,=\,0\,$ then {\it (ii)} becomes

{\it (ii')} $\;\,(x\,+\,y)^{[p]}\,=\,x^{[p]}\,+\,y^{[p]}$.\\
\notation $\,(L,\,[p])\,$ denotes a restricted Lie algebra of
characteristic $\,p$. 

\begin{example} \label{examrest} %{\bf{Examples}:}
{\rm (1) A typical example of a restricted Lie algebra is a Lie subalgebra 
$\,L\,$ of an associative algebra $\,\cal A\,$ stable under the $\,p$th
power map in $\,\cal A\,$ (in this case the $\,p$th power map defines the 
$\,[p]$-structure in $\,L\,$).

For an associative algebra $\,A\,$ consider a new operation defined 
by $\,[x,\,y]\,=\,x\,y\,-\,y\,x\,$, for all $\,x,\,y\in A\,$. This
gives $\,A\,$ a Lie algebra structure. Denote this Lie algebra
by $\,A^{(-)}\,$. % the pair $\,(A,\,[\;,\;]).\,$
If the base field of $\,A\,$ has characteristic $\,p,\,$ then 
$\,A^{(-)}\,$ carries a canonical restricted Lie algebra structure,
given by $\,x\longmapsto x^p\,$ (the $\,p$th power map).
In particular, $\,\mathfrak{gl}(V)\,:=\,({\rm End}(V))^{(-)},\,$ 
%_{\rm Lie}\,$ 
where $\,V\,$ a finite dimensional vector space over a field $\,F\,$
of characteristic $\,p,\,$ is a restricted Lie algebra.}

{\rm (2) Let $\,{\mathfrak U}\,$ be an arbitrary (non-associative) algebra.
A {\it derivation} of $\,{\mathfrak U}\,$ is a linear mapping 
$\,D\,:\,{\mathfrak U}\longrightarrow {\mathfrak U}\,$ satisfying the
rule for the derivative of a product, namely, 
$\,D(a\,b)\,=\,D(a)\,b\,+\,a\,D(b)\,$, for all $\,a,\,b\,\in\mathfrak U\,$. Let
$\,{\mathfrak D}({\mathfrak U})\,$ be the set of all derivations of 
$\,{\mathfrak U}$. Then 
$\,{\mathfrak D}({\mathfrak U})\,$ is a subalgebra of the Lie
algebra $\,(\End({\mathfrak U}))^{(-)}$. 
One has the Leibniz formula
\begin{equation} \label{leibniz}
D^k(ab)\,=\,\sum_{i=1}^k\,\binom{k}{i}\,D^i(a)\,D^{k-i}(b)
\end{equation}
for any $\,D\,\in\,{\mathfrak D}({\mathfrak U})\,$  
(this can be established by induction on
$\,k$). Now, assuming that the base field 
%of $\,{\mathfrak U}$, hence of $\,{\mathfrak D}({\mathfrak U}),\,$ 
is of characteristic $\,p\,$ and taking
$\,k\,=\,p\,$ in~(\ref{leibniz}), we have that the binomial
coefficients $\,\displaystyle\binom{p}{i}\,$ vanish for $\,1\leq i\leq p-1$.  
Hence~(\ref{leibniz}) reduces to 
\begin{equation}
D^p(ab)\,=\,D^p(a)\,b\,+\,a\,D^p(b),
\end{equation}
which implies $\,D^p\,\in\,{\mathfrak D}({\mathfrak U})\,$. 
Thus $\,{\mathfrak D}({\mathfrak U})\,$
is closed under the mapping $\,D\,\longmapsto\,D^p\,$.
 % as well as the Lie algebra compositions. 
Hence $\,{\mathfrak D}({\mathfrak U})\,$ is a
restricted Lie algebra.}
\end{example}
\medskip

Ideals, subalgebras and modules of restricted Lie algebras are defined in the
obvious way.
\begin{definition}
Let $\,(L,\,[p])\,$ be a restricted Lie algebra over $\,F$.
A {\bf restricted Lie subalgebra or $\,p$-subalgebra} (respectively {\bf
$\,p$-ideal)} of $\,L\,$ is a subalgebra (respectively ideal) which is
stable under the $\,p$-mapping $\,[p].\,$ %$\,x\,\longmapsto\,x^{[p]}$. 
Such a subalgebra 
(respectively, ideal) is regarded as a Lie $\,p$-algebra by taking its
$\,p$-mapping to be the restriction of that of $\,L$.
\end{definition}
%\notation $\,I\,\unlhd_p\,L\,$: a $\,p$-ideal of $\,(L,\,[p])$.

\begin{definition}
Let $\,(L_1,\,[p]_1)\,$ and $\,(L_2,\,[p]_2)\,$ be restricted Lie
algebras over $\,F$. A homomorphism $\,f\,:\,L_1\,\rightarrow\,L_2\,$ is
called  {\bf restricted} (or {\bf $\,p$-homomorphism}) if 
%$\,f\,$ is a homomorphism of Lie algebras and
$\,f(x^{[p]_1})\,=\,f(x)^{[p]_2},\;\forall\,x\,\in\,L_1$. 
%A {\bf $\,p$-isomorphism} is a bijective {\bf $\,p$-homomorphism}.
\end{definition}

\begin{definition}
A {\bf $\,p$-representation} (or {\bf restricted representation}) 
of a restricted Lie algebra $\,L\,$ in a vector space $\,V\,$ is a
restricted homomorphism from $\,L\,$ into $\,{\mathfrak{gl}\,}(V).\,$

A {\bf Lie $\,p$-module} for a Lie $\,p$-algebra $\,L\,$ is a
$\,L$-module $\,V\,$ such that 
\[
(x^{[p]})\cdot v \,=\,x\cdot (x\cdot(\cdots (x\cdot v)\cdots
))\;({p}\;\mbox{times}),\;\forall\,x\,\in\,L,\;v\in V.
\]
\end{definition}

\begin{example}{\rm 
If $\,L\,$ is a Lie $\,p$-algebra, then by \ref{restliealg}(iii)
$\,{\rm ad}\,:\,L\longrightarrow \mathfrak{gl}(L)\,$ is a  
$\,p$-rep\-re\-sent\-a\-tion of $\,L.\,$
Consequently, the centre $\,Z\,$ of $\,L\,$ is a $\,p$-ideal,
for $\,Z\,=\,\ker\,{\rm ad}\,$ (see \cite[p.\ 69]{stfa})}.
%since the centre is the kernel of the $\,p$-homomorphism $\,{\rm ad}.\,$ 
\end{example}

\subsubsection{Nilpotent, Semisimple and Toral Elements}\label{nste}

Let $\,(L,\,[p])\,$ be a restricted Lie algebra over $\,F$. 
Given $\,i\in\mathbb Z^+\,$ denote by
%$\,x^{[p]}\,$ the image of $\,x\in L\,$ under the $\,p$-mapping
%$\,[p]\,$ of $\,L,\,$ by 
$\,x^{[p]^i}\,$ the image of $\,x\,$ under
the $\,i$th iterate of $\,x\longmapsto x^{[p]}\,$ (with
$\,x^{[p]^0}\,=\,x\,$).

\begin{definition}
%Let $\,(L,\,[p])\,$ be a restricted Lie algebra over $\,F$. 
A $\,p$-ideal $\,I\,$ of $\,L\,$ is called {\bf $\,p$-nilpotent}
if there is $\,n\,\in\,\mathbb N\,$ such that $\,I^{[p]^n}\,=\,0$. An 
element $\,x\,\in\,L\,$ is called {\bf $\,p$-nilpotent} if there is 
$\,n\,\in\,\mathbb N\,$ such that $\,x^{[p]^n}\,=\,0$. The $\,p$-ideal
$\,I\,$ is called {\bf $\,p$-nil} if every element $\,x\,\in\,I\,$ is 
$\,p$-nilpotent. The set $\,{\cal N}(L)\,:=\,\{\,x\in L
\,/\,x^{[p]^e}\,=\,0\;\,\mbox{for}\;\,e>>0\,\}\,$ is called the {\bf
nilpotent cone} of $\,L.\,$
\end{definition}

Let $\,V\,$ be a finite dimensional vector space over a field $\,F$.  
An endomorphism
$\,\sigma\,:\,V\,\rightarrow\,V\,$ is called {\bf semisimple} if its
minimal polynomial has distinct roots in some field extension of
$\,F\,$ (so that $\,\sigma\,$ is diagonalizable after some base field 
extension). From general algebra we have the following characterization:
$\,\sigma\,$ is semisimple if the ideal, in $\,F[X]$, generated by
the minimum polynomial of $\,\sigma\,$ and its derivative, contains $\,1$.

Let $\,<x>\,$ denote the smallest restricted
subalgebra of $\,L\,$ containing $\,x$, i.e., the linear span 
%space of linear combinations 
of $\,\{ \ x^{[p]^i}\,/\,i\in\mathbb Z^+ \ \}\,$.

\begin{definition}
Let $\,(L,\,[p])\,$ be a restricted Lie algebra over $\,F$. An element 
$\,x\,\in\,L\,$ is called {\bf $\,p$-semisimple} (or {\bf semisimple}
for short) if
$\,x\,\in\,<x^{[p]}>\,$; {\bf toral} if $\,x^{[p]}\,=\,x$.
\end{definition}

\begin{proposition}{\rm\cite[p. 80]{stfa},~\cite[V.7]{seli}} \label{properties}
Let $\,(L,\,[p])\,$ be a restricted Lie algebra. Then the following
statements hold:

(1) Each toral element is $\,p$-semisimple.

(2) If $\,x\in L\,$ is $\,p$-semisimple, then the endomorphism 
$\,\psi (x)\,$ is semisimple for
every finite dimensional restricted representation
$\,\psi\,:\,L\,\rightarrow\,\mathfrak{gl}(V)$.

(3) If $\,x,\,y\,\in\,L$ are $\,p$-semisimple and $\,[x,y]\,=\,0$, then 
$\,x\,+\,y\,$ is $\,p$-semisimple.

(4) If $\,x\in L\,$ is $\,p$-semisimple, then $\,<x>\,$ consists of
$\,p$-semisimple elements.

(5) If $\,K\,$ is perfect and $\,L\,$ contains no nontrivial
$\,p$-nilpotent elements, then each
element in $\,L\,$ is $\,p$-semisimple.

(6) An endomorphism $\,\sigma\,\in\,End_F(V)\,$ is semisimple if and
only if it is semisimple as an element of the restricted Lie algebra 
$\,\mathfrak{gl}(V)\,:=\,{\rm End}(V)^{(-)}\,$.
\end{proposition}

The significance of semisimple elements rests on the following result.

\begin{theorem}{\rm\cite[p. 80]{stfa}}\label{sigsemi}
Let $\,(L,\,[p])\,$ be a restricted Lie algebra over $\,F$. For each 
$\,x\,\in\,L\,$ there exists $\,k\,\in\,\mathbb N\,$ such that  
$\,x^{[p]^k}\,$ is $\,p$-semisimple.
\end{theorem}

If $\,F\,$ is perfect, one can prove a much stronger result 
closely related to the well-known Jordan-Chevalley decomposition of an
endomorphism. 

\begin{theorem}{\rm\cite[p. 81]{stfa}}\label{jcd}
Let $\,F\,$ be a perfect field and let $\,(L,\,[p])\,$ be a finite
dimensional restricted Lie algebra over $\,F$. Then for any
$\,x\,\in\,L\,$ there are uniquely determined elements $\,x_n,\,x_s\,\in\,L\,$
such that 

(1) $\,x_n\,$ is $\,p$-nilpotent, $\,x_s\,$ is $\,p$-semisimple.

(2) $\,x\,=\,x_s\,+\,x_n,\;\,[x_s,\,x_n]\,=\,0$.   
\end{theorem}
The decomposition obtained can be refined in the case of an
algebraically closed field $\,K$. 

\begin{definition}
A $\,p$-mapping $\,[p]\,$ on $\,L\,$ is called {\bf nonsingular} 
if~$\,x^{[p]}\,\neq\,0$ for all $\,x\,\in\,L\setminus\,\{ 0\}$.
\end{definition}

The following useful result (due to N. Jacobson) can be found in 
\cite[p. 81]{stfa}.
\begin{theorem}{\rm\cite[p. 82]{stfa}} \label{3.6}
Let $\,(L,\,[p])\,$ be a finite dimensional restricted Lie algebra
over an algebraically closed field $\,K$. Then the following statements
hold:

(1) If $\,L\,$ is abelian and the $\,p$-mapping is nonsingular, then $\,L\,$
possesses a basis consisting of toral elements. 

(2) For any $\,x\,\in\,L\,$ there exist toral elements
$\,x_1,\,\ldots,\,x_r\,\in\,L,\,$ scalars
$\,c_1,\,\ldots,\,c_r\,\in\,K,\,$ and a $\,p$-nilpotent element
$\,y\in L\,$ such that
\[
x\,=\,y\,+\,\sum_{i=1}^r\,c_i\,x_i\,,\;\;\;[y,\,x_i]\,=\,[x_i,\,x_j]\,
=\,0\;\;\;\forall\, i,\,j.
\]
\end{theorem}

\begin{definition} \label{resttorus}
Let $\,(L,\,[p])\,$ be a restricted Lie algebra over $\,F$. A subalgebra
$\,T\subset L\,$ is called a {\bf torus} or a {\bf toral subalgebra} if
 $\,T\,$ is an abelian $\,p$-subalgebra, consisting of
$\,p$-semisimple elements.
% $\;\forall\,x\,\in\,T.\,$ %(i.e., $\,T\,$ has no
%nilpotent element other than zero).
\end{definition}

It follows from the definition and Proposition~\ref{properties}(2),
that if $\,T\,$ is a torus of $\,L\,$ and $\,\varphi\,$ is a
$\,p$-representation of $\,L,\,$ then $\,\varphi(T)\,$ is
diagonalizable (see \cite[4.5.5]{wint}).

\begin{lemma} \label{ppp}
Let $\,(L,\,[p])\,$ be a finite dimensional restricted Lie algebra
over an algebraically closed field $\,K$. If $\,(L,\,[p])\,$
contains no nonzero $\,p$-nilpotent elements, then $\,(L,\,[p])\,$ is toral. 
\end{lemma}\noindent
\begin{proof} By Theorem~\ref{jcd}(2), each element of $\,(L,\,[p])\,$
is semisimple, so it remains to prove
that $\,(L,\,[p])\,$ is abelian, i.e.,
$\,{\rm ad}\;x\,=\,0$, for each $\,x\,\in\,L$. As $\,{\rm ad}\;x\,$ is
diagonalizable ($\,{\rm ad}\;x\,$ being semisimple and $\,K\,$ algebraically
closed), we have to show that $\,{\rm ad}\;x\,$ has no nonzero eigenvalues.
Suppose, on the contrary, that $\,[x,\,y]\,=\,a\,y\;(a\neq 0)\,$ for some 
nonzero $\,y\in L$. Then 
\[
({\rm ad}\;y)^2\,(x)\,=\,[y,\,[y,\,x]]\,=\,-a\,[y,\,y]\,=\,0,\,\qquad\mbox{(*)}
\]
i.e., $\,[y,\,x]\,$ is an eigenvector of $\,{\rm ad}\;y\,$ of
eigenvalue $\,0$. 
Now write $\,x\,$ as a linear combination of eigenvectors of
$\,{\rm ad}\;y\,$ ($\,y\,$ is also semisimple).
Clearly, $\,[y,\,x]\,$ is a combination of $\,{\rm ad}\;y$-eigenvectors which
belong to nonzero eigenvalues, if any. This, however, contradicts (*).
%; after applying $\,{\rm
%ad}\;y\,$ to $\,x$, all that is left is a combination of eigenvectors which
%belong to nonzero eigenvalues, if any. This contradicts (*).
\end{proof}

\subsubsection{Restricted Universal Enveloping Algebras}

For restricted Lie algebras there is an analogue of the universal
enveloping algebra.
%, taking into account the additional structure of restrictedness. 
This structure plays an important role in the
representation theory of algebraic groups over a field of positive
characteristic. % as we shall see in Section~\ref{reptheo}.

\begin{definition} \label{ruel}
Let $\,(L,\,[p])\,$ be a restricted Lie algebra. A pair $\,({\mathfrak
u}(L),\,i)\,$ consisting of an associative $\,K$-algebra with unity
and a restricted homomorphism $\,i\,:\,L\,\longrightarrow\, 
{\mathfrak u}(L)^{(-)}\,$
is called a {\bf restricted universal enveloping algebra} if given any
associative $\,K$-algebra $\,A\,$ with unity and any restricted
homomorphism $\,f\,:\,L\,\longrightarrow\,A^{(-)}$, there is a unique
homomorphism $\,F\,:\,{\mathfrak u}(L)\,\longrightarrow\,A\,$ of
associative $\,K$-algebras such that $\,F\circ i\,=\,f$.
\end{definition}

The restricted universal enveloping algebra $\,{\mathfrak u}(L)\,$ is
the quotient of the ordinary universal enveloping
algebra $\,{\mathfrak U}(L)\,$ by the two-sided ideal generated by all
$\,x^p - x^{[p]}$, where $\,x\in L\,$ (see \cite[V,Theorem 12]{jac}).

The universal property of $\,{\mathfrak u}(L)\,$ shows the
uniqueness of the restricted universal enveloping algebra of $\,L\,$ up to
isomorphism (see \cite[I.8.1]{stfa}).

\begin{theorem}{\rm\cite[p. 91]{stfa}}  \label{ruelbasis}
Let $\,(L,\,[p])\,$ be a restricted Lie algebra. Then the following
statements hold:

(1) The restricted universal enveloping algebra of $\,L\,$ exists.

(2) If $\,({\mathfrak u}(L),\,i)\,$ is a restricted universal enveloping
algebra of $\,L\,$ and $\,(e_j)_{j\in J}\,$ is an ordered basis of
$\,L\,$ over $\,K$, then the elements
$\,i(e_{j_1})^{s_1}\,\cdots\,i(e_{j_n})^{s_n}\,$, where  
$\,j_1\,<\,\cdots\,<\,j_n,\;n\,\geq\,1$, $\,0\,\leq\,s_k\,\leq\,p-1$,
%$\,\;1\leq k\leq n\;$ 
form a basis of $\,{\mathfrak u}(L)\,$ over $\,K$. In particular, 
$\,i\,:\,L\,\longrightarrow\,{\mathfrak u}(L)\,$ is injective and 
$\,\dim_F\,{\mathfrak u}(L)\,=\,p^n\,$ if $\,\dim_F\,L\,=\,n$.
\end{theorem} 

We identify $\,L\,$ with its image $\,i(L).\,$ %as in the case of 
%$\,{\mathfrak U}(L)$, the ordinary universal enveloping algebra.
Note that the
finiteness of the dimension is preserved when passing from restricted
Lie algebras to their restricted enveloping algebras, in contrast with
the ordinary Lie algebras and their envelopes. \vspace{1.5ex}\\%However,
%an important property of $\,{\mathfrak U}(L)\,$ does not hold in 
%$\,{\mathfrak u}(L)$: in general, $\,{\mathfrak u}(L)\,$ is not free of zero 
%divisors.
\noindent 
\begin{remark}\label{extrest}
It follows from Definition~\ref{ruel} that a $\,p$-representation of
the restricted Lie algebra $\,L\,$ extends uniquely to a
representation of $\,\mathfrak u(L)\,$ and, conversely, 
any representation of $\,\mathfrak u(L)\,$ restricts to a  
$\,p$-representation of $\,L.\,$ %(\cite[Chapter 5]{stfa}, \cite[V.7]{jac}).
\end{remark}

%\newpage

\subsection{The Lie Algebra of an Algebraic Group}\label{tlaag}

Throughout this section let $\,K\,$ be an algebraically closed field
and $\,G\,$ be a connected algebraic group over $\,K$.
 
In this section I define the Lie algebra $\,{\cal L}(G)\,$ of the
group $\,G\,$ and show that if $\,{\rm char}\,K\,>\,0,\,$ then $\,{\cal
L}(G)\,$ is a restricted Lie algebra. I also recall some 
results related to the general structure of algebraic groups 
and their Lie algebras. Standard notions of
algebraic geometry used here can be found in the first chapter of 
\cite{bor2} or \cite{hum2}.  

Denote by $\,K[G]\,$
the coordinate ring of (the irreducible affine variety) $\,G$.
% $\,K[G]\,$ is 
%also regarded as a ring of functions from $\,G\,$ to $\,K\,$. 
Let $\,T_e(G)\,$ denote the tangent space of $\,G\,$ at the
identity element $\,e\in G$. $\,T_e(G)\,$ can be identified with the space
$\,{\rm Der}(K[G],\,K_e)\,$ of the point-derivations of $\,G\,$ at $\,e\,$
(for definition and notation see~\cite[1.3]{car1} or~\cite[I.3.3]{bor2}).

If $\,f\in K[G]\,$ and $\,g\in G,\,$ the map
$\,f^g\,:\,G\longrightarrow K\,$ defined by $\,f^g(t)\,=\,f(tg)\,$
lies in $\,K[G]\,$ (since right multiplication by $\,g\,$ is a morphism
of $\,G\,$ \cite[I.1.9]{bor2}). 
Let $\,\alpha_g\,:\,K[G]\longrightarrow K[G]\,$ be defined by 
$\,\alpha_g(f)\,=\,f^g$. Then $\,\alpha_g\,$ is a $\,K$-algebra
automorphism of $\,K[G].\,$ Furthermore,
$\,\alpha_{gh}\,=\,\alpha_g\alpha_h\,$ and so we have a homomorphism
from $\,G\,$ into the group of $\,K$-algebra automorphisms of $\,K[G]$.

Recall that the set of all derivations $\,{\rm Der}_K(K[G])\,$ of the
algebra $\,K[G]\,$ forms a Lie algebra under
the Lie multiplication $\,[D_1,D_2]\,=\,D_1D_2\,-\,D_2D_1\,$. 
(cf.~Example~\ref{examrest}(2)).
A derivation $\,D\in{\rm Der}_K(K[G])\,$ is said to be {\bf invariant}
if $\,D(f^g)\,=\,(Df)^g\,$ for all $\,f\in K[G]\,$ and $\,g\in G$. 
The set $\,{\rm Der}_K(K[G])^G\,$ of all invariant derivations forms a Lie
subalgebra of $\,{\rm Der}_K(K[G])\,$. If $\,K\,$ has characteristic
$\,p\,$ and $\,D\,$ is an invariant derivation, then so is $\,D^p\,$. 
This gives $\,{\rm Der}_K(K[G])^G\,$ a restricted Lie algebra structure
(see~\cite[1.3]{car1}). %the Lie algebra of invariant derivations

Given $\,D\in {\rm Der}_K(K[G])^G$, the map $\,f\longmapsto Df(e)\,$
is a point-derivation of $\,G\,$ at the identity element. This gives
rise to a map $\,\varphi : {\rm Der}_K(K[G])^G\longrightarrow 
{\rm Der}(K[G],\,K_e)\,$ from the Lie algebra of invariant derivations
of $\,K[G]\,$ to the space of point-derivations of $\,G\,$ at $\,e.\,$ 
By~\cite[I.3.4, I.3.5]{bor2},
$\,\varphi\,$ is an isomorphism of vector spaces. Since 
$\,{\rm Der}_K(K[G])^G\,$ has a Lie algebra structure, $\,\varphi\,$ gives a
restricted Lie algebra structure to $\,{\rm Der}(K[G],\,K_e).\,$ 
Thus, the tangent space $\,T_e(G)\,$ to $\,G\,$ at the
identity element has a canonical restricted Lie algebra structure.
This restricted Lie algebra is denoted by $\,\mathfrak g\,=\,{\cal L}(G)\,$ and
called the {\bf Lie algebra of the algebraic group $\,G$}.
The Lie algebra of an algebraic group $\,G,\,H,\,M,\ldots\,$ is often denoted
by the corresponding German character $\,\mathfrak{g,\,h,\,m,}\ldots\,$.
%We shall view $\,G\longmapsto{\cal L}(G)\,=\,T_e(G)\,$ as a functor
%from (affine) algebraic groups to restricted Lie algebras. 

We have $\,\dim\,{\cal L}(G)\,=\,\dim\,G\,$ since $\,G\,$ is a smooth
variety~\cite[1.3]{car1}. One can also define the Lie algebra of a
disconnected algebraic group, but then $\,{\cal L}(G)\,=
\,{\cal L}(G^{\circ})\,$, where $\,G^{\circ}\,$
is the connected component of $\,G\,$ containing the identity element
$\,e\in G\,$ \cite[I.3.6]{bor2}.

If $\,H\,$ is a closed subgroup of $\,G\,$ then 
%$\,{\cal L}(H)\subset {\cal L}(G)\,$ 
$\,\mathfrak h\,$ is a restricted subalgebra of $\,\mathfrak g\,$ 
\cite[I.3.8]{bor2}.
If $\,U\,$ is a unipotent subgroup of $\,G$, then $\,{\cal L}(U)\,$  
consists of $\,p$-nilpotent elements of the restricted Lie algebra 
$\,{\cal L}(G)\,$ \cite[I.4.8]{bor2}.
\medskip

Let $\,\phi : G \rightarrow G'\,$ be a homomorphism of algebraic groups.
There is an associated Lie algebra homomorphism
$\,{\rm d}\phi : {\cal L}(G) \rightarrow {\cal L}(G')$, given by
$\,{\rm d}\phi (D) = D \circ
\phi^*$. This homomorphism is called the {\bf differential of} $\,\phi\,$
\cite[Theorem~9.1]{hum2}. Here $\,\phi^* : K[G']\rightarrow K[G]\,$
denotes the {\it comorphism} associated with the morphism
$\,\phi : G\rightarrow G'\,$ \cite[1.5]{hum2}. 
If $\,\phi_1 : G_1\rightarrow G_2\,$ and 
$\,\phi_2 : G_2\rightarrow G_3\,$ are homomorphisms of algebraic
groups, then their differentials satisfy 
$\,{\rm d}(\phi_2\circ \phi_1)\,=\,{\rm d}(\phi_2)\circ{\rm d}
(\phi_1) \,$.

Let $\,V\,$ be a finite dimensional vector space over the field $\,K.\,$
A representation $\,\pi\,:\,G\longrightarrow \GL(V)\,$ which is also a
morphism of algebraic varieties is called a
{\bf rational representation} of $\,G$. 
The differential $\,{\rm d}\pi\,:\,\mathfrak g\longrightarrow 
\mathfrak{gl}(V)\,$ defines a $\,p$-representation of the restricted
Lie algebra $\,\mathfrak g\,=\,{\cal L}(G)\,$ \cite[I.3.10]{bor2}.

The {\bf adjoint representation} of an algebraic group $\,G\,$ is defined  
as follows.
For each $\,x\in G$, the inner automorphism $\,{\rm Int}\,x\,(y)\,= 
\,x\,y\,x^{-1}\;(x,\,y\,\in G\,)\,$ of $\,G\,$ preserves the identity
element. This induces a linear action of $\,G\,$ on the tangent space
$\,{\cal L}(G)\,=\,T_e(G),\,$ denoted by 
$\,\Ad\,:\,G\longrightarrow\GL({\cal L}(G))\,$ and called the
{\it adjoint representation} of $\,G\,$. 
%For each $\,x\in G$, $\,\Ad\,x\,=\,{\rm d}({\rm Int}\,x)\,$. 
The differential of $\,\Ad\,$ is $\,{\rm ad}\,:\,
{\cal L}(G)\longrightarrow \mathfrak{gl}({\cal L}(G)),\,$ where 
$\,({\rm ad}\,a)(b)\,=\,[a,\,b]\,$ for all $\,a,\,b\in {\cal L}(G)\,$
\cite[10.3, 10.4]{hum2}.

%\newpage

\subsection{General Notions Related to Algebraic Groups}\label{gennot}

Let $\,G\,$ be a connected reductive algebraic group over $\,K\,$ and 
let $\,T\,$ be a maximal torus of $\,G.\,$ 
Let $\,B\,$ be a Borel subgroup of $\,G\,$ containing $\,T.\,$ Then 
$\,B\,$ has a semidirect product decomposition $\,B\,=\,T\,U ,\,$ where 
$\,U\,=\,R_{\rm u}(B)\,$ is the unipotent radical of $\,B\,$
\cite[19.5]{hum2}. There is a unique Borel subgroup $\,B^-\subset G\,$ 
containing $\,T\,$ and such that $\,B\cap B^-\,=\,T.\,$ $\,B\,$ and $\,B^-\,$ 
are called {\bf opposite} Borel subgroups. We have $\,B^-\,=\,T\,U^-\,$, 
where $\,U^-\,=\,R_{\rm u}(B^-)\,$ \cite[Cor. 26.2]{hum2}. 
$\,U\,$ and $\,U^-\,$ are connected subgroups normalized by  
$\,T\,$ and satisfying $\,U\cap U^-\,=\,\{ e\}.\,$ They 
are maximal unipotent subgroups of $\,G.\,$ 

\subsubsection{The Root System}\label{therootsys}

Consider the minimal subgroups of positive dimension in $\,U\,$ and $\,U^-\,$ 
which are normalized by $\,T.\,$ These are connected unipotent groups,
%of dimension $\,1,\,$ so are 
isomorphic to the additive group $\,{\mathbf G}_a\,$ \cite[20.5]{hum2}. 
$\,T\,$ acts on each of them by conjugation, giving rise to
a homomorphism 
%$\,T\rightarrow\,{\rm Aut}\,{\mathbf G}_a\,$ 
from $\,T\,$ to the group of algebraic automorphisms of $\,{\mathbf G}_a.\,$ 
However, the only algebraic automorphisms of $\,{\mathbf G}_a\,$ are the 
maps $\,\lambda\longmapsto\mu\lambda\,$ for some $\,\mu\in K^*.\,$ 
Thus $\,{\rm Aut}\,{\mathbf G}_a\,$ is isomorphic to $\,{\mathbf G}_m,\,$
the multiplicative group of $\,K$.
Hence, each of our $\,1$-dimensional unipotent groups determines 
an element of $\,{\rm Hom}(T,{\mathbf
G}_m)\,=\,X\,$. (Here $\,X\,=\,X(T)\,$ is the so-called {\it character
group} of $\,T\,$.)
The elements of $\,X\,$ arising in this way are called the {\bf roots} of
$\,G$. 
%They are all nonzero elements of $\,X\,$. 
Distinct $\,1$-dimensional 
unipotent subgroups give rise to distinct roots. The roots form a finite
subset $\,R\,$ of $\,X\,$ (which is independent of the choice of the  
Borel subgroup $\,B\,$ containing $\,T\,$)~\cite[1.9]{car1}. 

For each root $\,\alpha\in R,\,$ the $\,1$-dimensional unipotent subgroup 
giving rise to it is denoted by $\,U_{\alpha}\,$. 
The $\,U_{\alpha}\,$ are called the {\bf root subgroups} of $\,G.\,$ 
The roots arising from $\,U^-\,$ are the negatives 
of the roots arising from $\,U.\,$ 
%We also have $\,G\,=\,\langle\, T,\,U_{\alpha}\,/\,\alpha\in R\,\rangle\,$
%(see~\cite[0.31]{dm}).

Let $\,W\,=\,N(T)/T\,$ be the {\bf Weyl group} of $\,G\,$ (see
\cite[24.1]{hum2}).
\begin{theorem}{\rm\cite[27.1]{hum2}}\label{ssars}
Let $\,G\,$ be a semisimple algebraic group, and \linebreak 
$\,E\,=\,X\otimes_{\mathbb Z}{\mathbb R}\,$. Then $\,R\,$ is an abstract 
root system in $\,E$, whose rank is the rank of $\,G\,$ and whose abstract 
Weyl group is isomorphic to $\,W$.
\end{theorem}
Thus, all results on abstract root systems are applicable to
$\,R.\,$ Here we summarize some of them.

A {\bf basis} of $\,R\,$ is a subset 
$\,\Delta\,=\,\{\alpha_1,\ldots,\alpha_{\ell}\},\,$ $\,\ell\,=\,{\rm
rank}\,G,\,$ which spans $\,E\,$ (hence is a basis of $\,E\,$) and
relative to which each root $\,\alpha\,$ has a (unique) expression
$\,\alpha\,=\,\sum\,c_i\,\alpha_i\,$, where the $\,c_i$'s are integers
of like sign. The elements of $\,\Delta\,$ are called {\bf simple
roots}. Bases do exist; in fact, $\,W\,$ permutes them simply
transitively, and each root lies in at least one basis. Moreover,
$\,W\,$ is generated by the reflections $\,s_{\alpha}\,$ 
($\,\alpha\in\Delta\,$), for any basis $\,\Delta\,$
(see~\cite[III]{hum1}). There is an inner product
$\,(\,\cdot\,,\,\cdot\,)\,$ on $\,E\,$ relative to which $\,W\,$ consists
of orthogonal transformations. The formula for
$\,s_{\alpha}\,$ becomes $\,s_{\alpha}(\beta)\,=\,\beta\,-\, 
\langle\beta,\,\alpha\rangle\,\alpha\,$, where  
$\,\langle\beta,\,\alpha\rangle\,=\,2(\beta,\,\alpha)/(\alpha,\,\alpha).\,$

Bases of $\,R\,$ correspond one-to-one to Borel subgroups
containing $\,T\,$ \cite[27.3]{hum2}. In particular, each choice
of a Borel subgroup $\,B\,$ containing $\,T\,$ amounts to a choice of
$\,\Delta,\,$ or to a choice of {\bf positive} roots $\,R^+\,$ (those
for which all $\,c_i\,$ above are nonnegative). Moreover, if $\,\Delta\,$ is a
basis of $\,R,\,$ then $\,G\,=\,\langle\, T,\,U_{\alpha}\,/\,\pm\alpha\in
\Delta\,\rangle\,$ \cite[Theorem 27.3]{hum2}.

We recall that $\,R\,$ is called {\bf irreducible} if  
$\,\Delta\,$ cannot be partitioned into two disjoint 
``orthogonal'' subsets \cite[10.4]{hum1}.  
Every root system is the disjoint union of
(uniquely determined) irreducible root systems in suitable subspaces
of $\,E\,$ \cite[11.3]{hum1}.  
Up to isomorphism, the irreducible root systems correspond 
one-to-one to the following {\bf Dynkin diagrams} \cite[11.4]{hum1}:

\begin{center}
\setlength{\unitlength}{1cm}
\begin{picture}(12,15)
%Dynkin diagram for type A
\put(0,14.5){$A_{\ell}\,(\ell\geq 1)$:}
\multiput(2.5,14.6)(1.2,0){6}{\circle*{0.2}}
\put(2.6,14.6){\line(1,0){3.6}}
\multiput(6.35,14.5)(0.2,0){4}{$\cdot$}
\put(7.4,14.6){\line(1,0){1.2}}
%Dynkin diagram for type B
\put(0,13){$B_{\ell}\,(\ell\geq 2)$:}
\multiput(2.5,13.1)(1.2,0){6}{\circle*{0.2}}
\put(2.6,13.1){\line(1,0){2.4}}
\multiput(5.15,13)(0.2,0){4}{$\cdot$}
\put(6.2,13.1){\line(1,0){1.2}}
\put(7.2,13.15){\line(1,0){1.2}}
\put(7.3,13.05){\line(1,0){1.2}}
\put(7.7,13){$>$}
%Dynkin diagram for type C
\put(0,11.5){$C_{\ell}\,(\ell\geq 3)$:}
\multiput(2.5,11.6)(1.2,0){6}{\circle*{0.2}}
\put(2.6,11.6){\line(1,0){2.4}}
\multiput(5.15,11.5)(0.2,0){4}{$\cdot$}
\put(6.2,11.6){\line(1,0){1.2}}
\put(7.3,11.65){\line(1,0){1.2}}
\put(7.3,11.55){\line(1,0){1.2}}
\put(7.8,11.5){$<$}
%Dynkin diagram for type D
\put(0,10){$D_{\ell}\,(\ell\geq 4)$:}
\multiput(2.5,10.1)(1.2,0){5}{\circle*{0.2}}
\put(8.5,10.65){\circle*{0.2}}
\put(8.5,9.55){\circle*{0.2}}
\put(2.6,10.1){\line(1,0){2.4}}
\multiput(5.15,10)(0.2,0){4}{$\cdot$}
\put(6.2,10.1){\line(1,0){1.2}}
\put(7.39,10.1){\line(2,1){1.2}}
\put(7.39,10.1){\line(2,-1){1.2}}
%Dynkin diagram for type E6
\put(0,8.5){$E_{6}$:}
\multiput(1.5,8.6)(1.2,0){5}{\circle*{0.2}}
\put(1.6,8.6){\line(1,0){4.8}}
\put(3.9,7.4){\circle*{0.2}}
\put(3.9,7.4){\line(0,1){1.2}}
%Dynkin diagram for type E7
\put(0,6.8){$E_{7}$:}
\multiput(1.5,6.9)(1.2,0){6}{\circle*{0.2}}
\put(1.6,6.9){\line(1,0){6}}
\put(3.9,5.7){\circle*{0.2}}
\put(3.9,5.7){\line(0,1){1.2}}
%Dynkin diagram for type E8
\put(0,5.1){$E_{8}$:}
\multiput(1.5,5.2)(1.2,0){7}{\circle*{0.2}}
\put(1.6,5.2){\line(1,0){7.2}}
\put(3.9,4){\circle*{0.2}}
\put(3.9,4){\line(0,1){1.2}}
%Dynkin diagram for type F4
\put(0,3.3){$F_{4}$:}
\multiput(1.5,3.4)(1.2,0){4}{\circle*{0.2}}
\put(1.5,3.4){\line(1,0){1.2}}
\put(3.9,3.4){\line(1,0){1.2}}
\put(2.7,3.45){\line(1,0){1.2}}
\put(2.7,3.35){\line(1,0){1.2}}
\put(3.2,3.3){$>$}
%Dynkin diagram for type G2
\put(0,1.8){$G_{2}$:}
\multiput(1.5,1.9)(1.2,0){2}{\circle*{0.22}}
\multiput(1.5,2)(0,-0.1){3}{\line(1,0){1.2}}
\put(2,1.8){$\langle$}

\end{picture}
\end{center}

\vspace*{-6ex}

The vertices of the Dynkin diagram correspond to the simple roots
$\,\alpha_i\,$ ($\,i=1,\ldots,\ell\,$). The vertices corresponding to
$\,\alpha_i,\,\alpha_j\,$ are joined by
$\,\langle\alpha_i,\,\alpha_j\rangle\,\langle\alpha_j,\,\alpha_i\rangle\,$
edges, with an arrow pointing to the shorter of the two roots if they
are of unequal length. 
The order of $\,s_{\alpha_i}\,s_{\alpha_j}\,$ in $\,W\,$ is
$\,2,\,3,\,4,\,$ or $\,6,\,$ according to whether $\,\alpha_i\,$ and 
$\,\alpha_j\,$ are joined by $\,0,\,1,\,2,\,$ or $\,3\,$ edges
($\,\alpha_i\,\neq\,\alpha_j\,$). The giving of the Dynkin diagram is
equivalent to the giving of the matrix of {\bf Cartan integers}
$\,\langle\alpha_i,\,\alpha_j\rangle\,$
($\,\alpha_i,\,\alpha_j\,\in\Delta\,$). See \cite[11.2]{hum1}. 
 
By~\cite[27.5]{hum2}, if $\,G\,$ is a semisimple algebraic group, 
then it has a decomposition $\,G\,=\,G_1\,\cdots\,G_n\,$, where each 
$\,G_i\,$ is a minimal Zariski closed connected normal subgroup of positive
dimension. Moreover, this decomposition corresponds precisely to the
decomposition of $\,R\,$ into its irreducible components.
An algebraic group is {\bf simple} if it is non-commutative and has no
Zariski closed, connected, normal subgroups other than itself and $\,e\,$.
We say that a simple group is of {\bf exceptional type} if its (irreducible)
root system has Dynkin diagram of type $\,E_6,\,E_7,\,E_8,\,F_4,\,$ or
$\,G_2\,$. A simple group is of {\bf classical type} if its (irreducible)
root system has Dynkin diagram of type $\,A_{\ell},\,B_{\ell},\, 
C_{\ell},\,$ or $\,D_{\ell}.\,$

\subsubsection{The Structure of the Lie Algebra}\label{structliealg}

Let $\,{\mathfrak g}\,=\,{\cal L}(G)\,$ be the Lie algebra of the 
connected reductive algebraic group $\,G.\,$ 
The various structural properties we have described for $\,G\,$
have analogues for the Lie algebra $\,\mathfrak g.\,$ See
\cite[IV.13.18]{bor2},~\cite[26.2]{hum2}. 

Let $\,T\,$ be a maximal torus of $\,G\,$ and $\,{\mathfrak t}\,=\,
{\cal L}(T).\,$ For each root $\,\alpha\in R\,$ let $\,U_{\alpha}\,$ be
the corresponding root subgroup of $\,G\,$ and $\,{\mathfrak g}_{\alpha}\,=
\,{\cal L}(U_{\alpha}).\,$ Then we have a %direct 
decomposition % of $\,{\mathfrak g}\,$ as vector space given by 
\[
{\mathfrak g}\,=\,{\mathfrak t}\oplus\sum_{\alpha\in R}
{\mathfrak g}_{\alpha}\;.
\]
This is called the {\bf Cartan decomposition} of $\,{\mathfrak g}\,$
\cite[1.13]{car1}.
Let $\,{\mathfrak n}\,=\,{\cal L}(U)\,$ and $\,{\mathfrak n^-}\,=\,{\cal L}
(U^-).\,$ Then $\,{\mathfrak n}\,=\,\displaystyle\sum_{\alpha\in R^+}
{\mathfrak g}_{\alpha}\,$ and $\,{\mathfrak n^-}\,=\,\displaystyle
\sum_{\alpha\in R^-}{\mathfrak g}_{\alpha}\,$, where
$\,R^-\,=\,-R^+\,$. Hence we can write 
\[
{\mathfrak g}\,=\,{\mathfrak t}\oplus{\mathfrak n}\oplus{\mathfrak n^-}.
\]
Each of the spaces $\,{\mathfrak g}_{\alpha}\,$ is $\,1$-dimensional and
invariant under the adjoint action of $\,T\,$ on $\,{\mathfrak g}.\,$
Moreover the $\,1$-dimensional representation of $\,T\,$ afforded by the
module $\,{\mathfrak g}_{\alpha}\,$ is $\,\alpha.\,$
(The roots of $\,G\,$ can be defined in the following
way: as $\,T\,$ acts on $\,G\,$ by conjugation, $\,T\,$ acts 
on the Lie algebra $\,{\mathfrak g}\,=\,{\cal L}(G)\,$
via the adjoint representation $\,\Ad\,$. 
%This representation is diagonalizable
%(as is any rational representation of a torus). 
For each $\,\alpha\in X\,$ (that is, for each character of $\,T\,$) let 
\[
{\mathfrak g}_{\alpha}\,=\,\{\,x\in{\mathfrak g}\,/\,{\rm Ad}t\cdot x\,=\,
\alpha(t)\,x\;\,\mbox{for all}\;\,t\in T\,\}.
\]
As $\,T\,$ is diagonalizable, $\,{\mathfrak g}\,$ is the direct
sum of the $\,{\mathfrak g}_{\alpha}$'s \cite[III.8.17]{bor2}.
Those $\,\alpha\,$ for
which $\,{\mathfrak g}_{\alpha}\neq 0\,$ are called the {\it weights}
of $\,T\,$ in $\,{\mathfrak g}.\,$  The set $\,R\,$ of the {\it nonzero}
weights of $\,T\,$ in $\,{\mathfrak g}\,$ is called the 
set of {\it roots} of $\,G\,$ with respect to $\,T.\,$
Thus the space $\,{\mathfrak g}\,$ is the direct sum of 
$\,{\mathfrak t}\,=\,{\cal L}(T)\,=\,{\mathfrak g}_0\,$ and 
of the $\,{\mathfrak g}_{\alpha},\;\alpha\in R.\,$ As
$\,{\mathfrak g}_{\alpha}\,$ is $\,1$-dimensional for each $\,\alpha\in
R,\,$ we can write $\,{\mathfrak g}_{\alpha}\,=\,K\,e_{\alpha}\,$.
We call $\,e_{\alpha}\,$ a {\bf root vector} corresponding to $\,\alpha$.)

Let $\,B\,=\,T\,U\,$ be a Borel subgroup of $\,G.\,$ 
%$\,T\,$ a maximal torus, $\,U\,$ a maximal unipotent subgroup. 
By~\cite[14.25]{bor2}, %as $\,G\,$ is semisimple, then 
$\,\mathfrak g\,=\,({\rm Ad}\,G)\cdot {\mathfrak b},\,$ where
$\,{\mathfrak b}\,=\,{\cal L}(B).\,$ Let
$\,{\mathfrak t}\,=\,{\cal L}(T)\,$ and $\,{\mathfrak n}\,=\,{\cal L}(U).\,$ 
%$\,{\mathfrak b}\,=\,{\cal L}(B)\,$ is a Borel subalgebra and $\,B\,$ is a
%Borel subgroup of $\,G\,$ containing the maximal torus $\,T.\,$
%Hence, any element $\,x\in{\mathfrak g}\,$ lies in a
%Borel subalgebra $\,{\mathfrak b}\,=\,{\mathfrak t}\,+\,{\mathfrak n}\,$ 
%so that $\,x\,=\,t\,+\,n\,$ with $\,t\in{\mathfrak t},\,n\in{\mathfrak n}\,$.
Clearly, $\,\mathfrak b\,$ is spanned by $\,\mathfrak t\,$ and all  
$\,e_{\alpha}\,$, where $\,\alpha\in R^+.\,$ Moreover, 
$\,\mathfrak b\,=\,\mathfrak t\,+\,\mathfrak n.\,$ We call 
$\,\mathfrak b\,$ a {\bf Borel subalgebra} of $\,\mathfrak g\,$.

\begin{lemma}{\rm\cite[4.8]{veld}}\label{veldkamp}
Let $\,B'\,=\,H'U'\,$ with $\,B'\,$ a connected solvable group,
$\,H'\,$ a maximal torus, and $\,U'\,$ the maximal unipotent subgroup
of $\,B'$. Let $\,{\mathfrak b}',\,{\mathfrak h}',\,$ and
$\,{\mathfrak n}'\,$ be the respective Lie algebras, so that
$\,{\mathfrak b}'\,=\,{\mathfrak h}'\,+\,{\mathfrak n}'.\,$ If
$\,t\in{\mathfrak h}'\,$ and $\,n\in{\mathfrak n}'\,$, then there
exists $\,n'\in{\mathfrak n}'\,$ such that $\,t\,+\,n'\,$ is  
(${\rm Ad}\,U'$)-conjugate to $\,t\,+\,n\,$, and $\,t\,$ and $\,n'\,$ commute.
\end{lemma}

Recall that $\,x\in{\mathfrak g}\,$ is called toral if 
$\,x^{[p]}\,=\,x\,$.
\begin{proposition}\label{anytoral}
Any toral element in $\,{\mathfrak g}\,$ is 
(${\rm Ad}\,G$)-conjugate to an element in $\,{\mathfrak t}\,$.
\end{proposition}\noindent
\begin{proof}
Let $\,x\in{\mathfrak g}\,$ be toral. %By Proposition~\ref{properties},
As $\,\mathfrak g\,=\,({\rm Ad}\,G)\cdot {\mathfrak b}\,$ \cite[14.25]{bor2},
we may assume that $\,x\,=\,t\,+\,n,\,$ 
where $\,t\in\mathfrak t\,$ and $\,n\in\mathfrak n\,$.
%every toral element 
%$\,x\in{\mathfrak g}\,$ is semisimple, hence $\,x\,=\,t\,+\,n\,$ with 
%$\,t\,$ a toral and $\,n\,$ a nilpotent elements of the Borel
%subalgebra $\,{\mathfrak b}\,=\,{\mathfrak t}\,+\,{\mathfrak n}\,$.
By Lemma~\ref{veldkamp}, $\,x\,$ is (${\rm Ad}\,U'$)-conjugate to 
$\,t\,+\,n'\,$ and $\,t\,$ and $\,n'\,$ commute. As
$\,x^{[p]}\,=\,x,\,$ we must have $\,(t\,+\,n')^{[p]}\,=\,t\,+\,n'\,$.
On the other hand,
$\,(t\,+\,n')^{[p]}\,=\,t\,+\,{n'}^{[p]}\,$, by Jacobson's
identity (see Definition~\ref{restliealg}(ii')) and the fact that 
$\,[t,\,n']\,=\,0.\,$
Hence, $\,\,{n'}^{[p]}\,=\,n'\,$. As $\,n'\,$ is nilpotent, we get 
$\,n'\,=\,0$, finishing the proof. 
\end{proof}

\begin{lemma}\label{centretoral}
Let $\,G\,$ be a reductive algebraic group, and 
$\,\mathfrak g\,=\,{\cal L}(G)\,$. Then the centre $\,\mathfrak z\,=\,
\mathfrak z(\mathfrak g)\,$ is a toral subalgebra of $\,\mathfrak g\,$.
\end{lemma}\noindent
\begin{proof}
Clearly, $\,\mathfrak z\,$ is ($\Ad\,G)$-stable. 
So a maximal torus $\,T\,$ of $\,G\,$ acts diagonalisably on $\,\mathfrak z\,$.
%%So $\,\mathfrak z\,$
%%decomposes into weight spaces relative to a maximal torus $\,T$. 
If the root space $\,\mathfrak z_{\alpha}=\{\,x\in{\mathfrak z}\,/\,
{\rm Ad}t\cdot x\,=\,
\alpha(t)\,x\;\,\mbox{for all}\;\,t\in T\,\}\,$ is nonzero for some
$\,\alpha\neq 0$, then $\,e_\alpha\in \mathfrak z\,$
(as $\,\mathfrak g_{\alpha}\,$ is $\,1$-dimensional). 
But $\,[e_{\alpha},e_{-\alpha}]\,\neq\, 0$, for every $\,\alpha\in
R\,$ \cite[8.3]{hum1}. Hence $\,\mathfrak z\subset \mathfrak t\, 
=\,{\cal L}(T)$. As $\,\mathfrak t\,$ is toral, so is any of its  
restricted subalgebras. The result follows.
\end{proof}

At this point we should emphasize some properties of the $\,p$-mapping
in $\,\mathfrak g\,$.

Let $\,{\bf G}_a\,$ denote the additive group $\,K^+$. 
Then $\,K[{\bf G}_a]\,=\,K[x]\,$ so that $\,{\cal L} 
({\bf G}_a)\,$ is spanned by the invariant derivation  
$\,D\,=\,\displaystyle\frac{\rm d}{{\rm d}x}\,$ \cite[9.3]{hum2}.   
Since 
$\,D^{[p]}\,x^n\,=\,n(n-1)\cdots(n-(p-1))\,x^{n-p}\,$ (or zero if
$\,n<p\,$) it follows that %, if $\,{\rm char}\,K\,>\,0\,$, 
the $\,p$-mapping is zero in $\,{\cal L}({\bf G}_a)\,$ (because the
product of $\,p\,$ consecutive integers is divisible by $\,p\,$).

Thus, %if $\,{\rm char}\,K\,=\,p\,>\,0\,$, then  
the $\,p$-mapping of $\,\mathfrak g\,$ vanishes on
$\,{\mathfrak g}_{\alpha}\,\cong\,{\cal L}(U_{\alpha}),\,$ for each
root $\,\alpha\in R,\,$ implying that $\,e_{\alpha}^{[p]}\,=\,0\,$.  

Now let $\,{\bf G}_m\,$ denote the multiplicative group
$\,(K,\,\cdot\,).\,$ Then $\,K[{\bf G}_m]\,=\,K[x,x^{-1}]\,$  
(the ring of Laurent polynomials). For $\,D\in {\cal L}({\bf G}_m),\,$
one has  $\,D\,x\,=\,a\,x\,$, where $\,a\in K\,$ \cite[I.3.9]{bor2}. 
It follows that $\,D^{[p]}\,x\,=\,a^p\,x\,$. Thus $\,{\cal L} 
({\bf G}_m)\,$ is isomorphic to the $\,1$-dimensional Lie algebra
$\,K\,$ with $\,p$-mapping $\,a\longmapsto a^p\,$.
From this it follows that $\,{\cal L}({\bf G}_m)\,$ is a 
$\,1$-dimensional torus.
%Also, if $\,H\,$ is a one dimensional toral subalgebra, then for each
%$\,h\in H,\,$ $\,h^{[p]}\,=\,k\,h\,$ for a nonzero $\,k\in K$.

A maximal torus $\,T\,$ of a connected algebraic group $\,G\,$
is isomorphic to a direct product $\,K^*\times\cdots\times K^*\,$
($\,\ell\,$ times), where $\,\ell\,=\,{\rm
rank}\,G\,$ \cite[III.8.5]{bor2}. Thus 
$\,{\cal L}(T)\,\cong\,K\oplus\cdots\oplus K\,$
($\,\ell\,$ times) is a toral subalgebra of $\,{\cal L}(G)\,$ 
\cite[III.8.2 Corollary]{bor2}.
%It is easy to check that the $\,p$-mapping $\,[p]\,$ is  
%nonsingular on $\,\mathfrak t\,=\,{\cal L}(T)\,$.
By Theorem~\ref{3.6}(1), $\,\mathfrak t\,$ has a basis consisting of
toral elements. This implies the following:
\begin{lemma} \label{toralbase}
$\mathfrak t\,=\,{\cal L}(T)\,$ contains $\,p^{\dim T}\,$ toral
elements. If $\,G\,$ is reductive, there are finitely many toral
$\,(\Ad\,G)$-conjugacy classes in $\,\mathfrak g$.
\end{lemma}\noindent
\begin{proof} 
Let $\{\,t_1,\ldots ,\,t_{\ell}\,\}$ be a basis of $\,\mathfrak t\,$
consisting of toral elements. Suppose $\,x\,\in\,\mathfrak t\,$ is toral,
that is, $\,x\,=\,\sum_{i=1}^{\ell}\,\lambda_i\,t_i\,$ and $\,x^{[p]}\,=\,x$.
Then Jacobson's identity shows that 
$\,x^{[p]}\,=\,\sum_{i=1}^{\ell}\,\lambda_i^p\,t_i\,=\,
\sum_{i=1}^{\ell}\,\lambda_i\,t_i\,=\,x\,$. Therefore,
$\,\lambda_i^p\,=\,\lambda_i$, implying $\,\lambda_i\,\in\,\mathbb
F_p,\,$ for each $\,i\leq \ell\,$. It follows that $\,\mathfrak t\,$ contains
$\,p^{\ell}\,$ toral elements. As $\,\dim\,T\,=\,\ell\,$, the first
part of the lemma is proved. Now let $\,y\,\in\,\mathfrak g\,$ be toral.
%By Proposition~\ref{properties}(1), $\,y\,$ is semisimple, and thus
By Proposition~\ref{anytoral}, $\,y\,$ is $\,(\Ad\,G)$-conjugate to a
toral element in $\,\mathfrak t\,$. This finishes the proof.
\end{proof}

\subsubsection{Special and Good Primes}
Let $\,R\,$ denote the root system of a reductive algebraic
group $\,G\,$ relative to a maximal torus $\,T\subset G.\,$
\begin{definition}\label{badprime}
A prime $\,p\,$ is said to be {\bf bad} for %a simple algebraic group 
$\,G\,$ if $\,p\,$ divides a coefficient of a root $\,\alpha\in R\,$
when expressed as a combination %$\,\alpha\,=\,\sum_in_i\alpha_i\,$ 
of simple roots. 
\end{definition} 
The bad primes for the simple algebraic groups are as follows:

none if $\,G\,$ is of type $\,A_{\ell}\,$ 

$\,p=2\,$ if $\,G\,$ is of type $\,B_{\ell},\,C_{\ell},\,D_{\ell}\,$

$\,p=2\,$ or $\,3\,$ if $\,G\,$ is of type $\,G_2,\,F_4,\,E_6,\,E_7\,$

$\,p=2,\,3\,$ or $\,5\,$ if $\,G\,$ is of type $\,E_8\,$ (see
\cite[p. 106]{stein2}). \\
Primes which are not bad for $\,G\,$ are called {\bf good}.

\begin{definition}\label{notvgp}
Let $\,G\,$ be a simple algebraic group. A prime $\,p\,$ is said to be
{\bf very good} for $\,G\,$ if either 
$\,G\,$ is not of type $\,A_{\ell}\,$ and $\,p\,$ is a good prime for $\,G,\,$ 
or
$\,G\,$ is of type $\,A_{\ell}\,$ and $\,p\,$ does not divide $\,\ell+1$.
\end{definition}

\begin{definition}\label{pspecial}
Let $\,\alpha_0\,$ be the maximal short root
and $\,\tilde{\alpha}\,$ the longest root in $\,R^+.\,$ Put
$\,d=(\tilde{\alpha},\,\tilde{\alpha})/(\alpha_0,\,\alpha_0).\,$ We say that
$\,p\,$ is {\bf special} for a reductive group $\,G\,$ if $\,p|d.\,$
\end{definition}
By Jantzen~\cite{jant2}, $\,p\,$ is special
for $\,G\,$ if either $\,p=3\,$ and $\,R\,$ has a component of type $\,G_2\,$
or $\,p=2\,$ and $\,R\,$ has a component of type $\,B_{\ell},\,C_{\ell},\,
\ell\geq 2,\,$ or $\,F_4.\,$

\medskip
 
Let $\,\mathfrak g\,$ be the restricted Lie algebra of a %connected
reductive algebraic group $\,G\,$ defined over an algebraically closed
field of characteristic $\,p > 0\,$. % ($\,0\,$ or any prime). 
The classification of nilpotent $\,(\Ad\,G)$-orbits in $\,\mathfrak
g\,$ reduces easily 
to the case where $\,G\,$ is simple, and the following important result holds.
\begin{theorem}\label{badhs} %{\rm\cite[Thm. 1]{holspa}}
There are finitely many nilpotent $\,(\Ad\,G)$-orbits in $\,\mathfrak  g\,$.
In other words, there are finitely many $\,[p]$-nilpotent conjugacy
classes in $\,\mathfrak g\,$.
\end{theorem}\noindent
\begin{proof}
For $\,p\,$ good (or zero), the result is proved by Richardson
by reducing the general case to the case $\,\mathfrak
g\,=\,\mathfrak{gl}(V)\,$ \cite{rich2} (resp.,
\cite{kos},~\cite{dynk}). See also \cite[I.5.6]{spst}.
% using the classification of unipotent elements in $\,G$. 
For $\,p\,$ bad, the result is established by Holt and
Spaltenstein by using computer calculations~\cite{holspa}.
\end{proof} 
\medskip

\subsubsection{Chevalley Basis}

Let $\,\mathfrak g_{\mathbb C}\,$ be the complex simple Lie algebra
associated with the root system $\,R.\,$ 
Denote by $\,\mathfrak h_{\mathbb C}\,$ a maximal toral subalgebra of 
$\,\mathfrak g_{\mathbb C}\,$. Let 
\[
\mathfrak g_{\mathbb C}\,=\,\mathfrak h_{\mathbb
C}\,\oplus\,\sum_{\alpha\in R}\mathfrak g_{\mathbb C,\alpha}
\]
be a Cartan decomposition of $\,\mathfrak g_{\mathbb C}\,$. 
Let $\,h_{\alpha}\,=\,\displaystyle\frac{2\alpha}{(\alpha,\,\alpha)}\,\in 
\mathfrak h_{\mathbb C}\,$ be the {\it coroot} corresponding to   
$\,\alpha\in R\,$ \cite[3.6.1]{car2},~\cite[8.2]{hum1}.
According to~\cite[25.2]{hum1}, one can choose root 
vectors $\,e_{\alpha}\in\mathfrak g_{\mathbb
C, \alpha}\,$ ($\,\alpha\in R\,$) so that

(a) $\,[e_{\alpha},\,e_{-\alpha}]\,=\,h_{\alpha}\,$.

(b) If $\,\alpha,\,\beta,\;\alpha\,+\,\beta\,\in R\,$ and
$\,[e_{\alpha},\,e_{\beta}]\,=\,N_{\alpha\,\beta}\,e_{\alpha+\beta}$,
then $\,N_{\alpha\beta}\,=\,-N_{-\alpha,\,-\beta}$.

(c) $\,N_{\alpha\beta}^2\,=\,q(r+1)\,\displaystyle\frac{(\alpha\,+
\,\beta,\,\alpha\,+\,\beta)}{(\beta,\,\beta)}\,$, 
where $\,\beta\,-\,r\alpha,\,\ldots ,\,\beta\,+\,q\alpha\,$ is
the $\,\alpha$-string through $\,\beta\,$.

Let $\,\Delta\,=\,\{\alpha_1,\,\ldots,\,\alpha_{\ell}\,\}\,$ be a basis of
$\,R\,$. By definition a {\bf Chevalley basis of $\,\mathfrak g_{\mathbb C}\,$}
is any basis $\,\{\,h_i=h_{\alpha_i},\,\alpha_i\in\Delta\,\}\cup 
\{\,e_{\alpha},\,\alpha\in R\,\}\,$ of $\,\mathfrak g_{\mathbb C}\,$
with the $\,e_{\alpha}$'s satisfying the conditions (a),(b),(c) above.  

Now fix a Chevalley basis
$\,\{\,h_i=h_{\alpha_i},\,\alpha_i\in\Delta\,\}\cup\{\,e_{\alpha},\, 
\alpha\in R\,\}\,$ of $\,\mathfrak g_{\mathbb C}\,$. Then one has 
\begin{align}\label{relations}
[h_i,\,h_j]\, &=\,0,\qquad 1\,\leq\, i,\,j\,\leq\,\ell\nonumber\\
[h_i,\,e_{\alpha}]\, &=\,\langle \alpha,\,\alpha_i\rangle\,e_{\alpha}, 
\qquad 1\,\leq i \leq\,\ell\,\nonumber\\
[e_{\alpha},\,e_{-\alpha}]\, &=\,h_{\alpha}, \qquad \alpha\,\in\,R\\ 
%\;\;\mbox{is a $\,\mathbb
%Z$-linear combination of $\,h_1,\,\ldots,\,h_{\ell}\,$}\\
[e_{\alpha},\,e_{\beta}]\, &=\,0\,\qquad\mbox{if}\;\,\alpha+\beta\not\in R,
\nonumber\\
[e_{\alpha},\,e_{\beta}]\, &=\,N_{\alpha,\beta}\,e_{\alpha+\beta}
\qquad\mbox{if}\;\,\alpha+\beta\in R,\nonumber
\end{align}
where $\,N_{\alpha,\beta}\,=\,\pm(r+1),\,$ with $\,r\,$ the greatest 
integer for which $\,\beta-r\alpha\,$ is a root 
\cite[\S 25]{hum1},~\cite[4.2]{car2}.
We denote $\,e_i\,=\,e_{\alpha_i},\; f_i\,=\,e_{-\alpha_i}\,$ 
and rewrite the relations~\eqref{relations}
accordingly.

Let $\,\mathfrak g_{\mathbb Z}\,$ be the $\,\mathbb Z$-span of the
Chevalley basis. By~\eqref{relations}, $\,\mathfrak g_{\mathbb Z}\,$
is a Lie algebra over $\,\mathbb Z.\,$ 
Let $\,\mathfrak g\,=\,{\cal L}(G)\,$ denote the (restricted) Lie
algebra of $\,G\,$.  By \cite[Sect.2.5]{bor1}, 
$\,\mathfrak g\,=\,\mathfrak g_{\mathbb Z}\otimes_{\mathbb Z}K,\,$
provided that $\,G\,$ is simply connected.

For convenience, we also denote by $\,e_{\alpha}\,$ and
$\,h_{\alpha}\,$ the corresponding basis elements $\,e_{\alpha}\otimes
1\,$ and $\,h_{\alpha}\otimes 1\,$ of $\,\mathfrak g\,$.
By the discussion at the end of Section~\ref{structliealg}, we have 
that % the $\,p$th power map in $\,\mathfrak g\,$ has the property 
$\,e_{\alpha}^{[p]}\,=\,0,\;h_{\alpha}^{[p]}\,=\, 
h_{\alpha}\,$ for all $\,\alpha\in R$.

\subsubsection{Graph Automorphisms}\label{graphaut}
Let $\,R\,$ be a root system and $\,\Delta\,$ a basis of $\,R$.
The automorphism group of %a root system
$\,R\,$ is the semidirect product of $\,W\,$ %(which is normal) 
and the group $\,\Gamma\,:=\,\{\sigma\in{\rm Aut}\,R\,/\,\sigma 
(\Delta)=\Delta\}\,$ of
{\it graph} (or {\it diagram}) automorphisms of $\,R$.
Each $\,\sigma\in\Gamma\,$ determines an automorphism of the Dynkin
diagram of $\,R.\,$ Conversely, each automorphism of the Dynkin
diagram determines an automorphism of $\,R\,$ \cite[12.2]{hum1}. 
If $\,R\,$ is irreducible, then $\,\Gamma\,$ is
trivial, except for types $\,A_{\ell}\,(\ell\geq 2),\;D_{\ell}\,$ and
$\,E_6\,$ \cite[Table 1, p.\ 66]{hum1}.

Now let $\,R\,$ be the (irreducible) root system of a simply connected
simple algebraic group $\,G.\,$
Each automorphism $\,\sigma\in\Gamma\,$ %of the Dynkin diagram of $\,R\,$
gives rise to a (rational) automorphism of $\,G,\,$ denoted by 
$\,\hat{\sigma}\,$ and called a {\bf graph automorphism of $\,G\,$}
\cite[12.2 and 12.3]{car2}.

Each graph automorphism $\,\hat{\sigma}\,$ of $\,G\,$ induces a Lie
algebra automorphism of $\,\mathfrak g\,=\,{\cal L}(G),\,$ also
denoted by $\,\hat{\sigma}.\,$ The latter has the property:
\[
\hat{\sigma}(e_{\alpha})\, =\,e_{\sigma(\alpha)},\quad
\hat{\sigma}(f_{\alpha})\, =\,f_{\sigma(\alpha)}\qquad\quad\mbox{for
all}\;\;\alpha\in\Delta\,.
\]

The compatibility between the graph automorphisms of $\,G\,$ and
$\,\mathfrak g\,$ is as follows. For $\,g\in G,\;x\in\mathfrak g\,$ we
have $\,\hat{\sigma} (({\rm Ad}\,g)\cdot x)\,=\,({\rm
Ad}\,g^{\hat{\sigma}})\cdot x^{\hat{\sigma}}\,$.

By~\cite[Proposition 12.2.3]{car2}, there exist graph automorphisms of
order $\,2\,$ of $\,A_{\ell}(K),\;\ell\geq 2;\;D_{\ell}(K),\;\ell\geq
4; \;E_{6}(K)$. There is also a graph automorphism of order $\,3\,$ of
$\,D_4(K)$.

\subsection{Basic Representation Theory}\label{reptheo}

%\subsubsection{Root Systems and Weights}\label{introback}
From now on let $\,G\,$ denote a simply connected simple algebraic
group (universal Chevalley group) over an algebraically closed field $\,K\,$ 
of characteristic $\,p>0$.

Fix a maximal torus $\,T\,$ of $\,G,\,$ and let 
$\,B\,=\,T\,U\,$ be a Borel subgroup 
containing $\,T$. %, where $\,U\,$ is the unipotent radical of $\,B$. 
Denote by $\,X\,=\,X(T)\,$ the character group of $\,T.\,$ 
Let $\,R\,$ be the (irreducible) root system of $\,G\,$ with respect to 
$\,T,\,$ $\,R^+\,$ (resp. $\,R^-\,$) the set of positive (resp. negative) 
roots relative to $\,B$, and 
$\,\Delta\,=\,\{\alpha_1,\ldots,\alpha_{\ell}\}\,$ the corresponding basis 
of simple roots, where $\,\ell\,=\,{\rm rank}\,G\,$. Let
$\,W\,=\,N_G(T)/T\,$ be the Weyl group 
of $\,G,\,$ $\,s_{\alpha}\,$ the reflection in $\,W\,$ corresponding to  
$\,\alpha\in R$. Put $\,s_i\,=\,s_{\alpha_i}\,$, and let $\,w_0\,$
denote the unique element of 
$\,W\,$ sending $\,R^+\,$ to $\,R^-\,$. 
The cardinality of $\,R\,$ is denoted by $\,|R|\,=\,2\,|R^+|,\,$ and
$\,\dim\,G\,=\,|R|\,+\,\ell\,$ \cite[26.3]{hum2}.

\subsubsection{Weights}\label{introback}
Let $\,(\;\cdot\;,\,\cdot\;)\,$ denote a
nondegenerate $\,W$-invariant symmetric bilinear form on $\,X,\,$
extended to $\,E\,=\,X\otimes_{\mathbb Z}\mathbb R\,$. Set
$\,\langle\alpha,\,\beta\rangle\,=\,\displaystyle\frac{2\,(\alpha,\,\beta)}
{(\beta,\,\beta)}\,$, for $\,\alpha,\,\beta\,\in\,E,\,$ as in
Section~\ref{therootsys}. Since $\,R\,$ is an abstract root system, 
$\,\langle\alpha,\,\beta\rangle\,\in\,\mathbb Z\,$ for all
$\,\alpha,\,\beta\,\in\,R\,$ \cite[\S 27]{hum2}. 

A vector $\,\lambda\,\in\,E\,$ is called a
{\bf weight} provided that $\,\langle\lambda,\,\alpha\rangle\in\mathbb
Z\,$ for all $\,\alpha\in R$.
The weights form a lattice $\,P\,$ of $\,E,\,$ in which the lattice 
$\,\mathbb ZR\,$ spanned by $\,R\,$ is a subgroup of finite index.
$\,P/\mathbb ZR\,$ is called the {\bf fundamental group} of $\,G.\,$
$\,W\,$ acts naturally on $\,P\,$ and $\,\mathbb ZR,\,$ and acts trivially 
on the fundamental group $\,P/\mathbb ZR\,$ \cite[14.7]{bor2}.

If $\,\Delta\,=\,\{\alpha_1,\ldots,\alpha_{\ell}\},\,$ %is a basis of $\,R,\,$
then $\,P\,$ is generated by the system of {\bf fundamental weights}
$\,\{\omega_1, \ldots ,\omega_{\ell}\},\,$ defined by 
$\,\langle \omega_i,\alpha_j\rangle = \delta_{ij}\,$.    
Define $\,\rho\,:=\,\sum_{i=1}^{\ell}\,\omega_i\,$ (it equals half the sum 
of the positive roots).

A weight $\,\lambda\,$ is said to be {\bf dominant} if 
$\,\langle\lambda, \alpha \rangle \geqslant 0\,$ for all roots 
$\,\alpha\in R^+$. Let $\,P_{++}\,$ denote the set of all dominant weights.
Clearly, $\,P_{++}\,=\,\displaystyle \sum_{i=1}^{\ell}\,\mathbb Z^+\omega_i\,$.
\linebreak A weight $\,\lambda\in P\,$ is $\,W$-conjugate to one and only one 
dominant weight. There is a natural partial ordering of $\,E\,$: 
given $\,\lambda,\,\mu\in E,\,$ we write $\,\mu\,\leq\,\lambda\,$ if and only
if $\,\lambda\,-\,\mu\,$ is a sum of positive roots \cite[Appendix]{hum2}.

As $\,G\,$ is simply connected, the group $\,X\,=\,X(T)\,$ 
%of rational characters of $\,T\,$ 
is the full weight lattice $\,P\,$ \cite[\S 31]{hum2}.

\subsubsection{Maximal Vectors and Irreducible Modules}\label{wamv}

Let $\,V\,$ be a finite dimensional vector space over $\,K.\,$ 
If $\,\pi\,:\,G\,\rightarrow\,GL(V)\,$ is a rational representation 
of $\,G\,$, then we call $\,V\,$ a {\bf rational $\,G$-module}. 
Denote $\,\pi(g)\,v\,=\,g\cdot v\,$ for all $\,g\in G,\,v\in V$.

Being a diagonalizable group, the maximal torus $\,T\,$ acts
completely reducibly on $\,V,\,$ so that
$\,V\,$ is the direct sum of weight spaces
\[
V_{\mu}\,:=\,\{\,v\in V\,/\,t\cdot v\,=\,\mu(t)\,v\;\;\mbox{for
all}\;t\in T\}\,,
\]
where $\,\mu\in X\,$ \cite[III.8.4]{bor2}.  
The dimension $\,\dim\,V_{\mu}\,$ is called the {\bf multiplicity} of $\,\mu$.
Put $\,{\cal X}(V)\,:=\,\{\,\mu\in X\,/\,V_{\mu}\,\neq\,(0)\,\}.\,$
The elements of $\,{\cal X}(V)\,$ are called the {\bf weights of
$\,V\,$}.
We call $\,v\in V_{\mu}\,$ a {\bf weight vector} (of weight $\,\mu\,$).
Put $\,{\cal X}_{++}(V)\,=\,{\cal X}(V)\cap P_{++}\,$.  

A (nonzero) vector $\,v\in V\,$ is called a {\bf maximal vector}
if it is fixed by all $\,u\in U.\,$ If $\,V\,$ is nonzero, then
maximal vectors exist \cite[31.2]{hum2}. 

If $\,V\,$ is
irreducible, then a maximal vector $\,v_0\,$ of weight $\,\lambda\,$ 
is unique up to a scalar multiple. Moreover, $\,\lambda\,$ is a dominant weight
of multiplicity $\,1\,$ and is called the {\bf highest weight} of $\,V.\,$
The vectors in $\,V_{\lambda}\,$ are called {\bf highest weight vectors}.
All other weights of $\,V\,$ are of the form
$\,\lambda\,-\,\sum\,c_{\alpha}\,\alpha,\,$ where $\,\alpha\in
R^+\,$ and $\,c_{\alpha}\in \mathbb Z^+\,$. They are permuted by $\,W,\,$ with
$\,W$-conjugate weights having the same multiplicity \cite[31.3]{hum2}. 
It follows that $\,{\cal X}(V)\,=\,W\cdot {\cal X}_{++}(V)\,$.

If $\,V'\,$ is an irreducible rational $\,G$-module with highest weight
$\,\lambda'\,$, then $\,V\,\cong\,V'\,$ (as
$\,G$-module) if and only if $\,\lambda\,=\,\lambda'\,$ 
\cite[31.3]{hum2}. We denote by $\,E(\lambda)\,$ the irreducible
rational $\,G$-module with highest weight $\,\lambda\,$.

In the next section we will review the method of construction of 
irreducible rational $\,G$-modules.

\subsubsection{The Construction of Irreducible Modules}

The construction of irreducible rational modules for semisimple
algebraic groups given below
shows a connection between the representation theory of such groups 
with the representation theory of {\it hyperalgebras}.
The main results of this section are due to Chevalley  
\cite{stein1}, \cite{wong1}.

Recall some notation:  %from Section~\ref{introback}.
$\,\mathfrak g_{\mathbb C}\,$ is the complex
simple Lie algebra associated with the root system $\,R,\,$ and 
$\,\Delta\,$ is a basis of $\,R$. Fix a Chevalley basis
$\,\{ h_{\beta},\,\mbox{$\beta\in\Delta$};$ \linebreak 
$\,\mbox{$e_{\alpha}$},\, \mbox{$\alpha\in R$} \}\,$
of $\,\mathfrak g_{\mathbb C}\,$ and  % (cf. Section~\ref{chevbas}).
let $\,\mathfrak g_{\mathbb Z}\,$ be the $\,\mathbb Z$-span of the
Chevalley basis. The basis of the restricted Lie algebra
$\,\mathfrak g\,=\,{\cal L}(G)\,=\,\mathfrak g_{\mathbb
Z}\otimes_{\mathbb Z}K\,$ is also denoted by $\,e_{\alpha}\,$ and
$\,h_{\beta}\,$ (that is, we identify $\,e_{\alpha}\,$ and
$\,h_{\beta}\,$ with $\,e_{\alpha}\otimes 1\,$ and
$\,h_{\beta}\otimes 1\,$ in $\,\mathfrak g\,$). 
The $\,p$th power map in $\,\mathfrak g\,$ has the property
$\,e_{\alpha}^{[p]}\,=\,0,\;h_{\alpha}^{[p]}\,=\,h_{\alpha}\,$ for all
$\,\alpha\in R$.

Using Kostant's Theorem~\cite{wong1}, which describes a $\,\mathbb
Z$-basis of the $\,\mathbb Z$-form $\,U_{\mathbb Z}\,$ of the
universal enveloping algebra $\,U\,=\,U(\mathfrak g_{\mathbb C})\,$
of $\,\mathfrak g_{\mathbb C}\,$ generated by all
$\,\displaystyle\frac{e^n_{\alpha}}{n!}\,$ ($\,\alpha\in
R,\;n\in\mathbb Z^+$), we can construct % (in Section~\ref{admzform})
an {\it admissible lattice} (a lattice stable under $\,U_{\mathbb Z}\,$)
in an arbitrary $\,\mathfrak g_{\mathbb C}$-module \cite[\S 27]{hum1}.

In particular, let $\,V\,=\,V_{\lambda}\,$ be an irreducible
$\,\mathfrak g_{\mathbb C}$-module with highest weight
$\,\lambda\in P_{++}\,$,
and let $\,v_0\,$ be a maximal vector of $\,V.\,$ This is 
a nonzero weight vector annihilated by all $\,e_{\alpha},\,\alpha\in
R^+\,$. By Theorem~\cite[27.1]{hum1}, $\,V_{\mathbb
Z}\,:=\,U_{\mathbb Z}\,v_0\,=\,U_{\mathbb Z}^-\,v_0\,$ is
% an admissible lattice, 
the (unique) smallest $\,\mathbb Z$-form of $\,V\,$ containing
$\,v_0\,$ and stable under $\,U_{\mathbb Z}\,$. In particular,
$\,V_{\mathbb Z}\,$ is an admissible lattice.

Now tensoring with $\,K,\,$ one obtains a $\,K$-space
$\,V_K(\lambda)\,:=\,V_{\mathbb Z}\otimes K\,$. By construction,
$\,V_K(\lambda)\,$ carries a canonical module structure over
the {\bf hyperalgebra} $\,U_K\,:=\,U_{\mathbb Z}\otimes_{\mathbb Z}K\,$
\cite[20.2]{hum1},~\cite[1.3]{wong1}. 
Note that $\,U_K\,$ is generated over $\,K\,$ by $\,x_{\alpha}^{(m)}\,=\, 
\displaystyle\frac{e_{\alpha}^m}{m!} \otimes 1,\,$ where $\,\alpha\in
R\,$ and $\,m\in\mathbb Z^+\,$.
% and that $\,x_{-\alpha}^{(m)}\,(\,\alpha\in R^+)\,$ generate 
%$\,U_K^-\,=\,U_{\mathbb Z}^-\otimes_{\mathbb Z}K\,$.   
% has basis $\,x_{-\alpha}^{(m)}\,(\,\alpha\in R^+)\,$.
Let $\,e_0\,=\,v_0\otimes 1\,$. Then
% denote a maximal vector in $\,V_K(\lambda),\,$ then
\[
U_K\cdot e_0\,=\,(U_{\mathbb Z}\otimes_{\mathbb Z}K)\cdot (v_0\otimes
1)\,=\,U_{\mathbb Z}\cdot v_0\otimes K\,=\,V_{\mathbb Z}\otimes K\,=\,
V_K(\lambda)\,.
\]
One easily sees that $\,V_K(\lambda)\,=\,U_K^-\cdot e_0\,$, where 
$\,U_K^-\,=\,U_{\mathbb Z}^-\otimes_{\mathbb Z}K\,$. 
Moreover, $\,V_K(\lambda)\,$ is the direct sum
of its weight spaces $\,V_K(\lambda)_{\mu}\,=\,(V_{\mathbb Z}\cap
V_{\lambda,\mu})\otimes_{\mathbb Z}K\,$, where $\,\mu\in {\cal
X}(V)\,$ \cite[Theorem 20.2]{hum1}.
% and it is finite dimensional \cite[p. 12]{wong1}. 
%Therefore, %by Theorem~\ref{connection},  
It follows that $\,V_K(\lambda)\,$ is a rational $\,G$-module of
highest weight $\,\lambda,\,$ whose dimension over $\,K\,$ equals
$\,\dim_{\mathbb C}\,V_{\lambda}\,$.

\begin{theorem}{\rm\cite{cps}}\label{connection}
There is a one-to-one correspondence between finite dimensional 
$\,U_K$-modules and finite dimensional rational $\,G$-modules.
\end{theorem}
Since $\,\mathfrak g_{\mathbb Z}\,$ embeds in
$\,U_{\mathbb Z}\,$, $\,V_K(\lambda)\,$ %=\,V_{\mathbb Z}\otimes K\,$
carries a canonical $\,\mathfrak g$-module structure.
%$\,\mathfrak g\,=\,\mathfrak g_{\mathbb Z}\otimes_{\mathbb Z}K\,$.
One can show that the corresponding representation 
$\,\mathfrak g\longrightarrow\mathfrak{gl}(V_K(\lambda))\,$ is nothing
but the differential of the rational representation
$\,G\longrightarrow\GL (V_K(\lambda))\,$ induced by the action of 
$\,U_K\,$ on $\,V_K(\lambda)\,$ (see Theorem~\ref{connection}). 

The rational $\,G$-module $\,V_K(\lambda)\,$ is called the
{\bf Weyl module} with highest weight $\,\lambda$.
The vector $\,e_0\in V_K(\lambda)\,$ is a maximal vector of
weight $\,\lambda\,$ relative to the action of a Borel subgroup
$\,B\,=\,T\,U\,$ of $\,G$.

There is also an intrinsic construction of Weyl modules. Let
$\,\lambda\,$ be a dominant weight of $\,X(T)\,\cong\,P\,$ 
and $\,\langle v \rangle\,$ a
$\,1$-dimensional module for $\,B\,$ affording $\,\lambda\,$
upon restriction to $\,T$. Now $\,G\times B\,$ acts on $\,K[G]$,
with $\,G\,$ acting on the left and $\,B\,$ on the right.
This yields an action of $\,G\times B\,$ on $\,K[G]   \otimes_K\langle
v  \rangle $.  Set  $\,H(\lambda)
= {(K[G]\otimes_K\langle v \rangle )}^B$,  where $\,B\,$
denotes the direct factor of $\,G \times B$.  One checks that $\,H(\lambda) =
D(\lambda)\otimes_K\langle v \rangle $, where $\,D(\lambda) = \{f
\in K[G]\,/\,  f(x) = \lambda(b)f(xb),  \; x \in G, \; b \in B\}$. 
It turns out that $\,H(\lambda)\,$ has a simple socle isomorphic to
$\,E(\lambda)\,$.  
%=\,E_G(\lambda)$.  The Weyl module is
%related to the  dual of this module (see \cite[31.6]{hum2}).  Indeed,
Moreover, $\,H(-\omega_0\lambda)^*\,$ is isomorphic to the Weyl module
$\,V_K(\lambda)\,$ \cite[Part II, 2.12]{jant1}.
\medskip

From now on denote by $\,V(\lambda)\,$ the Weyl module $\,V_K(\lambda)\,$.
Recall some fundamental properties of the Weyl modules.
%Proofs can be found in sections 2.13, 2.14, 5.11 of {\cite{jant1}}.

\begin{theorem}{\rm\cite[2.13-14,5.11]{jant1}}\label{univprop}~Let  
$\lambda$ be a dominant weight of \mbox{$\,X(T)\,\cong\,P$}.
\begin{itemize}
\item[(1)]  $V(\lambda)\,$ has a unique maximal submodule
$\,\Phi(\lambda)\,$ such that $\,V(\lambda) / \Phi(\lambda)
\cong E(\lambda).\,$ %is irreducible of highest weight $\,\lambda$. 
In particular $\,V(\lambda)\,$ is indecomposable.
\item[(2)] ${\rm dim}(V(\lambda))$ is given by the Weyl dimension formula
(see {\rm \cite[24.3]{hum1}}).
\item[(3)] Given a rational $\,G$-module $\,M\,$ generated by a
maximal weight vector of weight $\,\lambda,\,$ there is an epimorphism
$\,V(\lambda)\longrightarrow M\,$. In other words, $\,V(\lambda)\,$ is
a universal object among the rational $\,G$-modules of highest weight 
$\,\lambda$.
\end{itemize}
\end{theorem}

In general, the Weyl module $\,V(\lambda)\,$ 
is reducible. Its composition factors are of the form $\,E(\mu),\,$ 
where $\,\mu\leq\lambda$, and $\,E(\lambda)\,$ always occurs
in $\,V(\lambda)\,$ with (composition) multiplicity $\,1$.
%only once (as unique irreducible quotient module), together with 
%possibly other modules $\,E(\mu),\,$ where $\,\mu\,<\,\lambda\,$.
However, in one important case the Weyl module 
$\,V(\lambda)\,$ is irreducible and its dimension is known.
% is the {\bf Steinberg module} $\,{\rm St}\,=\,V((p-1)\rho)\,$.

\begin{theorem}[Steinberg]{\rm\cite{stein3}}\label{stmod}
The {\bf Steinberg module} $\,E((p-1)\rho)\,$ is an irreducible 
$\,G$-module of dimension $\,p^{|R^+|}\,$.
% and all other $\,E(\lambda),\;\lambda\in  
%\Lambda_p\,$, have strictly smaller dimension.
\end{theorem}

\subsubsection{Infinitesimally Irreducible Modules}\label{construct}

If $\,\pi\,:\,G\,\rightarrow\,GL(V)\,$ is a rational representation, 
%$\,V\,$ is a rational $\,G$-module, 
then the differential
$\,{\rm d}\pi\,:\,\mathfrak g\,\rightarrow\mathfrak{gl}(V)\,$  is a
representation of $\,\mathfrak g\,$ (see Section~\ref{tlaag}). 
Let $\,\lambda\,=\,\lambda(\pi)\,$ be the highest weight of
$\,V\,$. %%% (or of the representation $\,\pi\,$).

The differential $\,{\rm d}\mu\,: \mathfrak t\longrightarrow K\,$
of a weight $\,\mu\in X\,$ is a linear function on $\mathfrak t$.
The torus $\,\mathfrak t\,$ acts on the weight space $\,V_{\mu}\,$ 
via the differential $\,{\rm d}\mu\,$. As usual, abusing notation, 
we identify $\,\mu\,$ with $\,{\rm d}\mu\,$ and call $\,\mu\,=\,{\rm
d}\mu\,$ a {\it weight of $\,\mathfrak t\,$}.
However, the elements of $\,p\,X\,$ have zero differentials, so that
the weights of $\,\mathfrak t\,$ correspond one-to-one to the elements
of $\,X/pX\,$, a set of cardinality $\,p^{\ell}.\,$ 

A weight $\,\lambda\in P_{++}\,$ is called {\bf $\,p$-restricted} if 
$\,\lambda= \sum_ic_i\omega_i\,$ with $\,0\leq c_i \leq p-1\,$
%,\,c_i\in \mathbb Z_+\,$ 
for all $\,i\,$. Denote by $\,\Lambda_p\,$ the subset
of $\,P_{++}\,$ consisting of all $\,p$-restricted weights.

\begin{definition}\label{infirred}
A rational $\,K$-representation $\,\pi\,$ of $\,G\,$ is called {\bf
infinitesimally irreducible} if its differential $\,{\rm d}\pi\,$
defines an irreducible representation of the Lie algebra $\,\mathfrak
g\,=\,{\cal L}(G)\,$.
\end{definition}
\begin{lemma}{\rm\cite[6.2]{bor1}}\label{bor42}
Let $\,\pi\,:\,G\,\rightarrow\,GL(V)\,$ %$\,(\pi,\,V)\,$ 
be an infinitesimally irreducible rational representation 
of $\,G\,$ with highest weight $\,\lambda\,=\,\lambda(\pi)\,$. Then

(i) $\,V\,=\, U(\mathfrak n^-)\cdot V_{\lambda}\,$.

(ii) $\,V_{\lambda}\,$ is the only subspace of $\,V\,$ annihilated by 
$\,\mathfrak n\,$.
\end{lemma}
Denote by $\,M(G)\,$ the set of all infinitesimally irreducible 
representations of $\,G$. The following two theorems (due to Curtis and
Steinberg) are among the fundamental results of the modular
representation theory.
\begin{theorem}{\rm\cite{curt1},\cite{stein3}}\label{curt}

(i) A rational representation $\,\pi\,$ of $\,G\,$ is infinitesimally
irreducible if and only if $\,\lambda(\pi)\,\in\,\Lambda_p\,$ (that
is, $\,\lambda(\pi)\,$ is a $\,p$-restricted weight).

(ii) Each irreducible $\,p$-representation of the restricted Lie 
algebra $\,\mathfrak g\,=\,{\cal L}(G)\,$ is equivalent to the
differential of a unique infinitesimally irreducible representation of
$\,G$.
\end{theorem}

\begin{theorem}[Steinberg's Tensor Product Theorem]
{\rm\cite{stein3}}\label{bts}

For any irreducible rational $\,K$-rep\-re\-sent\-ation $\,\psi\,$ of 
$\,G\,$ there exist infinitesimally irreducible rep\-re\-sent\-a\-tions 
$\,\pi_0,\,\pi_1,\,\ldots,\,\pi_m\,$ such that 
\begin{equation}\label{decomp}
\psi\cong \pi_0\,\otimes\,\pi_1^{\rm Fr}\,\otimes
\,\ldots\,\otimes\,\pi_m^{{\rm Fr}^m}\,.
\end{equation}
Here $\,{\rm Fr}\,$ is the Frobenius endomorphism of the field $\,K$.
\end{theorem}

Let $\,{\cal X}(\psi)\,$ denote the set of weights of a rational 
$\,G$-representation $\,\psi\,$. If $\,\psi\,$ is irreducible, then 
by formula~\eqref{decomp} we have
%Formula~\eqref{decomp} shows that 
%the following equality holds for an irreducible representation $\,\psi\,$:
\[
{\cal X}(\psi)\,=\,{\cal X}(\pi_0)\,+\,p\,{\cal X}(\pi_1)\,+\,
\ldots\,+\,p^m\,{\cal X}(\pi_m)\,.
\]
%(here $\,{\cal X}(\psi)\,$ denotes the set of weights of a rational 
%$\,G$-representation $\,\psi\,$).

Steinberg's tensor product theorem can be stated in another way.
%Let $\,G\,$ be a semisimple algebraic group over a field $\,K\,$ of 
%characteristic $\,p> 0$.  
If $\,\lambda\,$ is a dominant weight, then it has a unique 
$\,p$-adic expansion $\lambda = \lambda_0 + p\lambda_1 +
\cdots + p^m\lambda_m$, where $\lambda_0, \ldots , \lambda_m\,$ are
$\,p$-restricted weights.
%Let $\,E(\lambda)\,$ denote the irreducible $\,G$-module with highest 
%weight $\,\lambda$.

\begin{theorem}~{\rm \cite{stein3}}~~Let $\,G\,$ be a simply connected
semisimple algebraic group and let $\,\lambda\,$ be a dominant weight
with $\,p$-adic expansion $\,\lambda = \lambda_0 + p\lambda_1 +
\cdots + p^m\lambda_m\,$. Then
%where each of $\lambda_0, \ldots , \lambda_m\,$ is restricted. Then
$$
E(\lambda)
\cong  E(\lambda_0) \otimes E(\lambda_1)^{Fr} \otimes \cdots \otimes
E(\lambda_m)^{{Fr}^m}.
$$
\end{theorem}

This theorem reduces many questions concerning irreducible representations
of $\,G\,$ to the study of the infinitesimally irreducible ones.
One of such questions is a description of systems of weights of  
irreducible rational representations of simple algebraic groups.
%over $\,K\,$ reduces to a detailed study of the weights of the
%infinitesimally irreducible representations.

\subsubsection{Weight Space Decomposition}

In this section, we explore the relationship between the 
weight spaces of the Weyl module $\,V(\lambda)\,$ and the  
weight spaces of its irreducible quotient
$\,E(\lambda)\,$. The main result (Theorem~\ref{premprin}) asserts that 
although in passing from  
the Weyl module to its irreducible top factor the dimensions of weight spaces
may decrease, weight spaces rarely disappear entirely.

Denote by $\,\pi_{\mathbb C}\,:\,\mathfrak g_{\mathbb
C}\longrightarrow \mathfrak{gl}(V_{\lambda})\,$ the irreducible complex
representation of the Lie algebra $\,\mathfrak g_{\mathbb C}\,$ in the
vector space $\,V_{\lambda}\,$. Let $\,{\cal X}(\pi_{\mathbb
C})\,$ denote the set of weights of this representation (that is,
$\,{\cal X}(\pi_{\mathbb C})\,=\,{\cal X}(V_{\lambda})\,$). Let 
$\,\pi\,:\,G\longrightarrow \GL(E(\lambda))\,$ be the irreducible
representation of $\,\mathfrak g\,$ in the vector space $\,E(\lambda).\,$
It follows from our discussion in
Section~\ref{construct} that
$\,{\cal X}(\pi)\,\subseteq\,{\cal X}(\pi_{\mathbb C})\,$.

Recall that
$\,{\cal X}(\pi)\,$ and $\,{\cal X}(\pi_{\mathbb C})\,$ are both
$\,W$-invariant, and any weight of $\,P\,$
is $\,W$-conjugate to precisely one dominant weight. Therefore, we can write

$\,{\cal X}(\pi)\,=\,W\cdot {\cal X}_{++}(\pi)\qquad$ and 
$\qquad{\cal X}(\pi_{\mathbb C})\,=\,W\cdot {\cal X}_{++}(\pi_{\mathbb
C})\,$. 

%\pagebreak 

By \cite[Chap. VIII \S 7]{bourb3}, it is true that 
$\,{\cal X}_{++}(\pi_{\mathbb C})\,=\,(\lambda(\pi)\,-\,Q_+)\,\cap\,P_{++},
\,$ where $\,Q_+\,=\,\{ \sum c_{k}\,\alpha_k\,/\,\alpha_k\in\Delta,\,
c_{\alpha_k}\in\mathbb Z^+\,\}\,$ and $\,(\lambda(\pi)\,-\,Q_+)\,=\, 
\{\,\lambda(\pi)\,-\,\beta\,/\,\beta\in Q_+\,\}$.

%%%%%%%%%%%% We may need this definition or not 
%%%%%%If $\,\gamma\in Q_+\,$, the we call the subset
%%%%%%%$\,Y(\gamma)\,=\,\mbox{$\{\alpha_k\in\Delta\,/\,c_{\alpha_k}\neq
%%%%%%%0\,\}$}\,$ the {\bf support} of $\,\gamma$.

Let $\,T\,$ be a maximal torus of $\,G\,$ with which the root system
$\,R\,$ is associated, and let
\[
V(\lambda)\,=\,\bigoplus_{\mu\in {\cal X}(\pi_{\mathbb
C})}\,V(\lambda)_{\mu}\,\quad\mbox{and}\,\quad
\Phi(\lambda)\,=\,\bigoplus_{\mu\in {\cal X}(\pi_{\mathbb
C})}\,\Phi(\lambda)_{\mu}
\]
be the decompositions of the $\,G$-modules $\,V(\lambda)\,$ and
$\,\Phi(\lambda)\,$ into the direct sum of weight subspaces with
respect to $\,T\,$. One has 
$\,V(\lambda)_{\mu}\,=\,(V_{\mathbb Z}\cap V_{\lambda,\mu}) 
\otimes_{\mathbb Z}K\,$ and
$\,\Phi(\lambda)_{\mu}\,\subseteq\,V(\lambda)_{\mu}\,$.
The quotient module $\,V(\lambda)/\Phi(\lambda)\,=\,E(\lambda)\,$
is irreducible and has highest weight $\,\lambda$. It is clear
that
$\,E(\lambda)_{\mu}\,\cong\,V(\lambda)_{\mu}/\Phi(\lambda)_{\mu}\,$ as
vector spaces. Since $\,\pi\,$ is realized in 
$\,E(\lambda)\,$, the following equality holds:
\begin{equation}\label{iiii}
{\cal X}(\pi)\,=\,\{\,\mu\in {\cal X}(\pi_{\mathbb C})\,/\,
V(\lambda)_{\mu}\,\neq\,\Phi(\lambda)_{\mu}\,\}.
\end{equation}

%Let $\,e(R)\,$ be the maximum
%of the squares of the ratios of the lengths of the roots in $\,R$.
%We recall that $\,e(R)=1\,$ if all roots in $\,R\,$ have the same length,
%$\,e(R)=2\,$ if $\,R\,$ is of type $\,B_{\ell},\,C_{\ell},\,F_4\,$ and
%$\,e(R)=3\,$ if $\,R\,$ is of type $\,G_2$.
%The following theorem is a consequence of~\eqref{iiii}.
%%Theorem~\ref{premprin} is equivalent to the following:
\begin{theorem}{\rm\cite[p. 169]{pre1}}\label{theop1}
%Let $\,p\,>\,e(R)\,$ 
Let $\,p\,$ be non-special for $\,G$, and
$\,\lambda\,\in\,\Lambda_p$. For $\,\lambda\,=\,\omega_1\,$, assume
also that $\,p\neq 2\,$ for groups of type $\,G_2\,$.
Then $\,V(\lambda)_{\mu}\,\neq\,\Phi(\lambda)_{\mu}\,$ for any weight
$\,\mu\,$ of the Weyl module $\,V(\lambda)$.
\end{theorem}
%Let $\,M(G)\,$ be the set of all infinitesimally irreducible
%representations of the group $\,G$.
In~\cite[p. 169]{pre1}, Premet shows that Theorem~\ref{theop1} is  
equivalent to the following:
\begin{theorem}\label{premprin}
%Let $\,p\,>\,e(R)$. 
Let $\,p\,$ be as in Theorem~\ref{theop1}. % non-special for $\,G$. 
Then for any $\,\pi\in M(G)\,$ the equality
$\,{\cal X}(\pi)\,=\,{\cal X}(\pi_{\mathbb C})\,$ holds. In particular,
\[
{\cal X}_{++}(\pi)\,=\,(\lambda(\pi)\,-\,Q_+)\,\cap\,P_{++}.
\]
\end{theorem}
%In passing from the Weyl module to its irreducible quotient the dimension
%of weight spaces may decrease.  However, this Theorem
%shows that weight spaces rarely disappear entirely.
%For $\,p\leq e(R)\,$ the inequality
For special primes the inequality
$\,{\cal X}(\pi)\,\neq\,{\cal X}(\pi_{\mathbb C})\,$ becomes an
ordinary occurrence. The simplest examples are the natural
representation of the group $\,B_{\ell}(K)\,$ for $\,p=2\,$ and 
the adjoint representation of the group of type $\,G_2\,$ for $\,p=3$.
\medskip

The following result tells us about the weights of irreducible
restricted \mbox{$\,{\mathfrak g}$-modules} (cf. Theorem~\ref{curt}(ii)).
%a rational $\,G$-module $\,V\,$ when considered
%as a restricted $\,{\mathfrak g}$-module.

\begin{corollary}{\rm\cite{pre1}}\label{pppweights}
%Let $\,p\,>\,e(R)$, Let $\,p\,$ be non-special for $\,G$, and let  
Let $\,p\,$ be as in Theorem~\ref{theop1}. Let $\,\phi\,$ be an
irreducible \mbox{$\,p$-representation} of %a Lie algebra 
$\,{\mathfrak g}\,$ with highest weight $\,\bar{\lambda}\,\in\,
P\,\otimes_{\mathbb Z}\,\mathbb F_p$, and $\,\lambda\,$ the unique  
inverse image of $\,\bar{\lambda}\,$ under the reduction homomorphism
$\,P\,\longrightarrow\,P\,\otimes_{\mathbb Z}\,\mathbb F_p\,$,  
lying in $\,\Lambda_p$.
Then the set of $\,\mathfrak t$-weights of $\,\phi\,$ coincides with
the image of  $\,W\cdot ((\lambda\,-\,Q_+)\,\cap\,P_{++})\,$
in $\,P\,\otimes_{\mathbb Z}\,\mathbb F_p\,$.
%under reduction modulo $\,p$.
\end{corollary}

\subsubsection{Weight Multiplicities}

The main unsolved problem concerning the modules $\,E(\lambda)\,$ is 
the determination of their weight multiplicities (or formal characters 
of $\,E(\lambda)\,$).
A lot of work has been done in the direction of solving this problem
in the last 30 years and good progress has been made.

Thanks to Steinberg's tensor product Theorem~\ref{bts}, it is enough  
to obtain this kind of information about the collection
$\,\{\,E(\lambda)\,/\,\lambda\in \Lambda_p\,\}\,$. However, we observe 
that even if $\,\lambda\in \Lambda_p\,$, some of the dominant weights below
$\,\lambda\,$ in the partial ordering defined in
Section~\ref{introback} may lie outside $\,\Lambda_p\,$
unless $\,R\,$ has type $\,A_1,\,A_2,\,B_2\,$ (cf. \cite{verma1})

For a given $\,\lambda\,$ and a given $\,p,\,$ it is possible (at least in 
principle) to compute effectively the weight multiplicities of 
$\,E(\lambda)\,$. Burgoyne~\cite{burg} carried out computer calculations
along these lines for small ranks and small primes $\,p$. The
underlying idea is to write down a square matrix over the integers  
(of size equal to the dimension 
of a weight space of $\,V_{\lambda}\,$). The number of elementary divisors
of the matrix divisible by $\,p\,$ counts the decrease in the
dimension of the weight space when we pass to $\,E(\lambda).\,$

In this work we use information on weights and their 
multiplicities given in \cite{buwil} (for Lie algebras of 
small rank) and in \cite{gise} (for Lie algebras of exceptional type).

\newpage

\part*{Chapter 3}
\part*{The Exceptional Modules}
\addcontentsline{toc}{section}{\protect\numberline{3}{The
Exceptional Modules}}
\setcounter{section}{3}
\setcounter{subsection}{0}
In this Chapter we give the definition of exceptional modules and
find a necessary condition for a module to be exceptional.

\subsection{Some Results on Centralizers}\label{centsection}

This section contains some results on centralizers of elements in the
algebraic group and its Lie algebra that will be used throughout the
main sections of this work.

If $\,x\,$ is an element of an algebraic group $\,G\,$ then
$\,C_G(x)\,=\,\{ g\in G\,/\,gx\,=\,xg\}\,$ is the {\bf centralizer} of 
$\,x\,$. This is a closed subgroup of $\,G\,$ \cite[8.2]{hum2}. 

We now discuss the relationship between centralizers in an algebraic group 
$\,G\,$ and in its Lie algebra $\,{\mathfrak g}\,=\,{\cal L}(G).\,$
Recall that $\,G\,$ acts on $\,{\mathfrak g}\,$ via the adjoint 
representation $\,{\rm Ad}\,:\,G\longrightarrow \GL(\mathfrak{g}).\,$ 
(cf. Section~\ref{tlaag}). Let $\,x\in G\,$ and $\,a\in{\mathfrak
g}.\,$ Define
\begin{gather*}
C_G(x)\,=\,\{ g\in G\,/\,x^{-1}gx\,=\,g\} \\
\mathfrak c_{\mathfrak g}(x)\,=\,\{ b\in {\mathfrak g}\,/\,{\rm
Ad}x\cdot b\,=\,b\}\\
C_G(a)\,=\,\{ g\in G\,/\,{\rm Ad}g\cdot a\,=\,a\}\\
\mathfrak c_{\mathfrak g}(a)\,=\,\{ b\in {\mathfrak g}\,/\,[b,a]\,=\,0\}.
\end{gather*}
The relationship between these subgroups and subalgebras is as follows. 
\begin{gather*}
{\cal L}(C_G(x))\subseteq \mathfrak c_{\mathfrak g}(x)\;\;\mbox{for all}\;\;
x\in G\;\;{\rm\cite[10.6]{hum2}}\,, \\
{\cal L}(C_G(a))\subseteq \mathfrak c_{\mathfrak g}(a)\;\;\mbox{for all}\;\;
a\in {\mathfrak g}\;\;{\rm\cite[9.1]{bor2}}. 
\end{gather*}
Equality holds in certain important special cases.  
\begin{lemma}{\rm \cite[III.Prop. 9.1]{bor2}}\label{semisim}
If $\,s\in G\,$ is semisimple,
then $\,{\cal L}(C_G(s))\,=\,\mathfrak c_{\mathfrak g}(s).\,$ 
If $\,a\in {\mathfrak g}\,$
is semisimple, then $\,{\cal L}(C_G(a))\,=\,\mathfrak c_{\mathfrak
g}(a)$. 
\end{lemma}

\begin{lemma}{\rm\cite[p.38]{slo}}\label{centra}
Let $\,a\in{\mathfrak g}\,$.
%(i) {\rm\cite[9.1]{bor2}} $\,{\cal L}(C_G(x))\,=\,{\mathfrak
%c}_{\mathfrak g}(x)\;\;\mbox{if}\;\;
%x\,$ is semisimple.(ii) {\rm\cite[p.38]{slo}}
If the characteristic of $\,K\,$ is
either $\,0\,$ or a very good prime for $\,G,\,$ then 
$\,{\cal L}(C_G(a))\,=\,{\mathfrak c}_{\mathfrak g}(a)\;$ for any 
$\,a\in{\mathfrak g}\,$.
\end{lemma}

Now $\,G\,$ acts on the nilpotent cone $\,{\cal N}\,$ of $\,\mathfrak g\,$ 
via the adjoint representation and $\,{\cal N}\,$ splits into finitely
many ($\Ad\,G$)-orbits (see \cite{baca1}, \cite{baca2}, \cite{pom1},
\cite{pom2}, \cite{holspa}).
%Let $\,{\cal Z}\,$ be an irreducible component of $\,{\cal N}.\,$ As 
%$\,{\cal N}\,$ is a cone, so is $\,{\cal Z}.\,$
%Let $\,\cal T\,$ be a torus of $\,\mathfrak g\,$ of dimension
%$\,\ell\,=\,{\rm rank}\,G.\,$ 
\begin{lemma}\label{help}
Let $\,{\cal Z}\,$ be an irreducible component of $\,{\cal N}\,$ 
and let $\,\cal T\,$ be a torus of $\,\mathfrak g\,$ of dimension
$\,\ell\,=\,{\rm rank}\,G.\,$ Then
$\,\dim\,{\cal Z}\,\leq\,\dim\,\mathfrak g\,-\,\ell$.
\end{lemma}\noindent
\begin{proof}
Suppose $\,\dim\,{\cal Z}\,>\,\dim\,\mathfrak g\,-\,\ell$. As 
$\,{\cal N}\,$ is a cone, so is $\,{\cal Z}.\,$ Hence, $\,0\,\in\,{\cal
Z}\cap{\cal T}\,$ and, as a consequence, $\,{\cal Z}\cap{\cal
T}\,\neq\,\emptyset\,$. Since 
$\,{\cal Z}\,$ and $\,\mathfrak t\,$ are irreducible varieties of 
$\,\mathfrak g,\,$ by the Affine Dimension Theorem~\cite[Prop. 7.1]{har}
we have that  $\,\dim\,{\cal Z}\cap\mathfrak t\,>\,0$. By
contradiction, the result follows.
\end{proof}
\begin{proposition}\label{othercases}
Let $\,x\,$ be a nilpotent element of $\,\mathfrak g\,$. Then 
\[
\dim\,C_G(x)\,\geq\,\ell\,.
\]
\end{proposition}\noindent
\begin{proof}
Suppose $\,\dim\,C_G(x)\,<\,\ell\,$. Then $\,\dim\,G\cdot x\,>\, 
\dim\,\mathfrak g\,-\,\ell\,$. But since $\,G\cdot x\subset {\cal
N},\,$ this implies  
$\,\dim\,{\cal N}\,>\,\dim\,\mathfrak g\,-\,\ell$, contradicting the
previous lemma. This proves the proposition.
\end{proof}
\begin{proposition}\label{semicases}
Let $\,a\,$ be a semisimple element of $\,\mathfrak g\,$. Then 
\[
\dim\,C_G(a)\,\geq\,\ell\,.
\]
\end{proposition}\noindent
\begin{proof}
Let $\,Z\,=\,\{ x\in\mathfrak g\,/\,x^{[p]}\,=\,t^{(p-1)}x\;\mbox{for
some}\;t\in K\}\,$ be the so called {\it cone} over the variety of
toral elements $\,\cal T\,$. Let $\,\mathfrak t\,$ be a torus of
dimension $\,\ell\,=\,{\rm rank}\,G\,$. The intersection $\,Z\cap
\mathfrak t\,$ consists of finitely many affine lines (as every toral
subalgebra has finitely many toral elements). Each irreducible
component $\,Z_i\,$ of $\,Z\,$ is homogeneous whence contains $\,0$. 
It follows that  $\,Z_i\cap \mathfrak t\,$ is non-empty. Hence, 
$\,\dim\,Z_i\leq\dim\,\mathfrak g\,-\,\ell\,+\,1$. Each component of
$\,\cal T\,$ is contained in one of the $\,Z_i$'s, but cannot be equal
to it as given a nonzero $\,y\in{\cal T}\cap Z_i\,$, the scalar
multiple $\,\lambda y\,$ is not toral, for $\,\lambda\not\in {\bf
F}_p\,$. However, $\,\lambda y\in Z_i\,$ as $\,Z_i\,$ is homogeneous. 
Therefore, each homogeneous component of $\,\cal T\,$ has dimension
$\,\leq\,\dim\,\mathfrak g\,-\,\ell\,$. Now proceed as in the proof 
of Proposition~\ref{othercases}.
\end{proof}

\subsection{The Exceptional Modules}\label{sec3.2}

%Let $\,K\,$ be an algebraically closed field of
%characteristic $\,p>0$, and 
Let $\,V\,$ be a finite dimensional vector space over $\,K$. 
%Let $\,G\,$ be a simply connected simple algebraic group over $\,K\,$ 
%and $\,{\mathfrak g}\,=\,{\cal L}(G)\,$ its (restricted) Lie algebra
%(cf. Section~\ref{tlaag}). 
If $\,G\,$ acts on $\,V\,$ via a rational representation
$\,\pi\,:\,G\,\longrightarrow\,GL(V),\,$ then $\,{\mathfrak g}\,$ acts
on $\,V$, via the differential 
$\,d\pi\,:\,{\mathfrak g}\,\longrightarrow\,{\mathfrak {gl}}\;(V)$.

For $\,v\in V\,$ and $\,x\in \mathfrak g$, put
%$\,\pi(g)\cdot v\,=\,g\cdot v\,$ for $\,g\in G\,$ and 
$\,d\pi(x)\cdot v\,=\,x\cdot v\,$.
Let $\,G\cdot v\,=\,\{\,g\cdot v\,/\,g\,\in\,G\,\}\,$
be the $\,G$-orbit of $\,v\,$ and  
$\,{\mathfrak g}\cdot v\,=\,\{\,x\cdot v\,/\,x\,
\in\,{\mathfrak g}\,\}.\,$ %be the $\,\mathfrak g$-orbit of $\,v\,$. 

\begin{definition}\label{isotropy}
%For any point $\,v\,\in\,V$, 
The {\bf isotropy (sub)group} (or
{\bf stabilizer}) of $\,v\in V\,$ in $\,G\,$ is $\,G_v\,=\,\{\,g\,\in\,G\,/\,
g\cdot v\,=\,v\,\}\,$ and the {\bf isotropy (Lie) subalgebra} of
$\,v\in V\,$ in $\,\mathfrak g\,$ is 
$\,{\mathfrak g}_v\,=\,\{\,x\,\in\,{\mathfrak g}\,/\,x\cdot v\,=\,0\,\}$. 
\end{definition}
\noindent
\begin{remark}\label{rem2}~{\bf (1)}~The stabilizer $\,{\mathfrak
g}_v\,=\,\{\,x\,\in\,{\mathfrak g}\,/\,
x\cdot v\,=\,0\,\}\,$ is a restricted Lie subalgebra of $\,{\mathfrak g}$, as 
$\,x^{[p]}\cdot v\,=\,x^p\cdot v$. Hence, by Lemma~\ref{ppp}, 
%Theorem~\ref{3.6}, 
%$\,{\mathfrak g}_v\,$ is algebraic and so 
either $\,{\mathfrak g}_v\,$ is a torus or $\,{\mathfrak g}_v\,$
contains a nonzero element $\,x\,$ such that $\,x^{[p]}\,=\,0.\,$
\end{remark}
By~\cite[Ex. 10.2]{hum2}, $\,{\cal L}(G_v)\,\subseteq\, 
{\mathfrak g}_v$.
For $\,x\,\in\,{\mathfrak g}$, let $\,V^x\,=\,\{\,v\,\in\,V\,/\,x\cdot
v\,=\,0\,\}\,$ and $\,x\cdot V\,=\,\{\,x\cdot v\,/\,v\in V\,\}$.

\begin{definition} \label{isgp}
For a rational action of an algebraic group $\,G\,$ on a vector space
$\,V$, we say that a subgroup $\,H\,\subset\,G\,$ is an {\bf isotropy
subgroup in general position} (or an ISGP for short) if $\,V\,$
contains a non-empty Zariski open subset $\,U\,$ whose points have
their isotropy subgroups conjugate to $\,H$. The points of $\,U\,$ are
called {\bf points in general position}.
\end{definition}

Isotropy subalgebras $\,{\mathfrak h}\,\subset\,{\mathfrak g}\,$ in general
position are defined similarly.

\begin{definition} \label{locfree}
A rational $\,G$-action $\,G\longrightarrow\GL(V)\,$ is said to be  
{\bf locally free} if the ISGP $\,H\,$ exists and equals $\,\{ e\}$.
\end{definition}
\begin{definition} \label{locfreega}
A linear $\,\mathfrak g$-action $\,\mathfrak
g\longrightarrow\mathfrak{gl}(V)\,$ is said to be {\bf locally free} if the
isotropy subalgebra in general position $\,\mathfrak h\,$ %exists and 
equals $\,\{ 0\}$.
\end{definition}

We are interested in studying $\,G$- and $\,\mathfrak g$-actions which are
not locally free. 
 
\subsubsection{Lie Algebra Actions in Positive Characteristic}

The following result is well-known and applies to all
characteristics. We include the proof for reader's convenience, as
references are hard to find.
\begin{proposition}\label{rem1}~~If there exists $\,0\neq\,v\in V\,$ such 
that $\,{\mathfrak g}_v\,=\,\{ 0\}$, then there exists a Zariski open subset 
$\,W\subseteq V\,$ such that $\,{\mathfrak g}_w\,=\,0\,$ for all
$\,w\in W$. 
\end{proposition}\noindent
\begin{proof}
Consider the morphism $\,\psi\,:\,{\mathfrak g}\times V\longrightarrow
V\times V\,$ given by $\,(x,\,v)\longmapsto (x\cdot v,\,v)\,$.
Let $\,(v_1,\,v_2)\in V\times V\,$. Then
\begin{align*}
\psi^{-1}(v_1,\,v_2)\,&=\,\{\,(x,\,v)\in {\mathfrak g}\times V\,/\,
(x\cdot v,\,v)\,=\,(v_1,\,v_2)\,\}\\
& =\,\{\,(x,\,v_2)\in {\mathfrak g}\times V\,/\,x\cdot v_2\,=\,v_1\,\}.
\end{align*}
Let $\,d\,$ denote the minimal dimension of the fibres of $\,\psi.\,$
Suppose $\,\psi^{-1}(v_1,\,v_2)\,\neq\,\emptyset\,$ and
let $\,(x,\,v_2),\;(y,\,v_2)\,\in\,\psi^{-1}(v_1,\,v_2)\,$. Then
$\,x\,-\,y\,\in\,{\mathfrak g}_{v_2}\,$. Hence,
$\,\psi^{-1}(v_1,\,v_2)\,=\,(\,{\mathfrak g}_{v_2}\,+\,y,\,v_2),\,$ where
$\,y\cdot v_2\,=\,v_1\,$. Thus, either $\,\dim\,\psi^{-1}(v_1,\,v_2)\,=\,\dim
{\mathfrak g}_{v_2}\,$ or $\,\psi^{-1}(v_1,\,v_2)\,=\,\emptyset\,$.

Let $\,U\,=\,\{\,(x,\,v)\in\mathfrak g\times V\,/\,%\dim {\mathfrak g}_v\,=
\,\dim\,\psi^{-1}(\psi(x,\,v))\,=\,d\,\}.\,$ By the theorem on  
dimension of fibres~\cite[4.1]{hum2}, $\,U\,$ is a non-empty Zariski
open subset of \mbox{$\,{\mathfrak g}\times V$}. Thus, $\,\dim {\mathfrak
g}_v\,=\,\dim\,\psi^{-1}(\psi(x,\,v))\,$
is the minimal possible for all $\,v\in V.\,$
As the composition $\,{\mathfrak g}\times
V{\stackrel{\psi}{\longrightarrow}}V\times V
{\stackrel{{\rm pr}_2}{\longrightarrow}}V \,$ is surjective, we have 
$\,d=0\,$ (by our assumption).
Now the projection $\,\pi_2\,:\,{\mathfrak g}\times V\longrightarrow V\,$
is surjective, whence dominant (see~\cite[4.1]{hum2}). 
Therefore, $\,\pi_2(U)\,$ contains a Zariski open
subset $\,W\,$ of $\,V.\,$ Thus, all points of $\,W\,$ have trivial
stabilizer, proving the proposition.
\end{proof}

\begin{proposition} \label{3.1.1}
If there is a non-empty Zariski open subset $\,W\,\subset\,V\,$ such that
$\,{\mathfrak g}_w\,\neq\,0,\;\;\forall\,w\in W$, then 
$\,{\mathfrak g}_v\,\neq\,0$, for all $\,v\in V$.
\end{proposition}\noindent
\begin{proof} Consider $\,{\cal C}\,=\,\{\,(x,\,v)\in {\mathfrak
g}\times V\,/\,x\cdot v\,=\,0\,\}$, the commuting variety of 
$\,{\mathfrak g}\times V$. Note that $\,{\cal C}\,$ is ``bihomogeneous'', 
i.e., $\,\forall\,\lambda,\,\mu\in K-\{ 0\},\;(\lambda\,x,\,\mu\,v)\,
\in\,{\cal C}$, whenever $\,(x,\,v)\,
\in\,{\cal C}$. Let $\,\bar{\cal C}\,$ be the closed subset of 
$\,\mathbb P({\mathfrak g})\,\times\,\mathbb P(V)\,$ corresponding to 
$\,{\cal C}$. As $\,\mathbb P({\mathfrak g})\,$ is complete, the 
projection 
$\,{\rm pr}_2\,:\,\mathbb P({\mathfrak g})\,\times\,\mathbb P(V)\,\longrightarrow\,
\mathbb P(V)\,$ is a closed map (see \cite[\S 6]{hum2}) so that
$\,{\rm pr}_2 (\bar{\cal C})\,$ is a closed set. 
Hence the image of $\,{\cal C}\,$ under the projection map  
$\,{\mathfrak g}\times V\longrightarrow V\,$ is also closed 
(as the closed subsets of $\,\mathbb P(V)\,$ are the images of the
Zariski closed subsets of $\,V\,$ corresponding to 
collections of homogeneous polynomials in $\,K[V]$). 
Let $\,\varepsilon\,:\,{\cal C}\longrightarrow V\,$ denote the
restriction of the projection map $\,{\mathfrak g}\times V\longrightarrow V\,$
to $\,{\cal C}.\,$ For each $\,v\in V$, the fibre 
$\,\varepsilon^{-1}(v)\,=\,\{\,(x,\,v)\in {\cal C}\,/\,x\in {\mathfrak g}_v\,\}
\,$ is isomorphic to a vector space and so is irreducible as an algebraic 
variety. Moreover, $\,\dim\,\varepsilon^{-1}(v)\,=\,\dim\,{\mathfrak g}_v$.

Now, for each $\,n\in\mathbb N$, consider the set
\[
{\cal C}_n\,=\,\{\,(x,\,v)\in {\cal C}\,/\,\dim\, {\mathfrak g}_v\,
\geq\,n\,\}\,=\,\{\,(x,\,v)\in {\cal C}\,/\,\dim\, 
\varepsilon^{-1}(\varepsilon(x,\,v))\,\geq\,n\,\}. 
\]
By \cite[AG.10.3]{bor2}, each $\,{\cal C}_n\,$ is a closed set.
% for any $\,n\in \mathbb N$. 
Let $\,s\,=\,\min\,\{\,\dim {\mathfrak g}_w\,/\,w\in W\,\}$. 
%As $\,\varepsilon\,$ is surjective, $\,\varepsilon ({\cal
%C}_s)\,\supseteq\,W$.  %As $\,\varepsilon\,$ is a
%bihomogeneous map, the closed set $\,{\cal C}_s\,$ is determined 
%by bihomogeneous polynomials and so
There exists a closed variety $\,\bar{{\cal C}_s}\,$ in $\,\bar{\cal
C}\,$ that corresponds to $\,{\cal C}_s$.
Therefore, $\,{\rm pr}_2(\bar{{\cal C}_s})\,$ is closed in $\,\mathbb P(V)$, 
implying that $\,\varepsilon ({\cal C}_s)\,$ is also closed in $\,V$. 
Now as $\,\varepsilon\,$ is surjective, $\,\varepsilon ({\cal
C}_s)\,\supseteq\,W$. Finally, as $\,W\,$ is dense in $\,V$, this implies 
$\,\varepsilon ({\cal C}_s)\,=\,V$, proving the proposition.
\end{proof}

\begin{definition} \label{nonlocfree}
A $\,{\mathfrak g}$-module $\,V\,$ is said to be {\bf non-free} if
$\,{\mathfrak g}_v\,\neq\,0\,$ for all $\,v\,\in\,V$. 
\end{definition}
\noindent
\begin{remark}~{\bf (2)} Comparing the two propositions above, we can
say that if $\,V\,$ is not a non-free $\,{\mathfrak g}$-module, then
it is locally free.
\end{remark}

By Lemma~\ref{centretoral}, the centre $\,{\mathfrak z(\mathfrak g)}\,$ of
the restricted Lie algebra $\,{\mathfrak g}\,=\,{\cal L}(G)\,$ is a
toral subalgebra. Hence, by Proposition~\ref{properties}(2), 
$\,{\mathfrak z(\mathfrak g)}\,$ acts diagonalisably on $\,V\,$. 
Observe that sometimes $\,{\mathfrak z(\mathfrak g)}\,$ 
%of the Lie algebra $\,{\mathfrak g}\,$ 
is nonzero but acts trivially on $\,V,\,$ implying that 
$\,{\mathfrak g}_v\,\supseteq\,{\mathfrak z
(\mathfrak g)}\,$ for all $\,v\in V$. This happens, for instance, when
$\,\mathfrak g\,=\,\mathfrak{sl}(\ell+1, K),\,$ $\,p|(\ell +1)\,$
and $\,V\,=\,\mathfrak g\,$ is the adjoint $\,\mathfrak g$-module. 
The following statement is true.

\begin{proposition} \label{3.1.a}
If there is a non-empty Zariski open subset $\,W\,\subset\,V\,$ such that
$\,{\mathfrak g}_w\,$ has a non-central element for all $\,w\in W$, then 
$\,{\mathfrak g}_v\,$ has a non-central element for all $\,v\in V$.
\end{proposition} \noindent
\begin{proof}
The proof repeats almost verbatim the proof of Proposition~\ref{3.1.1}.
%Similar to the proof of Proposition~\ref{3.1.1}. Note that in the
%proof of this case $\,s\,>\,\dim\,\mathfrak{z}(\mathfrak g)\,$.
\end{proof}

Our next definition is crucial for the rest of this work.
\begin{definition}\label{excpmod}
A $\,{\mathfrak g}$-module $\,V\,$ is called {\bf exceptional} if
for each $\,v\in V\,$ the isotropy subalgebra $\,{\mathfrak g}_v\,$ contains 
a non-central element (that is, $\,{\mathfrak g}_v\,\not\subseteq\,
{\mathfrak z(\mathfrak g)}\,$).
\end{definition}

\begin{example}\label{adjoint} {\rm
Clearly, if $\,\mathfrak g\,$ is non-abelian, then the trivial
$\,1$-dimensional $\,\mathfrak g$-module is exceptional.
For any non-abelian Lie algebra $\,{\mathfrak g},\,$ the adjoint
module is always
exceptional, as $\,{\mathfrak z(\mathfrak g)}\,$ is a proper subset of
$\,{\mathfrak g}_x\,$ for all $\,x\in{\mathfrak g}\,$ (note that 
$\,[x,\,x]\,=\,0\,$).}
\end{example}
For the centreless Lie algebras the notions of {\it non-free} and
{\it exceptional} modules coincide. 
It follows from the definitions that all exceptional modules are
non-free. The converse statement is not true, as the following lemma shows.
(The proof of this lemma will be given in Section~\ref{lact}.1.) 
\begin{lemma}\label{nonexcepnonfree}
Let $\,\mathfrak g\,=\,\mathfrak{sl}(np,K),\,$ where 
$\,p>2\,$ and $\,np\geq 4.\,$ Then the Steinberg module 
$\,V\,=\,E((p-1)\rho)\,$ is not an exceptional module, but it is non-free.
\end{lemma}\noindent

It follows from Proposition~\ref{3.1.a} that if there is a nonzero
$\,v\in V\,$ such that $\,{\mathfrak
g}_v\,\subset\,{\mathfrak z(\mathfrak g)}\,$ then $\,V\,$ is not
an exceptional module. In particular, locally free modules cannot be
exceptional.

\begin{lemma}\label{notbadhw}
If $\,V\,$ is an exceptional $\,\mathfrak g$-module, then so is each of its 
composition factors.
\end{lemma}\noindent
\begin{proof}
Let $\,V\,=\,V_0\,\supset\,V_1\,\supset\,V_2\,\supset\,\cdots\,\supset\, 
V_k\,\supset \,\{ 0\}\,$
be a composition series of $\,V.\,$ Take $\,0\neq\bar{v}\in\,V_i/V_{i+1}\,$
and let $\,v\in V_i\,$ be its preimage. Then $\,\mathfrak
g_v\,\subset\,\mathfrak g_{\bar{v}}\,$. The result follows. 
%$\,0\neq\bar{v}\in\,V_i/V_{i+1}\,$. Hence, the result follows. 
\end{proof}
\medskip

\subsection{A Necessary Condition}\label{necess}

Let $\,\pi\,:\,G\longrightarrow\GL(V)\,$ be a non-trivial faithful
rational representation of $\,G\,$ and 
let $\,d\pi\,:\,{\mathfrak g}\,\longrightarrow\,{\mathfrak{gl}}\;(V)\,$
be the differential of $\,\pi\,$ at $\,e\in G$.
Suppose that $\,\ker d\pi\,\subseteq\,{\mathfrak z}
({\mathfrak g})\,$ and assume that $\,V\,$ is an exceptional 
$\,{\mathfrak g}$-module.

Under these hypotheses, we want to study highest weights of $\,V.\,$ 
(We do not assume that $\,V\,$ is irreducible.)
Note that if $\,V\,$ has no $\,1$-dimensional $\,\mathfrak g$-submodules, then 
$\,\dim {\mathfrak g}_v\,<\,\dim {\mathfrak g}\,$ for each nonzero $\,v\in V$.

First we aim to find an upper bound for $\,\dim\,V.\,$

\begin{proposition} \label{prev}
Given $\,x\,\in\,{\mathfrak g},\,$ define $\,\varphi_x\,:\,G\,\times\,V^x\,
\longrightarrow\,V\,$ by setting $\,\varphi_x(g,v)\,=\,g\cdot v.\,$ 
Then for any $\,v\in \Im\varphi_x,\;\dim \varphi_x^{-1}(v) \,\geq\, \dim
C_G(x)$, where $\,C_G(x)\,=\,\{\,g\in G\,/\,(\Ad\,g)(x)\,=\,x\,\}$.
\end{proposition}\noindent
\begin{proof} 
First note that given $\,v\in \Im\varphi_x$, there exist
$\,h\,\in\,G\,$ and $\,v_1\,\in\,V^x\,$ such that $\,v\,=\,h\cdot
v_1\,$ (we set $\,h=e\,$ for $\,v\in V^x)$. 
By definition, %for $\,v\in V$,
\[
\varphi_x^{-1}(v) \,=\,\{\,(g,\,w)\,\in\,G\times V^x\,/\,g\cdot
w\,=\,v\,\}\,.
\]
Let $\,{\rm pr}_2\,:\,\varphi_x^{-1}(v)\,\longrightarrow\,V^x\,$ be the
projection map. Then its image is $\,G\cdot v\,\cap\,V^x$.
For $\,w\in\, {\rm pr}_2(\varphi_x^{-1}(v))$, if
$\,(g,\,w),\;(h,\,w)\,\in\,{\rm pr}_2^{-1}(w)\,\subset\,\varphi_x^{-1}(v)$,
then $\,g\cdot w\,=\,v\,=\,h\cdot w\,$ which implies $\,h\,\in\,G_v\,g$.
Hence, for each $\,w\in\; {\rm pr}_2(\varphi_x^{-1}(v))$, there exists
$\,g\,\in\,G\,$ such that $\,g\cdot w\,=\,v\,$ and
$\,(\,G_v\,g,\,w)\,\subset\,\varphi_x^{-1}(v)$, so that 
$\,\dim\, {\rm pr}_2^{-1}(w)\,=\,\dim\, 
\{(x\,g,\,w)\,/\,x\in G_v\}\,=\,\dim\, G_v$. Hence all fibres of
$\,{\rm pr}_2|_{\varphi_x^{-1}(v)}\,$ have the same dimension
equal to $\,\dim\, G_v$. Therefore,
\begin{equation} \label{var1}
\dim\, \varphi_x^{-1}(v)\,=\,\dim\, G_v\,+\,\dim\, (G\cdot v\,\cap\,V^x).
\end{equation}
Now $\,(G\cdot v\,\cap\,V^x)\,\supseteq\,C_G(x)\cdot v_1\,$ yielding
\begin{equation}\label{var2}
\begin{split}
\dim\, (G\cdot v\,\cap\,V^x)\, & \geq\,\dim\, (C_G(x)\cdot v_1) \\
& =\, \dim\,
C_G(x)\,-\,\dim\, (G_{v_1}\,\cap\,C_G(x)).
\end{split}
\end{equation}
Combining~(\ref{var1}) and~(\ref{var2}) we get
\begin{eqnarray} 
\dim\, \varphi_x^{-1}(v) & \geq & \dim\, G_v\,+\,\dim\, C_G(x)\,-\,
\dim\, (G_{v_1}\,\cap\,C_G(x)) \nonumber\\
& \geq & \dim\, G_v\,-\,\dim\, G_{v_1}\,+\,\dim\, C_G(x)\,=\,\dim\,
C_G(x), \nonumber
\end{eqnarray}
as $\,G_v\,\cong\,G_{v_1}$. This proves the proposition.
\end{proof}
\medskip

\begin{proposition} \label{dimxvxg}
Given an exceptional $\,{\mathfrak g}$-module $\,V$,
there exists a non-central $\,x\,\in\,{\mathfrak g}\,$ such that 
either $\,x^{[p]}\,=\,x\,$ or $\,x^{[p]}\,=\,0\,$ and 
\begin{equation}\label{eqa}
\dim x\,V\,\leq\,\dim\,G\,-\,\dim\,C_G(x)\,\leq\,|R|\,.
\end{equation} 
%(b) Let $\,x\,$ be as in (a). 
%If $\,p\,$ is not special for $\,G$, then $\,\dim\,C_G(x)\,\geq\,\ell\,$ and 
%so $\,\dim x\,V\,\leq\,|R|$.
\end{proposition}\noindent
\begin{proof}
Let $\,v\,\in\,V$. By Remark~\ref{rem2}(1), 
the stabilizer $\,{\mathfrak g}_v\,$ is either toral or contains a
non-zero element $\,x\,$ such that $\,x^{[p]}\,=\,0\,$.

Let $\,{\cal T}_1^*\,=\,\{\,x\in\,{\mathfrak g}\setminus\{ 0\}\,/\,x^{[p]}\,
=\,x\,\}\,$ and $\,{\cal N}_1^*\,
=\,\{\,x\in\, {\mathfrak g}\setminus\{ 0\}\,/\,x^{[p]}\,=\,0\,\}$.

By Lemma~\ref{centretoral}, $\,{\mathfrak z}(\mathfrak g)\,$ is toral,
hence consists of semisimple elements. 
If $\,{\mathfrak g}_v\,$ contains a nilpotent element $\,x$,  
then $\,x\not\in{\mathfrak z}(\mathfrak g)\,$. If $\,{\mathfrak
g}_v\,$ contains no nilpotent elements, then the $\,p$-mapping
$\,[p]\,$ is nonsingular on $\,{\mathfrak g}_v\,$, whence
$\,{\mathfrak g}_v\,$ is toral. Then by Theorem~\ref{3.6} 
$\,{\mathfrak g}_v\,$ is spanned by toral elements. Thus, $\,{\mathfrak g}_v\,$
contains at least one non-central toral element, as $\,V\,$ is
exceptional.  %otherwise, all toral elements spanning $\,{\mathfrak
	      %g}_v\,$ would be central, and so $\,{\mathfrak g}_v\,$
	      %contradicting the fact that V is exceptional.
Therefore, for each $\,v\in V,\,$ there is a non-central
$\,x\in \,{\cal N}_1^*\,\cup\,{\cal T}_1^*\,$ such that $\,v\in V^x$. Hence
\[
V\,=\,\bigcup_{x\in ({\cal N}_1^*\cup {\cal T}_1^*)\setminus\mathfrak
z(\mathfrak g)}\,V^x.
\]
If $\,x^{[p]}\,=\,x$, then there exists $\,y\in G\,$ such
that $\,(\Ad\,y)(x)\,\in\,{\mathfrak t}\,$ (Proposition~\ref{anytoral}).
% i.e., any toral element of
%$\,{\mathfrak g}\,$ is ($\Ad\,G$)-conjugate to an element in $\,{\mathfrak
%t}\,$, by Proposition~\ref{anytoral}. 
By Lemma~\ref{toralbase}(2), $\,{\mathfrak t}\cap 
{\cal T}_1^*\,$ is finite. Combining this with Theorem~\ref{badhs}, we
obtain that %By~\ref{toralbase}(2) and~\ref{finitenilp}, 
there are finitely many nilpotent and toral conjugacy classes in  
$\,{\mathfrak g}$. Let $\,n_1,\,n_2,\,\ldots\,,n_s\,$ 
(resp. $\,t_1,\,t_2,\,\ldots\,,t_{\ell}\,$) be 
representatives of the non-central conjugacy classes in $\,{\cal
N}_1^*\,$ (resp.,
$\,{\cal T}_1^*$). As $\,{\mathfrak g}_v\,$ contains a non-central
element from $\,{\cal N}_1^*\cup {\cal T}_1^*$, we have 
$\,v\,\in\,\left(\bigcup_{j=1}^{s}\;G\cdot V^{n_j}\right)\,\cup\,
\left(\bigcup_{i=1}^{\ell}\;\,G\cdot V^{t_i}\right)$. Therefore,
\begin{equation} \label{uniongti}
V\,=\,\left(\bigcup_{j=1}^{s}\;G\cdot V^{n_j}\right)\,\cup\,
\left(\bigcup_{i=1}^{\ell}\;\,G\cdot V^{t_i}\right)
\end{equation}

For any $\,x\,\in\,{\mathfrak g}$, the set $\,G\cdot V^{x}\,$ is the image of
the morphism
\[
\varphi_x\,:\,G\,\times\,V^x\,\longrightarrow\,V,\qquad\;(g,v)\,
\longmapsto\,g\cdot v\,.
\]
In particular, each $\,G\cdot V^{x}\,$ is constructible 
\cite[4.4]{hum2}. This implies that $\,G\cdot V^{x}\,$
contains an open dense subset of $\,\overline{G\cdot V^{x}}\,$ 
\cite[Cor.10.2]{bor2}. 
Hence at least one of the subsets $\,G\cdot V^{n_j},\;G\cdot V^{t_i}\,$ is 
Zariski dense in $\,V\,$ (note that the number of the subsets in the
decomposition~\eqref{uniongti} is finite). 
%(otherwise all $\,\overline{G\,V^{n_j}},\;
%\overline{G\,V^{t_i}}\,$ would be proper closed subsets of  
%$\,V,\,$ contradicting the decomposition~(\ref{uniongti})).
It follows that $\,V\,$ contains a non-empty Zariski dense subset of the form 
$\,G\cdot V^{x}.\,$ So $\,\varphi_x\,$ is dominant
for some non-central $\,x\in{\cal N}_1^*\cup {\cal T}_1^*$.

Now, using \cite[3.1]{hum2} and the theorem on dimension of fibres of
a morphism (see \cite[Theorem AG.10(ii)]{bor2}), we have
\begin{equation} \label{first}
\dim {\mathfrak g}\,+\,\dim V^x\,-\,\min_{v\in V}\;\,\dim\,  
\varphi^{-1}_x(v) \,=  \, \dim V.
\end{equation}
By Proposition~\ref{prev}, for any $\,v\in\, \Im\varphi_x,
\;\dim\, \varphi_x^{-1}(v)\,\geq\,\dim C_G(x)$, forcing
\begin{equation} \label{second}
\dim {\mathfrak g}\,+\,\dim V^x\,-\,\dim C_G(x)\,\geq\, \dim V.
\end{equation}
As $\,\dim{\mathfrak g}=\dim G\,$ and $\,\dim x\cdot V=\dim V\ -\
\dim V^x,\,$ we get $\,\dim x\cdot V\,\leq\,\dim\,G\,-\,\dim\,C_G(x)\,$.

Now if $\,x\,$ is toral, then it is semisimple by
Proposition~\ref{properties}(1). In this case Proposition~\ref{semicases}
applies. If $\,x\,$ is nilpotent we
apply Proposition~\ref{othercases}. Thus, in both cases the result follows.
\end{proof}

%\pagebreak

From now on we assume that $\,G\,$ is a simply connected simple
algebraic group over $\,K\,$. %and that $\,p\,$ is not special for $\,G.\,$ 
Fix a maximal torus $\,T\,$ of $\,G\,$ and 
a basis $\,\Delta\,$ of simple roots in the root system $\,R\,$ of
$\,G\,$ (with respect to $\,T\,$). Let $\,R^+\,$ denote the system of positive
roots with respect to $\,\Delta\,$ and 
let $\,B\,$ be the Borel subgroup
of $\,G\,$ generated by $\,T\,$ and the $\,1$-parameter unipotent
subgroups $\,U_{\alpha}\,=\,\{\,x_{\alpha}(t)\,/\,t\in K\,\}\,$,
where $\,\alpha\in R^+\,$. 
%containing $\,T\,$ corresponding to $\,R^+$.
Let $\,\tilde{\alpha}\,$ denote the maximal root in $\,R^+.\,$
Decompose $\,\mathfrak g\,=\,{\cal L}(G)\,$ into
root spaces with respect to $\,T\,$ giving a Cartan decomposition
\[
\mathfrak g\,=\,\mathfrak t\,\bigoplus\,\sum_{\alpha\in R}\,K\,e_{\alpha}
\]
where $\,\mathfrak t\,=\,\Lie (T)\,$ (cf. Section\ref{gennot}).
By Section~\ref{structliealg},
$\,\mathfrak g\,=\,({\rm Ad}\,G)\cdot\mathfrak b,\,$
where $\,\mathfrak b\,=\,{\cal L}(B).\,$ %with $\,B\,$ a Borel subgroup of 
%$\,G\,$ containing $\,T.\,$ e.g.~\cite[p. 87]{stfa}).
Put $\,{\cal E}\,=\,(\Ad G)\cdot e_{\tilde{\alpha}}\cup\{0\}.\,$
It is well-known (cf.~\cite{kraft},~\cite{pre4}) that $\,{\cal E}\,$
is a Zariski closed, conical subset of $\,\mathfrak g.\,$

The following result is well-known in the characteristic zero
case. It was generalized by A. Premet in~\cite{pre4} to the case
$\,p>0$. Here we reproduce the proof for the reader's convenience.
\begin{lemma}{\rm\cite[Lemma 2.3]{pre4}}\label{varep}
Suppose $\,p\,$ is non-special for $\,G$. 
Let $\,Z\,$ be a Zariski closed, conical, (${\rm Ad}\,G$)-invariant
subset of $\,\mathfrak g.\,$ Then either $\,Z\subseteq{\mathfrak z}
(\mathfrak g)\,$ or $\,{\cal E}\subseteq Z\,$.
\end{lemma}\noindent
\begin{proof}
If $\,x\in\mathfrak b\,$ then $\,x\,=\,t\,+\,\sum_{\alpha\in R^+}\,
n_{\alpha}\,e_{\alpha}\,$ where $\,t\in \mathfrak t\,$ and 
$\,n_{\alpha}\in K$. Define $\,{\rm Supp}(x)\,=\,\{\alpha\in R^+\,/\,
n_{\alpha}\neq 0\}$. 
As $\,\mathfrak g\,=\,({\rm Ad}\,G)\cdot\mathfrak b\,$ and $\,Z\,$ is 
$\,({\rm Ad}\,G)$-stable, $\,Z\cap \mathfrak b
\neq\{ 0\}$. If $\,Z\cap \mathfrak b \subseteq \mathfrak t$, then 
$\,({\rm Ad}\,x_{\alpha}(t))\cdot z\,=\,z\,$ for all $\,\alpha\in R^+,\,
t\in K,\,$ and $\,z\in Z\cap \mathfrak b $. It follows that 
$\,({\rm d}\alpha)(z)\,=\,0\,$ for every $\,\alpha\in R$. In other words,
$\,Z\cap \mathfrak b \subseteq\mathfrak z(\mathfrak g)\,$. As
$\,Z\,=\,({\rm Ad}\,G)\cdot Z\cap \mathfrak b\,$ we get $\,Z\subseteq 
\mathfrak z(\mathfrak g)$.

Now suppose that $\,{\rm Supp}(z)\,\neq\, \emptyset\,$ for some 
$\,z\in Z\cap \mathfrak b$. Let $\,X_*(T)\,$ denote the lattice of one 
parameter subgroups in $\,T\,$. If $\,\lambda(t)\in X_*(T)\,$ and
$\,\gamma\in R$, then $\,\gamma(\lambda(t))\,=\,t^{m(\gamma)}\,$ for some 
$\,m(\gamma)\in \mathbb Z.\,$ A standard argument shows that there
exists $\,\lambda(t)\,\in\,X_*(T)\,$ such that $\,m(\gamma)\,>\,0\,$
for every $\,\gamma\in R^+\,$ and the numbers
$\,m(\gamma),\,\gamma\in R^+\,$ are pairwise distinct. Let 
$\,r\,=\,\max\,\{m(\alpha)\,/\,\alpha\,\in\,{\rm Supp}(z)\}.\,$  As 
$\,Z\,$ is conical, 
\[
\{\,t^r\cdot ({\rm Ad}\,\lambda(t^{-1}))\cdot z\,/\,t\in
K^*\,\}\,\subseteq \,Z\,.
\]
By construction, there are $\,z_0,\,z_1,\ldots,\,z_r\,\in\,\mathfrak b\,$ 
such that $\,z_0\in Ke_{\delta}\setminus\{ 0\}\,$ for some
$\,\delta\in R^+\,$ and
\[
t^r\cdot ({\rm Ad}\,\lambda(t^{-1}))\cdot
z\,=\,z_0\,+\,tz_1\,+\,\cdots\,+\,t^r z_r\,.
\]
As $\,Z\,$ is Zariski closed we must have $\,z_0\in Z.\,$ If
$\,\delta\,$ is long, there is $\,w\in W\,$ such that
$\,w\delta\,=\,\tilde{\alpha}\,$. It follows that $\,({\rm
Ad}\,N_G(T))\cdot z_0\,$ contains $\,e_{\tilde{\alpha}}\,$. Thus in this
case $\,{\cal E}\subseteq Z.\,$ If $\,\delta\,$ is short, there is  
$\,w\in W\,$ such that $\,w\delta\,=\,{\alpha}_0\,$. So we may assume
that $\,\delta\,=\,{\alpha}_0\,$. By \cite{hogew}, 
$ \,\tilde{\alpha}\,\neq \,{\alpha}_0\,$ implies that $\,\mathfrak
g\,$ is simple. Hence, there is $\,\gamma\in R^+\,$ such that 
$\,[e_{\gamma},\,e_{\alpha_0}]\,\neq\,0$. Obviously, all roots in
$\,\alpha_0\,+\,\mathbb N\gamma\,$ are long. Let
$\,q\,=\,\max\,\{\,i\in\mathbb N\,/\,\alpha_0\,+\,i\gamma\in R\,\}.\,$
Applying the same argument as above to the subset 
\[
\{\,t^q\cdot ({\rm Ad}\,x_{\gamma}(t^{-1}))\cdot z\,/\,t\in
K^*\,\}\,
\]
we get $\,Z\cap Ke_{\alpha_0+q\gamma}\neq\,\{ 0\}.\,$ Since 
$\,\alpha_0+q\gamma\,$ is long, we are done.
\end{proof}
\begin{corollary}\label{maxcent}
Let $\,V\,$ be any rational $\,G$-module and let $\,x\in{\mathfrak
g}$. The maximum value of $\,\dim\,V^x,\;x\not\in 
{\mathfrak z(\mathfrak g)}$, is $\,\dim\,V^{e_{\tilde{\alpha}}}$. 
\end{corollary}\noindent
\begin{proof}
Given $\,d\in\mathbb Z^+,\,$ let 
$\,X_d\,=\,\{\,x\in\mathfrak g\,/\,\dim\,x\cdot V\,\leq\,d\}.\,$
It is clear that $\,X_d\,$ is ($\Ad\,G$)-invariant
and a conical Zariski closed subset of $\,\mathfrak g$. 
By Lemma~\ref{varep}, either $\,X_d\subseteq {\mathfrak z}({\mathfrak
g})\,$ or $\,{\cal E}\subseteq X_d\,$. Now take $\,d_1\,=\,\min\,\{
d\,/\, X_d\not\subseteq{\mathfrak z}({\mathfrak g})\}\,$. 
%% Such $\,d_1\,$ exists, since for $\,k=\dim V\,$ all $\,\dim\,x\cdot
%% V\,\leq\,k\,$.
Then $\,{\cal E}\subseteq X_{d_1}\,$ and the result follows.
\end{proof}
\vspace{2ex}\noindent
\begin{remark} \label{after}
It follows from Proposition~\ref{dimxvxg}(a) and this last Corollary
that 
\[
\dim\,e_{\tilde{\alpha}}\cdot V\,\leq\,|R|.
\]
\end{remark}

%\newpage
%\include{allcase}

\subsubsection{An Upper Bound for $\,\dim\,V\,$}

In this section we use results proved before to produce an upper
bound for the dimension of the exceptional modules.

%Let $\,V\,$ be an exceptional irreducible $\,\mathfrak g$-module.

Recall that $\,G\,$ is a simply connected simple algebraic group and 
$\,\mathfrak g\,=\,{\cal L}(G)\,$ (cf. Section~\ref{introback}).
%$\cong\,\mathfrak g_{\mathbb Z} 
%\otimes_{\mathbb Z} K\,$, where $\,\mathfrak g_{\mathbb Z}\,$ is a  
%$\,\mathbb Z$-span of the Chevalley basis
%$\,\{\,h_{\alpha},\,\alpha\in\Delta;\;e_{\alpha},\,\alpha
%\in R\,\}\,$ of $\,\mathfrak g_{\mathbb C}\,$. We also assume that 
%$\,e_{\alpha}\,=\,e_{\alpha}\otimes 1\,$ and put $\,h_{\alpha}\,= 
%\,h_{\alpha}\otimes 1\,$ for the basis of $\,\mathfrak g$.
%The $\,p$th power map in $\,\mathfrak g\,$ is invariant under the
%adjoint action of $\,G\,$ and has the property 
%$\,e_{\alpha}^{[p]}\,=\,0,\;h_{\alpha}^{[p]}\,=\,h_{\alpha},\,$ for all
%$\,\alpha\in R\,$ (cf. Section~\ref{introback}).
Let $\,\tilde{\alpha}\,$ be the highest root for $\,{\mathfrak g},\,$
and $\,E=e_{\tilde{\alpha}}\,\in\,{\mathfrak g}\,$ denote a highest
root vector. Then $\,E^{[p]}= 0.\,$  
%By Jacobson-Morozov's Theorem ???\cite[III, 4.3]{spst}, 
There exist $\,F=e_{-\tilde{\alpha}}\in{\mathfrak g}_{-\tilde{\alpha}}\,$ 
and $\,H=h_{\tilde{\alpha}}\in{\mathfrak h}\,$ 
such that $\,(E,H,F)\,$ form a standard basis of an 
$\,{\mathfrak{sl}}(2)$-triple $\,{\mathfrak s}_{\tilde{\alpha}}\,$ in 
$\,\mathfrak g$.  
  
Let $\,V\,$ be a finite dimensional exceptional $\,G$-module.
As $\,H\,=\,[E,\,F],\,$ we have $\,H\,V\,\subseteq\,E\,V\,+\,F\,V\,$ implying
\[
\dim\,H\,V\,\leq\,\dim\,E\,V\,+\,\dim\,F\,V\,=\,2\,\dim\,E\,V\,.
\]
Thus $\,\displaystyle
\frac{1}{2}\,\dim\,H\,V\,\leq\,\dim\,E\,V\,\leq\,|R|$, by
Remark~\ref{after}. % Corollary~\ref{maxcent}.

Now $\,V\,=\,\displaystyle\bigoplus_{\mu\in{\cal X}(V)}\,V_{\mu}$, 
where the $\,V_{\mu}$'s are the weight spaces of $\,V\,$ with respect to $\,T$.
Recall from Section~\ref{wamv}
that the differential of a weight $\,\mu\in X\,$ is a linear function
on $\,\mathfrak t\,$, also denoted by $\,\mu\,$. 
If $\,H\in\mathfrak t,\,$ then  
$\,H\cdot v\,=\,\mu(H)\,v,\,$ for all $\,v\in V_{\mu}\,$. Therefore, 
$\,H|_{V_{\mu}}\,=\,\mu (H)\,Id_{m_{\mu}}$.

Now the action of $\,H\,$ on $\,V\,$ induced by the differential
$\,{\rm d}\pi\,$ (at the identity $\,e\in G\,$) turns $\,V\,$ into a
$\,\mathbb Z/p\mathbb Z$-graded vector space
$\,V\,=\,\displaystyle\sum_{i\in\mathbb Z/p\mathbb Z}\,V_{i}\,$, where
\[
V_i\,=\,\bigoplus_{\mu(H)=i}\,V_{\mu}\,.
\]
Thus, we can calculate $\,\dim\,H\,V\,$ as follows:
\begin{equation} \label{hvhv}
\begin{split}
\dim\,H\,V\,& =\,\sum_{\mu (H)\not\equiv 0\pmod p} %\stackrel{\scriptstyle
                                                   %{\mu\in{\cal X}(V)}}
\,\dim\,V_{\mu}\\
&  =  \sum_{\mu\in{\cal X}_{++}(V)}\,m_{\mu}\,|\{\,\nu\in W\mu\,/\,
\nu(H)\not\equiv 0\pmod p\,\}|\,,
\end{split}
\end{equation}
where $\,m_{\mu}\,=\,\dim\,V_{\mu}\,$. 
The second equality in~\eqref{hvhv} comes from the facts that 
$\,{\cal X}(V)\,=\,W\cdot {\cal X}_{++}(V)\,$ 
and that the multiplicity is constant on the $\,W$-orbits of $\,{\cal X}(V)$.

Let $\,(\,\cdot\,,\,\cdot\,)\,$ denote a scalar product on the 
$\,\mathbb R$-span of $\,\Delta,\,$ invariant under the action of the
Weyl group $\,W\,$ of $\,R.\,$ We adjust $\,(\,\cdot\,,\,\cdot\,)\,$ in such 
a way that $\,(\alpha,\alpha)=2\,$ for every short root $\,\alpha\in R.\,$
As usual, $\,\langle\beta,\,\alpha\rangle\,=\,2\,(\beta,\,\alpha)/
(\alpha,\,\alpha),\,$ for $\,\beta,\,\alpha\in R$.

This adjustment on $\,(\,\cdot\,,\,\cdot\,)\,$ and the assumption that
$\,p\,$ is not special for $\,G\,$ imply  $\,p\nmid\frac
{(\tilde{\alpha},\,\tilde{\alpha})}{2}\,$. Hence 
\begin{equation}\label{mual}
\mu (H)\,=\,\langle\mu,\,\tilde{\alpha}\rangle\,=\,\frac{2\,(\mu,\,
\tilde{\alpha})}{(\tilde{\alpha},\,\tilde{\alpha})}
\not\equiv\,0\!\pmod{p}
\,\Longleftrightarrow\,(\mu,\,\tilde{\alpha})\not\equiv 0\!\pmod{p}
\end{equation}

Note that $\,w\mu (H)\,=\,\langle w\mu,\,\tilde{\alpha}\rangle\,=\,
\langle\mu,\,w^{-1}\tilde{\alpha}\rangle$. Hence 
$\,w\mu (H)\,\not\equiv\,0\pmod{p}
\,\Longleftrightarrow\,(w\mu,\,\tilde{\alpha})\,\not\equiv 0\pmod{p}$,
by~(\ref{mual}). Also if $\,w_1\in C_W(\mu)$, then
$\,w\,w_1\mu (H)\,=\,w\mu (H).\,$ Thus, we obtain that
\begin{multline*}
\sharp\,\{ w\in W\,/\,w\mu(H)\not\equiv 0\!\!\pmod p\}\,= \\
|C_W(\mu)|\cdot |\{\nu\,\in\,W\mu\,/\,\nu(H)\not\equiv 0\!\!\pmod p\}|\,.
\end{multline*}
Observe that $\,(w\mu,\,\tilde{\alpha})\,=
\,(\mu,\,w^{-1}\tilde{\alpha})$, as the scalar product is
$\,W$-invariant. Thus $\;w\mu (H)\,\not\equiv\,0\!\!\pmod{p}
\;\Longleftrightarrow\;(\mu,\,w^{-1}\tilde{\alpha})\,
\not\equiv 0\!\!\pmod{p}$. 
Also, $\,\{\,w^{-1}\tilde{\alpha}\,/\,w\in W\,\}\,=\,R_{long}$, the set
of all long roots in $\,R$, forms a subsystem of $\,R\,$ 
\cite[VI \S 1 Ex.14]{bourb2}. 
%since $\,\tilde{\alpha}\,$ is conjugate to all long roots of $\,R\,$
%Note that $\,R_{long}\,$ forms a subsystem of $\,R$. 
In particular, 
$\,R_{long}\,=\,R_{long}^+\,\cup\,-R_{long}^+\,$. Hence
%since $\,(\mu,\,\gamma)\,\not\equiv 0\pmod{p}\;\Longleftrightarrow\;
%(\mu,\,-\gamma)\,\not\equiv 0\pmod{p}$, then we calculate
\begin{multline*}
\sharp\,\{ w\in W\,/\,w\mu(H)\not\equiv 0\!\!\pmod p\}\,= \\
2\,|C_W(\tilde{\alpha})|\cdot 
|\{\,\gamma\in R^+_{long}\,/\,(\mu,\,\gamma)\,
\not\equiv\,0\!\!\pmod p\,\}|\,.
\end{multline*}

We summarize the discussion in:
\begin{lemma}\label{lem345}
Given $\,\mu\,\in\,{\cal X}_{++}(V)\,$ one has
\begin{multline*}\label{ineq345}
|C_W(\mu)|\cdot |\{\,w\mu\in W\mu\,/\,(w\mu,\,\tilde{\alpha})
\,\not\equiv\,0\pmod p\,\}|\\
=\,2\,|C_W(\tilde{\alpha})|\cdot
|\{\,\gamma\in R^+_{long}\,/\,(\mu,\,\gamma)\,
\not\equiv\,0\pmod p\,\}|.
\end{multline*}
\end{lemma}
\hfill$\Box$

Set $\,R_{\mu,p}\,=\,\{\,\gamma\in R_{long}\,/\,
(\mu,\,\gamma)\,\equiv\,0\!\pmod p\,\}\,$ and 
$\,R_{\mu,p}^+\,=\,R_{\mu,p}\,\cap R_{long}^+\,=\,\{\,\gamma\in R_{long}^+\,/\,
(\mu,\,\gamma)\,\equiv\,0\!\pmod p\,\}\,$. %=\,R_{\mu,p}\,\cap R_{long}^+$. 
We have
\begin{equation} \label{split345}
\begin{split}
\frac{1}{2}\,|\{\,\nu\in W\mu\,/\,(\nu,\,\tilde{\alpha})
\,\not\equiv\,0\!\!\pmod p\,\}| & =  \frac{|C_W(\tilde{\alpha})|}
{|C_W(\mu)|}\,|R_{long}^+-R^+_{\mu,p}| \vspace{1ex} \\
 & = \frac{|W|}{|C_W(\mu)|}\,\frac{|R_{long}^+-R^+_{\mu,p}|}{|R_{long}|}
 \vspace{1ex}\\
 & =  \frac{|W\mu|}{|R_{long}|}\,
|R_{long}^+-R^+_{\mu,p}|
\end{split}
\end{equation} 

Now observe that $\,R^+_{\mu,p}\,=\,R_{long}^+\,$ if and only if 
$\,(\mu,\,\gamma)\,\equiv\,0\pmod p\,$ for all $\,\gamma\in R_{long}^+$.
In this case the contribution of the weight $\,\mu\,$ to the 
sum~(\ref{hvhv}) is $\,0$. This fact gives rise to the following definition:

\begin{definition}\label{badweight}
A weight $\,\mu\in P(R)\,$ is called {\bf bad} if 
$\,\mu(h_{\gamma})\equiv 0\pmod p\,$ for all $\,\gamma\in R^+_{long}$, 
or equivalently, if $\,(\mu,\,\gamma)\equiv 0\pmod p\,$ for all 
$\,\gamma\in R^+_{long}$.
If $\,\mu\,$ is not a bad weight we say that $\,\mu\,$ is {\bf good}.
\end{definition}

Recall that each weight $\,\mu\in P\,$ can be written as 
$\,\mu\,=\,\sum\,a_i\,\omega_i$, with $\,a_i\in\mathbb Z_+$.

\begin{lemma}
Suppose $\,p\,$ is not special for $\,G$. Then 
$\,\mu\in P\,$ is bad if and only if 
$\,\mu\,=\,\displaystyle\sum_{i=1}^{\ell}\,a_i\,\omega_i$, where 
$\,a_i\in p\mathbb Z\,$ for all $\,i\,$.
%$\,p=\Char(K)\,$ is not special for $\,G\,$.
%with $\,p\neq 2\,$ if $\,R\,$ is of type $\,C_{\ell}\,$ and
%$\,p\neq 3\,$ if $\,R\,$ is of type $\,G_2$. 
\end{lemma}\noindent
\begin{proof}
%For $\,G\,$ of type $\,A_{\ell},\,D_{\ell},\,E_6,\,E_7\,$ or $\,E_8,\,$ 
If $\,R_{long}\,=\,R\,$, the lemma is obvious. 

If $\,R\,$ is of type $\,B_{\ell}$, then $\,R_{long}\cong D_{\ell}$. A basis
of $\,R_{long}\,$ is given by $\,\alpha_1=\varepsilon_1
-\varepsilon_2\,$, $\,\alpha_2=\varepsilon_2 -\varepsilon_3,\ldots,
\,\alpha_{\ell -1}=\varepsilon_{\ell -1} -\varepsilon_{\ell},
\,\alpha_{\ell}=\varepsilon_{\ell -1} +\varepsilon_{\ell}$.
Hence $\,(\mu,\,\alpha_i)=a_i\,$, if $\,1\leq i\leq \ell -1,\,$ and 
$\,(\mu,\,\alpha_{\ell})=(\mu,\,\varepsilon_{\ell-1} +\varepsilon_{\ell})=
a_{\ell -1} + a_{\ell}$. Therefore, $\,\mu\,$ is bad for $\,R\,$  
%$\,B_{\ell}\,$
if and only if $\,a_i\equiv 0\pmod p\,$ for all $\,i$.

If $\,R\,$ is of type $\,C_{\ell}$, then $\,R_{long}\cong A_1^{\ell}$. A basis
of $\,R_{long}\,$ is given by $\,2\varepsilon_1$, $\,2\varepsilon_2$,
$\,\ldots,\,$$\,2\varepsilon_{\ell -1},\,2\varepsilon_{\ell}$.
%The solution for the system of equations 
One has $\,(\mu,\,2\varepsilon_i)=
2(a_i +\cdots +a_{\ell}),\;1\leq i\leq \ell\,$.
% is the trivial one $\,\pmod p$, provided $\,p\neq 2$. 
Therefore, $\,\mu\,$ is bad for $\,R\,$ %$\,C_{\ell}\,$
if and only if $\,a_i\equiv 0\pmod p\,$ (recall that in this case
$\,p\neq 2$).
                          
For $\,R\,$ of type $\,G_2$, $\,R_{long}\cong A_2\,$ has basis 
$\,-2\varepsilon_1 +\varepsilon_2+\varepsilon_3,\,\varepsilon_1-2\varepsilon_2 
+\varepsilon_3$.
Thus $\,(\mu,\,-2\varepsilon_{1} +\varepsilon_{2}+\varepsilon_3)= 3a_2\,$ and
$\,(\mu,\,\varepsilon_1-2\varepsilon_2 +\varepsilon_3)= 3(a_1+a_2)\,$, implying
that $\,\mu\,$ is bad for $\,G_2\,$
if and only if $\,a_i\equiv 0\pmod p\,$ (in this case $\,p\neq 3$).

For $\,R\,$ of type $\,F_4$, $\,R_{long}\cong D_4\,$ has basis
$\,\varepsilon_1 -\varepsilon_2,\,\alpha_1=\varepsilon_2 -\varepsilon_3,
\,\alpha_2=\varepsilon_3 -\varepsilon_4,\,
\varepsilon_3 +\varepsilon_4$. Hence,
$\,(\mu,\,\varepsilon_{1} -\varepsilon_{2})= a_2 +a_3 +a_4$,
$\,(\mu,\,\alpha_1)=a_1,\;(\mu,\,\alpha_2)=a_2\,$ and 
$\,(\mu,\,\varepsilon_3 +\varepsilon_4)=a_2 +a_3$. 
Therefore, $\,\mu\,$ is bad for $\,F_4\,$
if and only if $\,a_i\equiv 0\pmod p\,$.
\end{proof}
\smallskip
\noindent
\begin{remark} It is easy to check that if $\,\mu\,=\,\sum_{i=1}^{\ell}\, 
a_i\,\omega_i\,$ is a bad weight with $\,a_k\neq 0$, then 
$\,\mu -\alpha_k\,$ is a good weight, unless $\,p=2\,$ and $\,R\,$ is of type
$\,A_1\,$ or $\,B_{\ell}\,$.
\end{remark}
\smallskip

We further observe that for $\,\mu\,$ good, $\,R_{\mu,p}\,$ is a
proper subsystem of $\,R_{long}\,$, so it is contained in a maximal 
subsystem of $\,R_{long}$. We refer to~\cite{borsie} for a list of
maximal subsystems in $\,R_{long}$.
%(Indeed, $\,R_{\mu,p}\,$ is closed in $\,R_{long}$.)?????

From (\ref{hvhv}) and~(\ref{split345}) 
we conclude that if $\,V\,$ is an exceptional $\,\mathfrak g$-module
then it satisfies
\[
r_p(V)\,:=\,\sum_{\stackrel{\scriptstyle\mu\;good}{\mu\in{\cal X}_{++}(V)}}
\,m_{\mu}\,\frac{|W\mu|}{|R_{long}|}\,
|R_{long}^+-R^+_{\mu,p}|\,=\,\frac{1}{2}\,\dim\,H\,V\,\leq\,\dim\,E\,V\,
\leq |R|
\]
But then
\[
\left(\,\sum_{\stackrel{\scriptstyle\mu\;good}{\mu\in{\cal X}_{++}(V)}}
\,m_{\mu}\,\frac{|W\mu|}{|R_{long}|}\,\right){\displaystyle
\min_{\mu\,good}}\,|R_{long}^+-R^+_{\mu,p}|\,\leq\,|R|
\]
implying
\[
s(V)\,:=\,\sum_{\stackrel{\scriptstyle\mu\;good}{\mu\in{\cal X}_{++}(V)}}
\,m_{\mu}\,|W\mu|\,\leq\,\frac{|R|\,\;|R_{long}|}
{\displaystyle\min_{\mu\;good}\,|R_{long}^+-R^+_{\mu,p}|}
\]
Now note that $\,|R_{long}^+-R^+_{\mu,p}|\,$ is minimal if and only if
%$\,|R^+_{\mu,p}|\,$ is maximal if and only if
$\,|R_{\mu,p}|\,$ is maximal. (Recall that $\,R_{\mu,p}\subseteq
R_{long}$.) 
Put $\,M\,=\,\displaystyle\min_{\mu\,good}\,|R_{long}^+-R^+_{\mu,p}|\,$ 
and $\,L_G\,=\,\displaystyle\frac{|R|\,\;|R_{long}|}
{\displaystyle\min_{\mu\,good}\,|R_{long}^+-R^+_{\mu,p}|}\,$. We
call $\,\lceil L_G\rceil\,$ the {\bf limit}. Table~\ref{table2} lists all
related subsystems and gives 
the values of $\,\lceil L_G\rceil\,$ for different types of $\,G$. 
In the column
headed by $\,MLS,\,$ we list the possible maximal subsystems of $\,R_{long}$. 

\begin{table}[tb]
\begin{center}
\begin{tabular}{|c|c|c|c|c|}\hline \hline
Type of $\,G$ & $R_{long}$  &  $MLS$ & 
$\,M\,$ &
$\,\lceil L_G \rceil\,$ \\ \hline \hline
$A_{\ell}$ & $A_{\ell}$  & $A_{\ell -1}$ &  $\ell$ & 
$\ell^3\,+\,2\ell^2\,+\,\ell$ \\ \hline
$B_{\ell}$ & $D_{\ell}$  & $D_{\ell -1}\,$ if $\,\ell\geq 5$  &  
  $2(\ell -1)$ &  $2\ell^3\,$ if $\,\ell\geq 5$ \\   
           &             & $A_{\ell -1}\,$ if $\,\ell\leq 4$  & 
           $\frac{(\ell -1)\ell}{2}$ & $8\ell^2\,$ if $\,\ell\leq 4$
           \\ \hline
$C_{\ell}$ & $A_1^{\ell}$  & $A_1^{\ell -1}$ &  $1$ & 
 $\,4\ell^3$ \\ \hline
$D_{\ell}$ & $D_{\ell}$  & $D_{\ell -1}\,$ if $\,\ell\geq 5$  &  
                           $2(\ell -1)$ & $2\ell^3\,-\,2\ell^2\,$ if
                           $\,\ell\geq 5$\\ 
           &             & $A_{\ell -1}\,$ if $\,\ell\leq 4$  & 
           $\frac{(\ell -1)\ell}{2}$ & $\frac{8\ell^2(\ell -1)}{(\ell -2)}$
           \\ \hline
$G_2$ & $A_2$  & $A_1$ &  $2$ & $36$ \\ \hline   
$F_4$ & $D_4$  & $A_3$ &  $6=2\cdot 3$ & $192$ \\ \hline   
$E_6$ & $E_6$  & $D_5$ &  $16=2^4$ & $324$ \\ \hline
$E_7$ & $E_7$  & $E_6$ &  $27=3^3$ & $588$ \\ \hline
$E_8$ & $E_8$  & $E_7$ &  $57=3\cdot 19$ & $1011$ \\ \hline \hline
\end{tabular}
\caption[The Limits]{The Limits\label{table2}}
\end{center}
\end{table}

We summarize these facts in the following theorem.
\begin{theorem}[The Necessary Condition]\label{???}
Let $\,p\,$ be a non-special prime for $\,G$.
Let $\,\pi\,:\,G\longrightarrow\GL(V)\,$
be a non-trivial, faithful, rational representation of $\,G\,$ such that
$\,\ker\,{\rm d}\pi\subseteq\mathfrak z(\mathfrak g)\,$.
If $\,V\,$ is an exceptional $\,\mathfrak g$-module, then it satisfies
the inequalities
\begin{equation}\label{equata}
r_p(V)\,:=\,\sum_{\stackrel{\scriptstyle\mu\;good}{\mu\in{\cal X}_{++}(V)}}
\,m_{\mu}\,\frac{|W\mu|}{|R_{long}|}\,
|R_{long}^+-R^+_{\mu,p}|\,\leq |R|\,,
\end{equation}
and
\begin{equation}\label{?ref123}
s(V)\,:=\,\sum_{\stackrel{\scriptstyle\mu\,good}{\mu\in{\cal X}_{++}(V)}}
\,m_{\mu}\,|W\mu|\;\leq\;\mbox{\bf limit}\,,
\end{equation}
where $\,m_{\mu}\,$ denotes the multiplicity of $\,\mu\in
{\cal X}_{++}(V)\,$ and the {\bf limit}s for the different types of   
algebraic groups are given in Table~\ref{table2}.
\end{theorem}
\hfill $\Box$

\newpage

\part*{Chapter 4}
\part*{The Procedure and The Results}
\addcontentsline{toc}{section}{\protect\numberline{4}{The
Procedure and The Results}}
\setcounter{section}{4}
\setcounter{subsection}{0}
The main section of this Chapter describes the techniques used to
classify the exceptional modules. Everywhere below the indexing of 
the simple roots in the basis
$\,\Delta\,=\,\{\,\alpha_1,\,\ldots\,,\alpha_{\ell}\,\}\,$ and of the
fundamental weights in the basis
$\,\{\,\omega_1,\,\ldots,\,\omega_{\ell}\,\}\,$
corresponds to Bourbaki's tables~\cite[Chap. VI, Tables I-IX]{bourb2}.

\subsection{Some Facts on Weights}

In this section we formulate some known facts and obtain a number of 
preliminary results that will be extensively used to prove lemmas
in the next sections.

Given a weight $\,\mu\,=\,a_1\,\omega_1\,+\,\cdots\,+
\,a_{\ell}\,\omega_{\ell}\,\in\,P\,$, %{\cal X}_{++}(V)$, 
the size of its orbit under the action of the Weyl group is given by
\[
|W\mu|\,=\,\frac{|\,W\,|}{|\,C_W(\mu)\,|},
\]
where $\,C_W(\mu)\,=\,\langle s_{\alpha_i}\,/\,a_i\,=\,0\rangle\,$ is
the centralizer of $\,\mu\,$ in $\,W\,$ \cite[1.10]{seitz1}.
%Note that $\,|W\mu|\,$ is independent of $\,p=\Char\,K\,$ (see Bourbaki ???).

\begin{lemma}\label{remark}
Let $\,\lambda,\,\mu\,\in\,P_{++}\,$ be such that %{\cal X}_{++}(V)\,$ and 
$\,\langle\lambda,\,\alpha_i\rangle\,=\,0\,$ if 
$\,\langle\mu,\,\alpha_i\rangle\,=\,0.\,$ Then $\,|W\mu|\,\geq\,|W\lambda|$.
If $\,\langle\mu,\,\alpha_i\rangle\,\neq\,0\,$ and 
$\,\langle\lambda,\,\alpha_i\rangle\,=\,0\,$ for some $\,i,\,$ then
$\,|W\mu|\,\geq\,2\,|W\lambda|$.
\end{lemma}\noindent
\begin{proof}
From the definition of centralizer it is clear that 
$\,C_W(\mu)\,$ is a subgroup of $\,C_W(\lambda)\,$ and a proper one
if $\,\langle\mu,\,\alpha_i\rangle\,\neq\,0\,$ and
$\,\langle\lambda,\,\alpha_i\rangle\,=\,0.\,$ This yields the assertion.
\end{proof}

If $\,Y\subseteq \Delta$, then
we write $\,W(Y)\,=\,\langle\,s_{\alpha}\,/\,\alpha\in Y\,\rangle\,$ for the
subgroup of $\,W\,$ generated by the reflections $\,s_{\alpha}\,$
corresponding to roots $\,\alpha\in Y$. 
If $\,Y\,$ is a root (sub)system, then $\,W(Y)\,$ denotes the Weyl
group associated with the system $\,Y\,$.
Moreover, we denote
$\,\Delta_i\,=\,\Delta-\{\alpha_i\},\;\,\Delta_{ij}\,=\,\Delta
-\{\alpha_i,\alpha_j\},\,$
and so on. Note that $\,C_W(\mu)\,=\,W(Y)$, for some $\,Y\subseteq \Delta$.

In the following lemmas, $\,i_0\,$ and $\,i_{-1}\,$ are both $\,0.\,$ 
Recall that $\,P\,(=\,P(G)\,)\,$ denotes the weight lattice of the
group $\,G\,$.
\begin{lemma}\label{orban}
Let $\,G\,$ be an algebraic group of type $\,A_{\ell}.\,$ Let 
$\,\mu\,=\,a_{i_1}\,\omega_{i_1}\,+\,a_{i_2}\,\omega_{i_2}\,+\,\cdots\,+
\,a_{i_m}\,\omega_{i_m}\,\in\,P$, %{\cal X}_{++}(V)$, 
with $\,i_{k-1}\,<\,i_{k}\,$ and $\,a_{i_k}\neq 0\,$ for $\,1\leq
k\leq m.\,$ Then 
\[
|W\mu|\,=\,\binom{i_2}{i_1}\,\cdots\,\binom{i_m}{i_{m-1}}\,
\binom{\ell+1}{i_m}.
\]
\end{lemma}\noindent
\begin{proof} For $\,G\,$ of type $\,A_{\ell},\,$
$\,|W|\,=\,(\ell+1)!\,$ and
$\,C_W(\mu)\,\cong\,S_{n_1+1}\times S_{n_2+1}\times\cdots\times
S_{n_{m+1}+1},\,$ 
where the $\,n_k$'s are obtained from the Dynkin diagram in the 
following way: $\,n_1= i_1 -1,\quad n_{m+1}= \ell - i_m,\;$
$\,n_k =i_k-1-i_{k-1}=$ number of 
nodes between the $\,i_{k-1}$th and the $\,i_{k}$th nodes, for $\,1<k\leq m$.  
Therefore, $\,|\,C_W(\mu)\,|\,=\,i_1!\,(i_2-i_1)!\,\cdots
\,(i_m-i_{m-1})!\,(\ell-i_m+ 1)!$, yielding the result.
\end{proof}

\begin{lemma}\label{orbbncndn}
Let $\,G\,$ be an algebraic group of type $\,B_{\ell},\,C_{\ell}\,$ or
$\,D_{\ell}$. Let $\,\mu\,=\,a_{i_1}\,\omega_{i_1}\,+\,
a_{i_2}\,\omega_{i_2}\,+\,\cdots\,+\,a_{i_m}\,\omega_{i_m}\,\in\, P$,
%%%%%%%%%%%%%%%{\cal X}_{++}(V)$, 
with $\,i_{k-1}\,<\,i_{k}\,$ and 
$\,a_{i_k}\neq 0\,$ for \mbox{$\,1\leq k\leq m\,$}. Then 
\[
|W\mu|\,=\,2^r\,\binom{i_2}{i_1}\,\cdots\,\binom{i_{m-1}}{i_{m-2}}\,
\binom{t}{i_{m-1}}\,\binom{\ell}{t}
\]
where $\,r=t=i_m\,$ if $\,G\,$ is not of type $\,D_{\ell}\,$ or 
$\,i_{m}<\ell -1;\,$ if $\,G\,$ is of type $\,D_{\ell},\,$
and $\,i_{m}=\ell-1\,$ or ($\,i_{m}=\ell\,$ and $\,i_{m-1}\leq
\ell-2\,$), then $\,r=\ell-1,\,t=\ell;\,$ for ($\,i_{m}=\ell\,$ and 
$\,i_{m-1}=\ell-1\,$), $\,r=\ell-1\,=\,t\,=\,i_{m-1}.\,$
\end{lemma}\noindent
\begin{proof}
For $\,G\,$ of type $\,B_{\ell}\,$ or $\,C_{\ell},\,$
$\,|W|\,=\,2^{\ell}\,\ell!.\,$ For $\,i_m\,<\,\ell-1\,$ we have
$\,C_W(\mu)\,\cong\,S_{n_1+1}\times S_{n_2+1}\times\cdots\times
S_{n_{m}+1}\times W(B_{\ell-i_m}\;\mbox{or}\;C_{\ell-i_m}),\,$ 
where the $\,n_k$'s ($\,1\leq k\leq m\,$) are obtained from the 
Dynkin diagram as in the proof of Lemma~\ref{orban}.
Thus $\,|\,C_W(\mu)\,|\,=\,i_1!\,(i_2-i_1)!\,\cdots
\,(i_m-i_{m-1})!\,2^{\ell-i_m}\,(\ell-i_m)!\,$ and so
\[
|W\mu|\,=\,2^{i_m}\,\binom{i_2}{i_1}\,\cdots\,\binom{i_{m-1}}{i_{m-2}}\,
\binom{i_m}{i_{m-1}}\,\binom{\ell}{i_m}.
\]
Hence, in these cases, $\,r=t=i_m.\,$
 
For $\,i_m\,=\,\ell-1,\,$ 
$\,C_W(\mu)\,\cong\,S_{n_1+1}\times S_{n_2+1}\times\cdots\times
S_{n_{m}+1}\times S_2\,$ so
$\,|\,C_W(\mu)\,|\,=\,i_1!\,(i_2-i_1)!\,\cdots
\,(i_{m-1}-i_{m-2})!\,(\ell-1-i_{m-1})!\,2!\,$ and this implies
\[
|W\mu|\,=\,2^{\ell-1}\,\binom{i_2}{i_1}\,\cdots\,\binom{i_{m-1}}{i_{m-2}}\,
\binom{\ell-1}{i_{m-1}}\,\binom{\ell}{\ell-1}.
\]
Here $\,r=\ell-1\,=t=i_m.\,$

For $\,i_m\,=\,\ell,\,$
$\,C_W(\mu)\,\cong\,S_{n_1+1}\times S_{n_2+1}\times\cdots\times
S_{n_{m}+1},\,$ so
\[
|W\mu|\,=\,2^{\ell}\,\binom{i_2}{i_1}\,\cdots\,\binom{i_{m-1}}{i_{m-2}}\,
\binom{\ell}{i_{m-1}}\,\binom{\ell}{\ell}\,,
\]
and $\,r=\ell=t=i_m.\,$ 

For $\,G\,$ of type $\,D_{\ell},\,$ 
$\,|W|\,=\,2^{\ell-1}\,\ell!.\,$ For $\,i_m\,<\,\ell-1\,$ we have
$\,C_W(\mu)\,\cong\,S_{n_1+1}\times S_{n_2+1}\times\cdots\times
S_{n_{m}+1}\times W(D_{\ell-i_m}),\,$ 
where the $\,n_k$'s ($\,1\leq k\leq m\,$) are obtained from the 
Dynkin diagram as in the proof of Lemma~\ref{orban}.
Thus $\,|\,C_W(\mu)\,|\,=\,i_1!\,(i_2-i_1)!\,\cdots
\,(i_m-i_{m-1})!\,2^{\ell-i_m-1}\,(\ell-i_m)!\,$ and so
\[
|W\mu|\,=\,2^{i_m}\,\binom{i_2}{i_1}\,\cdots\,\binom{i_{m-1}}{i_{m-2}}\,
\binom{i_m}{i_{m-1}}\,\binom{\ell}{i_m}.
\]
Here $\,r=t=i_m.\,$

For $\,i_m\,=\,\ell-1\,$ (or $\,i_m\,=\,\ell\,$ and
$\,i_{m-1}\,\leq\,\ell-2\,$),
$\,C_W(\mu)\,\cong\,S_{n_1+1}\times S_{n_2+1}\times\cdots\times
S_{n_{m-1}+1}\times S_{\ell-1-i_{m-1}+1}\,$ so
\[
|W\mu|\,=\,2^{\ell-1}\,\binom{i_2}{i_1}\,\cdots\,\binom{i_{m-1}}{i_{m-2}}\,
\binom{\ell}{i_{m-1}}\,\binom{\ell}{\ell}.
\]
Here $\,r=\ell-1,\,t=\ell.\,$ 
For $\,i_m\,=\,\ell\,$ and $\,i_{m-1}\,=\,\ell-1\,$
$\,C_W(\mu)\,\cong\,S_{n_1+1}\times S_{n_2+1}\times\cdots\times
S_{n_{m-1}+1},\,$ so
\[
|W\mu|\,=\,2^{\ell-1}\,\binom{i_2}{i_1}\,\cdots\,\binom{i_{m-2}}{i_{m-3}}\,
\binom{\ell-1}{i_{m-2}}\,\binom{\ell}{\ell-1}\,,
\]
whence $\,r=\ell-1\,=\,t=i_{m-1}.\,$ This proves the Lemma.
\end{proof}

%\pagebreak

\begin{lemma}\label{orbexcp}
Some formulae for exceptional groups.

(1) Let $\,G\,$ be of type $\,E_6.\,$ Then $\,|W|\,=\,2^7\cdot 3^4\cdot 5\,$,
\begin{gather*}
C_W(\omega_1)\,\cong\,C_W(\omega_6)\,\cong\,W(D_5),\quad 
C_W(\omega_2)\,\cong\,S_6,\\
C_W(\omega_3)\,\cong\,C_W(\omega_5)\,\cong\,S_2\times S_5,\quad 
C_W(\omega_4)\,\cong\,S_2\times S_3^2,\\
C_W(\omega_1\,+\,\omega_2)\,\cong\,C_W(\omega_1\,+\,\omega_3)\,\cong\,
C_W(\omega_2\,+\,\omega_6)\,\cong\,C_W(\omega_5\,+\,\omega_6)\,\cong\,S_5,\\
C_W(\omega_1\,+\,\omega_4)\,\cong\,C_W(\omega_3\,+\,\omega_4)\,\cong\,
C_W(\omega_3\,+\,\omega_5)\,\cong\,C_W(\omega_4\,+\,\omega_5)\\
\cong\,C_W(\omega_4\,+\,\omega_6)\,\cong\,S_2^2\times S_3,\\
C_W(\omega_1\,+\,\omega_5)\,\cong\,C_W(\omega_2\,+\,\omega_3)\,\cong\,
C_W(\omega_2\,+\,\omega_5)\,\cong\,C_W(\omega_3\,+\,\omega_6)\,\cong\,
S_2\times S_4,\\
C_W(\omega_1\,+\,\omega_6)\,\cong\,W(D_4),\quad
C_W(\omega_2\,+\,\omega_4)\,\cong\,S_3^2,\quad
C_W(\omega_2\,+\,\omega_3\,+\,\omega_5)\,\cong\,S_2^3.
\end{gather*}

(2) Let $\,G\,$ be of type $\,E_7.\,$ Then 
$\,|W|\,=\,2^{10}\cdot 3^4\cdot 5\cdot 7\,$, 
\begin{gather*}
C_W(\omega_1)\,\cong\,W(D_6),\quad C_W(\omega_2)\,\cong\,S_7,\quad
C_W(\omega_3)\,\cong\,S_2\times S_6,\\
C_W(\omega_4)\,\cong\,S_2\times S_3\times S_4,\quad 
C_W(\omega_5)\,\cong\,S_3\times  S_5,\quad 
C_W(\omega_6)\,\cong\,S_2\times W(D_5),\\ 
C_W(\omega_7)\,\cong\,W(E_6),\quad
C_W(\omega_1\,+\,\omega_2)\,\cong\,C_W(\omega_1\,+\,\omega_3)\,\cong\,
C_W(\omega_2\,+\,\omega_7)\,\cong\,S_6,\\
C_W(\omega_1\,+\,\omega_4)\,\cong\,S_2^2\times S_4,\quad
C_W(\omega_1\,+\,\omega_5)\,\cong\,C_W(\omega_2\,+\,\omega_4)\,\cong\,
S_3\times S_4,\\
C_W(\omega_1\,+\,\omega_6)\,\cong\,S_2\times W(D_4),\quad
C_W(\omega_1\,+\,\omega_7)\,\cong\,
C_W(\omega_6\,+\,\omega_7)\,\cong\,W(D_5),\\
C_W(\omega_3\,+\,\omega_7)\,\cong\,C_W(\omega_5\,+\,\omega_7)\,\cong\,
S_2\times S_5,\quad
C_W(\omega_4\,+\,\omega_6)\,\cong\,S_2^3\times S_3,\\
C_W(\omega_4\,+\,\omega_7)\,\cong\,S_2\times S_3^2,\quad
C_W(\omega_2\,+\,\omega_3\,+\,\omega_5)\,\cong\,S_2^2\times S_3.
\end{gather*}

(3) Let $\,G\,$ be of type $\,E_8.\,$ Then 
$\,|W|\,=\,2^{14}\cdot 3^5\cdot 5^2\cdot 7\,$,
\begin{gather*}
C_W(\omega_1)\,\cong\,W(D_7),\quad C_W(\omega_2)\,\cong\,S_8,\quad
C_W(\omega_3)\,\cong\,S_2\times S_7,\\
C_W(\omega_4)\,\cong\,S_2\times S_3\times S_5,\quad 
C_W(\omega_5)\,\cong\,S_4\times S_5,\\
C_W(\omega_6)\,\cong\,S_3\times W(D_5),\quad
C_W(\omega_7)\,\cong\,S_2\times W(E_6),\\
C_W(\omega_8)\,\cong\,W(E_7),\quad
C_W(\omega_1\,+\,\omega_4)\,\cong\,C_W(\omega_5\,+\,\omega_7)\,\cong\,
\,S_2^2\times S_5,\\
C_W(\omega_1\,+\,\omega_8)\,\cong\,W(D_6),\quad 
C_W(\omega_2\,+\,\omega_8)\,\cong\,S_7,\quad 
C_W(\omega_3\,+\,\omega_8)\,\cong\,S_2\times S_6,  %\\
\end{gather*}
\vspace*{-2ex}
\begin{gather*}
C_W(\omega_4\,+\,\omega_6)\,\cong\,S_2^2\times S_3^2,\quad
C_W(\omega_4\,+\,\omega_8)\,\cong\,S_2\times S_3\times S_4,\\
C_W(\omega_5\,+\,\omega_8)\,\cong\,S_3\times S_5,\quad
C_W(\omega_6\,+\,\omega_8)\,\cong\,S_2\times W(D_5),\\
C_W(\omega_7\,+\,\omega_8)\,\cong\,W(E_6),\quad
C_W(\omega_2\,+\,\omega_3\,+\,\omega_5)\,\cong\,S_2^2\times S_4.
\end{gather*}

(4) Let $\,G\,$ be of type $\,F_4.\,$ Then $\,|W|\,=\,2^{7}\cdot 3^2\,$, 
\begin{gather*}
C_W(\omega_1)\,\cong\,C_W(\omega_4)\,\cong\,W(C_3),\quad
C_W(\omega_2)\,\cong\,C_W(\omega_3)\,\cong\,S_2\times S_3,\\
C_W(\omega_1\,+\,\omega_2)\,\cong\,C_W(\omega_3\,+\,\omega_4)\,\cong\,
\,S_3,\quad
C_W(\omega_1\,+\,\omega_4)\,\cong\,W(C_2),\\
C_W(\omega_1\,+\,\omega_3)\,\cong\,C_W(\omega_2\,+\,\omega_3)\,\cong\,
C_W(\omega_2\,+\,\omega_4)\,\cong\,\,S_2^2, \\
C_W(\omega_1\,+\,\omega_2\,+\,\omega_4)\,\cong\,S_2.
\end{gather*}

(5) Let $\,G\,$ be of type $\,G_2.\,$ Then 
$\,|W|\,=\,12,\quad C_W(\omega_1)\,\cong\,C_W(\omega_2)\,\cong\,S_2\,$.
\end{lemma}
\begin{proof}
The information on $\,W(G)\,$ is extracted from Tables 5-8 of~\cite{bourb}.
Centralizers are found by using the remark preceeding
Lemma~\ref{remark}.
\end{proof}

\begin{lemma} \label{lamu}
Let $\,\lambda\in \Lambda_p\,$ and $\,\mu\,\in\,P_{++}\,$ be such that
$\,\mu < \lambda\,$ (that is, 
$\,\lambda - \mu\,=\,$ sum of positive roots). Then 
$\,{\cal X}_{\mathbb C}(\mu)\,\subseteq\,{\cal X}(\lambda)$, where
$\,{\cal X}_{\mathbb C}(\mu)\,$ denotes the set of weights of a
complex $\,{\mathfrak g}_{\mathbb C}$-module of highest weight $\,\mu\,$.
\end{lemma}\noindent
\begin{proof} As $\,{\cal X}_{\mathbb C}(\mu)\,$ and $\,{\cal
X}(\lambda)\,$ are $\,W\,$-invariant, it is enough to prove that 
\mbox{$\,{\cal X}_{++,\mathbb C}(\mu)\,\subseteq\,{\cal X}_{++}(\lambda)$}.

Let $\,\gamma\,\in\,{\cal X}_{++,\mathbb C}(\mu)\,=\,(\mu - Q_+) \cap
P_{++}$. Then $\,\mu\,-\,\gamma \,\in\,Q_+\,$ and $\,\gamma\in P_{++}$.
By hypothesis, $\,\lambda\,-\,\mu\,\in\,Q_+\,$ so 
$\,\lambda - \gamma + \gamma - \mu\,\in\,Q_+\,$ implying that
$\,\lambda - \gamma\,\in\,Q_+$. Hence, 
$\,\gamma\,\in\,(\lambda - Q_+) \cap P_{++}\,=\,\,{\cal
X}_{++}(\lambda)$, proving the lemma.
\end{proof}
\bigskip
\noindent

The following lemma is a very useful fact on multiplicities.

\begin{lemma}\label{mulmin}
Let $\,V\,$ be an $\,A_{\ell}(K)$-module of highest weight
$\,\lambda\,=\,\omega_i\,+\,\omega_j\,$, where $\,1\leq i<j \leq
\ell$. Then the weight $\,\mu\,=\,\omega_{i-1}\,+\,\omega_{j+1}\,\in
\,{\cal X}_{++}(V)\,$ has multiplicity $\,j-i+1\,$ 
if $\,p\,$ does not divide $\,j-i+2,\,$ and $\,j-i\,$ otherwise. 
\end{lemma}\noindent
\begin{proof}
Let $\,A\,$ be the Lie algebra (or group) generated by
$\,\{\,e_{\alpha_i=\alpha_1'},\ldots,\,e_{\alpha_j=\alpha_{j-i+1}'}\,\}.\,$ 
Note that $\,A\,$ is isomorphic to a Lie algebra (or group) of type
$\,A_{j-i+1}\,$. 
Consider the $\,A$-module $\,U\,=\,V|_A$. $\,U\,$ is a nontrivial
homomorphic image of $\,V(\tilde{\alpha}')\,\cong\,\mathfrak{sl}(j-i+2),\,$ 
the adjoint module with respect to $\,A$. 
The minuscule weight of $\,U\,$ is
$\,\mu\,=\,\omega'_{i-1}\,+\,\omega'_{j+1}\,=\,0'\,$. It 
has multiplicity $\,j-i+1\,$ if $\,p\,$ does not divide $\,j-i+2,\,$
and multiplicity $\,j-i\,$ otherwise (for 
the Lie algebra of type $\,A_n\,$ is simple if $\,p\,$ does not divide
$\,n+1\,$ and has one dimensional centre otherwise (in this case the
quotient algebra is simple)).
Now, by Smith's Theorem~\cite{smith}, the result follows.
\end{proof}

\subsubsection{Realizing $\,E(2\omega_1)\,$}\label{roapm}

In this subsection we describe the realization of the irreducible 
$\,B_{\ell}(K)$- or $\,D_{\ell}(K)$-module of highest weight $\,2\omega_1\,$,
denoted by $\,E(2\omega_1)\,$. We need the following definition 
due to Donkin and Jantzen.

\begin{definition}\label{weylfilt}{\rm\cite[11.5]{donk1}}
A descending chain $\,V\,=\,V_1\supset V_2\supset V_3\supset \cdots\,$
of submodules of a $\,G$-module $\,V\,$ is called a {\bf Weyl filtration}
if each $\,V_i/V_{i+1}\,$ is isomorphic to some $\,V_K(\lambda_i)\,$ with 
$\,\lambda_i\,\in\,P_{++}\,$.
\end{definition}

\begin{proposition}\label{filttensor}{\rm\cite[11.5.2]{donk1},\cite{mathi}}
Let $\,\lambda,\,\mu\,\in\,P_{++}\,$. Then $\,V_K(\lambda)\otimes V_K(\mu)\,$
has a Weyl filtration.
\end{proposition}

Let $\,V=E(\omega_1)\,$ be the irreducible module of highest
weight $\,\omega_1\,$ and \mbox{$\,\dim\,V = n$}. Let $\,f\,$ be a
non-degenerate quadratic form on $\,V\,$. There exist a basis of 
$\,V\,$ with respect to which $\,f\,$ has the form 
$\,f\,=\,x_1^2\,+\,x_2^2\,+\,\cdots\,+\,x_n^2\,$, where the $\,x_i$'s
are indeterminates.

Let $\,\mathfrak g\,=\,\mathfrak{so}(f)\,=\,\{ x\in\,{\rm End}(V)\,/\,
f(x\cdot v,\,w)\,+\,f(v,\,x\cdot w)\,=\,0,\;\mbox{for}\;v,\,w\in V\}$.
For $\,n\,=\,2\ell+1\,$, $\,\mathfrak g\,$ is a Lie algebra of type 
$\,B_{\ell}\,$ and for $\,n\,=\,2\ell\,$, $\,\mathfrak g\,$ is a Lie
algebra of type $\,D_{\ell}\,$.
Identify $\,\mathfrak{so}(f)\,$ with the space $\,{\rm Skew}_n\,$ 
of all $\,n\times n\,$ skew-symmetric matrices over $\,K$.

Let $\,G\,=\,\SO (f)\,$. Denote by $\,{\rm Symm}_n\,$ the space of
all $\,n\times n\,$ symmetric matrices over $\,K$. Then $\,{\rm Symm}_n\,$
and $\,S^2(V)\,$ are isomorphic as $\,G$-modules. Hence, we can
identify $\,S^2(V)\,$ with $\,{\rm Symm}_n\,$. 

The Lie algebra $\,\mathfrak{so}(f)\,$ acts on the space $\,{\rm Symm}_n\,$ as
follows. Given $\,x\,\in\,\mathfrak{so}(f)\,$ and $\,s\,\in\,{\rm Symm}_n\,$ 
one has $\,x(s)\,=\,xs\,-\,sx\,$. It is easy to see that $\,S^2(V)\,$ is an 
$\,\mathfrak{so}(f)$-module containing a highest
weight vector of weight $\,2\omega_1\,$. Hence, so does $\,{\rm Symm}_n\,$.

Let $\,U\,=\,{\rm Symm}_n\cap\mathfrak{sl}(n)\,$. Then $\,\mathfrak{so}(f)\,$  
acts on $\,U\,$. Since $\,\dim\,({\rm Symm}_n/U)= 1$, 
$\,U\,$ also has a highest weight vector of weight $\,2\omega_1\,$. 

By~Proposition~\ref{filttensor}, $\,{\rm End}(V)\,\cong\,V\otimes V\,
\cong\,{\rm Symm}_n\oplus {\rm Skew}_n\,$ has Weyl filtration. It follows  
that $\,S^2(V)\,$ has Weyl filtration as well~\cite[Chapter 4]{andjan}.

Let $\,V(2\omega_1)\,$ be the Weyl module with highest weight $\,2\omega_1\,$. 
Then, by Weyl's dimension formula \cite[24.3]{hum1},
$\,\dim\,V(2\omega_1)=\displaystyle\frac{(n+2)(n-1)}{2}=\mbox{$\dim\,
{\rm Symm}_n -1$}$. As $\,\dim\,S^2(V)\,=\,\displaystyle 
\frac{n\,(n+1)}{2}\,$, the members of the Weyl filtration of  
$\,S^2(V)\,$ are $\,V(2\omega_1)\,$ and $\,V(0),\,$ the
trivial module (by dimension reasons). \vspace{2ex}

{\bf Claim 1}: $\,V(2\omega_1)\,$ is a submodule of $\,S^2(V).\,$ 

Indeed, suppose $\,V(2\omega_1)\,\cong\,\displaystyle\frac{S^2(V)}{V(0)}\,$.
Let $\,\bar v\,$ be the highest weight vector of weight $\,2\omega_1\,$ 
in $\,\displaystyle\frac{S^2(V)}{V(0)}$. Let $\,v\,$ be
the preimage of $\,\bar v\,$ in $\,S^2(V)_{2\omega_1}\,$. All
$\,T$-weights of $\,S^2(V)\,$ are $\,\leq\,2\omega_1\,$ with respect
to our partial ordering. Hence, $\,v\,$ is a highest weight vector
of $\,S^2(V)$. Let $\,\tilde V\,$ denote the $\,G$-submodule of
$\,S^2(V)\,$ generated by $\,v\,$. By the universal property of Weyl 
modules (see Theorem~\ref{univprop}(3)), there exists
a $\,G$-epimorphism $\,\varphi\,:\,V(2\omega_1)\longrightarrow \tilde
V\,$. On the other hand, as $\,V(2\omega_1)\,\cong\,\displaystyle 
\frac{S^2(V)}{V(0)}\,$, the image of $\,\tilde V\,$ in 
$\,\displaystyle\frac{S^2(V)}{V(0)}\,$ is the whole $\,\displaystyle
\frac{S^2(V)}{V(0)}$. Thus, there is a $\,G$-epimorphism $\,\psi\,:\,
\tilde V\longrightarrow V(2\omega_1)\,$. It follows that 
$\,\tilde V\,\cong\,V(2\omega_1)\,$ by Theorem~\ref{univprop}(3),  
whence the claim. \vspace{2ex}

%\pagebreak

{\bf Claim 2}: $\,V(2\omega_1)\,\cong\,{\rm Symm}_n\cap\mathfrak{sl}(n)\,$.

First note that $\,G\,=\,\SO (f)\,$ acts by conjugation on the space  
$\,{\rm Mat}_n\,$ of all $\,n\times n\,$ matrices over $\,K$, and 
$\,{\rm Symm}_n\,$ is invariant under this action. Hence \mbox{$\,U\,=\, 
{\rm Symm}_n\cap\mathfrak{sl}(n)\,$} is a $\,G$-submodule of $\,{\rm
Symm}_n\,$.
%%%%Explanations for these facts.
%$\,M\in\SO(f)\,$ iff $\,MM^t\,=\,Id\,$ Hence, for 
%$\,S\in{\rm Symm}_n\,$ $\,(M^{-1}SM)^t\,=\,M^t S (M^{-1})^t\,=\,M^{-1}SM\,$. 
%ALso if $\,X\in{\rm Symm}_n\cap\mathfrak{sl}(n)\,$ then 
%$\,trace(M^{-1} XM)\,=\,trace( XM M^{-1})\,=\,trace(X)\,$,
%hence $\,U\,$ is a $\,G$-module.

Now $\,U\,$ contains all vectors of nonzero weight (otherwise  
$\,{\cal X}({\rm Symm}_n/U)\,$ would contain a nonzero weight,  
contradicting $\,\dim\,({\rm Symm}_n/U)\,=1\,$). Moreover, each weight
of $\,U\,$ is $\,\leq\,2\omega_1\,$, and $\,2\omega_1\,$ has
multiplicity $\,1\,$ (as this is true for $\,S^2(V)\,$). As 
$\,{\rm Symm}_n\,\cong\,S^2(V)\,$ as $\,G$-modules, Claim 1 now shows
that $\,U\,\cong\,V(2\omega_1)\,$, proving Claim 2.

{\bf Claim 3}: $\,V(2\omega_1)\,$ is an irreducible
$\,\mathfrak{so}(f)$-module if and only if $\,p\nmid n\,$.
If $\,p\mid n\,$, then $\,E(2\omega_1)\,\cong\,\displaystyle
\frac{{\rm Symm}_n\cap\mathfrak{sl}(n)}{scalars}\,$.

Let $\,\Phi(2\omega_1)\,$ be the maximal submodule of $\,V(2\omega_1)$.
Recall that $\,\displaystyle\frac{V(2\omega_1)}{\Phi(2\omega_1)}\,\cong 
\,E(2\omega_1)\,$. By Theorem~\ref{theop1}, $\,{\cal X}(E(2\omega_1))\,=\, 
{\cal X}(V(2\omega_1))\,$. One can show that each nonzero 
$\,\mu\,\in\,{\cal X}(V(2\omega_1))\,$ has multiplicity $\,1$. This implies
$\,\Phi(2\omega_1)\,=\,a\,E(0)\,$, for some $\,a\geq 0$. 

The group $\,G\,$ acts trivially on $\,\Phi(2\omega_1)\,$. 
As $\,\Phi(2\omega_1)\subseteq {\rm Symm}_n\cap
\mathfrak{sl}(n)\,$ and $\,G\,$ acts on$\,{\rm Symm}_n\,$ by
conjugation, $\,\Phi(2\omega_1)\,$ consists of scalar matrices, 
by Schur's Lemma~\cite[1.5]{isaacs}. Hence,  
$\,\dim\,\Phi(2\omega_1)\,\leq\,1\,$. As $\,\Phi(2\omega_1)\subseteq 
\mathfrak{sl}(n)\,$, $\,\Phi(2\omega_1)\,\neq\,0\,$ if and only if $\,p\mid n$.
Hence, if $\,p\nmid n$, then $\,V(2\omega_1)\,\cong\,E(2\omega_1)\,$
is irreducible. If $\,p\mid n$, then $\,E(2\omega_1)\,\cong\,\displaystyle
\frac{{\rm Symm}_n\cap\mathfrak{sl}(n)}{scalars}\,$, as claimed.

%\newpage

%\include{results} %%%%inside of which anidea.tex,
		  %%%%bnidea.tex,cnidea.tex, dnidea.tex, e6e7e8.tex
      
\subsection{The Procedure}\label{proced}

Let $\,V\,$ be a rational $\,G$-module. 
Our objective throughout this work is to classify the exceptional
$\,\mathfrak g$-modules. 
The procedure used in the classification relies heavily on the 
necessary conditions given by Theorem~\ref{???}. We deal with the set
$\,{\cal X}_{++}(V)\,$ of weights of the module $\,V$. Hence we need
to guarantee that weights really occur in $\,{\cal X}_{++}(V)$. 
 
{\bf From now on assume that $\,p\,$ is non-special for $\,G\,$ and
also that $\,p\neq 2\,$ if $\,V\,$ has highest weight $\,\omega_1\,$ 
for groups of type $\,G_2\,$}. 

Under these assumptions, Theorem~\ref{premprin} says that {\it the system
of weights of an infinitesimally irreducible representation
$\,\pi\,:\,G\longrightarrow \GL(V),\,$ with highest weight
$\,\lambda\in\Lambda_p\,$ coincides 
with the system of weights of an irreducible complex representation 
$\,\pi_{\mathbb C}\,$ of a Lie algebra $\,{\mathfrak g}_{\mathbb C}\,$
with the same highest weight.} In particular, the set of dominant
weights of the representation $\pi$ is 
$\,{\cal X}_{++}(\lambda)=\,(\lambda - Q_+)\,\cap\,P_{++}$.

Thus, in the sequel {\bf module} means {\bf an infinitesimally irreducible 
finite dimensional rational
$\,G$-module $\,V\,$ of highest weight $\,\lambda\in\Lambda_p\,$}.
In particular, $\,V\,$ is an irreducible (restricted) $\,\mathfrak g$-module
(cf. Theorem~\ref{curt}). \vspace{2ex}

We now describe our procedure.
Recall that by Theorem~\ref{???}, an exceptional 
$\,\mathfrak g$-module $\,V\,$ satisfies the inequalities
\begin{gather*}
s(V)\,=\,\sum_{\stackrel{\scriptstyle\mu\,good}{\mu\in{\cal X}_{++}(V)}}
\,m_{\mu}\,|W\mu|\,\leq\,\mbox{{\bf limit}}\qquad\qquad\mbox{and}
\vspace{1.3ex}\\
r_p(V)\,=\,\sum_{\stackrel{\scriptstyle\mu\;good}{\mu\in{\cal X}_{++}(V)}}
\,m_{\mu}\,\frac{|W\mu|}{|R_{long}|}\,
|R_{long}^+-R^+_{\mu,p}|\,\leq\,|R|\,.
\end{gather*}
Thus, we start by eliminating all possible modules of highest weight 
$\,\lambda\,$ for which 
\begin{equation}\label{?ref}
s(V)\,=\,\sum_{\stackrel{\scriptstyle\mu\,good}{\mu\in{\cal X}_{++}(V)}}
\,m_{\mu}\,|W\mu|\;>\;\mbox{{\bf limit}}\,, 
\end{equation} 
since by Theorem~\ref{???}, $\,G$-modules 
$\,V\,$ satisfying~\eqref{?ref} are {\bf not} exceptional.
In general, we proceed as follows:

Given a weight $\,\mu\in{\cal X}_{++}(V)\,$ we produce
weights $\,\mu =\mu_0,\,\mu_1,\,\mu_2,\,\ldots\,\in\,{\cal
X}_{++,\mathbb C}(\mu)\,\subseteq\,{\cal X}_{++}(V)\,$ and calculate 
the sum of the orbit sizes
\[
s_1(\mu)\,=\,\sum_{\stackrel{\scriptstyle \mu_i\in{\cal X}_{++,\mathbb
C}(\mu)}{\mu_i\,good}}\,|W\mu_i|\,.
\] 
If this sum is already bigger than the respective {\bf limit}, then we apply:

\begin{proposition}\label{criteria}
Let $\,\lambda\,$ be the highest weight of $\,V.\,$
If $\,\mu\,\in\,{\cal X}_{++}(V)\,$ and 
$\,s_1(\mu)\,=\,\displaystyle\sum_{\stackrel{\scriptstyle \mu_i\,good}{\mu_i
\in{\cal X}_{++,\mathbb C}(\mu)}}\,|W\mu_i|\,>\,\mbox{{\bf limit}},\,$ then 
$\,s(V)\,>\,\mbox{{\bf limit}}\,$. Hence $\,V\,$ is not an exceptional
$\,\mathfrak g$-module. In particular, if $\,\mu\,$ is a good weight
in $\,{\cal X}_{++}(V)\,$ and $\,|W\mu|\,>\,\mbox{{\bf limit}},\,$ 
then $\,V\,$ is not an exceptional $\,\mathfrak g$-module.
\end{proposition}\noindent
\begin{proof}
Just note that $\,{\cal X}_{++,\mathbb C}(\mu)\,\subseteq\,{\cal X}_{++}(V)\,$
and use~(\ref{?ref}).
\end{proof}

In the process, we usually obtain first certain reduction lemmas,
which in most cases say that if a module $\,V\,$ is such
that $\,{\cal X}_{++}(V)\,$ contains good weights with many (at least
$\,3\,$ or $\,4\,$) nonzero coefficients, then $\,V\,$ is not an exceptional
module. See, for example, Lemmas~\ref{4coef},~\ref{3orless},~\ref{3coefbn},  
\ref{3coefcn},~\ref{3cdn},~\ref{e61},~\ref{e71},~\ref{e81}.

In some cases we calculate the quantities  
$\,|R_{long}^+-R^+_{\mu_i,p}|\,$ for each of the weights 
$\,\mu =\mu_0,\,\mu_1,\,\mu_2,\,\ldots\,\in\,{\cal
X}_{++,\mathbb C}(\mu)\,\subseteq\,{\cal X}_{++}(V)\,$ and prove that 
\begin{equation}\label{rpv2}
r_p(V)\,=\,\sum_{\stackrel{\scriptstyle\mu\;good}{\mu\in{\cal X}_{++}(V)}}
\,m_{\mu}\,\frac{|W\mu|}{|R_{long}|}\,
|R_{long}^+-R^+_{\mu,p}|\,>\,|R|\,.
\end{equation}
If \eqref{rpv2} holds for $\,V\,$ then, again by Theorem~\ref{???},
$\,V\,$ is not exceptional.

Sometimes we assume that $\,V\,$ has highest weight $\,\mu\,$. Then we use 
information on the multiplicities of weights in $\,{\cal X}_{++}(V)\,$
to get the inequality~(\ref{?ref}) or \eqref{rpv2}.
% and prove that $\,V\,$ is not exceptional. 
We frequently use the works of Gilkey-Seitz~\cite{gise} and  
Burgoyne-Williamson~\cite{buwil}, which give 
tables of multiplicities of weights for representations of exceptional
and low-rank classical Lie algebras, respectively. 
%Theorem~\ref{supruzal} given below is also useful in many cases. 
In some cases, we calculate multiplicities by
using some Linear Algebra and the fact that 
$\,V\,$ has a basis consisting of the elements $\,f_{i_1}^{m_{i_1}}\, 
f_{i_2}^{m_{i_2}}\,\cdots\,f_{i_k}^{m_{i_k}}\cdot v_0\,$, where
$\,v_0\,$ is a highest weight vector of $\,V\,$ (cf. Lemma~\ref{bor42}). 

Using this procedure we end up with a small list of weights, whose
corresponding modules may or may not be exceptional.  
Then we have to decide precisely which ones are exceptional. 

To prove that a module $\,V\,$ is not exceptional it is enough to
exhibit a nonzero vector $\,v\in V\,$ for which $\,\mathfrak
g_v\,\subset\,\mathfrak z(\mathfrak g)\,$ or $\,\mathfrak g_v\,=\,0$,
but this is not an easy task as it depends on how the
representation is realized.

The following argument relates to the dimension of $\,V$. 
Let $\,\varepsilon\,=\,\dim\,\mathfrak z(\mathfrak g)\,$.

\begin{proposition}\label{dimcrit}
Let $\,V\,$ be an irreducible $\,\mathfrak g$-module. 
If $\,\dim\,V\,<\,\dim\,{\mathfrak g}\,-\,\varepsilon,\,$
%%% where $\,\varepsilon\,=\,\dim\,\mathfrak z(\mathfrak g),\,$ 
then $\,V\,$ is an exceptional module.
\end{proposition}\noindent
\begin{proof} In proving the proposition, we may assume that $\,\dim\,V\,>\,1$.
By the irreducibility, $\,V\,$ does not
have $\,1$-dimensional invariant subspaces, so that $\,\dim\,
{\mathfrak g}_v\,<\,\dim\,{\mathfrak g}.\,$ 

Suppose $\,V\,$ is not exceptional. Then there exists $\,0\neq v\,\in\,V\,$ 
such that $\,{\mathfrak g}_v\,\subseteq\,\mathfrak z(\mathfrak g)$. 
Consider the linear map $\,\psi_v\,:\,\mathfrak g\longrightarrow V\,$
given by $\,x\longmapsto x\cdot v\,$. Thus,
$\,\ker\,\psi_v\subseteq \mathfrak z(\mathfrak g)\,$ and, by Linear Algebra,
$\,\dim\,\psi_v(\mathfrak g)\,\geq \dim\,\mathfrak g\,-\,\dim\,
\mathfrak z(\mathfrak g)$. As \mbox{$\,\psi_v(\mathfrak g)\subseteq V\,$},
by contradiction, the result follows.
\end{proof} 
\vspace{2ex}\noindent
\begin{remark} The dimension criteria for a module to be exceptional
(given by Proposition~\ref{dimcrit}) is independent of $\,p$. 
\end{remark}

\medskip

The following proposition gives the list of all rational irreducible
representations of Chevalley groups whose weight multiplicities are
equal to $\,1$.
 
\begin{proposition}{\rm\cite[p. 13]{supruzal}}\label{supruzal}
Let $\,G\neq A_1(K)\,$ be a simply connected simple algebraic group over an
algebraically closed field $\,K\,$ of characteristic $\,p>0$. Define
the set of weights $\,\Omega=\Omega(G)\,$ as follows:
%\newpage

$\,\Omega(A_{\ell}(K))\,=\,\{\,\omega_i,\,a\,\omega_1,\,b\,\omega_{\ell},\,
c\,\omega_j\,+\,(p-1-c)\,\omega_{j+1};\;1\leq i\leq \ell,
\;1\leq j < \ell,\;\,0\leq a,\,b,\,c\,<\,p\,\}\,$;

$\,\Omega(B_{\ell}(K))\,=\,\{\,\omega_1,\,\omega_{\ell}\,\}\,$ for
$\,\ell\geq 3,\;p>2\,$;

$\,\Omega(C_{\ell}(K))\,=\,\{\,\omega_\ell,\;\mbox{for}\;\ell=2,\,3;\;\;
\omega_1,\,\omega_{\ell-1}+\frac{(p-3)}{2}\omega_{\ell},\;
\frac{(p-1)}{2}\omega_{\ell}\,\}\,$ for $\,p>2\,$ and
$\,\Omega(C_{\ell}(K))\,=\,\{\,\omega_1,\,\omega_\ell\,\}\,$ for
$\,p=2\,$;

$\,\Omega(D_{\ell}(K))\,=\,\{\,\omega_1,\,\omega_{\ell-1},\,\omega_\ell\,\}\,$;
\qquad 
$\,\Omega(E_6(K))\,=\,\{\,\omega_1,\,\omega_6\,\}\,$;

$\,\Omega(E_7(K))\,=\,\{\,\omega_7\,\}\,$;
\qquad
$\,\Omega(F_4(K))\,=\,\{\,\omega_4\,\}\,$ for $\,p=3\,$ and
$\,\emptyset\,$ for $\,p\neq 3\,$;

$\,\Omega(G_2(K))\,=\,\{\,\omega_1\,\}\,$ for $\,p\neq 3\,$ and
$\,\Omega(G_2(K))\,=\,\{\,\omega_1,\omega_2\,\}\,$ for $\,p=3\,$.

Let $\,\varphi\,$ be an irreducible rational representation of $\,G\,$
with highest weight
$\,\omega\,=\,\displaystyle\sum_{i=0}^k\,p^i\lambda_i$,
where $\,\lambda_i\,$ are the highest weights of the representations
from $\,M(G)$. \linebreak The multiplicities of all the weights of the given
representation $\,\varphi\,$ are equal to $\,1\,$ if and only if 
$\,\lambda_i=0\,$ or $\,\lambda_i\in\Omega(G)\,$ for $\,0\leq i\leq
k$, $\,\lambda_{i+1}\neq\omega_1\,$ for $\,p=2$,
$\,G=C_{\ell}(K),\,$$\,\lambda_i=\omega_{\ell}\,$ and for
$\,G=G_2(K),\,p=2,\,$ $\,\lambda_i=\omega_1\,$ or
$\,p=3,\,\lambda_i=\omega_2$. 
\end{proposition}

\begin{proposition}\label{twistexc}
Let $\,\hat{\sigma}\,$ be a graph automorphism of $\,\mathfrak g\,$ induced
by a graph automorphism $\,\sigma\,$ of $\,G$. If $\,V\,$ is an
exceptional $\,\mathfrak g$-module, then so is $\,V^{\hat{\sigma}}\,$.
\end{proposition}\noindent
\begin{proof}
The action of the graph automorphism on $\,G\,$ and $\,\mathfrak g\,$
is described in Section~\ref{graphaut}. Denote by $\,\rho_{\sigma}\,$
(respectively, $\,\rho_{\hat{\sigma}}\,$) the twisted representation
of $\,G\,$ (respectively, $\,\mathfrak g\,$). We have
$\,\rho_{\hat{\sigma}}(x)\cdot v\,=\,\rho(x^{\hat{\sigma}})\cdot v\,$ for all
$\,x\in\mathfrak g,\;v\in V\,$.

As $\,\sigma\,$ preserves positive roots, it preserves the
Borel subgroup of $\,G$. Hence, $\,V\,$ and $\,V^{\hat{\sigma}}\,$
have the same highest weight vector, but $\,V^{\hat{\sigma}}\,$ has
highest weight $\,\lambda^{\sigma}\,$, where $\,\lambda\,$ is the
highest weight of $\,V$. 

Now, let $\,^{\sigma}\mathfrak g_v\,$ denote the isotropy subalgebra
of $\,v\,$ with respect to the representation $\,\rho_{\hat{\sigma}}\,$.
If $\,\rho(x)\cdot v\,=\,0,\,$ then 
$\,\rho_{\hat{\sigma}}(x^{{\hat{\sigma}}^{-1}})\cdot v\,=\, 
\rho((x^{{\hat{\sigma}}^{-1}})^{\hat{\sigma}})\cdot v\,=\,0$. 
Hence, $\,\mathfrak g_v\,\subseteq\,^{\sigma}\mathfrak g_v\,$.
By symmetry, equality holds and the result follows.
%%%$\,\dim\,\mathfrak g_v\,=\,\dim\,^{\sigma}\mathfrak g_v\,$. The
%%%result follows.
\end{proof}
\vspace{2ex}\noindent
\begin{remark}
Proposition~\ref{twistexc} implies that if $\,V\,$ is not an
exceptional module, then $\,V^{\hat{\sigma}}\,$ is also not
exceptional.
\end{remark}

We call $\,V\,$ and $\,V^{\hat{\sigma}}\,$ {\bf graph-twisted}
modules. If the graph automorphism has order $\,2$, then we also say
that $\,V\,$ and $\,V^{\hat{\sigma}}\,$ are {\bf graph-dual} modules.

%\newpage
\subsection{The Results}\label{resul}

\subsubsection{Lie Algebras of Exceptional Type}

\begin{theorem}\label{frlaet}
Let $\,G\,$ be a simply connected simple algebraic group of exceptional
type, and $\,\mathfrak g\,=\,{\cal L}(G).\,$ 
Let $\,V\,$ be an infinitesimally irreducible $\,G$-module.
If the highest weight of $\,V\,$ is listed in Table~\ref{table3},
then $\,V\,$ is an exceptional $\,\mathfrak g$-module. If $\,p\,$
is non-special for $\,G\,$, then the modules listed
in Table~\ref{table3} are the only exceptional $\,\mathfrak g$-modules.
\end{theorem}

The proof of this theorem will be given in Section~\ref{proofexcp}.

%\input{anidea.tex}
%if old table is necessary, look at resulbin.tex

\subsubsection{Lie Algebras of Classical Type}\label{lact} 

In this case, after using the procedure described in
Section~\ref{proced}, for each
algebra type we end up with a short list of highest weights. 
From these lists we have to decide, using further criteria, whether the
corresponding modules are exceptional or not. 
Most of the modules that are proved to be exceptional satisfy
Proposition~\ref{dimcrit}.

We have reduction lemmas for each group/algebra type
which are proved in this section.
The proofs of Theorems~\ref{anlist},~\ref{listbn},~\ref{listcn},
and~\ref{dnlist} are rather technical and are given in 
Appendix~\ref{appendixa} of this work.

\subsubsection*{\ref{lact}.1. Algebras of Type $\,A\,$}\label{alresults}
\addcontentsline{toc}{subsubsection}{\protect\numberline{\ref{lact}.1}
Algebras of Type $\,A\,$}

By Theorem~\ref{???}, any $\,A_{\ell}(K)$-module $\,V\,$ such that 
\begin{equation}\label{allim}
s(V)\,=\,\displaystyle\sum_{\stackrel{\scriptstyle\mu\,good}
{\scriptstyle \mu\in{\cal X}_{++}(V)}}\,m_{\mu}\,|W\mu|\,
>\,\ell\,(\,\ell\,+\,1\,)^2
\end{equation}
or
\begin{equation}\label{rral}
r_p(V)\,=\,\sum_{\stackrel{\scriptstyle{\mu\in{\cal X}_{++}(V)}}
{\scriptstyle \mu\;good}}\,m_{\mu}\,
\frac{|W\mu|}{|R_{long}|}\,|R_{long}^+-R^+_{\mu,p}|\,>\,\ell\,(\ell+1)
\end{equation}
is not an exceptional $\,\mathfrak g$-module.  
For groups of type $\,A_{\ell}$, $\,|W|\,=\,(\ell+1)!\,$,
$\,|R_{long}|\,=\,|R|\,=\,\ell\,(\ell+1)\,$ and
$\,|R_{long}^+|\,=\,\displaystyle\frac{\ell\,(\ell+1)}{2}.\,$

Let $\,V\,$ be an $\,A_{\ell}(K)$-module. Here we prove some
reduction lemmas. 

\begin{lemma}\label{4coef} Let $\,\ell\geq 4.\,$
If $\,{\cal X}_{++}(V)\,$ contains a weight
$\,\mu\,=\,\displaystyle\sum_{i=1}^{\ell}\,b_i\,\omega_i\,$ with 4 or
more nonzero coefficients, then $\,V\,$ is not an exceptional 
$\,A_{\ell}(K)$-module.
\end{lemma}\noindent
\begin{proof}
Let $\,\mu=b_1\,\omega_1\,+\,\cdots\,+\,b_{\ell}\,\omega_{\ell}\,\in\,
{\cal X}_{++}(V)\,$ have at least 4 nonzero coefficients.
We want to prove that $\,s(V)\,>\,\ell(\ell + 1)^2$. 

\underline{Step 1}: {\bf Claim}:
If $\,\nu=c_1\,\omega_1\,+\,\cdots\,+\,c_{\ell}\,
\omega_{\ell}\,$ is a weight with exactly 4 nonzero coefficients, say 
$\,\nu=c_1\,\omega_{i_1}\,+\,c_{2}\,\omega_{i_2}\,
+\,c_{3}\,\omega_{i_3}\,+\,c_{4}\,\omega_{i_4}$, with
$\,i_1<i_2<i_3<i_4\,$ and $\,c_{j}\neq 0$, then
$\,|W\nu|\,>\,\ell(\ell + 1)^2$.

First we prove that
if $\,\bar{\mu}\,=\,d_{1}\,\omega_{1}\,+\,d_{2}\,\omega_{2}\,
+\,d_{3}\,\omega_{3}\,+\,d_{4}\,\omega_{4}\,$ (or graph-dually 
$\,\bar{\mu}\,=\,d_{\ell -3}\,\omega_{\ell -3}\,+\,d_{\ell-2}\,
\omega_{\ell -2}\,+\,d_{\ell -1}\,\omega_{\ell -1}\,+\,d_{\ell}\,
\omega_{\ell}\,$), with all $\,d_i\neq 0$, then
$\,|W\nu|\,>\,|W\bar{\mu}|$, that is,
$\,W\bar{\mu}\,$ is the smallest possible orbit for weights 
having exactly 4 nonzero coefficients.
Indeed, by Lemma~\ref{orban},
\begin{eqnarray} \label{an1}
|W\bar{\mu}| & = & \frac{(\ell +1)!}{(\ell -3)!}\,=\,(\ell +1)\,\ell
\,(\ell -1)\,(\ell -2)\qquad \mbox{and}\nonumber 
\end{eqnarray}
\begin{eqnarray} \label{an2}
|W\nu| & = & \frac{(\ell + 1)!}
{i_1!\,(i_2-i_1)!\,(i_3-i_2)!\,(i_4-i_3)!
\,(\ell-i_4+1)!}.\nonumber
\end{eqnarray}
Thus, as we prove in Appendix~\ref{appineq1}, for all $\,\ell\,\geq\,4$.
\begin{eqnarray}\label{ineq1}
\frac{|W\nu|}{|W\bar{\mu}|}\,=\,\frac{(\ell - 3)!}
{i_1!\,(i_2-i_1)!\,(i_3-i_2)!\,(i_4-i_3)!\,(\ell-i_4+1)!}
\,\geq\,1\,. %\mbox{for all}\;\ell\,\geq\,4.
\end{eqnarray} 

Now for $\,\ell\geq 4$, 
$\,|W\bar{\mu}|\,-\,\ell(\,\ell\,+\,1)^2\,=\,
\ell\,(\ell^3 -2\ell^2 -\ell +2)\,-\,\ell(\,\ell\,+\,1)^2\,=\,
%\ell\,(\ell^3 -3\ell^2 -3\ell +1)\,=\,
\ell\,[\ell(\ell^2 -3\ell -3)\,+\,1]\,>\,0\,$
(as $\,\ell^2 -3\ell -3\,$ is a
positive increasing function on $\,\ell$, provided $\,\ell\,\geq\,4\,$).
Therefore, $\,|W\nu|\,\geq\,|W\bar{\mu}|\,>\,\ell(\ell + 1)^2$,
proving the claim.

\underline{Step 2}: As $\,\mu\,$ has at least 
4 nonzero coefficients, by Lemma~\ref{remark}, 
\mbox{$\,|W\mu|\,>\,|W\nu|\,$} for some weight $\,\nu\,$ with exactly 
4 nonzero coefficients. Hence, by Step 1, $\,|W\mu|\,>\,\ell(\ell + 1)^2$.

Therefore, if $\,\mu\,\in\,{\cal X}_{++}(V)\,$ is a good weight with 
at least 4 nonzero coefficients, 
then $\,s(V)\,\geq\,|W\mu|\,>\,\ell(\ell + 1)^2\,$
and, by~(\ref{allim}), $\,V\,$ is not an exceptional module.
If $\,\mu\,\in\,{\cal X}_{++}(V)\,$ is a bad weight, then write 
$\,\mu\,=\,b_{1}\,\omega_{i_1}\,+\,b_{2}\,\omega_{i_2}\,
+\,b_{3}\,\omega_{i_3}\,+\,b_{4}\,\omega_{i_4}\,+\,\cdots$, (where
$\,b_1,\,b_2,\,b_3,\,b_4\,\geq 2\,$ are the first 4 nonzero coefficients of 
$\,\mu\,$ and $\,i_1<i_2<i_3<i_4$). Thus, the good weight
$\,\mu_1 \,=\,\mu-(\alpha_{i_2}+\cdots+\alpha_{i_3})\,=\,
b_{1}\,\omega_{i_1}\,+\,\omega_{i_2-1}\,+\,(b_{2}-1)\,\omega_{i_2}\,
+\,(b_{3}-1)\,\omega_{i_3}\,+\,\omega_{i_3+1}\,+\,b_{4}\,\omega_{i_4}\,+
\,\cdots\,\in\,{\cal X}_{++}(V)$. As $\,\mu_1\,$ has at least 4
nonzero coefficients, by Step 2, $\,s(V)\,\geq\,|W\mu_1|\,>\,\ell(\ell +
1)^2\,$. Hence, again by~(\ref{allim}), $\,V\,$ is not an exceptional
module. This proves the lemma.
\end{proof}

Now we can prove \\ 
{\bf Lemma~\ref{nonexcepnonfree}.} {\it 
Let $\,\mathfrak g\,=\,\mathfrak{sl}(np,K),\,$ where 
$\,p>2\,$ and $\,np\geq 4.\,$ Then the Steinberg module 
$\,V\,=\,E((p-1)\rho)\,$ is not an exceptional module, but it is non-free.}
\\ \noindent
\begin{proof}By Lemma~\ref{4coef}, $\,V\,$ is not exceptional. To prove that
$\,V\,$ is non-free, let $\,z\,=\,h_{\alpha_1}\,+\,2\,h_{\alpha_2}\,+ 
\,3\,h_{\alpha_3}\,+\,\cdots\,+\,(np-1)\,h_{\alpha_{np-1}}\,$.

First we claim that $\,z\,$ is a nonzero central element of $\,\mathfrak g.\,$
Indeed, note that 
$\,\alpha_j(z)\,=\,(j-1)\cdot(-1)\,+\,j\cdot 2\,+\,(j+1)\cdot(-1)\,
=\,0\;$ for all $\,j\,=\,2,\ldots,np-2\,$. 
Also $\,\alpha_1(z)\,=\,\alpha_{np-1}(z)\,=\,0.\,$ Hence,
$\,[z,\,e_{\alpha_j}]\,=\,\alpha_j(z)\,e_{\alpha_j}\,=\,0$,
for all $\,j\,$. Therefore, $\,[z,\,e_{\alpha}]\,=\,0\,$ for
all $\,\alpha\in R$. But $\,\{ e_{\alpha}\,/\,\alpha\in R\}\,$ generate
$\,\mathfrak g\,$, hence $\,[z,\,y]\,=\,0\,$ for all $\,y\in\mathfrak
g\,$, proving the claim.

Now we prove that $\,z\in \mathfrak g_v\,$ for all $\,v\in V$.  
We have
\begin{align*}
(p-1)\rho(z)\, &=\,(p-1)\,\sum_{i=1}^{np-1}\,\omega_i(z)\,=\,
(p-1)\,\sum_{i=1}^{np-1}\,i\,\frac{2(\omega_i,\,{\alpha_i})} 
{(\alpha_i,\,\alpha_i)}\\
&=\,(p-1)\,\sum_{i=1}^{np-1}\,i\,=\,\frac{(p-1)\,(np-1)\,np}{2}\,\equiv\,0
\pmod p.
\end{align*}
The other weights of $\,V\,=\,E((p-1)\rho)\,$ are of the form
$\,(p-1)\rho\,-\,\displaystyle\sum_{\beta>0}\,n_{\beta}\beta\,$, where
$\,n_{\beta}\in\mathbb Z^+\,$. Hence
$\,\mu(z)\,=\,0,\,$ for each $\,\mu\in{\cal X}_{++}(V)$.  
The vectors $\,v\in V\,$ are of the form
$\,v\,=\,\sum\,v_{\mu},\,$ with $\,v_{\mu}\in V_{\mu}\,$. Thus
$\,z\cdot v\,=\,\sum\,z\cdot v_{\mu}\,=\,\sum\mu(z)\,v_{\mu}\,=\,0,\,$ 
whence $\,z\in \mathfrak g_v\,$ for all $\,v\in V$.
Thus, $\,V\,$ is a non-free module.
\end{proof}

%\pagebreak

{\bf From now on, by Lemma~\ref{4coef}, it suffices to consider
$\,A_{\ell}(K)$-modules $\,V\,$ such that $\,{\cal X}_{++}(V)\,$
contains only weights with at most $\,3\,$ nonzero coefficients.}
The following lemmas reduce this assumption even further.
\begin{lemma}\label{3orless}
Let $\,\ell\,\geq\,4$. If $\,V\,$ is an $\,A_{\ell}(K)$-module such
that $\,{\cal X}_{++}(V)\,$ satisfies
any of the following conditions, then $\,V\,$ is not an exceptional
$\,\mathfrak g$-module.

(a) $\,{\cal X}_{++}(V)\,$ contains at least $2$ good weights
with $3$ nonzero coefficients.

(b) $\,{\cal X}_{++}(V)\,$ contains a bad weight
with $3$ nonzero coefficients.

(c) $\,{\cal X}_{++}(V)\,$ contains just one good weight
with $3$ nonzero coefficients, (at least) $2$ good weights
with $2$ nonzero coefficients and a nonzero minimal
weight or any other good weight.
\end{lemma}\noindent
\begin{proof}
Let  
$\,\mu\,=\,a\,\omega_{i_1}\,+\,b\,\omega_{i_2}\,+\,c\,\omega_{i_3}\,$
be a weight with $\,i_1<i_2<i_3\,$ and $\,a,\,b,\,c\,\geq\,1$. 
For $\,\ell\geq 4,\,$ the orbit size of $\,\mu\,$ is
\begin{equation} \label{orb3c}
|W\mu| =\frac{(\ell + 1)!}{i_1!\,(i_2-i_1)!\,
(i_3-i_2)!\,(\ell-i_3+1)!}\,\geq \, \frac{(\ell\,+\,1)!}{(\ell\,-\,2)!}\,
=\,\ell\,(\ell^2\,-\,1).\;
\end{equation}
This inequality is proved in Appendix~\ref{apporb3c}.

(a) If $\,{\cal X}_{++}(V)\,$ has at least $2$ good weights with exactly
$\,3\,$ nonzero coefficients, then 
$\,s(V)\,-\,\ell(\ell+1)^2\,\geq\,2\,\ell\,(\ell^2\,-\,1)\,-\,
(\ell^3\,+\,2\ell^2\,+\,\ell)\,=\,\ell\,(\ell^2\,-\,2\ell\,-\,3)\,>\,0$,
for all $\,\ell\,\geq\,4$. Hence, by~\eqref{allim}, $\,V\,$ is not exceptional.

(b) If $\,\mu\,=\,a\,\omega_{i_1}\,+\,b\,\omega_{i_2}\,+\,c\,\omega_{i_3}\,
\in\,{\cal X}_{++}(V)\,$ is bad, then $\,a,\,b,\,c\,\geq 2.\,$
Thus,
\mbox{$\mu_1=\mu-(\alpha_{i_1}+\cdots
+\alpha_{i_2})\,=\,\omega_{i_1-1}\,+\,(a-1)
\omega_{i_1}\,+\,(b-1)\omega_{i_2}\,+\,\omega_{i_2+1}\,
+\,c\,\omega_{i_3}\,$,}\linebreak
\mbox{$\mu_2=\mu-(\alpha_{i_2}+\cdots +\alpha_{i_3})\,=\,a\omega_{i_1}\,+
\,\omega_{i_2-1}\,+\,(b-1)\omega_{i_2}\,+\,(c-1)\omega_{i_3}\,
+\,\omega_{i_3+1}\,$}\linebreak
are good weights in $\,{\cal X}_{++}(V)$,
both with at least $3$ nonzero coefficients, hence satisfying (a).

(c) %Let $\,\mu\,$ be the only good weight in $\,{\cal X}_{++}(V)\,$
%having exactly 3 nonzero coefficients. Then, by~(\ref{orb3c}), 
%$\,|W\mu|\,\geq\,\ell\,(\ell^2\,-\,1)$. Now if
If $\,\nu=a\omega_{i_1}\,+\,b\omega_{i_2}\,$ is a weight
with $2$ nonzero coefficients, then $\,|W\nu|\,\geq\,
\ell\,(\ell\,+\,1)$, for all $\,\ell\geq 4$. Indeed, 
by Appendix~\ref{apporb2c}, 
\begin{eqnarray} \label{orb2c}
|W\nu| & = & \frac{(\ell + 1)!}{i_1!\,(i_2-i_1)!\,
(\ell-i_2+1)!}\,\geq\,\frac{(\ell\,+\,1)!}{(\ell\,-\,1)!}
\,=\,\ell\,(\ell\,+\,1).
\end{eqnarray} 
Thus, if $\,{\cal X}_{++}(V)\,$ contains just one good weight with $3$
nonzero coefficients, at least $2$ good weights with
$2$ nonzero coefficients and a minimal weight $\,\omega_i\neq 0\,$ 
or any other good weight (whose orbit size is $\,\geq 1$) then,
by~\eqref{orb3c} and~\eqref{orb2c}, $\,s(V)\,\geq
\,\ell\,(\ell^2\,-\,1)\,+\,2\,\ell\,(\ell\,+\,1)\,+\,1
\, =\,\ell\,(\,\ell\,+\,1\,)^2\,+\,1\,>\,\ell\,(\,\ell\,+\,1\,)^2\,$
and, by~\eqref{allim}, $\,V\,$ is not an exceptional module. This
proves the lemma.
\end{proof}

\begin{lemma}\label{nlf3c}
Let $\,\ell\,\geq\,4.\,$ If $\,V\,$ is an $\,A_{\ell}(K)$-module such
that $\,{\cal X}_{++}(V)\,$ contains a good
weight with $3$ nonzero coefficients, then $\,V\,$ is not an exceptional
$\,\mathfrak g$-module.
\end{lemma}\noindent
\begin{proof}
Let $\,\mu\,=\,a\,\omega_{i_1}\,+\,b\,\omega_{i_2}\,+\,c\,\omega_{i_3},\,$
with $\,1\,<\,i_1\,<\,i_2\,<\,i_3\,$ and $\,a,\,b,\,c\,$ nonzero, be a
good weight in $\,{\cal X}_{++}(V)$. 
%Bad weights of this form were considered in Lemma~\ref{3orless}(b).

\underline{Case 1}: Suppose $\,1<i_1\,<\,i_2\,<\,i_3\,<\,\ell\,$.\\
For $\,a\geq 1,\,b\geq 1,\,c\geq 1,\,$ 
$\,\mu\,=\,a\omega_{i_1}\,+\,b\omega_{i_2}\,+\,c\omega_{i_3}\;$ and 
$\,\mu_1\,=\mbox{$\mu-(\alpha_{i_1} +\cdots +\alpha_{i_3})$}
=\,\omega_{i_1-1}\,+\,(a-1)\,\omega_{i_1}\,+\,b\,\omega_{i_2}\,+
\,(c-1)\omega_{i_3}\,+\,\omega_{i_3+1}\;$
are good weights in $\,{\cal X}_{++}(V)$.
Hence, Lemma~\ref{3orless}(a) applies.

\underline{Case 2}: Suppose $\,1<i_1 < i_2 < i_3=\ell\;$ (or graph-dually
 $\,1=i_1 < i_2 < i_3<\ell\,$).

(i) For $\,a\geq 2,\,b\geq 1,\,c \geq 1\,$ (or 
$\,a\geq 1,\,b\geq 1,\,c\geq 2\,$),
$\,\mu\,=\,a\omega_{i_1}\,+\,b\omega_{i_2}\,+\,c\omega_{\ell}\;$ and
$\,\mu_1\,=\,\mu-(\alpha_{i_1} +\cdots +\alpha_{\ell})\,
=\,\omega_{i_1-1}\,+\,(a-1)\,\omega_{i_1}\,+\,b\,\omega_{i_2}\,+
\,(c-1)\omega_{\ell}\;$ are good weights in $\,{\cal X}_{++}(V)$. 
Hence, Lemma~\ref{4coef} or~\ref{3orless}(a) applies.

(ii) For $\,a=c=1,\,b\geq 2$,
$\,\mu\,=\,\omega_{i_1}\,+\,b\omega_{i_2}\,+\,\omega_{\ell}\;$ and 
$\,\mu_1=\mbox{$\mu-(\alpha_{i_1} +\cdots +\alpha_{i_2})$}
=\,\omega_{i_1-1}\,+\,(b-1)\,\omega_{i_2}\,+\,\omega_{i_2+1}\,+\, 
\omega_{\ell}\;$ are both good weights in $\,{\cal X}_{++}(V)$.
Thus, Lemma~\ref{4coef} or~\ref{3orless}(a) applies.

(iii) For $\,a=b=c=1,\,$
$\,\mu\,=\,\omega_{i_1}\,+\,\omega_{i_2}\,+\,\omega_{\ell}\,$ and 
$\,\mu_1\,=\,\mu\,-\,(\alpha_{i_1} +\cdots +\alpha_{i_2})\,
=\,\omega_{i_1-1}\,+\,\omega_{i_2+1}\,+\,\omega_{\ell}\,\in\,
{\cal X}_{++}(V)$. 
For $\,i_2+1\,<\,\ell,\,$ $\,\mu_1\,$ has 3 nonzero coefficients and 
Lemma~\ref{3orless}(a) applies. For $\,i_2+1\,=\,\ell,\,$ 
$\,\mu =\omega_{i_1}\,+\,\omega_{\ell-1}\,+\,\omega_{\ell},\;$ 
$\,\mu_1=\omega_{i_1-1}\,+\,2\omega_{\ell},\,$ 
$\,\mu_2=\mu_1\,-\,\alpha_{\ell}\,
=\,\omega_{i_1-1}\,+\,\omega_{\ell-1},\;$ 
$\,\mu_3\,=\,\mu_2\,-\,(\alpha_{i_1-1} +\cdots +\alpha_{\ell-1})\,
=\,\omega_{i_1-2}\,+\,\omega_{\ell}\,\in\,{\cal X}_{++}(V)\,$ and
Lemma~\ref{3orless}(c) applies. 

In any of that above cases, $\,V\,$ is not an exceptional module.

%\pagebreak

\underline{Case 3}: Suppose 
$\,i_1=1,\;i_3\,=\,\ell\;$ and $\,2<\,i_2\,<\ell-1\,$.

(i) For $\,a\,\geq\,1,\,b\,\geq\,2,\,c\,\geq\,1,\;$
$\,\mu\,=\,a\,\omega_{1}\,+\,b\,\omega_{i_2}\,+\,c\,\omega_{\ell}\;$ and 
$\,\mu_1=\mu\,-\,\alpha_{i_2}\,=\,a\,\omega_{1}\,+\,\omega_{i_2-1}\,
+\,(b-2)\,\omega_{i_2}\,+\,\omega_{i_2+1}\,+\,c\,\omega_{\ell}
\,\in\,{\cal X}_{++}(V)\,$ and Lemma~\ref{4coef} applies.

(ii) For $\,a\,\geq\,2,\,b\,=\,1,\,c\,\geq\,1\,$ (or
$\,a\,\geq\,1,\,b\,=\,1,\,c\,\geq\,2\,$),
$\,\mu =a\,\omega_{1}\,+\,\omega_{i_2}\,+\,c\omega_{\ell},\;$
$\,\mu_1\,=\,\mu\,-\,(\alpha_{1}+\cdots+\alpha_{i_2})\,
=\,(a-1)\,\omega_{1}\,+\,\omega_{i_2+1}\,+\,c\,\omega_{\ell},\;$
$\,\mu_2\,=\,\mu\,-\,(\alpha_{i_2}+\cdots+\alpha_{\ell})\,
=\,a\,\omega_{1}\,+\,\omega_{i_2-1}\,+\,(c-1)\,\omega_{\ell}\;$ and
$\,\mu_3\,=\,\mu\,-\,(\alpha_{1}+\cdots+\alpha_{\ell})\,
	    =\,(a-1)\,\omega_{1}\,+\,\omega_{i_2}\,+\,(c-1)\,\omega_{\ell}
\,\in\,{\cal X}_{++}(V).\,$ Thus,
$\,{\cal X}_{++}(V)\,$ has at least 2 good weights with 3 nonzero
coefficients and Lemma~\ref{3orless}(a) applies.

(iii)  For $\,a=b=c=1,\;$
$\,\mu\,=\,\omega_{1}\,+\,\omega_{i_2}\,+\,\omega_{\ell},\;$ 
$\,\mu_1\, =\,\mbox{$ \mu-(\alpha_{1}+\cdots+\alpha_{i_2})$}\,
=\,\omega_{i_2+1}\,+\,\omega_{\ell},\;$
$\,\mu_2\, =\,\mu-(\alpha_{i_2}+\cdots+\alpha_{\ell})\,
=\,\omega_1\,+\,\omega_{i_2-1}\;$ and
$\mu_3\,=\,\mbox{$\mu_2-(\alpha_1+\cdots+$}\alpha_{i_2-1})\,
=\,\omega_{i_2}\;$
are all good weights in $\,{\cal X}_{++}(V).\,$ Hence, by
Lemma~\ref{3orless}(c), $\,V\,$ is not an exceptional module.

{\it Hence, if $\,{\cal X}_{++}(V)\,$ contains a good weight
$\,\mu\,=\,a\,\omega_{i_1}\,+\,b\,\omega_{i_2}\,+\,c\,\omega_{i_3}\,$
(with $\,3\,$ nonzero coefficients), then we can assume 
$\,i_1=1,\;i_2=2,\;i_3\,=\,\ell\,$ (or graph-dually 
$\,i_1=1,\;i_2=\ell-1,\;i_3=\ell\,$).} This case is treated as follows.

\underline{Case 4}: $\,i_1=1,\;i_2=2,\;i_3\,=\,\ell\;$ (or graph-dually 
$\,i_1=1,\;i_2=\ell-1,\;i_3=\ell\,$). 

(i) For $\,a\geq 1,\,b\geq 2,\,c\geq 1\,$ (or
$\,a\geq 2,\,b \geq 1,\,c\geq 1\,$), 
$\,\mu\,=\,a\omega_{1}\,+\,b\omega_{2}\,+\,c\omega_{\ell},\;$
$\mu_1\,=\,\mu\,-\,(\alpha_1 +\alpha_2)\,=\,(a-1)\,\omega_{1}\,+\,
\mbox{$(b-1)\omega_2$}\,+\,\omega_{3}\,+\,c\omega_{\ell}\,
\in\,{\cal X}_{++}(V)$.
%As $\,\ell\geq 4$, Lemma~\ref{4coef} or~\ref{3orless}(a) applies.
For $\,a\geq 1,\,b\geq 1,\,c \geq 2,\,$ also  
%$\,\mu\,=\,a\omega_{1}\,+\,b\omega_{2}\,+\,c\omega_{\ell},\;$
$\,\mu_2\,=\,\mu-\alpha_{\ell}\,=\,a\,\omega_{1}\,+\,b\omega_2\,
+\,\omega_{\ell-1}\,+\,(c-2)\omega_{\ell}\,\in\,{\cal
X}_{++}(V)$. In both cases, as $\,\ell\geq
4$, Lemma~\ref{4coef} or~\ref{3orless}(a) applies.

{\it Therefore, if $\,{\cal X}_{++}(V)\,$ contains a good weight
$\,\mu\,$ with $\,3\,$ nonzero coefficients, then we can assume that 
$\,\mu\,=\,\omega_{1}\,+\,\omega_{2}\,+\,\omega_{\ell}\,$.}

(ii) Let $\,\mu\,=\,\mu_0\,=\,\omega_{1}\,+\,\omega_{2}\,+\,\omega_{\ell}\in
{\cal X}_{++}(V)$. Then $\,\mu_1\,=\,\omega_3\,+\,\omega_{\ell}\,$,
$\,\mu_2\,=\,\mu_1-(\alpha_3\,+\,\cdots\,+\,\alpha_{\ell})\,=\,\omega_2\in
{\cal X}_{++}(V)$. For $\,\ell\,\geq\,6\,$ and any prime $\,p$, 
\begin{align*}
s(V) &\geq\,(\ell + 1)\,\ell\,(\ell -1)\,+\,\displaystyle\frac{(\ell +
1)\,\ell\,(\ell-1)(\ell-2)}{3!}\,+\,\frac{(\ell\,+\,1)\,\ell}{2!}\\
& =\,\displaystyle\frac{(\ell + 1)\,\ell\,(\ell^2+\ell + 3)}{6}\,>\,
\ell\,(\ell+1)^2\,.
\end{align*}
Hence, by~\eqref{allim}, $\,V\,$ is not an exceptional module.

For $\,\ell=5\,$ and $\,p\geq 2\,$, $\,|R_{long}^+-R_{\mu,p}^+|\geq
8\,$. Thus $\,r_p(V)\,\geq\,\displaystyle\frac{120\cdot
8}{30}\,=\,32\,>\,30\,$.
For $\,\ell=4$, $\,\mu\,=\,\omega_1\,+\,\omega_2\,+\,\omega_4\,$,
$\,\mu_1\,=\,\omega_3\,+\,\omega_{4}\,$. For $\,p\geq 2\,$, 
$\,|R_{long}^+-R_{\mu,p}^+|\geq 6,\;|R_{long}^+-R_{\mu,p}^+|\geq 4$. 
Thus $\,r_p(V)\,\geq\,\displaystyle\frac{60\cdot
6}{20}\,+\,\frac{20\cdot 4}{20}\,=\,22\,>\,20\,$. 
Hence, for $\,\ell=4,\,5,\,$ by~\eqref{rral}, $\,V\,$ is not an
exceptional module. This proves the lemma.
\end{proof}

The main theorem for groups of type $\,A_{\ell}\,$ is as follows. Its proof 
is given in Appendix~\ref{appal}.% where we also prove lemmas
%for rank $\,\ell\,=\,1,\,2,\,3,\,4\,$ cases.

\begin{theorem}\label{anlist}
%Let $\,\ell\geq 4.\,$ The exceptional
If $\,V\,$ is an infinitesimally irreducible
$\,A_{\ell}(K)$-module with highest weight listed
in Table~\ref{tablealall} (p. \pageref{tablealall}), 
then $\,V\,$ is an exceptional $\,\mathfrak g$-module.
%If $\,V\,$ is an $\,A_{\ell}$-module having highest weight according
%to Table~\ref{leftan}, then $\,V\,$ is not classified.
If $\,V\,$ has highest weight different from the ones listed in Tables
~\ref{tablealall} or~\ref{leftan} (p. \pageref{leftan}), then
$\,V\,$ is not an exceptional $\,\mathfrak g$-module.
\end{theorem}

%\newpage

%\input{bnidea.tex}

\subsubsection*{\ref{lact}.2. Algebras of Type $\,B\,$}\label{blresults}
\addcontentsline{toc}{subsubsection}{\protect\numberline{\ref{lact}.2}
Algebras of Type $\,B\,$}

By Theorem~\ref{???}, any $\,B_{\ell}(K)$-module $\,V\,$ such that 
\begin{equation}\label{bllim}
s(V)\,=\,\displaystyle\sum_{\stackrel{\scriptstyle\mu\,good}
{\scriptstyle \mu\in{\cal X}_{++}(V)}}\,m_{\mu}\,|W\mu|\,
>\,\left\{
\begin{array}{ll}
2\,\ell^3 & \mbox{if}\;\ell\geq 5\\
8\,\ell^2 & \mbox{if}\;\ell\leq 4.
\end{array}\right.
\end{equation}
or
\begin{equation}\label{rrbl}
r_p(V)\,=\,\sum_{\stackrel{\scriptstyle{\mu\in{\cal X}_{++}(V)}}
{\scriptstyle \mu\;good}}\,m_{\mu}\,
\frac{|W\mu|}{|R_{long}|}\,|R_{long}^+-R^+_{\mu,p}|\,>\,2\,\ell^2
\end{equation}
is not an exceptional $\,\mathfrak g$-module.
Note that for $\,\ell =4$, $\,8\ell^2=2\ell^3\,$.
For groups of type $\,B_{\ell}$, $\,|W|\,=\,2^{\ell}\,\ell !$, 
$\,|R_{long}|\,=\,|D_{\ell}|\,=\,2\,\ell\,(\ell-1)\,$ and
$\,|R_{long}^+|\,=\,\ell\,(\ell-1).\,$ Recall that $\,p\geq 3$. 
The orbit sizes of weights are given by Lemma~\ref{orbbncndn}.

Let $\,V\,$ be a $\,B_{\ell}(K)$-module. We prove a reduction lemma.

\begin{lemma}\label{3coefbn}
Let $\,\ell\,\geq\,4$. If $\,{\cal X}_{++}(V)\,$ contains a
weight with 3 or more nonzero coefficients, then $\,V\,$
is not an exceptional $\,\mathfrak g$-module.
\end{lemma}\noindent
\begin{proof} 
I) Let $\,\mu=b_1\,\omega_1\,+\,\cdots\,+\,b_{\ell}\,\omega_{\ell}\,\in 
\,{\cal X}_{++}(V)\,$ be a good weight with at least 3 nonzero coefficients. 
We claim that $\,s(V)\,>\,2\ell^3\,$. 

\underline{Step 1}: 
Write $\,\mu\,=\,b_1\,\omega_{i_1}\,+\,b_2\,\omega_{i_2}\,+\,
b_3\,\omega_{i_3}\,+\cdots\,$, where $\,1\leq\,i_i<i_2<i_3\,$ are the
first 3 nonzero coefficients of $\,\mu$. By 
Lemma~\ref{remark},  $\,|W\mu|\,\geq\,|W\bar{\mu}|,\,$ where
$\,\bar{\mu}\,=\,d_1\,\omega_{i_1}\,+
\,d_2\,\omega_{i_2}\,+\,d_3\,\omega_{i_3}\,$ with $\,d_j\neq 0\,$ for
all $\,1\leq j\leq 3$. Therefore, to prove the claim, it
suffices to prove that $\,|W\bar{\mu}|\,>\,2\,\ell^3$, for any weight
$\,\bar{\mu}\,$ with exactly 3 nonzero coefficients.

\underline{Step 2}: Let $\,\bar{\mu}\,=\,d_1\,\omega_{i_1}\,+
\,d_2\,\omega_{i_2}\,+\,d_3\,\omega_{i_3}\,$ with $\,d_j\neq 0\,$ for
$\,1\leq j\leq 3\,$. We prove that $\,|W\bar{\mu}|\,>\,2\,\ell^3\,$ by
considering the different possibilities for $\,(i_1,\,i_2,\,i_3)$. 

(a) Let $\,(i_1,\,i_2,\,i_3)\,=\,(1,\,2,\,3)\,$. In this case denote
$\,\bar{\mu}\,$ by $\,\nu=c_1\,\omega_1\,+\,c_2\omega_2\,+\,c_{3}\,
\omega_{3}\,$, with $\,c_1,\,c_2,\,c_3\,$ all nonzero. Then
$\,|W\nu|\,=\,2^3\,\ell\,(\ell -1)\,(\ell -2)\,$. Thus, for
$\,\ell\geq 4$, $\,|W\nu|\,-\,2\,\ell^3 \,=\,6\ell^2\,(\ell-4)\,
+\,16\ell\,>\,0\,$ whence  $\,|W\nu|\,>\,2\,\ell^3 \,$.

{\bf Note}: For $\,\ell =4,\,$ any choice of $\,(i_1,\,i_2,\,i_3)\,$ implies
$\,|W\bar{\mu}|\,=\,2^6\cdot 
3\,>\,2^7=8\cdot 4^2$. Thus, from now on let $\,\ell\geq 5\,$ and 
$\,{\mathbf\nu=c_1\,\omega_1\,+\,c_2\omega_2\,+\,c_{3}\,
\omega_{3}}\,$, with $\,c_i\neq 0$.

(b)i) Let $\,1\leq i_1\,<\,i_2\,<\,i_3\,<\,\ell -1\,$ and 
$\,\bar{\mu}\,=\,d_1\,\omega_{i_1}\,+\,d_{2}\,\omega_{i_2}\,
+\,d_{3}\,\omega_{i_3}$, with $\,d_{j}\neq 0$. Then
$\,|W\bar{\mu}|\, =\, \displaystyle 2^{i_3}\,\binom{i_2}{i_1}\,
\binom{i_3}{i_2}\,\binom{\ell}{i_3}\,$.
%\frac{2^{\ell}\,\ell !}{i_1!\,(i_2 -i_1)!\,(i_3 -i_2)!\,2^{(\ell -i_3)}
%\,(\ell-i_3)!}\,$. 
Note that for $\,1\leq i_1< i_2 < i_3,\;3\leq i_3\leq\ell-3\,$ and 
$\,\ell\geq 6,\,$ we have $\,\displaystyle\binom{\ell}{i_3}\geq
\binom{\ell}{3}\,$ and $\,\displaystyle\binom{i_2}{i_1}\,
\binom{i_3}{i_2}\geq 6\,$. Thus,  
\[
\frac{|W\bar{\mu}|}{|W{\nu}|} \,= \, 
\frac{2^{i_3-3}}{\ell (\ell - 1)(\ell - 2)}\,\binom{i_2}{i_1}\,
\binom{i_3}{i_2}\,\binom{\ell}{i_3}
\,\geq\,\frac{2^{i_3-3}\cdot 6}{\ell (\ell - 1)(\ell -2)}\,\binom{\ell}{3}
\, =\, 2^{i_3-3}\geq 1\,.
\]
Hence, %for $\,1\leq i_1< i_2 < i_3,\,3\leq i_3\leq\ell-3\,$ and
for $\,\ell\geq 6,\,$ $\,|W\bar{\mu}|\,\geq\,|W{\nu}|\,>\,2\ell^3\,$.

ii) If $\,1\leq i_1< i_2 < i_3=\ell -2\,$ then, for $\,2\leq i_2\leq
\ell-3\,$ and $\,\ell\geq 5,\,$
\begin{eqnarray}
\frac{|W\bar{\mu}|}{|W{\nu}|} & = &\frac{2^{\ell
-6}}{(\ell-2)}\,\binom{i_2}{i_1}\,\binom{\ell -2}{i_2}\,
\geq\, \frac{2^{\ell
-6}}{(\ell-2)}\,\binom{i_2}{i_1}\,\binom{\ell-2}{2}
 \geq 2^{\ell-6}\,(\ell-3) \,\geq\,1\,.\nonumber
\end{eqnarray}

{\bf Therefore, if $\,\bar{\mu}\,=\,d_1\,\omega_{i_1}\,+\,d_{2}\,\omega_{i_2}\,
+\,d_{3}\,\omega_{i_3}$, with $\,d_{j}\neq 0\,$ and 
$\,1\leq i_1\,<\,i_2\,<\,i_3\,\leq\,\ell -2\,$, then
$\,|W\bar{\mu}|\,\geq\,|W{\nu}|\,>\,2\ell^3\,$. Hence, we can assume
$\,\ell -1\leq \,i_3\,$.} We deal with this case in the sequel.

(c) For $\,1\leq i_1< i_2 <i_3\,=\,\ell -1,\,$ 
$\,\bar{\mu}\,=\,b_1\,\omega_{i_1}\,+\,b_{2}\,\omega_{i_2}\,
+\,b_{3}\,\omega_{\ell -1}\,$ and \\
$\,|W\bar{\mu}|\, = \,\displaystyle 2^{\ell-1}\,\ell\,\binom{i_2}{i_1}\, 
\binom{\ell -1}{i_2}\,$.
%\frac{2^{\ell-1}\,\ell !}{i_1!\,(i_2 -i_1)!\,(\ell -i_2-1)!}\,$. 
Thus, for $\,2\leq i_2\leq  \ell-2\,$ and $\,\ell\geq 4,\,$
\begin{eqnarray}
\frac{|W\bar{\mu}|}{|W{\nu}|} & = & \frac{2^{\ell-4}\;\ell}{\ell\,(\ell -
1)\,(\ell -2)}\,\binom{i_2}{i_1}\,\binom{\ell -1}{i_2} \geq 
2^{\ell-4}\,\geq\,1\,.\nonumber
\end{eqnarray}

d)i) Let $\,1\leq i_1< i_2 <i_3\,=\,\ell$. Then  
$\,\mu\,=\,\bar{\mu}\,=\,b_1\,\omega_{i_1}\,+\,b_{2}\,\omega_{i_2}\,
+\,b_{3}\,\omega_{\ell}\;$ and 
$\,|W\bar{\mu}|\, = \,\displaystyle 2^{\ell}\,\ell\,\binom{i_2}{i_1}\, 
\binom{\ell -1}{i_2}\,$. 
%\frac{2^{\ell}\,\ell !} {i_1!\,(i_2 -i_1)!\,(\ell -i_2)!}\,$. 
Thus, for $\,3\leq i_2\leq\ell-2\,$ and $\,\ell\geq 5,\,$
\begin{eqnarray}
\frac{|W\bar{\mu}|}{|W{\nu}|} & = & \frac{2^{\ell-3}}{\ell(\ell - 1)
(\ell -2)}\,\binom{i_2}{i_1}\,\binom{\ell}{i_2} \geq 2^{\ell-4}\, \geq\,1\,.
\nonumber
\end{eqnarray}

ii) For $\,i_2\,=\,2\,$ (hence $\,i_1=1\,$) and $\,i_3=\ell$,  
$\,\mu\,=\,\bar{\mu}\,=\,b_1\,\omega_{1}\,+\,b_{2}\,\omega_{2}\,
+\,b_{3}\,\omega_{\ell}\,$. Then, for $\,\ell\geq 4,\,$
$\,|W\mu|\,=\,2^{\ell}\,\ell(\ell-1)\,>\,2\ell^3\,$.

iii) For $\,1\leq i_1 < i_2\,=\,\ell -1,\,i_3=\ell$, 
$\,\mu\,=\,\bar{\mu}\,=\,b_1\,\omega_{i_1}\,+\,b_{2}\,\omega_{\ell -1}\,
+\,b_{3}\,\omega_{\ell}\;$ and 
$\,|W\bar{\mu}|\,=\,2^{\ell}\,\ell\,\displaystyle\binom{\ell-1}{i_1}\,$.
Hence, for $\,1\leq i_1 \leq \ell-2\,$ and $\,\ell\geq 4,\,$
$\,|W\bar{\mu}|\,\geq\,2^{\ell}\,\ell\,(\ell-1)\,>\,2\ell^3\,$.

{\bf Therefore, if $\,{\cal X}_{++}(V)\,$ contains a good weight $\,\mu\,$
with at least 3 nonzero coefficients, then
$\,s(V)\,\geq\,|W\mu|\,>\,2\ell^3\,$. Hence, by~\eqref{bllim}, $\,V\,$
is not an exceptional module.}

II) Suppose that $\,\mu\,$ is a bad weight in $\,{\cal X}_{++}(V)$.
This implies $\,b_j\geq 3\,$ for all nonzero coefficients. 
We claim that it is possible to produce, from $\,\mu,\,$ a
good weight $\,\mu_1\in{\cal X}_{++}(V)\,$ having 3 or more nonzero 
coefficients. Indeed:

(a) if $\,1\leq\,i_i<i_2<i_3\,<\,\ell -1$, then 
$\,\mu\,=\,b_1\,\omega_{i_1}\,+\,b_2\,\omega_{i_2}\,+\,
b_3\,\omega_{i_3}\,+\cdots\;$ and
$\,\mu_1 = \mu-(\alpha_{i_2}+\cdots +\alpha_{i_3}) 
 = b_1\,\omega_{i_1}\,+\,\omega_{i_2-1}\,+\,(b_2-1)\,\omega_{i_2}\,+\,
(b_3-1)\,\omega_{i_3}\,+\,\omega_{i_3+1}\,+\,\cdots\;\in\,{\cal
X}_{++}(V)$.

(b) if $\,1\leq\,i_i<i_2<i_3\,=\,\ell -1$, then 
$\,\mu\,=\,b_1\,\omega_{i_1}\,+
\,b_2\,\omega_{i_2}\,+\,b_3\,\omega_{\ell -1}\,+\,\cdots\;$ and
$\,\mu_1 = \mu-(\alpha_{i_1}+\cdots +\alpha_{i_2})
 = \omega_{i_1-1}\,+\,(b_1-1)\,\omega_{i_1}\,+\,(b_2-1)\omega_{i_2}\,+\,
\omega_{i_2+1}\,+\,b_3\,\omega_{\ell -1}\,+\,\cdots\;\,\in\,{\cal
X}_{++}(V)$.

(c) if $\,1\leq\,i_i<i_2<i_3\,=\,\ell$, then $\,\mu\,=\,b_1\,\omega_{i_1}\,+
\,b_2\,\omega_{i_2}\,+\,b_3\,\omega_{\ell}\;$ and\linebreak 
$\,\mu_1\,=\,\mu-(\alpha_{i_1}+\cdots +\alpha_{\ell})\,=\,
\omega_{i_1-1}\,+\,(b_1-1)\,\omega_{i_1}\,+\,b_2\omega_{i_2}\,+\,
b_3\,\omega_{\ell}\,\in\,{\cal X}_{++}(V)$.

In all these cases, $\,\mu_1\,$ is a good weight with at least 3 nonzero 
coefficients, since $\,b_1,\,b_2,\,b_3\,\geq 3$. Hence, by Part I of
this proof, $\,s(V)\,\geq\,|W\mu_1|\,>\,2\ell^3\,$ and,
by~\eqref{bllim}, $\,V\,$ is not an exceptional module.
This proves the lemma.
\end{proof}

The main theorem for groups of type $\,B$ is as follows. Its proof is given
in Appendix~\ref{appbl}.

\begin{theorem}\label{listbn}
Suppose $\,p\,>\,2.\,$
If $\,V\,$ is an infinitesimally irreducible 
$\,B_{\ell}(K)$-module with highest weight listed in
Table~\ref{tableblall} (p. \pageref{tableblall}), 
then $\,V\,$ is an exceptional $\,\mathfrak g$-module.
If $\,V\,$ has highest weight different from the ones listed in Tables
~\ref{tableblall} or~\ref{leftbn} (p. \pageref{leftbn}), then
$\,V\,$ is not an exceptional $\,\mathfrak g$-module.
\end{theorem}

%\newpage

%\input{cnidea.tex}

\subsubsection*{\ref{lact}.3. Algebras of Type $\,C\,$}\label{clresults}
\addcontentsline{toc}{subsubsection}{\protect\numberline{\ref{lact}.3}
Algebras of Type $\,C\,$}

By Theorem~\ref{???}, any $\,C_{\ell}(K)$-module $\,V\,$ such that 
\begin{equation}\label{cllim}
s(V)\,=\,\displaystyle\sum_{\stackrel{\scriptstyle\mu\,good}
{\scriptstyle \mu\in{\cal X}_{++}(V)}}\,m_{\mu}\,|W\mu|\,
>\,4\,\ell^3
\end{equation}
or
\begin{equation}\label{rrcl}
r_p(V)\,=\,\sum_{\stackrel{\scriptstyle{\mu\in{\cal X}_{++}(V)}}
{\scriptstyle \mu\;good}}\,m_{\mu}\,
\frac{|W\mu|}{|R_{long}|}\,|R_{long}^+-R^+_{\mu,p}|\,>\,2\,\ell^2
\end{equation}
is not an exceptional $\,\mathfrak g$-module. For groups of 
type $\,C_{\ell}$, $\,|W|\,=\,2^{\ell}\,\ell !$, 
$\,|R_{long}|\,=\,|A_1^{\ell}|\,=\,2\,\ell\,$ and
$\,|R_{long}^+|\,=\,\ell.\,$ Recall that $\,p\geq 3.\,$
The orbit sizes of weights are given by Lemma~\ref{orbbncndn}.

Let $\,V\,$ be a $\,C_{\ell}(K)$-module. We prove a reduction lemma.

\begin{lemma}\label{3coefcn}
Let $\,\ell\,\geq\,6$. If $\,{\cal X}_{++}(V)\,$ contains a
weight with 3 or more nonzero coefficients, then $\,V\,$
is not an exceptional $\,\mathfrak g$-module.
\end{lemma}\noindent
\begin{proof} 
I) Let $\,\mu=b_1\,\omega_1\,+\,\cdots\,+\,b_{\ell}\,\omega_{\ell}\,\in\,
{\cal X}_{++}(V)\,$ be a good weight with at least 3 nonzero coefficients. 
We claim that $\,s(V)\,>\,4\ell^3\,$, for $\,\ell\geq 6$. 

\underline{Step 1}: 
Write $\,\mu\,=\,b_1\,\omega_{i_1}\,+\,b_2\,\omega_{i_2}\,+\,
b_3\,\omega_{i_3}\,+\cdots\,$, where $\,1\leq\,i_i<i_2<i_3\,$ are the
first 3 nonzero coefficients of $\,\mu$. By 
Lemma~\ref{remark},  $\,|W\mu|\,\geq\,|W\bar{\mu}|,\,$ where
$\,\bar{\mu}\,=\,d_1\,\omega_{i_1}\,+
\,d_2\,\omega_{i_2}\,+\,d_3\,\omega_{i_3}\,$ with $\,d_j\neq 0\,$ for
all $\,1\leq j\leq 3$. Therefore, to prove the claim, it
suffices to prove that $\,|W\bar{\mu}|\,>\,4\,\ell^3$, for any weight
$\,\bar{\mu}\,$ with exactly 3 nonzero coefficients.

\underline{Step 2}: Let $\,\bar{\mu}\,=\,d_1\,\omega_{i_1}\,+
\,d_2\,\omega_{i_2}\,+\,d_3\,\omega_{i_3}\,$ with $\,d_j\neq 0\,$ for
$\,1\leq j\leq 3\,$. We prove that $\,|W\bar{\mu}|\,>\,4\,\ell^3\,$ by
considering the different possibilities for $\,(i_1,\,i_2,\,i_3)$. 

(a) Let $\,(i_1,\,i_2,\,i_3)\,=\,(1,\,2,\,3)\,$ and denote
$\,\bar{\mu}\,$ by $\,\nu=c_1\,\omega_1\,+\,c_2\omega_2\,+\,c_{3}\,
\omega_{3}\,$, with $\,c_1,\,c_2,\,c_3\,$ all nonzero. Then 
$\,|W\nu|\,=\,\displaystyle\frac{2^{\ell}\,\ell !}
{2^{(\ell -3)}\,(\ell -3)!}\,=\,2^3\,\ell\,(\ell -1)\,(\ell
-2)\,$. Thus, for $\,\ell\geq 6$, $\,|W\nu|\,-\,4\,\ell^3 \,=\, 
\,4\ell^2\,(\ell-6)\,+\,16\ell\,>\,0\,$ whence $\,|W\nu|\,>\,4\,\ell^3\,$.

(b) Let $\,1\leq i_1\,<\,i_2\,<\,i_3\,<\,\ell -1\,$ and 
$\,\bar{\mu}\,=\,d_1\,\omega_{i_1}\,+\,d_{2}\,\omega_{i_2}\,
+\,d_{3}\,\omega_{i_3}$, with $\,d_{j}\neq 0$. Then 
$\,|W\bar{\mu}|\, =\, \displaystyle 2^{i_3}\,\binom{i_2}{i_1}\,
\binom{i_3}{i_2}\,\binom{\ell}{i_3}\,$. Thus, for $\,3\leq i_3\leq
\ell-2\,$ and $\,\ell\geq 6,\,$ $\,|W\bar{\mu}|\,-\,4\,\ell^3  
\,\geq\,2^3\,\ell\,(\ell -1)\,(\ell-2)\,-\,4\,\ell^3\,\geq
\,4\ell^2\,(\ell-6)\,+\,16\ell\,>\,0\,$.

{\bf Therefore, if $\,\bar{\mu}\,=\,d_1\,\omega_{i_1}\,+\,d_{2}\,\omega_{i_2}\,
+\,d_{3}\,\omega_{i_3}$, with $\,d_{j}\neq 0\,$ and 
$\,1\leq i_1\,<\,i_2\,<\,i_3\,\leq\,\ell -2\,$, then
$\,|W\bar{\mu}|\,>\,2\ell^3\,$. Hence, we can assume
$\,\ell -1\leq \,i_3\,$.} We deal with this case in the sequel.

(c) For $\,1\leq i_1\,<\,i_2\,<i_3\,=\,\ell -1,\,$ 
$\,\bar{\mu}\,=\,b_1\,\omega_{i_1}\,+\,b_{2}\,\omega_{i_2}\,
+\,b_{3}\,\omega_{\ell -1}\;$ and \\
$\,|W\bar{\mu}|\, = \,\displaystyle 2^{\ell-1}\,\ell\,\binom{i_2}{i_1}
\binom{\ell -1}{i_2}\,$. Thus, for $\,2\leq i_2\leq \ell-2\,$ and 
$\,\ell\geq 5$, $\,|W\bar{\mu}|\,\geq\,2^{\ell-1}\,\ell\, 
(\ell-1)\,(\ell-2)\,>\,4\ell^3\,$.

d) Let $\,1\leq i_1\,<\,i_2\,<i_3\,=\,\ell\,$ and 
$\,\mu\,=\,\bar{\mu}\,=\,b_1\,\omega_{i_1}\,+\,b_{2}\,\omega_{i_2}\,
+\,b_{3}\,\omega_{\ell}$. Then,
$\,|W\bar{\mu}|\, = \,\displaystyle 2^{\ell}\,\binom{i_2}{i_1}\, 
\binom{\ell}{i_2}\,$. Hence, for $\,2\leq i_2\leq\ell -1\,$ and
$\,\ell\geq 5,\,$ we have $\,|W\bar{\mu}|\,\geq\,2^{\ell}\,\ell\,(\ell-1)\, 
>\,4\ell^3\,$. 

{\bf Therefore, if $\,{\cal X}_{++}(V)\,$ contains a good weight $\,\mu\,$
with at least 3 nonzero coefficients, then
$\,s(V)\,\geq\,|W\mu|\,>\,2\ell^3\,$. Hence, by~\eqref{bllim}, $\,V\,$
is not an exceptional module.}

II) Suppose that $\,\mu\,$ is a bad weight in $\,{\cal X}_{++}(V)\,$
with at least 3 nonzero coefficients. As $\,p\geq 3$,
we have $\,b_j\geq 3\,$ for all nonzero coefficients. 
We claim that it is always possible to produce, from $\,\mu,\,$ a
good weight $\,\mu_1\in{\cal X}_{++}(V)\,$ having 3 or more nonzero 
coefficients. Indeed:

(a) if $\,1\leq\,i_i<i_2<i_3\,\leq\,\ell -1$, then 
$\,\mu\,=\,b_1\,\omega_{i_1}\,+\,b_2\,\omega_{i_2}\,+\,
b_3\,\omega_{i_3}\,+\cdots\;$ and
$\,\mu_1 = \mu-(\alpha_{i_2}+\cdots +\alpha_{i_3}) 
 = b_1\,\omega_{i_1}\,+\,\omega_{i_2-1}\,+\,(b_2-1)\,\omega_{i_2}\,+\,
(b_3-1)\,\omega_{i_3}\,+\,\omega_{i_3+1}\,+\,\cdots\,\in\,{\cal
X}_{++}(V)$.

(b) if $\,1\leq\,i_i<i_2<i_3\,=\,\ell$, then $\,\mu\,=\,b_1\,\omega_{i_1}\,+
\,b_2\,\omega_{i_2}\,+\,b_3\,\omega_{\ell}\;$ and 
$\,\mu_1\,=\,\mu-(\alpha_{i_1}+\cdots +\alpha_{\ell})\,=\,
\omega_{i_1-1}\,+\,(b_1-1)\,\omega_{i_1}\,+\,b_2\omega_{i_2}\,+\,
\omega_{\ell-1}\,+\,(b_3-1)\,\omega_{\ell}\,\in\,{\cal X}_{++}(V)$.

In all these cases, $\,\mu_1\,$ is a good weight with at least 3 nonzero 
coefficients, since $\,b_1,\,b_2,\,b_3\,\geq 3$. Hence, by Part I of
this proof, $\,s(V)\,\geq\,|W\mu_1|\,>\,4\ell^3\,$ and,
by~\eqref{bllim}, $\,V\,$ is not an exceptional module.
This proves the lemma.
\end{proof}

The main theorem for groups of type $\,C\,$ is as follows. Its proof is 
given in Appendix~\ref{appcl}.
\begin{theorem}\label{listcn}
Suppose $\,p\,>\,2.\,$
If $\,V\,$ is an infinitesimally irreducible 
$\,C_{\ell}(K)$-module with highest weight listed in
Table~\ref{tableclall} (p. \pageref{tableclall}), 
then $\,V\,$ is an exceptional $\,\mathfrak g$-module.
If $\,V\,$ has highest weight different from the ones listed in Tables
\ref{tableclall} or~\ref{leftcn} (p. \pageref{leftcn}), then
$\,V\,$ is not an exceptional $\,\mathfrak g$-module.
\end{theorem}

\newpage

\subsubsection*{\ref{lact}.4. Algebras of Type $\,D\,$}\label{resuldl}
\addcontentsline{toc}{subsubsection}{\protect\numberline{\ref{lact}.4}
Algebras of Type $\,D\,$}

By Theorem~\ref{???}, any $\,D_{\ell}(K)$-module $\,V\,$ such that 
\begin{equation}\label{dllim}
s(V)\,=\,\displaystyle\sum_{\stackrel{\scriptstyle\mu\,good}
{\scriptstyle \mu\in{\cal X}_{++}(V)}}\,m_{\mu}\,|W\mu|\,
>\,\left\{
\begin{array}{ll}
2\,\ell^2\,(\ell-1) & \mbox{if}\;\ell\geq 5\\
\frac{8\,\ell^2\,(\ell-1)}{(\ell-2)} & \mbox{if}\;\ell\leq 4.
\end{array}\right.
\end{equation}
or
\begin{equation}\label{rrdl}
r_p(V)\,=\,\sum_{\stackrel{\scriptstyle{\mu\in{\cal X}_{++}(V)}}
{\scriptstyle \mu\;good}}\,m_{\mu}\,
\frac{|W\mu|}{|R_{long}|}\,|R_{long}^+-R^+_{\mu,p}|\,>\,2\,\ell\,(\ell-1)
\end{equation}
is not an exceptional $\,\mathfrak g$-module.
For groups of type $\,D_{\ell}\,$, $\,|W|\,=\,2^{\ell-1}\,\ell !$, 
$\,|R|\,=\,|R_{long}|\,=\,2\ell(\ell-1)$. The orbit sizes of weights
are given by Lemma~\ref{orbbncndn}.

Let $\,V\,$ be a $\,D_{\ell}(K)$-module. We prove a reduction lemma.

\begin{lemma}\label{3cdn}
Let $\,\ell\,\geq\,5$. If $\,{\cal X}_{++}(V)\,$ contains a
weight with 3 or more nonzero coefficients, then $\,V\,$
is not an exceptional $\,\mathfrak g$-module.
\end{lemma}\noindent
\begin{proof}
I) Let $\,\mu=b_1\,\omega_1\,+\,\cdots\,+\,b_{\ell}\,\omega_{\ell}\,\in\,
{\cal X}_{++}(V)\,$ be a good weight with at least 3 nonzero coefficients.
We claim that $\,s(V)\,>\,2\ell^2\,(\ell-1)\,$. 

\underline{Step 1}: 
Write $\,\mu\,=\,b_1\,\omega_{i_1}\,+\,b_2\,\omega_{i_2}\,+\,
b_3\,\omega_{i_3}\,+\cdots\,$, where $\,1\leq\,i_i<i_2<i_3\,$ are the
first 3 nonzero coefficients of $\,\mu$. By 
Lemma~\ref{remark},  $\,|W\mu|\,\geq\,|W\bar{\mu}|,\,$ where
$\,\bar{\mu}\,=\,a\omega_{i_1}\,+\,b\omega_{i_2}\,+\,c\omega_{i_3}\,$
with $\,a,\,b,\,c\,\neq 0$. Therefore, to prove the claim, it
suffices to prove that $\,|W\bar{\mu}|\,>\,2\,\ell^2\,(\ell-1)$, for any weight
$\,\bar{\mu}\,$ with exactly 3 nonzero coefficients.

\underline{Step 2}: Let $\,\bar{\mu}\,=\,a\omega_{i_1}\,+
\,b\omega_{i_2}\,+\,c\omega_{i_3}\,$ with $\,a,\,b,\,c\,\neq 0$.  
We prove that $\,|W\bar{\mu}|\,>\,2\,\ell^2\,(\ell-1)\,$ by
considering the different possibilities for $\,(i_1,\,i_2,\,i_3)$.

(a) Let $\,1\leq i_1\,<\,i_2\,<\,i_3\,\leq\,\ell-3\,$ (hence $\,\ell
\geq 6\,$) and $\,\bar{\mu}\,=\,a\omega_{i_1}\,+\,b\omega_{i_2}\,
+\,c\omega_{i_3}\,$. Then for $\,3\leq i_3\leq \ell-3\,$ and $\,\ell\geq 6$,
$\,\displaystyle\binom{\ell}{i_3}\geq\binom{\ell}{3}\,$ and
$\,\displaystyle\binom{i_2}{i_1}\,\binom{i_3}{i_2}\geq 6\,$. Hence
\[
|W\bar{\mu}|\,=\, \displaystyle 2^{i_3}\,\binom{i_2}{i_1}\,\binom{i_3}{i_2}\,
\binom{\ell}{i_3}\,\geq\,2^{3}\,\ell\,(\ell-1)\,(\ell-2)\,>\,2\,
\ell^2\,(\ell-1)\,.
\]

(b) For $\,1\leq i_1\,<\,i_2\,<i_3=\ell-2,\,$ $\,\bar{\mu}\,=\,a\omega_{i_1}\,+
\,b\omega_{i_2}\,+\,c\omega_{\ell-2}\,$. Thus, for $\,2\leq i_2\leq \ell-3\,$
and $\,\ell\geq 5$,
\[
\begin{array}{lcl}
|W\bar{\mu}|\, & = & \displaystyle
\frac{2^{\ell-1}\,\ell !}{i_1!\,(i_2-i_1)!\,(\ell-2-i_2)!\,
2!\,2!}\, =\,2^{\ell -3}\,\ell\,(\ell-1)\,\binom{i_2}{i_1}\,
\binom{\ell-2}{i_2}\vspace{.5ex}\\
& \geq & 2^{\ell -3}\,\ell\,(\ell-1)\,(\ell-2)\, >\,
 2\,\ell^2\,(\ell-1) \,.
\end{array}
\]

(c) For $\,1\leq i_1\,<\,i_2\,\leq \ell-2,\,i_3=\ell-1\,$ (or by a
graph-twist $\,1\leq i_1\,<\,i_2\,\leq \ell-2,\,i_3\,=\,\ell\,$), 
$\,\bar{\mu}\,=\,a\omega_{i_1}\,+\,b\omega_{i_2}\,+\,c\omega_{\ell-1}\;$ 
(or $\,\bar{\mu}\,=\,a\omega_{i_1}\,+\,b\omega_{i_2}\,+\,c\omega_{\ell}\,$). 
Thus, for $\,2\leq i_2\leq \ell-2\,$ and $\,\ell\geq 5,\,$
\[
\begin{array}{lcl}
|W\bar{\mu}|\, & = & \displaystyle 2^{\ell -1}\,\binom{i_2}{i_1}\,
\binom{\ell}{i_2}\,\geq\,2^{\ell -1}\cdot 2\binom{\ell}{2}\,\geq\, 
2^{\ell-1}\,\ell\,(\ell-1)\,>\,2\,\ell^2\,(\ell-1) \,.
\end{array}
\]

(d) For $\,1\leq i_1\leq \ell-2,\,i_2\,=\,\ell-1\,$ and $\,i_3\,=\,\ell,\;$
$\,\bar{\mu}\,=\,a\omega_{i_1}\,+\,b\omega_{\ell-1}\,+\,c\omega_{\ell}\,$.
Thus, for $\,\ell\geq 5$,
$\,|W\bar{\mu}|\, = \displaystyle 2^{\ell-1}\,\ell\,\binom{\ell-1}{i_1}\geq
2^{\ell-1}\,\ell\,(\ell-1)\,>\,2\,\ell^2\,(\ell-1) \,$.

{\bf Therefore, if $\,{\cal X}_{++}(V)\,$ contains a good weight $\,\mu\,$
with at least 3 nonzero coefficients, then $\,s(V)\,\geq\,|W\mu|\,>\, 
2\ell^2\,(\ell-1)\,$. Hence, \mbox{{\bf by}~\eqref{dllim},} $\,V\,$
is not an exceptional module.}

II) Let $\,\mu\,$ be a bad weight in $\,{\cal X}_{++}(V)\,$ with at
least 3 nonzero coefficients. We have $\,b_j\geq 2\,$ for all nonzero
coefficients. We claim that it is  possible to produce, from $\,\mu,\,$ a
good weight $\,\mu_1\in{\cal X}_{++}(V)\,$ having 3 or more nonzero
coefficients. Indeed,

(a) if $\,1\leq\,i_1<i_2<i_3\,<\,\ell -2$, then
$\,\mu\,=\,b_1\,\omega_{i_1}\,+\,b_2\,\omega_{i_2}\,+\,
b_3\,\omega_{i_3}\,+\cdots\;$ and 
$\,\mu_1\,=\,\mu-(\alpha_{i_2}+\cdots +\alpha_{i_3})\,=\,
b_1\,\omega_{i_1}\,+\,\omega_{i_2-1}\,+\,(b_2-1)\,\omega_{i_2}\,+\,
(b_3-1)\,\omega_{i_3}\,+\,\omega_{i_3+1}\,+\,\cdots\,\in\,{\cal
X}_{++}(V).\,$ 

(b) if $\,1\leq\,i_1<i_2<i_3\,=\,\ell -2$. Then
$\,\mu\,=\,b_1\,\omega_{i_1}\,+
\,b_2\,\omega_{i_2}\,+\,b_3\,\omega_{\ell -2}\,+\,\cdots\;$ and
$\,\mu_1\,=\,\mu-(\alpha_{i_2}+\cdots +\alpha_{\ell -2})\,=\,
b_1\,\omega_{i_1}\,+\,\omega_{i_2-1}\,+\,
(b_2-1)\omega_{i_2}\,+\,(b_3-1)\,\omega_{\ell -2}\,+\,
\omega_{\ell-1}\,+\,\omega_{\ell}\,\in\,{\cal X}_{++}(V).\,$

(c) if $\,1\leq\,i_1<i_2<i_3\,=\,\ell-1\,$ (or by a graph-twist
$\,i_3\,=\,\ell\,$), then $\,\mu\,=\,b_1\,\omega_{i_1}\,+
\,b_2\,\omega_{i_2}\,+\,b_3\,\omega_{\ell-1}\,+\,\cdots\;\,$ 
(or $\,\mu\,=\,b_1\,\omega_{i_1}\,+
\,b_2\,\omega_{i_2}\,+\,b_3\,\omega_{\ell}\,$).

i) For $\,1\leq\,i_1<i_2<\ell-2,\,$ $\,{\cal X}_{++}(V)\,$ contains
$\,\mu_1\,=\,\mu-(\alpha_{i_1}+\cdots +\alpha_{i_2})\,=\,
\omega_{i_1-1}\,+\,(b_1-1)\,\omega_{i_1}\,+\,(b_2-1)\omega_{i_2}\,+\,
\omega_{i_2+1}\,+\,b_3\,\omega_{\ell-1}\,+\,\cdots$
(or $\,\mu_1\,=\,\mu-(\alpha_{i_1}+\cdots +\alpha_{i_2})\,=\,
\omega_{i_1-1}\,+\,(b_1-1)\,\omega_{i_1}\,+\,b_2\omega_{i_2}\,+\,
\omega_{i_2+1}\,+\,b_3\,\omega_{\ell}$.)

ii) For $\,i_2\,=\,\ell-2,\;$ 
$\,\mu\,=\,b_1\,\omega_{i_1}\,+
\,b_2\,\omega_{\ell-2}\,+\,b_3\,\omega_{\ell-1}\,+\,\cdots\;\,$ 
(or $\,\mu\,=\,b_1\,\omega_{i_1}\,+
\,b_2\,\omega_{\ell-2}\,+\,b_3\,\omega_{\ell}\,$) and 
$\,\mu_1\,=\,\mu-(\alpha_{i_1}+\cdots +\alpha_{\ell-2})\,=\,
\omega_{i_1-1}\,+\,(b_1-1)\,\omega_{i_1}\,+\,(b_2-1)\omega_{\ell-2}\,+
\,(b_3+1)\,\omega_{\ell-1}\,+\,\omega_{\ell}\;$
(or 
$\,\mu_1\,=\,\mu-(\alpha_{i_1}+\cdots +\alpha_{\ell-2})\,=\,
\omega_{i_1-1}\,+\,(b_1-1)\,\omega_{i_1}\,+\,(b_2-1)\omega_{\ell-2}\,+\,
\omega_{\ell-1}\,+\,(b_3+1)\,\omega_{\ell}$) $\,\in\,{\cal X}_{++}(V)$.

In all these cases, $\,\mu_1\,$ is a good weight with at least 3 nonzero
coefficients, since $\,b_1,\,b_2,\,b_3\,\geq 2$. Hence, by Part I of
this proof, $\,s(V)\,\geq\,|W\mu_1|\,>\,2\ell^3\,$ and,
by~\eqref{dllim}, $\,V\,$ is not an exceptional module.
This proves the lemma.
\end{proof}

The main theorem for groups of type $\,D\,$ is as follows.
Its proof is given in Appendix~\ref{appdl}.

\begin{theorem}\label{dnlist}
If $\,V\,$ is an infinitesimally irreducible 
$\,D_{\ell}(K)$-module with highest weight listed
in Table~\ref{tabledlall} (p. \pageref{tabledlall}), 
then $\,V\,$ is an exceptional $\,\mathfrak g$-module.
If $\,V\,$ has highest weight different from the ones listed in Tables
~\ref{tabledlall} or~\ref{leftdn} (p. \pageref{leftdn}), then
$\,V\,$ is not an exceptional $\,\mathfrak g$-module.

\end{theorem}

\newpage

\part*{Chapter 5}
\part*{The Proofs}
\addcontentsline{toc}{section}{\protect\numberline{5}{The Proofs}}
\setcounter{section}{5}
\setcounter{subsection}{0}
In this Chapter, Theorem~\ref{frlaet} (p. \pageref{frlaet}) 
is proved. The proof is divided into
lemmas for each group or algebra type.
Recall that $\,\dim\,\mathfrak g\,=\,|R|+\ell\,$ and
$\,\varepsilon=\dim\,\mathfrak z(\mathfrak g)\,$.
We use Lemma~\ref{orbexcp} for sizes of centralizers and orbits of weights.
The notation for roots is as follows. For $\,\gamma=\displaystyle 
\sum_{j=1}^{\ell}\,b_j\,\alpha_j$, we write $\,\gamma =
(b_1\,b_2\,\cdots\,b_{\ell})\,$. %, where $\,6\leq \ell\leq\,8$.

\subsection{Groups or Lie Algebras of Exceptional Type}
\label{proofexcp}

\begin{remark}\label{remA} 
Let $\,R\,$ be of type $\,E_6,\,E_7\,$ or $\,E_8$. Let  
$\,\lambda = \displaystyle\sum_{i=1}^{\ell}\,a_i\,\omega_i\,$ be a
dominant weight such that $\,a_k\neq 0$, for some $\,1\leq k\leq\ell$. Then 
$\,C_W(\lambda)\,\subseteq\,C_W(\omega_k).\,$ Hence, by Lemma~\ref{remark}, 
$\,|W\lambda|\,\geq\,|W\omega_k|$.
\end{remark}

\subsubsection{Type $E_6$}

In this case the orbit sizes of the fundamental weights are as follows.
\vspace{2ex}

\begin{center}
\begin{tabular}{|c|c|c|c|c|c|c|} 
\hline             
$\omega_i$  & $\omega_1$  &  $\omega_2$  & $\omega_3$  & $\omega_4$  
& $\omega_5$  & $\omega_6$   \\ \hline
$|W\omega_i|$  & $ 27 $  & $72$ & $ 216$ & $ 720$ & $216$ & $27$ \\ \hline
\end{tabular}
\end{center}
\vspace{2ex}

By Theorem~\ref{???}, any $\,E_6(K)$-module $\,V\,$ such that
\begin{equation} \label{324}
s(V)\,=\,\sum_{\stackrel{\scriptstyle\mu\in{\cal X}_{++}(V)}
{\scriptstyle\mu\;good}}\,m_{\mu}\,|W\mu|\,>\,324
\end{equation}
or
\begin{equation}\label{rr72}
r_p(V)\,=\,\sum_{\stackrel{\scriptstyle{\mu\in{\cal X}_{++}(V)}}
{\scriptstyle \mu\;good}}\,m_{\mu}\,
\frac{|W\mu|}{|R_{long}|}\,|R_{long}^+-R^+_{\mu,p}|\,>\,72
\end{equation}
is not an exceptional $\,\mathfrak g$-module.
For groups of type $\,E_6\,$, $\,|W|\,=\,2^7\cdot 3^4\cdot 5\,$ and 
$\,|R_{long}|\,=\,|R|\,=\,72\,=\,2^3 \cdot 3^2\,$.
The following are reduction lemmas.
\begin{lemma}\label{e61}
Let $\,V\,$ be an $\,E_6(K)$-module. If $\,{\cal X}_{++}(V)\,$ contains:

(a) a good weight with at least $3$ nonzero coefficients or

(b) a (good or bad) weight with nonzero coefficient for $\,\omega_4\,$
(this will be called the {\bf $\,\omega_4$-argument} hereafter),\\
then $\,V\,$ is not an exceptional $\,\mathfrak g$-module.
\end{lemma}\noindent
\begin{proof}
(a) Let $\,\mu= \displaystyle\sum_{i=1}^6\,a_i\,\omega_i\,$ (with at
least $3$ nonzero coefficients) be a good weight in $\,{\cal X}_{++}(V)$.
Then $\,s(V)\,\geq\,|W\mu|\,\geq\,2^4\cdot 3^3\cdot 5\,>\,324$.
Hence, by (\ref{324}), $\,V\,$ is not an exceptional $\,\mathfrak g$-module.
 
(b) Let $\,\mu= \displaystyle\sum_{i=1}^{6}\,a_i\,\omega_i\,$ be 
such that $\,a_4\neq 0$. By Remark~\ref{remA}, $\,|W\mu|\,\geq\,720\,$.

If $\,\mu\,$ is a good weight in $\,{\cal X}_{++}(V)$, then
$\,s(V)\,\geq\,720\,>\,324$. Hence, by (\ref{324}), $\,V\,$ is not exceptional.

If $\,\mu\,$ is a bad weight in $\,{\cal X}_{++}(V)\,$ (with
$\,a_4\neq 0\,$, hence $\,a_4\geq 2\,$), then $\,\mu_1\,=\,\mu-\alpha_4\,
=\,a_1\omega_1\,+\,(a_2+1)\omega_2\,+\,(a_3+1)\omega_3\,+\, 
(a_4-2)\omega_4\,+\,(a_5+1)\omega_5\,+\,a_6\omega_6\,\in\,{\cal X}_{++}(V)\,$ 
is a good weight with (at least) $3$ nonzero coefficients. 
Hence, by (a), $\,V\,$ is not exceptional. This proves the lemma.
\end{proof}

\begin{lemma}\label{e6r2}
Let $\,V\,$ be an $\,E_6(K)$-module. If $\,{\cal X}_{++}(V)\,$
contains a good weight with $\,2\,$ nonzero coefficients, then 
$\,V\,$ is not an exceptional $\,\mathfrak g$-module.
\end{lemma}\noindent
\begin{proof}
Let $\,\mu\,=\,a\omega_j\,+\,b\omega_k\,$ (with $\,1\leq j<k\leq 6\,$ 
and $\,a\geq 1,\,b\geq 1\,$) be a good weight in $\,{\cal X}_{++}(V)$.

(a) If $\,j=4\,$ or $\,k=4$, then the $\,\omega_4$-argument
(Lemma~\ref{e61}(b)) applies.

(b) Let $\,\mu\in\{\,a\omega_1\,+\,b\omega_2\,,\;
a\omega_1\,+\,b\omega_3\,,\;a\omega_2\,+\,b\omega_6\,,\;
a_5\omega_5 + a_6\omega_6\,,\;\}\,$. Then $\,s(V)\,\geq\,|W\mu|\geq 
\displaystyle\frac{2^7\cdot 3^4\cdot 5}{5!}=432\,>\,324\,$. Hence, 
by~(\ref{324}), $\,V\,$ is not exceptional.

(c) Let $\,\mu\in\{\,a\omega_1\,+\,b\omega_5\,,
\;a\omega_2\,+\,b\omega_3\,,\;a\omega_2\,+\,b\omega_5\,,
\;a\omega_3\,+\,b\omega_6\,\}$. 
Then $\,s(v)\,\geq\,|W\mu|\,=\,2^3\cdot 3^3\cdot 5=1080>324\,$. Hence,
by (\ref{324}), $\,V\,$ is not exceptional.

(d) Let $\,\mu\,=\,a\omega_1\,+\,b\omega_6\,\in\,{\cal X}_{++}(V)\,$. 

For $\,a\geq 2,\,b\geq 1\,$ (resp., $\,a\geq 1,\,b\geq 2\,$),  
$\,\mu_1=\mu -\alpha_1=\,(a-2)\omega_1\,+\,\omega_3\,+\,b\omega_6\,$
(resp., $\,\mu_1=\mu -\alpha_6\,=\,a\omega_1\,+\,\omega_5\,+\,
(b-2)\omega_6\,$) $\,\in\,{\cal X}_{++}(V)\,$. 
For $\,a\geq 3\,$ (resp., $\,b\geq 3\,$), Lemma~\ref{e61}(a) applies. 
For $\,a=2\,$ (resp., $\,b=2\,$), $\,\mu_1\,$ satisfies case (c)
(resp., (b)) of this proof. In any case $\,V\,$ is not exceptional. 

Let $\,\mu= \omega_1+\omega_6\,\in{\cal X}_{++}(V)$. Then
$\,\mu_1=\mu-(1\,0\,1\,1\,1\,1)\,=\,\omega_2\,\in\,{\cal X}_{++}(V)$. 
For $\,p\geq 2$, $\,|R_{long}^+-R^+_{\mu,p}|\,\geq\,16\,$ and
$\,|R_{long}^+-R^+_{\mu_1,p}|\,\geq\,20$. Thus
$\,r_p(V)\,\geq\,\displaystyle\frac{2\cdot 3^3\cdot 5\cdot 16}{2^3\cdot 3^2}
\,+\,\frac{2^3\cdot 3^2\cdot 20}{2^3\cdot 3^2}\,=\,80\,>\,72\,$. 
Hence, by~\eqref{rr72}, $\,V\,$ is not exceptional.

(e) Let $\,\mu\,=\,a\omega_3\,+\,b\omega_5\,$. Then $\,s(V)\,\geq\,|W\mu| 
% \frac{2^7\cdot 3^4\cdot 5}{3!2!2!}=
\,=\,2^3\cdot 3^3\cdot 5\,=\,432\,>\,324\,$. Hence, by (\ref{324}), $\,V\,$
is not exceptional. This proves the lemma.
\end{proof}
\vspace{1ex}

{\bf Claim 1}: ({\bf The $\,\omega_3$-(or $\,\omega_5$)-argument}).
{\it Let $\,p\geq 3$. If $\,V\,$ is an $\,E_6(K)$-module such that
$\,\omega_3\,$ (or graph-dually $\,\omega_5\,$) 
$\,\in\,{\cal X}_{++}(V)$, then $\,V\,$ is not an exceptional 
$\,\mathfrak g$-module.}

Indeed, $\,|W\omega_3|\,=\,2^3\cdot 3^3\,$ and, for $\,p\geq 3$, 
$\,|R_{long}^+-R^+_{\omega_3,p}|\,=\,25$. Thus
$\,r_p(V)\,\geq\,\displaystyle\frac{2^3\cdot 3^3\cdot 25}{2^3\cdot 3^2}\,= 
\,75\,>\,72\,$. Hence, by~\eqref{rr72}, $\,V\,$ is not exceptional, proving
the claim. \hfill $\Box$ \vspace{1ex}

\begin{lemma}\label{e6nlf}
Let $\,G\,$ be a simply connected simple algebraic group of type  
$\,E_6$. A $\,\mathfrak g$-module $\,V\,$ is exceptional if and only if  
its highest weight is $\,\omega_1,\,\omega_2\,$ or $\,\omega_6$.
\end{lemma}\noindent
\begin{proof}
($\,\Longleftarrow\,$) 
If $\,V\,$ has highest weight $\lambda= \omega_1\,$ (or graph-dually
$\lambda= \omega_6\,$), then $\,\dim\,V\,=\,27\,<\,78\,-\,\varepsilon\,$. 
Hence, by Proposition~\ref{dimcrit}, $\,V\,$ is an exceptional module.
If $\,V\,$ has highest weight $\,\omega_2\,=\,\tilde{\alpha}\,$ 
then $\,V\,$ is the adjoint module, which is exceptional
by Example~\ref{adjoint}.

($\,\Longrightarrow\,$) Let $\,V\,$ be an $\,E_6(K)$-module.

{\bf Claim 2}: {\it If $\,{\cal X}_{++}(V)\,$ contains a nonzero bad
weight, then $\,V\,$ is not an exceptional $\,\mathfrak g$-module.}

Indeed, let $\,\nu = \sum_{i=0}^6\,a_i\,\omega_i\,$ be a nonzero bad
weight in $\,{\cal X}_{++}(V)$. Hence $\,a_i\equiv 0\pmod p,\,$ for all $\,i\,$
and at least one $\,a_i\neq 0\,$ (in which case $\,a_i\geq 2$.)
 
(i) $\,a_4\neq 0\,\Longrightarrow\,\nu\,$ satisfies the
$\,\omega_4$-argument (Lemma~\ref{e61}(b)).

(ii) $\,a_3\neq 0\,$ (or graph-dually $\,a_5\neq 0\,$) 
$\,\Longrightarrow\,\mu=\nu - \alpha_3 =
(a_1+1)\omega_1 + a_2\omega_2 +(a_3-2)\omega_3 +(a_4+1)\omega_4 +
a_5\omega_5 +a_6\omega_6\in{\cal X}_{++}(V)$. Hence
the $\,\omega_4$-argument applies.

(iii) $\,a_2\neq 0\,\Longrightarrow\,\mu=\nu - \alpha_2 =
a_1\omega_1 + (a_2-2)\omega_2 +a_3\omega_3 +(a_4+1)\omega_4 +
a_5\omega_5 +a_6\omega_6\in{\cal X}_{++}(V)$. Hence
the $\,\omega_4$-argument applies.

(iv) $\,a_1\neq 0\,$(or graph-dually $\,a_6\neq 0\,$) $\,\Longrightarrow\,
\mu=\nu - \alpha_1 = (a_1-2)\omega_1 + a_2\omega_2 +
(a_3+1)\omega_3 +a_4\omega_4 + a_5\omega_5 +a_6\omega_6\in{\cal X}_{++}(V)$. 

For $\,p\geq 3,\,$ $\,|R_{long}^+-R^+_{\mu,p}|\,=\,21\,$.  
Thus, as $\,|W\mu|\geq 2^4\cdot 3^3\,$, $\,r_p(V)\,\geq\,\displaystyle
\frac{2^4\cdot 3^3\;\;21}{2^3\cdot 3^2}\,=\,126\,>\,72$. Hence
by~\eqref{rr72}, $\,V\,$ is not exceptional.

For $\,p=2\,$ and $\,a_1>2$, $\,\mu\,$ does not occur in $\,{\cal X}_{++}(V)$.
For $\,p= 2\,$ we can assume $\,a_1=2\,$ and $\,a_i=0\,(\,2\leq i\leq 6)$. 
Then $\,\nu=2\omega_1\,\in{\cal X}_{++}(V)$ only if $\,V\,$ has
highest weight $\,\lambda= \omega_1\,+\,\omega_5\,$ (for instance). 
Hence, by Lemma~\ref{e6r2}, $\,V\,$ is not exceptional. This proves Claim 2.

{\bf Hence, by Lemma~\ref{e6r2} and Claim 2,
if $\,{\cal X}_{++}(V)\,$ contains a weight with 
$2$ nonzero coefficients, then $\,V\,$ is not an exceptional module.}

{\bf Therefore, by Claim 2, we can assume
that $\,{\cal X}_{++}(V)\,$ contains only good weights with at most one 
nonzero coefficient.}

Let $\,\lambda\, = \,a\,\omega_i\,$ (with $\,a\geq 1\,$)
be a good weight in $\,{\cal X}_{++}(V)$.

(a) Let $\,a\geq 1\,$ and $\,\lambda\,=\,a\,\omega_4\,$. 
Then the $\,\omega_4$-argument (Lemma~\ref{e61}(b)) applies.

(b) Let $\,a\geq 1\,$ and $\,\lambda\,=\,a\,\omega_3\,$
(or graph-dually $\,\lambda\,=\,a\,\omega_5\,$). 

For $\,a\geq 2$,  $\,\mu\,=\,\lambda - \alpha_3\,=\,\omega_1\,+\,(a-2)
\omega_3\,+\,\omega_4\,\in\,{\cal X}_{++}(V)$. Hence, 
the $\,\omega_4$-argument applies. 

For $\,a=1\,$ and $\,p\geq 3\,$, $\,\lambda\,=\,\omega_3\,$ satisfies Claim 1.
For $\,p=2$, we may assume that $\,V\,$ has highest weight $\,\mu\,=\,
\omega_3\,$. Then $\,\mu_1\,=\,\mu-(112210)\,=\,\omega_6\,\in\,
{\cal X}_{++}(V)$. By~\cite[p. 414]{gise}, $\,m_{\mu_1}=4$. As
$\,|R_{long}^+-R^+_{\mu,2}|\,=\,20\,$ and
$\,|R_{long}^+-R^+_{\mu_1,2}|\,=\,16$, $\,r_2(V)\,\geq\,\displaystyle
\frac{2^3\cdot 3^3\cdot 20}{2^3\cdot 3^2}\,+\,4\cdot
\frac{3^3\cdot 16}{2^3\cdot 3^2}\,=\,84\,>\,72\,$. Hence,
by~\eqref{rr72}, $\,V\,$ is not exceptional.

(c) Let $\,a\geq 1\,$ and $\,\lambda\,=\,a\,\omega_1\,$ 
(or graph-dually $\,\mu\,=\,a\,\omega_6\,$) be a good weight in 
$\,{\cal X}_{++}(V)$.

For $\,a\geq 2$, $\,\mu_1\,=\,\lambda-\alpha_1\,=\,(a-2)\omega_1\,+\,\omega_3\,
\in\,{\cal X}_{++}(V)$. For $\,a\geq 3$, $\,\mu_1\,$ satisfies
Lemma~\ref{e6r2}. For $\,a=2\,$ (hence $\,p\geq 3\,$), $\,\mu_1\,$  
satisfies Claim 1. In these cases $\,V\,$ is not exceptional.

For $\,a=1$, we may assume that $\,V\,$ has highest weight $\,\mu=
\omega_1\,$ (or graph-dually $\,\mu\,=\,\omega_6\,$).
Then $\,\dim\,V\,=\,27\,<\,78\,-\,\varepsilon\,$. Hence,
by Proposition~\ref{dimcrit}, $\,V\,$ is an exceptional module.

(d) Let $\,a\geq 1\,$ and $\,\mu\,=\,a\,\omega_2\,$ be a good weight in 
$\,{\cal X}_{++}(V)$.   

For $\,a\geq 2$, $\,\mu_1\,=\,\mu-\alpha_2\,=\,(a-2)\omega_2\,+\,\omega_4\,
\in\,{\cal X}_{++}(V)\,$ satisfies the $\,\omega_4$-argument
(Lemma~\ref{e61}(b)). For $\,a=1\,$, we may assume that $\,V\,$ 
has highest weight $\,\omega_2\,=\,\tilde{\alpha}\,$. Then $\,V\,$ is
the adjoint module, which is exceptional by Example~\ref{adjoint}.
This proves the lemma.
\end{proof}

\subsubsection{Type $\,E_7$}

The orbit sizes of the fundamental weights for groups of type
$\,E_7\,$ are as follows.
\vspace{2ex}

\begin{center}
\begin{tabular}{|c|c|c|c|c|c|c|c|} 
\hline             
$\omega_i$  & $\omega_1$  &  $\omega_2$  & $\omega_3$  & $\omega_4$  
& $\omega_5$  & $\omega_6$  & $\omega_7$ \\ \hline
$|W\omega_i|$  & $ 126 $  & $576$ & $ 2016$ & $10080$ & $4032$ & $756$ 
& $56$ \\ \hline
\end{tabular}
\end{center}
\vspace{2ex}

By Theorem~\ref{???}, any $\,E_7(K)$-module $\,V\,$ such that
\begin{equation} \label{588}
s(V)\,=\,\sum_{\stackrel{\scriptstyle\mu\in{\cal X}_{++}(V)}
{\scriptstyle \mu\; good}}\,m_{\mu}\,|W\mu|\,>\,588
\end{equation}
or
\begin{equation}\label{rr126}
r_p(V)\,=\,\sum_{\stackrel{\scriptstyle{\mu\in{\cal X}_{++}(V)}}
{\scriptstyle \mu\; good}}\,m_{\mu}\,
\frac{|W\mu|}{|R_{long}|}\,|R_{long}^+-R^+_{\mu,p}|\,>\,126
\end{equation}
is not an exceptional $\,\mathfrak g$-module. For groups of type $\,E_7\,$,
$\,|W|\,=\,2^{10}\cdot 3^4\cdot 5\cdot 7\,$ and 
$\,|R_{long}|\,=\,|R|\,=\,126\,=\,2\cdot 3^2\cdot 7\,$.

\begin{lemma}\label{e71}
Let $\,V\,$ be an $\,E_7(K)$-module. If $\,{\cal X}_{++}(V)\,$ contains:

(a) a good weight with nonzero coefficient 
for $\,\omega_i\,$ for some $\,i\in\{ 3,\,4,\,5,\,6\}\,$ or

(b) a good weight with at least $2$ nonzero coefficients or

(c) a nonzero bad weight,\\
then $\,V\,$ is not an exceptional $\,\mathfrak g$-module.
\end{lemma}\noindent
\begin{proof}
(a) Let $\,\mu\,=\,\displaystyle\sum_{i=1}^7\,a_i\,\omega_i\,$ be a good
weight in $\,{\cal X}_{++}(V)$. If $\,a_k\neq 0\,$ then, by 
Remark~\ref{remA}, $\,|W\mu|\geq |W\omega_k|\,$. Thus, if $\,a_k\neq
0\,$ for $\,k\in\{ 3,\,4,\,5,\,6\,\}$, then $\,s(V)\,\geq\,756\,>\,588\,$.
Hence, by (\ref{588}), $\,V\,$ is not an exceptional module.

(b) Let $\,\mu\,= \,a\,\omega_j\,+\,b\,\omega_k\,$ (with 
$\,1\leq j<k\leq 7\,$ and $\,a\geq 1,\,b\geq 1\,$) be a good weight in 
$\,{\cal X}_{++}(V)$. If $\,j\,$ or $\,k\;\in\,\{ 3,4,5,6\,\}$, 
then case (a) of this lemma applies. 
 
If $\,\mu\,=\,a\,\omega_1\,+\,b\,\omega_2\,$ or 
$\,\mu\,=\,a\,\omega_2\,+\,b\,\omega_7 \,$, then
$\,|W\mu|\,=\,2^6\cdot 3^2\cdot 7\,=\,4032\,$. 
If $\,\mu\,=\,a\,\omega_1\,+\,b\,\omega_7\,$, then 
$\,|W\mu|\,=\,2^3\cdot 3^3\cdot 7\,=\,1512.\,$  In these cases,
$\,s(V)\,>\,588\,$. Hence, by (\ref{588}), $\,V\,$ is not exceptional.

(c) Let $\,\nu\,=\,\displaystyle\sum_{i=1}^7\,a_i\,\omega_i\,$ be
a nonzero bad weight in $\,{\cal X}_{++}(V)\,$ (that is, 
$\,a_i\equiv 0\pmod p,\,$ for all $\,i\,$, but at least one $\,a_i\neq 0$.
(Note that $\,a_i\neq 0\,\Longrightarrow\,a_i\geq 2$.)
Write $\,\Omega_{j,k,\ldots}\,=\,\displaystyle
\sum_{r\neq j,k,\ldots}\,a_r\,\omega_r\,$.

i) $\,a_1\neq 0\,\Longrightarrow\,\mu=\nu - \alpha_1 =
(a_1-2)\omega_1 + (a_3+1)\omega_3 + \Omega_{1,3}\,\in{\cal X}_{++}(V)$.

ii) $\,a_2\neq 0\,\Longrightarrow\,\mu=\nu - \alpha_2 =
(a_2-2)\omega_2 +(a_4+1)\omega_4 + \Omega_{2,4}\,
\in{\cal X}_{++}(V).\,$ 

iii) $\,a_3\neq 0\,\Longrightarrow\,\mu=\nu - \alpha_3 =
(a_1+1)\omega_1 + (a_3-2)\omega_3 +(a_4+1)\omega_4 + \Omega_{1,3,4}\,
\in{\cal X}_{++}(V).\,$

iv) $\,a_4\neq 0\,\Longrightarrow\,\mu=\nu - \alpha_4 =
(a_2+1)\omega_2 +(a_3+1)\omega_3 +(a_4-2)\omega_4 +\mbox{$(a_5 +1)\omega_5$}
+\Omega_{2,3,4,5}\,\in{\cal X}_{++}(V).\,$

v) $\,a_5\neq 0\,\Longrightarrow\,\mu=\nu - \alpha_5 =
(a_4+1)\omega_4 +(a_5-2)\omega_5 +(a_6+1)\omega_6 + \Omega_{4,5,6}\,
\in{\cal X}_{++}(V).\,$ 

vi) $\,a_6\neq 0\,\Longrightarrow\,\mu=\nu - \alpha_6 =
(a_5+1)\omega_5 +(a_6-2)\omega_6 +(a_7+1)\omega_7 + \Omega_{5,6}\,
\in{\cal X}_{++}(V).\,$ 

vii) $\,a_7\neq 0\,\Longrightarrow\,\mu=\nu - \alpha_7 =
(a_6+1)\omega_6 +(a_7-2)\omega_7 + \Omega_{6,7}\,\in{\cal X}_{++}(V).\,$

In all these cases, $\,\mu\,$ satisfies part (a) of this lemma. Hence
$\,V\,$ is not an exceptional module. This proves the lemma.
\end{proof}

\begin{lemma}\label{e72}
Let $\,G\,$ be a simply connected simple algebraic group of 
type $\,E_7$. A $\,\mathfrak g$-module $\,V\,$ is exceptional
if and only if its highest weight is $\,\omega_1\,$ or $\,\omega_7$.
\end{lemma}\noindent
\begin{proof}
($\,\Longleftarrow\,$) 
If $\,V\,$ is an $\,E_7(K)$-module of highest weight $\lambda= \omega_1\,$, 
then $\,V\,$ is the adjoint module, which is exceptional
by Example~\ref{adjoint}.

If $\,V\,$ is an $\,E_7(K)$-module of highest weight $\,\lambda\,=\,
\omega_7\,$, then $\,\dim\,V\,=\,56\,<\,133\,-\,\varepsilon\,$. 
Hence, by Proposition~\ref{dimcrit}, $\,V\,$ is an exceptional module.

($\,\Longrightarrow\,$) Let $\,V\,$ be an $\,E_7(K)$-module.
By Lemma~\ref{e71}, {\bf it suffices to consider modules $\,V\,$ such that
$\,{\cal X}_{++}(V)\,$ contains good weights with at most one nonzero 
coefficient.}

Let $\,\lambda\,=\,a\,\omega_i\,$ (with $\,a\geq 1\,$ and $\,1\leq i\leq 7\,$) 
be a good weight in $\,{\cal X}_{++}(V)$.
For $\,i\in\{ 3,\,4,\,5,\,6\}\,$, Lemma~\ref{e71}(a) applies.
Hence we can assume $\,i\in\{ 1,\,2,\,7\}\,$. These cases are treated in the
sequel.

(i) Let $\,a\geq 1\,$ and $\,\lambda\,=\,a_1\,\omega_1\,$. 
For $\,a\geq 2$, $\,\mu=\lambda - \alpha_1=(a-2)\omega_1 +\omega_3   
\in{\cal X}_{++}(V)\,$ satisfies Lemma~\ref{e71}(a).
For $\,a=1$, we may assume that $\,V\,$ has highest weight 
$\,\lambda =\omega_1\,$. Then $\,V\,$ is the adjoint module, 
which is exceptional by Example~\ref{adjoint}.

(ii) Let $\,a\geq 1\,$ and $\,\lambda\,=\,a\,\omega_2\,$. 
For $\,a\geq 2$, $\,\mu=\lambda - \alpha_2=(a-2)\omega_2 +\omega_4  
\in{\cal X}_{++}(V)\,$ satisfies Lemma~\ref{e71}(a).
For $\,a=1$, $\,\lambda =\omega_2\,$ and $\,\mu\,=\,\lambda -
(1223210)\,=\,\omega_7\,\in\,{\cal X}_{++}(V)$. Thus
$\,s(V)\,\geq\,576 + 56\,=\,632\,>\,588\,$. Hence, by~\eqref{588},
$\,V\,$ is not exceptional. 

(iii) Let $\,a\geq 1\,$ and $\,\lambda\,=\,a\,\omega_7\,$.  
For $\,a\geq 2$, $\,\mu=\lambda - \alpha_7=\omega_6 +(a-2)\omega_7  
\in{\cal X}_{++}(V)\,$ satisfies Lemma~\ref{e71}(a).
For $\,a=1$, we may assume that $\,V\,$ has highest weight $\,\lambda 
=\omega_7$. Then $\,\dim\,V\,=\,56\,<\,133\,-\varepsilon\,$. Hence, 
by Proposition~\ref{dimcrit}, $\,V\,$ is an exceptional module.
This proves the lemma.
\end{proof}

\subsubsection{Type $\,E_8$}

In this case the orbit sizes of the fundamental weights are as follows.
\vspace{2ex}

\begin{center}
\begin{tabular}{|c|c|} 
\hline             
$\omega_i$  & $|W\omega_i|$ \\ \hline
$\omega_1$  &  $ 1080 $     \\ \hline
$\omega_2$  &   $17280$          \\ \hline
$\omega_3$  &  $2^9\cdot 3^3\cdot 5 $        \\ \hline
$\omega_4$ &   $ 2^9\cdot 3^3\cdot 5\cdot 7 $      \\ \hline
$\omega_5$  &  $2^8\cdot 3^3\cdot 5 \cdot 7 $ \\ \hline
$\omega_6$  &  $ 2^6\cdot 3^3\cdot 5 \cdot 7 $ \\ \hline
$\omega_7$ &   $ 2^6\cdot 3\cdot 5 \cdot 7  $\\ \hline
$\omega_8$ &  $120$
\\ \hline
\end{tabular}
\end{center}
\vspace{2ex}

By Theorem~\ref{???}, any $\,E_8(K)$-module $\,V\,$ such that
\begin{equation} \label{1011}
s(V)\,=\,\sum_{\stackrel{\scriptstyle\mu\in{\cal X}_{++}(V)}
{\scriptstyle \mu\;good}}\,m_{\mu}\,|W\mu|\,>\,1011
\end{equation}
or
\begin{equation}\label{rr240}
r_p(V)\,=\,\sum_{\stackrel{\scriptstyle{\mu\in{\cal X}_{++}(V)}}
{\scriptstyle \mu\;good}}\,m_{\mu}\,
\frac{|W\mu|}{|R_{long}|}\,|R_{long}^+-R^+_{\mu,p}|\,>\,240
\end{equation}
is not an exceptional $\,\mathfrak g$-module.
For groups of type $\,E_8\,$, $\,|W|\,=\,2^{14}\cdot 3^5\cdot 5^2\cdot
7\,$ and $\,|R_{long}|\,=\,|R|\,=\,240\,=\,2^4\cdot 3\cdot 5$.

\begin{lemma}\label{e81}
Let $\,V\,$ be an $\,E_8(K)$-module. If $\,{\cal X}_{++}(V)\,$ contains:

(a) a good weight with nonzero coefficient for $\,\omega_i\,$, for some 
$\,i\in\{ 1,\ldots,7\}\;$ or

(b) a nonzero bad weight,\\
then $\,V\,$ is not an exceptional $\,\mathfrak g$-module.
\end{lemma}\noindent
\begin{proof}
(a) Let $\,\mu\,=\,\displaystyle\sum_{i=1}^8\,a_i\,\omega_i\,$ be a good
weight in $\,{\cal X}_{++}(V)$. If $\,a_k\neq 0\,$ then, by 
Remark~\ref{remA}, $\,|W\mu|\geq |W\omega_k|\,$. Thus, if $\,a_k\neq
0\,$ for $\,k\in\{ 1,\,2,\,3,\,4,\,5,\,6,\,7\,\}$ then, by table above, 
$\,s(V)\,\geq\,1080\,>\,1011\,$. Hence, by (\ref{1011}), $\,V\,$ is not an 
exceptional module.

(b) Let $\,\nu\,=\,\displaystyle\sum_{i=1}^8\,a_i\,\omega_i\,$ be
a nonzero bad weight in $\,{\cal X}_{++}(V)\,$ (that is, 
$\,a_i\equiv 0\pmod p,\,$ for all $\,i\,$, but at least one $\,a_i\neq 0$).
(Note that $\,a_i\neq 0\,\Longrightarrow\,a_i\geq 2$.)
Write $\,\Omega_{j,k,\ldots}\,=\,\displaystyle
\sum_{r\neq j,k,\ldots}\,a_r\,\omega_r\,$.

i) $\,a_1\neq 0\,\Longrightarrow\,\mu=\nu - \alpha_1 =
(a_1-2)\omega_1 + (a_3+1)\omega_3 + \Omega_{1,3}\in{\cal X}_{++}(V)$.

ii) $\,a_2\neq 0\,\Longrightarrow\,\mu=\nu - \alpha_2 =
(a_2-2)\omega_2 +(a_4+1)\omega_4 + \Omega_{2,4}
\in{\cal X}_{++}(V).\,$

iii) $\,a_3\neq 0\,\Longrightarrow\,\mu=\nu - \alpha_3 =
(a_1+1)\omega_1 + (a_3-2)\omega_3 +(a_4+1)\omega_4 + \Omega_{1,3,4}
\in{\cal X}_{++}(V).\,$ 

iv) $\,a_4\neq 0\,\Longrightarrow\,\mu=\nu - \alpha_4 =
(a_2+1)\omega_2 +(a_3+1)\omega_3 +(a_4-2)\omega_4 + (a_5 +1)\omega_5 +
\Omega_{2,3,4,5}\in{\cal X}_{++}(V).\,$

v) $\,a_5\neq 0\,\Longrightarrow\,\mu=\nu - \alpha_5 =
(a_4+1)\omega_4 +(a_5-2)\omega_5 +(a_6+1)\omega_6 + \Omega_{4,5,6}
\in{\cal X}_{++}(V).\,$

vi) $\,a_6\neq 0\,\Longrightarrow\,\mu=\nu - \alpha_6 =
(a_5+1)\omega_5 +(a_6-2)\omega_6 +(a_7+1)\omega_7 + \Omega_{5,6}
\in{\cal X}_{++}(V).\,$

vii) $\,a_7\neq 0\,\Longrightarrow\,\mu=\nu - \alpha_7 =
(a_6+1)\omega_6 +(a_7-2)\omega_7 + (a_8+1)\omega_8 + 
\Omega_{6,7}\in{\cal X}_{++}(V).\,$

viii) $\,a_8\neq 0\,\Longrightarrow\,\mu=\nu - \alpha_8 =
(a_7+1)\omega_7 +(a_8-2)\omega_8 +\Omega_{7,8}\in{\cal X}_{++}(V).\,$\\
In all these cases, $\, \mu\,$ satisfies part (a) of this lemma. Hence 
$\,V\,$ is not an exceptional module. This proves the lemma.
\end{proof}

\begin{lemma}\label{nlfe8}
The only exceptional $\,E_8(K)$-module is the adjoint module.
\end{lemma}\noindent
\begin{proof}
Let $\,V\,$ be an $\,E_8(K)$-module. By Lemma~\ref{e81}, {\bf we can assume  
that $\,{\cal X}_{++}(V)\,$ contains only good weights of the form $\,\lambda 
= a\,\omega_8\,$ (with $\,a\geq 1\,$).} 

For $\,a\geq 2$, $\,\mu=\lambda - \alpha_8=\omega_7 + (a-2)\omega_8\,
\in{\cal X}_{++}(V)$. Hence Lemma~\ref{e81}(a) applies.
For $\,a=1\,$, we can assume that $\,V\,$ has highest weight
$\,\lambda = \omega_8\,$. Then $\,V\,$ is the adjoint module, 
which is exceptional by Example~\ref{adjoint}.
This proves the lemma.
\end{proof}

\subsubsection{Type $\,F_4\,$}

For groups of type $\,F_4\,$ the orbit sizes for the fundamental
weights are as follows.

\begin{center}
\begin{tabular}{|c|c|c|c|c|} 
\hline             
$\omega_i$  & $\omega_1$  &  $\omega_2$  & $\omega_3$  & $\omega_4$  
 \\ \hline
$|W\omega_i|$  & $ 24 $  & $96$ & $ 96$ & $ 24$ \\ \hline
\end{tabular}
\end{center}
\vspace{2ex}

By Theorem~\ref{???}, any $\,F_4(K)$-module $\,V\,$ such that
\begin{equation} \label{192}
s(V)\,=\,\sum_{\stackrel{\scriptstyle\mu\in{\cal X}_{++}(V)}
{\scriptstyle\mu\;is good}}\,m_{\mu}\,|W\mu|\,>\,192
\end{equation}
or
\begin{equation}\label{rr48}
r_p(V)\,=\,\sum_{\stackrel{\scriptstyle{\mu\in{\cal X}_{++}(V)}}
{\scriptstyle \mu\; good}}\,m_{\mu}\,
\frac{|W\mu|}{|R_{long}|}\,|R_{long}^+-R^+_{\mu,p}|\,>\,48
\end{equation}
is not an exceptional $\,\mathfrak g$-module. For groups of type $\,F_4\,$,
$\,p\geq 3$, $\,|W|\,=\,2^7\cdot 3^2\,$, 
$\,|R|\,=\,48\,,\;\,|R_{long}|\,=\,|R(D_4)|\,=\,2^3\cdot 3$.
$\,R_{long}^+\,=\,\{ \varepsilon_i\pm\varepsilon_j\,/\,1\leq i<j\leq 4\}\,
=\,\{ (2 3 4 2),\,(1 3 4 2),\,(1 2 4 2),\,(1 2 2 0),\,(1 1 2 0),\,(0 1 2 0),
\,(0 1 2 2),\,(1 1 2 2),\,(1 2 2 2),\,(1 0 0 0)$, \linebreak
$\,(1 1 0 0),\,(0 1 0 0)\}$.
Furthermore,
\begin{eqnarray}
\omega_1 & = & 2\alpha_1 + 3\alpha_2 + 4\alpha_3 + 2\alpha_4 = (2342) 
\nonumber \\
\omega_2 & = & 3\alpha_1 + 6\alpha_2 + 8\alpha_3 + 4\alpha_4 = (3684) 
\nonumber \\
\omega_3 & = & 2\alpha_1 + 4\alpha_2 + 6\alpha_3 + 3\alpha_4 = (2463) 
\nonumber \\
\omega_4 & = & 1\alpha_1 + 2\alpha_2 + 3\alpha_3 + 2\alpha_4 = (1232)
\nonumber
\end{eqnarray}

\begin{lemma}\label{fared}
Let $\,V\,$ be an $\,F_4(K)$-module. If $\,{\cal X}_{++}(V)\,$ contains

(a) a good weight with $3$ nonzero coefficients or

(b) a good weight with $2$ nonzero coefficients or

(c) any nonzero bad weight, \\
then $\,V\,$ is not an exceptional $\,\mathfrak g$-module.
\end{lemma}\noindent
\begin{proof}
(a) Let $\,\mu\in{\cal X}_{++}(V)\,$ be a good weight (with $3$ nonzero
coefficients). Then $\,s(V)\,\geq\,|W\mu|\,=\,576\,>\,192\,$.
Hence, by \eqref{192}, $\,V\,$ is not exceptional.   

(b) Let $\,\mu\,=\,a\,\omega_i\,+\,b\,\omega_j\,$ (with $\,a\geq 1,\,b\geq 
1\,$ and $\,1\leq i<j\leq 4\,$) be a good weight in $\,{\cal X}_{++}(V)$. 

i) Let $\,\mu\in\{\,a\omega_1\,+\,b\omega_3\,,\;a\omega_2\,+\,b\omega_3\,,\;
a\omega_2\,+\,b\omega_4\,\}$. Then $\,s(V)\,\geq\,|W\mu|\geq\displaystyle 
\frac{2^7\cdot 3^2}{2!\cdot 2!}=2^5\cdot 3^2 =288>192\,$.

ii) Let $\,\mu\,=\,a\omega_1\,+\,b\omega_2\;$ or
$\;\mu\,=\,c\omega_3\,+\,d\omega_4\,$. Then $\,|W\mu|\,=\,2^6\cdot 3=192$.

If $\,\mu\,=\,a\omega_1\,+\,b\omega_2\,=\,a(2342)\,+\,b(3684)\,$, 
then $\,\mu_1=\mu - (1221)\,=\,a(2342) +(b-1)(3684)+(2463)\,=\,
a\omega_1\,+\,(b-1)\omega_2\,+\,\omega_3 \in {\cal X}_{++}(V)$.  
If $\,\mu\,=\,c\omega_3\,+\,d\omega_4\,=\,c(2463)\,+\,d(1232)\,$, then
$\,\mu_1=\mu - (0011)\,=\,(c-1)(2463) +(d-1)(1232)+(3684)\,=\,\omega_2\,+ 
\,(c-1)\omega_3\,+\,(d-1)\omega_4\,\in{\cal X}_{++}(V)$.  
In both cases, $\,s(V)\,\geq\,|W\mu|\,+\,|W\mu_1|\,\geq\,288\,>\,192\,$. 
Hence, by~\eqref{192}, $\,V\,$ is not exceptional.

iii) Let $\,\mu\,=\,a\omega_1\,+\,b\omega_4\,$ (with $\,a\geq 1,\,b\geq 1\,$)
be a good weight in $\,{\cal X}_{++}(V)$. As $\,\omega_1=
\tilde{\alpha}=(2342)\,$ and $\,\omega_4=(1232)\,$ are also roots,
$\,\mu_1=\omega_1 +\omega_4\,\in\,{\cal X}_{++}(V)$. Also
$\,\mu_2=\mu_1-(1111)\,=\,\omega_3\,\in\,{\cal X}_{++}(V)$. Thus 
$\,s(V)\,\geq\,|W\mu|\,+\,|W\mu_1|\,\geq 144 +96 =240>\,192$.
Hence, by~\eqref{192}, $\,V\,$ is not exceptional. This proves (b).

(c) Let $\,\mu = \displaystyle\sum_{i=1}^4\,a_i\,\omega_i\,$ be a nonzero 
bad weight in $\,{\cal X}_{++}(V)$. Hence, as $\,p\geq 3$,
the nonzero coefficients of $\,\mu\,$ are $\,\geq 3\,$.

i) $\,a_1\neq 0\,\Longrightarrow\,\mu_1\,=\,\mu -(1110)\,=\,(a_1-1)\omega_1
+(a_4+1)\omega_4 +\Omega_{1,4}\,\in{\cal X}_{++}(V)$. 

ii) $\,a_2\neq 0\,\Longrightarrow\,\mu_1=\mu -(1221)=(a_2-1)\omega_2 +
(a_3+1)\omega_3+\Omega_{2,3}\in{\cal X}_{++}(V)$. 

iii) $\,a_3\neq 0\,\Longrightarrow\,\mu_1=\mu -(0121)=(a_1+1)\omega_1 +
(a_3-1)\omega_3+\Omega_{1,3}\in{\cal X}_{++}(V)$. 

In all these cases, $\,\mu_1\,$ is a good weight in $\,{\cal X}_{++}(V)\,$ 
with $2$ nonzero coefficients. Hence, by part (b) of this lemma, 
$\,V\,$ is not exceptional.

iv) We can assume that $\,a_4\,$ is the only nonzero coefficient of
$\,\mu\,$. As $\,\mu\,$ is bad, $\,a_4\geq 3\,$. Thus 
$\,\mu_1\,=\mu-\alpha_4=\omega_3 +(a_4-2)\omega_4\in {\cal
X}_{++}(V)\,$ is a good weight with $2$ nonzero coefficients. Hence, by
part (b) of this proof, $\,V\,$ is not exceptional.
This proves the lemma.
\end{proof}

\begin{lemma}\label{f4nlf}
Let $\,V\,$ be an $\,F_4(K)$-module. Then $\,V\,$ is an exceptional 
$\,\mathfrak g$-module if and only if $\,V\,$ has highest weight
$\,\omega_1\,$ (that is, the adjoint module) or $\,\omega_4\,$.
\end{lemma}\noindent
\begin{proof}
($\,\Longleftarrow\,$)
If $\,V\,$ has highest weight $\,\omega_1\,$, then $\,V\,$ is the
adjoint module, which is exceptional by Example~\ref{adjoint}.
If $\,V\,$ has highest weight $\,\omega_4\,$ then
$\,\dim\,V\,=\,26\,$ (for $\,p\neq 3\,$) or $\,\dim\,V\,=\,25\,$ (for 
$\,p= 3\,$). In any case $\,\dim\,V\,<\,52\,-\,\varepsilon\,$.
Hence, by Proposition~\ref{dimcrit}, $\,V\,$ is an exceptional module.

($\,\Longrightarrow\,$) 
Recall that $\,p\geq 3\,$ for groups of type $\,F_4\,$.
By Lemma~\ref{fared}, {\bf we can assume that $\,{\cal X}_{++}(V)\,$
contains only good weights with at most one nonzero coefficient.}

Let $\,\mu\,=\,a\,\omega_i\,$ (with $\,a\geq 1\,$ and $\,1\leq i\leq 4\,$) 
be a good weight in $\,{\cal X}_{++}(V)$. 

(a) Let $\,a\geq 1\,$ and $\,\mu\,=\,a\,\omega_2\,$. Then
$\,\mu_1\,=\mu -(1221)=\,(a-1)\omega_2\,+\,\omega_3\,\in\,
{\cal X}_{++}(V)$. For $\,a_2\geq 2,\,$ $\,\mu_1\,$ satisfies 
Lemma~\ref{fared}(b). For $\,a=1,\,$ $\,\mu\,=\,\omega_2\,$,
$\,\mu_1\,=\,\omega_3\,$ and $\,\mu_2\,=\mu -(1342)=\,\omega_1 
\,\in{\cal X}_{++}(V)$. For $\,p\geq 3,\,$ these are all good weights. 
Thus $\,s(V)\,\geq\,|W\mu|\,+\,|W\mu_1|\,+\,|W\mu_2|\,=\, %\,96+96+24=
216\,>\,192\,$. Hence, \mbox{by \eqref{192},} $\,V\,$ is not exceptional.

(b) Let $\,a\geq 1\,$ and $\,\mu\,=\,a\omega_3\,$. Then
$\,\mu_1=\mu -(1231)=\,(a-1)\omega_3\,+\,\omega_4\,\in\,{\cal X}_{++}(V)$. 
For $\,a\geq 2,\,$ Lemma~\ref{fared}(b) applies. 
For $\,a=1,\,$ $\,\mu=\omega_3\,$, 
$\,\mu_1=\omega_4\,$ and $\,\mu_2=\,\mu-(0 1 2 1)\,=\,\omega_1\,\in
\,{\cal X}_{++}(V)$. For $\,p\geq 3$, these are all good weights and
$\,|R_{long}^+-R^+_{\omega_3,3}|\,=\,|R_{long}^+-R^+_{\omega_1,3}|=9,\,$ and
$\,|R_{long}^+-R^+_{\omega_4,3}|=6\,$. Thus $\,r_p(V)\,\displaystyle
=\,\frac{96\cdot 9}{24}\,+\,\frac{24\cdot 9}{24}\,+\,\frac{24\cdot 6}{24}
\,=\,51\,>\,48\,$. Hence, \mbox{by~(\ref{rr48}),} $\,V\,$ is not an exceptional
module.

(c) Let $\,a\geq 1\,$ and $\,\mu\,=\,a\,\omega_1\,$. Then  
$\,\mu_1=\mu -(1110)=\,(a-1)\omega_1\,+\,\omega_4\,\in{\cal X}_{++}(V)$. 
For $\,a\geq 2$, Lemma~\ref{fared}(b) applies. For $\,a=1$, we may
assume that $\,V\,$ has highest weight $\,\omega_1\,$, then $\,V\,$ is the
adjoint module, which is exceptional by Example~\ref{adjoint}.

(d) Let $\,a\geq 1\,$ and $\,\mu\,=\,a\,\omega_4\,$.
For $\,a\geq 2$, $\,\mu_1=\mu -\alpha_4=\,\omega_3\,+\,(a-2)\omega_4\,
\in{\cal X}_{++}(V)$. For $\,a\geq 3$, Lemma~\ref{fared}(b) applies. 
For $\,a=2\,$, $\,\mu_1\,=\,\omega_3\,$ satisfies case (b) above.
Hence we can assume that $\,V\,$ has highest weight $\,\omega_4$, 
then $\,\dim\,V\,=\,25\,$ (for $\,p=3\,$) or $\,26\,$ otherwise. 
In both cases $\,\dim\,V\,<\,52\,-\,\varepsilon\,$. Hence,  
by Proposition~\ref{dimcrit}, $\,V\,$ is an exceptional module.

This proves the lemma.
\end{proof}

\subsubsection{Type $\,G_2\,$}

By Theorem~\ref{???}, any $\,G_2(K)$-module $\,V\,$ such that
\begin{equation} \label{36}
s(V)\,=\,\sum_{\stackrel{\scriptstyle\mu\in{\cal X}_{++}(V)}
{\scriptstyle\mu\;good}}\,m_{\mu}\,|W\mu|\,>\,36
\end{equation}
or
\begin{equation}\label{rr12}
r_p(V)\,=\,\sum_{\stackrel{\scriptstyle{\mu\in{\cal X}_{++}(V)}}
{\scriptstyle \mu\;good}}\,m_{\mu}\,
\frac{|W\mu|}{|R_{long}|}\,|R_{long}^+-R^+_{\mu,p}|\,>\,12
\end{equation}
is not an exceptional $\,\mathfrak g$-module.
For groups of type $\,G_2\,$, $\,|W|\,=\,12\,$,
$\,|R|\,=\,12\,=\,2^2\cdot 3\,$, $\,|R_{long}|\,=\,6\,$ and
$\,R_{long}^+\,=\,\{ \alpha_2,\,3\alpha_1\,+\,\alpha_2,\,3\alpha_1\,
+\,2\alpha_2 \}$. Furthermore, $\,|W\omega_1|\,=\,|W\omega_2|\,=\,6\,$ and  
$\,|W(\omega_1 +\omega_2)|\,=\,12\,$, where
$\,\omega_1 \, =\, 2\alpha_1 + \alpha_2\;$ and
$\;\omega_2 \, =\, 3\alpha_1 + 2\alpha_2\,$.
\vspace{2ex}\\
\begin{remark}\label{g2rem} 
If $\,\mu= a\,\omega_1\,+\,b\,\omega_2\,$ and
$\,\Lambda\,=\,A\,\omega_1\,+\,B\,\omega_2\,$ are
dominant weights such that $\,A\geq a\,$ and $\,B\geq b$, then
$\,{\cal X}_{++,\mathbb C}(\mu)\,\subseteq\,{\cal X}_{++,\mathbb
C}(\Lambda)$. For $\,{\cal X}_{++,\mathbb C}(\Lambda)\,=\,(\Lambda\,
-\,Q_+)\cap P_{++}\,$ and 
$\,\Lambda\, -\,Q_+\,=\,(\Lambda\,-\,\mu)\,+(\mu -\,Q_+)$, where  
$\,\Lambda\,-\,\mu\,\in\,Q_+$. Hence $\,(\mu\,-\,Q_+)\,\subseteq\,(\Lambda\, 
-\,Q_+),\,$ which implies the claim. 
\end{remark} 

\begin{lemma}\label{g2nlf}
Let $\,V\,$ be a $\,G_2(K)$-module. Then $\,V\,$ is an exceptional 
$\,\mathfrak g$-module if and only if $\,V\,$ has highest weight
$\,\omega_1\,$ or $\,\omega_2.\,$
\end{lemma}\noindent
\begin{proof}
Recall that $\,p\neq 3\,$ for groups of type $\,G_2\,$ and also $\,p\neq
2\,$ if $\,V\,$ has highest weight $\,\omega_1\,$. 

($\,\Longleftarrow\,$)
If $\,V\,$ has highest weight $\,\omega_1\,$, then $\,\dim\,V\,=\,7\,$
(for $\,p\neq 2\,$) or $\,\dim\,V\,=\,6\,$ (for $\,p= 2\,$). In both
cases $\,\dim\,V\,<\,14\,-\,\varepsilon\,$. Hence, by 
Proposition~\ref{dimcrit}, $\,V\,$ is an exceptional module.
If $\,V\,$ has highest weight $\,\omega_2\,$ then $\,V\,$ is the
adjoint module, which is exceptional by Example~\ref{adjoint}.

($\,\Longrightarrow\,$) Let $\,V\,$ be a $\,G_2(K)$-module. 

First observe that bad weights with $2$ nonzero coefficients do not 
occur in $\,{\cal X}_{++}(V)$, otherwise it
would contradict Theorem~\ref{curt}(i).

{\bf Therefore we can assume that weights with $2$ nonzero coefficients 
occurring in $\,{\cal X}_{++}(V)\,$ are good weights.} 

I) Let $\,\mu\,=\,a\,\omega_1\,+\,b\,\omega_2\,$ (with $\,a\geq
1,\,b\geq 1\,$) be a good weight in $\,{\cal X}_{++}(V)$. 

For $\,a\geq 2\,$ {\bf or} $\,b\geq 2$, by Remark~\ref{g2rem},
$\,\mu_1=\omega_1\,+\,\omega_2\in{\cal X}_{++}(V)$. Hence also
$\,\mu_2=2\omega_1,\;\mu_3=\omega_2,\;\mu_4=\omega_1,\;\mu_5=0\,\in\,
{\cal X}_{++}(V)$.  Thus, for $\,p\geq 5$, $\,s(V)\,\geq\,12+12+6+6+6\,>
\,36\,$. Hence, by (\ref{36}), $\,V\,$ is not exceptional.

For $\,p=2,\,$ a weight of the form $\,a\,\omega_1\,+\,b\,\omega_2\,$,  
with $\,a\geq 2\,$ or $\,b\,\geq 2,\,$ does not occur in 
$\,{\cal X}_{++}(V)\,$ (otherwise it would contradict
Theorem~\ref{curt}(i)).

{\bf Hence if $\,{\cal X}_{++}(V)\,$ contains a good weight $\,\mu\,$ with $2$
nonzero coefficients, then we can assume that 
$\,\mu\,=\,\omega_1\,+\,\omega_2\,$.} We deal with this case in III(a)
below.

II) Let $\,\mu\,=\,a\,\omega_i\,$ (with $\,a\geq 1\,$ and $\,1\leq
i\leq 2\,$) be a weight in $\,{\cal X}_{++}(V)$.

(a) Let $\,a\geq 2\,$ and $\,\mu\,=\,a\omega_2\,$. Then 
$\,\mu_1\,=\,\mu-(\alpha_1+\alpha_2)\,=\,\omega_1\,+\,(a-1)\omega_2\,\in\,
{\cal X}_{++}(V)$. For $\,a\geq 3$, $\,\mu_1\,$ satisfies case I) of
this proof. Hence $\,V\,$ is not exceptional.

For $\,a= 2\,$ and $\,p\geq 5,\,$ $\,\mu\,=\,2\omega_2,\,$
$\,\mu_1=\omega_1 +\omega_2\,$, $\,\mu_2\,=\,\mu\,-\,\alpha_2\,= 
\,3\omega_1\,$, $\,\mu_3\,=\,\mu_1-(\alpha_1+\alpha_2)\,=\,2\omega_1\,$, 
$\,\mu_4\,=\,\mu_3-\alpha_1\,=\,\omega_2\,$ and 
$\,\mu_5\,=\,\mu_4-(\alpha_1+\alpha_2)\,=\,\omega_1\,\in\,{\cal X}_{++}(V)$.
Thus $\,s(V)\,\geq\,6\,+\,6\,+\,12\,+\,6\,+\,6\,+\,6\,=\,42\,>\,36\,$.
Hence, by~(\ref{36}), $\,V\,$ is not exceptional.
For $\,p=2,\,$ $\,2\omega_2\,$ does not occur in $\,{\cal X}_{++}(V)$.

{\bf Hence, if $\,{\cal X}_{++}(V)\,$ contains
$\,\mu\,=\,a\omega_2\,$, then we can assume that $\,a=1$.} This case
is treated in III(b) below. 

(b) Let $\,a\geq 2\,$ and $\,\mu\,=\,a\omega_1\,$. Then 
$\,\mu_1\,=\,\mu-\alpha_1\,=\,(a-2)\omega_1\,+\,\omega_2\,\in\,
{\cal X}_{++}(V)$. For $\,a\geq 4$, $\,\mu_1\,$ satisfies case I) of
this proof. Hence $\,V\,$ is not exceptional.
 
For $\,a= 3\,$, $\,\mu\,=\,3\omega_1\,$. For $\,p=2\,$, $\,\mu\,$ does
not occur in $\,{\cal X}_{++}(V)$. For $\,p>3$, we may assume that $\,V\,$
has highest weight $\,\mu\,=\,3\omega_1\,$. Then  
$\,\mu_1\,=\,\omega_1 +\omega_2\,$, $\,\mu_2\,=\,2\omega_1\,$,  
$\,\mu_3\,=\,\omega_2\,$, $\,\mu_4\,=\,\omega_1\,\in{\cal X}_{++}(V)$. 
By~\cite[p. 413]{gise}, $\,m_{\mu_1}=1,\,m_{\mu_2}=2,\,m_{\mu_3}=3,\, 
m_{\mu_4}=4\,$. Thus $\,s(V)\,\geq\,6 + 12 + 2\cdot 6 + 3\cdot 6 +
4\cdot 6\,=\,72\,>\,36\,$. Hence, by~(\ref{36}), $\,V\,$ is not exceptional. 

For $\,a= 2$, we can assume that $\,V\,$ has highest weight 
$\,\mu=2\omega_1\,$.
Let $\,U=E(\omega_1)\,$ be the irreducible module of highest
weight $\,\omega_1$. For $\,p\neq 2$, $\,\dim\,U\,=\,7\,$.
By Subsection~\ref{roapm} (p. \pageref{roapm}), \mbox{$\,\mathfrak{so}(f)=
\{x\in {\rm End}\,U\,/\,f(xv,w)\,+\,f(v,xw)=0 \}\,$}
is a Lie algebra of type $\,B_3$. By~\cite[p. 103]{hum1}, it is
possible to describe a Lie algebra of type $\,G_2\,$ as a subalgebra of  
$\,\mathfrak{so}(f).\,$ Now, as proved in case III(b)iv)
(p. \pageref{IIIbiv}), $\,E(2\omega_1)\,$ is not an
exceptional $\,\mathfrak{so}(f)$-module. Hence $\,E(2\omega_1)\,$ is also not
exceptional as a $\,G_2(K)$-module.

{\bf Hence, if $\,{\cal X}_{++}(V)\,$ contains
$\,\mu\,=\,a\omega_1\,$, then we can assume that $\,a=1$.} This case
is treated in III(c) below.

III) (a) Let $\,V\,$ be a $\,G_2(K)$-module of highest weight 
$\,\lambda\,=\,\omega_1\,+\,
\omega_2\,$. Then $\,\mu_1=2\omega_1\,$, $\,\mu_2=\omega_2\,$,
$\,\mu_3=\omega_1\,\in\,{\cal X}_{++}(V)$. By~\cite[p. 413]{gise},   
for $\,p\neq 3,\,7,\,$ $\,m_{\mu_2}\,=\,2\,$ and
$\,m_{\mu_3}\,=\,4\,$. Thus $\,s(V)\,\geq\,12 + 2\cdot 6 + 4\cdot
6\,=\,48\,>\,36\,$. Hence, by~\eqref{36}, $\,V\,$ is not exceptional

For $\,p=7$, $\,m_{\mu_1}\,=\,m_{\mu_2}\,=\,1\,$ and
$\,m_{\mu_3}\,=\,2\,$. As $\,|R_{long}^+-R^+_{2\omega_1,7}|\,=\,
|R_{long}^+-R^+_{\omega_1,7}|\,=\,2,\,$ 
$\,|R_{long}^+-R^+_{\omega_2,7}|\,=\,|R_{long}^+-R^+_{\lambda,
7}|\,=\,3,\,$ one has $\,r_7(V)\,\displaystyle
=\,\frac{12}{6}\cdot 3\,+\,\frac{6}{6}\cdot 2\,+\,\frac{6}{6}\cdot 3\,+\,
2\cdot\frac{6}{6}\cdot 2\,=\,15\,>\,12\,$. Hence, by~(\ref{rr12}), 
$\,V\,$ is not an exceptional module.

(b) Now we may assume that $\,V\,$ has highest weight $\,\omega_2= 
\tilde{\alpha}\,$. Then $\,V\,$ is the adjoint module, which is exceptional by
Example~\ref{adjoint}. 

(c) Finally, if $\,V\,$ has highest weight $\,\omega_1\,$, then 
$\,\dim\,V\,=\,6\,$ (for $\,p=2\,$) or $\,7\,$
(otherwise). In both cases, $\,\dim\,V\,<\,14\,-\,\varepsilon\,$, Hence,
by Proposition~\ref{dimcrit}, $\,V\,$ is an exceptional module.

This proves the lemma.
\end{proof}
\newpage

%\include{alt1anan} %%%%inside of which a123new.tex, bnbnbn.tex,
                   %%%%cncncn.tex, dndndn.tex
																		
\appendix
\part*{Appendix}

\section{The Proofs For The Classical Types}\label{appendixa}

Recall that $\,\varepsilon\,=\,\dim\,{\mathfrak z}({\mathfrak g})\,$
and that $\,\dim\,{\mathfrak g}\,=\,|R|\,+\,\ell\,$,
where $\,\ell\,=\,{\rm rank}\,G\,$ and $\,R\,$ is the set of roots
of $\,G\,$. The indexing of simple roots and fundamental weights
corresponds to Bourbaki's tables~\cite[Chap. VI, Tables I-IX]{bourb2}. 

\subsection{Groups or Lie Algebras of Type $\,A_{\ell}$}\label{appal}

\subsubsection{Type $\,A_{\ell}\,,\;\ell\geq 4$}\label{fpp}

{\bf Proof of Theorem~\ref{anlist} - First Part} \\ 
Let $\,\ell\geq 4$. Let $\,V\,$ be an $\,A_{\ell}(K)$-module.
By Lemmas~\ref{4coef},~\ref{3orless},~\ref{nlf3c},  
{\bf it suffices to consider modules $\,V\,$ such that
$\,{\cal X}_{++}(V)\,$ contains only weights with at most $\,2\,$
nonzero coefficients}, that is, weights of the form 
\[
\mu\,=\,a\,\omega_i\,+\,b\,\omega_j,\;\,\mbox{with}\,\;i\,<\,j\;\,
\mbox{and}\;\,a,\,b\,\geq\,0.
\]

I) {\bf Claim~1}:\label{bad2cal} 
{\it Let $\,\ell\geq 4$. Let $\,V\,$ be an $\,A_{\ell}(K)$-module. 
If $\,{\cal X}_{++}(V)\,$ contains a bad weight with 
$\,2\,$ nonzero coefficients, then $\,V\,$ is not an exceptional
module.}

Indeed, let $\,\mu\,=\,a\omega_i\,+\,b\omega_j\,$ (with $\,a\geq 1,\,
b\geq 1\,$) be a nonzero bad weight in $\,{\cal X}_{++}(V)\,$ 
(hence $\,a\geq 2,\,b \geq 2\,$). 

(i) If $\,1< i < j < \ell$, then $\mu_1\,=\,\mu-(\alpha_{i}+\cdots 
+\alpha_{j})\,=\,\omega_{i-1}\,+\,\mbox{$(a-1)\omega_{i}$}\,+\,
\mbox{$(b-1)\omega_{j}$}\,+\,\omega_{j+1}\,\in\,{\cal X}_{++}(V)\,$ 
is a good weight with 4 nonzero coefficients. Hence, by Lemma~\ref{4coef}, 
$\,V\,$ is not an exceptional module.

(ii) If $\,1=i,\;2< j < \ell\,$ (or graph-dually $\,1 <i <\ell -1,
\;j=\ell\,$), then \\ $\mu_1\,=\,\mu-(\alpha_{1}+\cdots +\alpha_{j})\,
=\,(a-1)\omega_{1}\,+\,(b-1)\omega_j\,+\,\omega_{j+1}\in\,{\cal X}_{++}(V)$.
As $\,\mu_1\,$ has $3$ nonzero coefficients, by Lemma~\ref{nlf3c},
$\,V\,$ is not exceptional.

(iii) Let $\,i=1,\,j=2\,$ (or graph-dually $\,i=\ell -1,\,j=\ell\,$) and
$\;\mu\,=\,a\omega_1\,+\,b\omega_2\,$. Then
$\;\mu_1\,=\,\mu\,-\,(\alpha_{1}+\alpha_{2})\,
=\,(a-1)\,\omega_{1}\,+\,(b-1)\,\omega_2\,+\,\omega_{3}\;$ 
is a good weight in $\,{\cal X}_{++}(V)\,$ with 3 nonzero coefficients. 
Hence, by~Lemma~\ref{nlf3c} (graph-dual case of 2(iii)),
$\,V\,$ is not exceptional.

(iv) Let $\,i=1,\,j=\ell\,$ ($\,a\geq 2,\,b\geq 2\,$) and
$\,\mu\,=\,a\omega_1\,+\,b\omega_{\ell}\,$. Then 
$\,\mu_1\,=\,\mu-\alpha_{1}\,=\,(a-2)\omega_1\,+\,\omega_2\,+\,
b\omega_{\ell}\,$, $\,\mu_2\,=\,\mu\,-\,\alpha_{\ell}\,=\,a\omega_{1}\,
+\,\omega_{\ell -1}\,+\,(b-2)\omega_{\ell}\;$ and
$\,\mu_3\,=\,\mu_1-\alpha_{\ell}\,=\,(a-2)\omega_1\,+\,\omega_2\,+\,
\omega_{\ell -1}\,+\,(b-2)\omega_{\ell}\,$ are
all good weights in $\,{\cal X}_{++}(V)$. For $\,\ell\geq 4\,$ and  
$\,a> 2\,$ or $\,b>2$, Lemma~\ref{3orless}(a) applies, whence 
$\,V\,$ is not an exceptional module.

If $\,a=b=2\,$ (hence $\,p=2\,$), then the good weights
$\mu_1\,=\,\mu-\alpha_{1}\,=\,\omega_2\,+\,2\omega_{\ell},\;$
$\;\mu_2\,=\,\mu-\alpha_{\ell}\,=\,2\omega_{1}\,+\,\omega_{\ell -1},\;$
$\,\mu_3\,=\,\mu-(\alpha_1+\cdots +\alpha_{\ell})\,=\,
		     \omega_1\,+\,\omega_{\ell}\;\;$ and
$\;\,\mu_4\,=\,\mu_2\,-\,\alpha_1\,=\,\omega_2\,+\,\omega_{\ell -1}\, 
\in\,{\cal X}_{++}(V)$. Thus, for $\,\ell\geq 4$,
\[
s(V)\,-\,\ell(\ell +1)^2\,=\,\displaystyle\frac{(\ell +1)\ell}{4}
(\ell^2\,-\,3\,\ell\,-2)>\,0\,.
\]
Hence, by~\eqref{allim}, $\,V\,$ is not exceptional. This proves the claim.
\hfill $\Box$ \vspace{1.5ex}

{\bf From now on we can assume that weights with 
$\,2\,$ nonzero coefficients occurring in $\,{\cal X}_{++}(V)\,$ are
good weights.}

II) Let $\,\mu\,=\,a\,\omega_i\,+\,b\,\omega_j\;$ 
(with $\,1\leq i\,<\,j\leq \ell\,$ and $\,a\geq 1,\,b\,\geq\,1\,$) 
be a good weight in $\,{\cal X}_{++}(V)$. \vspace{1.5ex} 

{\bf Claim~2}: \label{claim2al} 
{\it Let $\,\ell\geq 4\,$ and $\,2\leq i\,<\,j\leq \ell-1\,$. If 
$\,V\,$ is an $\,A_{\ell}(K)$-module such that
$\,{\cal X}_{++}(V)\,$ contains a good weight of the form
$\,\mu\,=\,a\,\omega_i\,+\,b\,\omega_j\;$ 
(with $\,a\geq 2,\,b\geq 1\,$ or $\,a\geq 1,\,b\geq 2\,$), then 
$\,V\,$ is not an exceptional $\,\mathfrak g$-module.}

Indeed, if $\,2\leq i<j\leq\ell-1\,$ and  
$\,\mu\,=\,a\omega_i\,+\,b\omega_j\,$, then
$\,\mu_1\,=\,\mbox{$\mu-(\alpha_{i}+\cdots +\alpha_{j})$}\,
=\,\omega_{i-1}\,+\,(a-1)\omega_{i}\,+\,(b-1)\omega_{j}\,
+\,\omega_{j+1}\,\in\,{\cal X}_{++}(V)$.
For $\,a\geq 2,\,b\,\geq\,2,\,$ $\,\mu_1\,$ is good with $4$ nonzero
coefficients. Hence, Lemma~\ref{4coef} applies. 
For $\,a\geq 2,\,b=1\,$ or $\,a=1,\,b\geq 2\,$
$\,\mu_1\,$ is good with $3$ nonzero coefficients.
Hence, by Lemma~\ref{nlf3c}, $\,V\,$ is not exceptional.
This proves the claim. \hfill $\Box$ \vspace{1.5ex}

{\bf Therefore, if $\,{\cal X}_{++}(V)\,$ contains a weight of the form
$\,\mu\,=\,a\,\omega_i\,+\,b\,\omega_j\;$ with 
$\,2\leq i<j\leq\ell-1\,$ and $\,a\geq 1,\,b\geq\,1$, then we can
assume that $\,a=b=1\,$.} These are treated in the sequel.

{\bf Claim 3}: \label{claim3} 
{\it Let $\,V\,$ be an $\,A_{\ell}(K)$-module. If 
$\,{\cal X}_{++}(V)\,$ contains a good weight of the form
$\,\mu\,=\,\omega_2\,+\,\omega_j\,$ (or graph-dually
$\,\mu\,=\,\omega_i\,+\,\omega_{\ell-1}\,$) satisfying the following 
conditions:

(a) $\,4\leq \,j\leq \ell-1\,$ (or graph-dually $\,2\leq
i\leq\ell-3\,$) for $\,\ell\geq 5\,$ and $\,p\geq 2\,$ 

(b) $\,j=3\,$ (or graph-dually $\,i=\ell-2\,$) for ($\,\ell\geq 6\,$
and $\,p\geq 2\,$) or ($\,\ell=5\,$ and $\,p\geq 3\,$),\;\, 
then $\,V\,$ is not an exceptional $\,\mathfrak g$-module.}

Indeed, let $\,\mu\,=\,\omega_2\,+\,\omega_j\in{\cal X}_{++}(V)$. 

(a) For $\,i=2,\,4\leq j\leq \ell-1\,$, 
$\,|W\mu|=\displaystyle\frac{j(j-1)}{2}\binom{\ell+1}{j}\geq
\frac{(\ell+1)\ell(\ell-1)(\ell-2)}{4}\,$. Thus, for $\,\ell\geq 8$,
$\,s(V)\,-\,\ell(\ell+1)^2\,\geq\,\displaystyle\frac{(\ell+1)\ell(\ell^2
-7\ell-2)}{4}\,>\,0\,$. Hence, \mbox{by~\eqref{allim},}
$\,V\,$ is not exceptional. 

Let $\,\mu=\omega_{2}\,+\,\omega_{4}\,$. Then $\,\mu_1\,=\,
\mu-(\alpha_2+\alpha_3+\alpha_4)\,=\,\omega_1\,+\,\omega_5
\,\in\,{\cal X}_{++}(V)$.
Thus, for $\,\ell=7$, $\,s(V)\,\geq\,420+280=700\,>\,448$; for 
$\,\ell=6$, $\,s(V)\,\geq\,210+105=315\,>\,294$. Hence,
by~\eqref{allim}, $\,V\,$ is not exceptional. For
$\,\ell=5$, $\,|R_{long}^+-R^+_{\mu,p}|\geq 8\,$ and   
$\,|R_{long}^+-R^+_{\mu_1,p}|\geq 8.\,$ Thus 
$\,r_p(V) \,\geq\,\displaystyle\frac{90\cdot 8}{30}\,+\,\frac{30\cdot
8}{30}\,=\,32\,>\,30\,$. Hence,
by~(\ref{rral}), $\,V\,$ is not an exceptional module.

Let $\,\mu\,=\,\omega_{2}\,+\,\omega_{5}\,$ (implying $\,\ell\geq 6\,$).
Then $\,\mu_1\,=\,\mu-(\alpha_2+\cdots+\alpha_5)\,=\,\omega_1\,+\,\omega_6
\,\in\,{\cal X}_{++}(V)$. For $\,\ell=7$, $\,s(V)\,\geq\,560+168=728\,
>\,448$. For $\,p\geq 2\,$ and $\,\ell=6$, $\,|R_{long}^+-R^+_{\mu,p}|
\geq 12\,$, thus $\,r_p(V)\,\geq\,\displaystyle\frac{210\cdot
12}{42}\,=\,60\,>\,42\,$. Hence, by~\eqref{rral}, $\,V\,$ is not exceptional. 

Let $\,\mu\,=\,\omega_{2}\,+\,\omega_{6}\,$ (implying $\,\ell=7\,$).
Then $\,\mu_1\,=\,\mu-(\alpha_2+\cdots+\alpha_6)\,=\,\omega_1\,+\,\omega_7
\,\in\,{\cal X}_{++}(V)$. Thus $\,s(V)\,\geq\,420+56=476\,>\,448\,$ and, 
by~\eqref{allim}, $\,V\,$ is not exceptional. 

(b) Let $\,\mu\,=\,\omega_2\,+\,\omega_3\,$. Then
$\,\mu_1\,=\,\mu-(\alpha_2+\alpha_3)\,=\,\omega_{1}\,+\,\omega_{4}\,\in\,
{\cal X}_{++}(V)$.  
For $\,p\neq 2,\,$ $\,|R_{long}^+\,-\,R_{\mu,p}^+|\,=\,3\ell -4\,$ 
and $\,|R_{long}^+\,-\,R_{\mu_1,p}^+|\,=\,4\ell-9.\,$ Thus,
for $\,p\neq 2\,$ and $\,\ell\geq 5\,$,
$\,
r_p(V)\,=\,\frac{(\ell-1)}{2}\cdot
(3\ell -4)\,+\,\frac{(\ell-1)\,(\ell-2)}{6}\cdot (4\ell-9)\,
=\,\frac{2\ell^2(\ell-3)\,+\,7\ell\,-\,3}{3}\,>\,\ell\,(\ell+1)\,.
\,$
For $\,p=2,\,$ $\,|R_{long}^+\,-\,R_{\mu,p}^+|\,=\,\ell\,$ 
and $\,|R_{long}^+\,-\,R_{\mu_1,p}^+|\,=\,3\ell-6.\,$ Thus, for $\,p=2$,
and $\,\ell\geq 6\,$, $\,r_2(V)\,=\,\frac{(\ell-1)}{2}\cdot
\ell\,+\,\frac{(\ell-1)\,(\ell-2)}{6}\cdot (3\ell-6)\,=\,
\frac{\ell^2(\ell-4)+7\ell-4}{2}\,>\,\ell\,(\ell+1)\,$.
In both these cases, by~\eqref{rral}, $\,V\,$ is not exceptional.
This proves the claim. \hfill $\Box$ \vspace{1.5ex}

{\bf Claim 3.a}: \label{claim3a} 
{\it Let $\,\ell\geq \mathbf 6\,$ and $\,V\,$ be an $\,A_{\ell}(K)$-module. 
Let $\,3\leq i<j\,$, \mbox{$\,4\leq j\leq \ell-1\,$} 
(or graph-dually $\,2\leq i\leq
\ell-3,\;i<j\leq \ell-2\,$) and $\,a\geq 1,\,b\geq 1\,$.
%Let $\,5\leq j\leq \ell-1\,$ (or graph-dually $\,2\leq i\leq
%\ell-4\,$) and $\,a\geq 1,\,b\geq 1\,$. 
If $\,{\cal X}_{++}(V)\,$ contains a good weight of the form
$\,\mu\,=\,a\,\omega_{i}\,+\,b\,\omega_{j}\,$,
%$\,\mu\,=\,a\,\omega_{2+k_1}\,+\,b\,\omega_{j-k_2}\,$ (or graph-dually
%$\,\mu\,=\,b\,\omega_{i+k_2}\,+\,a\,\omega_{\ell-1-k_1}\,$) for some
%$\,1\leq k_1\leq k_2\,$ such that $\,2+k_1\,<\,j-k_2\,$, 
then $\,V\,$ is not an exceptional $\,\mathfrak g$-module.}

Indeed, by Claim 2, we can assume $\,\mu\,=\,\omega_i\,+\,\omega_j\,$. 
For $\,3\leq i<j\,$, \mbox{$\,4\leq j\leq \ell-1\,$}, write
$\,\mu\,=\,\omega_{2+k_1}\,+\,\omega_{\ell-1-k_2}\,$, 
for some $\,1\leq k_1\,$ and $\,k_2\geq 0\,$ such that $\,2+k_1\,
<\,\ell-1-k_2\,$. 
By Claim~3 and graph-duallity, we can assume $\,1\leq k_1\leq k_2\,$. Then
$\,\mu_1\,=\,\omega_{2}\,+\,\omega_{\ell-1-k_2+k_1}\,\in\,{\cal
X}_{++,\mathbb C}(\mu)\subset{\cal X}_{++}(V)\,$ (by
Lemma~\ref{lamu}). As $\,4\leq \ell-1-k_2+k_1\,\leq\ell-1\,$, by Claim~3,
$\,V\,$ is not exceptional. This proves Claim~3.a. \hfill $\Box$ \vspace{1.5ex}

{\bf Claim 4}: \label{claim4}
{\it Let $\,V\,$ be an $\,A_{\ell}(K)$-module. If 
$\,{\cal X}_{++}(V)\,$ contains a good weight of the form
$\,\mu\,=\,a\,\omega_1\,+\,b\,\omega_j\,$ (or graph-dually
$\,\mu\,=\,b\,\omega_i\,+\,a\,\omega_{\ell}\,$) satisfying the following 
conditions:

(a) $\,a=2,\,b=1\,$ and $\,3\leq j\leq\ell-1\,$ (or graph-dually 
$\,2\leq i\leq\ell-2\,$) for $\,\ell\geq 4\;$ or

(b) $\,a=2,\,b=1\,$ and $\,j=2\,$ (or graph-dually $\,i=\ell-1\,$)
for ($\,\ell\geq 5\,$ and $\,p\geq 2\,$) or ($\,\ell=4\,$ and $\,p\geq 5\;$) or

(c) $\,a=1,\,b=2\,$ for $\,\ell\geq 4\,$ and $\,2\leq i,\,j\leq\ell-1\,$ or

(d) ($\,a\geq 2,\,b\geq 2\,$) or ($\,a\geq 3,\,b\geq 1\,$) or ($\,a\geq 1,
\,b\geq 3\,$) and $\,2\leq j\leq\ell-1\,$ for $\,\ell\geq 4\,$,\;\,
then $\,V\,$ is not an exceptional $\,\mathfrak g$-module.}

Indeed, let $\,\mu\,=\,a\,\omega_1\,+\,b\,\omega_j\,$ (with 
$\,2\leq j\leq\ell-1\,$) be a good weight in $\,{\cal X}_{++}(V)$.

(a) Let $\,a=2,\,b=1,\,$ $\,3\leq j\leq\ell-1\,$
and $\,\mu\,=\,2\omega_1\,+\,\omega_j\,$.
Then $\,\mu_1\,=\,\mu-\alpha_1\,=\,\omega_2\,+\,\omega_j\in{\cal
X}_{++}(V)\,$ satisfies Claim~3 (p. \pageref{claim3}). Hence, for 
($\,4\leq \,j\leq \ell-1\,$, $\,\ell\geq 5\,$ and $\,p\geq 2\,$) or 
($\,j=3\,$ ($\,\ell\geq 6\,$ and $\,p\geq 2\,$) or ($\,\ell=5\,$ and
$\,p\geq 3\,$)), $\,V\,$ is not an exceptional module. 

Let $\,\ell=5\,$, $\,p=2\,$ and $\,\mu\,=\,2\omega_1\,+\,\omega_3\,$.
Then $\,\mu_1\,=\,\omega_2\,+\,\omega_3\,$ and $\,\mu_2\,=\,\mu_1- 
(\alpha_2+\alpha_3)\,=\,\omega_1\,+\,\omega_4\,\in{\cal X}_{++}(V)$.
As $\,|R_{long}^+\,-\,R_{\mu,2}^+|\,=\,|R_{long}^+\,-\,R_{\mu_2,2}^+|\,=\,9\,$
and $\,|R_{long}^+\,-\,R_{\mu,2}^+|\,=\,5\,$, one has 
$\,r_p(V)\geq \displaystyle\frac{60\cdot 9}{30}\,+\,\frac{60\cdot 5}{30}\,+\,  
\frac{60\cdot 9}{30}\,=\,46\,>\,30\,$. Hence, by~\eqref{rral}, $\,V\,$
is not exceptional.

Now let $\,\ell=4,\,$ $\,p\geq 2\,$ and $\,\mu\,=\,2\omega_1\,+\,\omega_3\,$.
Then $\,\mu_1\,=\,\omega_2\,+\,\omega_3\,$ and $\,\mu_2\,=\,\omega_1\,
+\,\omega_4\,\in{\cal X}_{++}(V)$. As $\,|R_{long}^+\,-\,
R_{\mu,p}^+|\,\geq 6,\,|R_{long}^+\,-\,R_{\mu_1,p}^+|\,\geq\,4\,$
and $\,|R_{long}^+\,-\,R_{\mu_2,p}^+|\,\geq\,6\,$, one has 
$\,r_p(V)\geq \displaystyle\frac{30\cdot 6}{20}\,+\,\frac{30\cdot 4}{20}\,+\,  
\frac{20\cdot 6}{20}\,=\,21\,>\,20\,$. Hence, by~\eqref{rral}, $\,V\,$
is not exceptional. This proves part (a).

(b) Let $\,\mu\,=\,2\omega_1\,+\,\omega_2\,$. For $\,p=2,\,$ 
$\,\mu\,\in\,{\cal X}_{++}(V)\,$ only if $\,V\,$ has 
highest weight $\,\lambda\,=\,\omega_{1}\,+\,\omega_{2}\,+\,\omega_{2+i}\,+  
\,\omega_{\ell-i}\,$ (with $\,i\geq 1\,$ and $\,\ell\geq 5\,$; for
$\,\ell=4\,$ $\,\mu\,$ does not occur in $\,{\cal X}_{++}(V)\,$).  In
these cases, Lemma~\ref{4coef} applies. 

For $\,p\geq 3$, consider 
$\,\mu_1\,=\,\mu-(\alpha_{1}+\alpha_{2})\,=\,\omega_{1}\,+\,\omega_{3},\;$
$\,\mu_2\,=\,\mu-\alpha_1\,=\,2\omega_2\;$ and
$\,\mu_3\,=\,\mu_1-(\alpha_1 + \alpha_2+\alpha_{3})\,=\,\omega_{4}\in
\,{\cal X}_{++}(V)$. For $\,p\geq 3$, $\,|R_{long}^+-R_{\mu,p}^+|\geq  
\ell$, \linebreak $\,|R_{long}^+-R_{\mu_1,p}^+|=3\ell-4\,$, 
$\,|R_{long}^+-R_{\mu_2,p}^+|=2(\ell-1)\,$
and $\,|R_{long}^+-R_{\mu_3,2}^+|=4(\ell-3)$. Thus, for
$\,\ell\geq 5\,$, 
\begin{align*}
r_p(V) &\geq\displaystyle\ell\,+\,\frac{(\ell-1)\cdot
(3\ell-4)}{2}\,+\,\frac{2(\ell-1)}{2}\,+\,\frac{(\ell-1)(\ell-2)
\cdot 4(\ell-3)}{4!}\\
& =\,\frac{\ell\,(\ell+1)\,(\ell+2)}{6}\,>\,\ell\,(\ell+1)\,.
\end{align*} 
For $\,\ell=4\,$ and $\,p\geq 5$, $\,|R_{long}^+\,-\,R_{\mu,p}^+|=
7\,$. Thus, $\,r_p(V)\,\geq\,\displaystyle\frac{20\cdot 7}{20}\,+\,
\frac{30\cdot 8}{20}\,+\,\frac{10\cdot 6}{20}\,+\,\frac{5\cdot
4}{20}\,=\,21\,>\,20\,$ Hence, by~\eqref{rral}, $\,V\,$ is not exceptional.
%This proves part (b).

(c) Let $\,\mu\,=\,\omega_1\,+\,2\omega_j\,$. Then
$\,\mu_1\,=\,\mu-\alpha_j\,=\,\omega_1\,+\,\omega_{j-1}\,
+\,\omega_{j+1}\in{\cal X}_{++}(V)$. For $\,3\leq j\leq\ell-1\,$ and
$\,\ell\geq 4\,$, $\,\mu_1\,$ has $3$ nonzero coefficients, hence
Lemma~\ref{nlf3c} applies. 
Let $\,\mu\,=\,\omega_1\,+\,2\omega_2\,$. Then
$\,\mu_1\,=\,\mu-\alpha_2\,=\,2\omega_1\,+\,\omega_{3}\in{\cal
X}_{++}(V)\,$ satisfies case (a) of this claim. Hence, for $\,\ell\geq
4\,$, $\,V\,$ is not exceptional. %This proves part (c).

(d)i) Let $\,a\geq 2,\,b\geq 2,\,$ $\,2\leq j\leq \ell-1\,$ and  
$\,\mu\,=\,a\,\omega_1\,+\,b\,\omega_j\,$. Then
$\,\mu_1\,=\,\mu-(\alpha_{1}+\cdots +\alpha_{j})\,
=\,(a-1)\omega_{1}\,+\,(b-1)\omega_j\,+\,\omega_{j+1}\in
{\cal X}_{++}(V)\,$ is a good weight with $3$ nonzero coefficients.
Hence, for $\,\ell\geq 4\,$, Lemma~\ref{nlf3c} applies.

ii) For $\,a\geq 3,\,b= 1\,$ and $\,3\leq j\leq \ell-1\,$,
$\,\mu_2\,=\,\mu-\alpha_1\,=\,(a-2)\omega_1\,  
+\,\omega_2\,+\,\omega_j\in {\cal X}_{++}(V)\,$ is a good weight with
$3$ nonzero coefficients. Hence, for $\,\ell\geq 4\,$,
Lemma~\ref{nlf3c} applies.

Let $\,a\geq 3,\,b=1\,$ and $\,\mu=\,a\omega_{1}\,+\,\omega_{2}\,$.
Then $\,\mu_1\,=\mu-(\alpha_{1}+\alpha_{2}) =(a-1)\omega_{1}\, 
+\,\omega_{3}$, $\,\mu_2\,=\,\mu_1-(\alpha_1 +\alpha_2+\alpha_3)\, 
=\,(a-2)\omega_1\,+\,\omega_4\,$, $\,\mu_3\,=\,\mu-\alpha_1\,=\, 
(a-2)\omega_1\,+\,2\omega_{2}\,$ and
$\,\mu_4\,=\,\mu_3-(\alpha_1+\alpha_{2})\,=\,(a-3)\omega_1\,+\,
\omega_{2}\,+\,\omega_3\,$ are good weights in $\,{\cal X}_{++}(V)$.
For $\,a\geq 4\,$ and $\,\ell\geq 4\,$, $\,\mu_4\,$ satisfies
Lemma~\ref{nlf3c}. For $\,a=3\,$ and $\,\ell\geq 4$, $\,\mu_1\,=\,
2\omega_1\,+\,\omega_3\,$ satisfies part (a) of this claim.

iii) Let $\,a=1,\,b\geq 3,\,$ $\,2\leq j\leq\ell-1\,$ and 
$\,\mu\,=\,\omega_1\,+\,b\,\omega_j\,$. Then
$\,\mu_1\,=\,\mu-\alpha_j\,=\,\omega_1\,+\,\omega_{j-1}\,+\,(b-2)\omega_j
\,+\,\omega_{j+1}\,\in\,{\cal X}_{++}(V)\,$ has (at least) $3$ nonzero
coefficients. Hence, Lemma~\ref{nlf3c} or Lemma~\ref{4coef} applies.

This proves the claim. \hfill $\Box$ \vspace{1.5ex}

{\bf Claim 5}: \label{claim5al}
{\it Let $\,\ell\geq 4\,$. If $\,V\,$ is an $\,A_{\ell}(K)$-module
such that $\,{\cal X}_{++}(V)\,$ contains a good weight of the form
$\,\mu\,=\,a\,\omega_1\,+\,b\,\omega_{\ell}\,$ (or graph-dually
$\,\mu\,=\,b\,\omega_1\,+\,a\,\omega_{\ell}\,$) satisfying the following 
conditions:

(a) ($\,a\geq 2,\,b\geq 2\,$) or ($\,a\geq 3,\,b\geq 1\,$) or ($\,a\geq 1,
\,b\geq 3\,$) for $\,\ell\geq 4\;$ or  

(b) $\,a=2,\,b=1\,$ (or graph-dually $\,a=1,\,b=2\,$) for $\,\ell\geq
5\,$, \\
then $\,V\,$ is not an exceptional $\,\mathfrak g$-module.}

Indeed, let $\,\mu\,=\,a\,\omega_1\,+\,b\,\omega_{\ell}\,$ be a good
weight in $\,{\cal X}_{++}(V)$.

(a) For $\,a\geq 2,\,b\geq 2\,$,
$\,\mu_1\,=\,\mu-\alpha_{1}\,=\,(a-2)\omega_1\,+\,\omega_2\,+\,
	 b\omega_{\ell},\;$
$\,\mu_2\,=\,\mu-\alpha_{\ell}\,=\,a\omega_{1}\,+\,
		    \omega_{\ell -1}\,+\,(b-2)\omega_{\ell}\;$ and 
$\;\mu_3\,=\,\mu_1-\alpha_{\ell}\,=\,(a-2)\,\omega_1\,+\,\omega_2\,+\,
\omega_{\ell -1}\,+\,(b-2)\,\omega_{\ell}\,\in\,{\cal X}_{++}(V)$.
Hence, for $\,\ell\geq 4$, by~\eqref{allim}, $\,V\,$ is not
exceptional since  $\,s(V)\geq\displaystyle (\ell+1)\ell\,+\,2\cdot 
\frac{(\ell+1)\ell(\ell-1)}{2}\,+\,\frac{(\ell+1)\ell(\ell-1)(\ell-2)}{4}\,
=\,\frac{(\ell+1)\ell(\ell^2 +\ell + 2)}{4}\,>\,\ell(\ell+1)^2\,$. 
%Hence, by~\eqref{allim}, $\,V\,$ is not exceptional.

For $\,a\geq 3,\,b=1\,$ (or graph-dually $\,a=1,\,b\geq 3\,$), 
$\,\mu_1\,$ (resp., $\,\mu_2\,$) has $3$ nonzero coefficients. Hence,
Lemma~\ref{nlf3c} applies. This proves part (a).

(b) Let $\,\mu\,=\,2\omega_{1}\,+\,\omega_{\ell}\,$. Then 
$\,\mu_1=\mu-\alpha_{1}=\omega_2\,+\,\omega_{\ell}\;$ and
$\;\mu_2=\mbox{$\mu-(\alpha_1 + \cdots +\alpha_{\ell})=\omega_{1}$}\,\in
\,{\cal X}_{++}(V)$. 
For $\,p=2,\,$ $\,\mu\in {\cal X}_{++}(V)\,$ only if $\,V\,$ has highest 
weight $\,\lambda\,=\,\omega_1\,+\,\omega_{i+1}\,+\,\omega_{\ell-i}\,$,
for some $\,i\geq 1\,$. In these cases Lemma~\ref{nlf3c} applies.
For $\,p\geq 3,\,$ $\,\mbox{$|R_{long}^+-R_{\mu,p}^+|$}\,=\,2(\ell
-1)\,$, $\,|R_{long}^+-R_{\mu_1,p}^+|\,=\,3\ell -4\,$ and
$\,|R_{long}^+-R_{\mu_2,p}^+|\,=\,\ell\,$. Thus, for $\,\ell\geq 5$,
\[
r_p(V)\, \geq\,2(\ell-1)
\,+\,\frac{(\ell-1)}{2}\cdot (3\ell-4)\,+\,\frac{\ell}{\ell}\, =\,
\frac{3\ell^2\,-\,3\ell\,+\,2}{2}\,>\,\ell\,(\ell+1)\,.
\]
%Now $\,\displaystyle\frac{3\ell^2\,-\,3\ell\,+\,2}{2}\,>\,\ell\,(\ell+1)\;
%\Longleftrightarrow\;\ell^2\,-\,5\ell\,+\,2\,>\,0\;
%\Longleftrightarrow\;\ell\geq 5.\,$ 
Hence, by~\eqref{rral}, $\,V\,$ is not exceptional. This proves the
claim.  \hfill $\Box$ \vspace{1.5ex}

\begin{lemma}\label{III.iii}
Let $\,p\geq 5\,$ and $\,3\leq a < p\,$. 
Then the (irreducible) $\,A_{\ell}(K)$-module $\,V\,=\,E(a\omega_1)\,$
(of highest weight $\,a\omega_1\,$) is not an exceptional 
$\,\mathfrak g$-module.
\end{lemma}\noindent
\begin{proof}
The module $\,V\,$ is isomorphic to the space of all homogeneous polynomials  
of degree $\,a\,$ in the indeterminates $\,x_1,\,x_2,\ldots,\,x_{\ell+1}\,$.
We claim that, for \mbox{$\,3\leq a<p$}, the nonzero element
$\,v\,=\,\displaystyle\sum_{k=1}^{\ell+1}\,x^a_k\,\in\,V\,$ has
trivial isotropy subalgebra, that is, \mbox{$\,\mathfrak g_v\,=\,\{ 0\}$}.
For let 
$\,z\,=\,\displaystyle\sum_{i,j=1}^{\ell+1}\,\lambda_{ij}\,x_i\,\partial_j\,
\in\,\mathfrak{gl}(\ell+1)\,$. Then 
\[
z\cdot v\,=\,(\displaystyle\sum_{i,j=1}^{\ell+1}\,\lambda_{ij}\,x_i\, 
\partial_j)\cdot(\sum_{k=1}^{\ell+1}\,x^a_k)\,=\,
\displaystyle a\,\sum_{i,k=1}^{\ell+1}\,\lambda_{ik}\,x_i\,x_k^{a-1}\,.
\]
Thus, $\,z\in\mathfrak g_v\,\Longleftrightarrow\,z\cdot v\,=\,0\,
\Longleftrightarrow\,
\displaystyle
a\,\sum_{i,k=1}^{\ell+1}\,\lambda_{ik}\,x_i\,x_k^{a-1}\,=\,0\,$.
As $\,3\leq a<p$, this only happens if all
$\,\lambda_{ik}\,=\,0,\,$ implying $\,z=0$. This proves the claim.
The result follows.
\end{proof}\vspace{1ex}

{\bf Claim~6}: \label{claim6al}
{\it Let $\,\ell\geq 4$. If $\,V\,$ is an $\,A_{\ell}(K)$-module such that
$\,{\cal X}_{++}(V)\,$ contains 

(a) $\,\mu\,=\,a\,\omega_i\,$, with $\,a\geq 3\,$ and $\,2\leq i\leq  
\ell-1\;$ or

(b) $\,\mu\,=\,a\,\omega_1\,$ (or graph-dually
$\,\mu\,=\,a\,\omega_{\ell}\,$), with $\,a\geq 4\,$, \\
then $\,V\,$ is not an exceptional $\,\mathfrak g$-module.} Indeed:

(a) let $\,a\geq 3,\,$ $\,2\leq i\leq\ell-1\,$ and $\,\mu\,=\,a\omega_i\in
{\cal X}_{++}(V)$. Then $\,\mu_1\,=\,\mu\,-\,\alpha_{i}\,=\,\omega_{i-1}\,
+\,(a-2)\omega_{i}\,+\,\omega_{i+1}\,\in\,{\cal X}_{++}(V)\,$ 
is a good weight with 3 nonzero coefficients.
Hence, by Lemma~\ref{nlf3c}, $\,V\,$ is not an exceptional module.

(b) Let $\,a\geq 4\; (\mbox{or}\; 3)\,$ and $\,\mu\,=\,a\omega_1\,$.
For $\,p=2$, $\,\mu\,\in\,{\cal X}_{++}(V)\,$ only if
$\,V\,$ has highest weight $\,\lambda\,=\,\omega_1\,+\,\omega_2\,+\cdots 
+\,\omega_a\,+\,\omega_{\ell-a+2}\,+\cdots +\,\omega_{\ell}\,$ (with
$\,\frac{\ell+2}{2}\,>\,a\,$) or \mbox{$\,\lambda\,=\,\omega_2\,+\cdots 
+\,\omega_{a+1}\,+\,\omega_{\ell-a+1}\,+\cdots +\,\omega_{\ell}\,$}
(with $\,\frac{\ell}{2}\,>\,a\,$). In any of these cases,
Lemma~\ref{4coef} applies. 

i) Let $\,a\geq 5$. 
For $\,p\geq 3$, $\,\mu_1\,=\mu-\alpha_1=\,(a-2)\omega_{1}\,+\,\omega_{2}\,
\in\,{\cal X}_{++}(V)$. Hence for $\,\ell\geq 4$, $\,\mu_1\,$ satisfies
Claim~4(d) (p. \pageref{claim4}) and $\,V\,$ is not exceptional.

ii) Let $\,\mu\,=\,4\,\omega_1\,$. Then $\,\mu_1\,=\,2\omega_{1}\, 
+\,\omega_{2}\,\in\,{\cal X}_{++}(V)$. For ($\,\ell\geq 5\,$ and
$\,p\geq 2\,$) or ($\,\ell=4\,$ and $\,p\geq 5\,$), $\,\mu_1\,$ satisfies 
Claim~4(b) (p. \pageref{claim4}). Hence $\,V\,$ is not exceptional.
For $\,\ell=4\,$ and $\,p=3,\,$ $\,\mu\in {\cal X}_{++}(V)\,$ only  
if $\,V\,$ has highest weight $\,\lambda\,\in\{\,2\omega_1\,+\,2\omega_2\,+\,
2\omega_4\,,\,2\omega_1\,+\,\omega_2\,+\,2\omega_3\,+\,\omega_4\,,
\,\omega_1\,+\,2\omega_2\,+\,2\omega_3\,+\,2\omega_4\,\}$.
In these cases, Lemma~\ref{nlf3c} or Lemma~\ref{4coef} applies. 

This proves the claim. \hfill $\Box$ \vspace{1.5ex}

Now we deal with some particular cases.

P.1) Let $\,\ell=4\,$ and $\,p=3$. We may assume that $\,V\,$ has highest  
weight $\,\mu\,=\,2\omega_1\,+\,\omega_2\,$ (or graph-dually
$\,\mu\,=\,\omega_3\,+\,2\omega_4\,$). Then
$\,\mu_1\,=\,\omega_{1}\,+\,\omega_{3}\,$, $\,\mu_2\,=\,2\omega_2\,$
and $\,\mu_3\,=\,\omega_4\in {\cal X}_{++}(V)$. 
Let $\,v_o\,$ be a highest weight vector of $\,V$. 
Applying $\,e_1,\,e_2\,$ to the nonzero weight vectors $\,v_1\,=\,f_1\,
f_2\,v_o\,$ and $\,v_2\,=\,f_2\,f_1\,v_o\,$ (of weight $\,\mu_1\,$), and 
using relations~\eqref{relations}, one proves that, for $\,p\neq 2$,
$\,v_1,\,v_2\,$ are linearly independent. Hence, for $\,p\neq 2$, 
$\,m_{\mu_1}\geq 2\,$. Now for $\,p=3,\,$ 
$\,|R_{long}^+\,-\,R_{\mu,3}^+|\,=\,|R_{long}^+\,-\,R_{\mu_2,3}^+|\,=\,6\,$,
$\,|R_{long}^+\,-\,R_{\mu_1,3}^+|\,=\,8\,$ and 
$\,|R_{long}^+\,-\,R_{\mu_3,3}^+|\,=\,4.\,$ Thus 
$\,r_3(V)\,=\,\displaystyle\frac{20\cdot
6}{20}\,+\,2\cdot\frac{30\cdot 8}{20}\,+\,\frac{10\cdot
6}{20}\,+\,\frac{5\cdot 4}{20}\,=\,34\,>\,20\,$.
Hence, by~(\ref{rral}), $\,V\,$ is not an exceptional module.

P.2) Let $\,\ell=4\,$. We may assume that $\,V\,$ has highest  
weight $\,\mu\,=\,2\omega_1\,+\,\omega_4\,$ (or graph-dually
$\,\mu\,=\,\omega_1\,+\,2\omega_4\,$). Then
$\,\mu_1\,=\,\omega_{2}\,+\,\omega_{4}\,$, $\,\mu_2\,=\,\omega_1\in  
{\cal X}_{++}(V)$. By~\cite[p. 168]{buwil}, for $\,p\geq 3\,$,
$\,m_{\mu_2}\geq 3\,$. Also for $\,p\geq 3$, 
$\,|R_{long}^+\,-\,R_{\mu,3}^+|\geq 6\,$,  
$\,|R_{long}^+\,-\,R_{\mu_1,3}^+|\,=\,8\,$ and 
$\,|R_{long}^+\,-\,R_{\mu_3,3}^+|\,=\,4.\,$ Thus,
$\,r_p(V)\,\geq\,\displaystyle\frac{20\cdot
6}{20}\,+\,\frac{30\cdot 8}{20}\,+\,3\cdot\frac{5\cdot 4}{20}\,=\,21\,>\,20\,$.
Hence, by~(\ref{rral}), $\,V\,$ is not an exceptional module.
\vspace{1.5ex}

{\bf Therefore, if $\,{\cal X}_{++}(V)\,$ contains a weight of the form
$\,\mu\,=\,a\,\omega_i\,+\,b\,\omega_j\;$ with $\,1 \leq i< j \leq \ell\,$,
then we can assume that $\,a=b=1\,$.} These cases are treated in the 
sequel. We start with some particular cases.

P.3) Let $\,\ell=5\,$ and $\,p=2\,$. We may assume that $\,V\,$ has 
highest weight
\mbox{$\,\mu\,=\,\omega_2\,+\,\omega_3\,$} (or graph-dually
$\,\mu\,=\,\omega_3\,+\,\omega_4\,$). Then $\,\mu_1\,=\,\omega_1\,+\,\omega_4\,
\in\,{\cal X}_{++}(V)$. By Lemma~\ref{mulmin}, 
$\,m_{\mu_1}=2\,$. As $\,|R_{long}^+-R^+_{\mu,2}|=5\,$ and
$\,|R_{long}^+-R^+_{\mu_1,2}|=9\,$, one has
$\,r_2(V)\,\geq\,\displaystyle\frac{60\cdot 5}{30}
\,+\,2\cdot\frac{60\cdot 9}{30}\,=\,46\,>\,30\,$.
Hence, by~\eqref{rral}, $\,V\,$ is not exceptional. 

P.4) Let $\,\ell=4\,$. We may assume that $\,V\,$ has highest weight
$\,\mu\,=\,\omega_2\,+\,\omega_3\,$. Then $\,\mu_1\,=\,\omega_1\, 
+\,\omega_4\,\in\,{\cal X}_{++}(V)$.  
For $\,p\neq 3,\,$ by~\cite[p. 167]{buwil}, $\,m_{\mu_1}=2$
Thus, for $\,p\geq 5\,$, $\,r_p(V)\,=\,\frac{3}{2}\cdot 8\,+\,2\cdot
7\,=\,26\,>\,20$. Hence, by~\eqref{rral}, $\,V\,$ is not exceptional.

U.1) For $\,\ell=4\,$ and $\,p=2,\,3$, the $\,A_4(K)$-modules of
highest weight $\,\mu\,=\,\omega_2\,+\,\omega_3\,$ are unclassified 
{\bf (N. 4 in Table~\ref{leftan})}.

P.5) Let $\,\ell \geq 4\,$ and $\,\mu\,=\,3\omega_1\,$.
For $\,p=2\,$, see proof of Claim~6 (b) (p.\pageref{claim6al}).  
For $\,p=3,\,$ $\,\mu\in {\cal X}_{++}(V)\,$ only if $\,V\,$ has
highest weight $\,\lambda\,=\,2\omega_1\,+\,\omega_{2+j}\,+\, 
\omega_{\ell-j}\,$ or $\,\lambda\,=\,\omega_1\,+\,\omega_{2+j}\,+\, 
\omega_{2+j+k}\,+\,\omega_{\ell-(j+k)}\,+\,\omega_{\ell-j}\,$  
(for some $\,j\geq 1,\,k\geq 0\,$ and $\,\ell\geq 3\,$), for instance.
In these cases, Lemma~\ref{nlf3c} applies.
For $\,p\geq 5,\,$ we can assume that $\,V\,$ has highest weight 
$\,\mu\,=\,3\,\omega_1\,$. Hence,
by Lemma~\ref{III.iii} (p. \pageref{III.iii}), $\,V\,$ is not exceptional. 
\vspace{1.5ex}

{\bf Claim 7}: \label{claim7al}
{\it Let $\,\ell\geq 7\,$ and $\,1\leq i < j,\;4\leq j\leq \ell-3\,$ 
(or graph-dually $\,4\leq i\leq \ell-3,\;i < j < \ell\,$).
If $\,V\,$ is an $\,A_{\ell}(K)$-module such that
$\,{\cal X}_{++}(V)\,$ contains a good weight of the form
$\,\mu\,=\,a\,\omega_i\,+\,b\,\omega_j\,$ (with $\,a\geq 1,\,
b\geq 1\,$), then $\,V\,$ is not an exceptional $\,\mathfrak g$-module.}

Indeed, for $\,\ell\geq 7\,$, by Claim~2 (p.\pageref{claim2al}) or Claim~4 
(p.\pageref{claim4}), we can assume $\,\mu\,=\,\omega_i\,+\,\omega_j\,$. 
For $\,4\leq j\leq \ell-3\,$,
$\,|W\mu|\,=\,\displaystyle\binom{j}{i}\binom{\ell+1}{j}\,\geq\,
\frac{4\,(\ell+1)\ell(\ell-1)(\ell-2)}{4!}$. 
Thus, for $\,\ell\geq 10$, $\,s(V)\,-\,\ell(\ell +1)^2\,\geq \,\displaystyle
%\frac{(\ell+1)\ell(\ell-1)(\ell-2)}{6}\,-\,\ell(\ell +1)^2\,=\,
\frac{(\ell +1)\ell}{6} (\ell^2\,-\,9\,\ell\,-\,4)>\,0$. 
Hence, \mbox{by~\eqref{allim},} $\,V\,$ is not exceptional. 

For $\,7\leq \ell\leq 9\,$, $\,2\leq i<j\,$ and $\,4\leq j\leq\ell-3\,$,
$\,\mu\,=\,\omega_i\,+\,\omega_j\,$ satisfies Claim~3 (p. \pageref{claim3})  
or Claim~3.a (p. \pageref{claim3a}). Hence $\,V\,$ is not exceptional.

Let $\,7\leq \ell\leq 9\,$ and $\,\mu\,=\,\omega_1\,+\,\omega_j\,$
(with $\,4\leq j\leq\ell-3\,$). Then $\,\mu_1\,=\,\omega_{j+1}\,\in
{\cal X}_{++}(V)$. 
For ($\,\ell=9\,$ and $\,4\leq j\leq 6\,$) or ($\,\ell=8\,$ and $\,j=5\,$), 
one can easily prove that $\,s(V)\,\geq\,|W\mu|\,+\,|W\mu_1|\,>\,
\ell(\ell+1)^2\,$. Hence, by~\eqref{allim}, $\,V\,$ is not exceptional.
%For $\,j=4\,$ $\,\ell=9,\,$ $\,s(V)\,\geq\,850+252\,=\,1102\,>\,900\,$; 
%$\,j=5,6\,$$\,\ell=9,\,$ $\,s(V)\,\geq\,|W\mu|=1260\,>\,900$;  
%for $\,j=5$ $\,\ell=8,\,$ $\,s(V)\,\geq\, 630+84\,=\,714\,>\,648$.  

Let $\,\mu\,=\,\omega_{1}\,+\,\omega_{4}\,$ and $\,\mu_1\,=\,\omega_5\,$.
Then for $\,p\geq 2$, $\,|R_{long}^+-R^+_{\mu,2}|\geq\,3\,(\ell-2)\,$ and
$\,|R_{long}^+-R^+_{\mu_1,p}|\,=\,5\,(\ell-4)\,$.
Thus, for $\,\ell=8\,$, $\,r_p(V)\,\geq\,\displaystyle\frac{504+\cdot 18}{72}
\,>\,72\,$. For $\,\ell=7,\,$ 
$\,r_p(V)\,\geq\,\displaystyle\frac{280\cdot 15}{56}\,+\,\frac{56\cdot 15}
{56}\,=\,90\,>\,56\,$. Hence, by~(\ref{rral}), $\,V\,$ is not exceptional.
This proves the claim. \hfill $\Box$ \vspace{1.5ex}

In the following cases, for 
$\,p\neq 2$, $\,|R_{long}^+-R^+_{\omega_1\,+\,\omega_j,p}|\,=\,
\ell+(j-1)(\ell-j+1)\,$ and for $\,p=2$, $\,|R_{long}^+-R^+_{\omega_1\,+\,
\omega_j,2}|\,=\,(j-1)\,(\ell-j+2)\,$.
Also for any $\,p$, $\,|R_{long}^+-R^+_{\omega_j,p}|\,=\,j\,(\ell-j+1)\,$.
\vspace{1.5ex}

{\bf Claim 8}: \label{claim8al} 
{\it Let ($\,\ell\geq 6\,$ for $\,p\geq 3\,$) or ($\,\ell\geq 8\,$ for 
$\,p=2\,$). If $\,V\,$ is an $\,A_{\ell}(K)$-module such that
$\,{\cal X}_{++}(V)\,$ contains $\,\mu\,=\,\omega_1\,+\,\omega_{3}\,$
(or graph-dually $\,\mu\,=\,\omega_{\ell-2}\,+\,\omega_{\ell}\,$), 
then $\,V\,$ is not an exceptional $\,\mathfrak g$-module.}

Indeed, let $\,\mu\,=\,\omega_{1}\,+\,\omega_{3}\,$. Then
$\,\mu_1\,=\,\mu-(\alpha_1+\alpha_2+\alpha_3)\,=\,\omega_{4}\,\in\,
{\cal X}_{++}(V)$. For $\,p\neq 2,\,$ $\,|R_{long}^+-R_{\mu,p}^+|\,
=\,3\ell -4\,$ and $\,|R_{long}^+-R_{\mu_1,p}^+|\,=\,4(\ell -3)\,$.
Thus, for $\,\ell\geq 6\,$, $\,r_p(V)\,\geq\,
\frac{(\ell-1)\,(3\ell-4)}{2}\,+\,\frac{(\ell-1)\,(\ell-2)\,4(\ell-3) }{4!}
\, =\,\frac{(\ell-1)\,(\ell^2 + 4\ell - 6)}{6}\,>\,\ell\,(\ell+1)\,$.
For $\,p=2,\,$ $\,|R_{long}^+\,-\,R_{\mu,2}^+|\,=\,2(\ell -1)\,$ and 
$\,|R_{long}^+-R_{\mu_1,2}^+|\,=\,4(\ell -3)\,$. Thus, for $\,\ell\geq 8\,$,
$\,r_2(V)\, \geq\,\frac{(\ell-1)\,2(\ell -1)}{2}\,+\, 
\frac{(\ell-1)(\ell-2)\,4(\ell-3)}{4!}\,=\,\frac{(\ell-1)\,\ell\,(\ell +1)}
{6}\,>\,\ell(\ell+1)\,$.
In both cases, by~\eqref{rral}, $\,V\,$ is not exceptional.
This proves the claim. \hfill $\Box$ \vspace{1.5ex}

{\bf Hence for $\,\ell\geq 4$, if $\,V\,$ is an $\,A_{\ell}(K)$-module
such that $\,{\cal X}_{++}(V)\,$ contains a weight  
$\,\mu\,=\,a\omega_i\,$ (with $\,a\geq 3\,$ and \mbox{$\,1\leq
i\leq\ell\,$}), then $\,V\,$ is not an exceptional $\,\mathfrak g$-module. 
Therefore we can assume $\,a\leq 2\,$.}
%Recall that $\,\varepsilon\,=\,\dim\,\mathfrak z(\mathfrak g)\,$.
\vspace{1.5ex}

{\bf Claim 9}: \label{claim9al} 
{\it Let $\,V\,$ be an $\,A_{\ell}(K)$-module. If $\,{\cal X}_{++}(V)\,$
contains

(a) $\,2\,\omega_i\,$ with $\,3\leq i\leq\ell-2\,$ (hence $\,\ell\geq 5\,$)\;
or

(b) $\,2\,\omega_2\,$ (or graph-dually $\,2\,\omega_{\ell-1}\,$) 
for ($\,\ell\geq 6\,$ and $\,p\geq 3\,$) or ($\,\ell\geq 4\,$ and 
$\,p=2\,$), \quad
then $\,V\,$ is not an exceptional $\,\mathfrak g$-module.} 

Indeed:

(a) let $\,\mu\,=\,2\omega_i\in{\cal X}_{++}(V)$. Then
$\,\mu_1\,=\,\mu-\alpha_i\,=\,\omega_{i-1}\,+\,\omega_{i+1}\,
\in\,{\cal X}_{++}(V)$. For $\,3\leq\,i\,\leq\ell-2\,$ (implying 
$\,\ell\geq 5\,$), $\,\mu_1\,$ satisfies Claim~3 (p. \pageref{claim3}) 
or Claim~3.a (p. \pageref{claim3a}). Hence $\,V\,$ is not exceptional.

(b) Let $\,\mu\,=\,2\omega_2\,$ (or graph-dually $\,\mu\,=\,
2\omega_{\ell-1}\,$). Then $\,\mu_1\,=\,\omega_{1}\,+\,\omega_{3}\,
\in\,{\cal X}_{++}(V)$. For ($\,\ell\geq 6\,$ and $\,p\geq 3\,$) or
($\,\ell\geq 8\,$ and $\,p=2\,$), $\,\mu_1\,$ satisfies Claim~8  
(p. \pageref{claim8al}). Hence, $\,V\,$ is not exceptional.
For $\,4\leq\ell\leq 7\,$ and $\,p=2,\,$ $\,\mu\in {\cal X}_{++}(V)\,$ only if 
the highest weight of $\,V\,$ is
$\,\lambda\,=\,\omega_2\,+\,\omega_{3+i}\,+\,\omega_{\ell-i}\,$ (for some
$\,i\geq 0\,$). In this case, Lemma~\ref{nlf3c} applies. 
This proves the claim. \hfill $\Box$ \vspace{1.5ex}

P.6) For $\,\ell=5\,$ and $\,p\geq 3\,$, we may assume that $\,V\,$ has
highest weight $\,\mu\,=\,2\omega_2\,$. Then $\,\mu_1\,=\,\omega_{1}\,
+\,\omega_{3}\,$ and $\,\mu_2\,=\,\omega_4\,\in\,{\cal X}_{++}(V)$. By
Proposition~\ref{supruzal}, $\,m_{\mu_2}\geq 2\,$. As
$\,|R_{long}^+-R^+_{\mu,p}|\,=\,|R_{long}^+-R^+_{\mu_2,p}|\,=\,8\,$
and $\,|R_{long}^+-R^+_{\mu,p}|\,=\,11\,$, one has 
$\,r_p(V)\,\geq\,\displaystyle\frac{15\cdot 8}{30}\,+\,\frac{60\cdot
11}{30}\,+\,2\cdot \frac{15\cdot 8}{30}\,=\,34\,>\,30\,$. Hence,
by~\eqref{rral}, $\,V\,$ is not exceptional.

U.2) For $\,\ell=4\,$ and $\,p\geq 3\,$, the $\,A_4(K)$-modules of
highest weight $\,2\omega_2\,$ (or graph-dually $\,2\omega_3\,$)
are unclassified {\bf (N. 7 in Table~\ref{leftan}).}\vspace{1.5ex}

In the next particular cases, we may
assume that $\,V\,$ is an $\,A_{\ell}(K)$-module of 
highest weight $\,\mu\,=\,\omega_{1}\,+\,\omega_{3}\,$. Then
$\,\mu_1\,=\,\omega_{4}\,\in\,{\cal X}_{++}(V)$.

P.7) For $\,\ell=5\,$ and $\,p\geq 3\,$, $\,|R_{long}^+\,-\,R_{\mu,p}^+|\, 
=\,11\,$ $\,|R_{long}^+\,-\,R_{\mu_1,p}^+|\,=\,8\,$. 
By Lemma~\ref{mulmin}, $\,m_{\mu_1}=3$. Thus, $\,r_p(V)\,=\,\frac{60 
\cdot 11}{30}\,+\,3\cdot\frac{15\cdot 8}{30}\, =\,34\,>\,30\,$.

P.8) For $\,\ell=7\,$ and $\,p=2\,$, $\,|R_{long}^+\,-\,R_{\mu,2}^+|\,=\,12\,$ 
$\,|R_{long}^+\,-\,R_{\mu_1,2}^+|\,=\,16\,$. By Lemma~\ref{mulmin}, 
$\,m_{\mu_1}=2$. Thus, $\,r_2(V)\,=\,\frac{168\cdot 12}{56}\,+\,2\cdot
\frac{70\cdot 16}{56}\, =\,76\,>\,56\,$.

P.9) For $\,\ell=6\,$ and $\,p=2\,$, $\,|R_{long}^+\,-\,R_{\mu,2}^+|\,=\,10\,$ 
$\,|R_{long}^+\,-\,R_{\mu_1,2}^+|\,=\,12\,$. By Lemma~\ref{mulmin}, 
$\,m_{\mu_1}=2$. Thus, $\,r_2(V)\,=\,\frac{105\cdot 10}{42}\,+\,2\cdot
\frac{35\cdot 12}{42}\, =\,45\,>\,42\,$.

In all these cases, by~\eqref{rral}, $\,V\,$ is not exceptional.

U.3) For ($\,\ell=4\,$ and $\,p\geq 2\,$) or ($\,\ell=5\,$ and  
$\,p=2\,$), the $\,A_{\ell}(K)$-module of highest weight $\,\omega_1\,+ 
\,\omega_3\,$ (or graph-dually $\,\omega_{\ell-2}\,+\,\omega_{\ell}\,$) is 
unclassified {\bf (N. 6 or N. 3 in Table~\ref{leftan}).} \vspace{1.5ex}

{\bf Claim 10}: \label{claim10al} 
{\it Let $\,\ell\geq 6\,$ and $\,p\geq 2\,$. 
If $\,V\,$ is an $\,A_{\ell}(K)$-module such that
$\,{\cal X}_{++}(V)\,$ contains $\,\mu\,=\,\omega_1\,+\,\omega_{\ell-2}\,$
(or graph-dually $\,\mu\,=\,\omega_3\,+\,\omega_{\ell}\,$), 
then $\,V\,$ is not an exceptional $\,\mathfrak g$-module.}

Indeed, let $\,\mu\,=\,\omega_1\,+\,\omega_{\ell-2}\,$. Then
$\,\mu_1\,=\,\omega_{\ell-1}\,\in\,{\cal X}_{++}(V)$.
For $\,p\geq 2$, $\,|R_{long}^+-R^+_{\mu,p}|\,\geq\,4(\ell-3)\,$ and
$\,|R_{long}^+-R^+_{\mu_1,p}|\,=\,2\,(\ell-1)\,$. Thus, for $\,\ell\geq 6$,
\[
r_p(V)=\frac{(\ell-1)(\ell-2)\cdot 4(\ell -3)}{6}\,+\, 
\frac{2(\ell-1)}{2}=\frac{(\ell-1)[2\ell(\ell-5) + 15]}{3}\,>\, 
\ell(\ell+1)\,.
\]
Hence, by~(\ref{rral}), $\,V\,$ is not exceptional, proving the claim.
\hfill $\Box$ \vspace{1.5ex}

{\bf Claim 11}: \label{claim11al} 
{\it Let ($\,\ell\geq 9\,$ for $\,p\geq 3\,$) or ($\,\ell\geq 11\,$ for 
$\,p=2\,$). If $\,V\,$ is an $\,A_{\ell}(K)$-module such that
$\,{\cal X}_{++}(V)\,$ contains $\,\mu\,=\,\omega_1\,+\,\omega_{\ell-1}\,$
(or graph-dually $\,\mu\,=\,\omega_2\,+\,\omega_{\ell}\,$), 
then $\,V\,$ is not an exceptional $\,\mathfrak g$-module.}

Indeed, let $\,\mu\,=\,\omega_1\,+\,\omega_{\ell-1}\,$. Then
$\,\mu_1\,=\,\omega_{\ell}\,\in\,{\cal X}_{++}(V)$.
For $\,p\geq 3$, $\,|R_{long}^+-R^+_{\mu,p}|\,=\,(3\ell-4)\,$ and
$\,|R_{long}^+-R^+_{\mu_1,p}|\,=\,\ell\,$. Thus, for $\,\ell\geq 9\,$,
\[
r_p(V)\,=\,\frac{(\ell-1)\cdot(3\ell-4)}{2}\,+\,1\,=\,\frac{3\ell^2\,
-\,7\ell\,+\,6}{2}\,>\,\ell\,(\ell+1)\,.
\]
For $\,p=2$, $\,|R_{long}^+-R^+_{\mu,p}|\,=\,3(\ell-2)\,$ and
$\,|R_{long}^+-R^+_{\mu_1,p}|\,=\,\ell\,$. Thus, for $\,\ell\geq 11\,$,
\[
r_p(V)\,=\,\frac{(\ell-1)\cdot 3(\ell-2)}{2}\,+\,1\,=\,\frac{3\ell^2\,
-\,9\ell\,+\,8}{2}\,>\,\ell\,(\ell+1)\,.
\]
Hence, by~\eqref{rral}, $\,V\,$ is not an exceptional module. 
This proves the claim.\hfill $\Box$ \vspace{1.5ex}

The following are particular cases.

P.10) Let $\,\ell=10\,$ and $\,p=2\,$. We may assume that $\,V\,$ has 
highest weight
$\,\mu\,=\,\omega_1\,+\,\omega_9\,$. Then $\,\mu_1\,=\,\omega_{10}
\in\,{\cal X}_{++}(V)$. By Lemma~\ref{mulmin}, for $\,p=2\,$,
$\,m_{\mu_1}=8\,$. Thus, $\,r_2(V)\,\geq\,\displaystyle\frac{495\cdot 24}{110}
\,+\,8\cdot\frac{11\cdot 10}{110}\,=\,116\,>\,110\,$.

P.11) Let $\,\ell=9\,$ and $\,p=2\,$. We may assume that $\,V\,$ has 
highest weight
$\,\mu\,=\,\omega_1\,+\,\omega_8\,$. Then $\,\mu_1\,=\,\omega_{9}
\in\,{\cal X}_{++}(V)$. By Lemma~\ref{mulmin}, for $\,p=2\,$,
$\,m_{\mu_1}=8\,$. Thus, $\,r_2(V)\,\geq\,\displaystyle\frac{360\cdot 21}{90}
\,+\,8\cdot\frac{10\cdot 9}{90}\,=\,92\,>\,90\,$.

P.12) Let $\,\ell=8\,$ and $\,p\geq 3\,$. We may assume that $\,V\,$ 
has highest 
weight $\,\mu\,=\,\omega_1\,+\,\omega_7\,$. Then $\,\mu_1\,=\,\omega_{8}
\in\,{\cal X}_{++}(V)$. By Lemma~\ref{mulmin}, for $\,p\geq 3\,$,
$\,m_{\mu_1}=7\,$. Thus, $\,r_p(V)\,\geq\,\displaystyle\frac{252\cdot 20}{72}
\,+\,7\cdot\frac{9\cdot 8}{72}\,=\,77\,>\,72\,$.

P.13) Let $\,\ell=7\,$. We may assume that $\,V\,$ has highest 
weight $\,\mu\,=\,\omega_1\,+\,\omega_6\,$. Then $\,\mu_1\,=\,\omega_{7}
\in\,{\cal X}_{++}(V)$. By Lemma~\ref{mulmin}, for $\,p\neq 7\,$,
$\,m_{\mu_1}=6\,$. Thus, for $\,p\neq 2,\,7\,$, 
$\,r_p(V)\,\geq\,\displaystyle\frac{168\cdot 17}{56}
\,+\,6\cdot\frac{8\cdot 7}{56}\,=\,57\,>\,56\,$.

In all these cases, by~\eqref{rral}, $\,V\,$ is not exceptional. 
 \vspace{1.5ex}

U.4) For ($\,4\leq\ell\leq 8\,$ and $\,p= 2\,$) or 
($\,\ell=7\,$ and $\,p=7\,$) or ($\,4\leq\ell\leq 6\,$ and $\,p\geq 3\,$), 
the $\,A_{\ell}(K)$-modules of highest weight $\,\omega_1\,+\, 
\omega_{\ell-1}\,$ (or graph-dually $\,\omega_{2}\,+\,\omega_{\ell}\,$) are 
unclassified {\bf (N. 6 in Table~\ref{leftan}).}
\vspace{1.5ex}

{\bf Claim 12}:\label{claim12al} 
{\it Let $\,\ell\geq 3\,$ and $\,p=3\,$. Let $\,V\,$ be an 
$\,A_{\ell}(K)$-module of highest weight $\,\mu\,=\, 
\omega_{1}\,+\,\omega_{2}\,$. Then $\,V\,$ is not an exceptional
$\,\mathfrak g$-module.}

Indeed, as $\,V\,$ has highest weight $\,\mu\,=\,(3-1-1)\omega_1\,+\,
\omega_2\,$, $\,V\,$ is isomorphic to the truncated symmetric power
$\,T_3\,$ (the $\,3^{rd}$-graded
component of the truncated symmetric algebra). See Appendix~\ref{trunc}.
%%can be realized in the `truncated' symmetric power 
%%$\,T_{\ell+1}\,$ (see definition in Appendix~\ref{trunc}).
Consider the action of $\,\mathfrak{gl}({\ell+1})\,$ on $\,T_{3}\,$.
Let $\,v\,=\,x_1^2\,x_2\,+\,x_2^2\,x_3\,+\,\cdots\,+\,x_{\ell}^2\,x_{\ell+1}
\,+\,x_{\ell+1}^2\,x_1\,\in\,T_{3}\,$ and
$\,X\,=\,\displaystyle\sum_{1\leq i,\,j\leq \ell+1}\,\lambda_{ij}\,
x_i\partial_j\,$ with $\,\displaystyle\sum_i\lambda_{ii}\,=\,0\,$ be such that 
$\,X\cdot v\,=\,0\,$. By direct calculations, one shows that
$\,X=0\,$. Hence, $\,\mathfrak g_v\,\subseteq\,\mathfrak z(\mathfrak g)\,$ and
this module is not exceptional.
This proves the claim. \hfill $\Box$ \vspace{1.5ex}

U.5) For $\,p\neq 3\,$ and $\,\ell\geq 4\,$, the 
$\,A_{\ell}(K)$-module of highest weight $\,\mu\,=\, 
\omega_{1}\,+\,\omega_{2}\,$ is unclassified
{\bf (N. 5 in Table~\ref{leftan}).} \vspace{1.5ex}

P.14) \label{adjal14}
Let $\,V\,$ be an $\,A_{\ell}(K)$-module of highest weight 
$\,\mu\,=\,\omega_{1}\,+\,\omega_{\ell}\,$. Then, for
$\,\ell\geq 1\,$, $\,V\,$ is the adjoint module, which is exceptional
by Example~\ref{adjoint} \mbox{\bf (N. 2 in Table~\ref{tablealall})}.

{\bf Therefore, for $\,\ell\geq 4$, 
if $\,V\,$ is an $\,A_{\ell}(K)$-module such that
$\,{\cal X}_{++}(V)\,$ contains a weight with 2 nonzero coefficients,
then $\,V\,$ is not an exceptional module, unless $\,V\,$ has
highest weight $\,\omega_1\,+\,\omega_{\ell}\,$ (the adjoint module
which is exceptional) or the highest weight of $\,V\,$ is one of }\linebreak
{\bf N. 3,\,4,\,5 or 6 in Table~\ref{leftan}.}

{\bf From now on we may assume that $\,V\,$ is an $\,A_{\ell}(K)$-module such
that $\,{\cal X}_{++}(V)\,$ contains only weights with at most
one nonzero coefficient.}

P.15)\label{IIIbp15} Let $\,\mu\,=\,2\,\omega_1\,$ (or graph-dually  
$\,\mu\,=\,2\,\omega_{\ell}\,$). 
%$\;\mu_1\,=\,\mu-\alpha_1\,=\,\omega_{2}\,\in\,{\cal X}_{++}(V).\,$

For $\,p=2,\,$ $\,\mu\in {\cal X}_{++}(V)\,$ only if $\,V\,$ has
highest weight $\,\lambda\,=\,\omega_1\,+\,\omega_{2+j}\,+\,\omega_{\ell-j}\,$
(for some $\,j\geq 0\,$ and $\,\ell\geq 3\,$) or
$\,\lambda\,=\,\omega_2\,+\,\omega_{2+j}\,+\,\omega_{\ell-j}\,+\, 
\omega_{\ell}\,$ (for some $\,j\geq 1\,$ and $\,\ell\geq 5\,$), for
instance. In these cases, Lemma~\ref{nlf3c} or Lemma~\ref{4coef} applies.

For $\,p\geq 3\,$, we may assume that $\,V\,$ is an $\,A_{\ell}(K)$-module of
highest weight $\,\mu\,=\,2\omega_1\,$.  By~\cite[Table 1]{ave},
for $\,\ell\geq 2\,$, $\,\dim\,V\,=\,\displaystyle\frac{(\ell
+1)(\ell +2)}{2}\,<\,\ell^2 + 2\ell\,-\,\varepsilon\,$. 
Hence, by Proposition~\ref{dimcrit}, $\,V\,$ is an
exceptional module {\bf (N. 3 in Table~\ref{tablealall})}. \vspace{1.5ex}

{\bf Hence if $\,{\cal X}_{++}(V)\,$ contains a weight  
$\,\mu\,=\,a\omega_i\,$ (with $\,a\geq 1\,$ and \mbox{$\,1\leq
i\leq\ell\,$}), then we can assume $\,a=1\,$.} 

{\bf Claim 13}: \label{claim13al}
{\it Let $\,V\,$ be an $\,A_{\ell}(K)$-module. If $\,{\cal X}_{++}(V)\,$  
contains

(a) $\,\mu\,=\,\omega_i\,$, with $\;5\leq i\leq \ell-4\;$ for $\,\ell\geq 10\;$
or 

(b) $\,\mu\,=\,\omega_4\,$ (or graph-dually
$\,\mu\,=\,\omega_{\ell-3}\,$), for $\;\ell\geq 12$, \\
then $\,V\,$ is not an exceptional $\,\mathfrak g$-module.} Indeed:
  
(a) let $\,5\leq i\leq \ell-4\,$ (hence $\,\ell\geq 9\,$) and 
$\,\mu\,=\,\omega_i\in{\cal X}_{++}(V)$. Then, for $\,\ell\geq 11\,$,
\[
r_p(V)\,=\,\frac{i\,(\ell-i+1)}{\ell(\ell+1)}
\binom{\ell+1}{i}\,\geq\,5\cdot 5\cdot\frac{(\ell+1)\ell(\ell-1)(\ell-2)
(\ell-3)}{5!\cdot (\ell+1)\ell}\,\geq\,\ell (\ell +1)\,.
\]
For $\,\ell=10\,$ and $\,\mu\,=\,\omega_5\,$ (or graph-dually
$\,\mu\,=\,\omega_{6}\,$), $\,\mbox{$|R_{long}^+-R_{\mu,p}^+|$}\, 
=\,5(\ell-4)\,$. Thus $\,r_p(V)\,\geq\,\displaystyle\frac{9\cdot
8\cdot 7\cdot 30}{5!}\,=\,126\,>\,110\,$. Hence for $\,\ell\geq 10$, 
by~\eqref{rral}, $\,V\,$ is not exceptional. 

(b) Let $\,\mu\,=\,\omega_4\,$ (or graph-dually
$\,\mu\,=\,\omega_{\ell-3}\,$) $\,\in\,{\cal X}_{++}(V)$. 
$\,|R_{long}^+-R^+_{\mu,p}|\,=\,4(\ell-3)\,$. Thus for $\,\ell\geq 12\,$, 
\[
r_p(V)\,=\,
\frac{(\ell+1)\ell(\ell-1)(\ell-2)}{(\ell+1)\ell\, 4!}\,4(\ell-3)\,=\, 
\frac{(\ell-1)\,(\ell-2)\,(\ell-3)}{6}\,>\,\ell\,(\ell+1)\,.
\]
Hence, by~\eqref{rral}, $\,V\,$ is not exceptional, proving the claim.
\hfill $\Box$ \vspace{1.5ex}

U.6) For $\,\ell=9\,$, the $\,A_9(K)$-module of highest weight 
$\,\omega_5\,$ is unclassified \mbox{\bf (N. 8 in Table~\ref{leftan}).}

P.16) Let $\,\mu\,=\,\omega_4\,$ (or graph-dually
$\,\mu\,=\,\omega_{\ell-3}\,$) $\,\in\,{\cal X}_{++}(V)$. 
$\,|R_{long}^+-R^+_{\mu,p}|\,=\,4(\ell-3)\,$ Thus for $\,\ell\geq 12\,$, 
by~\eqref{rral}, $\,V\,$ is not exceptional since 
\[
r_p(V)\,=\,
\frac{(\ell+1)\ell(\ell-1)(\ell-2)}{(\ell+1)\ell\, 4!}\,4(\ell-3)\,=\, 
\frac{(\ell-1)\,(\ell-2)\,(\ell-3)}{6}\,>\,\ell\,(\ell+1)\,.
\] 
For $\,4\leq \ell\leq 6\,$, we may assume that $\,V\,$ is an
$\,A_{\ell}(K)$-module of highest weight  
$\,\mu=\omega_4\,$. Then $\,\dim\,V\,=\,\displaystyle 
\binom{\ell+1}{4}\,<\,\ell\,(\ell+1)\,+\,\ell\,-\,\varepsilon\,$.
Hence, by Proposition~\ref{dimcrit},  
$\,V\,$ is exceptional {\bf (N. 4, 5 and 6 in Table~\ref{tablealall})}.

U.7) For $\,7\leq \ell\leq 11\,$ the $\,A_{\ell}(K)$-modules of
highest weight $\,\mu=\omega_4\,$ are unclassified {\bf (N. 9 in 
Table~\ref{leftan}).}

P.17) Let $\,V\,$ be an $\,A_{\ell}(K)$-module of highest weight  
$\,\mu=\omega_3\,$ (or graph-dually $\,\mu=\omega_{\ell-2}\,$). Then,
for $\,3\leq \ell\leq 7,\,$ $\,\dim\,V\,=\,\displaystyle 
\binom{\ell+1}{3}\,<\,\ell\,(\ell+1)\,+\,\ell\,-\,\varepsilon\,$.
Hence, by Proposition~\ref{dimcrit}, $\,V\,$ is an exceptional module
(N. 4 and 6 in Table~\ref{tablealall}).

U.8) For $\,\ell\geq 8\,$ the $\,A_{\ell}(K)$-modules of
highest weight $\,\mu=\omega_3\,$ are unclassified {\bf (N. 10 in 
Table~\ref{leftan}).}

P.18) If $\,V\,$ is an $\,A_{\ell}(K)$-module of highest weight 
$\,\mu=\omega_2\,$ then, for $\,\ell\geq 2\,$, $\,\dim\,V\,=\,\displaystyle 
\binom{\ell+1}{2}\,<\,\ell\,(\ell+1)\,+\,\ell\,-\,\varepsilon\,$. 
Hence, by Proposition~\ref{dimcrit}, $\,V\,$
is an exceptional module {\bf (N. 5 in Table~\ref{tablealall})}.

P.19) \label{IIIp19} Let $\,V\,$ be an $\,A_{\ell}(K)$-module of highest 
weight $\,\mu=\omega_1\,$. Then, for $\,\ell\geq 2\,$ or ($\,\ell=1\,$
and $\,p\neq 2\,$), $\,\dim\,V\,=\,\ell+1\,<\,\ell^2 + 2\ell\,-\, 
\varepsilon\,$. Thus, by Proposition~\ref{dimcrit}, $\,V\,$ is
exceptional module. 

For $\,\ell=1\,$ and $\,p=2,\,$ 
the Lie algebra of type $\,A_1\,$ is Heisenberg (that is, its derived 
subalgebra is the centre, which coincides with the scalar matrices).
We claim that the standard module $\,V\,$ for $\,\mathfrak{sl}(2)\,$
is exceptional. (Indeed, for any $\,v\in V$, it is possible to 
construct a nonzero nilpotent matrix which annihilates $\,v\,$. 
Hence, $\,{\mathfrak g}_v\not\subset\mathfrak z(\mathfrak g),\,$
as no nonzero nilpotent matrix is in the centre. This proves the claim.)

{\it Therefore, for $\,\ell\geq 1$, if $\,V\,$ is an $\,A_{\ell}(K)$-module of
highest weight $\,\mu=\omega_1\,$, then $\,V\,$ is exceptional}
{\bf (N. 4 and 1 in Table~\ref{tablealall})}.

{\bf Therefore, for $\,\ell\geq 4,\,$ if the highest weight of $\,V\,$
is listed in Table~\ref{tablealall}, then $\,V\,$ is an exceptional 
$\,A_{\ell}(K)$-module.
If $\,V\,$ has highest weight listed in Table~\ref{leftan}, then
$\,V\,$ is unclassified. If the highest weight of $\,V\,$ is not in
Table~\ref{tablealall} or Table~\ref{leftan}, then $\,V\,$ is not an
exceptional $\,\mathfrak g$-module.}

This proves Theorem~\ref{anlist} for $\,\ell\geq 4$.
\hfill$\Box$

%\newpage
%\input{a123new.tex}

\subsubsection{Type $\,A_{\ell}\,$ - Small Rank Cases}\label{tasr}

This is the Second Part of the Proof of Theorem~\ref{anlist}: 
for groups of small rank. 
\subsubsection*{\ref{tasr}.1. Type $\,A_1$}
\addcontentsline{toc}{subsubsection}{\protect\numberline{\ref{tasr}.1}
Type $\,A_1\,$}

For groups of type $\,A_1\,$, $\,|W|=2.\,$ The limit for the inequality
(\ref{allim}) is $\,1(1+1)^2=4$. The dominant weights are of the form
$\,\mu\,=\,a\omega_1\,$ with $\,a\in\mathbb Z^+\,$. Hence, 
$\,|W\mu|=2\,$ for all nonzero weights.\vspace{2ex} \\
{\bf Proof of Theorem~\ref{anlist}\; for $\,\ell=1$.}\\
Let $\,V\,$ be an $\,A_1(K)$-module and suppose
$\,\mu\,=\,a\omega_1\,\in {\cal X}_{++}(V)$.

For a prime $\,p,\,$ if $\,a\geq p$, 
then $\,\mu=\,a\,\omega_1\,$ does not occur in $\,{\cal X}_{++}(V),\,$ 
otherwise the highest weight of $\,V\,$ would have coefficient
$\,\geq p,\,$ contradicting Theorem~\ref{curt}(1).

If $\,p>5\,$ and $\, 5\leq a < p\,$ then, for $\,a\,$ odd (resp. $\,a\,$ even),
$\,a\omega_1,\,3\omega_1,\,\omega_1\,$ (resp. $\,a\,\omega_1,\, 
4\omega_1,\,2\omega_1\,$) $\,\in\,{\cal X}_{++}(V)$. Hence,
$\,s(V)\,>\,4\,$ and, by~(\ref{allim}), $\,V\,$ is not an exceptional module.
For $\,p\geq 5\,$ and $\,3\leq a < p,\,$ the modules with highest
weight $\,a\,\omega_1\,$ are not exceptional, by Lemma~\ref{III.iii}
(p. \pageref{III.iii}).

For $\,p\geq3,\,$ if $\,V\,$ has highest weight
$\,\mu\,=\,2\omega_1\,$, then $\,V\,$ is the adjoint module, which is 
exceptional by Example~\ref{adjoint}.
For any $\,p,\,$ if $\,V\,$ has highest weight $\,\mu\,=\,\omega_1,\,$
then $\,V\,$ is exceptional, by case P.19 (p. \pageref{IIIp19}) of First
Part of Proof of Theorem~\ref{anlist}. 

This proves Theorem~\ref{anlist} for groups of type $\,A_1\,$. 
\hfill$\Box$

\subsubsection*{\ref{tasr}.2. Type $\,A_2$}
\addcontentsline{toc}{subsubsection}{\protect\numberline{\ref{tasr}.2}
Type $\,A_2\,$}

For groups of type $\,A_2,\,$ $\,|W|=6.\,$ The limit for 
(\ref{allim}) is $\,18\,$ and for \eqref{rral} is $\,6$. The dominant 
weights are of the form
$\,\mu\,=\,a\,\omega_1\,+\,b\,\omega_2,\,$ with $\,a,\,b\in\mathbb Z^+$.
If $\,a\,$ and $\,b\,$ are both nonzero, then $\,|W\mu|\,=\,6$. If
either $\,a\,$ or $\,b\,$ (but not both) is nonzero, then $\,|W\mu|\,=\,3$.
\vspace{2ex}\\
{\bf Proof of Theorem~\ref{anlist}\; for $\,\ell=2$.} 
Let $\,V\,$ be an $\,A_2(K)$-module.

Let $\,\mu\,=\,a\omega_1\,+\,b\omega_2\,$ (with $\,a\geq 1,\,b\geq 1\,$)
be a weight in $\,{\cal X}_{++}(V)$.

{\bf Claim 1}: \label{bada2} {\it Bad weights with $\,2\,$ nonzero 
coefficients do not occur in $\,{\cal X}_{++}(V)$.}
Indeed, if $\,\mu\,=\,a\omega_1\,+\,b\omega_2\,$ is a bad weight in   
$\,{\cal X}_{++}(V)$, then the highest weight of $\,V\,$ must have at least
one coefficient $\,\geq p$, contradicting Theorem~\ref{curt}(i).
\hfill $\Box$ \vspace{1ex}

{\bf From now on we can assume that weights with $\,2\,$ nonzero coefficients  
occurring in $\,{\cal X}_{++}(V)\,$ are good weights.} 

{\bf Claim 2}: \label{cl2a2} 
{\it Let $\,p\geq 5\,$ and ($\,a\geq 1,\,b\geq 3\,$) or 
($\,a\geq 3,\,b\geq 1\,$) or ($\,a=b=2\,$). If $\,V\,$ is an 
$\,A_2(K)$-module such that $\,{\cal X}_{++}(V)\,$ contains a good weight
$\,\mu\,=\,a\omega_1\,+\,b\omega_2\,$, then $\,V\,$ is not an exceptional  
$\,\mathfrak g$-module.}

Indeed, let $\,\mu\,=\,a\omega_1\,+\,b\omega_2\,$ be a good weight in
$\,{\cal X}_{++}(V)\,$.

(a) Let $\,a\geq 4,\,b\geq 2\,$ (or graph-dually $\,a\geq 2,\,b\geq 4\,$).  
For $\,p=2\,$ or $\,3\,$ such $\,\mu\,$ does not occur in $\,{\cal X}_{++}(V)$.
For $\,p\geq 5,\,$ 
$\,\mu_1\,=\,\mu-(\alpha_1+\alpha_2)\,=\,(a-1)\omega_1\,+\,(b-1)\omega_2,\;$
$\,\mu_2\,=\,\mu-\alpha_1\,=\,(a-2)\omega_1\,+\,(b+1)\omega_2\;$ and
$\,\mu_3\,=\,\mu_1-\alpha_1\,=\,(a-3)\omega_1\,+\,b\omega_2\,\in\,
{\cal X}_{++}(V)$.
By Claim 1, these are all good weights. Thus, $\,s(V)\,\geq\,24\,>\,18\,$.
Hence, by~(\ref{allim}), $\,V\,$ is not an exceptional module.

(b) Let $\,a=3,\,b\geq 2\,$ (or graph-dually $\,a\geq 2,\,b=3\,$) and  
$\,\mu\,=\,3\omega_1\,+\,b\omega_2\,$. For $\,p=2\,$ or $\,3\,$ such
$\,\mu\,$ does not occur in $\,{\cal X}_{++}(V)$.
For $\,p\geq 5,\,$ consider
$\,\mu_1\,=\,\mu-(\alpha_1+\alpha_2)\,=\,2\omega_1\,+\,(b-1)\omega_2,\;$
$\,\mu_2\,=\,\mu-\alpha_1\,=\,\omega_1\,+(b+1)\omega_2,\;$
$\,\mu_3\,=\,\mu_1-(\alpha_1+\alpha_2)\,=\,\omega_1\,+\,(b-2)\omega_2\,\in\,
{\cal X}_{++}(V)$. For $\,p\geq 5,\,$  
$\,\mu,\,\mu_1,\,\mu_2,\,\mu_3\,$ are all good weights. Thus
$\,s(V)\,\geq\,21\,>\,18$. Hence, by~(\ref{allim}), $\,V\,$ is not exceptional.

(c) Let $\,a=b=2\,$ and $\,\mu\,=\,2\omega_1\,+\,2\omega_2\,$. 
Then $\,{\cal X}_{++}(V)\,$ also contains
$\,\mu_1\,=\,\mu-(\alpha_1+\alpha_2)\,=\,\omega_1\,+\,\omega_2,\;$
$\,\mu_2\,=\,\mu-\alpha_1\,=\,3\omega_2\;$ and
$\,\mu_3\,=\,\mu-\alpha_2\,=\,3\omega_1.\,$ \linebreak 
For $\,p\geq 3$, $\,|R_{long}^+-R^+_{\mu,p}|\,=\,|R_{long}^+- 
R^+_{\mu_1,p}|\,=\,3\;$ and, for $\,p\geq 5,\,$
$\,|R_{long}^+-R^+_{\mu_2,p}|\,=\,|R_{long}^+-R^+_{\mu_3,p}|\,=\,2\,$.
Thus, for $\,p\geq 5,\,$ $\,r_p(V)\,\geq\,8\,>\,6$. Hence,
by~\ref{rral}, $\,V\,$ is not exceptional.

{\bf Therefore if $\,{\cal X}_{++}(V)\,$ contains a good weight of the form
$\,\mu\,=\,a\omega_1\,+\,b\omega_2\,$, with $\,a\geq 1,\,b\geq 1\,$, then
we can assume $a\leq 1$ or $\,b\leq 1$.}

(d) Let $\,a=1,\,b\geq 5\,$ (or graph-dually $\,a\geq 5,\,b=1\,$) and  
$\,\mu\,=\,\omega_1\,+\,b\omega_2\,$. For $\,p=2\,$ or $\,3$, such $\,\mu\,$
does not occur in $\,{\cal X}_{++}(V)$. For $\,p\geq 5$,  
$\,{\cal X}_{++}(V)\,$ also contains $\,\mu_1\,=\,\mu-\alpha_2\,=\, 
2\omega_1\,+\,(b-2)\omega_2\,$. As $\,b\geq 5,\,$ 
case (b) above applies to $\,\mu_1\,$. Hence $\,V\,$ is not exceptional.

(e) For $\,a=1,\,b=4\,$ (or graph-dually $\,a=4,\,b=1\,$),  
$\,\mu\,=\,\omega_1\,+\,4\omega_2\,$. For $\,p=2\,$ or $\,3$, such $\,\mu\,$
does not occur in $\,{\cal X}_{++}(V)$. For $\,p\geq 5$, 
$\,{\cal X}_{++}(V)\,$ also contains
$\,\mu_1\,=\,\mu-\alpha_2\,=\,2\omega_1\,+\,2\omega_2\;$ which
satisfies case (c). Hence, $\,V\,$ is not an exceptional module.

(f) For $\,a=3,\,b=1\,$ (or graph-dually $\,a=1,\,b=3\,$),  
$\,\mu\,=\,3\omega_1\,+\,\omega_2\,$. For $\,p=2,\,3\,$ such
$\,\mu\,$ does not occur in $\,{\cal X}_{++}(V)$.
For $\,p\geq 5,\,$ $\,{\cal X}_{++}(V)\,$ also contains the good weights
$\,\mu_1\,=\,\mu-(\alpha_1+\alpha_2)\,=\,2\omega_1,\;$
$\,\mu_2\,=\,\mu-\alpha_1\,=\,\omega_1\,+\,2\omega_2\;$ and
$\,\mu_3\,=\,\mu_1-\alpha_1\,=\,\omega_2\,$.
For $\,p\geq 5$, $\,|R_{long}^+-R^+_{\mu,p}|\,=\,|R_{long}^+-R^+_{\mu_2,p}|\,= 
\,3\,$, $\,|R_{long}^+-R^+_{\mu_1,p}|\,=\,|R_{long}^+-R^+_{\mu_3,p}|\,=\,2$.
Thus $\,r_p(V)\,\geq\,\displaystyle\frac{6\cdot 3}{6}\,+\,  
\frac{3\cdot 2}{6}\,+\,\frac{6\cdot 3}{6}\,+\,\frac{3\cdot 2}{6}\,=\,8\,>\,6$.
Hence, by~\eqref{rral}, $\,V\,$ is not an exceptional module.

This proves Claim~2. \hfill $\Box$\vspace{1.5ex}

{\bf Claim 3}: \label{cl3a2}
{\it Let $\,p\geq 5\,$ and $\,a\geq 5\,$. If $\,V\,$ is an 
$\,A_2(K)$-module such that $\,{\cal X}_{++}(V)\,$ contains a weight
$\,\mu\,=\,a\omega_1\,$ (or graph-dually $\mu\,=\,a\omega_2\,$), 
then $\,V\,$ is not an exceptional $\,\mathfrak g$-module.}

Indeed, let $\,a\geq 5\,$ and $\,\mu\,=\,a\omega_1\,$. 
For $\,p\leq\,3,\,$ $\,\mu\,$ does not occur in $\,{\cal X}_{++}(V)$.
For $\,p\geq 5$, $\,\mu_1\,=\,\mu-\alpha_1\,=\,(a-2)\omega_1\,+\,\omega_2
\,\in\,{\cal X}_{++}(V)\,$ satisfies Claim~2. Hence, $\,V\,$ is not
exceptional, proving the claim.  \hfill $\Box$\vspace{1.5ex}

Now we deal with some particular cases.

P.1) Let $\,p=3\,$. We may assume that $\,V\,$ has highest weight  
$\,\mu\,=\,2\omega_1\,+\,2\omega_2\,$. Then $\,V\,$ is the Steinberg
module and $\,\dim\,V\,=\,3^3\,$ (see Theorem~\ref{stmod}). 
One can show that $\,m_{\mu_1}=2,\,$ where $\,\mu_1\,=\,
\mu-(\alpha_1+\alpha_2)\,=\,\omega_1\,
+\,\omega_2\,\in{\cal X}_{++}(V)$. Thus $\,r_p(V)\,=\,9\,>\,6\,$. Hence,
by~\eqref{rral}, $\,V\,$ is not exceptional.

U.1) Let $\,a=2,\,b=1\,$ (or graph-dually $\,a=1,\,b=2\,$) and 
$\,\mu\,=\,2\omega_1\,+\,\omega_2\,$. For $\,p=2,\,$ $\,\mu\,$  
does not occur in $\,{\cal X}_{++}(V)$. 
For $\,p\geq 3,\,$ the $\,A_2(K)$-modules of highest weight
$\,2\omega_1\,+\,\omega_2\,$ (or graph-dually $\,\omega_1\,+\,2\omega_2\,$)
are unclassified {\bf (N. 1 in Table~\ref{leftan}).}
%% $\,{\cal X}_{++}(V)\,$ also contains
%%$\,\mu_1\,=\,\mu-\alpha_1\,=\,2\omega_2\;$ and
%%$\;\mu_2\,=\,\mu-(\alpha_1+\alpha_2)\,=\,\omega_1.\,$

P.2) Let $\,V\,$ be an $\,A_2(K)$-module such that  
$\,\mu\,=\,4\omega_1\,$ (or graph-dually $\,\mu\,=\,4\omega_2\,$)
$\,\in\,{\cal X}_{++}(V)\,$. Then $\,V\,$ is not exceptional.
(Indeed, for $\,p\leq 3\,$ $\,\mu\,$ does not occur in $\,{\cal X}_{++}(V)$. 
For $\,p\geq 5,\,$ we can assume that $\,V\,$ has highest weight 
$\,\mu\,=\,4\omega_1\,$. Hence, by Lemma~\ref{III.iii}, $\,V\,$ is not  
exceptional.)

{\bf Therefore if $\,{\cal X}_{++}(V)\,$ contains a good weight 
$\,\mu\,=\,a\omega_1\,+\,b\omega_2\,$, with $\,2\,$ nonzero  
coefficients, then we can assume $\,a=b=1\,$.}

P.3) Let $\,V\,$ be an $\,A_2(K)$-module such that  
$\,\mu\,=\,3\omega_1\,$ (or graph-dually $\,\mu\,=\,3\omega_2\,$)
$\,\in\,{\cal X}_{++}(V)\,$. Then $\,V\,$ is not exceptional.
(Indeed, for $\,p=2\,$ $\,\mu\,$ does not occur in $\,{\cal X}_{++}(V)$. 
For $\,p= 3,\,$ $\,\mu\in {\cal X}_{++}(V)\,$ only if $\,V\,$ has 
highest weight $\,\lambda\,=\, 2\omega_1\,+\,2\omega_2\,$. Then 
$\,V\,$ is the Steinberg module, which is not exceptional by case P.1) above. 
For $\,p\geq 5,\,$ we can assume that $\,V\,$ has highest weight 
$\,\mu\,=\,3\omega_1\,$. Hence, by Lemma~\ref{III.iii}, $\,V\,$ is not  
exceptional.)

P.4) If $\,V\,$ is an $\,A_2(K)$-module of highest weight 
$\,\mu\,=\,\omega_1\,+\,\omega_2\,$ then $\,V\,$ is the adjoint module,
which is exceptional by Example~\ref{adjoint}
{\bf (N. 2 in Table~\ref{tablealall})}.

{\bf From now on we can assume that $\,{\cal X}_{++}(V)\,$
contains only weights of the form $\,\mu\,=\,a\omega_i\,$ (with 
$\,a\leq 2\,$).} Observe that for $\,p=2$, $\,2\omega_i\,$ does not 
occur in $\,{\cal X}_{++}(V)\,$. 

P.5) For $\,p\geq 3,\,$ if $\,V\,$ is an $\,A_2(K)$-module of highest weight
$\,\mu\,=\,2\omega_1\,$ (or graph-dually $\,\mu\,=\,2\omega_2\,$), 
then $\,V\,$ is an exceptional module by case P.15 
(p. \pageref{IIIbp15}) of First Part of Proof of Theorem~\ref{anlist}
{\bf (N. 3 in Table~\ref{tablealall})}.

P.6) If $\,V\,$ is an $\,A_2(K)$-module of highest weight  
$\,\mu\,=\,\omega_1\,$ (or graph-dually $\,\mu\,=\,\omega_2\,$), 
then $\,V\,$ is an exceptional module by case P.19 (p. \pageref{IIIp19})
of First Part of Proof of Theorem~\ref{anlist}
{\bf (N. 4 in Table~\ref{tablealall})}.

{\bf Therefore, if $\,V\,$ is an $\,A_2(K)$-module such that
$\,{\cal X}_{++}(V)\,$ contains a weight 
$\,\mu\,=\,a\omega_1\,$ (or graph-dually $\,\mu\,=\,a\omega_2\,$)
with $\,a\geq 3,\,$ then $\,V\,$ is not exceptional.
If $\,V\,$ has highest weight $\,\lambda\,\in\,\{
2\omega_1,\,2\omega_2,\,\omega_1,\,\omega_2 \}\,$ then 
$\,V\,$ is an exceptional module (N. 3,\,4 in Table~\ref{tablealall}).}

This proves Theorem~\ref{anlist} for groups of type $\,A_2$.
\hfill$\Box$

\subsubsection*{\ref{tasr}.3. Type $\,A_3$}
\addcontentsline{toc}{subsubsection}{\protect\numberline{\ref{tasr}.3}
Type $\,A_3\,$}

For groups of type $\,A_3\,$, $\,|W|=24.\,$ The limit for 
(\ref{allim}) is $\,48\,$ and for \eqref{rral} is $\,12$. The orbit sizes are 
$\,|W\mu|\,=\,24\,$ for $\,\mu\,$ having 3 nonzero coefficients, 
$\,|W\mu|\,=\,12\,$ for $\,\mu\,$ having 2 nonzero coefficients, and
$\,|W\mu|\,=\,6\,$ or  $\,4\,$ for $\,\mu\,$ having 1 nonzero coefficient.
First we prove a reduction lemma.
\begin{lemma}\label{a333} Let $\,V\,$ be an $\,A_3(K)$-module.
If $\,{\cal X}_{++}(V)\,$ contains $\,2\,$ good weights with 
$\,3\,$ nonzero coefficients and any other nonzero good 
weight, then $\,V\,$ is not an exceptional $\,\mathfrak g$-module.
\end{lemma}\noindent
\begin{proof}
If $\,\mu\in{\cal X}_{++}(V)\,$ has $\,3\,$ nonzero coefficients, then  
$\,|W\mu|=24$. Hence, by assumption,
$\,s(V)\,     %\displaystyle\sum_{\mu_i\;good}\,|W\mu_i|\,
\geq\,48+ k$, where $\,k>1$. Thus,
by~(\ref{allim}), $\,V\,$ is not an exceptional module and the lemma is proved.
\end{proof}\vspace{1.5ex}\noindent
{\bf Proof of Theorem~\ref{anlist}\; for $\,\ell=3$.}
Let $\,V\,$ be an $\,A_3(K)$-module.

{\bf Claim 1}: \label{cl1a3} {\it Bad weights with $\,3\,$ nonzero  
coefficients do not occur in $\,{\cal X}_{++}(V)$.} 
Indeed, otherwise the highest weight of $\,V\,$ would have at least
one coefficient $\,\geq p,\,$ contradicting Theorem~\ref{curt}(i).
\hfill $\Box$ 

{\bf Therefore we can assume that weights with $\,3\,$ nonzero coefficients 
occurring in $\,{\cal X}_{++}(V)\,$ are good weights.}

{\bf Claim 2}: \label{cl2a3} {\it Let $\,V\,$ be an $\,A_3(K)$-module.
If $\,{\cal X}_{++}(V)\,$ contains  
a bad weight with $\,2\,$ nonzero coefficients, then $\,V\,$ is not 
an exceptional $\,\mathfrak g$-module.} 

We prove this claim in two steps.

(a) Let $\,\mu\,=\,a\omega_1\,+\,b\omega_2\,$ 
(or graph-dually $\,\mu\,=\,b\omega_2\,+\,a\omega_3\,$)
be a bad weight in $\,{\cal X}_{++}(V)$. For $\,p=2\,$ such $\,\mu\,$
does not occur in $\,{\cal X}_{++}(V)$. For $\,p\geq 3,\,$  
$\,{\cal X}_{++}(V)\,$ contains the good weights
$\,\mu_1\,=\mu-(\alpha_1+\alpha_2)\,=\,(a-1)\omega_1\,+\,(b-1)\omega_2\,+
\,\omega_3,\;$
$\,\mu_2\,=\mu_1-(\alpha_1+\alpha_2)\,=\,(a-2)\omega_1\,+\,(b-2)\omega_2
\,+\,2\omega_3,\;$
$\,\mu_3\,=\mu-\alpha_2\,=\,(a+1)\omega_1\,+\,(b-2)\omega_2\,+
\,\omega_3.\,$ Thus, as $\,a\geq 3,\,b\geq 3\,$,
$\,{\cal X}_{++}(V)\,$ satisfies Lemma~\ref{a333}. 
%%If $\,p=2,\,$ $\,a,\,b\geq 2$, then
%%$\,\mu\,$ does not occur in $\,{\cal X}_{++}(V),\,$ 
%%otherwise $\,V\,$ would have highest weight 
%%$\,\lambda\,=\,a\,\omega_1\,+\,b\,\omega_2\,+\,c\omega_3\,$ with at least 
%%one coefficient $\,\geq 2\,$, contradicting Theorem~\ref{curt}(1).

(b) Let $\,\mu\,=\,a\omega_1\,+\,c\omega_3\,$ be a bad weight in 
$\,{\cal X}_{++}(V)$. For $\,p=2\,$ such $\,\mu\,$ does not occur  
in $\,{\cal X}_{++}(V)$. For $\,p\geq 3,\,$ 
$\,{\cal X}_{++}(V)\,$ contains the good weights
$\,\mu_1\,=\mu-(\alpha_1+\alpha_2+\alpha_3)\,=\,(a-1)\omega_1\,+\,(c-1)
\omega_3\,$, 
$\,\mu_2\,=\mu-\alpha_1\,=\,(a-2)\omega_1\,+\,\omega_2\,+\,c\omega_3,\;$
$\,\mu_3\,=\mu-\alpha_3\,=\,a\omega_1\,+\,\omega_2\,+\,(c-2)\omega_3,\;$
$\,\mu_4\,=\mu_2-(\alpha_1+\alpha_2+\alpha_3)\,=\,\,(a-3)\omega_1\,+\,
\omega_2\,+\,(c-1)\omega_3.\;$ Hence, Lemma~\ref{a333} applies. 

In both cases, $\,V\,$ is not exceptional, proving the claim.
\hfill $\Box$\vspace{1.5ex}

{\bf From now on, by Claims 1 and 2, we can assume that weights with 
$\,2\,$ or more nonzero coefficients occurring in $\,{\cal X}_{++}(V)\,$ 
are good weights.}

{\bf Claim 3}: \label{cl3a3}  {\it Let $\,V\,$ be an $\,A_3(K)$-module. 
If $\,{\cal X}_{++}(V)\,$ contains a good weight
$\,\mu\,=\,a\omega_1\,+\,b\omega_2\,+\,c\omega_3\,$ such that
($\,a\geq 2,\,b\geq 1,\,c\geq 1\,$) or ($\,a\geq 1,\,b\geq 2,\,c\geq 1\,$)
or ($\,a\geq 1,\,b\geq 1,\,c\geq 2\,$), then $\,V\,$ is not 
an exceptional $\mathfrak g$-module.} 

Indeed, let $\,\mu\,=\,a\omega_1\,+\,b\omega_2\,+\,c\omega_3\,$ be a good
weight in $\,{\cal X}_{++}(V)$, with $\,3\,$ nonzero coefficients.

(a) Let $\,a\geq 3,\,b\geq 1,\,c\geq 1\,$ (or graph-dually $\,a\geq 1,
\,b\geq 1,\,c\geq 3\,$) and
 $\,\mu\,=\,a\omega_1\,+\,b\omega_2\,+\,c\omega_3\,$.
For $\,p=2,\,$ $\,\mu\,$ does not occur in $\,{\cal X}_{++}(V)$.
For $\,p\geq 3,\,$ $\,{\cal X}_{++}(V)\,$ also contains 
%$\,\mu\,=\,a\omega_1\,+\,b\omega_2\,+\,c\omega_3,\;$    
$\,\mu_1\,=\mu-\alpha_1\,=\,(a-2)\omega_1\,+\,(b+1)\omega_2\,+
\,c\omega_3,\;$ and 
$\,\mu_2\,=\mu-(\alpha_1+\alpha_2)\,=\,(a-1)\omega_1\,+\,(b-1)\omega_2\,+
\,(c+1)\omega_3.\,$
By Claims 1 and 2, $\,\mu_1\,$ and $\,\mu_2\,$ are good weights.
Hence, by Lemma~\ref{a333}, $\,V\,$ is not exceptional.

(b) Let $\,a=2,\,b\geq 1,\,c\geq 1\,$ (or graph-dually $\,a\geq 1,
\,b\geq 1,\,c=2\,$) and $\,\mu\,=\,2\omega_1\,+\,b\omega_2\,+\,c\omega_3\,$.
For $\,p=2\,$ such $\,\mu\,$ does not occur in $\,{\cal X}_{++}(V)$.
For $\,p\geq 3$, $\,{\cal X}_{++}(V)\,$ also contains 
$\,\mu_1\,=\mu-\alpha_1\,=\,(b+1)\omega_2\,+\,c\omega_3\,$,
$\,\mu_2\,=\mu-(\alpha_1+\alpha_2)\,=\,\omega_1\,+\,(b-1)\omega_2\,+
\,(c+1)\omega_3\,$,
$\,\mu_3\,=\mu-(\alpha_1+\alpha_2+\alpha_3)=\,\omega_1\,+\,b\omega_2\,+
\,(c-1)\omega_3.\,$ 
By Claims 1 and 2, these are all good weights. Thus % or if $\,c=1\,$,
$\,s(V)\,>\,48\,$ (since $\,|W\mu|=24\,$ and 
$\,|W\mu_i|\,\geq\,12\,$ for all $\,i$). Hence, by~\eqref{allim}, 
$\,V\,$ is not an exceptional module. 

(c) Let $\,a=1,\,b\geq 1,\,c\geq 2\,$ (or graph-dually $\,a\geq 2,
\,b\geq 1,\,c=1\,$) and $\,\mu\,=\,\omega_1\,+\,b\omega_2\,+\,c\omega_3\,$.
For $\,p=2\,$ $\,\mu\,$ does not occur in $\,{\cal X}_{++}(V)$.
For $\,p\geq 3,\,$ $\,{\cal X}_{++}(V)\,$ also contains
$\,\mu_1=\mu-(\alpha_1+\alpha_2+\alpha_3)=b\omega_2\,+\,(c-1)
\omega_3,\;$
$\mu_2=\mu-(\alpha_1+\alpha_2)=(b-1)\omega_2\,+\,(c+1)\omega_3,\;$
$\mu_3\,=\mu-(\alpha_2+\alpha_3)\,=\,2\omega_1\,+\,(b-1)\omega_2\,+
\,(c-1)\omega_3,\;$
$\,\mu_4\,=\mu_1-(\alpha_2+\alpha_3)\,=\,\omega_1\,+\,(b-1)\omega_2\,+
\,(c-2)\omega_3.\,$ By Claims 1 and 2, these are all good weights, 
hence Lemma~\ref{a333} applies. 

{\it Hence, if $\,{\cal X}_{++}(V)\,$ contains a good weight of the 
form $\,\mu\,=\,a\omega_1\,+\,b\omega_2\,+\,c\omega_3\,$ (with
$\,3\,$ nonzero coefficients), then we can assume $\,a=c=1\,$.}

(d) Let $\,a=1,\,b\geq 2,\,c=1\,$ and 
$\,\mu=\omega_1\,+\,b\omega_2\,+\,\omega_3\,$. 
For $\,p=2\,$ such $\,\mu\,$ does not occur in $\,{\cal X}_{++}(V)$.
For $\,p\geq 3,\,$ $\,{\cal X}_{++}(V)$ also contains the good weights
$\,\mu_1\,=\mu-(\alpha_1+\alpha_2)\,=\,(b-1)\omega_2\,+\,2\omega_3,\;$
$\,\mu_2\,=\mu-(\alpha_2+\alpha_3)\,=\,2\omega_1\,+\,(b-1)\omega_2,\;$\\
$\,\mu_3\,=\mu_1-(\alpha_2+\alpha_3)\,=\,\omega_1\,+\,(b-2)\omega_2\,+
\,\omega_3.\,$ For $\,b\geq 3,\,$ Lemma~\ref{a333} applies.
For $\,b=2,\,$ $\,s(V)\,\geq\,60\,>\,48\,$ and, by~\eqref{allim},
$\,V\,$ is not exceptional. This proves the claim.
\hfill $\Box$ \vspace{1.5ex}

{\bf Claim 4}: \label{cl4a3}  {\it Let ($\,a\geq 2,\,b\geq 2\,$) or
($\,a\geq 1,\,b\geq 3\,$) or ($\,a\geq 3,\,b\geq 1\,$). If
$\,V\,$ is an $\,A_3(K)$-module such that $\,{\cal X}_{++}(V)\,$ contains 
a good weight $\,\mu\,=\,a\omega_1\,+\,b\omega_2\,$ (or graph-dually 
$\,\mu\,=\,b\omega_2\,+\,a\omega_3\,$), then $\,V\,$ is not 
an exceptional $\mathfrak g$-module.} 

Indeed, let $\,\mu\,=\,a\omega_1\,+\,b\omega_2\,$ (with $\,a\geq 1,
\,b\geq 1\,$) (or graph-dually $\,\mu\,=\,b\omega_2\,+\,a\omega_3\,$) 
be a good weight in $\,{\cal X}_{++}(V)$. Then $\,\mu_1\,=\mu-(\alpha_1 
+\alpha_2)\,=\,\mbox{$(a-1)\omega_1$}\,+\,(b-1)\omega_2\,+\,\omega_3\,\in
\,{\cal X}_{++}(V)$.

i) For $\,a\geq 3,\,b\geq 2\,$ or $\,a\geq 2,\,b\geq 3\,$,
$\,\mu_1\,$ satisfies Claim~3 (p. \pageref{cl3a3}). 

ii) For $\,a=b=2\,$ (hence $\,p\geq 3\,$) $\,\mu\,=\,2\omega_1\,+\, 
2\omega_2\,$, $\,\mu_1\,=\,\omega_1\,+\,\omega_2\,+\,\omega_3\,$, 
$\,\mu_2\,=\mu-\alpha_2=\,3\omega_1\,+\,\omega_3\,$, 
$\,\mu_3\,=\mu_1-(\alpha_1+\alpha_2+\alpha_3)\,=\,\omega_2\,\in
\,{\cal X}_{++}(V)$. Thus, $\,s(V)\geq 12+24 +12+6=54\,>\,48\,$. Hence,
by~(\ref{allim}), $\,V\,$ is not exceptional.

iii) Let $\,a\geq 3,\,b=1\,$ and $\,\mu\,=\,a\omega_1\,+\,\omega_2\,$.
For $\,p=2,\,$ such $\,\mu\,$ does not occur in 
$\,{\cal X}_{++}(V)$. For $\,p\geq 3$, 
$\,\mu_1\,=\mu-(\alpha_1+\alpha_2)\,=\,(a-1)\omega_1\,+\,\omega_3,\;$
$\,\mu_2\,=\mu-\alpha_1\,=\,(a-2)\omega_1\,+\,2\omega_2,\;$
$\,\mu_3\,=\mu_2-(\alpha_1+\alpha_2)\,=\,(a-3)\omega_1\,+\,\omega_2\,
+\,\omega_3\,\in\,{\cal X}_{++}(V)$. By Claims 1 and 2, these are all 
good weights. Thus, for $\,a\geq 4$, $\,s(V)\geq 60\,>\,48\,$.
For $\,a=3\,$, also $\,\mu_4\,=\mu_3-(\alpha_2+\alpha_3)\,=\,\omega_1\,\in\,
{\cal X}_{++}(V)$. Thus, $\,s(V)\geq 52\,>\,48\,$. Hence, by~(\ref{allim}),
$\,V\,$ is not exceptional.

iv) Let $\,a=1,\,b\geq 3\,$ and $\,\mu\,=\,\omega_1\,+\,b\omega_2\,$. 
For $\,p=2,\,$ such $\,\mu\,$ does not occur in $\,{\cal X}_{++}(V)$.
For $\,p\geq 3$, $\,\mu_1\,=\mu-\alpha_2\,=\,2\omega_1\,+\, 
(b-2)\omega_2\,+\,\omega_3\,\in\,{\cal X}_{++}(V)\,$ satisfies 
Claim~3 (p. \pageref{cl3a3}). Hence, $\,V\,$ is not exceptional. 
This proves Claim~4. \hfill $\Box$ \vspace{1.5ex}

{\bf Claim 5}: \label{cl5a3}  {\it Let ($\,a\geq 2,\,b\geq 3\,$) or
($\,a\geq 3,\,b\geq 2\,$) or ($\,a\geq 4,\,b\geq 1\,$) or 
($\,a\geq 1,\,b\geq 4\,$ ). If $\,V\,$ is an $\,A_3(K)$-module such 
that $\,{\cal X}_{++}(V)\,$ contains 
a good weight $\,\mu\,=\,a\omega_1\,+\,b\omega_3\,$ (or graph-dually 
$\,\mu\,=\,b\omega_1\,+\,a\omega_3\,$), then $\,V\,$ is not 
an exceptional $\mathfrak g$-module.} Indeed:

i) let $\,a\geq 3,\,c\geq 2\,$ (or graph-dually $\,a\geq 2,\,c\geq 3\,$)
and $\,\mu\,=\,a\omega_1\,+\,c\omega_3\,$.
For $\,p=2,\,$ $\,\mu\,$ does not occur in $\,{\cal X}_{++}(V)$. 
For $\,p\geq 3$, $\,\mu_1\,=\mu-\alpha_1\,=\,(a-2)\omega_1\,
+\,\omega_2\,+\,c\omega_3\,\in\,{\cal X}_{++}(V)\,$ satisfies
Claim~4. Hence $\,V\,$ is not exceptional.

ii) Let $\,a\geq 4,\,c=1\,$ (or graph-dually $\,a=1,\,c\geq 4\,$) and 
$\,\mu\,=\,a\omega_1\,+\,\omega_3\,$. For $\,p=2,\,$ $\,\mu\,$ does not 
occur in $\,{\cal X}_{++}(V)$. For $\,p\geq 3$, 
$\,\mu_1\,=\mu-\alpha_1\,=\,(a-2)\omega_1\,+\,\omega_2\,+\,\omega_3\,
\in\,{\cal X}_{++}(V)\,$ satisfies Claim~3. Hence $\,V\,$ is not exceptional. 
This proves the claim.  \hfill $\Box$ \vspace{1.5ex}

{\bf Claim 7}: \label{cl7a3}  {\it
Let $\,V\,$ be an $\,A_3(K)$-module. If $\,{\cal X}_{++}(V)\,$ contains 

(a) $\,\mu\,=\,a\omega_1\,$ (or graph-dually $\,\mu\,=\,a\omega_3\,$)
with $\,a\geq 4\,$ or

(b) $\,\mu\,=\,a\omega_2\,$ with $\,a\geq 4\,$, \\  
then $\,V\,$ is not an exceptional $\mathfrak g$-module.} 
                                                            
Indeed, let $\,a\geq 5\,$ and $\,\mu\,=\,a\omega_1\,\in\,{\cal X}_{++}(V)$.
Then $\,\mu_1\,=\mu-\alpha_1\,=\,(a-2)\omega_1\,+\,\omega_2\in
{\cal X}_{++}(V)\,$ satisfies Claim~4 (p. \pageref{cl4a3}). 

Now let $\,\mu\,=\,4\omega_1\,$.
For $\,p=2,\,$ $\,\mu\,$ does not occur in $\,{\cal X}_{++}(V)$.
For $\,p=3,\,$ $\,\mu\in {\cal X}_{++}(V)\,$ only if $\,V\,$ has highest 
weight $\,\lambda\,=\,2\omega_1\,+\,2\omega_2\,+\,2\omega_3\,$. In this case, 
$\,V\,$ is not exceptional, by Claim~3 (p. \pageref{cl3a3}).
For $\,p\geq 5,\,$ $\,\mu_1\,=\mu-\alpha_1\,=\,2\omega_1\,+\,\omega_2\,$,
$\,\mu_2\,=\mu_1-(\alpha_1+\alpha_2)\,=\,\omega_1\,+\,\omega_3\,$,
$\,\mu_3\,=\mu_1-\alpha_1\,=\,2\omega_2\,\in\,{\cal X}_{++}(V)$.
Also for $\,p\geq 5,\,$ $\,|R_{long}^+-R^+_{\mu,p}|\,=\,3,\,$
$\,|R_{long}^+-R^+_{\mu_1,p}|\,=\,|R_{long}^+-R^+_{\mu_2,p}|\,=\,5,\,$ and
$\,|R_{long}^+-R^+_{\mu_3,p}|\,=\,4.\,$
Thus $\,r_p(V)\,\geq\,1\,+\,5\,+\,5\,+\,2\,=\,13\,>\,12\,$. Hence,
by~(\ref{rral}), $\,V\,$ is not an exceptional module. This proves (a).

For (b), if $\,a\geq 4\,$ and $\,\mu\,=\,a\omega_2\,\in\, 
{\cal X}_{++}(V)$, then $\,\mu_1\,=\mu-\alpha_2\,=\,\omega_1\,+\,(a-2)
\omega_2\,+\,\omega_3\in{\cal X}_{++}(V)$. In this case, $\,\mu_1\,$ satisfies
Claim~3 (p. \pageref{cl3a3}). Hence $\,V\,$ is not exceptional, proving 
this claim. \hfill $\Box$ \vspace{1.5ex}

Now we deal with some particular cases.

P.1) Let $\,V\,$ be an $\,A_3(K)$-module such that $\,{\cal X}_{++}(V)\,$ 
contains $\,\mu\,=\,3\omega_1\,+\,\omega_3\,$ (or graph-dually $\,\mu\,=\,
\omega_1\,+\,3\omega_3\,$). For $\,p=2,\,$ $\,\mu\,$ does not 
occur in $\,{\cal X}_{++}(V)$. For $\,p=3$,
$\,\mu\,\in\,{\cal X}_{++}(V)\,$ only if $\,V\,$ has 
highest weight $\,\lambda\,=\,2\omega_1\,+\,2\omega_2\,$ or
$\,\lambda\,=\,2\omega_1\,+\,\omega_2\,+\,2\omega_3$. In these cases
$\,V\,$ is not exceptional, by Claims~4 or~3, respectively.
For $\,p\geq 5$, we can assume that $\,V\,$ has highest 
weight $\,\mu\,=\,3\omega_1\,+\,\omega_3$.
Then $\,\mu_1\,=\mu-\alpha_1\,=\,\omega_1\,+\,\omega_2\,+\,\omega_3,\;$
$\,\mu_2\,=\mu-(\alpha_1+\alpha_2+\alpha_3)\,=\,2\omega_1,\;$
$\,\mu_3\,=\mu_1-(\alpha_1+\alpha_2+\alpha_3)\,=\,\omega_2,\;$ and
$\,\mu_4\,=\mu_1-(\alpha_1+\alpha_2)\,=\,2\omega_3\, 
\in\,{\cal X}_{++}(V)$. For $\,p\geq 5,\,$ by~\cite[p. 167]{buwil},
$\,m_{\mu_1}=m_{\mu_4}=1,\,m_{\mu_2}=m_{\mu_3}=3.\,$ Thus,
$\,s(V)\,=\,12+24+ 3\cdot 4+3\cdot 4+ 4 \,=\,52>48\,$. Hence, 
by~(\ref{allim}), $\,V\,$ is not exceptional.

P.2) Let $\,V\,$ be an $\,A_3(K)$-module of highest weight 
$\,\mu\,=\,2\omega_1\,+\,2\omega_3\,$ (hence, $\,p\geq 3\,$). Then 
$\,\mu_1\,=\mu-(\alpha_1+\alpha_2+\alpha_3)\,=\,\omega_1\,+\,\omega_3,\;$
$\,\mu_2\,=\mu-\alpha_1\,=\,\omega_2\,+\,2\omega_3,\;$ and
$\,\mu_3\,=\mu-\alpha_3\,=\,2\omega_1\,+\,\omega_2\,
\in\,{\cal X}_{++}(V)$. By~\cite[p. 166]{buwil}, 
for $\,p\geq 3,\,$ $\,m_{\mu_1}=2,\,$ thus
$\,s(V)\,=\,12+2\cdot 12+ 12+12\,=\,60>48\,$. Hence, by~(\ref{allim}), 
$\,V\,$ is not exceptional.

P.3) Let $\,V\,$ be an $\,A_3(K)$-module such that $\,{\cal X}_{++}(V)\,$ 
contains $\,\mu\,=\,3\omega_2\,$.
For $\,p=2,\,$ $\,\mu\,$ does not occur in $\,{\cal X}_{++}(V)$.
For $\,p=3,\,$ $\,\mu\in {\cal X}_{++}(V)\,$ only if $\,V\,$ has
highest weight $\,\lambda\,=\,2\omega_1\,+\,\omega_2\,+\,2\omega_3\,$
or $\,\lambda\,=\,2\omega_1\,+\,2\omega_2\,$ (or graph-dually
$\,\lambda\,=\,2\omega_2\,+\,2\omega_3\,$). In these cases, $\,V\,$ is not
exceptional, by Claim~3 or Claim~4, respectively.
For $\,p\geq 5$, we can assume that $\,V\,$ has highest weight
$\,\mu\,=\,3\omega_2$. Then the good weights
$\,\mu_1\,=\mu-\alpha_2\,=\,\omega_1\,+\,\omega_2\,+\,\omega_3\,$,
$\,\mu_2\,=\mu_1-(\alpha_1+\alpha_2)\,=\,2\omega_3\,$,
$\,\mu_3\,=\mu_1-(\alpha_2+\alpha_3)\,=\,2\omega_1\,$,
$\,\mu_4\,=\mu_3-\alpha_1\,=\,\omega_2\,\in\,{\cal X}_{++}(V)$.
By~\cite[p. 166]{buwil}, $\,m_{\mu_4}=2.\,$ Thus,
$\,s(V)\,=\,6+ 24+ 4+ 4+ 2\cdot 6\,=\,50\,>\,48\,$. Hence, by~\eqref{allim}, 
$\,V\,$ is not exceptional.

P.4) Now we can assume that $\,V\,$ is an $\,A_3(K)$-module of highest weight 
$\,\mu\,=\,\omega_1\,+\,\omega_2\,+\,\omega_3\,$. Then
$\,{\cal X}_{++}(V)\,$ also contains
$\,\mu_1\,=\mu-(\alpha_1+\alpha_2+\alpha_3)\,=\,\omega_2\,$,
$\,\mu_2\,=\mu-(\alpha_1+\alpha_2)\,=\,2\omega_3\;$ and
$\,\mu_3\,=\mu-(\alpha_2+\alpha_3)\,=\,2\omega_1\,$.
By~\cite[p. 166]{buwil}, for
$\,p\geq 5,\,$ $\,m_{\mu_2}=m_{\mu_3}=2\,$ and $\,m_{\mu_1}\geq 3.\,$ Thus,
$\,s(V)\,\geq\,24+3\cdot 6 +2\cdot 4 +2\cdot 4\,=\,58\,>\,48\,$. Hence, by
\eqref{allim}, $\,V\,$ is not exceptional. For $\,p=3,\,$ 
$\,m_{\mu_2}=m_{\mu_3}=1\,$ and $\,m_{\mu_1}=2.\,$ Also 
$\,|R_{long}^+-R^+_{\mu,3}|\,=\,5,\,$
$\,|R_{long}^+-R^+_{\mu_1,3}|\,=\,4,\,$
$\,|R_{long}^+-R^+_{\mu_2,3}|\,=\,|R_{long}^+-R^+_{\mu_3,3}|\,=\,3$.
Hence, $\,r_3(V) \,\geq\,2\cdot 5\,+\,4\,+\,1\,+\,1\,>\,12\,$
and, by~(\ref{rral}), $\,V\,$ is not an exceptional module. 
For $\,p=2,\,$ if $\,V\,$ has highest weight 
$\,\mu\,=\,\omega_1\,+\,\omega_2\,+\,\omega_3\,$,
then $\,V\,$ is the Steinberg module (see Theorem~\ref{stmod}). One
can prove, using Lemma~\ref{bor42}, that $\,m_{\mu_1}\,=\,4\,$. 
Thus, $\,r_2(V)\,\geq\,16\,>\,12\,$ and, by~\eqref{rral}, $\,V\,$ is not
exceptional.

P.5) Let $\,V\,$ be an $\,A_3(K)$-module such that $\,{\cal X}_{++}(V)\,$ 
contains $\,\mu\,=\,3\omega_1\,$.
For $\,p=2,\,$ $\,\mu\,$ does not occur in $\,{\cal X}_{++}(V)$.
For $\,p=3,\,$ $\,\mu\in {\cal X}_{++}(V)\,$ only if $\,V\,$ has highest 
weight $\,\lambda\,=\,2\omega_1\,+\,\omega_2\,+\,\omega_3\,$ or
$\,\lambda\,=\,\omega_1\,+\,2\omega_2\,+\,2\omega_3\,$.
In these cases $\,V\,$ is not exceptional, by Claim~3 (p. \pageref{cl3a3}).
For $\,p\geq 5,\,$ we can assume that $\,V\,$ has highest weight
$\,\mu\,=\,3\omega_1\,$. Hence, by Lemma~\ref{III.iii} (p. \pageref{III.iii})
$\,V\,$ is not exceptional.

{\bf Hence, if $\,{\cal X}_{++}(V)\,$ contains a weight of the 
form $\,\mu\,=\,a\omega_i\,+\,b\omega_j\,$, then we can assume 
($\,a\leq 1,\,b\leq 2\,$) or ($\,a\leq 2,\,b\leq 1\,$).} 
These cases are treated in the sequel.
 
P.6) Let $\,V\,$ be an $\,A_3(K)$-module of highest weight 
$\,\mu\,=\,2\omega_1\,+\,\omega_2\,$ (or graph-dually 
$\,\mu\,=\,\omega_2\,+\,2\omega_3\,$) (hence $\,p\geq 3\,$). Then 
$\,\mu_1\,=\mu-(\alpha_1+\alpha_2)\,=\,\omega_1\,+\,\omega_3\;$ and
$\,\mu_2\,=\mu-\alpha_1\,=\,2\omega_2\,\in\,{\cal X}_{++}(V)$. 
By~\cite[p. 166]{buwil}, $\,m_{\mu_1}=2,\,m_{\mu_2}=1$. We have
$\,|R_{long}^+-R^+_{\mu,p}|\,=\,5\,(\mbox{for}\;p\neq 3)\,$ and
$\,4\,(\mbox{for}\;p=3),\,$ $\,|R_{long}^+-R^+_{\mu_1,p}|\,=\,5,\,$ 
$\,|R_{long}^+-R^+_{\mu_2,p}|\,=\,4.\,$ 
Thus, for $\,p\geq 3$, $\,r_p(V) \,\geq\,4\,+\,2\cdot 5\,+\,2\,=\,16\,>\,12\,$.
Hence, by~(\ref{rral}), $\,V\,$ is not exceptional.

P.7) Let $\,V\,$ be an $\,A_3(K)$-module of highest weight 
$\,\mu\,=\,\omega_1\,+\,2\omega_2\,$ (or graph-dually 
$\,\mu\,=\,2\omega_2\,+\,\omega_3\,$) (hence $\,p\geq 3\,$). Then
$\,\mu_1\,=\mu-(\alpha_1+\alpha_2)\,=\,\omega_2\,+\,\omega_3\,$,
$\,\mu_2\,=\mu-\alpha_2\,=\,2\omega_1\,+\,\omega_3\;$ and
$\,\mu_3\,=\mu_2-(\alpha_1+\alpha_2+\alpha_3)\,=\,\omega_1\,\in\,
{\cal X}_{++}(V)$.
By~\cite[p. 166]{buwil}, $\,m_{\mu_1}=2,\,m_{\mu_2}=1,\,m_{\mu_3}=3\,$. 
Thus, $\,s(V)\,=\,12+2\cdot 12+12+3\cdot 4\,=\,60>48\,$. Hence,
by~(\ref{allim}), $\,V\,$ is not an exceptional module.

P.8) Let $\,V\,$ be an $\,A_3(K)$-module such that $\,{\cal X}_{++}(V)\,$ 
contains $\,\mu\,=\,2\omega_1\,+\,\omega_3\,$ (or graph-dually 
$\,\mu\,=\,\omega_1\,+\,2\omega_3\,$).
For $\,p=2,\,$ $\,\mu\,$ does not occur in $\,{\cal X}_{++}(V)$.
For $\,p\geq 3$, we can assume that $\,V\,$ has highest weight $\,\mu$.
Thus $\,\mu_1\,=\mu-\alpha_1\,=\,\omega_2\,+\,\omega_3\;$ and
$\,\mu_2\,=\mu-(\alpha_1+\alpha_2+\alpha_3)\,=\,\omega_1\, 
\in\,{\cal X}_{++}(V)$. By~\cite[p. 166]{buwil},
for $\,p\neq 5,\,p\geq 3,\,$ $\,m_{\mu_2}=3,\,$ and
for $\,p=5,\,$ $\,m_{\mu_2}=2$. Also,
$\,|R_{long}^+-R^+_{\mu,p}|\,=\,5\,\,$ 
$\,|R_{long}^+-R^+_{\mu_1,p}|\,=\,5,\,$ and
$\,|R_{long}^+-R^+_{\mu_2,p}|\,=\,3.\,$ 
Thus, for $\,p\neq 5,\,$ $\,r_p(V)\,\geq\,5\,+\, 5\,+\,3\,>\,12\,$.
Hence, by~(\ref{rral}), $\,V\,$ is not an exceptional module.

U.1) For $\,p=5,\,$ the $\,A_3(K)$-modules of highest weight 
$\,2\omega_1\,+\,\omega_3\,$ (or graph-dually $\,\omega_1\,+\, 
2\omega_3\,$) are unclassified {\bf (N. 2 in Table~\ref{leftan}).}

{\bf Hence, if $\,{\cal X}_{++}(V)\,$ contains a weight of the form
$\,\mu\,=\,a\omega_i\,+\,b\omega_j\,$ (with $\,a\geq 1,\,b\geq 1\,$) 
then we can assume $\,a=b= 1\,$.} These cases are
treated in the sequel.

P.9) Let $\,V\,$ be an $\,A_3(K)$-module of highest weight 
$\,\mu\,=\,\omega_1\,+\,\omega_3\,$. Then $\,V\,$ is the adjoint module,  
which is exceptional by Example~\ref{adjoint} {\bf (N. 2 in 
Table~\ref{tablealall})}.

P.10) Let $\,V\,$ be an $\,A_3(K)$-module of highest weight 
$\,\mu\,=\,\omega_1\,+\,\omega_2\,$ (or graph-dually 
$\,\mu\,=\,\omega_2\,+\,\omega_3\,$). For $\,p=3,\,$
by Claim~12 (p. \pageref{claim12al}), $\,V\,$ is not exceptional.

U.2) For $\,p\neq 3,\,$ the $\,A_3(K)$-modules of highest weight 
$\,\mu\,=\,\omega_1\,+\,\omega_2\,$ are unclassified {\bf (N. 5 in 
Table~\ref{leftan})}.

{\bf Therefore, if $\,{\cal X}_{++}(V)\,$ contains a weight with 
$\,2\,$ nonzero coefficients, then $\,V\,$ is not an exceptional
module, unless $\,V\,$ has highest weight $\,\omega_1\,+\,\omega_3\,$
(the adjoint module, which is exceptional) or the highest weight of
$\,V\,$ is ($\,2\omega_1\,+\,\omega_3\,$ (or graph-dually 
$\,\omega_1\,+\,2\omega_3\,$) for $\,p=5\,$) or 
($\,\omega_1\,+\,\omega_2\,$ (or graph-dually
$\,\omega_2\,+\,\omega_3\,$) for $\,p\neq 3\,$) in which cases $\,V\,$ is 
unclassified.}

{\bf From now on we can assume that $\,{\cal X}_{++}(V)\,$ contains only 
weights with at most one nonzero coefficient.} 

P.11) Let $\,V\,$ be an $\,A_3(K)$-module such that 
$\,\mu\,=\,2\omega_1\,$ (or graph-dually $\,\mu\,=\,2\omega_3\,$) 
$\in\,{\cal X}_{++}(V)$. For $\,p=2\,$, $\,\mu\,\in\,{\cal X}_{++}(V)\,$  
only if $\,V\,$ has highest weight $\,\lambda\,=\,\omega_1\,+\,\omega_2\,
+\,\omega_3\,$. By case P.4) this module is not exceptional.
For $\,p\geq 3,\,$ we can assume that $\,V\,$ has highest weight
$\,\mu\,=\,2\omega_1\,$ (or graph-dually $\,\mu\,=\,2\omega_3$). 
Then $\,V\,$ is an exceptional module by case P.15 (p. \pageref{IIIbp15}) 
of First Part of Proof of Theorem~\ref{anlist} {\bf (N. 3 in 
Table~\ref{tablealall})}. 

P.12) If $\,V\,$ is an $\,A_3(K)$-module of highest weight 
$\,\mu\,=\,\omega_1\,$ (or graph-dually $\,\mu\,=\,\omega_3$), 
then $\,V\,$ is an exceptional module by case P.19 (p. \pageref{IIIp19})
of First Part of Proof Theorem~\ref{anlist} {\bf (N. 4 and 1 in  
Table~\ref{tablealall})} .

U.3) Let $\,V\,$ be an $\,A_3(K)$-module such that 
$\,\mu\,=\,2\omega_2\,\in\,{\cal X}_{++}(V)$.
For $\,p=2,\,$ $\,\mu\,$ does not occur in $\,{\cal X}_{++}(V)$. 
For $\,p\geq 3,\,$ we can assume that $\,V\,$ has highest weight 
$\,2\omega_2\,$. These modules are unclassified  
{\bf (N. 7 in Table~\ref{leftan}).}

P.13) If $\,V\,$ is an $\,A_3(K)$-module of highest weight
$\,\mu\,=\,\omega_2,\;$ then $\,\dim\,V=6\,<\,15\,-\,\varepsilon\,$.
Thus, by Proposition~\ref{dimcrit}, $\,V\,$ is an exceptional module
{\bf (N. 5 in Table~\ref{tablealall})}.

{\bf Hence if $\,V\,$ is an $\,A_3(K)$-module such that
$\,{\cal X}_{++}(V)\,$ has weights of the form
$\,\mu\,=\,a\omega_i\,$ (with $\,1\leq i\leq 3\,$ and $\,a\geq 3$), 
then $\,V\,$ is not an exceptional $\,\mathfrak g$-module. If $\,V\,$ has
highest weight $\,\lambda\,\in\,\{\omega_1,\,\omega_2,\,\omega_3,\,2\omega_1, 
2\omega_3\},\,$ then $\,V\,$ is an exceptional module.
If $\,V\,$ has highest weight 
$\,2\omega_2\,$ (for $\,p\geq 3\,$), then $\,V\,$ is unclassified.}

This finishes the proof of Theorem~\ref{anlist} for groups of type
$\,A_3$.
\hfill $\Box$

%\newpage
%\input{bnbnbn.tex}

\subsection{Groups or Lie Algebras of Type $\,B_{\ell}$}\label{appbl}
In this section we prove Theorem~\ref{listbn}. Recall that 
for groups of type $\,B_{\ell}\,$, $\,p\geq 3\,$, 
$\,|W|\,=\,2^{\ell}\,\ell!\,$, $\,|R|\,=\,2\ell^2\,,\,|R_{long}|\, 
=\,2\ell(\ell-1)\,$.

\subsubsection{Type $\,B_{\ell}\,,\;\ell\geq 5$} \label{fppbl}

{\bf Proof of Theorem~\ref{listbn} - First Part} \\
Let $\,\ell\geq 5$. 
By Lemma~\ref{3coefbn}, {\bf it suffices to consider $\,B_{\ell}(K)$-modules
such that $\,{\cal X}_{++}(V)\,$ contains only weights with 2 or
less nonzero coefficients.} 

First suppose that $\,{\cal X}_{++}(V)\,$ contains weights of the form
$\,\mu\,=\,a\,\omega_{i}\,+\,b\,\omega_{j}\,$ with $\,1\leq i<j \leq\ell\,$ and
$\,a,\,b\,\geq 1$.

I) {\bf Claim 1}:\label{claim1bl}
{\it If $\,{\cal X}_{++}(V)\,$ contains a bad weight with 
$\,2\,$ nonzero coefficients, then $\,V\,$ is not an exceptional module.}

Indeed, let $\,\mu\,=\,a\,\omega_{i}\,+\,b\,\omega_{j}\,$ 
(with $\,1\leq i<j \leq\ell\,$ and $\,a,\,b\geq 1\,$) be a bad weight
in $\,{\cal X}_{++}(V)$. Then $\,a,\,b\,\geq 3$.

(a) For $\,1\leq i<j <\ell-1\,$,
$\,\mu_1\,=\,\mu-(\alpha_i+\cdots +\alpha_j)\,=\,\omega_{i-1}\,+\,(a-1)
\,\omega_{i}\,+\,(b-1)\,\omega_{j}\,+\,\omega_{j+1}\in
\,{\cal X}_{++}(V)$.

(b) For $\,1\leq i<\,j =\ell-1\,$,
$\,\mu\,=\,a\,\omega_{i}\,+\,b\,\omega_{\ell -1}\;$ and 
$\,\mu_1\,=\,\mu-(\alpha_i+\cdots +\alpha_{\ell -1})\,=\,\omega_{i-1}\,
+\,(a-1)\,\omega_{i}\,+\,(b-1)\,\omega_{\ell-1}\,+\,2\omega_{\ell}\in
\,{\cal X}_{++}(V)$.

(c) For $\,1\,<\, i\,<\,j =\ell\,$,
$\,\mu\,=\,a\,\omega_{i}\,+\,b\,\omega_{\ell}\;$ and 
$\,\mu_1\,=\,\mu-(\alpha_i+\cdots +\alpha_{\ell})\,=\,\omega_{i-1}\,+\,(a-1)
\,\omega_{i}\,+\,b\,\omega_{\ell}\,\in
\,{\cal X}_{++}(V)\,$.

(d) For $\,1= i,\,j =\ell\,$,
$\,\mu\,=\,a\,\omega_{1}\,+\,b\,\omega_{\ell}\;$ and 
$\,\mu_1\,=\,\mu-\alpha_1\,=\,(a-2)\omega_{1}\,+\,
\,\omega_{2}\,+\,b\,\omega_{\ell}\,\in
\,{\cal X}_{++}(V)\,$.\\
In all these cases, $\,\mu_1\,$ 
is a good weight with 3 or more nonzero coefficients. Thus,
Lemma~\ref{3coefbn} applies and $\,V\,$ is not exceptional. 
This proves the claim. (Note that this claim is true for $\,\ell\geq 4$.)
\hfill $\Box$ 

{\bf Therefore if $\,{\cal X}_{++}(V)\,$ contains weights with 2
nonzero coefficients, then we can assume that these are good weights.}

II) Let $\,\mu\,=\,a\,\omega_{i}\,+\,b\,\omega_{j}\,$ (with $\,a,\,b\geq 1\,$)
be a good weight in $\,{\cal X}_{++}(V)$.

{\bf Claim 2}: \label{claim2bl}
{\it Let $\,\ell\geq 5$. If $\,V\,$ is a $\,B_{\ell}(K)$-module such
that $\,{\cal X}_{++}(V)\,$ contains a good weight $\,\mu\,=\,a\, 
\omega_{i}\,+\,b\,\omega_{\ell-2}\,$ (with $\,a\geq 1,\,b\geq 1\,$  
and $\,1\leq i \leq \ell-3\,$), then $\,V\,$ is not an exceptional 
$\,\mathfrak g$-module.} 

Indeed, let $\,\mu\,=\,a\,\omega_{i}\,+\,b\,\omega_{\ell-2}\,$.
Then for $\,1\leq i \leq\ell -3\,$,
\[
|W\mu|\,=\,\displaystyle %\frac{2^{\ell}\,\ell!}{i!\,(\ell-2-i)!\,2^2\,2!}\,
2^{\ell -3}\,\ell\,(\ell -1)\,\binom{\ell-2}{i}\,\geq\,2^{\ell-3}\,
\ell\,(\ell-1)\,(\ell-2)\,.
\]
Thus, for $\,\ell\geq 6$, $\,s(V)\geq 
2^{\ell-3}\,\ell\,(\ell-1)\,(\ell-2)\,>\,2\ell^3$. 
For $\,\ell=5,\,$ if $\,\mu\,=\,a\omega_{2}\,+\,b\omega_{3}\,$, then 
$\,\mu_1\,=\,\mu-(\alpha_2+\cdots +\alpha_3)\,=\,\omega_{1}\,+\,(a-1)
\omega_{2}\,+\,(b-1)\omega_{3}\,+\,\omega_{4}\,\in\,{\cal X}_{++}(V)$. 
Thus $\,s(V)\,\geq\,2^4\cdot 5 \cdot 3\,+\,2^6\cdot 5\,>\,2\cdot 5^3\,$. 
If $\,\mu\,=\,a\omega_{1}\,+\,b\omega_{3}\,$, then 
$\,\mu_1\,=\,(a-1)\omega_{1}\,+\,(b-1)\omega_{3}\,+\,\omega_{4}\, 
\in\,{\cal X}_{++}(V)$. In this case 
$\,s(V)\,\geq\,2^4\cdot 5\cdot 3\,+\,2^4\cdot 5\,>\,2\cdot 5^3\,$.
In any of these cases, by~\eqref{bllim}, $\,V\,$ is not 
exceptional, proving the claim. \hfill $\Box$ \vspace{1.5ex}

{\bf Claim 3}: \label{claim3bl} 
{\it Let $\,\ell\geq 6$. If $\,V\,$ is a $\,B_{\ell}(K)$-module such
that $\,{\cal X}_{++}(V)\,$ contains a good weight $\,\mu\,=\,a\, 
\omega_{i}\,+\,b\,\omega_{j}\,$ (with $\,a\geq 1,\,b\geq 1\,$ and 
$\,1\leq i < j,\; 3\leq j\leq \ell-3\,$), then $\,V\,$ is not an 
exceptional $\,\mathfrak g$-module.} 

Indeed, for $\,1\leq i < j,\;3\leq j\leq \ell -3\,$ (hence $\,\ell\geq 6\,$)
we have $\,\displaystyle\binom{j}{i}\geq 3\,$ and  
$\,\displaystyle\binom{\ell}{i}\geq\binom{\ell}{3}\,$. Hence
$\,|W\mu|\,=\,\displaystyle 2^j\,\binom{j}{i}\,\binom{\ell}{j}\,\geq\,
2^3\cdot 3\,\binom{\ell}{3}\,=\,2^2\,\ell\,(\ell-1)\,(\ell-2)$. 
Thus, for $\,\ell\geq 6$, $\,s(V)\,-\,2\ell^3\,\geq\,2^2\,\ell\,(\ell-1)\, 
(\ell-2)\,-\,2\ell^3\,=\,2\ell^2\,(\ell -6)\,+\,8\ell\,>\,0\,$.
Hence, by~\eqref{bllim}, $\,V\,$ is not exceptional. This proves the claim.
\hfill $\Box$ \vspace{1.5ex}

{\bf Therefore if $\,{\cal X}_{++}(V)\,$ contains
a good weight $\,\mu\,=\,a\,\omega_{i}\,+\,b\,\omega_{j}\,$  
(with $\,a\geq 1,\,b\geq 1\,$), then we can assume that ($\,i=1,\,j=2\,$) or
($\,1\leq i < j\,$ and $\,j\in\{\ell -1,\,\ell\}\,$).}
These cases are treated in the sequel. \vspace{1ex}

{\bf Claim 4}: \label{claim4bl} 
{\it Let $\,\ell\geq 5$. If $\,V\,$ is a $\,B_{\ell}(K)$-module such
that $\,{\cal X}_{++}(V)\,$ contains
a good weight $\,\mu\,=\,a\,\omega_{i}\,+\,b\,\omega_{\ell-1}\,$  
(with $\,a\geq 1,\,b\geq 1\,$ and $\,1\leq i \leq \ell-2\,$), then  
$\,V\,$ is not an exceptional $\,\mathfrak g$-module.}
 
Indeed, for $\,1\leq i\leq\ell-2\,$ and $\,\ell\geq 5$,
$\,s(V)\,\geq\,|W\mu|\,=\,\displaystyle2^{\ell-1}\,\ell\,\binom{\ell-1}{i}
\,\geq\,2^{\ell-1}\,\ell\,(\ell-1)\,>\,2\ell^3\,$. Hence,
by~\eqref{bllim}, $\,V\,$ is not exceptional, proving the claim.\hfill $\Box$
\vspace{1.5ex}

{\bf Claim 5}: \label{claim5bl}
{\it Let $\,\ell\geq 5$. If $\,V\,$ is a $\,B_{\ell}(K)$-module such
that $\,{\cal X}_{++}(V)\,$ contains
a good weight $\,\mu\,=\,a\,\omega_{i}\,+\,b\,\omega_{\ell}\,$  
(with $\,a\geq 1,\,b\geq 1\,$ and $\,1\leq i \leq \ell-1\,$), then  
$\,V\,$ is not an exceptional $\,\mathfrak g$-module.} Indeed:

(a) let $\,2\leq i\leq \ell-2\,$ and $\,\mu\,=\,a\,\omega_{i}\,+\,
b\,\omega_{\ell}\,$ (with $\,a\geq 1,\,b\geq 1\,$). Then for $\,\ell\geq 5$, 
$\,|W\mu|\,=\,\displaystyle 2^{\ell}\,\binom{\ell}{i}\,\geq \, 
2^{\ell}\,\binom{\ell}{2}\,$. Thus, for $\,\ell\geq 5$, $\,s(V)\,\geq
\,2^{\ell-1}\,\ell\,(\ell-1)\,>\,2\ell^3\,$. Hence, by~\eqref{bllim},
$\,V\,$ is not exceptional.

(b) Let $\,i=\ell -1,\,j=\ell\,$ and $\,\mu\,=\,a\,\omega_{\ell-1}\,+\,
b\,\omega_{\ell}\,$ (with $\,a\geq 1,\,b\geq 1\,$).  Then
$\,\mu_1\,=\,\mu-(\alpha_{\ell-1}+\alpha_{\ell})\,=\,
\omega_{\ell-2}\,+\,(a-1)\omega_{\ell-1}\,+\,b\omega_{\ell}\,
\in\,{\cal X}_{++}(V)$. For $\,\ell\geq 5,\,$ $\,\mu_1\,$ satisfies
(a) of this proof. Hence $\,V\,$ is not exceptional.

(c) Let $\,i=1,\;j=\ell\,$ and $\,\mu\,=\,a\,\omega_{1}\,+\,b\,
\omega_{\ell}\,$ (with $\,a\geq 1,\,b\geq 1\,$). Then $\,s(V)\,\geq\,
|W\mu|\,=\,\displaystyle\frac{2^{\ell}\,\ell!}{(\ell-1)!}\,=\,2^{\ell}\,
\ell\,>\,2\,\ell^3,\,$ for all $\,\ell\geq 7$.

Let $\,\ell=5,\,6\,$ and $\,\mu\,=\,a\,\omega_{1}\,+\,b\,\omega_{\ell}\,$.
For $\,a\geq 1,\,b\geq 2,\,$ $\,\mu_1\,=\,\mu-\alpha_{\ell}\,=\,a\omega_1\,+\,
\omega_{\ell-1}\,+\,(b-2)\omega_{\ell}\,\in\,{\cal X}_{++}(V)$.
Hence, for $\,b\geq 3,\,$ Lemma~\ref{3coefbn} applies and, for
$\,b=2,\,$ $\,\mu_1\,$ satisfies Claim~4 (p. \pageref{claim4bl}).

For $\,a\geq 2,\,b=1,\,$ $\,\mu_1\,=\,\mu-(\alpha_1+\cdots+ 
\alpha_{\ell})\,=\,(a-1)\omega_1\,+\,\omega_{\ell}\,\in\,{\cal X}_{++}(V)$.  
Thus, for $\,\ell=5,\,6$, $\,s(V)\,\geq\,|W\mu|+|W\mu_1|\,\geq\,2\cdot
2^{\ell}\,\ell\,>\,2\cdot \ell^3$.

Hence we may assume $\,a=b=1\,$ and $\,\mu\,=\,\omega_{1}\,+\,\omega_{\ell}\,$.
Then $\,\mu_1\,=\,\omega_{\ell}\,\in\,{\cal X}_{++}(V)$. For $\,\ell=6$,  
$\,s(V)\,\geq\,|W\mu|+|W\mu_1|\,\geq\,2^6\cdot 6\,+\,2^6\,=\,2^6\cdot
7\,>\,2\cdot 6^3\,$.

For $\,\ell=5$, we can assume that $\,V\,$ has highest weight 
$\,\mu\,=\,\omega_{1}\,+\,\omega_{5}\,$. Then $\,\mu_1\,=\,\omega_5\, 
\in\,{\cal X}_{++}(V)$. Let %$\,\mu\,$ be the highest weight of $\,V\,$ 
$\,v_0\,$ be a highest weight vector of $\,V\,$.
Consider $\,v_1\,=\,f_5\,f_1\,v_0.\,$  Note that
$\,v_1\neq 0\,$ and $\,e_i\,v_1\,=\,0\,$ for $\,i\in\{ 2,3,4\}.\,$ Thus, 
$\,v_1\,$ is a highest weight vector for the action of $\,A\,=\,\langle
e_{\pm\alpha_i}\,/\,2\leq i\leq4\rangle\,\cong\,A_3\,$ on $\,V.\,$ 
The weight of $\,v_1\,$ with respect to the toral subalgebra $\,\langle
h_{\alpha_i}\,/\,2\leq i\leq4\rangle\,$ is 
$\,\omega_2\,+\,\omega_4\,=\,\overline{\omega_1}\,+\,\overline{\omega_3}.\,$
Thus, the $\,A$-module generated by $\,v_1\,$ is a homomorphic image of
$\,\mathfrak{sl}(4)\,$. Its zero weight is $\,\mu_1\,=\,\omega_5\,=\,\bar 0$.
As $\,p>2,\,$ $\,\mathfrak{sl}(4)\,$ is a simple Lie
algebra, hence irreducible. It follows that
$\,m_{\mu_1}\,=\,\dim\,V_{\mu_1}\,\geq 3,\,$ and 
$\,s(V)\,\geq\,2^5\cdot 5\,+\,3\cdot 2^5\,=\,2^8\,>\,2\cdot
5^3$. 

Hence in any of these cases,  by~\eqref{bllim}, $\,V\,$ is not an 
exceptional module. This proves Claim~5. \hfill $\Box$ \vspace{1.5ex}

{\bf Therefore if $\,{\cal X}_{++}(V)\,$ contains
a good weight $\,\mu\,=\,a\,\omega_{i}\,+\,b\,\omega_{j}\,$  
(with $\,a\geq 1,\,b\geq 1\,$), then we can assume that $\,i=1,\,j=2\,$.}

(d) Let $\,\mu\,=\,a\,\omega_{1}\,+\,b\,\omega_{2}\,$ (with
$\,a,\,b\geq 1\,$) be a good weight in $\,{\cal X}_{++}(V)$.
Then $\,\mu_1\,=\,\mu-(\alpha_1+\alpha_{2})\,=\,(a-1)\omega_{1}\,+\,
(b-1)\omega_2\,+\,\omega_{3}\,\in\,{\cal X}_{++}(V)$.

i) For $\,a\geq 2,\,b\geq 2,\,$ $\,\mu_1\,$ is a good weight
with $3$ nonzero coefficients. Hence Lemma~\ref{3coefbn} applies. 

ii) For $\,a=1,\,b\geq 2,\,$ $\,\mu_1\,=\,(b-1)\omega_2\,+\,\omega_{3}\, 
\in\,{\cal X}_{++}(V)$. For $\,a\geq 2,\,b=1,\,$ $\,\mu_1\,=\,(a-1)
\omega_1\,+\,\omega_{3}\,\in\,{\cal X}_{++}(V)$. In both cases,  
for $\,\ell\geq 6$, $\,\mu_1\,$ satisfies Claim~3 (p. \pageref{claim3bl}).  
For $\,\ell=5\,$, $\,\mu_1\,$ satisfies Claim~2 (p. \pageref{claim2bl}).
In these cases, $\,V\,$ is not exceptional.

iii) For $\,a=b=1,\,$ we may assume that $\,V\,$ has highest
weight $\,\mu\,=\,\omega_{1}\,+\,\omega_{2}\,$. Then
$\,\mu_1=\mu-(\alpha_1+\alpha_2)\,=\,\omega_3,\;$
$\,\mu_2=\mu_1-(\alpha_3+\cdots +\alpha_{\ell})\,=\,\omega_2,\;$
$\,\mu_3=\mu_2-(\alpha_2+\cdots +\alpha_{\ell})\,=\,\omega_1,\,$
$\,\mu_4=\mu-(\alpha_2+\cdots +\alpha_{\ell})\,=\,2\omega_1\,
\in\,{\cal X}_{++}(V)$.

Let $\,v_0\,$ be a highest weight vector of $\,V\,$. Apply $\,e_2\,e_1\,$ 
and $\,e_1\,e_2\,$ to the nonzero vectors $\,v_1\,=\,f_1\,f_2\,v_0,\,$
$\,v_2\,=\,f_2\,f_1\,v_0\,$ (of weight $\,\mu_1\,$) to prove that  
$\,v_1,\,v_2\,$ are linearly independent for $\,p\neq 3.\,$ 
%As $\,v_1,\,v_2\,$ have weight $\,\mu_1,\,$ 
This implies $\,m_{\mu_1}\geq 2\,$. Thus, for $\,p\neq 3\,$ and
$\,\ell\geq 6$, $\,s(V)\,\geq\,4\,\ell\,(\ell-1)\, 
+\,2\displaystyle\frac{4\ell\,(\ell-1)\,(\ell-2)}{3}\,+\,2\,\ell\,
(\ell-1)\,=\,\displaystyle\frac{\ell\,(\ell-1)\,(8\ell-2)}{3}>\,2\,\ell^3\,$. 
For $\,\ell=5$, $\,s(V)\,\geq\,2^4\cdot
5\,+\,2\cdot 2^4\cdot 5\,+\,2^3\cdot 5\,+\,2\cdot 2\cdot 5\,>\,2\cdot 5^3.$
Hence, by~\eqref{bllim}, $\,V\,$ is not exceptional.

For $\,p=3,\,$ $\,|R_{long}^+-R^+_{\mu,3}|\,=\,
|R_{long}^+-R^+_{\mu_2,3}|\,=\,4\ell-7,\,$
$\,|R_{long}^+-R^+_{\mu_1,3}|\,=\,3(2\ell-5),\,$ and
$\,|R_{long}^+-R^+_{\mu_3,3}|\,=\,
|R_{long}^+-R^+_{\mu_4,3}|\,=\,2(\ell-1).\,$ Thus, for $\,\ell\geq 4,\,$
$\,r_3(V) \geq\,2(4\ell-7)\,+\,2(\ell-2)(2\ell-5)\,+\,4\ell-7\,+\,2\,+\,2\,
=\, 4\ell^2\,-\,6\ell\,+\,3\,>\,2\ell^2\,$.
In this case, by~\eqref{rrbl}, $\,V\,$ is not an exceptional module.

{\bf Hence, for $\,\ell\geq 5,\,$ if $\,{\cal X}_{++}(V)\,$
contains weights with 2 nonzero coefficients, then $\,V\,$ is not an
exceptional module.} 

{\bf From now on we only need to consider $\,B_{\ell}(K)$-modules such
that $\,{\cal X}_{++}(V)\,$ contains only weights with at most 
one nonzero coefficient.}

III) Let $\,\mu\,=\,a\,\omega_i\,$ (with $\,1\leq i\leq\ell\,$ and 
$\,a\geq 1\,$) be a weight in $\,{\cal X}_{++}(V)$.

(a)i) Let $\,a\geq 3,\,$ $\,2\leq i\leq \ell -2\,$ and
$\,\mu\,=\,a\,\omega_i\,$. Then
$\,\mu_1\,=\,\mu-\alpha_i\,=\,\omega_{i-1}\,+\,(a-2)\omega_{i}\,+\,
\omega_{i+1}\,\in{\cal X}_{++}(V)$. As $\,\mu_1\,$ has $3$ nonzero 
coefficients, Lemma~\ref{3coefbn} applies (for $\,\mu\,$ bad or good).

ii) Let $\,a\geq 3\,$ and $\,\mu\,=\,a\,\omega_{\ell-1}\,$. Then
the good weight $\,\mu_1\,=\,\mu-\alpha_{\ell-1}\,=\,\omega_{\ell-2}\,
+\,(a-2)\omega_{\ell-1}\,+\,2\omega_{\ell}\,\in\,{\cal
X}_{++}(V)$. Hence Lemma~\ref{3coefbn} applies.

iii) Let $\,a\geq 3\,$ and $\,\mu\,=\,a\,\omega_{\ell}\,$. Then
$\,\mu_1\,=\mu-\alpha_{\ell}=\omega_{\ell-1}\,+\, 
\mbox{$(a-2)\omega_{\ell}$}\,\in\,{\cal X}_{++}(V)\,$ satisfies
Claim~5 (p. \pageref{claim5bl}).

iv) Let $\,a\geq 3\,$ and $\,\mu\,=\,a\,\omega_1\,\in\,{\cal X}_{++}(V)$. 

iv.1) For $\,a\geq 4,\,$ $\,\mu_1\,=\mu-\alpha_1=\,\mbox{$(a-2)\omega_{1}$}\, 
+\,\omega_{2},\;$ and $\;\mu_2\,=\,\mu-(\alpha_1+\alpha_2)\,=\, 
(a-3)\omega_{1}\,+\,\omega_{3}\,\in\,{\cal X}_{++}(V)$.
For $\,\ell\geq 6,\,$ $\,\mu_2\,$ satisfies Claim~3
(p. \pageref{claim3bl}). For $\,\ell=5,\,$ $\,\mu_2\,$ satisfies Claim~2
(p. \pageref{claim2bl}). Hence $\,V\,$ is not exceptional.

iv.2) For $\,a=3,\,$ $\,\mu\,=\,3\,\omega_1\,$. If $\,p=3,\,$ then 
$\,\mu\in {\cal X}_{++}(V)\,$ only if $\,V\,$ has highest
weight $\,\lambda\,=\,2\omega_1\,+\,\omega_i\,$ (for some $\,2\leq  
i\leq\ell-1\,$) or $\,\lambda\,=\,\omega_1\,+\,\omega_i\,+\,\omega_j\,$
(for some $\,2\leq i\leq j\leq\ell-1\,$) or $\,\lambda\,=\,2\omega_i\,+
\,\omega_j\,$ (for some $\,2\leq i <j\leq\ell-1\,$), for instance. 
In any of these cases, by Part II of this proof or  
by Lemma~\ref{3coefbn}, $\,V\,$ is not an exceptional module.

For $\,p\geq 5,\,$ $\,\mu\,=\,3\omega_1,\,$ $\,\mu_1=\omega_1\,+ 
\,\omega_{2},\,$ $\,\mu_2\,=\,\omega_{3},\,$
$\,\mu_3\,=\mu_2-(\alpha_3+\cdots+\alpha_{\ell})=\,\omega_{2},\;$
$\,\mu_4\,=\mu_3-(\alpha_2+\cdots+\alpha_{\ell})=\,\omega_{1},\;$
$\,\mu_5\,=\mu_1-(\alpha_2+\cdots+\alpha_{\ell})=\,2\omega_{1}\in
{\cal X}_{++}(V)$. Also, $\,|R_{long}^+-R^+_{\mu,p}|\,=\,|R_{long}^+- 
R^+_{\mu_4,p}|\,=\,|R_{long}^+-R^+_{\mu_5,p}|\,=\,2(\ell-1),\,$
$\,|R_{long}^+-R^+_{\mu_1,p}|\,=\,2(2\ell-3),\,$ 
$\,|R_{long}^+-R^+_{\mu_2,p}|\,=\,3(2\ell-5),\,$ and
$\,|R_{long}^+-R^+_{\mu_3,p}|=4\ell-7.\,$
Thus, for $\,\ell\geq 5\,$ (and for $\,\ell=4\,$),
$\,r_p(V) \,\geq\,2\,+\,2\cdot 2(2\ell-3)\,+\,2(\ell-2)(2\ell-5)\,
+\,4\ell-7\,+\,2\,+\,2\, = 4\ell^2\,-\,6\ell\,+\,13\,>\,2\ell^2\,$.
Hence, by\eqref{rrbl}, $\,V\,$ is not exceptional. 

{\bf Hence, for $\,\ell\geq 5\,$ if $\,V\,$ is a $\,B_{\ell}(K)$-module 
such that $\,{\cal X}_{++}(V)\,$ contains weights
$\,\mu\,=\,a\,\omega_i$, with $\,1\leq i\leq\ell\,$ and 
\mbox{$\,a\geq 3$,} then $\,V\,$ is not an exceptional
$\,\mathfrak g$-module. Therefore we may assume $\,a\leq 2\,$.} These cases
are treated in the sequel.\vspace{1.5ex}

{\bf Claim 6}: \label{claim6bl} 
{\it Let $\,\ell\geq 6\,$ and $\,4\leq i\leq \ell-1\,$. 
If $\,V\,$ is a $\,B_{\ell}(K)$-module such that $\;\omega_{i}\,
\in\,{\cal X}_{++}(V)\,$, then $\,V\,$ is not an exceptional 
$\,\mathfrak g$-module.}

Indeed, if $\;\omega_{i}\,\in\,{\cal X}_{++}(V)\,$, then also
$\;\omega_{i-1},\;\cdots,\,\omega_{4},\;\omega_{3},\;\omega_{2},\; 
\omega_{1}\,\in\,{\cal X}_{++}(V)$. Now $\,\displaystyle
\sum_{j=1}^{4}\,|W\omega_j|\,=\, 2\ell\,+\,2\ell(\ell-1)\,+\,\displaystyle
\frac{2^2\ell(\ell-1)(\ell-2)}{3}\,+\,\frac{2\ell(\ell-1)(\ell-2)(\ell-3)}{3}
\,=\, \displaystyle\frac{2\ell^4-8\ell^3 +16\ell^2 - 4\ell}{3}\,$.
Thus for $\,\ell\geq 6$, 
\[
s(V)\,-\,2\ell^3\,\geq\,\displaystyle
\frac{2\ell^3(\ell -7)\,+\,4\ell(4\ell -1)}{3}\,>\,0\,.
\]
%For $\,\ell=6,\,$ $\,s(V)\,\geq %\sum_{j=1}^{4}\,|W\omega_j|\,=
%\,2^4\cdot 3\cdot 5\,+\,2^5\cdot 5\,+\,2^2\cdot 3\cdot 5\,+\,2^2\cdot
%3\,=\,2^3\cdot 59\,>\,2\cdot 6^3$. 
Hence for $\,\ell\geq 6$, 
by~\eqref{bllim}, $\,V\,$ is not exceptional, proving the
claim. \hfill $\Box$ \vspace{1.5ex}
 
(b) Let $\,a=2\,$ and $\,\mu\,=\,2\,\omega_i\,$ (with $\,1\leq i\leq
\ell\,$) be a weight in $\,{\cal X}_{++}(V)$.

i) For $\,3\leq i\leq \ell -3,\,$ $\,\mu_1\,=\,\mu-\alpha_i\,=\, 
\omega_{i-1}\,+\,\omega_{i+1}\,\in\,{\cal X}_{++}(V)\,$
satisfies Claim 2 or Claim 3. For $\,i=2,\,$  
$\,\mu_1\,=\,\omega_1\,+\,\omega_3\,$ satisfies Claim 3
(for $\,\ell\geq 6\,$) or Claim 2 (for $\,\ell=5\,$). For $\,i=\ell
-2,\,$ $\,\mu_1\,=\,\omega_{\ell-3}\,+\,\omega_{\ell-1}\,$ satisfies
Claim~4 (p. \pageref{claim4bl}). In any of these cases, $\,V\,$ is not
an exceptional module.

ii) For $\,i=\ell-1\,$, $\,\mu_1\,=\,\mu-\alpha_{\ell-1}\,=\, 
\omega_{\ell-2}\,+\,2\omega_{\ell}\,\in\,{\cal X}_{++}(V)\,$ satisfies
Claim~5 (p. \pageref{claim5bl}), for $\,\ell\geq 5$.

iii) Let $\,i=\ell\,$ and $\,\mu\,=\,2\,\omega_{\ell}\,$. Then
$\,\mu_1\,=\,\omega_{\ell-1}\,\in\,{\cal X}_{++}(V)$. Hence, for 
$\,\ell\geq 6,\,$ Claim~6 (p. \pageref{claim6bl})
applies and $\,V\,$ is not exceptional.

For $\,\ell=5,\,$ we can assume that $\,V\,$ has highest weight
$\,\mu\,=\,2\,\omega_{5}\,$. Then $\,\mu_1\,=\,\mu-\alpha_{5}\,=\, 
\omega_{4},\;$ $\,\mu_2\,=\,\mu-(\alpha_{4}+2\alpha_{5})\,=\,\omega_{3},\;$
$\,\mu_3\,=\,\mu-(\alpha_3+2\alpha_{4}+3\alpha_{5})\,=\,\omega_{2},\;$
$\,\mu_4\,=\,\mu_3-(\alpha_2+\alpha_{3}+\alpha_{4}+\alpha_5)\,=\,\omega_{1}\;$
are good weights in $\,{\cal X}_{++}(V).\,$ 
Consider the Lie algebra $\,A\,=\,\langle e_{\pm\alpha_i}\,/\,3\leq i\leq
5\rangle\,\cong\,B_3.\,$ The module $\,V|_A\,$ has highest weight 
$\,\mu\,=\,2\omega_5\,=\,2\overline{\omega_3}\,$ with respect to the
toral subalgebra
$\,\langle h_{\alpha_i}\,/\,3\leq i\leq 5\rangle.\,$ 
Then, by Smith's Theorem~\cite{smith}, the zero weight $\,\mu_3\,=\, 
\omega_{2}\,=\,\bar 0\,$ of $\,V|_A\,$ has multiplicity
$\,m_{\mu_3}\geq 2\,$ for $\,p\geq 3\,$~\cite[p. 167]{buwil}. Thus
$\,s(V)\,\geq\,2^5\,+\,2^4\cdot 5\,+\,2^4\cdot 5\,+\,2\cdot 2^3\cdot
5\,+\,2\cdot 5\,=\,2\cdot 141\,>\,2\cdot 5^3\,$. Hence, by~\eqref{bllim},
$\,V\,$ is not an exceptional module.

iv)\label{IIIbiv} Let $\,i=1\,$ and $\,\mu\,=\,2\,\omega_1\,$. 

By Section~\ref{roapm}, $\,E(2\omega_1)\,\cong\,\displaystyle
\frac{{\rm Symm}_n\cap\mathfrak{sl}(n)}{scalars}\,$, where $\,n\,=\,2\ell+1$.  
Consider a generic enough diagonal matrix 
\[
D\,=\,\left(
\begin{array}{cccc} t_1 & 0 & \cdots & 0 \\
0 & t_2 & \cdots & 0 \\ \vdots &  & \ddots & \vdots\\
0 & 0 & \cdots & t_n \end{array}\right)
\]
with $\,n\,$ distinct eigenvalues.
Look at $\,D\,$ inside $\,E(2\omega_1)\,\cong\,\displaystyle 
\frac{{\rm Symm}_n\cap\mathfrak{sl}(n)}{scalars}\,$. $\,\mathfrak
g_D\,$ denotes the stabilizer of $\,D\,$ in $\,\mathfrak g\,$.
We have $\,M\in\mathfrak g_D\,\Longleftrightarrow\,({\rm ad}\,D)(M)\,=\,  
MD-DM\,=\,b\,{\rm Id}\,$, where $\,b\in K$. Hence, $\,({\rm
ad}\,D)^2(M)\,=\,[D,\, [D,\,M]]\,=\,0\,$. As $\,D\,$ is semisimple, so
is $\,{\rm ad}\,D\,$ whence $\,[D,\, [D,\,M]]\,=\,0\,$ implies  
$\,[D,\,M]\,=\,0\,$. It follows that $\,M\,$ is a diagonal matrix.  
As $\,M\,$ is also skew-symmetric, we have \mbox{$\,M=0\,$}. Therefore,  
$\,E(2\omega_1)\,$ is not an exceptional module, for $\,\ell\geq 2$.

{\bf Hence if $\,{\cal X}_{++}(V)\,$ contains weights
$\,\mu\,=\,a\,\omega_i$, with $\,1\leq i\leq\ell\,$ and 
\mbox{$\,a\geq 2$,} then $\,V\,$ is not exceptional. Therefore we may assume 
$\,a=1\,$.} This case is treated in the sequel.

(c) Let $\,V\,$ be a $\,B_{\ell}(K)$-module such that
$\,\mu\,=\,\omega_i\,\in\,{\cal X}_{++}(V)$. 

i) For $\,4\leq i\leq \ell-1\,$ and $\,\ell\geq 6,\,$ see Claim~6
(p. \pageref{claim6bl}).

For $\,\ell=5,\,$ we can assume that $\,V\,$ has highest weight
$\,\mu\,=\,\omega_{4}\,$. Then $\,\mu_1\,=\,\omega_{3},\,$ $\,\mu_2\,=\,
\omega_{2},\,$ $\,\mu_3\,=\,\omega_{1}\,\in\,{\cal X}_{++}(V).\,$
Consider the Lie algebra $\,A=\mbox{$\langle
e_{\pm\alpha_i}\,/\,3\leq i\leq 5\rangle$}
\,\cong\,B_3.\,$ $\,A\,$ acts on $\,V\,$ and with respect to the 
toral subalgebra
$\,\langle h_{\alpha_i}\,/\,3\leq i\leq 5\rangle\,$ of $\,A,\,$
$\,V|_A\,$ has highest weight $\,\mu\,=\,\omega_4\,=\,\overline 
{\omega_2}\,$. Then $\,V|_A\,$ corresponds to the adjoint 
$\,B_3(K)$-representation. Its zero weight $\,\mu_2\,=\,\omega_2\,=\,\bar 0\,$ 
has multiplicity $\,m_{\mu_2}=3\,$ for $\,p\geq 3\,$ 
\cite[p. 167]{buwil}. Thus
$\,s(V)\,\geq\,2^4\cdot 5\,+\,2^4\cdot 5\,+\,3\cdot 2^3\cdot
5\,+\,2\cdot 5\,=\,2\cdot 5\cdot 29\,>\,2\cdot 5^3\,$. Hence, 
by~\eqref{bllim}, $\,V\,$ is not an exceptional module. 

ii) If $\,V\,$ has highest weight $\,\mu=\omega_3$, then $\,V\,$ is
unclassified {\bf (N. 2 in Table~\ref{leftbn})}. 

iii) If $\,V\,$ has highest weight $\,\mu=\omega_2\,$, then 
$\,V\,$ is the adjoint module, which is exceptional by
Example~\ref{adjoint} {\bf (N. 2 in Table~\ref{tableblall})}. 

iv)\label{IIIciv} If $\,V\,$ has highest weight $\,\mu\,=\,\omega_1\,$ then,  
by~\cite[p. 300]{vinonis}, $\,\dim\,V\,=\,2\ell+1\,<\,2\ell^2
-\ell\,-\,\varepsilon,\,$ for any $\,\ell\geq 2.\,$ Thus, by
Proposition~\ref{dimcrit}, $\,V\,$ is an exceptional module 
{\bf (N. 1 in Table~\ref{tableblall})}.

v) Let $\,\mu\,=\,\omega_{\ell}\,$. Then, for $\,\ell\geq 12$, 
$\,s(V)\,\geq\,|W\mu|\,=\,2^{\ell}\,>\,2\cdot \ell^3\,$. Hence,  
by~\eqref{bllim}, $\,V\,$ is not an exceptional module. 

For $\,\ell=5,\,$ we can assume that $\,V\,$ has highest weight
$\,\mu\,=\,\omega_5\,$. Then, by~\cite[p.301]{vinonis},
$\,\dim\,V\,=\,2^5=32\,<\,55\,-\,\varepsilon\,$.
For $\,\ell=6,\,$ let $\,V\,$ have highest weight
$\,\mu\,=\,\omega_6\,$. Then, by~\cite[p.301]{vinonis},
$\,\dim\,V\,=\,2^6=64\,<\,78\,-\,\varepsilon\,$.
Thus, by Proposition~\ref{dimcrit}, for $\,\ell=5,\,6\,$
$\,V\,$ is an exceptional module {\bf (N. 3 in Table~\ref{tableblall})}. 
{\bf For $\,7\leq\ell\leq 11\,$ these modules are unclassified (N. 3 in
Table~\ref{leftbn})}.

This finishes the First Part of the Proof of Theorem~\ref{listbn}.
\hfill$\Box$

%\newpage

\subsubsection{Type $\,B_{\ell}\,$ - Small Rank Cases}\label{tbsr}

In this section, we prove Theorem~\ref{listbn} for groups of small
rank.
\subsubsection*{\ref{tbsr}.1. Type $\,B_2\,$ or $\,C_2\,$}
\addcontentsline{toc}{subsubsection}{\protect\numberline{\ref{tbsr}.1}
Type $\,B_2\,$ or $\,C_2\,$}

For groups of type $\,B_2\,$ or $\,C_2\,$,
$\,|W|\,=\,2^2\,2!\,=\,8,\,|R|= 8,\,|R_{long}|=
4$. The limit for~\eqref{bllim} is $\,8\cdot 2^2\,=\,32$. 
The dominant weights are of the form $\,\mu\,=\,a\,\omega_{1}\,+\,b\, 
\omega_2$, with $\,a,\,b\in\mathbb Z^+$. If $\,a\,$ and  
$\,b\,$ are both nonzero, then $\,|W\mu|\,=\,8\,$. If only one
coefficient is nonzero, then $\,|W\mu|\,=\,4$. For type $\,B_2\,$,
$\,R_{long}^+\,=\,\{\,\alpha_1,\,\alpha_1\,+\,2\alpha_2\,\}\,$ and
for type $\,C_2\,$,
$\,R_{long}^+\,=\,\{\,2\alpha_1\,+\,\alpha_2,\,\alpha_2\,\}\,$.

{\bf Proof of Theorem~\ref{listbn}, for $\,\ell=2$.}~
Let $\,V\,$ be a $\,B_2(K)$-module.

{\bf Claim 1}: {\it Bad weights of the form $\,\mu\,=\,a\omega_{1}\,+\, 
b\omega_2\,$, with $\,a\,$ and $\,b\,$ both nonzero, do not occur in  
$\,{\cal X}_{++}(V)$.} Indeed, otherwise $\,V\,$ would have highest
weight with at least one coefficient $\,\geq\,p\,$, contradicting
Theorem~\ref{curt}(i). \hfill $\Box$ \vspace{1.5ex}

{\bf Therefore, we can assume that weights with $2$ nonzero coefficients
occurring in $\,{\cal X}_{++}(V)\,$ are good weights.} \vspace{1.5ex}

{\bf Claim 2}: \label{cl2b2}
{\it Let ($\,a\geq 3,\,b\geq 1\,$) or ($\,a\geq 1,\,b\geq 3\,$) or   
($\,a=b=2\,$) or ($\,a=1,\,b=2\,$ for $\,p\neq 5\,$) or
($\,a=2,\,b=1\,$ for $\,p\geq 5\,$). If $\,V\,$ is a $\,B_2(K)$-module
such that $\,\mu\,=\,a\,\omega_{1}\,+\,b\,\omega_2\,$ is
a good weight in $\,{\cal X}_{++}(V)$, then $\,V\,$ is not an
exceptional $\,\mathfrak g$-module.} Indeed:

(a) let ($\,a\geq 3,\,b\geq 2\,$) or ($\,a\geq 2,\,b\geq 3\,$) and 
$\,\mu\,=\,a\,\omega_{1}\,+\,b\,\omega_2\,$. For $\,p=3\,$ such
$\,\mu\,$ does not occur in $\,{\cal X}_{++}(V)$. For $\,p\geq 5$, 
$\,\mu_1=\mu-\alpha_{1}=(a-2)\omega_{1}\,+\,(b+2)\omega_{2},\;$
$\,\mu_2\,=\,\mu-(\alpha_{1}+\alpha_2)\,=\,(a-1)\omega_{1}\,+\,
b\omega_{2},\;$ $\,\mu_3\,=\,\mu-\alpha_{2}\,=\,(a+1)\omega_{1}\,+\, 
(b-2)\omega_{2},\;$\linebreak 
$\,\mu_4\,=\,\mu_2-\alpha_{2}\,=\,a\,\omega_{1}\, 
+\,(b-2)\omega_{2},\;$ $\,\mu_5\,=\,\mu_4-(\alpha_{1}+\alpha_{2})\,=\,
(a-1)\omega_1\,+\,(b-2)\omega_2,\;$ $\,\mu_6\,=\,\mu_5-(\alpha_1+\alpha_2)\, 
=\,(a-2)\omega_1\,+\,(b-2)\omega_2\,\in\,{\cal X}_{++}(V)$. By Claim~1,
the weights with 2 nonzero coefficients are all good. Thus, 
$\,s(V)\,>\,32\,$ and, \mbox{by~\eqref{bllim},} $\,V\,$ is not an exceptional 
module.

{\it Hence, if $\,{\cal X}_{++}(V)\,$ contains weights with 2
nonzero coefficients, then we can assume $\,a\leq 2\,$ {\bf or}
$\,b\leq 2\,$.} We deal with these cases in the sequel.

(b) Let  $\,a=b=2\,$ and $\,\mu\,=\,2\omega_{1}\,+\,2\omega_2\,$. Then
$\,\mu_1\,=\,4\omega_{2},\;$ $\,\mu_2\,=\,\omega_{1}\,+\,2\omega_{2},\;$ 
$\,\mu_3\,=\,3\omega_{1},\;$ $\,\mu_4\,=\,2\omega_{1},\;$  
$\,\mu_5\,=\,\omega_1,\;$ $\,\mu_6\,=\,\mu_2-(\alpha_1+\alpha_2)\, 
=\,2\omega_2\,\in\,{\cal X}_{++}(V)$. For $\,p\geq 5$,
$\,s(V)\,\geq\,36\,>\,32\,$. Hence, by~\eqref{bllim}, $\,V\,$ is not 
exceptional. For $\,p=3$, $\,|R_{long}^+-R_{\mu,3}^+|\,=\,
|R_{long}^+-R_{\mu_1,3}^+|\,=\,|R_{long}^+-R_{\mu_6,3}^+|\,=\,1$
and $\,|R_{long}^+-R_{\mu_2,3}^+|\,=\,|R_{long}^+-R_{\mu_4,3}^+|\,=\,
|R_{long}^+-R_{\mu_5,3}^+|\,=\,2$. Thus,
$\,r_3(V)\,\geq\,12\,>\,8\,$. Hence, by~\eqref{rrbl}, $\,V\,$ is not 
exceptional. 

(c) Let $\,a=1,\,b\geq 5\,$ and $\,\mu\,=\,\omega_{1}\,+\,b\,\omega_2\,$. 
For $\,p=3\,$ such $\,\mu\,$ does not occur in $\,{\cal X}_{++}(V)$.  
For $\,p\geq 5$, $\,\mu_1\,=\,\mu-\alpha_{2}\,=\,2\omega_{1}\,+\,(b-2)
\omega_{2}\,\in\,{\cal X}_{++}(V)\,$ satisfies case (a) of this
proof. Hence $\,V\,$ is not exceptional.

(d) Let $\,a=1,\,b=4\,$ and $\,\mu\,=\,\omega_{1}\,+\,4\omega_2\,$. 
For $\,p=3\,$ such $\,\mu\,$ does not occur in $\,{\cal X}_{++}(V)$.  
For $\,p\geq 5$, $\,\mu_1\,=\,\mu-\alpha_{2}\,=\,2\omega_{1}\,+\,2
\omega_{2}\,\in\,{\cal X}_{++}(V)\,$ satisfies case (b) of this
proof. Hence $\,V\,$ is not exceptional.

(e) Let $\,a\geq 4,\,b=1\,$ and $\,\mu\,=\,a\omega_{1}\,+\,\omega_2\,$. 
For $\,p=3\,$ such $\,\mu\,$ does not occur in $\,{\cal X}_{++}(V)$.  
For $\,p\geq 5$, $\,\mu_1\,=\,\mu-\alpha_{1}\,=\,(a-2)\omega_{1}\,+\,3
\omega_{2}\,\in\,{\cal X}_{++}(V)\,$ satisfies II(a). Hence $\,V\,$ is
not exceptional.

{\it Hence, if $\,{\cal X}_{++}(V)\,$ contains weights of the form 
$\,\mu\,=\,a\omega_{1}\,+\,b\omega_2\,$ (with 2 nonzero coefficients),
then we can assume $\,a=1,\,b\leq 3\,$ or $\,a\leq 3,\,b=1\,$.} 
These cases are treated as follows.

(f) Let $\,a=1,\,b=3\,$ and $\,\mu\,=\,\omega_{1}\,+\,3\omega_2\,$. 
For $\,p=3\,$ such $\,\mu\,$ does not occur in $\,{\cal X}_{++}(V)$.  
For $\,p\geq 5$, $\,\mu_1\,=\,\mu-\alpha_{2}\,=\,2\omega_{1}\,+\,
\omega_{2}\,$, $\,\mu_2\,=\,\mu_1-(\alpha_1+\alpha_{2})\,=\,\omega_{1}\,+\,
\omega_{2}\,$, $\,\mu_3\,=\,\mu-(\alpha_1+\alpha_{2})\,=\,3\omega_{2}\,$,
$\,\mu_4\,=\,\mu_2-(\alpha_1+\alpha_{2})\,=\,\omega_{2}\,
\in\,{\cal X}_{++}(V)$. For $\,p\geq 5$, 
$\,|R_{long}^+\,-\,R_{\mu,p}^+|\,\geq\,1,\,$
$\,|R_{long}^+\,-\,R_{\mu_1,p}^+|\,=\,|R_{long}^+\,-\,R_{\mu_2,p}^+|\,=\,2\,$
and $\,|R_{long}^+\,-\,R_{\mu_3,p}^+|\,=\,|R_{long}^+\,-\,R_{\mu_4,p}^+|\,=\,
|R_{long}^+\,-\,R_{\mu_4,p}^+|\,=\,1$. Thus,
$\,r_p(V)\,\geq\,12\,>\,8\,$ and, by~\eqref{rrbl}, $\,V\,$ is not exceptional. 

(g) Let $\,a=3,\,b=1\,$ and $\,\mu\,=\,3\omega_{1}\,+\,\omega_2\,$. 
For $\,p=3\,$ such $\,\mu\,$ does not occur in $\,{\cal X}_{++}(V)$.  
For $\,p\geq 5$, $\,\mu_1\,=\,\mu-\alpha_1\,=\,\omega_{1}\,+\,
3\omega_{2}\,\in\,{\cal X}_{++}(V)\,$ satisfies case (f) of this
proof. Hence $\,V\,$ is not exceptional.

(h) Let $\,a=1,\,b=2\,$ and $\,\mu\,=\,\omega_{1}\,+\,2\omega_2\,$. 
Then $\,\mu_1\,=\,\mu-(\alpha_1+\alpha_{2})\,=\,2\omega_{2}\,$,  
$\,\mu_2\,=\,\mu-\alpha_{2}\,=\,2\omega_{1}\,$,  
$\,\mu_3\,=\,\mu_1-\alpha_2\,=\,\omega_{1}\,\in\,{\cal X}_{++}(V)$.  
For $\,p\neq 5$, $\,|R_{long}^+\,-\,R_{\mu,p}^+|\,=\,
|R_{long}^+\,-\,R_{\mu_2,p}^+|\,=\,|R_{long}^+\,-\,R_{\mu_3,p}^+|\,=\,2\,$
and $\,|R_{long}^+\,-\,R_{\mu_1,p}^+|\,=\,1$. Thus,
$\,r_p(V)\,\geq\,9\,>\,8\,$. Hence, by~\eqref{rrbl}, $\,V\,$ is not 
exceptional. 

(i) Let $\,a=2,\,b=1\,$ and $\,\mu\,=\,2\omega_{1}\,+\,\omega_2\,$. 
Then $\,\mu_1\,=\,\mu-\alpha_1\,=\,3\omega_{2}\,$, 
$\,\mu_2\,=\,\mu-(\alpha_1+\alpha_{2})\,=\,\omega_1\,+\,\omega_{2}\,$,
$\,\mu_3\,=\,\mu_2-(\alpha_1+\alpha_{2})\,=\,\omega_{2}\,
\in\,{\cal X}_{++}(V)$. For $\,p\geq 5$, $\,|R_{long}^+\,-\,R_{\mu,p}^+|\,=\,
|R_{long}^+\,-\,R_{\mu_2,p}^+|\,=\,2\,$ and $\,|R_{long}^+\,-\,R_{\mu_1,p}^+| 
\,=\,|R_{long}^+\,-\,R_{\mu_3,p}^+|\,=\,1$. Thus
$\,r_p(V)\,\geq\,10\,>\,8\,$ and, by~\eqref{rrbl}, $\,V\,$ is not exceptional. 

This proves Claim 2. \hfill $\Box$ \vspace{1.5ex}

{\bf Claim 3}: \label{cl3b2}
{\it Let $\,V\,$ be a $\,B_2(K)$-module such that $\,{\cal X}_{++}(V)\,$ 
contains 

(a) $\,\mu\,=\,a\omega_1\,$ with $\,a\geq 3\,$ \; or 

(b) $\,\mu\,=\,a\omega_2\,$ with $\,a\geq 5\,$ or ($\,a=4\,$ for
$\,p\neq 5\,$). \\
Then $\,V\,$ is not an exceptional $\,\mathfrak g$-module.}

Indeed, let $\,a\geq 4\,$ and $\,\mu\,=\,a\,\omega_1\,\in\,{\cal
X}_{++}(V)$. For $\,p=3\,$ such $\,\mu\,$ does not occur in 
$\,{\cal X}_{++}(V)$. For $\,p\geq 5,\,$ 
$\,\mu_1\,=\,\mu-\alpha_1\,=\,(a-2)\omega_{1}\,+\,2\omega_2\,\in\, 
{\cal X}_{++}(V)\,$ satisfies Claim~2 (p. \pageref{cl2a2}).
Hence $\,V\,$ is not exceptional.

Let $\,\mu\,=\,3\omega_1\,\in\,{\cal X}_{++}(V)$. (Hence $\,p\geq 3\,$.) Then
$\,\mu_1\,=\,\mu-\alpha_1\,=\,\omega_{1}\,+\,2\omega_{2}\,\in\, 
{\cal X}_{++}(V)\,$ satisfies Claim~2 for $\,p\neq 5\,$.
In this case $\,V\,$ is not exceptional. 
For $\,p=5,\,$ consider also $\,\mu_2\,=\,2\omega_2\,$,
$\,\mu_3\,=\,2\omega_{1}\,$, $\,\mu_4\,=\,\omega_{1}\,\in\,{\cal X}_{++}(V)$.  
As $\,|R_{long}^+\,-\,R_{\mu,5}^+|\,=\,|R_{long}^+\,-\,R_{\mu_3,5}^+|\,=\,
|R_{long}^+\,-\,R_{\mu_4,5}^+|\,=\,2\,$ and $\,|R_{long}^+\,-\, 
R_{\mu_1,5}^+|\,=\,|R_{long}^+\,-\,R_{\mu_2,5}^+|\,=\,1$, 
$\,r_p(V)\,\geq\,9\,>\,8\,$. Hence, by~\eqref{rrbl}, $\,V\,$ is  
not exceptional. This proves (a).
 
For (b), let $\,a\geq 4\,$ and $\,\mu\,=\,a\,\omega_2\,$. 
For $\,a\geq 5\,$ and $\,p=3\,$, $\,\mu\,$ does not occur in  
$\,{\cal X}_{++}(V)$. Let $\,a\geq 5\,$ and $\,p\geq 5$. Then 
$\,\mu_1\,=\,\mu-\alpha_2\,=\,\omega_{1}\,+\,(a-2) 
\omega_2\,\in\,{\cal X}_{++}(V)\,$ satisfies Claim~2 (p. \pageref{cl2a2}). 
Hence $\,V\,$ is not exceptional.

Let $\,\mu\,=\,4\omega_2\,\in\,{\cal X}_{++}(V)$. Then
$\,\mu_1\,=\,\mu-\alpha_2\,=\,\omega_1 + 2\omega_2\,\in\,{\cal
X}_{++}(V)\,$ satisfies Claim~2 for $\,p\neq 5$. 
Hence $\,V\,$ is not exceptional.
This proves Claim~3. \hfill $\Box$ \vspace{1.5ex}

Now we deall with some particular cases.

P.1) For $\,p=5\,$, let $\,V\,$ be a $\,B_2(K)$-module of highest weight 
$\,\mu\,=\,4\omega_2\,$. Then $\,\mu_1\,=\,\omega_1 + 2\omega_2\,$,
$\,\mu_2\,=\,\mu_1-\alpha_2\,=\,2\omega_1\,$,  
$\,\mu_3\,=\,\mu_1-(\alpha_1+\alpha_2)\,=\,2\omega_2\,$ and 
$\,\mu_4\,=\,\mu_3-\alpha_2\,=\,\omega_1\,\in\,{\cal X}_{++}(V)$. 
Let $\,v_0\,$ be a highest weight vector of $\,V\,$. To prove that
$\,m_{\mu_3}\,=\,2$, consider the nonzero weight vectors  
$\,v_1\,=\,f_2\,f_1\,f_2\,v_0\,$ and $\,v_2\,=\,f_1\,f_2^2\,v_0\,$  
of weight $\,\mu_3\,$. One can show that $\,e_{\tilde{\alpha}}\,v_1\neq 0\,$  
and $\,e_{\tilde{\alpha}}\,v_2= 0\,$, by using relations~\ref{relations}.  
This implies that $\,v_1\,$ and $\,v_2\,$ are linearly independent.  
Hence $\,m_{\mu_3}\,=\,2$. As $\,|R_{long}^+\,-\,R_{\mu,5}^+|\,=\,
|R_{long}^+\,-\,R_{\mu_1,5}^+|\,=\,|R_{long}^+\,-\,R_{\mu_3,5}^+|\,=\,1\,$
and $\,|R_{long}^+\,-\,R_{\mu_2,5}^+|\,=\,2$, one has 
$\,r_5(V)\,\geq\,9\,>\,8\,$. Hence, by~\eqref{rrbl}, $\,V\,$ is not 
exceptional.

P.2) For $\,p=3\,$, let $\,V\,$ be a $\,B_2(K)$-module of highest weight 
$\,\mu\,=\,2\omega_{1}\,+\,\omega_2\,$. Then $\,\mu_1\,=\,3\omega_{2}\,$, 
$\,\mu_2\,=\,\omega_1\,+\,\omega_{2}\,$, $\,\mu_3\,=\,\omega_{2}\,
\in\,{\cal X}_{++}(V)$. Let $\,v_0\,$ be a highest weight vector
of $\,V\,$. To prove that $\,m_{\mu_2}\,=\,2$, consider the nonzero
weight vectors $\,v_1\,=\,f_1\,f_2\,v_0\,$ and $\,v_2\,=\,f_2\,f_1\,v_0\,$  
of weight $\,\mu_2\,$. One can show that $\,e_1\,v_1=0\,$ and  
$\,e_1\,v_2\neq 0\,$, % and that $\,e_2\,v_1\neq 0\,$ and $\,e_2\,v_2=0\,$
by using relations~\ref{relations}. This implies
that $\,v_1\,$ and $\,v_2\,$ are linearly independent.  
Hence $\,m_{\mu_2}\,=\,2$. As $\,|R_{long}^+-R_{\mu,3}^+|\,=\,2\,$
and $\,|R_{long}^+-R_{\mu_2,3}^+|\,=\,|R_{long}^+-R_{\mu_1,3}^+| 
\,=\,|R_{long}^+\,-\,R_{\mu_3,3}^+|\,=\,1$, one has 
$\,r_3(V)\,\geq\,\frac{8\cdot 2}{4}+ 2 \frac{8}{4} + 1\,=\,9\,>\,8\,$.
Hence, by~\eqref{rrbl}, $\,V\,$ is not exceptional. 

P.3) Let $\,V\,$ be a $\,B_2(K)$-module of
highest weight $\,\mu\,=\,3\omega_2\,$. Then $\,V\,$ is isomorphic to
the $\,C_2(K)$-module $\,\tilde{V}\,$ with highest weight $\,\mu\,=\,
3\omega_1\,$. Recall that $\,\mathfrak{sp}_n\,$ is a Lie subalgebra of
$\,\mathfrak{sl}_{2n}\,$, hence a $\,\mathfrak{sl}_4$-module can be 
considered as a $\,\mathfrak{sp}_2$-module. By Lemma~\ref{III.iii}
(p. \pageref{III.iii}), the $\,\mathfrak{sl}_{4}$-module with highest 
weight $\,3\omega_1\,$ is not exceptional. Hence neither is  
$\,\tilde{V}\cong V\,$. 
 
P.4) For $\,p=5,\,$ let $\,V\,$ be a $\,B_2(K)$-module of highest
weight $\,\mu\,=\,\omega_{1}\,+\,2\omega_2\,$. Then
$\,\mu_1\,=\,2\omega_{2}\,$, $\,\mu_2\,=\,2\omega_{1}\,$,  
$\,\mu_3\,=\,\omega_{1}\,\in\,{\cal X}_{++}(V)$.  
Let $\,v_0\,$ be a highest weight vector
of $\,V\,$. To prove that $\,m_{\mu_3}\,=\,2$, consider the nonzero
weight vectors $\,v_1\,=\,f_2\,f_1\,f_2\,v_0\,$ and
$\,v_2\,=\,f_1\,f_2^2\,v_0\,$ of weight $\,\mu_3\,$. One can show that
$\,e_{\tilde{\alpha}}\,v_1\neq 0\,$ and $\,e_{\tilde{\alpha}}\,v_2=0\,$,  
by using relations~\ref{relations}. This implies that $\,v_1\,$ and 
$\,v_2\,$ are linearly independent. Hence $\,m_{\mu_3}\,=\,2$.
As $\,|R_{long}^+-R_{\mu,5}^+|\,=\,1\,$,
$\,|R_{long}^+-R_{\mu_2,5}^+|\,=\,|R_{long}^+-R_{\mu_3,5}^+|\,=\,2\,$
and $\,|R_{long}^+-R_{\mu_1,5}^+|\,=\,1$, one has  
$\,r_5(V)\,\geq\,\displaystyle\frac{8}{4} + \frac{4\cdot 2}{4} + 1 +
2\frac{4\cdot 2}{4}\,=\,9\,>\,8\,$. Hence, by~\eqref{rrbl}, $\,V\,$ is not 
exceptional. 

P.5) Let $\,V\,$ be a $\,B_2(K)$-module of highest weight 
$\,\mu\,=\,\omega_{1}\,+\,\omega_2\,$. For $\,p=3$, by Claim~12
(p. \pageref{claim12al}) of First Part of Proof of Theorem~\ref{anlist}, 
$\,V\,$ is not an exceptional $\,\mathfrak{sl}_4$-module. Hence  
$\,V\,$ is not an exceptional $\,\mathfrak{sp}_2$-module.
{\bf For $\,p\neq 3$, if $\,V\,$ has highest weight $\,\mu\,$, 
then $\,V\,$ is unclassified ( N. 1 in Table~\ref{leftbn}).}

{\bf Therefore, if $\,V\,$ is a $\,B_2(K)$-module such that
$\,{\cal X}_{++}(V)\,$ contains a weight with $2$  
nonzero coefficients, then $\,V\,$ is not an exceptional module,
unless $\,V\,$ has highest weight $\,\mu\,=\,\omega_{1}\,+\,\omega_2\,$ for 
$\,p\neq 3\,$) in which case $\,V\,$ is unclassified.} 

{\bf From now on we can assume that $\,{\cal X}_{++}(V)\,$ contains
only weights with at most one nonzero coefficient.}

P.6) If $\,V\,$ is a $\,B_2(K)$-module of highest weight 
$\,\mu\,=\,2\omega_1\,$ then, by case III(b)iv) (p. \pageref{IIIbiv}) 
of First Part of Proof of Theorem~\ref{listbn}, $\,V\,$ is not exceptional.

P.7) If $\,V\,$ is a $\,B_2(K)$-module of highest weight 
$\,\mu\,=\,2\omega_2\,$, then $\,V\,$ is the adjoint module. 
Hence, by Example~\ref{adjoint},
$\,V\,$ is an exceptional module {\bf (N. 2 in Table~\ref{tableblall})}.

P.8) If $\,V\,$ has highest weight $\,\mu\,=\,\omega_1\,$ then, by case
III(c)iv) (p. \pageref{IIIciv}) of First Part of Proof of  
Theorem~\ref{listbn}, $\,V\,$ is
an exceptional module {\bf (N. 1 in Table~\ref{tableblall})}. 

P.9) If $\,V\,$ has highest weight $\,\mu\,=\,\omega_2\,$ then,
by~\cite[p. 301]{vinonis}, $\,\dim\,V\,=\,2^2\,<\,10\,-\,\varepsilon\,$.
Hence, by Proposition~\ref{dimcrit}, $\,V\,$ is exceptional 
{\bf (N. 3 in Table~\ref{tableblall})}.  

{\bf Hence if $\,V\,$ is a $\,B_2(K)$-module such that
$\,{\cal X}_{++}(V)\,$ contains $\,\mu\,=\,a\omega_1\,$ 
(for $\,a\geq 2\,$) or $\,\mu\,=\,a\omega_2\,$ (for $\,a\geq 3\,$), then 
$\,V\,$ is not exceptional. If $\,V\,$ has highest weight $\,\lambda\,\in\,
\{2\omega_2,\,\omega_2,\omega_1\}\,$, then $\,V\,$ is an exceptional module.}

This finishes the proof of Theorem~\ref{listbn} for $\,\ell=2$.
\hfill $\Box$

\subsubsection*{\ref{tbsr}.2. Type $\,B_3$}
\addcontentsline{toc}{subsubsection}{\protect\numberline{\ref{tbsr}.2}
Type $\,B_3\,$}
           
For groups of type $\,B_3\,$, $\,|W|\,=\,2^3\,3!\,=\,48$,
$\,|R|=2\cdot 3^2,\,|R_{long}|= 12\,$. 
The limit for~\eqref{bllim} is $\,8\cdot 3^2\,=\,72$.
If $\,\mu\in{\cal X}_{++}(V)\,$ has $\,3\,$ nonzero coefficients, then
$\,|W\mu|\,=\,48$. If $\,\mu\in{\cal X}_{++}(V)\,$ 
has $\,2\,$ nonzero coefficients, then $\,|W\mu|\,=\,24$. Recall that
$\,p\geq 3$. 
$\,R_{long}^+\,=\,\{\,\alpha_1,\,\alpha_2,\,\alpha_1+\alpha_2,\,
\alpha_2+2\alpha_3,\,\alpha_1+\alpha_2+2\alpha_3,\,\alpha_1+2\alpha_2 
+2\alpha_3\,\}.\,$ The following are reduction lemmas.
\begin{lemma}\label{bn3} Let $\,V\,$ be a $\,B_3(K)$-module.
If $\,{\cal X}_{++}(V)\,$ contains 

(a) $\,2\,$ good weights with $\,3\,$ nonzero coefficients or 

(b) $\,1\,$ good weight with $\,3\,$ nonzero coefficients and  
$\,2\,$ good weights with $\,2\,$ nonzero coefficients,
then $\,V\,$ is not an exceptional $\,\mathfrak g$-module.
\end{lemma}\noindent
\begin{proof}
If (a) or (b) holds, then $\,s(V)\,>\,72$. Thus, by (\ref{bllim}), 
$\,V\,$ is not an exceptional module, proving the lemma. 
\end{proof}
\begin{corollary}\label{corbn3} If $\,V\,$ is a $\,B_3(K)$-module such that
$\,{\cal X}_{++}(V)\,$ contains a weight with $\,3\,$ nonzero
coefficients, then $\,V\,$ is not an exceptional $\,\mathfrak g$-module.
\end{corollary}
\begin{proof}
First we observe that bad weights with $\,3\,$ nonzero coefficients do not
occur in $\,{\cal X}_{++}(V)\,$ (otherwise $\,V\,$ would have highest
weight with at least one coefficient $\,\geq p\,$, contradicting
Theorem~\ref{curt}(i)). {\bf Therefore, we may assume that weights
with $\,3\,$ nonzero coefficients occurring in $\,{\cal X}_{++}(V)\,$ are good
weights.}

Let $\,\mu\,=\,a\,\omega_1\,+\,b\,\omega_2\,+\,c\,\omega_3\,$ be a good
weight in $\,{\cal X}_{++}(V)\,$ with $\,a,\,b,\,c\,\geq 1\,$. 

(a) For $\,a\geq 1,\,b\geq 2,\,c\geq 1$, 
$\,\mu\,=\,a\,\omega_1\,+\,b\,\omega_2\,+\,c\,\omega_3\,$ and 
$\,\mu_1\,=\mu-(\alpha_2+\alpha_3)\,=\,(a+1)\omega_1\,+\,(b-1)\omega_2\,+
\,c\omega_3\,$ are both good weights in $\,{\cal X}_{++}(V)\,$ (with
$\,3\,$ nonzero coefficients). Hence, by Lemma~\ref{bn3}(a), $\,V\,$
is not exceptional.

(b) Let $\,a\geq 2,\,b=1,\,c\geq 1\,$. Then 
$\,\mu\,=\,a\,\omega_1\,+\,\omega_2\,+\,c\,\omega_3\;$ and
$\,\mu_1\,=\mu-(\alpha_1+\alpha_2+\alpha_3)\,=\,(a-1)\,\omega_1\,+\,
\omega_2\,+ \,c\,\omega_3\,\in\,{\cal X}_{++}(V)\,$ are
good weights. As $\,\mu_1\,$ has 3 nonzero coefficients,
Lemma~\ref{bn3}(a) applies. 

(c) Let $\,a=b=1,\,c\geq 1\,$. Then $\,\mu\,=\,\omega_1\,+\,\omega_2\,+ 
\,c\,\omega_3,\;$ $\,\mu_1\,=\mu-(\alpha_1+\alpha_2+\alpha_3)\,=\,\omega_2\,+
\,c\,\omega_3\,\;$ and $\;\mu_2\,=\mu_1-(\alpha_2+\alpha_3)\,=\, 
\omega_1\,+\,c\,\omega_3\,\in\,{\cal X}_{++}(V)$. Hence, by
Lemma~\ref{bn3}(b), $\,V\,$ is not exceptional. This proves the
corollary.
\end{proof}\vspace{2ex}\noindent
{\bf Proof of Theorem~\ref{listbn}, for $\,\ell=3$.}\\ 
Let $\,V\,$ be a $\,B_3(K)$-module. By Corollary~\ref{corbn3}, {\bf it
suffices to consider $\,B_3(K)$-modules $\,V\,$ such that 
$\,{\cal X}_{++}(V)\,$ contains only weights with at most $\,2\,$  
nonzero coefficients.}

I) {\bf Claim 1}: {\it If $\,{\cal X}_{++}(V)\,$ contains a bad weight with 
$\,2\,$ nonzero coefficients, then $\,V\,$ is not an exceptional module.}

Indeed, let $\,\mu\,=\,a\omega_i\,+\,b\omega_j\,$ (with $\,a,\,b\,\geq
1\,$) be a bad weight in $\,{\cal X}_{++}(V)\,$. Hence $\,a\geq 3,\,b\,\geq
3\,$.

(a) Let $\,\mu\,=\,a\omega_1\,+\,b\omega_2\,$. Then 
$\,\mu_1\,=\mu-(\alpha_1+\alpha_2)\,=\,(a-1)\omega_1\,+\,(b-1)\omega_2\,+
\,2\omega_3\,$ is a good weight in $\,{\cal X}_{++}(V)\,$ with $\,3\,$
nonzero coefficients. 

(b) Let $\,\mu\,=\,a\omega_1\,+\,b\omega_3\,$. Then
$\,\mu_1\,=\mu-\alpha_3\,=\,a\omega_1\,+\,\omega_2\,+\,(b-2)\omega_3\,$
is a good weight in $\,{\cal X}_{++}(V)\,$ with $\,3\,$ nonzero
coefficients. 

(c) Let $\,\mu\,=\,a\omega_2\,+\,b\omega_3\,$. Then the good weight
$\,\mu_1\,=\mu-(\alpha_2+\alpha_3)\,=\,\omega_1\,+\,(a-1)\omega_2\,+\,b\,
\omega_3\,\in\,{\cal X}_{++}(V)\,$ has $\,3\,$ nonzero
coefficients.

In all these cases Corollary~\ref{corbn3} applies. Hence $\,V\,$ is not 
exceptional, proving the claim. \hfill $\Box$

{\bf Therefore, we can assume that weights with $\,2\,$ nonzero
coefficients occurring in $\,{\cal X}_{++}(V)\,$ are good weights.}

{\bf Claim 2}: \label{cl2b3}
{\it Let ($\,a\geq 2,\,b\,\geq 1\,$) or $\,a\geq 1,\,b\,\geq
2\,$. If $\,V\,$ is a $\,B_3(K)$-module such that 
$\,\mu\,=\,a\omega_i\,+\,b\omega_j\,\in\,{\cal X}_{++}(V)$, then 
$\,V\,$ is not an exceptional $\,\mathfrak g$-module.}

(a) Let $\,\mu\,=\,a\omega_2\,+\,b\omega_3\,$. Then: 

i) for $\,a\geq 2,\,b\geq 1\,$, %$\,\mu\,=\,a\omega_2\,+\,b\omega_3\,$
$\,\mu_1\,=\mu-(\alpha_2+\alpha_3)\,=\,\omega_1\,+\,(a-1)\omega_2\,+\,b\,
\omega_3\,\in\,{\cal X}_{++}(V)$. As $\,\mu_1\,$ has $\,3\,$ nonzero
coefficients, Corollary~\ref{corbn3} applies. 
%Hence $\,V\,$ is not exceptional.

ii) For $\,a=1,\,b\geq 2,\,$  $\,\mu\,=\,\omega_2\,+\,b\,\omega_3,\;$
$\,\mu_1\,=\mu-(\alpha_2+\alpha_3)\,=\,\omega_1\,+\,b\,\omega_3,\,$ 
$\,\mu_2\,=\mu_1-\alpha_3\,=\,\omega_1\,+\,\omega_2\,+\,(b-2)\omega_3\;$
and $\,\mu_3\,=\mu_2-(\alpha_1+\alpha_2+\alpha_3)\,=\,\omega_2\,+\,(b-2)
\omega_3\,\in\,{\cal X}_{++}(V)$. Thus $\,s(V)\,\geq\,84\,>\,72\,$. Hence,  
by~(\ref{bllim}), $\,V\,$ is not exceptional.

(b) Let $\,\mu\,=\,a\omega_1\,+\,b\omega_2\,$. Then 
 
i) for $\,a\geq 1,\,b\geq 2$, the good weight 
$\,\mu_1\,=\mu-(\alpha_1+\alpha_2)\,=\,(a-1)\omega_1\,+\,(b-1)\omega_2\,+
\,2\omega_3\,\in\,{\cal X}_{++}(V)$. For $\,a\geq 2$, $\,\mu_1\,$ has
$\,3\,$ nonzero coefficients and Corollary~\ref{corbn3} applies.
For $\,a=1$, $\,\mu_1\,$ satisfies case (a)ii) of this proof. 
Hence $\,V\,$ is not exceptional.

ii) For $\,a\geq 3,\,b=1,\,$ $\,\mu\,=\,a\,\omega_1\,+\,\omega_2,\;$
$\,\mu_1\,=\mu-(\alpha_1+\alpha_2)\,=\,(a-1)\omega_1\,+\,2\omega_3,\;$
and $\;\mu_2\,=\,\mu_1-\alpha_1\,=\,(a-3)\omega_1\,+\,\omega_2\,+\,2\omega_3\,
\in\,{\cal X}_{++}(V)$. For $\,a\geq 4$, $\,\mu_2\,$ has 3 nonzero  
coefficients and Corollary~\ref{corbn3} applies. For $\,a=3$, 
$\,\mu_1\,$ satisfies case (a)ii) of this proof. Hence $\,V\,$ is not 
exceptional.

iii) For $\,a=2,\,b=1,\,$ $\,\mu\,=\,2\,\omega_1\,+\,\omega_2,\;$
$\,\mu_1\,=\mu-(\alpha_1+\alpha_2)\,=\,\omega_1\,+\,2\omega_3,\;$\linebreak
$\,\mu_2\,=\,\mu_1-\alpha_3\,=\,\omega_1\,+\,\omega_2,\;$
$\,\mu_3\,=\,\mu_2-(\alpha_1+\alpha_2)\,=\,2\omega_3,\;$
$\,\mu_4\,=\,\mu_3-\alpha_3\,=\,\omega_2\,\in\,{\cal X}_{++}(V)$. Thus
$\,s(V)\,\geq\,84\,>\,72\,$. Hence, by~\eqref{bllim},
$\,V\,$ is not exceptional.

(c) Let $\,\mu\,=\,a\omega_1\,+\,b\omega_3\,$. Then

i) for $\,a\geq 1,\,b\geq 3$, $\,\mu_1\,=\mu-\alpha_3\,=\,a\omega_1\,
+\,\omega_2\,+\,(b-2)\omega_3\,\in\,{\cal X}_{++}(V)$. For 
$\,a\geq 3,\,b\geq 1,\,$ $\,\mu_2\,=\mu-\alpha_1\,=\,(a-2)\omega_1\, 
+\,\omega_2\,+\,b\omega_3\,\in\,{\cal X}_{++}(V)$. As $\,\mu_1\,$ and
$\,\mu_2\,$ are good weights with 3 nonzero coefficients,
Corollary~\ref{corbn3} applies in both cases. For $\,a=b=2$, note that
$\,\mu_2\,$ satisfies case (a)ii) of this proof. Hence $\,V\,$ is  
not exceptional.

ii) For $\,a=2,\,b= 1,\,$ $\,\mu\,=\,2\,\omega_1\,+\,\omega_3,\;$  
$\,\mu_1\,=\,\mu-\alpha_1\,=\,\omega_2\,+\,\omega_3,\;$
$\,\mu_2\,=\,\mu-(\alpha_1+\alpha_2+\alpha_3)\,=\,\omega_1\,+\,\omega_3,\;$
and $\;\mu_3\,=\,\mu_2-(\alpha_1+\alpha_2+\alpha_3)\,=\,\omega_3\,
\in\,{\cal X}_{++}(V)$. Thus $\,s(V)\,\geq\,24\,+\,24\,+\,24\,+\,8\,
=\,80\,>\,72\,$ and, by~(\ref{bllim}), $\,V\,$ is not exceptional. 

iii) Let $\,\mu\,=\,\omega_1\,+\,2\,\omega_3\,\in\,{\cal
X}_{++}(V)$. (Hence $\,p\geq 3\,$.) Then
$\,\mu_1\,=\,\omega_1\,+\,\omega_2,\,$
$\,\mu_2\,=\,\mu-(\alpha_1+\alpha_2+\alpha_3)\,=\,2\omega_3,\;$ 
$\,\mu_3\,=\,\mu_1-(\alpha_2+\alpha_3)\,=\,2\omega_1,\;$
$\,\mu_4\,=\,\mu_1-(\alpha_1+\alpha_2+\alpha_3)\,=\,\omega_2,\;$ and 
$\,\mu_5\,=\,\mu_4-(\alpha_2+\alpha_3)\,=\,\omega_1\,\in\,{\cal X}_{++}(V)$.
Thus, $\,s(V)\,=\,24\,+\,24\,+\,8\,+\,6\,+\,12\,+\,6\,=\,80\,>\,72\,$.
Hence, by~(\ref{bllim}), $\,V\,$ is not an exceptional module. 

This proves Claim 2. \hfill $\Box$ \vspace{1.5ex}

{\bf Claim 3}:\label{cl3b3} {\it Let $\,V\,$ be a $\,B_3(K)$-module. 
If $\,{\cal X}_{++}(V)\,$ contains $\,\omega_1\,+\,\omega_2\,$ (for
$\,p\geq 3\,$) {\bf or} $\,\omega_2\,+\,\omega_3\,$ (for $\,p\geq
5\,$), then $\,V\,$ is not an exceptional $\,\mathfrak g$-module.}

Indeed, assume that $\,\omega_1\,+\,\omega_2\;\in\,{\cal X}_{++}(V)$.
Then $\,\mu_1\,=\,2\omega_3,\;$ $\,\mu_2\,=\,2\omega_1,\;$
$\,\mu_3\,=\,\omega_2,\;$ $\,\mu_4\,=\,\mu_3-(\alpha_2+\alpha_3)\,=\,
\omega_1\,\in\,{\cal X}_{++}(V)$. For $\,p\geq 3$, 
$\,|R_{long}^+-R_{\mu,p}^+|\,\geq\,5\,$,
$\,|R_{long}^+-R_{\mu_3,p}^+|\,=\,5,\;|R_{long}^+-R_{\mu_1,p}^+|\,=\,3,\;
|R_{long}^+-R_{\mu_2,p}^+|\,=\,|R_{long}^+\,-\,R_{\mu_4,p}^+|\,=\,4$.
Thus $\,r_p(V)\,\geq\,\frac{24\cdot 5}{12} +\frac{8\cdot 3}{12}+
\frac{6\cdot 4}{12}+\frac{12\cdot 5}{12}+\frac{6\cdot 4}{12}\,=\,21\,>
\,18\,$. Hence, by~(\ref{rrbl}), $\,V\,$ is not exceptional.

Now suppose $\,\omega_2\,+\,\omega_3\;\in\,{\cal X}_{++}(V)$. Then  
$\,\mu_1\,=\, \omega_1\,+\,\omega_3\;$ and 
$\,\mu_3\,=\mu_1-(\alpha_1+\alpha_2+\alpha_3)\, 
=\,\omega_3\,\in\,{\cal X}_{++}(V)$.
For $\,p\geq 5,\,$ $\,|R_{long}^+\,-\,R_{\mu,5}^+|\,=\,
|R_{long}^+\,-\,R_{\mu_1,5}^+|\,=\,5,\;|R_{long}^+\,-\,R_{\mu_2,5}^+|\,=\,3$.
Thus $\,r_p(V)\,=\,\displaystyle\frac{24\cdot 5}{12}
\,+\,\frac{25\cdot 5}{12}\,+\, \frac{8\cdot 3}{12}\,=\,22\,>\,18$.
Hence, by~\eqref{rrbl}, $\,V\,$ is not an exceptional module.
This proves Claim~3.\hfill $\Box$ \vspace{1.5ex}

{\bf Claim 4}: \label{cl4b3}
{\it Let $\,V\,$ be a $\,B_3(K)$-module such that $\,{\cal X}_{++}(V)\,$ 
contains 

(a) $\,\mu\,=\,a\omega_1\,$ with $\,a\geq 3\,$ \; or 

(b) $\,\mu\,=\,a\omega_2\,$ with $\,a\geq 2\,$ \; or

(c) $\,\mu\,=\,a\omega_3\,$ with $\,a\geq 3\,$. \\
Then $\,V\,$ is not an exceptional $\,\mathfrak g$-module.}
Indeed:

(a) Let $\,a\geq 3\,$ and $\,\mu\,=\,a\omega_1\in{\cal X}_{++}(V)$. Then 
$\,\mu_1\,=\mu-\alpha_1\,=\,(a-2)\omega_1\,+\,\omega_2\,\in
\,{\cal X}_{++}(V)$. For $\,a\geq 4$, $\,\mu_1\,$ satisfies Claim~2
(p. \pageref{cl2b3}); for $\,a=3$, $\,\mu_1\,$ satisfies Claim~3
(p. \pageref{cl3b3}). Hence $\,V\,$ is not exceptional. 

(b) Let $\,a\geq 2\,$ and $\,\mu\,=\,a\omega_2\in{\cal X}_{++}(V)$. 
Then $\,\mu_1\,=\mu-\alpha_2\,=\,\omega_1\,+\,(a-2)\omega_2\,+\,2\omega_3\,
\in\,{\cal X}_{++}(V)$. For $\,a\geq 3\,$, $\,\mu_1\,$ has $\,3\,$
nonzero coefficients and Corollary~\ref{corbn3} applies. For
$\,a=2\,$, $\,\mu_1\,$ satisfies Claim~2 (p. \pageref{cl2b3}).
Hence $\,V\,$ is not exceptional. 

(c) Let $\,a\geq 3\,$ and $\,\mu\,=\,a\omega_3\,\in{\cal X}_{++}(V)$. 
Then $\,\mu_1\,=\mu-\alpha_3\,=\,\omega_2\,+\,(a-2)\omega_3\,\in\,
{\cal X}_{++}(V)$. For $\,a\geq 4$, $\,\mu_1\,$ satisfies Claim~2
(p. \pageref{cl2b3}). Hence $\,V\,$ is not exceptional.

For $\,a=3\,$ and $\,p=3,\,$ $\,\mu\in{\cal X}_{++}(V)\,$ only if 
$\,V\,$ has highest weight $\,\lambda\,=\,\omega_1\,+\,\omega_2\,+\,
\omega_3,\,$ for instance. In this case $\,V\,$ is not exceptional, by
Corollary~\ref{corbn3}.
Let $\,a=3\,$ and $\,p\geq 5\,$. Then $\,\mu_1\,=\,\omega_2+\omega_3\in\,
{\cal X}_{++}(V)\,$ satisfies Claim~3(b). Hence $\,V\,$ is not
exceptional. This proves the claim.\hfill $\Box$ \vspace{1.5ex}

{\bf Therefore, if $\,V\,$ is a $\,B_3(K)$-module such that 
$\,{\cal X}_{++}(V)\,$ contains a weight of the
form $\,\mu\,=\,a\omega_i\,+\,b\omega_j\,$ with $\,a\geq 1,\,b\,\geq 1\,$,
then we can assume $\,\mu\,=\,\omega_1\,+\,\omega_3\,$ or 
($\,\mu\,=\,\omega_2\,+\,\omega_3\,$ for $\,p=3\,$).} These cases
are treated in the sequel.

P.1) For $\,p=3,\,$ let $\,V\,$ be a $\,B_3(K)$-module of highest weight
$\,\mu\,=\,\omega_2\,+\,\omega_3\,$. Then $\,\mu_1\,=\, \omega_1\,+ 
\,\omega_3\;$ and $\,\mu_3\,=\mu_1-(\alpha_1+\alpha_2+\alpha_3)\, 
=\,\omega_3\,\in\,{\cal X}_{++}(V)$. By~\cite[p. 167]{buwil}, 
for $\,p=3,\,$ $\,m_{\mu_1}=2,\,m_{\mu_3}=4$. Thus $\,s(V)\,\geq\,
24\,+\,2\cdot 24\,+\,4\cdot 8\,=\,104\,>\,72.\,$ Hence, 
by~\eqref{bllim}, $\,V\,$ is not exceptional.

U.1) {\bf If $\,V\,$ is a $\,B_3(K)$-module of highest weight 
$\,\mu\,=\,\omega_1\,+\,\omega_3\,$,  
then $\,V\,$ is unclassified (N. 5 in Table~\ref{leftbn}).}  

{\bf Therefore, if $\,{\cal X}_{++}(V)\,$ contains a weight of the
form $\,\mu\,=\,a\omega_i\,+\,b\omega_j\,$ with $\,a,\,b\,\geq 1\,$,
then $\,V\,$ is not an exceptional module, unless $\,V\,$ has highest
weight $\,\omega_1\,+\,\omega_3\,$ in which case $\,V\,$ is
unclassified.}

{\bf From now on, we can assume that $\,{\cal X}_{++}(V)\,$ 
contains only weights with at most one nonzero coefficient.}

P.2) If $\,V\,$ is a $\,B_3(K)$-module of highest weight 
$\,\mu\,=\,2\omega_1\,$ then, by
case III(b)iv) (p. \pageref{IIIbiv}) of the First Part of Proof of 
Theorem~\ref{listbn}, $\,V\,$ is not an exceptional module.

U.2) {\bf If $\,V\,$ is a $\,B_3(K)$-module of highest weight  
$\,\mu\,=\,2\omega_3\,$, then
$\,V\,$ is unclassified (N. 4 in Table~\ref{leftbn}).}
%$\,\mu_1\,=\,\omega_2,\;$ $\,\mu_2\,=\,\omega_1\,$.?????????

P.3) If $\,V\,$ is a $\,B_3(K)$-module of highest weight  
$\,\mu\,=\,\omega_2\,$, then $\,V\,$
is the adjoint module, which is exceptional by Example~\ref{adjoint}
{\bf (N. 2 in Table~\ref{tableblall})}.

P.4) If $\,V\,$ is a $\,B_3(K)$-module of highest weight 
$\,\mu\,=\,\omega_1\,$ then, 
by case III(c)iv) (p. \pageref{IIIciv}) of First Part of Proof of 
Theorem~\ref{listbn}, $\,V\,$ is an exceptional module 
{\bf (N. 1 in Table~\ref{tableblall})}.

P.5) If $\,V\,$ is a $\,B_3(K)$-module of highest weight 
$\,\mu\,=\,\omega_3\,$, then
$\,\dim\,V\,=\,2^3=8\,<\,21-\varepsilon\,$. Hence, by
Proposition~\ref{dimcrit}, $\,V\,$ is an exceptional module {\bf (N. 3 in  
Table~\ref{tableblall})}.

{\bf Therefore, if $\,{\cal X}_{++}(V)\,$ contains $\,\mu\,=\,a\omega_i\,$ 
(for $\,a\geq 2\,$ and $\,i=1,\,2\,$) or $\,\mu\,=\,a\omega_3\,$ 
(for $\,a\geq 3\,$), then $\,V\,$ is not exceptional. If $\,V\,$ has 
highest weight $\,\lambda\,\in\,\{\omega_1,\,\omega_2,\omega_3\}\,$, 
then $\,V\,$ is exceptional. If $\,V\,$ has highest weight $\,2\omega_3\,$, 
then $\,V\,$ is unclassified.}

This finishes the proof of Theorem~\ref{listbn} for $\,\ell=3$.
\hfill$\Box$

\subsubsection*{\ref{tbsr}.3. Type $\,B_4$}
\addcontentsline{toc}{subsubsection}{\protect\numberline{\ref{tbsr}.3}
Type $\,B_4\,$}

For groups of type $\,B_4\,$, $\,|W|\,=\,2^4\,4!\,=\,384,\,|R|=2\cdot
4^2,\,|R_{long}|= 24\,$. The limit for~\eqref{bllim} is  
$\,8\cdot 4^2\,=\,2^7$. 
$\,R_{long}^+\,=\,\{\,\alpha_1,\,\alpha_2,\,\alpha_3,\,\alpha_1+\alpha_2,\,
\,\alpha_2+\alpha_3,\,\alpha_1+\alpha_2+\alpha_3,\,\alpha_3+2\alpha_4,
\,\alpha_2+\alpha_3+2\alpha_4,\,
\alpha_1+\alpha_2+\alpha_3+2\alpha_4,\,\alpha_2+2\alpha_3+2\alpha_4,
\,\alpha_1+\alpha_2+2\alpha_3+2\alpha_4,
\,\alpha_1+2\alpha_2+2\alpha_3+2\alpha_4\,\}.\,$\vspace{2ex}\\
{\bf Proof of Theorem~\ref{listbn} for $\,\ell=4$.} 
Let $\,V\,$ be a $\,B_4(K)$-module.

By Lemma~\ref{3coefbn}, {\bf it suffices to consider modules $\,V\,$
such that $\,{\cal X}_{++}(V)\,$ contains only weights with 2 or
less nonzero coefficients.} Moreover, by Claim 1
(p. \pageref{claim1bl}) of First Part of Proof of
Theorem~\ref{listbn}, we can assume that weights with $\,2\,$ nonzero  
coefficients occurring in $\,{\cal X}_{++}(V)\,$ are good weights.

{\bf Claim 1}: \label{cl1b4}
{\it Let $\,V\,$ be a $\,B_4(K)$-module.
For ($\,1\leq i < j\leq 3\,$) or ($\,2\leq i < j\leq 4\,$), let
$\,a\geq 1,\,b\geq 1\,$. For $\,i=1,\,j=4\,$, let $\,a\geq 2,\,b\geq
1\,$ or $\,a\geq 1,\,b\geq 2\,$. If $\,{\cal X}_{++}(V)\,$ contains  
a good weight $\,\mu\,=\,a\,\omega_{i}\,+\,b\,\omega_j\,$, then 
$\,V\,$ is not an exceptional $\,\mathfrak g$-module.} Indeed:

(a) let $\,\mu\,=\,a\omega_{2}\,+\,b\omega_3\,$ (with 
$\,a\geq 1,\,b\geq 1\,$) be a good weight in $\,{\cal X}_{++}(V)$. Then
$\,\mu_1\,=\mu-(\alpha_2+\alpha_3)=\,
\omega_1\,+\,(a-1)\omega_2\,+\,(b-1)\omega_3\,+\,2\omega_4\in 
{\cal X}_{++}(V)$. Thus $\,s(V)\,\geq\,|W\mu|\,+\,|W\mu_2|\,\geq\,
\displaystyle\frac{2^4\,4!}{2!2!}\,+\,\frac{2^4\,4!}{2!\,2!}\, 
=\,2^6\cdot 3\,>\,2^7\,$. Hence, by~(\ref{bllim}), $\,V\,$ is not exceptional.

(b) Let $\,\mu\,=\,a\omega_{2}\,+\,b\omega_4\,$ (with 
$\,a\geq 1,\,b\geq 1\,$) be a good weight in $\,{\cal X}_{++}(V)$. Then
$\,\mu_1\,=\mu-(\alpha_2+\alpha_3+\alpha_4)
=\,\omega_1\,+\,(a-1)\omega_2\,+\,b\omega_4\in{\cal X}_{++}(V)$.
Thus $\,s(V)\geq\,|W\mu|\,+\,|W\mu_1|\,\geq\,\displaystyle
\frac{2^4\,4!}{2!2!}\,+\,\frac{2^4\,4!}{3!}\,=\,2^5\cdot
5\,>\,2^7$. Hence, by~(\ref{bllim}), $\,V\,$ is not exceptional.

(c) Let $\,\mu\,=\,a\omega_{3}\,+\,b\omega_4\,$ (with 
$\,a\geq 1,\,b\geq 1\,$) be a good weight in $\,{\cal X}_{++}(V)$. Then
$\,\mu_1=\mu-(\alpha_3+\alpha_4)
=\,\omega_2\,+\,(a-1)\omega_3\,+\,b\omega_4\in{\cal X}_{++}(V)$.
For $\,a\geq 2$, $\,\mu_1\,$ has $3$ nonzero coefficients, hence 
Lemma~\ref{3coefbn} applies. For $\,a=1\,$, $\,\mu_1\,$ satisfies case 
(b) of this proof. Hence $\,V\,$ is not exceptional.

(d) Let $\,\mu\,=\,a\omega_{1}\,+\,b\omega_3\,$ (with 
$\,a\geq 1,\,b\geq 1\,$) be a good weight in $\,{\cal X}_{++}(V)$. 
Then $\,\mu_1\,=\,\mu-(\alpha_3+\alpha_4)\,=\, 
a\omega_1\,+\,\omega_2\,+\,(b-1)\omega_3\,\in\,{\cal X}_{++}(V)$. 
For $\,b\geq 2,\,$ $\,\mu_1\,$ has $\,3\,$ nonzero coefficients and
Lemma~\ref{3coefbn} applies. For $\,b=1$, $\,s(V)\,\geq\,|W\mu|\,+\, 
|W\mu_1|\,\geq\,\displaystyle\frac{2^4\,4!}{2!\,2!}\,+\,\frac{2^4\,4!} 
{2^2\,2!}\,=\,2^4\cdot 3^2\,>\,2^7\,$. Hence,
by~(\ref{bllim}), $\,V\,$ is not an exceptional module. 

(e) Let $\,\mu\,=\,a\,\omega_{1}\,+\,b\,\omega_2\,$ (with 
$\,a\geq 1,\,b\geq 1\,$) be a good weight in $\,{\cal X}_{++}(V)$.
Then $\,\mu_1\,=\,\mu-(\alpha_1+\alpha_2)\,
=\,(a-1)\omega_1\,+\,(b-1)\omega_2\,+\,\omega_3\,\in\,{\cal X}_{++}(V)$. 

For $\,a\geq 2,\,b\geq 2,\,$ $\,\mu_1\,$ is a good weight with 3
nonzero coefficients and Lemma~\ref{3coefbn} applies.
For $\,a=1,\,b\geq 2,\,$, $\,\mu_1\,$ satisfies case (a).
For $\,a\geq 2,\,b= 1,\,$ $\,\mu_1\,$ satisfies case (d) of this proof. 

For $\,a=b=1$, $\,\mu\,=\,\omega_{1}\,+\,\omega_2\,$. Then $\,\mu_1\,=\, 
\omega_3,\;$ $\,\mu_2\,=\mu_1-(\alpha_3+\alpha_4)=\,\omega_2,\;$
$\,\mu_3\,=\mu-(\alpha_2+\alpha_3+\alpha_4)=\,2\omega_1,\;$ and
$\,\mu_4\,=\mu_3-(\alpha_1+\cdots+\alpha_4)=\,\omega_1\,\in{\cal X}_{++}(V)$. 
For $\,p\geq 3,\,$ $\,|R_{long}^+\,-\,R_{\mu,p}^+|=10,\;
|R_{long}^+\,-\,R_{\mu_1,p}^+|=4,\;|R_{long}^+\,-\,R_{\mu_2,p}^+|=9,
\;|R_{long}^+\,-\,R_{\mu_3,p}^+|=5,\;|R_{long}^+\,-\,R_{\mu_4,p}^+|
=6$. Thus $\,r_p(V) \,=\,\frac{2^4\cdot
3\cdot 10}{2^3\cdot 3}\,+\,\frac{2^5\cdot 9}{2^3\cdot 3}
\,+\,\frac{2^3\cdot 3\cdot 9}{2^3\cdot 3}\,+\,
\frac{2^3\cdot 3\cdot 6}{2^3\cdot 3}\,=\,47\,>\,32=2\cdot 4^2\,$. Hence,
by~(\ref{rrbl}), $\,V\,$ is not an exceptional module.

(f) Let $\,\mu\,=\,a\omega_{1}\,+\,b\omega_4\,\in\,{\cal X}_{++}(V)$. Then

i) for $\,a\geq 2,\,b\geq 1,\,$ $\,\mu_1\,=\,\mu-\alpha_1\,=\,(a-2)\omega_1\, 
+\,\omega_2+\,b\omega_4\,\in\,{\cal X}_{++}(V)$. For $\,a\geq 3$, $\,\mu_1\,$
has $\,3\,$ nonzero coefficients and Lemma~\ref{3coefbn} applies.
For $\,a=2,\,$ $\,\mu_1\,$ satisfies case (b) of this proof. 
Hence $\,V\,$ is not exceptional.

ii) For $\,a\geq 1,\,b\geq 2,\,$ $\,\mu_1\,=\,\mu-\alpha_4\,=\,a\, 
\omega_1\,+\,\omega_3+\,(b-2)\omega_4\,\in\,{\cal X}_{++}(V)$.  
For $\,b\geq 3$, $\,\mu_1\,$ has $\,3\,$ nonzero coefficients and  
Lemma~\ref{3coefbn} applies. For $\,b=2,\,$ $\,\mu_1\,$ satisfies case
(d) of this proof. Hence $\,V\,$ is not exceptional. This proves
Claim~1. \hfill $\Box$ \vspace{1.5ex}

{\bf Claim 2}: \label{cl2b4}
{\it Let $\,V\,$ be a $\,B_4(K)$-module such that $\,{\cal X}_{++}(V)\,$ 
contains 

(a) $\,\mu\,=\,a\omega_1\,$ or $\,\mu\,=\,a\omega_4\,$  
with $\,a\geq 3\,$ \; or 

(b) $\,\mu\,=\,a\omega_2\,$ or $\,\mu\,=\,a\omega_3\,$ 
with $\,a\geq 2\,$. \\
Then $\,V\,$ is not an exceptional $\,\mathfrak g$-module.}

Indeed:

(a) Let $\,a\geq 3$. If $\,\mu\,=\,a\,\omega_{4}\in{\cal X}_{++}(V)$, 
then $\,\mu_1\,=\,\mu-\alpha_4\,=\,\omega_3\,+\,(a-2)\omega_4
\,\in\,{\cal X}_{++}(V)$. If $\,\mu=\,a\,\omega_{1}\,$, then 
$\,\mu_1\,=\,\mu-\alpha_1\,=\,(a-2)\omega_1\,+\,\omega_2\,\in\,
{\cal X}_{++}(V)$. In both cases, $\,\mu_1\,$ satisfies Claim~1
(p. \pageref{cl1b4}). Hence $\,V\,$ is not exceptional. 
 
(b) Let $\,a\geq 2$. If $\,\mu\,=\,a\,\omega_{3}\in{\cal X}_{++}(V)$, 
then $\,\mu_1\,=\,\mu-(\alpha_3+\alpha_4)\,=\,\omega_2\,+\,(a-1)\omega_3
\,\in\,{\cal X}_{++}(V)\,$ satisfies Claim~1.
If $\,\mu\,=\,a\,\omega_{2}\in{\cal X}_{++}(V)$, then
$\,\mu_1\,=\,\mu-\alpha_2\,=\,\omega_1\,+\,(a-2)\,\omega_2\,+\,\omega_3
\,\in\,{\cal X}_{++}(V)$. In this case, for $\,a\geq 3,\,$ $\,\mu_1\,$
has $\,3\,$ nonzero coefficients and Lemma~\ref{3coefbn} applies. 
For $\,a=2,\,$ $\,\mu_1\,$ satisfies Claim~1. 
Hence $\,V\,$ is not exceptional. This proves Claim~2.

Now we deal with a particular case.

P.1) Let $\,V\,$ be a $\,B_4(K)$-module of highest weight 
$\,\mu\,=\,\omega_{1}\,+\,\omega_4\,$. Then $\,\mu_1\,=\,\mu-(\alpha_1+ 
\cdots+\alpha_4)\,=\,\omega_4\,\in\,{\cal X}_{++}(V)$.
By \cite[p. 168]{buwil}, $\,m_{\mu_1}\,=\,3\,$ for $\,p=3,\,$ and
$\,m_{\mu_1}\,=\,4\,$ for $\,p\neq 3.\,$ For $\,p\geq 3,\,$ 
$\,|R_{long}^+\,-\,R_{\mu,p}^+|\,=\,9,\;|R_{long}^+\,-\,R_{\mu_1,p}^+|\,=\,6$.
Thus $\,r_p(V)\,\geq\,\displaystyle\frac{2^6\cdot 9}{2^3\cdot 3}\,+\,3\cdot
\frac{2^4\cdot 6}{2^3\cdot 3}\,=\,36\,>\,32=2\cdot 4^2\,$. Hence, 
by~(\ref{rrbl}), $\,V\,$ is not exceptional.

{\bf Hence, if $\,V\,$ is a $\,B_4(K)$-module such that
$\,{\cal X}_{++}(V)\,$ contains a weight with $2$ nonzero coefficients, 
then $\,V\,$ is not exceptional.}

{\bf From now on we can assume that $\,{\cal X}_{++}(V)\,$ contains
only weights with at most one nonzero coefficient.}

P.2) If $\,V\,$ is a $\,B_4(K)$-module of highest weight 
$\,\mu\,=\,2\,\omega_{1}\,$ then, by case III(b)iv) 
(p. \pageref{IIIbiv}) of First Part of Proof
of Theorem~\ref{listbn}, $\,V\,$ is not exceptional.

U.1) {\bf If $\,V\,$ is a $\,B_4(K)$-module of highest weight 
$\,\mu\,=\,2\,\omega_{4}\,$, then $\,V\,$ is unclassified 
(N. 4 in Table~\ref{leftbn}).}

P.3) If $\,V\,$ is a $\,B_4(K)$-module of highest weight 
$\,\mu\,=\,\omega_{4}\,$, then $\,\dim\,V\,=\,2^4\,<\,36-\varepsilon\,$. 
Hence, by Proposition~\ref{dimcrit}, $\,V\,$ is exceptional 
{\bf (N. 3 in Table~\ref{tableblall})}.

P.4) If $\,V\,$ is a $\,B_4(K)$-module of highest weight
$\,\mu\,=\,\omega_{2}\,$, then $\,V\,$ is the adjoint module, 
which is exceptional by 
Example~\ref{adjoint} {\bf (N. 2 in Table~\ref{tableblall})}.

P.5) If $\,V\,$ is a $\,B_4(K)$-module of highest weight 
$\,\mu\,=\,\omega_{1}\,$ then, by case III(c)iv) (p. \pageref{IIIciv})
of First Part of the Proof of Theorem~\ref{listbn}, $\,V\,$ is an 
exceptional module {\bf (N. 1 in Table~\ref{tableblall})}.

U.2) {\bf If $\,V\,$ is a $\,B_4(K)$-module of highest weight 
$\,\mu\,=\,\omega_{3},\,$ then $\,V\,$ is unclassified 
(N. 2 in Table~\ref{leftbn}).}

{\bf Hence if $\,{\cal X}_{++}(V)\,$ contains $\,\mu\,=\,a\omega_i\,$ 
(for $\,a\geq 2\,$ and $\,i=1,\,2,\,3\,$) or $\,\mu\,=\,a\omega_4\,$ 
(for $\,a\geq 3\,$), then $\,V\,$ is not exceptional. If $\,V\,$ has 
highest weight $\,\lambda\,\in\,\{\omega_1,\,\omega_2,\,\omega_4\}\,$, 
then $\,V\,$ is exceptional. If $\,V\,$ has highest weight $\,2\omega_4\,$ or
$\,\omega_3\,$, then $\,V\,$ is unclassified.}

This finishes the proof of Theorem~\ref{listbn} for $\,\ell=4$.
\hfill $\Box$

%\newpage
%\input{cncncn.tex}

\subsection{Groups or Lie Algebras of type $\,C_{\ell}$}\label{appcl}
In this section we prove Theorem~\ref{listcn}. Recall that for
groups of type $\,C_{\ell}\,$, $\,p\geq 3$, 
$\,|W|\,=\,2^{\ell}\,\ell!\,$, $\,|R|\,=\,2\ell^2\,,\,|R_{long}|\, 
=\,2\ell\,$, $\,R_{long}\,=\,\{ 2\varepsilon_i\,/\,1\leq i\leq \ell \}$.

\subsubsection{Type $\,C_{\ell}\,,\;\ell\geq 6$}\label{fppcl}

%\begin{lemma}\label{3coefcn}

{\bf Proof of Theorem~\ref{listcn} - First Part}\\
Let $\,\ell\geq 6$. 
By Lemma~\ref{3coefcn}, {\bf it suffices to consider $\,C_{\ell}(K)$-modules 
$\,V\,$ such that $\,{\cal X}_{++}(V)\,$ contains only weights with 2 or
less nonzero coefficients.}

First suppose that $\,{\cal X}_{++}(V)\,$ contains weights with 
$\,2\,$ nonzero coefficients. 
%Let $\,\mu\,=\,a\,\omega_{i}\,+\,b\,\omega_{j}\,$ with $\,1\leq i<j\leq\ell\,$
%and $\,a,\,b\,\geq 1\,$ be a weight in $\,{\cal X}_{++}(V)$.

I) {\bf Claim 1}:\label{claim1cl}
{\it For $\,\ell\geq 6,\,$ if $\,{\cal X}_{++}(V)\,$
contains a bad weight with $\,2\,$ nonzero coefficients, then $\,V\,$
is not an exceptional module.}

Indeed, let $\,\mu\,=\,a\,\omega_{i}\,+\,b\,\omega_{j}\,$ with 
$\,1\leq i<j \leq\ell\,$ and $\,a,\,b\,\geq 1\,$ be a bad weight in 
$\,{\cal X}_{++}(V)$. Hence $\,a\geq 3,\,b\,\geq 3$.

(a) For $\,1\leq i<j \leq\ell-1,\,$ $\,\mu\,=\,a\omega_{i}\,+\,b\omega_{j},\;$
$\,\mu_1\,=\,\mu-(\alpha_i+\cdots +\alpha_j)\,=\,\omega_{i-1}\,+\,(a-1)
\,\omega_{i}\,+\,(b-1)\,\omega_{j}\,+\,\omega_{j+1}\in\,{\cal X}_{++}(V)$.

(b) For $\,1\,<\, i\,<\,j =\ell,\,$ $\,\mu\,=\,a\,\omega_{i}\,+\,b\,
\omega_{\ell}\;$ and $\,\mu_1\,=\,\mu-(\alpha_i+\cdots +\alpha_{\ell})\, 
=\,\omega_{i-1}\,+\,(a-1)\,\omega_{i}\,+\,\omega_{\ell-1}\,+\,(b-1) 
\omega_{\ell}\,\in\,{\cal X}_{++}(V)\,$.

(c) For $\,1= i,\,j =\ell,\,$ $\,\mu\,=\,a\,\omega_{1}\,+\,b\, 
\omega_{\ell}\;$ and $\,\mu_1\,=\,\mu-\alpha_1\,=\,(a-2)\omega_{1}\,+\,
\,\omega_{2}\,+\,b\,\omega_{\ell}\,\in\,{\cal X}_{++}(V)\,$.\\
In all these cases $\,\mu_1\,$ is a good weight with $\,3\,$ or more nonzero
coefficients. Hence, by Lemma~\ref{3coefcn}, $\,V\,$ is not an
exceptional module. This proves the claim. \hfill $\Box$ \vspace{1.5ex}

{\bf Therefore, for $\,\ell\geq 6,\,$ if $\,{\cal X}_{++}(V)\,$
contains weights with $\,2\,$ nonzero coefficients, then we can assume
that these are good weights.} We deal with these cases in the
sequel. First we prove two claims.

{\bf Claim 2}:\label{claim2cl} {\it Let $\,\ell\geq 6$. If $\,V\,$ is
a $\,C_{\ell}(K)$-module such that $\,\omega_{\ell-1}\,\in\, 
{\cal X}_{++}(V)$, then $\,V\,$ is not an exceptional $\,\mathfrak g$-module.}

Indeed, $\,|W\omega_{\ell-1}|\,=\,2^{\ell-1}\,\ell\,$ and, for $\,p\geq 3$,
$\,|R_{long}^+-R^+_{\omega_{\ell-1},p}|\,=\,(\ell-1)$.
Thus, for $\,\ell\geq 6$, $\,r_p(V)\,\geq\,\displaystyle
\frac{2^{\ell-1}\,\ell\cdot (\ell-1)}{2\ell}\,=\,2^{\ell-2}(\ell-1)\,>
\,2\ell^2\,$. Hence, by~\eqref{rrcl}, $\,V\,$ is not an exceptional
module, proving the claim. \hfill $\Box$ \vspace{1.5ex}

{\bf Claim 3}: \label{claim3cl} {\it Let $\,\ell\geq 6$. If $\,V\,$ is 
a $\,C_{\ell}(K)$-module such that $\,\omega_{\ell}\,\in\,
{\cal X}_{++}(V)$, then $\,V\,$ is not an exceptional $\,\mathfrak g$-module.}

Indeed, if $\,\omega_{\ell}\in{\cal X}_{++}(V)$, then 
$\,\mu_1\,=\,\mu-(\alpha_{\ell-1}+\alpha_{\ell})\,=\,\omega_{\ell-2}\in 
{\cal X}_{++}(V)$. For $\,p\geq 3$, $\,|R_{long}^+-R^+_{\omega_{\ell},p}|=
\ell\,$ and $\,|R_{long}^+-R^+_{\omega_{\ell-2},p}|=(\ell-2)$. Thus
for $\,\ell\geq 6$, $\,r_p(V)\,\geq\,\displaystyle\frac{2^{\ell}\cdot 
\ell}{2\ell}+\frac{2^{\ell-3}\ell(\ell -1)\cdot (\ell-2)}{2\ell}\,=\,2^{\ell-4}
[8+(\ell-1)(\ell-2)]\,>\,2\ell^2\,$. Hence, by~\eqref{rrcl}, $\,V\,$ is not
exceptional, proving the claim. \hfill $\Box$ \vspace{1.5ex}

II) Now let $\,\mu\,=\,a\,\omega_{i}\,+\,b\,\omega_{j}\,$ (with $\,1\leq
i<j \leq\ell\,$ and $\,a,\,b\,\geq 1\,$) be a good weight in $\,{\cal 
X}_{++}(V)$. First we deal with some particular cases. 

{\bf Claim 4}: \label{claim4cl}
{\it Let $\,\ell\geq 6$. If $\,V\,$ is 
a $\,C_{\ell}(K)$-module such that $\,{\cal X}_{++}(V)\,$ contains a
weight $\,\mu\,=\,a\omega_1\,+\,b\omega_{\ell}\,$ (with $\,a\geq 1,\,
b\geq 1\,$), then $\,V\,$ is not an exceptional $\,\mathfrak g$-module.}

Indeed, for $\,a\geq 2,\,b\geq 2$, $\,\mu_1\,=\mu-(\alpha_1+\cdots+
\alpha_{\ell})=\,(a-1)\omega_1\,+\,\omega_{\ell-1}\,+\,(b-1)\omega_{\ell}
\in{\cal X}_{++}(V)\,$ has $3$ nonzero coefficients. For $\,a=1,\,b\geq 3$,  
$\,\mu\,=\,\omega_1\,+\,b\omega_{\ell}\,$. Then $\,\mu_2\,=\mu-\alpha_{\ell}
=\,\omega_1\,+\,\omega_{\ell-1}\,+\,(b-2)\omega_{\ell}\,\in{\cal
X}_{++}(V)\,$ has $3$ nonzero coefficients. 
In both cases Lemma~\ref{3coefcn} applies.

For $\,a\geq 2,\,b=1\,$, $\,\mu\,=\,a\omega_1\,+\,\omega_{\ell}\,$ and
$\,\mu_1=\,(a-1)\omega_1\,+\,\omega_{\ell-1}\in{\cal X}_{++}(V)$. 
For $\,a=1,\,b=2\,$, $\,\mu\,=\,\omega_1\,+\,2\omega_{\ell}\,$ and 
$\,\mu_2=\,\omega_1\,+\,\omega_{\ell-1}\in{\cal X}_{++}(V)$. 
In both cases, for $\,\ell\geq 6$, 
$\,s(V)\,\geq\,2^{\ell}\ell\,+\,2^{\ell-1}\ell(\ell-1)\,
=\,2^{\ell-1}\ell(\ell+1)\,>\,4\ell^3\,$. Hence, by~\eqref{bllim}, $\,V\,$ is
not exceptional.
For $\,a=b=1$, $\,\mu\,=\,\omega_1\,+\,\omega_{\ell}\,$ and 
$\,\mu_1=\,\omega_{\ell-1}\in{\cal X}_{++}(V)$. Hence, by Claim 2, $\,V\,$ 
is not exceptional. This proves Claim 4.\hfill $\Box$ \vspace{1.5ex}

{\bf Claim 5}:\label{claim5cl} 
{\it Let $\,\ell\geq 6$. If $\,V\,$ is 
a $\,C_{\ell}(K)$-module such that $\,{\cal X}_{++}(V)\,$ contains a
weight $\,\mu\,=\,a\omega_1\,+\,b\omega_{\ell-1}\,$ (with $\,a\geq 1,\,
b\geq 1\,$), then $\,V\,$ is not an exceptional $\,\mathfrak g$-module.}

Indeed, for $\,a\geq 1,\,b\geq 2$, $\,\mu_1=\mu-\alpha_{\ell-1}=
\,a\omega_1\,+\,\omega_{\ell-2}\,+\,(b-2)\omega_{\ell-1}\,+\,\omega_{\ell}
\in{\cal X}_{++}(V)\,$. Hence Lemma~\ref{3coefcn} applies.

For $\,a\geq 1,\,b=1\,$, $\,\mu_2\,=\mu-(\alpha_1+\cdots+\alpha_{\ell-1})=\,
(a-1)\omega_1\,+\,\omega_{\ell}\in{\cal X}_{++}(V)$. For $\,a\geq 2$, 
$\,\mu_2\,$ satisfies Claim 4. For $\,a=1$, $\,\mu_2\,$ satisfies Claim 3.
In any case $\,V\,$ is not exceptional, proving Claim 5.\hfill $\Box$ 
\vspace{1.5ex}

{\bf Claim 6}: \label{claim6cl}
{\it Let $\,\ell\geq 6$. If $\,V\,$ is 
a $\,C_{\ell}(K)$-module such that $\,{\cal X}_{++}(V)\,$ contains a
weight $\,\mu\,=\,a\omega_i\,+\,b\omega_{\ell-1}\,$ (with $\,2\leq i
\leq\ell-2\,$ and $\,a\geq 1,\,b\geq 1\,$), then $\,V\,$ is not an 
exceptional $\,\mathfrak g$-module.}

Indeed, $\,\mu_1\,=\mu-(\alpha_i+\cdots+\alpha_{\ell-1})\,=\,\omega_{i-1}\,
+\,(a-1)\,\omega_{i}\,+\,(b-1)\omega_{\ell-1}\,+\,\omega_{\ell}\in 
{\cal X}_{++}(V)$. For $\,2\leq i\leq\ell-2\,$ ($\,a\geq 2,\,b\geq 1\,$) or  
($\,a\geq 1,\,b\geq 2\,$), $\,\mu_1\,$ has $3$ nonzero coefficients. Hence
Lemma~\ref{3coefcn} applies. For $\,a=b=1\,$ and $\,\ell\geq 6$, 
$\,s(V)\,\geq\,\displaystyle 2^{\ell-1}\ell\binom{\ell-1}{i}\,+\,
2^{\ell}\binom{\ell}{i-1}\,\geq\,2^{\ell-1}\ell(\ell-1) + 2^{\ell}\ell\,
=\,2^{\ell-1}\ell(\ell+1)\,>\,4\ell^3\,$. Hence, by~\eqref{bllim}, $\,V\,$ is
not exceptional. This proves Claim 6.\hfill $\Box$ \vspace{1.5ex}

{\bf Claim 7}: \label{claim7cl} {\it Let $\,\ell\geq 6$. If $\,V\,$ is 
a $\,C_{\ell}(K)$-module such that $\,{\cal X}_{++}(V)\,$ contains a
weight $\,\mu\,=\,a\omega_i\,+\,b\omega_{\ell}\,$ (with $\,2\leq i
\leq\ell-1\,$ and $\,a\geq 1,\,b\geq 1\,$), then $\,V\,$ is not an 
exceptional $\,\mathfrak g$-module.}

Indeed, $\,\mu_1\,=\,\mu-(\alpha_i+\cdots+\alpha_{\ell})\,=\,
\omega_{i-1}\,+\,(a-1)\,\omega_{i}\,+\,\omega_{\ell-1}\,+\,(b-1)
\omega_{\ell}\in {\cal X}_{++}(V)$. For ($\,a\geq 2,\,b\geq 1\,$) or  
($\,a\geq 1,\,b\geq 2\,$), $\,\mu_1\,$ has $3$ or
more nonzero coefficients, hence Lemma~\ref{3coefcn} applies.
For $\,a=b=1\,$ and $\,3\leq i\leq \ell-1\,$, $\,\mu_1\,$ satisfies Claim 6.
For $\,i=2\,$, $\,\mu_1\,$ satisfies Claim 5.
In any case $\,V\,$ is not an exceptional module, proving Claim 7. 
\hfill $\Box$ \vspace{1.5ex}

{\bf Therefore, if $\,{\cal X}_{++}(V)\,$ contains a good weight 
$\,\mu\,=\,a\,\omega_{i}\,+\,b\,\omega_{j}\,$ (with \mbox{$\,1\leq
i<j \leq\ell\,$} and $\,a,\,b\,\geq 1\,$), then we can assume that 
$\,1\leq i< j\leq\ell-2\,$.} These cases are treated as follows.

(a) Let $\,1\leq i < j=\ell-2\,$ and $\,\mu\,=\,a\omega_{i}\,+\,b
\omega_{\ell-2}\,$ with $\,a,\,b\,\geq 1\,$. Then for $\,\ell\geq 6$,
$\,s(V)\,\geq\,|W\mu|\,=\,2^{\ell-3}\ell(\ell-1)\binom{\ell-2}{i}\geq
2^{\ell-3}\ell(\ell-1)(\ell-2)\,>\,4\ell^3\,$. 
%Hence, by~\eqref{bllim}, $\,V\,$ is not exceptional.   

(b) Let $\,1\leq i < j=\ell-3\,$ and $\,\mu\,=\,a\omega_{i}\,+\,b
\omega_{\ell-3}\,$ with $\,a,\,b\,\geq 1\,$. Then for $\,\ell\geq 6$,
$\,s(V)\,\geq\,|W\mu|\,=\,2^{\ell-3}\ell(\ell-1)(\ell-2)\binom{\ell-3}{i}\geq
2^{\ell-3}\ell(\ell-1)(\ell-2)\,>\,4\ell^3\,$. 
%Hence, by~\eqref{bllim}, $\,V\,$ is not exceptional.   

(c) Let $\,1\leq i < j,\;4\leq j\leq\ell-4\,$ (hence $\,\ell\geq 8\,$) and 
$\,\mu\,=\,a\omega_{i}\,+\,b\omega_{j}\,$ with $\,a,\,b\,\geq 1\,$. Then 
for $\,\ell\geq 8$, $\,s(V)\,\geq\,|W\mu|\,=\,\displaystyle
2^j\,\binom{j}{i}\,\binom{\ell}{j}\,\geq\,2^4\cdot 4\binom{\ell}{4}\,\geq\, 
\frac{2^3\,\ell(\ell-1)(\ell-2)(\ell-3)}{3}\,>\,4\ell^3\,$.

Hence in all these cases, by~\eqref{bllim}, $\,V\,$ is not an
exceptional module.   

{\bf Hence if $\,{\cal X}_{++}(V)\,$ contains a weight 
$\,\mu\,=\,a\,\omega_{i}\,+\,b\,\omega_{j}\,$ (with \mbox{$\,1\leq
i<j$,} \mbox{$\,4\leq j \leq\ell\,$} and $\,a,\,b\,\geq 1\,$), then $\,V\,$ is
not an exceptional module. Therefore we can assume that 
$\,1\leq i< j\leq 3\,$.} These cases are treated in the sequel.
 
(d) Let $\,i=1,\,j=3\,$ and $\,\mu\,=\,a\omega_{1}\,+\,b\omega_{3}\,$ 
with $\,a,\,b\,\geq 1\,$. Then $\,\mu_1\,=\mu-(\alpha_1+\cdots+\alpha_{3})\,
=\,(a-1)\omega_1\,+\,(b-1)\omega_3\,+\,\omega_4\in{\cal X}_{++}(V)$. 
For $\,a\geq 2,\,b\geq 2\,$, $\,\mu_1\,$ has $3$ nonzero coefficients. Hence
Lemma~\ref{3coefcn} applies. For $\,a\geq 2,\,b=1\,$ or $\,a=1,\,b\geq 2\,$, 
$\,\mu_1\,$ satisfies case II(c). For $\,a=b=1$,
$\,\mu\,=\,\omega_{1}\,+\,\omega_{3}\,$ and 
$\,|R_{long}^+-R^+_{\mu,p}|\,=\,3\,$. Thus for $\,\ell\geq 4\,$,
$\,r_p(V)\,\geq\,\frac{2^2\,\ell(\ell-1)(\ell-2)\cdot 3}{2\ell}\,=\,
6(\ell-1)(\ell-2)\,>\,2\ell^2\,$. Hence, by~\eqref{rrcl}, 
$\,V\,$ is not exceptional.

(e) Let $\,i=2,\,j=3\,$ and $\,\mu\,=\,a\omega_{2}\,+\,b\omega_{3}\,$ 
with $\,a,\,b\,\geq 1\,$. Then $\,\mu_1\,=\mu-(\alpha_2+\alpha_{3})\,
=\,\omega_1\,+\,(a-1)\omega_2\,+\,(b-1)\omega_3\,+\,\omega_4\in
{\cal X}_{++}(V)$. For ($\,a\geq 2,\,b\geq 1\,$) or ($\,a\geq 1,\,b\geq 2\,$), 
$\,\mu_1\,$ has $3$ or more nonzero coefficients. Hence
Lemma~\ref{3coefcn} applies. For $\,a=b=1\,$, $\,\mu_1\,$ satisfies case 
II(c). Hence $\,V\,$ is not exceptional.

(f)\label{IIfi} Let $\,i=1,\,j=2\,$ and $\,\mu\,=\,a\omega_{1}\,+\, 
b\omega_{2}\,$ with $\,a,\,b\,\geq 1\,$. Then $\,\mu_1\,=\mu-(\alpha_1
+\alpha_2)\,=\,(a-1)\omega_1\,+\,(b-1)\omega_2\,+\,\omega_3\in
{\cal X}_{++}(V)$. 

i) For $\,a\geq 2,\,b\geq 2\,$, $\,\mu_1\,$ has $3$ 
nonzero coefficients, hence Lemma~\ref{3coefcn} applies.
For $\,a= 1,\,b\geq 2\,$, $\,\mu_1\,$ satisfies case II(e).
For $\,a\geq 2,\,b=1\,$, $\,\mu_1\,$ satisfies case II(d). 
In all these cases $\,V\,$ is not an exceptional module.

ii) For $\,a=b=1\,$, we may assume that $\,V\,$ has highest weight %(!!!!)
$\,\mu\,=\,\omega_{1}\,+\,\omega_{2}\,$. Then $\,\mu_1=\,\omega_3,\;$
$\,\mu_2=\mu_1-(\alpha_3+2\alpha_4+\cdots+2\alpha_{\ell-1}+\alpha_{\ell})\,
=\,\omega_1\,\in\,{\cal X}_{++}(V)$.
Now $\,|R_{long}^+-R^+_{\mu,p}|\,=\,2,\,$ 
$\,|R_{long}^+-R^+_{\mu_1,p}|\,=\,3,\,$
$\,|R_{long}^+-R^+_{\mu_2,p}|\,=\,1\,$. By Smith's Theorem~\cite{smith},
$\,m_{\mu_1}\,=\,2\,$ for $\,p\neq 3$. Thus for $\,p\neq 3\,$ and
$\,\ell\geq 4$, %%reduce to A(2) to get adjoint module and mult 2
\[
\begin{array}{lcl}
r_p(V) & \geq & \displaystyle
\frac{2^2\,{\ell}\,(\ell-1)\cdot 2}{2\,\ell}\,+\,
2\,\frac{2^2\,{\ell}\,(\ell-1)\,(\ell-2)}{3\cdot 2\,\ell}\,3\,+\,
\frac{2\,{\ell}}{2\,\ell}\,=\,
\frac{4\ell^2\,-\,8\ell\,+\,5}{3}\,>\,2\,\ell^2\,.
\end{array}
\]
Hence, by~(\ref{rrcl}), $\,V\,$ is not an exceptional module. 
For $\,p=3,\,$ as $\,\mathfrak{sp}_n\,$ is a Lie subalgebra of
$\,\mathfrak{sl}_{2n}\,$, the $\,\mathfrak{sl}_{2n}$-module
of highest weight $\,\omega_1\,+\,\omega_2\,$ can be considered as 
a $\,\mathfrak{sp}_n$-module (also of highest weight 
$\,\omega_1\,+\,\omega_2\,$). Hence, by Claim~12
(p. \pageref{claim12al}) of First Part of Proof of 
Theorem~\ref{anlist}, $\,V\,$ is not exceptional.

{\bf Hence, for $\,\ell\geq 6$, if $\,{\cal X}_{++}(V)\,$ contains a weight
$\,\mu\,=\,a\omega_{i}\,+\,b\omega_{j}\,$ (with $\,a\geq 1,\,b\geq
1\,$), then $\,V\,$ is not an exceptional module.}

{\bf From now on we can assume that $\,{\cal X}_{++}(V)\,$ contains
only weights with at most one nonzero coefficient.}

III) Let $\,\mu\,=\,a\,\omega_i\,$ (with $\,a\geq 1\,$) be a weight in
$\,{\cal X}_{++}(V)$.

(a)i) Let $\,2\,\leq\,i\,\leq\,\ell -1\,$, $\,a\geq 2\,$ and 
$\,\mu\,=\,a\,\omega_i\,$. Then $\,\mu_1\,=\,\mu-\alpha_i\,=\,
\omega_{i-1}\,+\,(a-2)\omega_{i}\,+\,\omega_{i+1}\,\in\,{\cal X}_{++}(V)$.

For $\,a\geq 3\,$, $\,\mu_1\,$ is a good weight with $3$ nonzero
coefficients, hence Lemma~\ref{3coefcn} applies (for $\,\mu\,$ bad or good).
For $\,a=2,\,$ $\,\mu\,=\,2\,\omega_i\;$ and $\,\mu_1\,=\,\mu-\alpha_i\, 
=\,\omega_{i-1}\,+\,\omega_{i+1}\in {\cal X}_{++}(V)$. Hence, by Part
II of this proof, $\,V\,$ is not exceptional.

ii) Let $\,i=\ell\,$, $\,a\geq 2\,$ and $\,\mu\,=\,a\,\omega_{\ell}\,$.
Then $\,\mu_1\,=\,\mu-\alpha_{\ell}\,=\,2\omega_{\ell-1}\,+\,(a-2) 
\omega_{\ell}\,$ and $\,\mu_2\,=\,\mu_1-\alpha_{\ell-1}\,=\,
\omega_{\ell-2}\,+\,(a-1)\omega_{\ell}\,\in\,{\cal X}_{++}(V)$. 
As $\,\mu_2\,$ satisfies Claim 7 (p. \pageref{claim7cl}),
$\,V\,$ is not exceptional.

iii) Let $\,i=1\,$, $\,a\geq 2\,$ and $\,\mu\,=\,a\,\omega_{1}\,$.
Then $\,\mu_1\,=\,\mu-\alpha_1\,=\,(a-2)\omega_1\,+\,\omega_2\,
\in\,{\cal X}_{++}(V)$. For $\,a\geq 4$, $\,\mu_1\,$ satisfies case
II(f)i) (p. \pageref{IIfi}) of this proof. 

For $\,a=3\,$ and $\,p\geq 5$, we may assume that $\,V\,$ has highest weight
$\,\mu\,=\,3\omega_1\,$. By Lemma~\ref{III.iii}
(p. \pageref{III.iii}), $\,V\,$ is not
exceptional as an $\,\mathfrak{sl}_{2n}$-module. Hence $\,V\,$ is not
exceptional as an $\,\mathfrak{sp}_{n}$-module, since 
$\,\mathfrak{sp}_n\,$ is a Lie subalgebra of $\,\mathfrak{sl}_{2n}\,$.
For $\,p=3,\,$ $\,\mu\,\in\,{\cal X}_{++}(V)\,$ only if $\,V\,$ has 
highest weight $\,\lambda\,=\,2\omega_1\,+\,\omega_3\,$ or 
$\,\lambda\,=\,\omega_1\,+\,2\omega_2\,$ (for instance). In these
cases $\,V\,$ is not exceptional by part II of this proof.

Let $\,a=2$. We may assume that $\,V\,$ has highest weight $\,\mu\,=\,
2\omega_1= \tilde{\alpha}$. Then $\,V\,$ is the adjoint module, which is
exceptional by Example~\ref{adjoint} {\bf (N. 1 in Table~\ref{tableclall})}.

{\bf Therefore, if $\,{\cal X}_{++}(V)\,$ contains a weight  
$\,\mu\,=\,a\,\omega_i\,$ (with $\,a\geq 1\,$), then we can assume $\,a=1\,$}. 

(b) Let $\,\mu\,=\,\omega_{i}\,$ be a weight in $\,{\cal X}_{++}(V)\,$.

i) For $\,i=\ell-1\,$, see Claim~2 (p. \pageref{claim2cl}). 
For $\,i=\ell$, see Claim~3 (p. \pageref{claim3cl}). {\bf
Therefore we can assume that $\,1\leq i\leq \ell-2\,$.}

ii) For $\,i=\ell-2\,$, $\,\mu\,=\,\omega_{\ell-2}\,$ and 
$\,\mu_1\,=\,\mu-(\alpha_{\ell-3}+2\alpha_{\ell-2}+2\alpha_{\ell-1}
+\alpha_{\ell})\,=\,\omega_{\ell-4}\,\in\,{\cal X}_{++}(V)$. As
$\,|R_{long}^+-R^+_{\mu,p}|=\ell-2,\;|R_{long}^+-R^+_{\mu_1,p}|=\ell-4\,$
one has, for $\,\ell\geq 6$,
\[
\begin{array}{rcl}
r_p(V) & \geq &\displaystyle\frac{2^{\ell-3}\ell(\ell-1)\cdot
(\ell-2)}{2\ell}\,+\,\frac{2^{\ell-7}\ell(\ell-1)(\ell-2)(\ell-3)\cdot
(\ell-4)}{3\cdot 2\ell}\vspace{1ex}\\
& = & \displaystyle\frac{2^{\ell-8}\,(\ell-1)(\ell-2) 
(\ell^2-7\ell+60)}{3}\,>\,2\,\ell^2\,.
\end{array}
\]
Hence, by~\eqref{rrcl}, $\,V\,$ is not an exceptional module.

iii) Let $\,i=\ell-3\,$ and $\,\mu\,=\,\omega_{\ell-3}\,$. Then 
$\,|R_{long}^+-R^+_{\mu,p}|=\ell-3\,$ and $\,|W\mu|= \displaystyle 
\frac{2^{\ell-4}\,\ell\,(\ell-1)\,(\ell-2)}{3}\,$. Thus for
$\,\ell\geq 7$, 
\[
r_p(V)\,\geq\,\displaystyle\frac{2^{\ell-4}\,\ell(\ell-1) 
(\ell-2)\cdot (\ell-3)}{3\cdot 2\ell}\,=\,\frac{2^{\ell-5}\,(\ell-1)(\ell-2)
(\ell-3)}{3}\,>\,2\ell^2\,.
\]
Hence, by~\eqref{rrcl}, $\,V\,$ is not exceptional. {\bf For $\,\ell=6\,$  
if $\,V\,$ has highest weight $\,\omega_3\,$, then $\,V\,$ is unclassified 
(N. 2 in Table~\ref{leftcn})}.

iv) Let $\,4\leq i\leq\ell-4\,$ (hence $\,\ell\geq 8\,$) and 
$\,\mu\,=\,\omega_i\,$. Then $\,|W\mu|=\displaystyle 2^i\binom{\ell}{i}\,$
and $\,|R_{long}^+-R^+_{\mu,p}|=\ell-i\,$. Thus for $\,\ell\geq 8,\,$
$\,r_p(V)\,\geq\,\displaystyle 2^i\binom{\ell}{i}\frac{(\ell-i)}{2\ell}\,\geq 
\,\frac{2^4\,\ell(\ell-1)(\ell-2)\,(\ell-3)}{4!}\frac{4}{2\ell}\,=\,
\,\frac{2^2\,(\ell-1)(\ell-2)(\ell-3)}{3}\,>\,2\ell^2\,$. Hence, 
by~\eqref{rrcl}, $\,V\,$ is not an exceptional module.

v) Let $\,i=3$. {\bf If $\,V\,$ has highest weight $\,\mu\,=\,\omega_3\,$,
then $\,V\,$ is unclassified (N. 2 in Table~\ref{leftcn})}.

vi)\label{IIIbvi} Let $\,i=2\,$. We may assume that $\,V\,$ has highest weight
$\,\mu\,=\,\omega_2\,$. By~\cite[Theorem 2]{presup},
$\,\dim\,V\,=\,(2\ell+1)\,(\ell-1)\,-\,\nu\,
<\,2\ell^2 +\ell\,-\,\varepsilon\,$, for any $\,\ell\geq 2$. (Here
$\,\nu=1\,$ (if $\,p\nmid \ell\,$) and $\,\nu=2\,$ (if $\,p|\ell\,$).)
Thus, by Proposition~\ref{dimcrit}, $\,V\,$ is an exceptional module 
{\bf (N. 3 in Table~\ref{tableclall})}.

vii)\label{IIIbvii} Finally, we may assume that $\,V\,$ has highest
weight $\,\omega_1\,$. Then
by~\cite[Theorem 2]{presup}, $\,\dim\,V\,=\,2\ell\,<\,2\ell^2+\ell\,-\, 
\varepsilon\,$, for any $\,\ell\geq 2.\,$ Thus, by
Proposition~\ref{dimcrit}, $\,V\,$ is an exceptional module 
{\bf (N. 2 in Table~\ref{tableclall})}.

{\bf Hence, for $\,\ell\geq 6\,$, if $\,{\cal X}_{++}(V)\,$
contains a weight $\,\mu\,=\,a\,\omega_i\,$ with ($\,a\geq 3\,$ and
$\,1\leq i\leq \ell\,$) or ($\,a=2\,$ and $\,2\leq i\leq \ell\,$) or
($\,a=1\,$ and $\,4\leq i\leq \ell\,$), then $\,V\,$ is not an
exceptional module. If $\,V\,$ has highest weight $\,\lambda\in\{
2\omega_1,\,\omega_1,\,\omega_2\}\,$, then $\,V\,$ is an exceptional
module. If $\,V\,$ has highest weight $\,\omega_3\,$, then $\,V\,$ is
unclassified.}

This ends the First Part of Proof of Theorem~\ref{listcn}.
\hfill $\Box$
\newpage

\subsubsection{Type $\,C_{\ell}\,$ - Small Rank Cases}\label{clsr}
In this section, we prove Theorem~\ref{listcn} for groups of small
rank.

\subsubsection*{\ref{clsr}.1. Type $\,C_2\,$}
\addcontentsline{toc}{subsubsection}{\protect\numberline{\ref{clsr}.1}
Type $\,C_2\,$}

As groups of type $\,C_2\,$ and $\,B_2\,$ are isomorphic, the proof for
this type was done in Subsection~\ref{tbsr}.1 (p. \pageref{tbsr}).

\subsubsection*{\ref{clsr}.2. Type $\,C_3\,$}
\addcontentsline{toc}{subsubsection}{\protect\numberline{\ref{clsr}.2}
Type $\,C_3\,$}

For groups of type $\,C_3\,$, $\,|W|\,=\,2^3\,3!\,=\,2^4\cdot 3$, 
$\,|R|=2\cdot 3^2\,$ and $\,|R_{long}|=2\cdot 3$. The
limits for~(\ref{cllim}) and~(\ref{rrcl}), in this 
case, are $\,2^2\cdot 3^3\,$ and $\,2\cdot 3^2,\,$ respectively. 
$\,R_{long}^+\,=\,\{\,2\varepsilon_1,\,2\varepsilon_2,\,2\varepsilon_3\,\}$.
\begin{lemma}\label{cn3}
Let $\,V\,$ be a $\,C_3(K)$-module. 
If $\,{\cal X}_{++}(V)\,$ contains

(a) $\,1\,$ good weight with $\,3\,$ coefficients $\,\not\equiv 0\pmod p\,$ or

(b) $\,2\,$ good weights with $\,3\,$ nonzero coefficients 
(one of them with at least $\,2\,$ coefficients 
$\,\not\equiv 0\pmod p\,$ and the other with at least $\,1\,$
coefficient $\,\not\equiv 0\pmod p\,$), \\
then $\,V\,$ is not an exceptional module.
\end{lemma}\noindent
\begin{proof}
If $\,\mu\,$ is a weight with $3$ nonzero coefficients, then
$\,|W\mu|\,=\,2^4\cdot 3$. Let $\,\mu\in{\cal X}_{++}(V)\,$ satisfy (a). 
Then $\,r_p(V)\,\geq\,\displaystyle\frac{2^4\cdot 3\cdot
3}{2\cdot 3}\,=\,2^3\cdot 3\,>\,2\cdot 3^2\,$. Hence, by~(\ref{rrcl}),
$\,V\,$ is not an exceptional module.

If the condition (b) holds, then
$\,r_p(V)\,\geq\,\displaystyle\frac{2^4\cdot 3\cdot
2}{2\cdot 3}\,+\,\frac{2^4\cdot 3}{2\cdot 3}=\,2^3\cdot 3\,>\,2\cdot
3^2\,$. Hence, by~(\ref{rrcl}), $\,V\,$ is not an exceptional module.
\end{proof}
\begin{corollary}\label{corc3}
If $\,{\cal X}_{++}(V)\,$ contains a weight with
$\,3\,$ nonzero coefficients, then $\,V\,$ is not an exceptional module.
\end{corollary}\noindent
\begin{proof}
First note that bad weights with $3$ nonzero coefficients do not occur
in $\,{\cal X}_{++}(V)$, otherwise $\,V\,$ would have highest weight
with at least $1$ coefficient $\,\geq p\,$ contradicting Theorem~\ref{curt}(i).

{\bf Therefore, we can assume that weights with $3$ nonzero
coefficients occurring in $\,{\cal X}_{++}(V)\,$ are good weights.}

Let $\,\mu\,=\,a\,\omega_1\,+\,b\,\omega_2\,+\,c\,\omega_3\,$
(with $\,a\geq 1,\,b\geq 1,\,c\geq 1\,$) be a good weight in 
$\,{\cal X}_{++}(V)$.
If $\,a,\,b,\,c\,$ are all $\,\not\equiv 0\pmod p$, then Lemma~\ref{cn3}(a)
applies. Hence, we can assume that at least one of $\,a,\,b,\,c\,$ is 
$\,\equiv 0\pmod p\,$. Recall that $\,p\geq 3\,$. 

i) Let $\,a\equiv b\equiv 0,\;c\not\equiv 0\pmod p$. Then 
$\,\mu_1\,=\,\mu-\alpha_1\,=\,(a-2)\omega_1\,+\,(b+1)\omega_2\, 
+\,c\omega_3\,\in\,{\cal X}_{++}(V)\,$ is a good weight with 
$3$ coefficients $\,\not\equiv 0\pmod p$. Hence Lemma~\ref{cn3}(a) applies.

ii) Let $\,a\equiv c\equiv 0,\;b\not\equiv 0\pmod p$. Then 
$\,\mu_1\,=\,\mu-\alpha_3\,=\,a\omega_1\,+\,(b+2)\omega_2\,+\,(c-2)\omega_3,\;$
$\,\mu_2\,=\,\mu_1-\alpha_2\,=\,(a+1)\omega_1\,+\,b\omega_2\,+\,(c-1)\omega_3\,
\in\,{\cal X}_{++}(V)$. As $\,\mu_2\,$ has $3$ coefficients
$\,\not\equiv 0\pmod p$, Lemma~\ref{cn3}(a) applies.

iii) Let $\,b\equiv c\equiv 0,\;a\not\equiv 0\pmod p$. Then 
$\,\mu_1\,=\,\mu-\alpha_3\,=\,a\omega_1\,+\,(b+2)\omega_2\,+\,(c-2)\omega_3\,
\in\,{\cal X}_{++}(V)\,$ and it satisfies Lemma~\ref{cn3}(a).

iv) Let $\,a\equiv 0,\;b\not\equiv 0,\;c\not\equiv 0\pmod p$. Then 
$\,\mu_1\,=\,\mu-\alpha_1\,=\,(a-2)\omega_1\,+\,\mbox{$(b+1)\omega_2$}
\,+\,c\omega_3\,\in\,{\cal X}_{++}(V)\,$ and Lemma~\ref{cn3}(b) applies.
 
v) Let $\,b\equiv 0,\;a\not\equiv 0,\;c\not\equiv 0\pmod p$. Then 
$\,\mu_1\,=\,\mu-\alpha_2\,=\,(a+1)\omega_1\,+\,\mbox{$(b-2)\omega_2$}
\,+\,(c+1)\omega_3\,\in\,{\cal X}_{++}(V)$. Hence Lemma~\ref{cn3}(b) 
applies, provided $\,(a+1)\not\equiv 0\pmod p\,$ or $\,(c+1)\not\equiv
0\pmod p$. If $\,(a+1)\equiv (c+1)\equiv 0\pmod p,\,$ then $\,\mu_1\,$
satisfies case ii) above.

vi) Let $\,c\equiv 0,\;a\not\equiv 0,\;b\not\equiv 0\pmod p$. Then 
$\,\mu_1\,=\,\mu-\alpha_3\,=\,a\omega_1\,+\,(b+2)\omega_2\,+\,(c-2)\omega_3\,
\in\,{\cal X}_{++}(V)$, hence Lemma~\ref{cn3}(b) applies.

In any of these cases $\,V\,$ is not exceptional. This proves the corollary.
\end{proof} \vspace{2ex}\noindent
{\bf Proof of Theorem~\ref{listcn} for $\,\ell=3$.}
Let $\,V\,$ be a $\,C_3(K)$-module. 

By Corollary~\ref{corc3}, {\bf 
it suffices to consider modules $\,V\,$ such that $\,{\cal X}_{++}(V)\,$ 
contains only weights with $\,2\,$ or less nonzero coefficients.}

I) {\bf Claim}: {\it If $\,{\cal X}_{++}(V)\,$ contains bad weights with
$2$ nonzero coefficients, then $\,V\,$ is not an exceptional
$\,\mathfrak g$-module.}

Indeed, let $\,\mu\,=\,a\omega_i\,+\,b\omega_j\,$ (with $\,a\geq
1,\,b\geq 1\,$) be a bad weight in $\,{\cal X}_{++}(V)\,$ 
(hence $\,a\geq 3,\,b\geq 3\,$). If $\,\mu\,=\,a\omega_1\,+\,b 
\omega_2\,$, then $\,\mu_1\,=\,\mu-(\alpha_1+\alpha_2)\,=\,(a-1)\omega_1\, 
+\,(b-1)\omega_2\,+\,\omega_3\,\in\,{\cal X}_{++}(V)$. 
If $\,\mu\,=\,a\omega_1\,+\,b\omega_3\,$, then $\,\mu_1\,=\mu-(\alpha_1 
+\alpha_2+\alpha_3)\,=\,(a-1)\omega_1\,+\,\omega_2\,+\,(b-1)\omega_3\,
\in\,{\cal X}_{++}(V)$. If $\,\mu\,=\,a\omega_2\,+\,b\omega_3\,$. 
Then $\,\mu_1\,=\mu-(\alpha_2 +\alpha_3)\,=\,\omega_1\,+\,a\omega_2\,+
\,(b-1)\,\omega_3\,\in\, {\cal X}_{++}(V)$. 
In all these cases, $\,\mu_1\,$ is a good weight with $3$ nonzero coefficients.
Hence, by Corollary~\ref{corc3}, $\,V\,$ is not an exceptional module, proving
the claim.

{\bf Therefore, if $\,{\cal X}_{++}(V)\,$ contains weights with
$2$ nonzero coefficients, then we can assume that these are good weights.}

II) Let $\,\mu\,=\,a\omega_i\,+\,b\omega_j\,$ (with $\,a\geq 1,\,b\geq
1\,$) be a good weight in $\,{\cal X}_{++}(V)$.  

(a) Let $\,a\geq 2\,$ {\bf or} $\,b\geq 2\,$ and 
$\,\mu\,=\,a\omega_1\,+\,b\omega_3\,$.

i) For $\,a\geq 2\,$ {\bf and} $\,b\geq 2\,$, $\,\mu_1\,=\mu- 
(\alpha_1+\alpha_2+\alpha_3)\,=\,(a-1)\omega_1\,+\,\omega_2\, 
+\,(b-1)\omega_3\in\,{\cal X}_{++}(V)\,$. % has $3$ nonzero coefficients. 

ii) For $\,a= 1,\,b\geq 3\,$, $\,\mu\,=\,\omega_1\,+\,b\omega_3\,$ and
$\,\mu_1\,=\mu-\alpha_3\,=\,\omega_1\,+\,2\omega_2\,+\,(b-2)\omega_3\,\in
\,{\cal X}_{++}(V)\,$. % has $3$ nonzero coefficients.

iii) For $\,a\geq 3,\,b= 1\,$, $\,\mu\,=\,a\omega_1\,+\,\omega_3\,$ and  
$\,\mu_1\,=\mu-\alpha_1\,=\,(a-2)\omega_1\,+\,\omega_2\,+\,b\omega_3\,\in
\,{\cal X}_{++}(V)\,$.\\  % has $3$ nonzero coefficients.
In all these cases $\,\mu_1\,$ has $3$ nonzero coefficients. Hence,
by Corollary~\ref{corc3}, $\,V\,$ is not exceptional.

iv)\label{IIaivcl} For $\,a=2,\,b= 1,\,$ $\,\mu=\mu_0=\,2\,\omega_1\, 
+\,\omega_3,\;$ $\,\mu_1\,=\,\mu-\alpha_1\,=\,\omega_2\,+\,\omega_3,\;$
$\,\mu_2\,=\,\mu-(\alpha_1+\alpha_2+\alpha_3)\,=\,\omega_1\,+\,\omega_2\,
\in\,{\cal X}_{++}(V)$. For each $\,\mu_i\,$,
$\,|R_{long}^+-R^+_{\mu_i,p}|\geq 2\,$ and $\,|W\mu_i|= 2^3\cdot 3\,$.
Thus $\,r_p(V)\,\geq\,\displaystyle3\cdot \frac{2^3\cdot 3\cdot 2}{2\cdot
3}\,=\,24\,>\,2\cdot 3^2\,$. Hence, by~(\ref{rrcl}), $\,V\,$ is not an
exceptional module. 

v)\label{IIavcl} Let $\,a=1,\,b= 2\,$ and $\,\mu\,=\,\omega_1\,+\,2
\omega_3\,$. Then $\,\mu_1\,=\,\mu-(\alpha_2+\alpha_3)\,=\,2\omega_1\,
+\,\omega_3\,\in\,{\cal X}_{++}(V)\,$ satisfies case II(a)iv). Hence 
$\,V\,$ is not exceptional. 

{\bf Hence, if $\,{\cal X}_{++}(V)\,$ contains a weight 
$\,\mu\,=\,a\omega_1\,+\,b\omega_3\,$ with $\,a\geq 2\,$ {\bf or}  
$\,b\geq 2\,$, then $\,V\,$ is not exceptional. Therefore, we can
assume $\,a=b= 1\,$.} This case is treated in (d).

(b) Let $\,a\geq 2\,$ {\bf or} $\,b\geq 2\,$ and 
$\,\mu\,=\,a\omega_2\,+\,b\omega_3\,$.

i) For $\,a\geq 1,\,b\geq 2\,$, $\,\mu_1\,=\mu-(\alpha_2+\alpha_3)\, 
=\,\omega_1\,+\,a\omega_2\,+\,(b-1)\,\omega_3\,\in\,{\cal
X}_{++}(V)\,$ has $3$ nonzero coefficients. Hence, by
Corollary~\ref{corc3}, $\,V\,$ is not exceptional. 

ii) Let $\,a\geq 2,\,b=1\,$ and $\,\mu\,=\,a\omega_2\,+\,\omega_3\,$. 
Then $\,\mu_1\,=\mu-\alpha_2=\,\omega_1\,+\,\mbox{$(a-2)\omega_2$}\,
+\,2\omega_3\,\in\,{\cal X}_{++}(V)$. For $\,a\geq 3$, Corollary~\ref{corc3}
applies. For $\,a=2,\,$ $\,\mu_1\,$ satisfies case II(a)v)
(p. \pageref{IIavcl}) of this proof.

%%iii) Let $\,a=2,\,b=1\,$ and $\,\mu\,=\,2\omega_2\,+\,\omega_3\,$. Then
%%$\,\mu_1\,=\mu-\alpha_2\,=\,\omega_1\,+\,2\omega_3\,\in\,{\cal X}_{++}(V)\,$ 
%%satisfies case (b?????)iii), $\,V\,$ is not an exceptional module.
%
%iv) For $\,a=1,\,b=1,\,$  $\,\mu\,=\,\omega_2\,+\,\omega_3,\;$
%$\,\mu_1\,=\mu-(\alpha_2+\alpha_3)\,=\,\omega_1\,+\,\omega_2,\;$
%$\,\mu_2\,=\mu_1-(\alpha_1+\alpha_2)\,=\,\omega_3,\;$
%$\,\mu_3\,=\mu_2-(\alpha_2+\alpha_3)\,=\,\omega_1\,
%\in\,{\cal X}_{++}(V)$. Thus, for any $\,p\geq 3,\,$ 
%$\,r_p(V)\,\geq\,\frac{2^3\cdot 3\cdot 3}{2\cdot 3}\,+\,
%\frac{2^3\cdot 3\cdot 2}{2\cdot 3}\,+\,\frac{2^3\cdot 3}{2\cdot 3}\, 
%+\,\frac{2\cdot 3}{2\cdot 3}\,=\,25\,>\,2\cdot 3^2\,$. Hence, 
%by~(\ref{rrcl}), $\,V\,$ is not an exceptional module.

{\bf Hence, if $\,{\cal X}_{++}(V)\,$ contains a weight 
$\,\mu\,=\,a\omega_2\,+\,b\omega_3\,$ with $\,a\geq 2\,$ {\bf or}  
$\,b\geq 2\,$, then $\,V\,$ is not exceptional. Therefore, we can
assume $\,a=b= 1\,$.} This case is treated in (d).

(c) Let $\,a\geq 2\,$ {\bf or} $\,b\geq 2\,$ and 
$\,\mu\,=\,a\omega_1\,+\,b\omega_2\,$.

i) For $\,a\geq 2,\,b\geq 2$, $\,\mu_1\,=\,\mu-(\alpha_1+\alpha_2)\,=\,
(a-1)\omega_1\,+\,(b-1)\omega_2\,+\,\omega_3\,\in\,{\cal
X}_{++}(V)\,$ has $\,3\,$ nonzero coefficients. Hence 
Corollary~\ref{corc3} applies.

ii) For $\,a=1,\,b\geq 2$, $\,\mu\,=\,\omega_1\,+\,b\omega_2,\;$
$\,\mu_1\,=\mu-\alpha_2\,=\,2\omega_1\,+\,(b-2)\omega_2\,+\,\omega_3\,
\in\,{\cal X}_{++}(V)$. For $\,b\geq 3\,$, Corollary~\ref{corc3} applies.
For $\,b=2\,$, $\,\mu_1\,$ satisfies case II(a)iv) (p. \pageref{IIaivcl}).

iii) Let $\,a\geq 3,\,b=1\,$ and $\,\mu\,=\,a\,\omega_1\,+\,\omega_2\,$. 
Then $\,\mu_1\,=\mu-(\alpha_1+\alpha_2)\,=\,(a-1)\omega_1\,+\,\omega_3\, 
\in\,{\cal X}_{++}(V)$. For $\,a\geq 4$, $\,\mu_1\,$ satisfies case
II(a)iii). For $\,a=3$, $\,\mu_1\,$ satisfies case II(a)iv) of this proof 
(p. \pageref{IIaivcl}).

iv)  For $\,a=2,\,b=1,\,$ $\,\mu\,=\,2\,\omega_1\,+\,\omega_2,\;$
$\,\mu_1\,=\mu-(\alpha_1+\alpha_2)\,=\,\omega_1\,+\,\omega_3\,$,\linebreak
$\,\mu_2\,=\,\mu_1-(\alpha_1+\alpha_2+\alpha_3)\,=\,\omega_2,\;$
$\,\mu_3\,=\,\mu-\alpha_1\,=\,2\omega_2\,\in\,{\cal X}_{++}(V)$. Thus
$\,r_p(V)\,\geq\,\displaystyle 2\cdot\frac{ 2^3\cdot 3\cdot 2}{2\cdot
3}\,+\,2\cdot\frac{ 2^2\cdot 3}{2\cdot 3}\,=\,20\,>\,2\cdot 3^2\,$.  
Hence, by~(\ref{rrcl}), $\,V\,$ is not an exceptional module.

{\bf Hence, if $\,{\cal X}_{++}(V)\,$ contains a weight 
$\,\mu\,=\,a\omega_1\,+\,b\omega_2\,$ with $\,a\geq 2\,$ {\bf or}  
$\,b\geq 2\,$, then $\,V\,$ is not exceptional. Therefore, we can
assume $\,a=b= 1\,$.} This case is treated in the sequel.

(d) Let $\,a=1,\,b=1\,$ and $\,\mu\,=\,\omega_i\,+\,\omega_j\,$ 
(with $\,1\leq i<j\leq 3\,$).

i) If $\,V\,$ has highest weight $\,\mu\,=\,\omega_2\,+\,\omega_3\,$, then
$\,\mu_1\,=\mu-(\alpha_2+\alpha_3)\,=\,\omega_1\,+\,\omega_2,\;$
$\,\mu_2\,=\mu_1-(\alpha_1+\alpha_2)\,=\,\omega_3,\;$
$\,\mu_3\,=\mu_2-(\alpha_2+\alpha_3)\,=\,\omega_1\,
\in\,{\cal X}_{++}(V)$. For $\,p\geq 3,\,$ 
$\,r_p(V)\,\geq\,\displaystyle\frac{2^3\cdot 3\cdot 3}{2\cdot 3}\,+\,
\frac{2^3\cdot 3\cdot 2}{2\cdot 3}\,+\,\frac{2^3\cdot 3}{2\cdot 3}\, 
+\,\frac{2\cdot 3}{2\cdot 3}\,=\,25\,>\,2\cdot 3^2\,$. Hence, 
by~(\ref{rrcl}), $\,V\,$ is not an exceptional module.

ii) If $\,V\,$ has highest weight $\,\mu\,=\,\omega_1\,+\,\omega_2\,$,
then $\,\mu_1\,=\,\omega_3\,$, $\,\mu_2\,=\,\omega_1\in{\cal
X}_{++}(V)$. By~\cite[p. 167]{buwil}, for $\,p\neq 3\,$,
$\,m_{\mu_1}=2,\,m_{\mu_2}\geq 3$. Thus 
$\,r_p(V)\,\geq\,\displaystyle\frac{2^3\cdot 3\cdot 2}{2\cdot 3}\,+\,
2\frac{2^3\cdot 3}{2\cdot 3}\,+\,3\frac{2\cdot 3}{2\cdot 3}\, 
=\,19\,>\,2\cdot 3^2\,$. Hence, 
by~(\ref{rrcl}), $\,V\,$ is not an exceptional module.
For $\,p=3\,$, recall that $\,\mathfrak{sp}_3\,$ is a Lie subalgebra
of $\,\mathfrak{sl}_6\,$. Hence, by Claim~12 (p. \pageref{claim12al})
of First Part of Proof of Theorem~\ref{anlist}, $\,V\,$ is not exceptional.
%For $\,p=3,\,$ as $\,\mathfrak{sp}_n\,$ is a Lie subalgebra of
%$\,\mathfrak{sl}_{2n}\,$, the $\,\mathfrak{sl}_{2n}$-module
%of highest weight $\,\omega_1\,+\,\omega_2\,$ can be considered as 
%a $\,\mathfrak{sp}_n$-module (also of highest weight 
%$\,\omega_1\,+\,\omega_2\,$). Hence, by case II(f)v) of First Part of
%Proof of Theorem~\ref{anlist}, $\,V\,$ is not exceptional.

iii) If $\,V\,$ has highest weight $\,\mu\,=\,\omega_1\,+\,\omega_3\,$, then
$\,\mu_1\,=\,\mu-(\alpha_2+\alpha_3)\,=\,2\omega_1\,$ and
$\,\mu_2\,=\,\mu_1-\alpha_1\,=\,\omega_2\,\in\,{\cal X}_{++}(V)$. 
By~\cite[p. 167]{buwil} for $\,p\geq 3$, $\,m_{\mu_2}\geq 2$. 
Thus $\,r_p(V)\,\geq\,\displaystyle\frac{2^3\cdot 3\cdot 3}{2\cdot 3}\,+\,
\frac{2\cdot 3}{2\cdot 3}\,+\,2\frac{2^2\cdot 3\cdot 2}{2\cdot 3}\, 
=\,21\,>\,2\cdot 3^2\,$. Hence, 
by~(\ref{rrcl}), $\,V\,$ is not an exceptional module.

{\bf Hence, if $\,V\,$ is a $\,C_3(K)$-module such that
$\,{\cal X}_{++}(V)\,$ contains weights with $2$  
nonzero coefficients, then $\,V\,$ is not an exceptional module.}

{\bf From now on we can assume that $\,{\cal X}_{++}(V)\,$ contains
only weights with at most one nonzero coefficient.}

III) Let $\,\mu\,=\,a\omega_i\,$ with $\,a\geq 1\,$ be a weight in 
$\,{\cal X}_{++}(V)$. 

(a)i) Let $\,a\geq 3\,$ and $\,\mu\,=\,a\omega_2\,$.
Then $\,\mu_1\,=\mu-\alpha_2\,=\,\omega_1\,+\,(a-2)\omega_2\,+\,\omega_3\, 
\in{\cal X}_{++}(V)\,$ has $\,3\,$ nonzero coefficients. Hence
Corollary~\ref{corc3} applies.

ii) Let $\,a\geq 3\,$ and $\,\mu\,=\,a\,\omega_3\,$.
Then $\,\mu_1\,=\mu-\alpha_3\,=\,2\omega_2\,+\,(a-2)\omega_3\,\in
\,{\cal X}_{++}(V)$. For $\,a\geq 4\,$, $\,\mu_1\,$ satisfies case 
II(b)i). For $\,a=3\,$, $\,\mu_1\,$ satisfies case II(b)ii). 

%%i) For $\,a=p,\,$ $\,\mu\,\in\,{\cal X}_{++}(V)\,$ only 
%%if $\,V\,$ has highest weight 
%%$\,\lambda\,=\,(k\omega_1\,+\,k\omega_2\,+\,(p-k)\omega_3\,$ 
%%($\,1\leq k\leq p-1\,$), but then by Corollary~\ref{corc3}
%%$\,V\,$ is not an exceptional module.
%%If $\,a=rp,\,r\geq 2,\,$ then $\,\mu\,$ does not occur in $\,{\cal
%%X}_{++}(V).\,$ 

iii) Let $\,a\geq 3\,$ and $\,\mu\,=\,a\omega_1\,$.
Then $\,\mu_1\,=\mu-\alpha_1\,=\,(a-2)\omega_1\,+\,\omega_2 
\in{\cal X}_{++}(V)$. For $\,a\geq 5,\,$ $\,\mu_1\,$ satisfies case 
II(c)iii). For $\,a=4,\,$ $\,\mu_1\,$ satisfies case II(c)iv).
For $\,a=3\,$ and $\,p\geq 5\,$, by Lemma~\ref{III.iii} (p. \pageref{III.iii}),
$\,V\,$ is not exceptional. For $\,p=3$, $\,\mu\,=\,3\omega_1\,\in\,
{\cal X}_{++}(V)\,$ only if $\,\lambda\,=\,2\omega_1\,+\,\omega_3\,$ or 
$\,\lambda\,=\,\omega_1\,+\,2\omega_2\,$ (for instance). In these
cases $\,V\,$ is not exceptional, by part II of this proof.
 
%For $\,p=3,\,$ as $\,\mathfrak{sp}_n\,$ is a Lie subalgebra of
%$\,\mathfrak{sl}_{2n}\,$, the $\,\mathfrak{sl}_{2n}$-module
%of highest weight $\,3\omega_1\,$ can be considered as 
%a $\,\mathfrak{sp}_n$-module (also of highest weight
%$\,3\omega_1\,$). Hence, by case II(f)v) of First Part of Proof of 
%Theorem~\ref{anlist}, $\,V\,$ is not exceptional.

{\bf Hence, if $\,{\cal X}_{++}(V)\,$ contains weights
$\,\mu\,=\,a\omega_i\,$ with $\,a\geq 3\,$, then $\,V\,$ is not 
exceptional. Therefore, we can assume $\,a\leq 2\,$.} We deal with
these cases in the sequel.

(b)i) Let $\,a=2\,$ and $\,\mu\,=\,2\omega_2\,$. 
Then $\,\mu_1\,=\,\omega_1\,+\,\omega_3,\;$
$\,\mu_2\,=\,2\omega_1,\;$ $\,\mu_3\,=\,\omega_2\,\in\,{\cal X}_{++}(V)$. 
Thus, for $\,p\geq 3,\,$ $\,r_p(V)\geq\,\frac{2^2\cdot 3\cdot
2}{2\cdot 3} + \frac{2^3\cdot 3\cdot 3}{2\cdot 3}+ 
\frac{2\cdot 3}{2\cdot 3} + \frac{2^2\cdot 3\cdot 2}{2\cdot 3}\,
=\,21\,>\,2\cdot 3^2\,$. Hence, by~(\ref{rrcl}), $\,V\,$ is not 
an exceptional module.

ii) Let $\,a=2\,$ and $\,\mu\,=\,2\omega_3\,$. 
Then $\,\mu_1\,=\,\mu-\alpha_3\,=\,2\omega_2\in\,{\cal X}_{++}(V)\,$ 
satisfies case III(b)i).
%$\,\mu_2\,=\,\mu_1-\alpha_2\,=\,\omega_1\,+\,\omega_3,\;$
%$\,\mu_3\,=\,\mu_2-(\alpha_2+\alpha_3)\,=\,2\omega_1,\;$
%$\,\mu_4\,=\,\mu_3-\alpha_1\,=\,\omega_2\,\in\,{\cal X}_{++}(V)$. 
%Thus, for $\,p\geq 3,\,$ $\,r_p(V)\,\geq\,\frac{2^3\cdot 3}{2\cdot 3}\,+\,
%\frac{2^2\cdot 3\cdot 2}{2\cdot 3}\,+\,\frac{2^3\cdot 3\cdot 3}{2\cdot 3}\, 
%\frac{2\cdot 3}{2\cdot 3}\,+\frac{2^2\cdot 3\cdot 2}{2\cdot 3}\,
%=\,25\,>\,2\cdot 3^2\,$. Hence, by~(\ref{rrcl}), $\,V\,$ is not 
%an exceptional module.

iii) Let $\,a=2\,$. We may assume that $\,V\,$ has highest
weight $\,\mu\,=\,2\omega_1\,$. Then $\,V\,$ is the
adjoint module, which is exceptional by Example~\ref{adjoint} {\bf (N. 1 in
Table~\ref{tableclall})}.

{\it From now on we may assume that $\,a=1\,$ and that $\,V\,$ has highest
weight $\,\omega_i\,$}. 

iv) If $\,V\,$ has highest weight $\,\omega_3\,$ then, 
by~\cite[Theorem 2]{presup},
$\,\dim\,V\,=\,14\,<\,21\,-\,\varepsilon\,$. Thus, by
Proposition~\ref{dimcrit}, $\,V\,$ is an exceptional module {\bf (N. 4 in 
Table~\ref{tableclall})}.

v) If $\,V\,$ has highest weight
$\,\mu\,=\,\omega_2\,$, then by case III(b)vi) (p. \pageref{IIIbvi})
of First Part of Proof of Theorem~\ref{listcn}, $\,V\,$ is an
exceptional module {\bf (N. 3 in Table~\ref{tableclall})}.

vi) Finally, if $\,V\,$ has highest weight $\,\omega_1\,$, then 
by case III(b)vii) (p. \pageref{IIIbvii}) of First Part of Proof of  
Theorem~\ref{listcn}, $\,V\,$ is an exceptional module 
{\bf (N. 2 in Table~\ref{tableclall})}.

{\bf Hence, if $\,V\,$ is a $\,C_3(K)$-module such that
$\,{\cal X}_{++}(V)\,$ contains weights
$\,\mu\,=\,a\omega_i\,$ with $\,a\geq 3\,$ or ($\,a=2\,$ and
$\,i=2,\,3\,$), then $\,V\,$ is not exceptional. If $\,V\,$ has highest
weight $\,\lambda\in\{ 2\omega_1,\,\omega_1,\,\omega_2,\,\omega_3\}\,$,
then $\,V\,$ is an exceptional module.} 

This finishes the Proof of Theorem~\ref{listcn} for $\,\ell=3$.
\hfill $\Box$

\subsubsection*{\ref{clsr}.3. Type $\,C_4$}
\addcontentsline{toc}{subsubsection}{\protect\numberline{\ref{clsr}.3}
Type $\,C_4\,$}
 
For groups of type $\,C_4\,$, $\,|W|\,=\,2^4\,4!\,=\,384$,
$\,|R|=2\cdot 4^2\,$ and $\,|R_{long}|=2\cdot 4$. The limit 
for~\eqref{cllim} is $\,4\cdot 4^3\,=\,2^8\,$ and for~\eqref{rrcl} is
$\,2\cdot 4^2$. $\,R_{long}^+\,=\,\{\,2\varepsilon_i\,/\,
1\leq i\leq 4\,\}.\,$ The following are reduction lemmas.

\begin{lemma}\label{cn4}
Let $\,V\,$ be a $\,C_4(K)$-module. If $\,{\cal X}_{++}(V)\,$ contains 

(a) $\,1\,$ good weight with $\,4\,$ nonzero coefficients or

(b) $\,2\,$ good weights with $\,3\,$ nonzero coefficients or

(c) $\,1\,$ good weight with $\,3\,$ nonzero coefficients and 
$\,2\,$ good weights with $\,2\,$ nonzero coefficients, 
then $\,V\,$ is not an exceptional $\,\mathfrak g$-module.
\end{lemma}\noindent
\begin{proof}
For (a), let $\,\mu\in{\cal X}_{++}(V)\,$ be a good weight with
$\,4\,$ nonzero coefficients. Then $\,s(V)\,\geq\,|W\mu|=2^7\cdot 3\,
>\,2^8\,$. Hence, by~(\ref{cllim}), $\,V\,$ is not exceptional.

If a weight $\,\mu\,$ has $\,3\,$ nonzero coefficients,
then $\,|W\mu|=2^6\cdot 3$. Thus if (b) holds, then
$\,s(V)\,\geq\,2\cdot 2^6\cdot 3\,>\,2^8\,$. Hence, by~(\ref{cllim}), 
$\,V\,$ is not exceptional.

If a weight $\,\mu\,$ has $\,2\,$ nonzero coefficients, then 
$\,|W\mu|\geq 2^4\cdot 3.\,$ Thus if (c) holds, then
$\,s(V)\,\geq\,2^6\cdot 3\,+\,2\cdot 2^4\cdot 3\,=\,2^5\cdot 3^2\,
>\,2^8\,$. Hence, by~(\ref{cllim}), $\,V\,$ is not an exceptional module. 
This proves the lemma.
\end{proof}

{\bf Note}: Bad weights with $\,4\,$ nonzero coefficients do not occur in
$\,{\cal X}_{++}(V)$, otherwise $\,V\,$ would have highest weight with at
least one coefficient $\,\geq p$, contradicting Theorem~\ref{curt}(i).

\begin{lemma}\label{bgd3c4} Let $\,V\,$ be a $\,C_4(K)$-module.
If $\,{\cal X}_{++}(V)\,$ contains a weight with $\,3\,$ nonzero
coefficients, then $\,V\,$ is not an exceptional $\,\mathfrak g$-module.
\end{lemma}
\begin{proof}
First note that if $\,\mu\,$ is a bad weight with $\,3\,$ nonzero 
coefficients in $\,{\cal X}_{++}(V)$, then it is easy to produce (from 
$\,\mu\,$) a good weight in $\,{\cal X}_{++}(V)\,$ satisfying 
Lemma~\ref{cn4}. {\bf Therefore, if $\,{\cal X}_{++}(V)\,$  contains
weights with $\,3\,$ nonzero coefficients, then we can assume that
these are good weights.} 

Let $\,\mu\,=\,a\,\omega_{i}\,+\,b\,\omega_j\,+\,c\,\omega_k\,$ 
(with $\,1\leq i<j<k\leq 4\,$ and $\,a\geq 1,\,b\geq 1,\,c\geq 1\,$) 
be a good weight in $\,{\cal X}_{++}(V)$.     

(a) Let $\,\mu\,=\,a\omega_{1}\,+\,b\omega_2\,+\,c\omega_3\,$. Then
$\,\mu_1\,=\,\mu-(\alpha_1+\alpha_2+\alpha_3)\,=\,
(a-1)\omega_{1}\,+\,b\omega_2\,+\,(c-1)\omega_3\,+\,\omega_4\,\in
\,{\cal X}_{++}(V)$.

(b) Let $\,\mu\,=\,a\omega_{1}\,+\,b\omega_2\,+\,c\omega_4\,$. Then
$\,\mu_1\,=\,\mu-(\alpha_1+\alpha_2+\alpha_3+\alpha_4)\,=\,
(a-1)\omega_{1}\,+\,b\omega_2\,+\,\omega_3\,+\,(c-1)\omega_4\,\in
\,{\cal X}_{++}(V)$.

For $\,a\geq 2,\,b\geq 1,\,c\geq 2\,$, in both cases, $\,\mu_1\,$ 
satisfies Lemma~\ref{cn4}(a). For ($\,a=1,\,b\geq 1,\,c\geq 2\,$) or 
($\,a\geq 2,\,b\geq 1,\,c\geq 2\,$), in both cases, Lemma~\ref{cn4}(b) 
applies. 

Let $\,a=b=c=1\,$. Then in case (a) $\,\mu\,=\,\omega_{1}\,+\,\omega_2\,
+\,\omega_3,\;$ $\,\mu_1\,=\,\omega_2\,+\,\omega_4,\;$ and 
$\,\mu_2\,=\,\mu-(\alpha_2+\alpha_3)\,=\,2\omega_{1}\,+\,\omega_4\,\in
\,{\cal X}_{++}(V).\,$ In case (b) $\,\mu\,=\,\omega_{1}\,+\,\omega_2\,
+\,\omega_4,\;$ $\,\mu_1\,=\,\omega_2\,+\,\omega_3,\;$ and 
$\,\mu_2\,=\,\mu-(\alpha_2+\alpha_3+\alpha_4)\,=\,2\omega_{1}\,+\,
\omega_3\,\in\,{\cal X}_{++}(V)$. Thus Lemma~\ref{cn4}(c) applies 
in both cases.

(c) Let $\,\mu\,=\,a\omega_{1}\,+\,b\omega_3\,+\,c\omega_4\,$.
Then $\,\mu_1\,=\,\mu-(\alpha_3+\alpha_4)\,=\,a\omega_{1}\,+\,\omega_2\,
+\,b\omega_3\,+\,(c-1)\omega_4\,\in\,{\cal X}_{++}(V)$.
For $\,a\geq 1,\,b\geq 1,\,c\geq 2\,$, $\,\mu_1\,$ satisfies 
Lemma~\ref{cn4}(a). For $\,a\geq 1,\,b\geq 1,\,c=1,\,$ 
%$\,\mu\,=\,a\omega_{1}\,+\,b\omega_3\,+\,\omega_4,\;$ $\,\mu_1\,=\,a 
%\omega_{1}\,+\,\omega_2\,+\,b\omega_3\,\in\,{\cal X}_{++}(V)$. Hence, 
Lemma~\ref{cn4}(b) applies.

(d) Let $\,\mu\,=\,a\omega_{2}\,+\,b\omega_3\,+\,c\omega_4\,$.
Then $\,\mu_1=\mu-(\alpha_2+\alpha_3+\alpha_4)=\,\omega_{1}\,+\,
\mbox{$(a-1)\omega_2$}\,+\,(b+1)\omega_3\,+\,(c-1)\omega_4\;$ and $\,\mu_2\,
=\,\mu_1-\alpha_3\,=\,\omega_{1}\,+\,a\omega_2\,+\,(b-1)\omega_3\,+\,c
\omega_4\,\in\,{\cal X}_{++}(V)$. For $\,a\geq 1,\,b\geq 1,\,c\geq 1\,$, 
Lemma~\ref{cn4}(a) or (b) applies. This proves the lemma.
\end{proof} \vspace{2ex}\noindent
{\bf Proof of Theorem~\ref{listcn} for $\,\ell=4$.} 
Let $\,V\,$ be a $\,C_4(K)$-module.

By Lemma~\ref{bgd3c4}, {\bf it suffices to consider modules $\,V\,$ such
that $\,{\cal X}_{++}(V)\,$ contains only weights with at most $2$
nonzero coefficients.} 

I) {\bf Claim 1}: \label{cl1c4} {\it Let $\,V\,$ be a $\,C_4(K)$-module.
If $\,{\cal X}_{++}(V)\,$ contains a bad weight with $2$ nonzero  
coefficients, then $\,V\,$ is not an exceptional $\,\mathfrak g$-module.}     

Indeed, let $\,\mu\,=\,a\,\omega_{i}\,+\,b\,\omega_j\,$ be a bad weight 
with $2$ nonzero coefficients in $\,{\cal X}_{++}(V)$. Hence $\,a\geq 3,\,
b\geq 3\,$. 

If $\,\mu\,=\,a\omega_{1}\,+\,b\omega_j\;$ (with $\,j=2,\,3\,$), 
then $\,\mu_1\,=\,\mu-(\alpha_1+\cdots +\alpha_j)\,=\,(a-1)\omega_{1}\,
+\,(b-1)\omega_j\,+\,\omega_{j+1}\,\in\,{\cal X}_{++}(V)$. If  
$\,\mu\,=\,a\omega_{1}\,+\,b\omega_4\,$, then 
$\,\mu_1\,=\,\mu-(\alpha_1+\cdots+\alpha_4)\,=\,(a-1)\omega_{1}\,+\,
\omega_3\,+\,(b-1)\omega_4\,\in\,{\cal X}_{++}(V)$. 
If $\,\mu\,=\,a\omega_{3}\,+\,b\omega_4\,$, then 
$\,\mu_1\,=\,\mu-(\alpha_3+\alpha_4)\,=\,\omega_{2}\,+\,
a\omega_3\,+\,(b-1)\omega_4\,\in\,{\cal X}_{++}(V)$. 
In all these cases $\,\mu_1\,$ is a good weight with $3$ nonzero 
coefficients, hence Lemma~\ref{bgd3c4} applies.

If $\,\mu\,=\,a\omega_{2}\,+\,b\omega_3\,$, then 
$\,\mu_1\,=\mu-(\alpha_2+\alpha_3)=\,\omega_1\,+\,(a-1)\omega_{2}\,+\,
\mbox{$(b-1)\omega_3$}\,+\,\omega_4\,\in\,{\cal X}_{++}(V)$. 
If $\,\mu\,=\,a\omega_{2}\,+\,b\omega_4\,$, then 
$\,\mu_1\,=\mu-(\alpha_2+\cdots+\alpha_4)=\,\omega_1\,+\,(a-1)
\omega_{2}\,+\,\omega_3\,+\,(b-1)\omega_4\,\in\,{\cal X}_{++}(V)$. 
In these cases $\,\mu_1\,$ is a good weight with $4$ nonzero
coefficients, hence Lemma~\ref{cn4}(a) applies.

In all these cases $\,V\,$ is not an exceptional module, proving the
Claim 1. \hfill $\Box$ \vspace{1.2ex}

{\bf Therefore, if $\,{\cal X}_{++}(V)\,$ contains weights with 
$2$ nonzero coefficients, then we can assume that these are good weights.}

II) Let $\,\mu\,=\,a\,\omega_{i}\,+\,b\,\omega_j\,$, (with $\,1\leq i<j
\leq 4\,$ and $\,a\geq 1,\,b\geq 1\,$) be a good weight in 
$\,{\cal X}_{++}(V)$.

{\bf Claim 2}: \label{cl2c4} {\it Let $\,V\,$ be a $\,C_4(K)$-module.
If $\,{\cal X}_{++}(V)\,$ contains a good weight 
$\,\mu\,=\,a\,\omega_{2}\,+\,b\,\omega_3\,$ (with $\,a\geq 1,\,b\geq
1\,$), then $\,V\,$ is not an exceptional $\,\mathfrak g$-module.}

Indeed, for ($\,a\geq 1,\,b\geq 2\,$) or ($\,a\geq 2,\,b\geq 1\,$), 
$\,\mu_1\,=\,\mu-(\alpha_2+\alpha_3)\,=\,\omega_1\,+\,(a-1)\omega_2\,+\,
(b-1)\omega_3\,+\,\omega_4\,\in\,{\cal X}_{++}(V)\,$ has $3$ nonzero
coefficients. Hence Lemma~\ref{bgd3c4} applies.  
For $\,a=b=1\,$, $\,\mu\,=\,\omega_{2}\,+\,\omega_3\,$, $\,|W\mu|\,=\,2^5
\cdot 3\,$ and $\,|R_{long}^+-R^+_{\mu,p}|\,=\,3\,$. Thus
$\,r_p(V)\,\geq\,\displaystyle\frac{2^5\cdot 3\cdot 3}{2\cdot 4}\,=\,36 \,>\,
2\cdot 4^2\,$. Hence, by~\eqref{rrcl}, $\,V\,$ is not exceptional.
This proves Claim 2.\hfill $\Box$ \vspace{1.5ex}

{\bf Claim 3}:\label{claim3c4} {\it Let $\,V\,$ be a $\,C_4(K)$-module.
If $\,{\cal X}_{++}(V)\,$ contains a good weight $\,\mu\,=\,a\,\omega_{i} 
\,+\,b\,\omega_4\,$ (with $\,1\leq i\leq 3\,$ and $\,a\geq 1,\,b\geq
1\,$), then $\,V\,$ is not an exceptional $\,\mathfrak g$-module.}

(a) Let $\,\mu\,=\,a\,\omega_{2}\,+\,b\,\omega_4\in{\cal X}_{++}(V)$. Then 

i) for ($\,a\geq 1,\,b\geq 2\,$) or ($\,a\geq 2,\,b\geq 1\,$), 
$\,\mu_1\,=\,\mu-(\alpha_2+\alpha_3+\alpha_4)\,=\,\omega_1\,+\,
(a-1)\omega_2\,+\,\omega_3\,+\,(b-1)\,\omega_4\,\in\,{\cal X}_{++}(V)$.
As $\,\mu_1\,$ has $3$ or more nonzero coefficients, Lemma~\ref{bgd3c4} or 
Lemma~\ref{cn4}(a) applies.

ii) Let $\,a=b=1\,$ and $\,\mu\,=\,\omega_{2}\,+\,\omega_4\,$. Then 
$\,|W\mu|\,=\,2^5\cdot 3\,$ and $\,|R_{long}^+-R^+_{\mu,p}|\,=\,4\,$. Thus
$\,r_p(V)\,\geq\,\displaystyle\frac{2^5\cdot 3\cdot 4}{2\cdot 4}\,=\,48\,>\,
2\cdot 4^2\,$. Hence, by~\eqref{rrcl}, $\,V\,$ is not exceptional.

(b) Let $\,\mu\,=\,a\,\omega_{1}\,+\,b\,\omega_4\,\in{\cal X}_{++}(V)$. Then: 

i) for $\,a\geq 1,\,b\geq 2,\,$ $\,\mu_1\,=\,\mu-\alpha_4\,=\,a\omega_{1}\,
+\,2\omega_3\,+\,(b-2)\omega_4\,$ and $\,\mu_2\,=\,\mu_1-\alpha_3\,=\,
a\omega_1\,+\,\omega_2+\,(b-1)\omega_4\,\in\,{\cal X}_{++}(V)$. 
As $\,\mu_2\,$ has $3$ nonzero coefficients, Lemma~\ref{bgd3c4} applies.

ii) For $\,a\geq 2,\,b=1,\,$ $\,\mu\,=\,a\,\omega_{1}\,+\,\omega_4\;$ and 
$\,\mu_1\,=\,\mu-\alpha_1\,=\,(a-2)\,\omega_1\,+\,\omega_2\,+\,\omega_4
\,\in\,{\cal X}_{++}(V)$. For $\,a\geq 3$, $\,\mu_1\,$ has $3$ nonzero 
coefficients, hence Lemma~\ref{bgd3c4} applies. For $\,a=2$, $\,\mu_1\,$
satisfies case (a)ii) of this claim.

iii) Let $\,a=b=1\,$ and $\,\mu\,=\,\omega_{1}\,+\,\omega_4\,$. Then
$\,\mu_1\,=\,\mu-(\alpha_3+\alpha_4)\,=\,\omega_1\,+\,\omega_2\,\in\,
{\cal X}_{++}(V)$. As $\,|R_{long}^+-R^+_{\mu,p}|\,=\,4\,$ and 
$\,|R_{long}^+-R^+_{\mu_1,p}|\,=\,2\,$, $\,r_p(V)\,\geq\,
\displaystyle\frac{2^6\cdot 4}{2\cdot 4}\,+\,\frac{2^4\cdot 3\cdot 2}
{2\cdot 4}\,=\,38\,>\,2\cdot 4^2\,$. Hence, by~\eqref{rrcl}, $\,V\,$ 
is not exceptional.

(c) Let $\,\mu\,=\,a\,\omega_{3}\,+\,b\,\omega_4\,\in{\cal X}_{++}(V)$. Then:

i) for $\,a\geq 1,\,b\geq 2,\,$ $\,\mu_1\,=\,\mu-(\alpha_3+\alpha_4)\,
=\,\omega_2\,+\,a\omega_3\,+\,(b-1)\,\omega_4\,\in\,{\cal X}_{++}(V)\,$ 
is a good weight with $3$ nonzero coefficients. Hence
Lemma~\ref{bgd3c4} applies.

ii) Let $\,a\geq 1,\,b=1\,$ and $\,\mu\,=\,a\,\omega_{3}\,+\,\omega_4\,$.
Then $\,\mu_1\,=\,\mu-(\alpha_3+\alpha_4)\,=\,\omega_2\,+\,a\,\omega_3\,\in\,
{\cal X}_{++}(V)\,$ satisfies Claim 2. Hence 
$\,V\,$ is not an exceptional module. This proves Claim 3.
\hfill $\Box$ \vspace{1.5ex}

{\bf Therefore, if $\,{\cal X}_{++}(V)\,$ contains a good weight 
$\,\mu\,=\,a\,\omega_{i}\,+\,b\,\omega_j\,$ (with $\,a\geq 1,\,b\geq 1\,$), 
then we can assume that $\,i=1\,$ and $\,j=2\,$ or $\,3\,$.}
These cases are treated as follows.

(d) Let $\,\mu\,=\,a\,\omega_{1}\,+\,b\,\omega_3\,\in{\cal X}_{++}(V)$. Then

i) for $\,a\geq 1,\,b\geq 2,\,$ 
$\,\mu_1\,=\,\mu-\alpha_3\,=\,a\,\omega_{1}\,+\,\omega_2\,+\,(b-2)\,
\omega_3\,+\,\omega_4\,\in\,{\cal X}_{++}(V)$. As $\,\mu_1\,$ is a good 
weight with $3$ nonzero coefficients, Lemma~\ref{bgd3c4} applies.

ii) Let $\,a\geq 2,\,b=1\,$ and $\,\mu\,=\,a\,\omega_{1}\,+\,\omega_3\,$. 
Then $\,\mu_1\,=\,\mu-\alpha_1\,=\,(a-2)\,\omega_1\,+\,\omega_2\,+\,
\omega_3\,\in\,{\cal X}_{++}(V)$. For $\,a\geq 3$, Lemma~\ref{bgd3c4} 
applies. For $\,a=2,\,$ $\,\mu_1\,$ satisfies Claim~2 (p. \pageref{cl2c4}).
Hence $\,V\,$ is not an exceptional module.

iii) Let $\,a=b=1\,$ and $\,\mu\,=\,\omega_{1}\,+\,\omega_3\,$. Then
$\,|W\mu|\,=\,2^5\cdot 3,\,$ $\,|R_{long}^+-R^+_{\mu,p}|\,=\,3\,$. Thus
$\,r_p(V)\,\geq\,\displaystyle\frac{2^5\cdot 3\cdot 3}{2\cdot 4}\,=\,36 \,>\,
2\cdot 4^2\,$. Hence, by~\eqref{rrcl}, $\,V\,$ is not exceptional.
%Then $\,\mu_1\,=\,\mu-(\alpha_1+\cdots+\alpha_3)\,=\,\omega_4,\;$
%$\,\mu_2\,=\,\mu-(\alpha_2+2\alpha_3+\alpha_4)\,=\,2\omega_1,\;$
%$\,\mu_3\,=\,\mu_1-(\alpha_3+\alpha_4)\,=\,\omega_2
%\,\in\,{\cal X}_{++}(V),\;$ 

(e) Let $\,\mu\,=\,a\,\omega_{1}\,+\,b\,\omega_2\,\in{\cal X}_{++}(V)$. Then: 

i) for $\,a\geq 2,\,b\geq 2,\,$ $\,\mu_1\,=\,\mu-(\alpha_1+\alpha_2)\,
=\,(a-1)\,\omega_{1}\,+\,(b-1)\,\omega_2\,+\,\omega_3\,\in\,
{\cal X}_{++}(V)\,$ has $3$ nonzero coefficients. Hence
Lemma~\ref{bgd3c4} applies.

ii) For $\,a=1,\,b\geq 2,\,$ $\,\mu\,=\,\omega_{1}\,+\,b\,\omega_2\;$ and 
$\,\mu_1\,=\,\mu-\alpha_2\,=\,2\omega_1\,+\,(b-2)\,\omega_2\,+\,\omega_3\,
\in\,{\cal X}_{++}(V)$. For $\,b\geq 3,\,$ Lemma~\ref{bgd3c4} applies.
For $\,b=2,\,$ $\,\mu_1\,$ satisfies case (d)ii). Hence $\,V\,$ is not 
exceptional.

iii) Let $\,a\geq 2,\,b=1\,$ and $\,\mu\,=\,a\,\omega_{1}\,+\,\omega_2\,$. 
Then $\,\mu_1\,=\,\mu-(\alpha_1+\alpha_2)\,=\,(a-1)\,\omega_1\,+\,\omega_3\,
\in \,{\cal X}_{++}(V)\,$ satisfies case (d). Hence $\,V\,$ is not 
exceptional.

iv) Let $\,a=b=1\,$. We may assume that $\,V\,$ has highest weight
$\,\mu\,=\,\omega_{1}\,+\,\omega_2\,$. Then $\,\mu_1\,=\,\omega_3\,$,
$\,\mu_2\,=\,\mu_1-(\alpha_2+2\alpha_3+\alpha_4)\,=\,\omega_1\,\in\,
{\cal X}_{++}(V)$. By~\cite[p. 168]{buwil}, for $\,p\neq 3,\,$ 
$\,m_{\mu_1}=2\,$ and $\,m_{\mu_2}\geq 4\,$. Thus, as 
$\,|R_{long}^+-R^+_{\mu,p}|\,=\,2,\,$ $\,|R_{long}^+-R^+_{\mu_1,p}|\,=\,3\,$
and $\,|R_{long}^+-R^+_{\mu_2,p}|\,=\,1\,$, $\,r_p(V)\,\geq\,
\displaystyle\frac{2^4\cdot 3\cdot 2}{2\cdot 4}\,+\,2\cdot
\frac{2^5\cdot 3}{2\cdot 4}\,+\,4\cdot\frac{2\cdot 4}
{2\cdot 4}\,=\,40\,>\,2\cdot 4^2\,$. Hence, by~\eqref{rrcl}, $\,V\,$ 
is not exceptional. For $\,p=3$, by Claim~12 (p. \pageref{claim12al})
of First Part of Proof of Theorem~\ref{anlist}, $\,V\,$ is not exceptional.

{\bf Hence if $\,V\,$ is an $\,C_4(K)$-module such that 
$\,{\cal X}_{++}(V)\,$ contains weights with $2$ nonzero coefficients,
then $\,V\,$ is not an exceptional module.} 

{\bf From now on we can assume that $\,{\cal X}_{++}(V)\,$ contains only 
weights with at most one nonzero coefficient.}

III) Let $\,\mu\,=\,a\,\omega_{i}\,$ (with $\,a\geq 1\,$) be a weight in 
$\,{\cal X}_{++}(V)$.  

(a)i) Let $\,a\geq 2\,$ and $\,\mu\,=\,a\,\omega_{4}\,$. 
Then $\,\mu_1\,=\,\mu-\alpha_4\,=\,2 \omega_3\,+\,(a-2)\omega_4\,$ and
$\,\mu_2\,=\,\mu_1-\alpha_3\,=\,\omega_2\,+\,(a-1)\,\omega_4
\,\in\,{\cal X}_{++}(V)$. For $\,a\geq 2$, $\,\mu_2\,$ satisfies Claim
3 (p. \pageref{claim3c4}). 

ii) Let $\,a\geq 2\,$ and $\,\mu\,=\,a\,\omega_{3}\,$.     
Then $\,\mu_1\,=\,\mu-\alpha_3\,=\,\omega_2\,+\,(a-2)\,\omega_3\,+\,\omega_4
\,\in\,{\cal X}_{++}(V)$. For $\,a\geq 3$, Lemma~\ref{bgd3c4} applies.
For $\,a=2\,$, $\,\mu_1\,$ satisfies Claim 3.

iii) Let $\,a\geq 2\,$ and $\,\mu\,=\,a\,\omega_{2}\,$.     
Then $\,\mu_1\,=\,\mu-\alpha_2\,=\,\omega_1\,+\,(a-2)\,\omega_2\,+\,\omega_3
\,\in\,{\cal X}_{++}(V)$. For $\,a\geq 3$, Lemma~\ref{bgd3c4} applies.
For $\,a=2$, $\,\mu_1\,$ satisfies case II(d)iii).

iv) Let $\,a\geq 2\,$ and $\,\mu\,=\,a\,\omega_{1}\,$.  
Then $\,\mu_1\,=\,\mu-\alpha_1\,=\,(a-2)\omega_1\,+\,\omega_2\,\in\,
{\cal X}_{++}(V)$. For $\,a\geq 4$, $\,\mu_1\,$ satisfies case II(e)iii.

For $\,a=3\,$ and $\,p\geq 5\,$ we can assume that $\,V\,$ has highest weight
$\,3\omega_1\,$. By Lemma~\ref{III.iii} (p. \pageref{III.iii}),
$\,V\,$ is not exceptional.
For $\,p=3\,$, $\,\mu\,\in\,{\cal X}_{++}(V)\,$ only if $\,V\,$ has
highest weight $\,\lambda\,=\,\omega_1\,+\,2\omega_2\,$ or 
$\,\lambda\,=\,(3-k)\omega_1\,+\,k\omega_3\,$ (for some
$\,1\leq k\leq 2\,$). In these cases, by Part II(e) or II(d) of this proof, 
$\,V\,$ is not an exceptional module.

Let $\,a=2\,$. If $\,V\,$ has highest weight $\,2\,\omega_{1}\,$, then 
$\,V\,$ is the adjoint module, which is exceptional by Example~\ref{adjoint}
{\bf (N. 1 in Table~\ref{tableclall})}.

{\bf Hence if $\,{\cal X}_{++}(V)\,$ contains a weight 
$\,\mu\,=\,a\,\omega_{i}\,$ (with $\,a\geq 1\,$ and $\,1\leq i\leq 4\,$), 
then we can assume $\,a=1\,$ and that $\,V\,$ has highest weight
$\,\omega_i\,$.} These cases are treated in the sequel. 

(b)i) {\bf If $\,V\,$ has highest weight $\,\omega_{3}\,$ (resp.,
$\,\omega_{4}\,$). Then $\,V\,$ is unclassified (N. 2 (resp., 4) in 
Table~\ref{leftcn})}.

ii) If $\,V\,$ has highest weight $\,\omega_{2}\,$ then, by case
III(b)vi) (p. \pageref{IIIbvi}) of First Part of Proof of 
Theorem~\ref{listcn}, 
$\,V\,$ is an exceptional module {\bf (N. 3 in Table~\ref{tableclall})}.
 
iii) If $\,V\,$ has highest weight $\,\omega_{1}\,$ then, by case
III(b)vii) (p. \pageref{IIIbvii}) of First Part of Proof of 
Theorem~\ref{listcn},
$\,V\,$ is an exceptional module {\bf (N. 2 in Table~\ref{tableclall})}.

{\bf Hence if $\,V\,$ is an $\,C_4(K)$-module such that 
$\,{\cal X}_{++}(V)\,$ contains $\,\mu\,=\,a\,\omega_i\,$ with
($\,a\geq 3\,$ and $\,1\leq i\leq 4\,$) or ($\,a\geq 2\,$ and $\,2\leq
i\leq 4\,$), then $\,V\,$ is not exceptional. If $\,V\,$ has highest weight
$\,\lambda\in\{ 2\omega_1,\,\omega_1,\,\omega_2\}\,$, then $\,V\,$ is  
an exceptional module. If $\,V\,$ has highest weight $\,\lambda\in\{  
\omega_3,\,\omega_4\}\,$, then $\,V\,$ is unclassified.} 

This finishes the proof of Theorem~\ref{listcn} for $\,\ell=4$.
\hfill $\Box$

\subsubsection*{\ref{clsr}.4. Type $\,C_5$}
\addcontentsline{toc}{subsubsection}{\protect\numberline{\ref{clsr}.4}
Type $\,C_5\,$}

For groups of type $\,C_5\,$, $\,|W|\,=\,2^5\,5!\,=\,2^8\cdot 3\cdot
5$, $\,|R|=2\cdot 5^2\,$, $\,|R_{long}|= 2\cdot 5$. The limit 
for~(\ref{cllim}) is $\,2^2\cdot 5^3\,$ and for~\eqref{rrcl} is 
$\,2\cdot 5^2$. $\,R_{long}^+\,=\,\{\,2\varepsilon_i\,/\,
1\leq i\leq 5\,\}.\,$ The following are reduction lemmas.

\begin{lemma}\label{cn5}
Let $\,V\,$ be a $\,C_5(K)$-module. If $\,{\cal X}_{++}(V)\,$ contains 
$\,1\,$ good weight with $\,3\,$ (or more) nonzero coefficients,
%%%%%(b) $\,1\,$ good weight with $\,3\,$ nonzero coefficients,\\
then $\,V\,$ is not an exceptional $\,\mathfrak g$-module.
\end{lemma}\noindent
\begin{proof}
If a weight $\,\mu\,$ has $\,4\,$ (or more) nonzero coefficients, then
$\,|W\mu|\geq 2^7\cdot 3\cdot 5\,$. Thus if $\,{\cal X}_{++}(V)\,$ contains 
$\,1\,$ good weight with $\,4\,$ (or more) nonzero coefficients,
then $\,s(V)\,\geq\,2^7\cdot 3\cdot 5\,>\,4\cdot 5^3$. Hence,
by~(\ref{cllim}), $\,V\,$ is not an exceptional module. 

If $\,\mu\,$ has $\,3\,$ nonzero coefficients, then 
either (i) $\,|W\mu|=2^6\cdot 3\cdot 5\,$ or 
(ii) $\,|W\mu|=2^7\cdot 5\,$ or (iii) $\,|W\mu|=2^5\cdot 3\cdot 5.\,$
If $\,\mu\in {\cal X}_{++}(V)\,$ is a good weight with $\,3\,$ nonzero 
coefficients satisfying case (i) or (ii), then $\,s(V)\geq |W\mu|\,> 
\,2^2\cdot 5^3\,$. Hence, by~(\ref{cllim}), 
$\,V\,$ is not exceptional. If $\,\mu\,$ satisfies (iii),
then $\,\mu\,=\,a\,\omega_1\,+\,b\,\omega_2\,+\,c\,\omega_3\,$.
For $\,a\geq 1,\,b\geq 1,\,c\geq 1\,$, $\,\mu_1\,=\mu-(\alpha_1+ 
\alpha_2+\alpha_3)=\,(a-1)\omega_1\,+\,b\omega_2\,+\,\mbox{$(c-1)\omega_3$}
\,+\,\omega_4\,\in\,{\cal X}_{++}(V)\,$ is a good weight such that
$\,|W\mu_1|\,\geq\,2^5\cdot 3\cdot 5\,$. Thus
$\,s(V)\,\geq\,2^5\cdot 3\cdot 5\,+\,2^5\cdot 3\cdot 5\,=\,2^6\cdot
3\cdot 5\,>\,2^2\cdot 5^3\,$. Hence, by~(\ref{cllim}), 
$\,V\,$ is not an exceptional module. This proves the lemma.
\end{proof}

{\bf Note}: Bad weights with $\,5\,$ nonzero coefficients do not occur in
$\,{\cal X}_{++}(V)$, otherwise $\,V\,$ would have highest weight with at
least one coefficient $\,\geq p$, contradicting Theorem~\ref{curt}(i).

\begin{corollary}\label{bad34c5}
Let $\,V\,$ be a $\,C_5(K)$-module. If $\,{\cal X}_{++}(V)\,$ contains
a bad weight with $\,3\,$ or $\,4\,$ nonzero coefficients, then  
$\,V\,$ is not an exceptional $\,\mathfrak g$-module.
\end{corollary}\noindent
\begin{proof} Just note that if $\,\mu\,$ is a bad weight with $\,3\,$
or $\,4\,$ nonzero coefficients in $\,{\cal X}_{++}(V)$, 
then it is easy to produce (from $\,\mu\,$) a good weight in 
$\,{\cal X}_{++}(V)\,$ satisfying Lemma~\ref{cn5}. This proves the corollary.
\end{proof}
%%\vspace{2ex}\noindent
\newpage\noindent
{\bf Proof of Theorem~\ref{listcn} for $\,\ell=5$.} 
Let $\,V\,$ be a $\,C_5(K)$-module. 

By Lemma~\ref{cn5} and
Corollary~\ref{bad34c5}, {\bf it suffices to consider modules $\,V\,$
such that $\,{\cal X}_{++}(V)\,$ contains only 
weights with 2 or less nonzero coefficients.}

I) {\bf Claim 1}: \label{cl1c5} {\it Let $\,V\,$ be a $\,C_5(K)$-module. 
If $\,{\cal X}_{++}(V)\,$ contains bad weights with $2$ nonzero 
coefficients, then $\,V\,$ is not an exceptional $\,\mathfrak g$-module.}

Indeed, let $\,\mu\in{\cal X}_{++}(V)\,$ be a bad weight
with $2$ nonzero coefficients. Then, by Claim 1
(p. \pageref{claim1cl}) of First Part of Proof of Theorem~\ref{listcn},  
one can produce (from $\,\mu\,$) good weights in $\,{\cal
X}_{++}(V)\,$ satisfying Lemma~\ref{cn5}. This proves the claim.

{\bf Therefore, we can assume that weights with $2$ nonzero
coefficients occurring in $\,{\cal X}_{++}(V)\,$ are good weights.}
First we deal with some particular cases.\hfill $\Box$ \vspace{1.5ex}

{\bf Claim 2}: \label{cl2c5} {\it Let $\,V\,$ be a $\,C_5(K)$-module. 
If either $\,\omega_1\,+\,\omega_4\,$ or $\,\omega_1\,+\,\omega_5\,$  
or $\,\omega_1\,+\,\omega_3\,$ is a weight in $\,{\cal X}_{++}(V)$, 
then $\,V\,$ is not an exceptional $\,\mathfrak g$-module.}

Indeed, if $\,\mu\,=\,\omega_1\,+\,\omega_4\in{\cal X}_{++}(V)$, then 
$\,r_p(V)\,\geq\,\displaystyle\frac{|W\mu|\cdot
|R_{long}^+-R^+_{\mu,p}|}{2\cdot 5}\,=\,\frac{2^6\cdot 5\cdot
4}{2\cdot 5}\,=\,2^7\,>\,2\cdot 5^2\,$.
If $\,\mu\,=\,\omega_1\,+\,\omega_5\in{\cal X}_{++}(V)$, then 
$\,r_p(V)\,\geq\,\frac{2^5\cdot 5\cdot
5}{2\cdot 5}\,=\,2^4\cdot 5\,>\,2\cdot 5^2\,$.
If $\,\mu\,=\,\omega_1\,+\,\omega_3\in{\cal X}_{++}(V)$, then 
$\,r_p(V)\,\geq\,\frac{2^4\cdot 3\cdot 5\cdot
3}{2\cdot 5}\,=\,2^3\cdot 3^2\,>\,2\cdot 5^2\,$. Hence in all these
cases, by~\eqref{rrcl}, $\,V\,$ is not exceptional, proving the
claim.\hfill $\Box$ \vspace{1.5ex}

II) Let $\,\mu\,=\,a\,\omega_{i}\,+\,b\,\omega_j\,$ (with $\,a\geq
1,\,b\geq 1\,$ and $\,1\leq i<j\leq 5\,$) be a good weight in 
$\,{\cal X}_{++}(V)$.

{\bf Claim 3}:\label{cl3c5} {\it Let $\,V\,$ be a $\,C_5(K)$-module.
If $\,{\cal X}_{++}(V)\,$ contains a good weight $\,\mu\,=\,a\omega_{i}\, 
+\,b\omega_{j}\,$ (with $\,2\leq i<j \leq 4\,$ and $\,a\geq 1,\,b\geq
1$), then $\,V\,$ is not an exceptional $\,\mathfrak g$-module.} Indeed:

i) for ($\,a\geq 2,\,b\geq 1\,$) or ($\,a\geq 1,\,b\geq 2\,$),
$\,\mu_1\,=\,\mu-(\alpha_i+\cdots +\alpha_j)\,=\,\omega_{i-1}\,+\,(a-1)
\,\omega_{i}\,+\,(b-1)\,\omega_{j}\,+\,\omega_{j+1}\in
\,{\cal X}_{++}(V)\,$ is a good weight with $3$ or more nonzero
coefficients. Hence, Lemma~\ref{cn5} applies.

ii) Let $\,a=b=1$. If $\,\mu\,=\,\omega_{2}\,+\,\omega_{3}\,$, then
$\,\mu_1\,=\,\omega_{1}\,+\,\omega_{4}\in{\cal X}_{++}(V)$.
If $\,\mu\,=\,\omega_{2}\,+\,\omega_{4}\,$, then
$\,\mu_1\,=\,\omega_{1}\,+\,\omega_{5}\in{\cal X}_{++}(V)$. In both cases
$\,\mu_1\,$ satisfies \mbox{Claim 2}.  
If $\,\mu\,=\,\omega_{3}\,+\,\omega_{4}\,$, then
$\,\mu_1\,=\,\omega_{2}\,+\,\omega_{5}\,$ and $\,\mu_2\,=\,\mu_1-
(\alpha_2+\cdots +\alpha_5)\,=\,\omega_{1}\,+\,\omega_{4}
\in{\cal X}_{++}(V)$. Hence Claim 2 applies. This proves Claim 3.\hfill $\Box$ 
\vspace{1.5ex}

{\bf Claim 4}: \label{cl4c5} {\it Let $\,V\,$ be a $\,C_5(K)$-module.
If $\,{\cal X}_{++}(V)\,$ contains a good weight 
$\,\mu\,=\,a\omega_{i}\,+\,b\omega_{5}\,$ (with $\,1\leq i \leq 4\,$
and $\,a\geq 1,\,b\geq 1$), then $\,V\,$ is not an exceptional
$\,\mathfrak g$-module.} Indeed, 

i) let $\,\mu\,=\,a\omega_{4}\,+\,b\omega_{5}\,$. Then
$\,\mu_1\,=\,\mu-(\alpha_4+\alpha_{5})\,=\,\omega_3\,+\,a\,\omega_{4}\,+
\,(b-1)\omega_{5}\,\in\,{\cal X}_{++}(V)$. For $\,a\geq 1,\,b\geq 2,\,$
$\,\mu_1\,$ satisfies Lemma~\ref{cn5}. For $\,a\geq 1,\,b=1,\,$ 
$\,\mu_1\,$ satisfies Claim 3. Hence $\,V\,$ is not exceptional.

ii) Let $\,\mu\,=\,a\omega_{i}\,+\,b\omega_{5}\,$ (with $\,2\leq i \leq 3\,$
and $\,a\geq 1,\,b\geq 1$). Then $\,\mu_1\,=\,\mu-(\alpha_i+\cdots 
+\alpha_5)\,=\,\omega_{i-1}\,+\,(a-1)\omega_{i}\,+\,\omega_{4}\,+\,(b-1)
\omega_{5}\in\,{\cal X}_{++}(V)$. 

ii.1) For ($\,a\geq 1,\,b\geq 2\,$) or 
($\,a\geq 2,\,b\geq 1\,$), $\,\mu_1\,$ satisfies Lemma~\ref{cn5}. 

ii.2) Let $\,a=b=1\,$. If $\,\mu\,=\,\omega_{2}\,+\,\omega_{5}\,$,
then $\,\mu_1\,=\,\omega_{1}\,+\,\omega_{4}\in{\cal X}_{++}(V)\,$
satisfies Claim 2. If $\,\mu\,=\,\omega_{3}\,+\,\omega_{5}\,$, then
$\,\mu_1\,=\,\omega_{2}\,+\,\omega_{4}\,\in{\cal X}_{++}(V)\,$
satisfies Claim 3.

iii) Let $\,\mu\,=\,a\omega_{1}\,+\,b\omega_{5}\,$ with $\,a\geq
1,\,b\geq 1$. Then:

iii.1) for $\,a\geq 1,\,b\geq 2\,$, $\,\mu_1\,=\,\mu-(\alpha_1+\cdots 
+\alpha_5)\,=\,(a-1)\omega_{1}\,+\,\omega_{4}\,+\,(b-1)
\omega_{5}\in\,{\cal X}_{++}(V)$. For $\,a\geq 2,\,b\geq 2\,$, 
Lemma~\ref{cn5} applies. For $\,a=1,\,b\geq 2\,$, $\,\mu_1\,$
satisfies case i) of this proof.

iii.2) For $\,a\geq 2,\,b=1\,$, $\,\mu\,=\,a\omega_{1}\,+\,\omega_{5}\,$
and $\,\mu_1\,=\,\mu-\alpha_1\,=\,(a-2)\,\omega_{1}\,+\,\omega_2\,+\,
\omega_{5}\,\in\,{\cal X}_{++}(V)$. For $\,a\geq 3\,$ 
Lemma~\ref{cn5} applies. For $\,a=2\,$, $\,\mu_1\,$ satisfies case ii)
of this proof. 

iii.3) For $\,a=b=1\,$, $\,\mu\,=\,\omega_{1}\,+\,\omega_{5}\,$
satisfies Claim~2 (p. \pageref{cl2c5}). This proves Claim 4.
\hfill $\Box$ \vspace{1.5ex}

{\bf Therefore, if $\,{\cal X}_{++}(V)\,$ contains a good weight 
$\,\mu\,=\,a\omega_{i}\,+\,b\omega_{j}\,$ (with $\,a\geq 1,\,b\geq 1$), then 
we can assume $\,i=1\,$ and $\,2\leq j\leq 4\,$.}
These cases are treated as follows.

(a) Let $\,\mu\,=\,a\omega_{1}\,+\,b\omega_{j}\,$ (with $\,j=3,\,4\,$)
be a good weight in $\,{\cal X}_{++}(V)\,$. Then:

i) for $\,a\geq 1,\,b\geq 2$, $\,\mu_1\,=\,\mu-\alpha_j\, 
=\,a\omega_1\,+\,\omega_{j-1}\,+\,(b-2)\omega_{j}\,+\,\omega_{j+1}\,
\in\,{\cal X}_{++}(V)\,$ satisfies Lemma~\ref{cn5}. 

ii) For $\,a\geq 2,\,b=1$, $\,\mu\,=\,a\omega_{1}\,+\,\omega_{j}\,$ and
$\,\mu_1\,=\,\mu-\alpha_1\,=\,(a-2)\omega_1\,+\,\omega_{2}\,+\,\omega_{j}\,
\in\,{\cal X}_{++}(V)$. For $\,a\geq 3$, $\,\mu_1\,$ satisfies
Lemma~\ref{cn5}. For $\,a= 2$, $\,\mu_1\,$ satisfies \mbox{Claim 3}
(p. \pageref{cl3c5}).

iii) For $\,a=b=1$, $\,\mu\,=\,\omega_{1}\,+\,\omega_{j}\,$ (with
$\,j=3,\,4\,$) satisfies Claim~2(p. \pageref{cl2c5}).

(b) Let $\,\mu\,=\,a\omega_{1}\,+\,b\omega_{2}\,$ with $\,a\geq
1,\,b\geq 1\,$. Then $\,\mu_1\,=\,\mu-(\alpha_1+\alpha_2)\, 
=\,(a-1)\omega_1\,+\,(b-1)\omega_{2}\,+\,\omega_{3}\,\in\,{\cal X}_{++}(V)$.

i)\label{IIbitp} For $\,a\geq 2,\,b\geq 2$, $\,\mu_1\,$
satisfies Lemma~\ref{cn5}. For $\,a=1,\,b\geq 2$, $\,\mu_1\,$
satisfies Claim 3 (p. \pageref{cl3c5}). For $\,a\geq 2,\,b=1$, $\,\mu_1\,$
satisfies case (a) above. 

ii) Let $\,a=b=1\,$. We may assume that $\,V\,$ has highest weight
$\,\mu\,=\,\omega_{1}\,+\,\omega_{2}\,$. Then
$\,\mu_1\,=\,\omega_{3}\,\in\,{\cal X}_{++}(V)$. By Smith's
Theorem~\cite{smith}, for $\,p\neq 3,\,$ $\,m_{\mu_1}=2\,$. Thus, for
$\,p\geq 5,\,$ $\,r_p(V)\,\geq\,\displaystyle\frac{2^4\cdot 5\cdot
2}{2\cdot 5}\,+\,2\cdot\frac{2^4\cdot 5\cdot 3}{2\cdot
5}\,=\,2^6\,>\,2\cdot 5^2\,$. Hence, \mbox{by~\eqref{rrcl},} $\,V\,$ is not
exceptional. For $\,p=3\,$, by Claim~12 (p. \pageref{claim12al})
of First Part of Proof of Theorem~\ref{anlist}, $\,V\,$ is not exceptional.

{\bf Hence if $\,V\,$ is an $\,C_5(K)$-module such that 
$\,{\cal X}_{++}(V)\,$ contains weights with $2$ nonzero coefficients,
then $\,V\,$ is not an exceptional module.} 

{\bf From now on we can assume that $\,{\cal X}_{++}(V)\,$ contains only 
weights with at most one nonzero coefficient.}
 
III) Let $\,\mu\,=\,a\,\omega_{i}\,$ (with $\,a\geq 1\,$) be a weight
in $\,{\cal X}_{++}(V)$. 

(a)i) Let $\,a\geq 2\,$ and $\,2\leq i\leq 4$. Then 
$\,\mu_1\,=\,\mu-\alpha_i\,=\,\omega_{i-1}\,+\,(a-2)\omega_i\,+\,\omega_{i+1}
\,\in\,{\cal X}_{++}(V)$. For $\,a\geq 3\,$, $\,\mu_1\,$ satisfies  
Lemma~\ref{cn5} (for $\,\mu\,$ good or bad). For $\,a=2\,$, by Claims
2,\,3 or 4 of this proof, $\,V\,$ is not exceptional.

ii) Let $\,a\geq 2\,$ and $\,\mu\,=\,a\,\omega_{5}\,$. Then 
$\,\mu_1\,=\,\mu-\alpha_5\,=\,2\omega_4\,+\,(a-2)\,\omega_5\;$ and
$\,\mu_2\,=\,\mu_1-\alpha_4\,=\,\omega_3\,+\,(a-1)\omega_5\,\in
\,{\cal X}_{++}(V)$. As $\,\mu_2\,$ satisfies Claim~2
(p. \pageref{cl2c5}) or Claim~4 (p. \pageref{cl4c5}), 
$\,V\,$ is not an exceptional module.

iii) Let $\,a\geq 2\,$ and $\,\mu\,=\,a\,\omega_{1}\,$. Then
$\,\mu_1\,=\,\mu-\alpha_1\,=\,(a-2)\omega_1\,+\,\omega_2\,\in 
\,{\cal X}_{++}(V)$. For $\,a\geq 4$, $\,\mu_1\,$ satisfies case
II(b)i) (p. \pageref{IIbitp}) of this proof.

For $\,a=3\,$ and $\,p\geq 5\,$ we may assume that $\,V\,$ has highest weight
$\,3\omega_1\,$. By Lemma~\ref{III.iii} (p. \pageref{III.iii}), 
$\,V\,$ is not exceptional.
For $\,p=3\,$, $\,\mu\,\in\,{\cal X}_{++}(V)\,$ only if $\,V\,$ has
highest weight $\,\lambda\,=\,\omega_1\,+\,2\omega_2\,$ or 
$\,\lambda\,=\,(3-k)\omega_1\,+\,k\omega_3\,$ (for some
$\,1\leq k\leq 2\,$). In these cases, by II(b) or II(a) of this proof, 
$\,V\,$ is not an exceptional module.

Let $\,a=2\,$. We may assume that $\,V\,$ has highest weight 
$\,2\,\omega_{1}\,$. Then 
$\,V\,$ is the adjoint module, which is exceptional by Example~\ref{adjoint}
{\bf (N. 1 in Table~\ref{tableclall})}.

{\bf Hence if $\,{\cal X}_{++}(V)\,$ contains  
$\,\mu\,=\,a\,\omega_{i}\,$ (with $\,a\geq 1\,$ and $\,1\leq i\leq 5\,$), 
then we can assume $\,a=1\,$ and that $\,V\,$ has highest weight 
$\,\omega_{i}\,$.} These cases are treated in the sequel. 

(b)i) Suppose that $\,V\,$ has highest weight $\,\omega_{3}\,$ (resp.,
$\,\omega_{4}\,$, $\,\omega_{5}\,$), then $\,V\,$ is unclassified
{\bf (N. 2 (resp., 3, 4) in Table~\ref{leftcn})}.

ii) If $\,V\,$ has highest weight $\,\omega_{2}\,$ then, by case
III(b)vi) (p. \pageref{IIIbvi}) of First Part of Proof of Theorem~\ref{listcn},
$\,V\,$ is an exceptional module {\bf (N. 3 in Table~\ref{tableclall})}.
 
iii) If $\,V\,$ has highest weight $\,\omega_{1}\,$ then, by case
III(b)vii) (p. \pageref{IIIbvii}) of First Part of Proof of 
Theorem~\ref{listcn}, $\,V\,$ is an exceptional module {\bf (N. 2 in  
Table~\ref{tableclall})}.

{\bf Hence if $\,V\,$ is an $\,C_5(K)$-module such that 
$\,{\cal X}_{++}(V)\,$ contains $\,\mu\,=\,a\,\omega_i\,$ with
($\,a\geq 3\,$ and $\,1\leq i\leq 5\,$) or ($\,a\geq 2\,$ and  
$\,2\leq i\leq 5\,$), then $\,V\,$ is not exceptional. If $\,V\,$ has
highest weight $\,\lambda\in\{ 2\omega_1,\,\omega_1,\,\omega_2\}\,$,  
then $\,V\,$ is an exceptional module. If $\,V\,$ has highest weight 
$\,\lambda\in\{ \omega_3,\,
\omega_4,\,\omega_{5}\}\,$, then $\,V\,$ is unclassified.}

This finishes the proof of Theorem~\ref{listcn} for $\,\ell=5$.
\hfill $\Box$

%\newpage
%\input{dndndn.tex}

\newpage
\subsection{Groups or Lie Algebras of type $\,D_{\ell}$}\label{appdl}

In this section we prove Theorem~\ref{dnlist}. For groups of type  
$\,D_{\ell}\,$, $\,|W|\,=\,2^{\ell-1}\ell!$,
$\,|R|\,=\,|R_{long}|\,=\,2\ell(\ell-1)\,$.

\subsubsection{Type $\,D_{\ell}\,,\;\ell\geq 5$}\label{fppdl}

{\bf Proof of Theorem~\ref{dnlist} - First Part} \\
Let $\,\ell\geq 5$. 
By Lemma~\ref{3cdn}, {\bf it suffices to consider $\,D_{\ell}(K)$-modules 
$\,V\,$ such that $\,{\cal X}_{++}(V)\,$ contains only weights with at
most $2$ nonzero coefficients.}

I) {\bf Claim 1}:\label{claim1dl} {\it Let $\,V\,$ be a $\,D_{\ell}(K)$-module.
If $\,{\cal X}_{++}(V)\,$ contains a bad weight with $2$ nonzero 
coefficients, then $\,V\,$ is not an exceptional $\,\mathfrak g$-module.}

Indeed, let $\,\mu\,=\,a\,\omega_{i}\,+\,b\,\omega_{j}\,$ (with 
$\,1\leq i<j\leq \ell\,$ and $\,a,\,b\,\geq 1\,$) be a bad weight in
$\,{\cal X}_{++}(V)\,$ (hence $\,a\geq 2,\,b\,\geq 2$).

(a) For $\,1\leq i<j <\ell-2,\,$ 
$\,\mu_1\,=\,\mu-(\alpha_i+\cdots +\alpha_j)\,=\,\omega_{i-1}\,+\,(a-1)
\,\omega_{i}\,+\,(b-1)\,\omega_{j}\,+\,\omega_{j+1}\in\,{\cal X}_{++}(V)$.

(b) For $\,1\leq i<\,j =\ell-2,\,$
$\,\mu\,=\,a\,\omega_{i}\,+\,b\,\omega_{\ell -2}\;$ and
$\,\mu_1\,=\,\mu-(\alpha_i+\cdots +\alpha_{\ell -2})\,=\,
\omega_{i-1}\,+\,(a-1)\,\omega_{i}\,+\,(b-1)\,\omega_{\ell-2}\,+\, 
\omega_{\ell-1}\,+\,\omega_{\ell}\in\,{\cal X}_{++}(V)$.

(c) For $\,1\leq\, i\,\leq\,\ell-2,\;j =\ell-1\,$ (or by a graph-twist 
$\,1\leq\, i\,\leq\,\ell-2,\;j=\ell\,$), $\,\mu\,=\,a\omega_{i}\,+\, 
b\omega_{\ell-1}\,$ (or $\,\mu\,=\,a\omega_{i}\,+\,b\omega_{\ell}\,$) and 
$\,\mu_1\,=\,\mu-(\alpha_i+\cdots +\alpha_{\ell-1})\,
=\,\omega_{i-1}\,+\,(a-1)\omega_{i}\,+\,(b-1)\omega_{\ell-1}\,+
\,\omega_{\ell}\,$
(or $\,\mu_1\,=\,\mu-(\alpha_i+\cdots +\alpha_{\ell-2}+\alpha_{\ell})\,
=\,\omega_{i-1}\,+\,(a-1)\omega_{i}\,+\,\omega_{\ell-1}\,+
\,(b-1)\omega_{\ell}\,$) $\,\in\,{\cal X}_{++}(V)$.

(d) For $\,i =\ell-1,\,j=\ell,\,$ $\,\mu\,=\,a\,\omega_{\ell-1}\,+\, 
b\,\omega_{\ell}\,$, $\,\mu_1\,=\,\mu-\alpha_{\ell-1}\,
=\,\omega_{\ell-2}\,+\,\mbox{$(a-2)\omega_{\ell-1}$}\,+\,b\omega_{\ell}\;$
and $\,\mu_2\,=\,\mu_1-(\alpha_{\ell-2}+\alpha_{\ell})\,
=\,\omega_{\ell-3}\,+\,(a-1)\omega_{\ell-1}\,+\,
\mbox{$(b-1)\omega_{\ell}$}\,\in\,{\cal X}_{++}(V)$.

In all these cases, $\,\mu_1\,$ or $\,\mu_2\,$
is a good weight in $\,{\cal X}_{++}(V)$,  with $3$ or more nonzero 
coefficients. Hence Lemma~\ref{3cdn} applies. This proves the Claim. 
%(Note that this claim is also true for $\,\ell=4$.)
\hfill $\Box$ \vspace{1ex}

{\bf Therefore, we can assume that weights with $2$ nonzero
coefficients occurring in $\,{\cal X}_{++}(V)\,$ are good weights.}

II) Let $\,\mu\,=\,a\,\omega_{i}\,+\,b\,\omega_{j}\,$ with  
$\,1\leq i<j\leq\ell\,$ and $\,a,\,b\,\geq 1\,$ be a good weight 
in $\,{\cal X}_{++}(V)$.

{\bf Claim 2}:\label{claim2dl} {\it Let $\,\ell\geq 5$. If $\,V\,$ is
a $\,D_{\ell}(K)$-module such that $\,{\cal X}_{++}(V)\,$
contains a weight $\,\mu\,=\,a\,\omega_{i}\,+\,b\,\omega_{\ell-2}\,$
(with  $\,1< i\leq \ell-3\,$ and $\,a,\,b\,\geq 1\,$),
then $\,V\,$ is not an exceptional $\,\mathfrak g$-module.}

Indeed, for $\,1\,<\,i\,<\,j=\ell-2\,$ and $\,a\geq 1,\,b\,\geq 1$, 
$\,\mu_1\,=\,\mu-(\alpha_i+\cdots +\alpha_{\ell-2})\,=\,\omega_{i-1}\,+\,
(a-1)\,\omega_i\,+\,(b-1)\,\omega_{\ell-2}\,+\,\omega_{\ell-1}\,+\,
\omega_{\ell}\in\,{\cal X}_{++}(V)$. As $\,\mu_1\,$ is a good weight  
with $3$ or more nonzero coefficients, Lemma~\ref{3cdn} applies.

Let $\,i=1,\;j=\ell-2\,$ and $\,\mu\,=\,a\,\omega_{1}\,+\, 
b\,\omega_{\ell-2}\,$. Then $\,\mu_1\,=\mu-(\alpha_1+\cdots 
+\alpha_{\ell-2}) =\,(a-1)\,\omega_1\,+\,(b-1)\,\omega_{\ell-2}\, 
+\,\omega_{\ell-1}\,+\,\omega_{\ell}\,\in\,{\cal X}_{++}(V)$.
For ($\,a\geq 2,\,b\geq 1\,$) or \mbox{($\,a\geq 1,\,b\geq 2\,$)},
$\,\mu_1\,$ is a good weight with 3 or more nonzero coefficients, hence
Lemma~\ref{3cdn} applies. For $\,a=b=1,\,$ $\,\mu\,$ and
$\,\mu_1\,$ are good weights in $\,{\cal X}_{++}(V)$. Thus for $\,\ell\geq 5$,
$\,s(V)\,\geq\,|W\mu|\,+\,|W\mu_1|\,= 
%\displaystyle\frac{2^{\ell-1}\,\ell!}{(\ell-3)!\,2!\,2!}\,+\, 
%\frac{2^{\ell-1}\,\ell!}{(\ell-1)!}\,=
\,2^{\ell -3}\,\ell(\ell -1)(\ell-2)\,+\,2^{\ell-1}\,\ell\,=\,
2^{\ell -3}\,\ell[(\ell -1)(\ell-2)\,+\,4]\,>\,2\,\ell^2\,(\ell-1)\,$.
Hence, by~(\ref{dllim}), $\,V\,$ is not an exceptional module. This
proves the claim.\hfill $\Box$ \vspace{1.5ex}

{\bf Claim 3}:\label{claim3dl} {\it Let $\,1\leq i \leq j,\;3\leq
j\leq \ell-3\,$ (hence $\,\ell\geq 6\,$). If $\,V\,$ is
a $\,D_{\ell}(K)$-module such that $\,{\cal X}_{++}(V)\,$ contains 
a good weight $\,\mu\,=\,a\,\omega_{i}\,+\,b\,\omega_{j}\,$ (with  
$\,a\geq 1,\,b\geq 1\,$),  
then $\,V\,$ is not an exceptional $\,\mathfrak g$-module.}

Indeed, for $\,3\leq j\leq \ell-3$, $\displaystyle\binom{j}{i}\geq 3\,$ 
and $\,\displaystyle\binom{\ell}{j}\geq\binom{\ell}{3}$. 
Thus for $\,\ell\geq 6$, $\,s(V)\,\geq\,|W\mu| =\displaystyle 
2^j\,\binom{j}{i}\,\binom{\ell}{j}\,\geq\, 2^3\cdot 3\,\binom{\ell}{3}
\, \geq\, 2^2\,\ell\,(\ell-1)\,(\ell-2)\, > 2\,\ell^2\,(\ell-1)\,$. 
Hence, by~\eqref{dllim}, $\,V\,$ is not an exceptional module. This proves the
claim.\hfill $\Box$ \vspace{1.5ex}

{\bf Therefore, for $\,\ell\geq 5$, if $\,{\cal X}_{++}(V)\,$ contains
a good weight $\,\mu\,=\,a\,\omega_{i}\,+\,b\,\omega_{j}\,$ (with
$\,a,\,b\,\geq 1\,$), then we can assume that ($\,i=1,\,j=2\,$) or
(\mbox{$\,1\leq i < j$,} \mbox{$\ell-1\leq j \leq\ell\,$}).} These
cases are treated in the sequel.

(a) Let $\,\mu\,=\,a\,\omega_{1}\,+\,b\,\omega_{2}\,$ (with $\,a\geq
1,\,b\geq 1\,$). Then $\,\mu_1\,=\,\mu-(\alpha_1+\alpha_{2})\,=\,(a-1)
\omega_1\,+\,(b-1)\omega_2\,+\,\omega_{3}\,\in\,{\cal X}_{++}(V)$.

i) For $\,a\geq 2,\,b\geq 2,\,$ $\,\mu_1\,$ has 3 nonzero
coefficients, hence Lemma~\ref{3cdn} applies. For ($\,a\geq
2,\,b=1\,$) or ($\,a=1,\,b\geq 2\,$), $\,\mu_1\,$ satisfies Claim 3
(for $\,\ell\geq 6\,$) or Claim 2 (for $\,\ell=5\,$).

ii) Let $\,a=b=1\,$ and $\,\mu\,=\,\omega_{1}\,+\,\omega_{2}\,$. Then
$\,\mu_1\,=\,\omega_{3}\;$ and $\,\mu_2\,=\,\mu_1-(\alpha_2+2\alpha_{3}+ 
2\alpha_4\,+\cdots\,+2\alpha_{\ell-2}+\alpha_{\ell-1}+\alpha_{\ell})\,
=\,\omega_{1}\,\in\,{\cal X}_{++}(V)$. For $\,p\geq 3,\,$ 
$\,|R_{long}^+-R^+_{\mu,p}|\geq 4\ell-7$,
$\,|R_{long}^+-R^+_{\mu_1,p}|=3(2\ell-5),\,$
$\,|R_{long}^+-R^+_{\mu_2,p}|=2(\ell-1)$. Thus, for $\,\ell\geq 5$, 
\begin{align*}
r_p(V) &\geq\,\displaystyle
\frac{2^2\ell(\ell-1)\cdot (4\ell-7)}{2\ell(\ell-1)}\,+\,
\frac{2^3\,\ell(\ell-1)(\ell-2)\cdot 3(2\ell-5)}{3!\cdot 2\ell(\ell-1)}\,+\,
\frac{2\,\ell\cdot 2(\ell-1)}{2\ell(\ell-1)}\vspace{1ex}\\
 & =\,2(4\ell-7)\,+\,2(\ell-2)(2\ell-5)\,+\,2\,=\, 4\ell^2\,-\,10\ell\,+\, 
8\,>\,2\ell\,(\ell-1)\,.
\end{align*}
Hence for $\,p\geq 3\,$, by~(\ref{rrdl}), $\,V\,$ is not exceptional. 
For $\,p=2\,$, we may assume that $\,V\,$ has highest weight 
$\,\omega_{1}\,+\,\omega_{2}\,$. By Smith's Theorem~\cite{smith},
$\,m_{\mu_1}=2\,$. Thus for $\,\ell\geq 5$,  
\[
s(V)\,\geq\,2^2\ell(\ell-1)\,+\,2\cdot\displaystyle
\frac{2^3\ell(\ell-1)\,(\ell-2)}{3!}\,=\,\frac{4\ell(\ell-1)\,(2\ell-1)}{3}\,
>\,2\ell^2(\ell-1)\,.
\]
Hence, by~(\ref{dllim}), $\,V\,$ is not an exceptional module.
%Note that $\,\gamma\in R_{long}^+\,$ implies
%$\,\gamma=\varepsilon_i\,-\,\varepsilon_j,\,$ for $\,1\leq i<j\leq
%\ell\,$ or $\,\gamma=\varepsilon_i\,+\,\varepsilon_{\ell},\,$ 
%for $\,1\leq i<\ell\,$ or $\,\gamma=\varepsilon_i\,+\,\varepsilon_{j},\,$
%for $\,1\leq i<j<\ell.\,$  $\,R^+_{\mu,3}\,=\,\{\gamma\in
%R_{long}^+\,/\,(2\varepsilon_1+\varepsilon_2,\gamma)\equiv
%0\pmod 3\},\,$ so $\,|R_{long}^+-R^+_{\mu,3}|=4\ell-7.\,$\\
%$\,R^+_{\mu_1,3}\,=\,\{\gamma\in R_{long}^+\,/\,(\varepsilon_1+
%\varepsilon_2+\varepsilon_3, \gamma)\equiv
%0\pmod 3\},\,$ so $\,|R_{long}^+-R^+_{\mu_1,3}|=3(2\ell-5).\,$\\
%$\,R^+_{\mu_2,3}\,=\,\{\gamma\in R_{long}^+\,/\,(\varepsilon_1,\gamma)\equiv
%0\pmod 3\},\,$ so $\,|R_{long}^+-R^+_{\mu_2,3}|=2(\ell-1).\,$\\
%%

(b) Let $\,1\leq i\leq \ell-2,\;j=\ell-1\,$ (or by a graph-twist 
$\,1<i\leq \ell-2,\;\mbox{$j=\ell$}\,$) and $\,\mu\,=\,a\,\omega_{i}\,+\,b\, 
\omega_{\ell-1}\,$ (resp., $\,\mu\,=\,a\,\omega_{i}\,+\,b\,\omega_{\ell}\,$).
For $\,a\geq 1,\,b\geq 1,\,$ $\,\mu_1\,=\mu-(\alpha_i+\cdots + 
\alpha_{\ell-1})=\,\omega_{i-1}\,+\,(a-1)\omega_i\,+\,(b-1)\omega_{\ell-1}
\,+\,\omega_{\ell}\,$ (resp., $\,\mu_1\,=\mu-(\alpha_i+\cdots 
+\alpha_{\ell-2}+ \alpha_{\ell})=\,\omega_{i-1}\,+\,(a-1)\omega_i\, 
+\,\omega_{\ell-1}\,+\,(b-1)\omega_{\ell}\,$). 

i)\label{IIbidl} For $\,2\leq i\leq \ell-2\,$,
$\,|W\mu|\,+\,|W\mu_1|\,\geq\,\displaystyle
2^{\ell-1}\left[\binom{\ell}{i}\,+\,\binom{\ell}{i-1}\right]\,= 
\,2^{\ell-1}\,\binom{\ell+1}{i}\,$ (by Theorem~\ref{usual}(b)). Thus
for $\,\ell\geq 5$, 
\[
s(V)\,\geq\,\displaystyle 2^{\ell-1}\,\binom{\ell+1}{i}\,\geq\,
2^{\ell-2}\,(\ell+1)\ell\,>\,2\,\ell^2\,(\ell -1)\,.
\]
Hence, by~\eqref{dllim}, $\,V\,$ is not an exceptional module. 

ii) Let $\,i=1,\,j=\ell-1\,$ (or by a graph-twist $\,i=1,\,j=\ell\,$)
and $\,\mu\,=\,a\,\omega_{1}\,+\,b\,\omega_{\ell-1}\;$ (resp.,
$\,\mu\,=\,a\,\omega_{1}\,+\,b\,\omega_{\ell}\,$).  

ii.1) For $\,a\geq 2,\,b\geq 1,\,$ $\,\mu_1\,=\mu-\alpha_{1}=\, 
(a-2)\omega_{1}\,+\,\omega_{2}\,+\,b\omega_{\ell-1}\,$  
(resp., $\,\mu_1\,=\mu-\alpha_{1}=\,(a-2)\omega_{1}\,+\,\omega_{2}\,+\,
b\omega_{\ell}\,$) $\,\in{\cal X}_{++}(V)$. For $\,a\geq 3,\,$
Lemma~\ref{3cdn} applies. For $\,a=2$, $\,\mu_1\,$ satisfies case
(b)i) above.

ii.2) For $\,a\geq 1,\,b\geq 2,\,$ $\,\mu_1\,=\,\mu-\alpha_{\ell-1}\,=\, 
a\omega_{1}\,+\,\omega_{\ell-2}\,+\,(b-2)\omega_{\ell-1}\,$
(or $\,\mu_1\,=\,\mu-\alpha_{\ell}\,=\,a\omega_{1}\,+\,\omega_{\ell-2}\,+\,
(b-2)\omega_{\ell}\,$) $\,\in\,{\cal X}_{++}(V)$. For $\,a\geq 3,\,$
Lemma~\ref{3cdn} applies. For $\,a=2$, $\,\mu_1\,$ satisfies Claim 2 
(p. \pageref{claim2dl}).
 
ii.3) Let $\,a=b=1\,$ and $\,\mu\,=\,\omega_{1}\,+\,\omega_{\ell-1}\,$
(or by a graph-twist $\,\mu\,=\,\omega_{1}\,+\,\omega_{\ell}\,$). 

ii.3.a)For $\,\ell\geq 8$, $\,s(V)\,\geq\,|W\mu|\,=\,2^{\ell-1}\,\ell\,>\, 
2\,\ell^2\,(\ell-1)\,$. Hence, by~\eqref{bllim}, 
\mbox{$\,V\,$ is not an exceptional module}.

ii.3.b) For $\,\ell=7,$ $\,\mu\,=\,\omega_{1}\,+\,\omega_{6}\,$ and
$\,\mu_1\,=\,\omega_7\,\in\,{\cal X}_{++}(V)$. As
$\,|R_{long}^+-R^+_{\mu,p}|\,\geq\,21\,$ and 
$\,|R_{long}^+-R^+_{\mu_1,p}|=21$, one has $\,r_p(V) \geq\,\displaystyle 
\frac{2^6\cdot 7\cdot 21}{2\cdot 7\cdot 6}\,+\,\frac{2^6\cdot 21} 
{2\cdot 7\cdot 6}\,=\,2^7\,>\,2\cdot 7\cdot 6\,$.
Hence, by~(\ref{rrdl}), $\,V\,$ is not an exceptional module.

ii.3.c) For $\,\ell=6,\,$ $\,\mu\,=\,\omega_{1}\,+\,\omega_{5}\,$ and
$\,\mu_1\,=\,\omega_6\,\in\,{\cal X}_{++}(V)$. For $\,p\geq 3$,
$\,|R_{long}^+-R^+_{\mu,p}|=20\,$ and $\,|R_{long}^+-R^+_{\mu_1,p}|=15$. Thus
$\,r_p(V) \geq\,\displaystyle\frac{2^5\cdot 6\cdot 20}{2\cdot 6\cdot
5}\,+\,\frac{2^5\cdot 15}{2\cdot 6\cdot 5}\,=\,2^3\cdot 3^2\,>\,2\cdot 6\cdot 5
\,$. Hence, by~(\ref{rrdl}), $\,V\,$ is not exceptional.  

For $\,p=2\,$, we may assume that $\,V\,$ has highest weight 
$\,\mu\,=\,\omega_{1}\,+\,\omega_{5}\,$. By Smith's
Theorem~\cite{smith}, $\,\mu_1\,=\,\omega_6\,$ has multiplicity
$\,m_{\mu_1}\,\geq\,2$. Now as $\,|R_{long}^+-R^+_{\mu,2}|\,=\, 
|R_{long}^+-R^+_{\mu_1,2}|=15$, one has $\,r_2(V)\geq
\displaystyle\frac{2^5\cdot 6\cdot 15}{2\cdot 6\cdot 5}\,+\,2\frac{2^5\cdot
15}{2\cdot 6\cdot 5}\,=\,2^6\,>\,2\cdot 6\cdot 5\,$. Hence,
by~(\ref{rrdl}), $\,V\,$ is not an exceptional module. 

%Consider the parabolic subgroup $P_J$ where $J =
%{\alpha_1,alpha_2,\alpha_3,\alpha_4,\alpha_5}$ to apply Smith's Theorem.
%Then $\omega_{1}\,+\,\omega_{5}\,$ corresponds to the adjoint module
%w.r.t. $ A_5\cong <e_{i}/ i=1....5>$, then $\,m_{\mu_1}\,=\,5$ for 
%$p\neq 2,3$ (that divide $\ell+1=6$ in this case) or 4 otherwise.
%veja outra(original??) explicacao abaixo em %%

ii.3.d) For $\,\ell=5,\,$ we may assume that $\,V\,$ has highest weight  
$\,\mu\,=\,\omega_{1}\,+\,\omega_{4}\,$ Then $\,\mu_1\,=\,\omega_{5}\,
\in\,{\cal X}_{++}(V)$. By Smith's Theorem~\cite{smith}, for $\,p\neq 5$,
$\,m_{\mu_1}=4$. For $\,p\neq 2$, $\,|R_{long}^+-R^+_{\mu,p}|\,=\,14,
\,|R_{long}^+-R^+_{\mu_1,p}|=10$. Thus, for $\,p\neq 2,\,5\,$,
$\,r_p(V) \,\geq\,\frac{2^4\cdot 5\cdot 14}{2\cdot 5\cdot 4}\,+\, 
4\frac{2^4\cdot 10}{2\cdot 5\cdot 4}\,=\,2^2\cdot 11\,>\,2\cdot 5\cdot 4\,$.
Hence, by~(\ref{rrdl}), $\,V\,$ is not exceptional. 
{\bf For $\,p=2\,$ or $\,p=5,\,$  %  $\,m_{\mu_1}\,=\,3\,$
these modules are unclassified (N. 1 in Table~\ref{leftdn})}.

(c) Let $\,i =\ell-1,\;j=\ell\,$ and $\,\mu\,=\,a\,\omega_{\ell-1}\, 
+\,b\,\omega_{\ell}\,$. Then $\,\mu_1\,=\,\mu-(\alpha_{\ell-2}+\alpha_{\ell-1}
+\alpha_{\ell})\,=\,\omega_{\ell-3}\,+\,(a-1)\,\omega_{\ell-1}\,+\,
(b-1)\,\omega_{\ell}\,\in\,{\cal X}_{++}(V)$.

i) For $\,a\geq 2,\,b\geq 2\,$, $\,\mu_1\,$ has $3$ nonzero
coefficients, hence Lemma~\ref{3cdn} applies.  
For ($\,a\geq 2,\,b=1\,$) or ($\,a=1,\,b\geq 2\,$), 
$\,\mu_1\,$ satisfies case (b)i).

ii) For $\,a=b=1\,$, $\,\mu\,=\,\omega_{\ell-1}\,+\,\omega_{\ell}\;$ and  
$\,\mu_1\,=\,\omega_{\ell-3}\,\in\,{\cal X}_{++}(V)$. 

ii.1) For $\,\ell\geq 7$, by~\eqref{dllim}, $\,V\,$ is not an
exceptional module, since
\[
s(V)\,\geq\,\displaystyle 2^{\ell-1}\,\ell\,+\,
\frac{2^{\ell-4}\,\ell(\ell-1)(\ell -2)}{3}\,=\,\frac{2^{\ell-4}\,\ell\,
[(\ell-1)(\ell -2)\,+\,24]}{3}\,>\,2\,\ell^2\,(\ell-1)\,.
\]

ii.2) For $\,\ell=6\,$, we may assume that $\,V\,$ has highest weight
$\,\mu\,=\,\omega_{5}\,+\,\omega_{6}\,$. Then $\,\mu_1\,=\,\omega_{3}\;$ and 
$\,\mu_2\,=\,\mu_1-(\alpha_2+2\alpha_3+2\alpha_4+\alpha_5+\alpha_6)\,=\,
\omega_1\,\in\,{\cal X}_{++}(V)$.
By Proposition~\ref{supruzal}, some $\,m_{\mu_i}\geq 2.\,$ Thus
$\,s(V)\,\geq\,2^5\cdot 6\,+\,m_{\mu_1}\,2^5\cdot 5\,+\,m_{\mu_2}\,2\cdot
6\,>\,2\cdot 6^2\cdot 5\,$. Hence, by~(\ref{dllim}), $\,V\,$ is not 
exceptional.

ii.3) For $\,\ell=5,\,$  $\,\mu\,=\,\omega_{4}\,+\,\omega_{5}\;$ and  
$\,\mu_1\,=\,\omega_{2}\,\in\,{\cal X}_{++}(V)$. For $\,p\neq 2,\,$
$\,|R_{long}^+-R^+_{\mu,p}|\,=\,14,\,\,|R_{long}^+-R^+_{\mu_1,p}|=13$. Thus
$\,r_p(V)\,\geq\,\displaystyle\frac{2^4\cdot 5\cdot 14}{2\cdot 5\cdot 4}\,
+\,\frac{2^3\cdot 5\cdot 13}{2\cdot 5\cdot 4}\,=\,41\,>\,2\cdot 5\cdot 4\,$.
Hence, by~(\ref{rrdl}), $\,V\,$ is not an exceptional module. 
{\bf For $\,p=2$, the $\,D_{\ell}(K)$-module $\,V\,$ of highest weight  
$\,\omega_{4}\,+\,\omega_{5}\,$ is unclassified (N. 2 in Table~\ref{leftdn})}.

{\bf Hence, for $\,\ell\geq 5\,$, if $\,{\cal X}_{++}(V)\,$ contains
a weight with $2$ nonzero coefficients, then $\,V\,$ is not an
exceptional module, unless $\,\ell=5\,$ and $\,V\,$ has highest weight
($\,\omega_{4}\,+\,\omega_{5}\,$ for $\,p=2\,$) or 
($\,\omega_{1}\,+\,\omega_{4}\,$ or $\,\omega_{1}\,+\,\omega_{5}\,$  
for $\,p=2,\,5\,$), in which cases $\,V\,$ is unclassified.}

{\bf From now on we can assume that $\,{\cal X}_{++}(V)\,$ contains 
only weights with at most one nonzero coefficient.} 

{\bf Claim 4}:\label{claim4dl}
{\it Let $\,\ell\geq 6\,$. If $\,V\,$ is a $\,D_{\ell}(K)$-module such
that $\,\omega_{\ell-2}\,\in {\cal X}_{++}(V)$, then $\,V\,$ is not  
an exceptional $\,\mathfrak g$-module.}

Indeed, $\,|W\omega_{\ell-2}|\,=\,2^{\ell-3}\,\ell\,(\ell-1)$. Thus 
for $\,\ell\geq 7$, $\,s(V)\geq \mbox{$2^{\ell-3}\,\ell\,(\ell-1)$}
>\,2\,\ell^2\,(\ell-1)\,$. Hence, by~\eqref{dllim}, $\,V\,$ is not 
exceptional. 
For $\,\ell=6,\,$ $\,\mu\,=\,\omega_4,\;\mu_1\,=\,\omega_2\,\in\,
{\cal X}_{++}(V)$. As $\,|R_{long}^+-R^+_{\mu,p}|\,\geq\,16,\;
|R_{long}^+-R^+_{\mu_1,p}|\,\geq\,15$, one has $\,r_p(V)\,\geq\,\displaystyle
\frac{2^4\cdot 3\cdot 5\cdot 16}{2\cdot 6\cdot 5}\,+\,\frac{2^2\cdot 3
\cdot 5\cdot 15}{2\cdot 6\cdot 5}\,=\,79\,>\,2\cdot 6 \cdot 5\,$.
Hence, by~(\ref{rrdl}), $\,V\,$ is not an exceptional module,
proving the claim. \hfill $\Box$ \vspace{2ex}

III) Let $\,\mu\,=\,a\,\omega_i\,$ (with $\,a\geq 1\,$) be a weight in 
$\,{\cal X}_{++}(V)$. 

(a)i) Let $\,a\geq 2\,$ and $\,2\,\leq\,i\,\leq\,\ell-3$. Then
$\,\mu_1\,=\,\mu-\alpha_{i}\,=\,\omega_{i-1}\,+\,(a-2)\,\omega_{i}\,+\,
\omega_{i+1}\,\in\,{\cal X}_{++}(V)$. For $\,a\geq 3\,$ (and $\,\mu\,$
good or bad), $\,\mu_1\,$ satisfies Lemma~\ref{3cdn}.
For $\,a=2\,$ and $\,2\leq i\leq \ell-4\,$, $\,\mu_1\,$ satisfies
Claim 3 (p. \pageref{claim3dl}). For $\,a=2\,$ and $\,i=\ell-3\,$, 
$\,\mu_1\,$ satisfies Claim 2 (p. \pageref{claim2dl}). 

ii) Let $\,a\geq 2\,$ and $\,\mu\,=\,a\,\omega_{\ell-2}\,$.
Then $\,\mu_1\,=\,\mu-\alpha_{\ell-2}\,=\,\omega_{\ell-3}\,+\,(a-2)\,
\omega_{\ell -2}\,+\,\omega_{\ell-1}\,+\,\omega_{\ell}\,\in
\,{\cal X}_{++}(V)\,$ has $3$ nonzero coefficients. 
Hence Lemma~\ref{3cdn} applies. % (for $\,\mu\,$ good or bad)

iii) Let $\,a\geq 2\,$ and $\,\mu\,=\,a\,\omega_{\ell-1}\,$ (or by a
graph-twist $\,\mu\,=\,a\,\omega_{\ell}\,$). Then 
$\,\mu_1\,=\,\mu-\alpha_{\ell-1}\,=\,\omega_{\ell-2}\,+\,(a-2)\,
\omega_{\ell-1}\,$ (resp., $\,\mu_1\,=\,\mu-\alpha_{\ell}\,=\, 
\omega_{\ell-2}\,+\,(a-2)\,\omega_{\ell}\,$) $\,\in\,{\cal X}_{++}(V)$. 
For $\,a\geq 3\,$, $\,\mu_1\,$ satisfies case II(b)i) (p. \pageref{IIbidl}). 
For $\,a=2\,$ and $\,\ell\geq 6$, $\,\mu_1\,$ satisfies Claim 4 
(p. \pageref{claim4dl}). {\bf For $\,\ell=5,\,$ if $\,V\,$ has highest  
weight $\,2\omega_4\,$ (or $\,2\omega_5\,$), then $\,V\,$ is
unclassified (N. 4 in Table~\ref{leftdn})}.
%$\,m_{\mu_2}=3,\,$ but this module is unclassified %??????? 

iv) Let $\,a\geq 2\,$ and $\,\mu\,=\,a\,\omega_1\,\in\,{\cal X}_{++}(V)$. Then:

iv.1) For $\,a\geq 4,\,$ $\,\mu_1\,=\,\mu-\alpha_{1}\,=\,(a-2)\omega_{1}\,+\,
\omega_{2}\,$, $\,\mu_2\,=\,\mu_1-(\alpha_{1}+\alpha_2)\,=\,(a-3) 
\omega_{1}\,+\,\omega_{3}\in{\cal X}_{++}(V)$. As $\,a\geq 4$, 
for $\,\ell\geq 6$, $\,\mu_2\,$ satisfies Claim 3
(p. \pageref{claim3dl}). For $\,\ell=5$, $\,\mu_2\,$ satisfies
Claim 2 (p. \pageref{claim2dl}).

iv.2) For $\,a=3,\,$ $\,\mu\,=\,3\,\omega_1,\;$ $\,\mu_1\,=\,\omega_{1}\, 
+\,\omega_{2},\;$ $\,\mu_2\,=\,\omega_{3}\;$ and $\,\mu_3\,=\,\omega_{1}\,\in\,
{\cal X}_{++}(V).\;$ For $\,p=3,\,$ $\,\mu\,\in\,{\cal X}_{++}(V)\,$ only if
$\,V\,$ has highest weight $\,\lambda\,=\,\omega_1\,+\,2\omega_2\,$ or
$\,\lambda\,=\,(3-k)\omega_1\,+\,k\omega_3\,$ (for some $\,1\leq k\leq
2$). In these cases, by Claims 2 or 3, $\,V\,$ is not exceptional.
For $\,p\geq 5,\,$ $\,|R_{long}^+-R^+_{\mu,p}|\,=\,
|R_{long}^+-R^+_{\mu_3,p}|=2(\ell-1),\,$
$\,|R_{long}^+-R^+_{\mu_1,p}|=2(2\ell-3),\,$ and
$\,|R_{long}^+-R^+_{\mu_2,p}|=3(2\ell-5)$. Thus, for $\,\ell\geq 5\,$,
$\,r_p(V) \geq\,2\,+\,4(2\ell-3)\,+\,2(\ell-2)(2\ell-5)\,+\,2\, =\,
4\ell^2\,-\,10\ell\,+\,12\,>\,2\ell\,(\ell-1)\,$. Hence, 
by~(\ref{rrdl}), $\,V\,$ is not an exceptional module.
%$\,R^+_{\mu,p}\,=\,\{\gamma\in R_{long}^+\,/\,(3\varepsilon_1,\gamma)\equiv
%0\pmod p\},\,$ so $\,|R_{long}^+-R^+_{\mu,p}|=2(\ell-1).\,$\\
%$\,R^+_{\mu_1,p}\,=\,\{\gamma\in R_{long}^+\,/\,(2\varepsilon_1+ 
%\varepsilon_2,\gamma)\equiv 0\pmod p\},\,$ so $\,|R_{long}^+- 
%R^+_{\mu_1,p}|=2(2\ell-3).\,$\\ $\,R^+_{\mu_2,p}\,=\,\{\gamma\in
%R_{long}^+\,/\,(\varepsilon_1+\varepsilon_2+\varepsilon_3, \gamma)\equiv
%0\pmod p\},\,$ so $\,|R_{long}^+-R^+_{\mu_2,p}|=3(2\ell-5).\,$\\
%$\,R^+_{\mu_3,p}\,=\,\{\gamma\in R_{long}^+\,/\,(\varepsilon_1,\gamma)\equiv
%0\pmod p\},\,$ so $\,|R_{long}^+-R^+_{\mu_3,p}|=2(\ell-1).\,$\\

iv.3) Let $\,a=2\,$. We may assume that $\,V\,$ has highest weight
$\,\mu\,=\,2\,\omega_1\,$. By the same argument used in case III(b)iv)
(p. \pageref{IIIbiv}) of First Part of Proof of
Theorem~\ref{listbn}, one proves that $\,V\,$ is not exceptional.

{\bf Therefore, for $\,\ell\geq 5\,$, if $\,V\,$ is a
$\,D_{\ell}(K)$-module such that $\,{\cal X}_{++}(V)\,$ contains a weight  
$\,\mu\,=\,a\,\omega_i\,$ (with $\,a\geq 1\,$), then we can assume $\,a=1\,$}. 

(b) Let $\,\mu\,=\,\omega_{i}\,$ be a weight in $\,{\cal X}_{++}(V)\,$.

i) For $\,i=\ell-2\,$, see Claim 4 (p. \pageref{claim4dl}) .

ii) Let $\,4\,\leq\,i\,\leq\,\ell-4\,$ (hence $\,\ell\geq 8\,$) and 
$\,\mu\,=\,\omega_{i}\,$. Then $\,|W\mu|\, =\, \displaystyle  
2^i\,\binom{\ell}{i}\,\geq\,\frac{2^4\,\ell(\ell-1)(\ell-2)(\ell
-3)}{4!}\,$. Thus for $\,\ell\geq 8$, by~\eqref{dllim}, $\,V\,$ is not
an exceptional module, since $\,
s(V)\,\geq\,\displaystyle\frac{2\,\ell\,(\ell-1)\,(\ell -2)(\ell-3)}{3}\,>\,
2\,\ell^2\,(\ell-1)\,$.
Hence, by~\eqref{dllim}, $\,V\,$ is not exceptional. 

iii) Let $\,i=\ell-3\,$ and $\,\mu\,=\,\omega_{\ell-3}\,$.
Thus for $\,\ell\geq 8$, $\,s(V)\,\geq\,|W\mu|\,=\,\displaystyle
\frac{2^{\ell-4}\,\ell(\ell-1)(\ell-2)}{3}\,>\,2\,\ell^2\,(\ell-1)\,$.
Hence, by~\eqref{dllim}, $\,V\,$ is not exceptional.

For $\,\ell=7,\,$ $\,\mu\,=\,\omega_4\,$ and  
$\,\mu_1\,=\,\mu-(\alpha_3+2\alpha_4+2\alpha_5+\alpha_6+\alpha_7)\,=
\,\omega_2\,\in\,{\cal X}_{++}(V)$. As $\,|R_{long}^+-R^+_{\mu,p}|\geq
24\,$ and $\,|R_{long}^+-R^+_{\mu_1,p}|\,\geq\,20$, one has
$\,r_p(V)\,\geq\,\displaystyle\frac{2^4\cdot 5\cdot 7\cdot 24}{2\cdot
7\cdot 6}\,+\,\frac{2^2\cdot 3\cdot 7\cdot 20}{2\cdot 7\cdot 6}\,=\, 
2^2\cdot 3^2\cdot 5\,>\,2\cdot 7\cdot 6\,$. Hence, by~\eqref{rrdl},
$\,V\,$ is not exceptional.

iv) Let $\,i=3\,$. We may assume that $\,V\,$ has highest weight 
$\,\mu\,=\,\omega_3\,$. Then {\bf for $\,\ell\geq 5$, $\,V\,$ is  
unclassified (N. 5 in Table~\ref{leftdn})}. 

v) Let $\,i=2\,$. We can assume that $\,V\,$ has highest weight 
$\,\mu\,=\,\omega_2\,$. Then $\,V\,$ is the adjoint module, which is 
exceptional by Example~\ref{adjoint} {\bf (N. 2 in Table~\ref{tabledlall})}.

vi) Let $\,i=1\,$. We may assume that $\,V\,$ has highest weight 
$\,\mu\,=\,\omega_{1}\,$. Then for $\,\ell\geq 4$,  
$\,\dim\,V\,=\,2\ell\,<\,2\ell(\ell-1)+\ell\,-\,\varepsilon\,$. Hence, by
Proposition~\ref{dimcrit}, $\,V\,$ is an exceptional module 
{\bf (N. 1 in Table~\ref{tabledlall})}.

vii) Let $\,i=\ell\,$ and $\,\mu\,=\,\omega_{\ell}\,$ (or by a graph-twist 
$\,\mu\,=\,\omega_{\ell-1}\,$). For $\,p\geq 2$, 
$\,|R_{long}^+-R^+_{\mu,p}|\,=\,\displaystyle\frac{(\ell-1)\ell}{2}\,$.
Thus, for $\,\ell\geq 11\,$,
$\,r_p(V)\,\geq\,\displaystyle\frac{2^{\ell-1}\,(\ell-1)\ell}{2\cdot 
2\ell(\ell-1)}\,=\,2^{\ell-3}\,>\,2\ell(\ell-1)\,$. Hence,
by~\eqref{rrbl}, $\,V\,$ is not an exceptional module. 
 
For $\,5\leq \ell\leq 7,\,$ if $\,V\,$ has highest weight
$\,\mu\,=\,\omega_{\ell}\,$ (or $\,\mu\,=\,\omega_{\ell-1}\,$), then
$\,\dim\,V\,=\,2^{\ell-1}\,<\,2\ell^2\,-\,\ell\,-\,\varepsilon$. Hence,
by Proposition~\ref{dimcrit}, $\,V\,$ is an exceptional module {\bf (N. 4
in Table~\ref{tabledlall})}. {\bf For $\,8\leq\ell\leq 10,\,$ these modules 
are unclassified (N. 6 in Table~\ref{leftdn})}.

{\bf Hence, for $\,\ell\geq 5\,$, if a $\,D_{\ell}(K)$-module $\,V\,$ 
has highest weight listed in Table~\ref{tabledlall}
(p. \pageref{tabledlall}), then $\,V\,$ an exceptional module.
If the highest weight of $\,V\,$ is different from the ones listed in
Tables~\ref{tabledlall} or~\ref{leftdn} (p. \pageref{leftdn}), then 
$\,V\,$ is not an exceptional module.}

This finishes the proof of Theorem~\ref{dnlist} for $\,\ell\geq 5$.
\hfill $\Box$

\subsubsection{Type $\,D_4\,$}

For groups of type $\,D_4\,$, $\,|W|\,=\,2^6\cdot 3$,
$\,|R|\,=\,|R_{long}|\,=\,2^3\cdot 4$. The limit for~\eqref{dllim} is
$\,2^6\cdot 3\,$ and for ~\eqref{rrdl} is $\,2^3\cdot 4$. First we prove a
reduction lemma.

\begin{lemma}\label{d4off}
Let $\,V\,$ be a $\,D_4(K)$-module. If $\,{\cal X}_{++}(V)\,$ contains 

(a) a (good) weight with $4$ nonzero coefficients and any other good weight or

(b) 2 good weights with $3$ nonzero coefficients and any other good weight,\\
then $\,V\,$ is not an exceptional module.
\end{lemma}\noindent
\begin{proof}
First note that weights with $4$ nonzero coefficients occurring in 
$\,{\cal X}_{++}(V)\,$ must be good weights. So
let $\,\mu\,=\,a\,\omega_{1}\,+\,b\,\omega_{2}\,+\,c\,\omega_{3}\,+\,d
\,\omega_{4}\,$ (with $\,a,\,b,\,c,\,d\,\geq 1\,$) be a (good) weight
in $\,{\cal X}_{++}(V)$. Then $\,|W\mu|= 2^6\cdot 3\,$. Thus, if (a)
holds, then $\,s(V)\,\geq\,2^6\cdot 3\,+\,k\,>\,2^6\cdot 3\,$, for
some $\,k>0$. Hence, by~\eqref{dllim}, $\,V\,$ is not exceptional.

%$\,\mu_1\,=\,\mu-(\alpha_1+\alpha_2)\,=\,(a-1)
%\omega_{1}\,+\,(b-1)\omega_{2}\,+\,(c+1)\omega_{3}\,+\,(d+1)\omega_{4}\;$ and
%$\,\mu_2\,=\,\mu-(\alpha_2+\alpha_3+\alpha_4)\,=\,(a+1)\omega_{1}\,+\,
%b\omega_{2}\,+\,(c-1)\omega_{3}\,+\,(d-1)\omega_{4}\,\in\,{\cal X}_{++}(V)$. 

For (b), let $\,\mu\,$ be a weight with $3$ nonzero coefficients. Then
$\,|W\mu|\,=\,2^5\cdot 3\,$. Thus, if $\,{\cal X}_{++}(V)\,$ contains 
2 good weights with $3$ nonzero coefficients and any other good weight, then 
$\,s(V)\,\geq\, %\sum_{\mu_i\,good}\,|W\mu_i|\,\geq\, 
2\cdot 2^5\cdot 3\,+\,k\,>\,2^6\cdot 3,\,$ where $\,k\,>\,0$. 
Hence, by~(\ref{dllim}), $\,V\,$ is not an exceptional module. This
proves the lemma. 
\end{proof}
\begin{corollary}\label{34d4}
Let $\,V\,$ be a $\,D_4(K)$-module. If $\,{\cal X}_{++}(V)\,$ contains

(a) a (good) weight with $4$ nonzero coefficients or

(b) a weight with $3$ nonzero coefficients, \\
then $\,V\,$ is not an exceptional module.
\end{corollary}\noindent
\begin{proof}
For (a), it is easy to show that one can always produce another good weight
in $\,{\cal X}_{++}(V)\,$ from a good weight $\,\mu\in{\cal X}_{++}(V)\,$
with $4$ nonzero coefficients. Hence, by Lemma~\ref{d4off}(a), $\,V\,$
is not exceptional. For (b) we prove some claims.

{\bf Claim 1}:\label{claim1d4} {\it Let $\,V\,$ be a $\,D_4(K)$-module.
If $\,{\cal X}_{++}(V)\,$ contains a bad weight with 
$3$ nonzero coefficients, then $\,V\,$ is not an exceptional module.} 

Indeed, let $\,\mu\,=\,a\omega_{i}\,+\,b\omega_{j}\,+\,c\omega_{k}\,$
(with $\,a,\,b,\,c\,\geq 1\,$ and $\,1\leq i<j<k\leq 4\,$)
be a bad weight in $\,{\cal X}_{++}(V)$. Hence $\,a\geq 2,\,b\geq
2,\,c\geq 2\,$. 

i) Let $\,\mu\,=\,a\omega_{1}\,+\,b\omega_{2}\,+\,c\omega_{3}\,$
(or by graph-twists $\,\mu\,=\,a\omega_{1}\,+\,b\omega_{2}\,+\,
c\omega_{4}\,$, $\,\mu\,=\,b\omega_{2}\,+\,c\omega_{3}\,+\,a\omega_{4}\,$).  
Then $\,\mu_1\,=\,\mu-(\alpha_1+\alpha_2)\,=\,
(a-1)\omega_{1}\,+\,(b-1)\omega_{2}\,+\,(c+1)\omega_{3}\,+\,\omega_{4}\,
\in\,{\cal X}_{++}(V)\,$ has $4$ nonzero coefficients. 
Hence, by part (a), $\,V\,$ is not an exceptional module. 

ii) Let $\,\mu\,=\,a\omega_{1}\,+\,b\omega_{3}\,+\,c\omega_{4}\,$. Then
$\,\mu_1\,=\,\mu-(\alpha_1+\alpha_2+\alpha_3)\,=\,(a-1)
\omega_{1}\,+\,(b-1)\omega_{3}\,+\,(c+1)\omega_{4}\,$ and 
$\,\mu_2\,=\,\mu_1-\alpha_4\,=\,(a-1)\omega_{1}\,+\,\omega_2\,+\,
(b-1)\omega_{3}\,+\,(c-1)\omega_{4}\,\in\,{\cal X}_{++}(V)$. As
$\,\mu_2\,$ has $4$ nonzero coefficients, part (a) applies.

%iii) Let $\,\mu\,=\,a\omega_{2}\,+\,b\omega_{3}\,+\,c\omega_{4}\,$. Then
%$\,\mu_1\,=\,\mu-(\alpha_2+\alpha_3+\alpha_4)\,=\,\omega_{1}\,
%+\,a\,\omega_{2}\,+\,(b-1)\,\omega_{3}\,+\,(c-1)\,\omega_{4}\,\in\,
%{\cal X}_{++}(V)\,$ has $4$ nonzero coefficients. Hence, by part
%(a), $\,V\,$ is not exceptional. This proves the claim.

{\bf Therefore we can assume that weights with $3$ nonzero coefficients
occurring in $\,{\cal X}_{++}(V)\,$ are good weights.}
 
1) Let $\,\mu\,=\,a\omega_{1}\,+\,b\omega_{2}\,+\,c\omega_{3}\,$
(or by graph-twists $\,\mu\,=\,a\omega_{1}\,+\,b\omega_{2}\,+\,
c\omega_{4}\,$, $\,\mu\,=\,b\omega_{2}\,+\,c\omega_{3}\,+\,a\omega_{4}\,$). 
Then $\,\mu_1\,=\,\mu-(\alpha_1+\alpha_2)\,=\,
(a-1)\omega_{1}\,+\,(b-1)\omega_{2}\,+\,(c+1)\omega_{3}\,+\,\omega_{4}\,$,
$\,\mu_2\,=\,\mu_1-\alpha_3\,=\,(a-1)\omega_{1}\,+\,b
\omega_{2}\,+\,(c-1)\omega_{3}\,+\,\omega_{4}\,\in\,{\cal X}_{++}(V)$.

For $\,a\geq 2,\,b\geq 2,\,c\geq 1\,$, $\,\mu_1\,$ has $4$ nonzero  
coefficients, hence by (a) $\,V\,$ is not exceptional.
For ($\,a=1,\,b\geq 2,\,c\geq 1\,$) or ($\,a\geq 2,\,b=1,\,c\geq 1\,$)
or ($\,a=1,\,b=1,\,c\geq 2\,$), 
$\,\mu,\,\mu_1,\,\mu_2\,$ satisfy Lemma~\ref{d4off}(b).

For $\,a=b=c=1,\,$ $\,\mu\,=\,\omega_{1}\,+\,\omega_{2}\,+\,\omega_{3}\,$
(or $\,\mu\,=\,\omega_{1}\,+\,\omega_{2}\,+\,\omega_{4}\,$). Then
$\;\mu_1\,=\,2\omega_{3}\,+\,\omega_{4},\;$
$\;\mu_2\,=\,\omega_{2}\,+\,\omega_{4},\;$
$\;\mu_3\,=\,\mu_2-(\alpha_2+\alpha_4)\,=\,\omega_{1}\,+\,\omega_{3},\;$
and $\,\mu_4\,=\,\mu_3-(\alpha_1+\alpha_2+\alpha_3)\,=\,\omega_{4}\,
\in\,{\cal X}_{++}(V)$. Thus $\,s(V)\,\geq\,2^5\cdot 3\,+\,2^5\,+\,
2^4\cdot 3\,+\,2^5\,+\,2^3\,=\,2^3\cdot 3^3\,>\,2^6\cdot 3\,$. Hence,
by~(\ref{dllim}), $\,V\,$ is not an exceptional module.

2) Let $\,\mu\,=\,a\omega_{1}\,+\,b\omega_{3}\,+\,c\omega_{4}\,$ (with
$\,a,\,b,\,c\,$) be a good weight in $\,{\cal X}_{++}(V)$. 

2.i) For $\,a\geq 2,\,b\geq 1,\,c\geq 1\,$, $\,\mu_1\,=\,\mu-\alpha_1\,
=\,(a-2)\,\omega_1\,+\,\omega_{2}\,+\,b\,\omega_{3}\,+\,c\,\omega_{4}\,$,
$\,\mu_2\,=\,\mu_1-(\alpha_2+\alpha_3+\alpha_4)\,=\,(a-1)\omega_{1}\, 
+\,\omega_{2}\,+\,(b-1)\omega_{3}\,+\,(c-1)\omega_{4}\in {\cal X}_{++}(V)$.
For $\,a\geq 3$, $\,\mu_1\,$ has $4$ nonzero coefficients. Hence part (a)
of this corollary applies. For $\,a=2,\,b\geq 1,\,c\geq 1\,$, 
$\,\mu,\,\mu_1,\,\mu_2\,$ satisfy Lemma~\ref{d4off}(b). Hence 
$\,V\,$ is not exceptional.

2.ii) Let $\,a=1,\,b\geq 1,\,c\geq 2\,$ and $\,\mu\,=\,\omega_{1}\,+\,
b\omega_{3}\,+\,c\omega_{4}\,$. Then $\,\mu_1\,=\,\mu-\alpha_4\,=\,
\omega_1\,+\,\omega_{2}\,+\,b\omega_{3}\,+\,(c-2)\omega_{4}\,$,
$\,\mu_2\,=\,\mu_1-(\alpha_1+\alpha_2+\alpha_3)\,=\,
\omega_{2}\,+\,(b-1)\omega_{3}\,+\,(c-1)\omega_{4}\,\in\,{\cal X}_{++}(V)$.
For $\,c\geq 3$, $\,\mu_1\,$ has $4$ nonzero coefficients. Hence part (a)
of this corollary applies. For $\,a=1,\,b\geq 1,\,c=2\,$, 
$\,\mu,\,\mu_1,\,\mu_2\,$ satisfy Lemma~\ref{d4off}(b). Hence 
$\,V\,$ is not exceptional.

2.iii) Let $\,a=b=c=1\,$ and $\,\mu\,=\,\omega_{1}\,+\,\omega_{3}\,+\,
\omega_{4}\,$. Then, for $\,p\geq 2,\,$ $\,|R_{long}^+-R^+_{\mu,p}|\geq 8\,$.
Thus $\,r_p(V)\,\geq\,\displaystyle\frac{2^5\cdot 3\cdot 8} 
{2\cdot 4\cdot 3}\,=\,32\,>\,2\cdot 4\cdot 3\,$. Hence,
by~(\ref{rrdl}), $\,V\,$ is not exceptional.  
This proves the corollary.

%3) Let $\,\mu\,=\,a\omega_{2}\,+\,b\omega_{3}\,+\,c\omega_{4}\,$ (with
%$\,a,\,b,\,c\,$) be a good weight in $\,{\cal X}_{++}(V)$. Then
%$\,\mu_1\,=\,\mu-(\alpha_2+\alpha_3+\alpha_4)\,=\,\omega_{1}\,
%+\,a\,\omega_{2}\,+\,(b-1)\,\omega_{3}\,+\,(c-1)\,\omega_{4},\,$
%$\,\mu_2\,=\,\mu-(\alpha_2+\alpha_3)\,=\,\omega_1\,+\,(a-1)\,\omega_{2}\,
%+\,(b-1)\,\omega_{3}\,+\,(c+1)\,\omega_{4}\in{\cal X}_{++}(V)$. 
%
%3.i) For ($\,a\geq 1,\,b\geq 2,\,c\geq 1\,$) or ($\,a\geq 1,\, 
%b\geq 1,\,c\geq 2\,$), $\,\mu,\,\mu_1,\,\mu_2\,$ satisfy 
%Lemma~\ref{d4off}(b). Hence $\,V\,$ is not exceptional.  
%
%3.ii) Let $\,a\geq 1,\,b=c= 1\,$ and $\,\mu\,=\,a\omega_{2}\,+\, 
%\omega_{3}\,+\,\omega_{4}\,$ Then $\,\mu_1\,=\,\omega_{1}\,+\,a\omega_{2},\;$ 
%$\,\mu_2\,=\,\omega_1\,+\,(a-1)\,\omega_{2}\,+\,2\omega_{4},\;$ and
%$\,\mu_3\,=\,\omega_1\,+\,(a-1)\,\omega_{2}\,+\,2\omega_{3} 
%\,\in\,{\cal X}_{++}(V)$. 
%For $\,a\geq 2,\,$ Lemma~\ref{d4off}(b) applies. For $\,a=b=c=1\,$
%$\,\mu\,=\,\omega_{2}\,+\,\omega_{3}\,+\,\omega_{4},\;$
%$\,\mu_1\,=\,\omega_{1}\,+\,\omega_{2}\,\in\,{\cal X}_{++}(V)$.
%For $\,p\geq 2$,  $\,|R_{long}^+-R^+_{\mu,p}|\geq 6\,$ and 
%$\,|R_{long}^+-R^+_{\mu_1,p}|\geq 6$. Thus $\,r_p(V)\, \geq\,
%\frac{2^5\cdot 3\cdot 6}{2\cdot 4\cdot 3}\,+\,\frac{2^4\cdot 3\cdot 6}
%{2\cdot 4\cdot 3}\,=\,36\,>\,2\cdot 4\cdot 3\,$.
%Hence, by~(\ref{rrdl}), $\,V\,$ is not exceptional. 
\end{proof}\vspace{2ex}\noindent
{\bf Proof of Theorem~\ref{dnlist} for $\,\ell=4$.}
Let $\,V\,$ be a $\,D_4(K)$-module. 

By Corollary~\ref{34d4}, {\bf it
suffices to consider modules $\,V\,$ such that $\,{\cal X}_{++}(V)\,$
contains only weights with at most $2$ nonzero coefficients.}

I) Let $\,\mu\,=\,a\,\omega_i\,+\,b\,\omega_j\,$ (with $\,1\leq i<j\leq 4\,$ 
and $\,a\geq 1,\,b\geq 1\,$) be a weight in $\,{\cal X}_{++}(V)$.

(a) Let $\,\mu\,=\,a\,\omega_1\,+\,b\,\omega_2\,$ (or by graph-twists
$\,\mu\,=\,b\,\omega_2\,+\,a\,\omega_3\,$, $\,\mu\,=\,b\,\omega_2\,+\,
a\,\omega_4\,$). Then $\,\mu_1\,=\,\mu-(\alpha_1+\alpha_2)\,=\, 
(a-1)\omega_{1}\,+\,(b-1)\omega_{2}\,+\,\omega_3\,+\,\omega_4\,\in\, 
{\cal X}_{++}(V)$.

i) If $\,\mu\,$ is a bad weight (hence $\,a\geq 2,\,b\geq 2\,$), then 
$\,\mu_1\,$ has $4$ nonzero coefficients. Hence, by
Corollary~\ref{34d4}(a), $\,V\,$ is not exceptional.

ii)\label{Iaiid4} Let $\,\mu\,$ be a good weight. Then for ($\,a\geq
1,\,b\geq 2\,$) or ($\,a\geq 2,\,b\geq 1\,$), $\,\mu_1\,$ has $3$
nonzero coefficients. Hence, by Corollary~\ref{34d4}(b), $\,V\,$ is
not exceptional.

{\bf Therefore, if $\,{\cal X}_{++}(V)\,$ contains weights
$\,\mu\,=\,a\,\omega_1\,+\,b\,\omega_2\,$ (or $\,\mu\,=\,b\,\omega_2\,
+\,a\,\omega_3\,$, $\,\mu\,=\,b\,\omega_2\,+\,a\,\omega_4\,$), then 
we can assume that $\,a=b=1\,$.} These cases are treated in (c) below.

(b) Let $\,\mu\,=\,a\,\omega_1\,+\,b\,\omega_3,\,$ (or by graph-twists
$\,\mu\,=\,a\,\omega_1\,+\,b\,\omega_4\,$, $\,\mu\,=\,a\,\omega_3\,+\,
b\,\omega_4\,$). Then $\,\mu_1\,=\,\mu-(\alpha_1+\alpha_2+\alpha_3)\, 
=\,(a-1)\omega_{1}\,+\,(b-1)\omega_3\,+\,\omega_4\,\in{\cal X}_{++}(V)$. 

i) If $\,\mu\,$ is a bad weight (hence $\,a\geq 2,\,b\geq 2\,$), then 
$\,\mu_1\,$ satisfies Corollary~\ref{34d4}(b). Hence $\,V\,$ is not 
exceptional. 

ii) Let $\,\mu\,$ be a good weight. Then, for $\,a\geq 2,\,b\geq 2\,$,
$\,\mu_1\,$ satisfies Corollary~\ref{34d4}(b). Hence $\,V\,$ is not 
exceptional. 

Let $\,a\geq 2,\,b=1\,$ and $\,\mu\,=\,a\,\omega_1\,+\,\omega_3\,$
(or by a graph-twist $\,\mu\,=\,\omega_1\,+\,a\,\omega_3\,$). Then
$\,\mu_1\,=\,\mu-\alpha_1\,=\,(a-2)\omega_1\,+\,\omega_{2}\,+\,\omega_3\,
\in\,{\cal X}_{++}(V)$. For $\,a\geq 3$, $\,\mu_1\,$ satisfies
Corollary~\ref{34d4}(b). For $\,a=2,\,$ $\,\mu\,=\mu_0=\,
2\omega_1\,+\,\omega_3\,$ $\,\mu_1\,=\,\omega_{2}\,+\,\omega_3\,$, 
$\,\mu_2\,=\,\mu_1-(\alpha_2+\alpha_3)\,=\,\omega_{1}\,+\,\omega_4\, 
\in\,{\cal X}_{++}(V)$. For $\,p\geq 2$, $\,|R_{long}^+-R^+_{\mu_i,p}| 
\geq 6\,$. Thus $\,r_p(V) \geq\displaystyle\frac{2^5\cdot 6}{2\cdot  
4\cdot 3}\,+\,\frac{2^4\cdot 3\cdot 6}{2\cdot 4\cdot 3}\,+\,
\frac{2^5\cdot 6}{2\cdot 4\cdot 3}\,=\,28\,>\,2\cdot 4\cdot
3\,$. Hence, by~(\ref{rrdl}), $\,V\,$ is not exceptional. 

%For $\,a=1,\,b\geq 2,\,$ $\,\mu\,=\,\omega_1\,+\,b\,\omega_3,\,$ 
%$\,\mu_1\,=\,\mu-\alpha_3\,=\,\omega_1\,+\,\omega_{2}\,+\,(b-2)\omega_3\,
%\in\,{\cal X}_{++}(V)$. For $\,b\geq 3$, $\,\mu_1\,$ satisfies
%Corollary~\ref{34d4}(b). For $\,b=2,\,$ $\,\mu\,=\mu_0=\,
%\omega_1\,+\,2\omega_3\,$ $\,\mu_1\,=\,\omega_{1}\,+\,\omega_2\,$, 
%$\,\mu_2\,=\,\mu_1-(\alpha_1+\alpha_2)\,=\,\omega_{3}\,+\,\omega_4\, 
%\in\in\,{\cal X}_{++}(V)$. For $\,p\geq 2$, $\,|R_{long}^+-R^+_{\mu_i,p}| 
%\geq 6\,$. Thus $\,r_p(V) \geq\,\frac{2^5\cdot 6}{2\cdot 4\cdot 3}\,+\,
%\frac{2^4\cdot 3\cdot 6}{2\cdot 4\cdot 3}\,+\,
%\frac{2^5\cdot 6}{2\cdot 4\cdot 3}\,=\,28\,>\,2\cdot 4\cdot
%3\,$. Hence, by~(\ref{rrdl}), $\,V\,$ is not exceptional. 

{\bf Therefore if $\,{\cal X}_{++}(V)\,$ contains weights
$\,\mu\,=\,a\,\omega_i\,+\,b\,\omega_j\,$ with $\,1\leq i<j\leq 4\,$
and  $\,a\geq 1,\,b\geq 1\,$, then we can assume that $\,a=b=1\,$.} 
These cases are treated in the sequel.

(c)i) Let $\,a=b=1\,$ and $\,\mu\,=\,\omega_1\,+\,\omega_2\,$ (or by
graph-twists $\,\mu\,=\,\omega_2\,+\,\omega_3\,$, $\,\mu\,=\,\omega_2\, 
+\,\omega_4\,$). Then $\,\mu_1\,=\,\omega_3\,+\,\omega_4,\;$ 
$\,\mu_2\,=\,\mu_1-(\alpha_2+ \alpha_3+\alpha_4)\,=\,\omega_1\,
\in\,{\cal X}_{++}(V)$.  
For $\,p\geq 3,\,$ $\,|R_{long}^+-R^+_{\mu,p}|\,\geq\,9,\, 
|R_{long}^+-R^+_{\mu_1,p}|= 9,\,|R_{long}^+-R^+_{\mu_2,p}|\,=\,6.\,$
Thus $\,r_p(V)\geq\,\displaystyle\frac{2^4\cdot 3\cdot 9}{2\cdot 4\cdot 3}\,+\,
\frac{2^5\cdot 9}{2\cdot 4\cdot 3}\,+\,\frac{2^3\cdot 6}{2\cdot 4\cdot 3}\,
=\,32\,>\,2\cdot 4\cdot 3\,$. Hence, by~(\ref{rrdl}), $\,V\,$ is not
an exceptional module.  

For $\,p=2,\,$ we may assume that $\,V\,$ has highest weight
$\,\mu\,=\,\omega_1\,+\,\omega_2\,$. Then, by~\cite[p. 169]{buwil},  
$\,m_{\mu_1}= 2,\,m_{\mu_2}=6$. As $\,|R_{long}^+-R^+_{\mu,2}|\,
=\,|R_{long}^+-R^+_{\mu_i,2}|=6$, one has
$\,r_2(V) \geq\,\displaystyle\frac{2^4\cdot 3\cdot 6}{2\cdot 4\cdot 3}\,+\,
2\frac{2^5\cdot 6}{2\cdot 4\cdot 3}\,+\,6\frac{2^3\cdot 6}{2\cdot 4\cdot 3}\,
=\,2^3\cdot 5\,>\,2\cdot 4\cdot 3\,$. Hence, by~(\ref{rrdl}), 
$\,V\,$ is not an exceptional module.

ii) Let $\,a=b=1\,$. We may assume that $\,V\,$ has highest weight
$\,\mu\,=\,\omega_1\,+\,\omega_3\,$ (or by graph-twists 
$\,\mu\,=\,\omega_1\,+\,\omega_4\,$, $\,\mu\,=\,\omega_3\,+\,\omega_4\,$). 
Then $\,V\,$ is unclassified {\bf (N. 3 in Table~\ref{leftdn})}.

{\bf Hence if $\,V\,$ is a $\,D_4(K)$-module such that $\,{\cal
X}_{++}(V)\,$ contains $\,\mu\,=\,a\omega_i\,+\,b\omega_j\,$ (with
$\,1\leq i<j\leq 4\,$ and  $\,a\geq 1,\,b\geq 1\,$, then 
$\,V\,$ is not an exceptional module, unless $\,V\,$ has highest weight
$\,\lambda\,\in\,\{ \omega_1\,+\,\omega_3,\;\omega_1\,+\,\omega_4,\;
\omega_3\,+\,\omega_4\}\,$ in which cases $\,V\,$ is unclassified.}

{\bf Therefore we can assume that $\,{\cal X}_{++}(V)\,$ contains only
weights with at most one nonzero coefficient.} 

II) Let $\,\mu\,=\,a\,\omega_i\,$ (with $\,a\geq 1\,$) be a weight in
$\,{\cal X}_{++}(V)$.

(a) Let $\,a\geq 2\,$ and $\,\mu\,=\,a\,\omega_1\,$ (or by
graph-twists $\,\mu\,=\,a\,\omega_3\,$, $\,\mu\,=\,a\,\omega_4\,$).
Then $\,\mu_1\,=\,\mu-\alpha_1\,=\,(a-2)\omega_1\,+\,\omega_2\in
\,{\cal X}_{++}(V)$.

i) For $\,a\geq 4$, $\,\mu_1\,$ satisfies case I(a)ii)
(p. \pageref{Iaiid4}) of this proof. 

ii) For $\,a=3,\,$ $\,\mu\,=\,3\omega_1\,$, $\,\mu_1\,=\,\omega_1\,+\,
\omega_2\,$, $\,\mu_2\,=\,\mu_1-(\alpha_1+\alpha_2)\,=\,\omega_3\,+
\,\omega_4\,$ and $\,\mu_3\,=\,\mu_2-(\alpha_2+\alpha_3+\alpha_4)\,=\,
\omega_1\,\in\,{\cal X}_{++}(V)$. As $\,|R_{long}^+-R^+_{\mu_i,p}|\geq
6,\,$ $\,r_p(V)\geq\,\displaystyle\frac{2^4\cdot 3\cdot 6}
{2\cdot 4\cdot 3}\,+\,\frac{2^4\cdot 3\cdot 6}{2\cdot 4\cdot 3}\, 
+\,\frac{2^3\cdot 6}{2\cdot 4\cdot 3}\,=\,26\,>\,2\cdot 4\cdot 3\,$. 
Hence, by~(\ref{rrdl}), $\,V\,$ is not an exceptional module.  

iii) Let $\,a=2\,$ and $\,\mu\,=\,2\omega_1\,$. For $\,p=2\,$, 
$\,2\omega_1\in{\cal X}_{++}(V)\,$ only if $\,V\,$ has highest weight
$\,\lambda=\omega_1+\omega_3+\omega_4\,$. In this case $\,V\,$ is not 
exceptional, by Corollary~\ref{34d4}(b). For $\,p\geq 3$, we may assume that 
$\,V\,$ has highest weight $\,\mu\,=\,2\omega_1\,$ (or by graph-twists
$\,\mu\,=\,2\omega_3\,$, $\,\mu\,=\,2\omega_4\,$).
Then, by the same argument used in 
case III(b)iv) (p. \pageref{IIIbiv}) of First Part of Proof of
Theorem~\ref{listbn}, $\,V\,$ is not an exceptional module. 

(b) Let $\,a\geq 2\,$ and $\,\mu\,=\,a\,\omega_2\,$. Then
$\,\mu_1\,=\,\mu-\alpha_2\,=\,\omega_1\,+\,(a-2)\omega_2\,+\,\omega_3\,+\,
\omega_4\,\in\,{\cal X}_{++}(V)$. Hence, by Corollary~\ref{34d4} (a)
or (b), $\,V\,$ is not exceptional.
 
{\bf Therefore, if $\,{\cal X}_{++}(V)\,$ contains a weight $\,\mu\, 
=\,a\,\omega_i\,$ with $\,a\geq 1$, then we can assume $\,a=1\,$ and
that $\,V\,$ has highest weight $\,\omega_i\,$.} These
cases are treated in the sequel.

(c) If $\,V\,$ has highest weight $\,\omega_1\,$ (or by graph-twists 
$\,\omega_3\,$, $\,\omega_4\,$), then $\,\dim\,V\,=\,2^3\,<\, 
20\,-\,\varepsilon\,$. Hence, by Proposition~\ref{dimcrit}, $\,V\,$ is
an exceptional module {\bf (N. 1 (resp., 3) in Table~\ref{tabledlall})}.
  
(d) If $\,V\,$ has highest weight $\,\omega_2\,$, then $\,V\,$ is the
adjoint module, which is exceptional by Example~\ref{adjoint} {\bf (N. 2 in
Table~\ref{tabledlall})}.

{\bf Hence if $\,V\,$ is a $\,D_4(K)$-module such that
$\,{\cal X}_{++}(V)\,$ contains a weight $\,\mu\,=\,a\,\omega_i\,$
with $\,a\geq 2$, then $\,V\,$ is not exceptional. If  
$\,V\,$ has highest weight $\,\omega_i\,$ (for any $\,1\leq i\leq
4\,$), then $\,V\,$ is an exceptional module.}

This finishes the proof of Theorem~\ref{dnlist} for $\,\ell=4$.
\hfill $\Box$

\newpage
\section{Equalities and Inequalities}

\subsection{Basic Combinatorics}

Here we recall a couple of equalities and inequalities often used when
dealing with binomials. They all can be found in any elementary book
on Combinatorial Mathematics, as for instance~\cite{ando}.

The number of ways of choosing $\,r\,$ objects from $\,n\,$ given
objects, without taking order into account, is given and denoted by
\[
\binom{n}{r}\,=\,\frac{n!}{(n-r)!\,r!}\,.
\]
Important properties of the numbers $\,\displaystyle
\binom{n}{r}\,$ are given in the following theorem. 
The convention that $\,\displaystyle\binom{n}{0}\,=\,1\,$ 
is followed.

\begin{theorem}\label{usual}

(a) $\displaystyle\binom{n}{r}\,=\,\binom{n}{n-r}\,$,\; for 
$\,0\leq r\leq n\,$.

(b) $\displaystyle\binom{n}{r}\,=\,\binom{n-1}{r-1}\,+\,\binom{n-1}{r}\,$. 
\vspace{.6ex}

(c) $\,i(\ell+1-i)\,<\,(i+1)(\ell-i)\,$.
\vspace{.6ex}

(d) $\,\displaystyle{\binom{\ell+1}{i}\,<\,
\binom{\ell+1}{i+1}}$,\qquad 
if $\,2\,\leq\, i+1\,\leq\,\displaystyle\frac{1}{2}(\ell+1)\,$. 
\end{theorem}

\subsection{The inequality~(\ref{ineq1})}\label{appineq1}
Let $\,n\in\mathbb N\,$ be such that 
$\,a_1\,+\,a_2\,+\,\cdots\,+\,a_k\,=\,n\,$ and consider the binomial
expression 
\[
\binom{n}{a_1,\,a_2,\,\cdots,\,a_k}\,=\,\frac{n!}{a_1!\,a_2!\,\cdots\,a_k!}
\]
i.e., the number of ways of dividing a set $\,S\,$ with $\,n\,$ elements
into an ordered $\,k$-tuple of subsets with $\,a_1,\,a_2,\,\cdots,\,a_k\,$
elements.

Now assume that each subset is non-empty, that is, $\,a_i\geq 1\,$ for
$\,1\leq i\leq k.\,$ In this case we have
\begin{eqnarray}\label{aaaa}
\binom{n}{a_1,\,a_2,\,\cdots,\,a_k}\,\geq\,n\,(n-1)\,\cdots\,(n-k+2). & &
\end{eqnarray}
The right hand side of this inequality is the number of ordered
$\,(k-1)$-tuples of elements of the set $\,S$. Let 
$\,X\,=\,(\,x_1,\,\ldots,\,x_{k-1}\,)\,$ be a typical $\,(k-1)$-tuple.
We can associate to $\,X\,$ a decomposition of $\,S\,$ into
$\,k\,$ subsets $\,S_1,\,\ldots,\,S_k,\,$ with $\,a_1,\,\ldots,\,a_k\,$
elements respectively, as follows.

First assign $\,x_i\,$ to $\,S_i\,$ for $\,1\leq i\leq k-1.\,$
Now complete $\,S_1\,$ by using elements of $\,S\,$ in cyclic order
starting with $\,x_1.\,$ Then complete $\,S_2\,$ in the same way starting
with $\,x_2\,$, etc. At each stage omit previously numbered elements of 
$\,S.\,$ Finally, assign remaining elements of $\,S\,$ to $\,S_k.\,$
This constructs distinct $\,k$-tuples $\,S_1,\,\ldots,\,S_k,\,$
from distinct $\,(k-1)$-tuples $\,X,\,$ proving~(\ref{aaaa}).

As corollary of~(\ref{aaaa}) we have
\[
\binom{\ell+1}{i_1,\,(i_2-i_1),\,(i_3-i_2),\,(i_4-i_3),\,(\ell-i_4+1)}\,
\geq\,(\ell+1)\,\ell\,(\ell-1)\,(\ell-2) \,.
\]
Hence the inequality~(\ref{ineq1}) holds, that is:
\[
\frac{(\ell-3)!}{i_1!\,(i_2-i_1)!\,(i_3-i_2)!\,(i_4-i_3)!\,(\ell-i_4+1)!}\,
\geq\,1.
\]

\subsection{The inequality~(\ref{orb3c})}\label{apporb3c}

For $\,3\leq i_3\leq\ell-2\,$, by Theorem~\ref{usual}(d),
$\,\displaystyle\binom{\ell+1}{i_3}\geq \binom{\ell+1}{3}\,$. Thus 
for $\,3\leq i_3\leq\ell-2\,$,
\begin{eqnarray*}
\displaystyle
\frac{(\ell+1)!}{i_1!\,(i_2-i_1)!\,(i_3-i_2)!\,(\ell-i_3+1)!} & = &
\displaystyle\binom{i_2}{i_1}\,\binom{i_3}{i_2}\,\binom{\ell+1}{i_3}
\vspace{1ex} \\
& \geq  & \displaystyle 2\cdot 3\cdot \binom{\ell+1}{3}
\,=\, (\ell+1)\,\ell\,(\ell-1)\,. 
\end{eqnarray*}

For $\,i_3=\ell\,$ and $\,2\leq i_2\leq\ell-2\,$, we have
\begin{eqnarray*}
\displaystyle
\frac{(\ell+1)!}{i_1!\,(i_2-i_1)!\,(\ell-i_2)!\,1!} & = &
\displaystyle\binom{i_2}{i_1}\,(\ell+1)\,\binom{\ell}{i_2}\,
 \geq \, \displaystyle 2\,(\ell+1)\, \binom{\ell}{i_2}\vspace{1ex}\\
& \geq & \,\displaystyle 2\,(\ell+1)\,\binom{\ell}{2}
= \,(\ell+1)\,\ell\,(\ell-1)\,.%,\;\,\mbox{for}\,\;2\leq i_2\leq\ell-2.
\end{eqnarray*}

For $\,i_2=\ell-1\,$, $\,i_3=\ell\,$ and $\,1\leq i_1\leq\ell-2\,$,
\[
\frac{(\ell+1)!}{i_1!\,(\ell-1-i_1)!\,1!} \, \geq 
(\ell+1)\,\ell\,(\ell-1)\,.
\]
For $\,i_3=\ell-1\,$ and $\,2\leq i_2\leq\ell-2\,$, we have
\begin{eqnarray*}
\displaystyle
\frac{(\ell+1)!}{i_1!\,(i_2-i_1)!\,(\ell-1-i_2)!\,2!} & = & \displaystyle
\binom{i_2}{i_1}\,\frac{(\ell+1)\,\ell}{2}\,\binom{\ell-1}{i_2}
\vspace{2ex} \\
& \geq & \displaystyle 2\,\frac{(\ell+1)\,\ell}{2}\,\binom{\ell-1}{2}
 \geq  (\ell+1)\,\ell\,(\ell-1).
\end{eqnarray*}

\subsection{The inequality~(\ref{orb2c})}\label{apporb2c}
For $\,2\leq i_2\leq\ell-1\,$, by Theorem~\ref{usual}(d),
$\,\displaystyle\binom{\ell+1}{i_2}\geq \binom{\ell+1}{2}\,$. Thus 
for $\,2\leq i_2\leq\ell-1\,$,
\[
\frac{(\ell+1)!}{i_1!\,(i_2-i_1)!\,(\ell-i_2+1)!} = \displaystyle
\binom{i_2}{i_1}\,\binom{\ell+1}{i_2}\,\geq\,2\,\frac{(\ell+1)\,\ell}{2}\,
 \geq  (\ell+1)\,\ell.
\]\vspace{2ex}

For $\,i_2=\ell\,$ and $\,1\leq i_1\leq\ell-1\,$,
$\;\displaystyle
\frac{(\ell+1)!}{i_1!\,(\ell-i_1)!} \,= \,
(\ell+1)\,\binom{\ell}{i_1}\,\geq\,(\ell+1)\,\ell\,$.

\section{The Truncated Polynomial Modules} \label{trunc}

Let $\,n\geq 2\,$ be an integer. Let $\,p\,$ be a prime and fix an
algebraically closed field of characteristic $\,p$.

Let $\,S\,=\,K[\,x_1,\ldots,\,x_{n}\,]\,$ be the ring of polynomials in 
$\,n\,$ commuting indeterminates $\,x_1,\ldots,\,x_{n}\,$. 
Write $\,x\,=\,(\,x_1,\ldots,\,x_{n}\,)\,$ for the ``row vector''
of indeterminates. The general linear group $\,G\,=\,\GL(n,K)\,$ acts
naturally on $\,S\,$ by linear substitutions on the left by
\[
(g\cdot f)(x)\,=\,f(xg),\qquad g\in\GL(n,K),\;f\in S\,,
\]
where the product $\,xg\,$ denotes the matrix multiplication as usual. The
$\,K$-algebra $\,S\,$ has a natural grading by homogeneous degree:
\[
S\,=\,\bigoplus_{d\geq 0}\,S_d\,,
\]
where $\,S_d\,$ is the span of the monomials in $\,S\,$ of total degree
$\,d\,$ in $\,x_1,\ldots,\,x_{n}\,$. For each $\,d\geq 0$, $\,S_d\,$ is
a $\,G$-submodule of $\,S,\,$ isomorphic to the $\,d$th symmetric power
of the natural module for $\,G.\,$ Its dimension is $\,\displaystyle
\binom{n+d-1}{d}\,$,
the number of unordered partitions of $\,d\,$ into not more than 
$\,n\,$ parts, that is, the number of nonnegative integral solutions to
the equation $\,a_1\,+\,\cdots\,+\,a_n\,=\,d\,$.

Let $\,I\,=\,\langle x_1^p,\,\ldots,\,x^p_n\rangle,\,$ the ideal of $\,S\,$ 
generated by the $\,p$th powers. $\,I\,$ is a $\,G$-submodule of $\,S.\,$
The algebra $\,\displaystyle\frac{S}{I}\,$ is called the {\bf
truncated symmetric algebra} in $\,n\,$ indeterminates.
%% because $(x\,+\,y)^p\,=\,x^p\,+\,y^p\,$. Recalling that G acts by
%% linear substituitions.
For each nonnegative integer $\,d$, let $\,T_d\,$ denote the image of $\,S_d\,$
in $\,\displaystyle\frac{S}{I}\,$. $\,T_d\,$ is the so called 
{\bf truncated symmetric power} (or the $\,d^th$-graded component of
the truncated symmetric algebra).
Clearly, $\,T_d\,=\,0\,$ for $\,d> n(p-1)\,$. 
%%because in S_{np} for the sum of the powers of n variable to be 
%%np, at least one of them must have power at least p 
Thus:
\[
\frac{S}{I}\,=\,\bigoplus_{d=0}^{n(p-1)}\,T_d\,.
\]
It turns out that $\,T_d\,$ is simple as a $\,G$-module 
for each $\,d\,$. (This follows 
from the main result of \cite{doty}). 
All weight spaces of $\,T_d\,$ have multiplicity one (as they are monomials).

For each $\,d\in \mathbb N,\,$ we have 
$\,d\,=\,(p-1)\,s\,+\,k,\,$ where $\,0\leq s,\,0\leq k\,<\,p-1\,$.
Let $\,v\,=\,x_1^{p-1}\,x_2^{p-1}\cdots x_s^{p-1}\,x_{s+1}^k\,\in\,T_d\,$.
Then $\,v\,$ is a maximal vector of $\,T_d\,$ of 
weight $\,\lambda\,=\,k\,(\varepsilon_1\,+\,\cdots\,+\,
\varepsilon_{s+1})\,+\,(p-1-k)\,(\varepsilon_1\,+\,\cdots\,+\,
\varepsilon_s)\,=\,(p-1-k)\,\omega_s\,+\,k\,\omega_{s+1}\,$.
$\,v\,$ is unique up to a scalar multiple.

The other weights of $\,T_d\,$ are smaller than $\,\lambda\,$ with 
respect to our partial ordering.
All dominant weights of $\,T_d\,$ belong to $\,(\lambda\,-\,Q^+)\cap 
P_{++}\,$.

\end{document}